\documentclass[oneside,english,a4paper]{scrbook}
\usepackage[T1]{fontenc}
\usepackage[latin9]{inputenc}
\usepackage{verbatim}
\usepackage{mathrsfs}
\usepackage{amsthm}
\usepackage{amsmath}
\usepackage{amssymb}

\makeatletter

\providecommand{\tabularnewline}{\\}

\numberwithin{equation}{section}
\numberwithin{figure}{section}
\numberwithin{table}{section}
  \theoremstyle{plain}
  \newtheorem*{thm*}{\protect\theoremname}
\theoremstyle{plain}
\newtheorem{thm}{\protect\theoremname}[section]
  \theoremstyle{definition}
  \newtheorem{problem}[thm]{\protect\problemname}
  \theoremstyle{plain}
  \newtheorem*{prop*}{\protect\propositionname}
  \theoremstyle{plain}
  \newtheorem*{conjecture*}{\protect\conjecturename}
  \theoremstyle{plain}
  \newtheorem*{cor*}{\protect\corollaryname}
  \theoremstyle{definition}
  \newtheorem{defn}[thm]{\protect\definitionname}
  \theoremstyle{remark}
  \newtheorem*{rem*}{\protect\remarkname}
  \theoremstyle{plain}
  \newtheorem{prop}[thm]{\protect\propositionname}
  \theoremstyle{remark}
  \newtheorem{rem}[thm]{\protect\remarkname}
  \theoremstyle{definition}
  \newtheorem{example}[thm]{\protect\examplename}
  \theoremstyle{plain}
  \newtheorem{cor}[thm]{\protect\corollaryname}
  \theoremstyle{remark}
  \newtheorem{notation}[thm]{\protect\notationname}
  \theoremstyle{plain}
  \newtheorem{lem}[thm]{\protect\lemmaname}
  \theoremstyle{plain}
  \newtheorem{conjecture}[thm]{\protect\conjecturename}

\usepackage{extarrow}
\usepackage{enumerate}
\usepackage{fontenc}
\usepackage[T1]{fontenc}
\usepackage{graphicx}
\usepackage[colorlinks=true, linkcolor=blue]{hyperref}
\usepackage{latexsym}
\usepackage{mathrsfs}
\usepackage{mathtext}
\usepackage{tikz}
\usepackage{textcomp}
\usepackage{upgreek}
\usepackage{xy}

\usetikzlibrary{arrows}
\usetikzlibrary{matrix}
\usetikzlibrary{shapes}
\usetikzlibrary{snakes}
\usetikzlibrary{matrix}

\DeclareMathOperator{\Aff}{\textup{Aff}}
\DeclareMathOperator{\Alg}{\textup{Alg}}
\DeclareMathOperator{\Ann}{\textup{Ann}}
\DeclareMathOperator{\Arr}{\textup{Arr}}
\DeclareMathOperator{\Art}{\textup{Art}}
\DeclareMathOperator{\Ass}{\textup{Ass}}
\DeclareMathOperator{\Aut}{\textup{Aut}}
\DeclareMathOperator{\Autsh}{\underline{\textup{Aut}}}
\DeclareMathOperator{\Bi}{\textup{B}}
\DeclareMathOperator{\CAdd}{\textup{CAdd}}
\DeclareMathOperator{\CAlg}{\textup{CAlg}}
\DeclareMathOperator{\CMon}{\textup{CMon}}
\DeclareMathOperator{\CPMon}{\textup{CPMon}}
\DeclareMathOperator{\CRings}{\textup{CRings}}
\DeclareMathOperator{\CSMon}{\textup{CSMon}}
\DeclareMathOperator{\CaCl}{\textup{CaCl}}
\DeclareMathOperator{\Cart}{\textup{Cart}}
\DeclareMathOperator{\Cl}{\textup{Cl}}
\DeclareMathOperator{\Coh}{\textup{Coh}}
\DeclareMathOperator{\Coker}{\textup{Coker}}
\DeclareMathOperator{\Cov}{\textup{Cov}}
\DeclareMathOperator{\Der}{\textup{Der}}
\DeclareMathOperator{\Div}{\textup{Div}}
\DeclareMathOperator{\End}{\textup{End}}
\DeclareMathOperator{\Endsh}{\underline{\textup{End}}}
\DeclareMathOperator{\Ext}{\textup{Ext}}
\DeclareMathOperator{\Extsh}{\underline{\textup{Ext}}}
\DeclareMathOperator{\FCoh}{\textup{FCoh}}
\DeclareMathOperator{\FGrad}{\textup{FGrad}}
\DeclareMathOperator{\FMod}{\textup{FMod}}
\DeclareMathOperator{\FVect}{\textup{FVect}}
\DeclareMathOperator{\Fibr}{\textup{Fibr}}
\DeclareMathOperator{\Fix}{\textup{Fix}}
\DeclareMathOperator{\Funct}{\textup{Funct}}
\DeclareMathOperator{\GL}{\textup{GL}}
\DeclareMathOperator{\GRis}{\textup{GRis}}
\DeclareMathOperator{\GRiv}{\textup{GRiv}}
\DeclareMathOperator{\Gal}{\textup{Gal}}
\DeclareMathOperator{\Gl}{\textup{Gl}}
\DeclareMathOperator{\Grad}{\textup{Grad}}
\DeclareMathOperator{\Hilb}{\textup{Hilb}}
\DeclareMathOperator{\Hl}{\textup{H}}
\DeclareMathOperator{\Hom}{\textup{Hom}}
\DeclareMathOperator{\Homsh}{\underline{\textup{Hom}}}
\DeclareMathOperator{\ISym}{\textup{Sym}^*}
\DeclareMathOperator{\Imm}{\textup{Im}}
\DeclareMathOperator{\Irr}{\textup{Irr}}
\DeclareMathOperator{\Iso}{\textup{Iso}}
\DeclareMathOperator{\Isosh}{\underline{\textup{Iso}}}
\DeclareMathOperator{\Ker}{\textup{Ker}}
\DeclareMathOperator{\LAdd}{\textup{LAdd}}
\DeclareMathOperator{\LAlg}{\textup{LAlg}}
\DeclareMathOperator{\LMon}{\textup{LMon}}
\DeclareMathOperator{\LPMon}{\textup{LPMon}}
\DeclareMathOperator{\LRings}{\textup{LRings}}
\DeclareMathOperator{\LSMon}{\textup{LSMon}}
\DeclareMathOperator{\Left}{\textup{L}}
\DeclareMathOperator{\Loc}{\textup{Loc}}
\DeclareMathOperator{\M}{\textup{M}}
\DeclareMathOperator{\Map}{\textup{Map}}
\DeclareMathOperator{\Mod}{\textup{Mod}}
\DeclareMathOperator{\Ob}{\textup{Ob}}
\DeclareMathOperator{\Obj}{\textup{Obj}}
\DeclareMathOperator{\PDiv}{\textup{PDiv}}
\DeclareMathOperator{\PGL}{\textup{PGL}}
\DeclareMathOperator{\Pic}{\textup{Pic}}
\DeclareMathOperator{\Picsh}{\underline{\textup{Pic}}}
\DeclareMathOperator{\Pro}{\textup{Pro}}
\DeclareMathOperator{\Proj}{\textup{Proj}}
\DeclareMathOperator{\QAdd}{\textup{QAdd}}
\DeclareMathOperator{\QAlg}{\textup{QAlg}}
\DeclareMathOperator{\QCoh}{\textup{QCoh}}
\DeclareMathOperator{\QMon}{\textup{QMon}}
\DeclareMathOperator{\QPMon}{\textup{QPMon}}
\DeclareMathOperator{\QRings}{\textup{QRings}}
\DeclareMathOperator{\QSMon}{\textup{QSMon}}
\DeclareMathOperator{\R}{\textup{R}}
\DeclareMathOperator{\Riv}{\textup{Riv}}
\DeclareMathOperator{\SFibr}{\textup{SFibr}}
\DeclareMathOperator{\SchI}{\textup{SchI}}
\DeclareMathOperator{\Sh}{\textup{Sh}}
\DeclareMathOperator{\Soc}{\textup{Soc}}
\DeclareMathOperator{\Spec}{\textup{Spec}}
\DeclareMathOperator{\Specsh}{\underline{\textup{Spec}}}
\DeclareMathOperator{\Stab}{\textup{Stab}}
\DeclareMathOperator{\Supp}{\textup{Supp}}
\DeclareMathOperator{\Sym}{\textup{Sym}}
\DeclareMathOperator{\TMod}{\textup{TMod}}
\DeclareMathOperator{\Top}{\textup{Top}}
\DeclareMathOperator{\Tor}{\textup{Tor}}
\DeclareMathOperator{\Vect}{\textup{Vect}}
\DeclareMathOperator{\alt}{\textup{ht}}
\DeclareMathOperator{\car}{\textup{char}}
\DeclareMathOperator{\codim}{\textup{codim}}
\DeclareMathOperator{\degtr}{\textup{degtr}}
\DeclareMathOperator{\depth}{\textup{depth}}
\DeclareMathOperator{\divis}{\textup{div}}
\DeclareMathOperator{\et}{\textup{et}}
\DeclareMathOperator{\ffpSch}{\textup{ffpSch}}
\DeclareMathOperator{\h}{\textup{h}}
\DeclareMathOperator{\ilim}{\displaystyle{\lim_{\longrightarrow}}}
\DeclareMathOperator{\ind}{\textup{ind}}
\DeclareMathOperator{\indim}{\textup{inj dim}}
\DeclareMathOperator{\lf}{\textup{LF}}
\DeclareMathOperator{\ord}{\textup{ord}}
\DeclareMathOperator{\pd}{\textup{pd}}
\DeclareMathOperator{\plim}{\displaystyle{\lim_{\longleftarrow}}}
\DeclareMathOperator{\pr}{\textup{pr}}
\DeclareMathOperator{\pt}{\textup{pt}}
\DeclareMathOperator{\rk}{\textup{rk}}
\DeclareMathOperator{\tr}{\textup{tr}}
\DeclareMathOperator{\type}{\textup{r}}
\DeclareMathOperator*{\colim}{\textup{colim}}

\usepackage{enumitem}
\setenumerate[1]{label={\arabic*)}} 
\usepackage[a4paper]{geometry}

\makeatother

\usepackage{babel}
  \providecommand{\conjecturename}{Conjecture}
  \providecommand{\corollaryname}{Corollary}
  \providecommand{\definitionname}{Definition}
  \providecommand{\examplename}{Example}
  \providecommand{\lemmaname}{Lemma}
  \providecommand{\notationname}{Notation}
  \providecommand{\problemname}{Problem}
  \providecommand{\propositionname}{Proposition}
  \providecommand{\remarkname}{Remark}
  \providecommand{\theoremname}{Theorem}
\providecommand{\theoremname}{Theorem}

\begin{document}
\newgeometry{top=3cm,bottom=3cm} \thispagestyle{empty} \begin{titlepage} \begin{center} \includegraphics[scale=0.2]{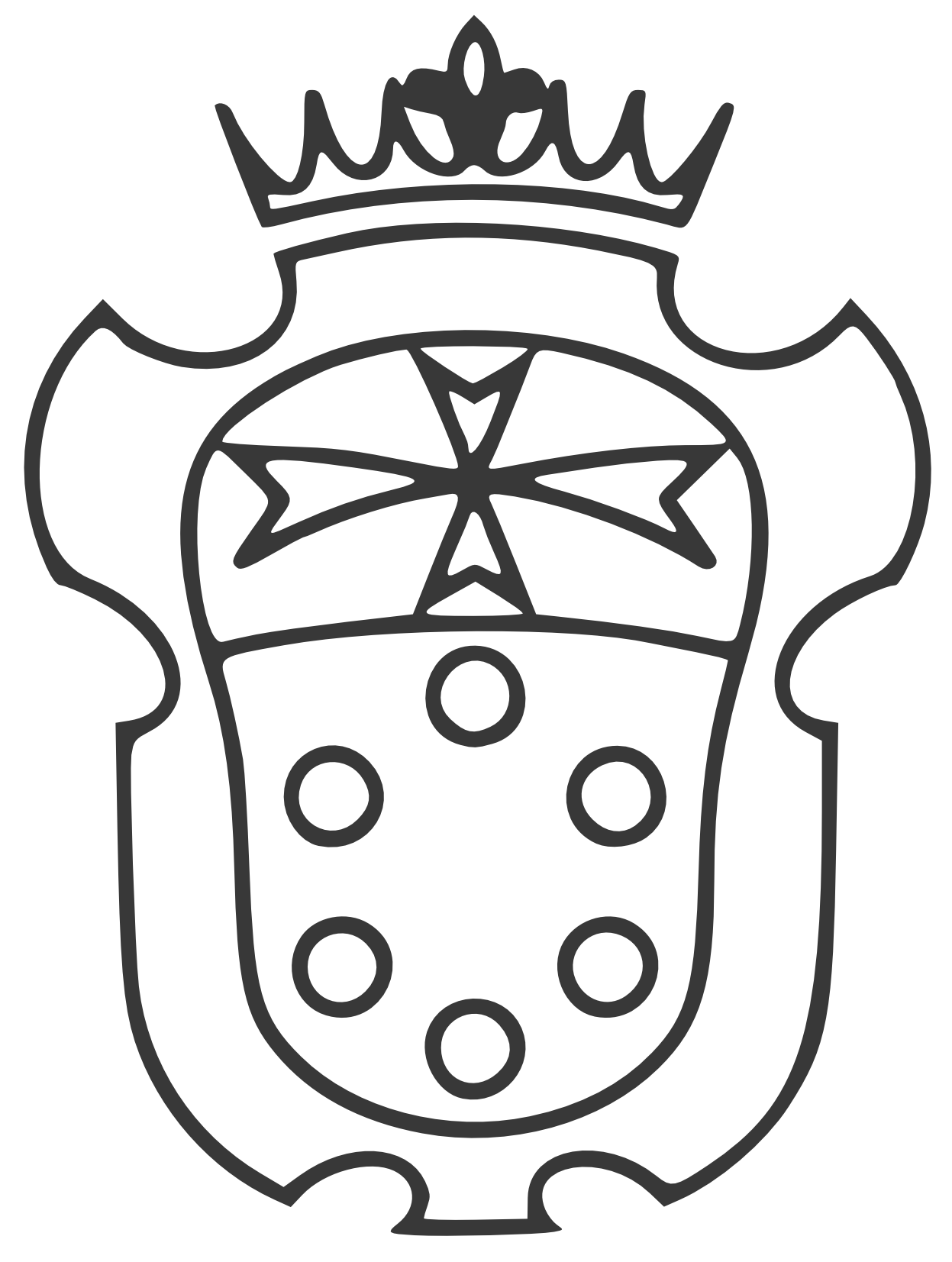} \par\end{center}
\begin{center} \textbf{\textsc{\Large Scuola Normale Superiore di Pisa}}\\ {\Large \rule[0.0066\textheight]{0.66\columnwidth}{1pt}}\textsc{\Large }\\ 
\textsc{\Large Classe di Scienze}\vspace{-1ex} \par\end{center} 
\begin{center} {\large Corso di Perfezionamento in Matematica} \par\end{center}{\large \par} {\large \medskip{} }{\large \par} \begin{center} 
\textsc{\Large Tesi di perfezionamento} \par\end{center}{\Large \par}
\begin{center} {\large \vfill{} } \par\end{center}{\large \par} \begin{center} 
\textsc{\Huge Stacks of} \par\end{center}{\Huge \par} \begin{center} \textsc{\Huge ramified Galois covers} \par\end{center}{\Huge \par} 
\begin{center} {\large \medskip{} } \par\end{center}{\large \par} \begin{center} {\Large Fabio Tonini} \par\end{center}{\Large \par} \begin{center} {\large \vfill{} } \par\end{center}{\large \par}
\begin{center} 
\textbf{\large Relatore:}{\large{} Prof.\ Angelo Vistoli}{\Large }\\ 
{\Large \rule[0.0066\textheight]{0.56\columnwidth}{1pt}}{\large }\\ {\large May 2, 2013 } \par\end{center}{\large \par}
\end{titlepage} \restoregeometry

\tableofcontents{}

\global\long\def\A{\mathbb{A}}

\global\long\def\Ab{(\textup{Ab})}

\global\long\def\C{\mathbb{C}}

\global\long\def\Cat{(\textup{cat})}

\global\long\def\Di#1{\textup{D}(#1)}

\global\long\def\E{\mathcal{E}}

\global\long\def\F{\mathbb{F}}

\global\long\def\GCov{G\textup{-Cov}}

\global\long\def\Gcat{(\textup{Galois cat})}

\global\long\def\Gfsets#1{#1\textup{-fsets}}

\global\long\def\Gm{\mathbb{G}_{m}}

\global\long\def\GrCov#1{\textup{D}(#1)\textup{-Cov}}

\global\long\def\Grp{(\textup{Grps})}

\global\long\def\Gsets#1{(#1\textup{-sets})}

\global\long\def\HCov{H\textup{-Cov}}

\global\long\def\MCov{\textup{D}(M)\textup{-Cov}}

\global\long\def\MHilb{M\textup{-Hilb}}

\global\long\def\N{\mathbb{N}}

\global\long\def\PGor{\textup{PGor}}

\global\long\def\PGrp{(\textup{Profinite Grp})}

\global\long\def\PP{\mathbb{P}}

\global\long\def\Pj{\mathbb{P}}

\global\long\def\Q{\mathbb{Q}}

\global\long\def\RCov#1{#1\textup{-Cov}}

\global\long\def\RR{\mathbb{R}}

\global\long\def\Sch{\textup{Sch}}

\global\long\def\WW{\textup{W}}

\global\long\def\Z{\mathbb{Z}}

\global\long\def\acts{\curvearrowright}

\global\long\def\alA{\mathscr{A}}

\global\long\def\alB{\mathscr{B}}

\global\long\def\arr{\longrightarrow}

\global\long\def\arrdi#1{\xlongrightarrow{#1}}

\global\long\def\catC{\mathscr{C}}

\global\long\def\catD{\mathscr{D}}

\global\long\def\catF{\mathscr{F}}

\global\long\def\catG{\mathscr{G}}

\global\long\def\comma{,\ }

\global\long\def\covU{\mathcal{U}}

\global\long\def\covV{\mathcal{V}}

\global\long\def\covW{\mathcal{W}}

\global\long\def\duale#1{{#1}^{\vee}}

\global\long\def\fasc#1{\widetilde{#1}}

\global\long\def\fsets{(\textup{f-sets})}

\global\long\def\iL{r\mathscr{L}}

\global\long\def\id{\textup{id}}

\global\long\def\la{\langle}

\global\long\def\odi#1{\mathcal{O}_{#1}}

\global\long\def\ra{\rangle}

\global\long\def\set{(\textup{Sets})}

\global\long\def\sets{(\textup{Sets})}

\global\long\def\shA{\mathcal{A}}

\global\long\def\shB{\mathcal{B}}

\global\long\def\shC{\mathcal{C}}

\global\long\def\shD{\mathcal{D}}

\global\long\def\shE{\mathcal{E}}

\global\long\def\shF{\mathcal{F}}

\global\long\def\shG{\mathcal{G}}

\global\long\def\shH{\mathcal{H}}

\global\long\def\shI{\mathcal{I}}

\global\long\def\shJ{\mathcal{J}}

\global\long\def\shK{\mathcal{K}}

\global\long\def\shL{\mathcal{L}}

\global\long\def\shM{\mathcal{M}}

\global\long\def\shN{\mathcal{N}}

\global\long\def\shO{\mathcal{O}}

\global\long\def\shP{\mathcal{P}}

\global\long\def\shQ{\mathcal{Q}}

\global\long\def\shR{\mathcal{R}}

\global\long\def\shS{\mathcal{S}}

\global\long\def\shT{\mathcal{T}}

\global\long\def\shU{\mathcal{U}}

\global\long\def\shV{\mathcal{V}}

\global\long\def\shW{\mathcal{W}}

\global\long\def\shX{\mathcal{X}}

\global\long\def\shY{\mathcal{Y}}

\global\long\def\shZ{\mathcal{Z}}

\global\long\def\st{\ | \ }

\global\long\def\stA{\mathcal{A}}

\global\long\def\stB{\mathcal{B}}

\global\long\def\stC{\mathcal{C}}

\global\long\def\stD{\mathcal{D}}

\global\long\def\stE{\mathcal{E}}

\global\long\def\stF{\mathcal{F}}

\global\long\def\stG{\mathcal{G}}

\global\long\def\stH{\mathcal{H}}

\global\long\def\stI{\mathcal{I}}

\global\long\def\stJ{\mathcal{J}}

\global\long\def\stK{\mathcal{K}}

\global\long\def\stL{\mathcal{L}}

\global\long\def\stM{\mathcal{M}}

\global\long\def\stN{\mathcal{N}}

\global\long\def\stO{\mathcal{O}}

\global\long\def\stP{\mathcal{P}}

\global\long\def\stQ{\mathcal{Q}}

\global\long\def\stR{\mathcal{R}}

\global\long\def\stS{\mathcal{S}}

\global\long\def\stT{\mathcal{T}}

\global\long\def\stU{\mathcal{U}}

\global\long\def\stV{\mathcal{V}}

\global\long\def\stW{\mathcal{W}}

\global\long\def\stX{\mathcal{X}}

\global\long\def\stY{\mathcal{Y}}

\global\long\def\stZ{\mathcal{Z}}

\global\long\def\then{\ \Longrightarrow\ }

\global\long\def\l{\textup{l}}

\chapter{Introduction. }

Let $G$ be a flat, finite group scheme finitely presented over a
base scheme $S$. In this thesis we study $G$-Galois covers of very
general schemes. A morphism $f\colon X\arr Y$ is called a \emph{cover}
if it is finite, flat and of finite presentation. We define a (ramified)
$G$\emph{-cover} as a cover $f\colon X\arr Y$ with an action of
$G$ on $X$ such that $f$ is $G$-invariant and $f_{*}\odi X$ is
fppf-locally isomorphic to the regular representation $\odi Y[G]$
as $\odi Y[G]$-comodule. This definition is a natural one: it generalizes
the notion of $G$-torsors and, under suitable hypothesis, coincides
with the usual definition of Galois cover when the group $G$ is constant
(see for example \cite{Pardini1991,Alexeev2011,Easton2008}). Moreover,
as explained below, in the abelian case $G$-covers are tightly related
to the theory of equivariant Hilbert schemes (see for example \cite{Nakamura2001,Stillman2002,Haiman2002,Alexeev2003}).
We call $\GCov$ the stack of $G$-covers and the aim of this thesis
will be to describe its structure. 

We denote by $\mu_{n}$ the diagonalizable group over $\Z$ with character
group $\Z/n\Z$. In many concrete problems, one is interested in a
more direct and concrete description of a $G$-cover $f\colon X\arr Y$.
This is very simple and well known when $G=\mu_{2}$: such a cover
$f$ is given by an invertible sheaf $\shL$ on $Y$ with a section
of $\shL^{\otimes2}$. Similarly, when $G=\mu_{3}$, a $\mu_{3}$-cover
$f$ is given by a pair $(\shL_{1},\shL_{2})$ of invertible sheaves
on $Y$ with maps $\shL_{1}^{\otimes2}\arr\shL_{2}$ and $\shL_{2}^{\otimes2}\arr\shL_{1}$
(see \cite[§ 6]{Arsie2004}).

In general, however, there is no comparable description of $G$-covers.
Very little is known when $G$ is not abelian, except for $G=S_{3}$
(see \cite{Easton2008}) and the case of Galois covers with groups
$G=D_{n},A_{4},S_{4}$ having regular total space (see \cite{Tokunaga1994,Tokunaga2002,Reimpell1999}).
In the non-Galois case, there exist a general description of covers
of degree $3$ and $4$ (see \cite{Miranda1985,Pardini1989,Hahn1999})
and of Gorenstein covers of degree $d\geq3$ (see \cite{Casnati1996,Casnati1996a}).

Even in the abelian case, the situation becomes complicated very quickly
when the order of $G$ grows. The paper that inspires our work is
\cite{Pardini1991}; here the author describes $G$-covers $X\arr Y$
when $G$ is an abelian group, $Y$ is a smooth variety over an algebraically
closed field of characteristic prime to $|G|$ and $X$ is normal,
in terms of certain invertible sheaves on $Y$, generalizing the description
given above for $G=\mu_{2}$ and $G=\mu_{3}$.

We now outline the content of this thesis remarking the results obtained
and we will follow the division in chapters.

\section{Preliminaries on Galois covers.}

Chapter $2$ is dedicated to explaining the basic properties of $G$-covers.
Our first result is:
\begin{thm*}
{[}\ref{pro:BG open in GCov},\ref{cor:GCov is algebraic}{]} The
stack $\GCov$ is algebraic and finitely presented over $S$. Moreover
$\Bi G$, the stack of $G$-torsors, is an open substack of $\GCov$.
\end{thm*}
As particular cases, we will study $G$-covers for the groups $G=\mu_{2},\mu_{3},\alpha_{p}$,
where $p$ is a prime and $\alpha_{p}$ is the kernel of the Frobenius
map $\mathbb{G}_{a}\arr\mathbb{G}_{a}$ over $\F_{p}$. Example \ref{ex: Romagny example},
suggested by Prof. Romagny, shows that covers that are generically
$\alpha_{p}$-torsors are not $\alpha_{p}$-cover in general. Moreover,
we will prove an unexpected result, that is that $\Bi\alpha_{p}=\RCov{\alpha_{p}}$
or, in other words, that every $\alpha_{p}$-cover is an $\alpha_{p}$-torsor
(see \ref{prop:Balphap alphapCov}). As we will see, the situation
is completely different for Galois covers of linearly reductive groups
(introduced below), even in the diagonalizable case. For instance
we will show that $\GCov$ is almost never irreducible. This motivates
the following definition. The \emph{main irreducible component of
$\GCov$, }denoted by $\stZ_{G}$, is the schematic closure of $\Bi G$
in $\GCov$.\emph{ }Notice that, if $S$ is irreducible, then $\Bi G$
is irreducible as well and therefore $\stZ_{G}$ is an irreducible
component of $\GCov$. 

In the last part of the chapter we study examples of isomorphisms
$\GCov\simeq\HCov$ using the notion of bitorsor: given finite, flat
and finitely presented group schemes $G$ and $H$ over a scheme $S$
a $(G,H)$-bitorsor is an $S$-scheme which is simultaneously a left
$G$-torsor and a right $H$-torsor and such that the actions are
compatible. Notice that the existence of a $(G,H)$-bitorsor implies
that $G$ and $H$ are locally isomorphic, but the converse is false.
The $(G,H)$-bitorsors correspond to isomorphisms $\Bi G\simeq\Bi H$
(see \cite[Chapter III, Remarque 1.6.7]{Giraud1971}). We will give
a proof of this fact and we will also show that they induce isomorphisms
$\GCov\simeq\HCov$ (see \ref{prop:bitorsors and isomorphisms BG - BH}
and \ref{thm:bitorsors and GCov}). Moreover if the $G$-cover $X\arr Y$
is sent to the $H$-cover $X'\arr Y$ through one of these isomorphisms,
then $X$ and $X'$ are fppf locally isomorphic over $S$, étale locally
if $G$ or $H$ is étale and therefore they share many geometric properties,
like reduceness, smoothness, geometrical connectedness and irreducibility
and, in the étale case, regularity and regularity in codimension $1$.
We will use this construction when we will study $(\mu_{3}\rtimes\Z/2\Z)$-covers
and $S_{3}$-covers in the last chapter.

\section{Galois covers under diagonalizable group schemes.}

Chapter $3$ of this thesis essentially coincides with the article
\cite{Tonini2011}. In it, we concentrate on the case when $G$ is
a finite diagonalizable group scheme over $\Z$; thus, $G$ is isomorphic
to a finite direct product of group schemes of the form $\mu_{d}$
for $d\geq1$. We consider the dual finite abelian group $M=\Hom(G,\Gm)$
so that, by standard duality results (see \cite{Grothendieck1970}),
$G$ is the fppf sheaf of homomorphisms $M\arr\Gm$ and a decomposition
of $M$ into a product of cyclic groups yields the decomposition of
$G$ into a product of $\mu_{d}$'s. Although Chapter $4$ study Galois
covers for general groups, we have decided to consider the diagonalizable
case first, because in this case $G$-covers have a very explicit
description in terms of sequences of invertible sheaves. Indeed a
$G$-cover over $Y$ is of the form $X=\Spec\alA$ where $\alA$ is
a quasi-coherent sheaf of algebras over $Y$ with a decomposition
\begin{equation}
\alA=\bigoplus_{m\in M}\shL_{m}\text{ s.t. }\shL_{0}=\odi Y\comma\shL_{m}\text{ invertible}\text{ and }\shL_{m}\shL_{n}\subseteq\shL_{m+n}\text{ for all }m,n\in M\label{eq:decomposition for diagonalizable groups introduction}
\end{equation}
So a $G$-cover corresponds to a sequence of invertible sheaves $(\shL_{m})_{m\in M}$
with maps $\psi_{m,n}\colon\shL_{m}\otimes\shL_{n}\arr\shL_{m+n}$
satisfying certain rules and our principal aim will be to simplify
the data necessary to describe such covers. For instance $G$-torsors
correspond to sequences where all the maps $\psi_{m,n}$ are isomorphisms.
Therefore, if $G=\mu_{l}$, a $G$-torsor is simply given by an invertible
sheaf $\shL=\shL_{1}$ and an isomorphism $\shL^{\otimes l}\simeq\odi{}$. 

When $G=\mu_{2}$ or $G=\mu_{3}$ the description given above shows
that the stack $\GCov$ is smooth, irreducible and very easy to describe.
In the general case its structure turns out to be extremely intricate.
For instance, as we will see, $\GCov$ is almost never irreducible.
The existence of the 'special' irreducible component $\stZ_{G}$ parallels
what happens in the theory of $M$-equivariant Hilbert schemes (see
\cite[Remark 5.1]{Haiman2002}). It turns out that this theory and
the theory of $G$-covers are deeply connected: given an action of
$G$ on $\A^{r}$, induced by elements $\underline{m}=m_{1},\dots,m_{r}\in M$,
the equivariant Hilbert scheme $M\text{-Hilb }\A^{r}$, that we will
denote by $\MHilb^{\underline{m}}$ to underline the dependence on
the sequence $\underline{m}$, can be viewed as the functor whose
objects are $G$-covers with an equivariant closed immersion in $\A^{r}$.
The forgetful map $\vartheta\colon\MHilb^{\underline{m}}\arr\GCov$
is smooth with geometrically irreducible fibers onto an open substack
$U_{\underline{m}}$ of $\GCov$. Moreover it is surjective, that
is an atlas, provided that $\underline{m}$ contains all the elements
in $M-\{0\}$ (\ref{pro:MHlb --> MCov has irreducible fibers}). This
means that $U_{\underline{m}}$ and $\MHilb^{\underline{m}}$ share
several geometric properties, like connectedness, irreducibility,
smoothness or reduceness. Moreover $\vartheta^{-1}(\stZ_{G})$ coincides
with the main irreducible component of $\MHilb^{\underline{m}}$,
first studied by Nakamura in \cite{Nakamura2001}.

We will prove the following results on the structure of $\GCov$.
\begin{thm*}
{[}\ref{thm:Mcov geom connected},\ref{cor:Mcov reducible},\ref{pro:smooth DMCov},\ref{rem:MCov for M=00003DZfour}{]}
When $G$ is a finite diagonalizable group scheme over $\Z$, the
stack $\GCov$ is
\begin{itemize}
\item flat and of finite type with geometrically connected fibers,
\item smooth if and only if $G\simeq0,\mu_{2},\mu_{3},\mu_{2}\times\mu_{2}$,
\item normal if $G\simeq\mu_{4}$,
\item reducible if $|G|\geq8$ and $G\not\simeq(\mu_{2})^{3}$.
\end{itemize}
The above properties continue to hold if we replace $\GCov$ by $\MHilb^{\underline{m}}$,
provided that each nonzero element of $M$ belongs to the sequence
$\underline{m}$.
\end{thm*}
We do not know whether $\GCov$ is integral for $G\simeq\mu_{5},\mu_{6},\mu_{7},(\mu_{2})^{3}$.
So $\GCov$ is usually reducible, its structure is extremely complicated
and we have little hope of getting to a real understanding of the
components not containing $\Bi G$. Therefore we will focus on the
main irreducible component $\stZ_{G}$ of $\GCov$. The main idea
behind the study of $G$-covers when $G$ is diagonalizable, inspired
by the results in \cite{Pardini1991}, is to try to decompose the
multiplications $\psi_{m,n}\in\shL_{m+n}\otimes\shL_{m}^{-1}\otimes\shL_{n}^{-1}$
as a tensor product of sections of other invertible sheaves. Following
this idea we will construct parametrization maps $\pi_{\underline{\E}}\colon\stF_{\underline{\E}}\arr\stZ_{G}\subseteq\GCov$,
where $\stF_{\underline{\E}}$ are 'nice' stacks, for example smooth
and irreducible, whose objects are those decompositions. 

This construction can be better understood locally, where a $G$-cover
over $Y=\Spec R$ is just $X=\Spec A$, where $A$ is an $R$-algebra
with an $R$-basis $\{v_{m}\}_{m\in M}$, $v_{0}=1$ $(\shL_{m}=\odi Yv_{m})$,
so that the multiplications are elements $\psi_{m,n}\in R$ such that
$v_{m}v_{n}=\psi_{m,n}v_{m+n}$.

Consider $a\in R$, a collection of natural numbers $\E=(\E_{m,n})_{m,n\in N}$
and set $\psi_{m,n}=a^{\E_{m,n}}$. The condition that the product
structure on $A=\oplus_{m}Rv_{m}$ defined by the $\psi_{m,n}$ yields
an associative, commutative $R$-algebra, i.e. makes $\Spec A$ into
a $G$-cover over $\Spec R$, translates into some additive relations
on the numbers $\E_{m,n}$. Call $\duale K_{+}$ the set of collections
$\E$ satisfying those relations. More generally given $\underline{\E}=\E^{1},\dots,\E^{r}\in\duale K_{+}$
we can define a parametrization
\[
R^{r}\ni(a_{1},\dots,a_{r})\arr\psi_{m,n}=a_{1}^{\E_{m,n}^{1}}\cdots a_{r}^{\E_{m,n}^{r}}
\]
This is essentially the local behavior of the map $\pi_{\underline{\E}}\colon\stF_{\underline{\E}}\arr\GCov$.
In the global case the elements $a_{i}$ will be sections of invertible
sheaves.

From this point of view the natural questions are: given a $G$-cover
over a scheme $Y$ when does there exist a lift to an object of $\stF_{\underline{\E}}(Y)$?
Is this lift unique? How can we choose the sequence $\underline{\E}$?

The key point is to give an interpretation to $\duale K_{+}$ (that
also explains this notation). Consider $\Z^{M}$ with canonical basis
$(e_{m})_{m\in M}$ and define $v_{m,n}=e_{m}+e_{n}-e_{m+n}\in\Z^{M}/\langle e_{0}\rangle$.
If $p\colon\Z^{M}/\langle e_{0}\rangle\arr M$ is the map $p(e_{m})=m$,
the $v_{m,n}$ generate $\Ker p$. Now call $K_{+}$ the submonoid
of $\Z^{M}/\langle e_{0}\rangle$ generated by the $v_{m,n}$, $K=\Ker p$
its associated group and also consider the torus $\shT=\Homsh(\Z^{M}/\langle e_{0}\rangle,\Gm)$,
which acts on $\Spec\Z[K_{+}]$. By construction we have that a collection
of natural numbers $(\E_{m,n})_{m,n\in M}$ belongs to $\duale K_{+}$
if and only if the association $v_{m,n}\arr\E_{m,n}$ defines an additive
map $K_{+}\arr\N$. Therefore, as the symbol suggests, we can identify
$\duale K_{+}$ with $\Hom(K_{+},\N)$, the dual monoid of $K_{+}$.
Its elements will be called \emph{rays}. More generally a monoid map
$\psi\colon K_{+}\arr(R,\cdot)$, where $R$ is a ring, yields a multiplication
$\psi_{m,n}=\psi(v_{m,n})$ on $\bigoplus_{m\in M}Rv_{m}$ and therefore
we obtain a map $\Spec\Z[K_{+}]\arr\stZ_{G}$. We will prove that
(see \ref{cor:MCov as global quotient}):
\begin{thm*}
We have $\stZ_{G}\simeq[\Spec\Z[K_{+}]/\shT]$ and $\Bi G\simeq[\Spec\Z[K]/\shT]$.
\end{thm*}
Notice that the whole $\GCov$ has a similar description as quotient,
but we have to consider non cancellative monoids. We introduce the
following notation: given $\alpha\in\N$, we set $0^{\alpha}=1$ if
$\alpha=0$ and $0^{\alpha}=0$ otherwise. Given $\underline{\E}=\E^{1},\dots,\E^{r}\in\duale K_{+}$
we have defined a map $\pi_{\underline{\E}}\colon\stF_{\underline{\E}}\arr\stZ_{G}$.
Notice that if $\underline{\gamma}$ is a subsequence of $\underline{\E}$
then $\stF_{\underline{\gamma}}$ is an open substack of $\stF_{\underline{\E}}$
and $(\pi_{\underline{\E}})_{|\stF_{\underline{\gamma}}}=\pi_{\underline{\gamma}}$.
The lifting problem for the maps $\pi_{\underline{\E}}$ clearly depends
on the choice of the sequence $\underline{\E}$. Considering larger
$\underline{\E}$ allows us to parametrize more covers, but also makes
uniqueness of the lifting unlikely. In this direction we have proved
that:
\begin{thm*}
{[}\ref{pro:characterization of points of Zphi}{]} Let $k$ be an
algebraically closed field and suppose we have a collection $\underline{\E}$
whose rays generate the rational cone $\duale K_{+}\Q$. Then the
map of groupoids $\stF_{\underline{\E}}(k)\arr\stZ_{G}(k)$ is essentially
surjective. In other words a $G$-cover of $\Spec k$ in the main
component $\stZ_{G}$ has a multiplication of the form $\psi_{m,n}=0^{\E_{m,n}}$
for some $\E\in\duale K_{+}$.
\end{thm*}
On the other hand small sequences $\underline{\E}$ can guarantee
uniqueness but not existence. The solution we have found is to consider
a particular class of rays, called extremal, that have minimal non
empty support. Set $\underline{\eta}$ for the sequence of all extremal
rays (that is finite). Notice that extremal rays generate $\duale K_{+}\Q$.
We prove that:
\begin{thm*}
{[}\ref{pro:smooth locus of Z[T +]}, \ref{thm:fundamental theorem for the smooth locus of ZM}{]}
The smooth locus $\stZ_{G}^{\textup{sm}}$ of $\stZ_{G}$ is of the
form $[X_{G}/\shT]$ where $X_{G}$ is a smooth toric variety over
$\Z$ (whose maximal torus is $\Spec\Z[K]$). Moreover $\pi_{\underline{\eta}}\colon\stF_{\underline{\eta}}\arr\stZ_{G}$
induces an isomorphism of stacks
\[
\pi_{\underline{\eta}}^{-1}(\stZ_{G}^{\textup{sm}})\arrdi{\simeq}\stZ_{G}^{\textup{sm}}
\]

\end{thm*}
Among the extremal rays there are special rays, called smooth, that
can be defined as extremal rays $\E$ whose associated multiplication
$\psi_{m,n}=0^{\E_{m,n}}$ yields a cover in $\stZ_{G}^{\textup{sm}}$.
Set $\underline{\xi}$ for the sequence of smooth extremal rays. It
turns out that the theorem above holds if we replace $\underline{\eta}$
with $\underline{\xi}$.

If, given a scheme $X$, we denote by $\Picsh X$ the category whose
objects are invertible sheaves on $X$ and whose arrows are arbitrary
maps of sheaves, we also have: 
\begin{thm*}
{[}\ref{cor:lift when Picsh is the same}{]} Consider a $2$-commutative
diagram    \[   \begin{tikzpicture}[xscale=2.0,yscale=-1.0]     \node (A0_0) at (0, 0) {$X$};     \node (A0_1) at (1, 0) {$\stF_{\underline \E}$};     \node (A1_0) at (0, 1) {$Y$};     \node (A1_1) at (1, 1) {$\GCov$};     \path (A0_0) edge [->]node [auto] {$\scriptstyle{}$} (A0_1);     \path (A1_0) edge [->,dashed]node [auto] {$\scriptstyle{}$} (A0_1);     \path (A0_0) edge [->]node [auto,swap] {$\scriptstyle{f}$} (A1_0);     \path (A0_1) edge [->]node [auto] {$\scriptstyle{\pi_{\underline \E}}$} (A1_1);     \path (A1_0) edge [->]node [auto] {$\scriptstyle{}$} (A1_1);   \end{tikzpicture}   \] where
$X,Y$ are schemes and $\underline{\E}$ is a sequence of elements
of $\duale K_{+}$. If $\Picsh Y\arrdi{f^{*}}\Picsh X$ is fully faithful
(resp. an equivalence) the dashed lifting is unique (resp. exists
and is unique).
\end{thm*}
In particular the theorems above allow us to conclude that:
\begin{thm*}
{[}\ref{thm:fundamental theorem for the smooth locus of ZM}, \ref{thm:fundamental theorem for locally factorial schemes}{]}
Let $Y$ be a locally noetherian and locally factorial scheme. A cover
$\chi\in\GCov(Y)$ such that $\chi_{|k(p)}\in\stZ_{G}^{\textup{sm}}(k(p))$
for any $p\in Y$ with $\codim_{p}Y\leq1$ lifts uniquely to $\stF_{\underline{\xi}}(Y)$.
\end{thm*}
An interesting problem is to describe all (smooth) extremal rays.
This seems very difficult and it is related to the problem of finding
$\Q$-linearly independent sequences among the $v_{m,n}\in K_{+}$.
A natural way of obtaining extremal rays is trying to describe $G$-covers
with special properties. The first examples of them arise looking
at covers with normal total space. Indeed in \cite{Pardini1991} the
author is able to describe the multiplications yielding regular $G$-covers
of a discrete valuation ring. This description, using the language
introduced above, yields a sequence $\underline{\delta}=(\E^{\phi})_{\phi\in\Phi_{M}}$
of smooth extremal rays, where $\Phi_{M}$ is the set of surjective
maps $M\arr\Z/d\Z$ with $d>1$. We will define a stratification of
$\GCov$ by open substacks $\Bi G=U_{0}\subseteq U_{1}\subseteq\cdots\subseteq U_{|G|-1}=\GCov$
and we will prove that there exists an explicitly given sequence $\underline{\E}$
of smooth  extremal rays (defined in \ref{pro:classification sm int ray htwo})
containing $\underline{\delta}$ such that:
\begin{thm*}
{[}\ref{thm:fundamental theorem for hleqone}, \ref{thm:fundamental thm for hleqtwo}{]}
We have $U_{2}\subseteq\stZ_{G}^{\textup{sm}}$ and $\pi_{\underline{\E}}\colon\stF_{\underline{\E}}\arr\stZ_{G}$
induces isomorphisms of stacks
\[
\pi_{\underline{\E}}^{-1}(U_{2})\arrdi{\simeq}U_{2}\comma\pi_{\underline{\delta}}^{-1}(U_{1})=\pi_{\underline{\E}}^{-1}(U_{1})\arrdi{\simeq}U_{1}
\]

\end{thm*}
Theorem above implies that $\MHilb\A^{2}$ is smooth and irreducible
(\ref{cor:MHilbmn smooth and irreducible}). In this way we get an
alternative proof of the result in \cite{Maclagan2002} (later generalized
in \cite{Maclagan2010}) in the particular case of equivariant Hilbert
schemes.
\begin{thm*}
{[}\ref{thm:for hleqone}, \ref{thm:fundamental thm locally factoria hleqtwo}{]}
Let $Y$ be a locally noetherian and locally factorial scheme and
$\chi\in\GCov(Y)$. If $\chi_{|k(p)}\in U_{1}$ (resp. $\chi_{|k(p)}\in U_{2}$)
for all $p\in Y$ with $\codim_{p}Y\leq1$, then $\chi$ lifts uniquely
to $\stF_{\underline{\delta}}(Y)$ (resp. $\stF_{\underline{\E}}(Y)$).
\end{thm*}
Notice that $\underline{\E}=\underline{\delta}$ if and only if $G\simeq(\mu_{2})^{l}$
or $G\simeq(\mu_{3})^{l}$ (\ref{pro:when sigmaM is empty}). Finally
we prove:
\begin{thm*}
{[}\ref{thm:regular in codimension 1 covers}, \ref{thm:NC in codimension one}{]}
Let $Y$ be a locally noetherian and locally factorial integral scheme
with $\dim Y\geq1$ and such that $|M|\in\odi Y^{*}$. Let also $f\colon X\arr Y$
be a $G$-cover. If $X$ is regular in codimension $1$ (resp. normal
crossing in codimension $1$ (see \ref{def:normal crossing codimesion one}))
then $f$ comes from a unique object of $\stF_{\underline{\delta}}(Y)$
(resp. $\stF_{\underline{\gamma}}(Y)$, where $\underline{\delta}\subseteq\underline{\gamma}\subseteq\underline{\E}$
is an explicitly given sequence).
\end{thm*}
Note that one can replace {}``regular in codimension $1$'' with
{}``normal'' in the above theorem because $G$-covers have Cohen-Macaulay
fibers. The part concerning covers that are regular in codimension
$1$ is essentially a rewriting of Theorem $2.1$ and Corollary $3.1$
of \cite{Pardini1991} extended to locally noetherian and locally
factorial schemes, while the last part generalizes Theorem $1.9$
of \cite{Alexeev2011}.

\section{Equivariant affine maps and monoidality.}

In Chapter $4$ we focus on the problem of describing Galois covers
for a general finite, flat and of finite presentation group scheme
$G$ over a given base ring $R$. This problem can be stated as follows.
\begin{problem}
\label{prob: data for general G covers introduction} Given an $R$-scheme
$T$, describe $G$-covers over $T$ in terms of locally free sheaves
over $T$ (without an action of $G$), maps among them and the representation
theory of $G$ over the base ring $R$. 
\end{problem}
Denote by $\Loc T$ ($\Loc^{G}T$) the category of locally free sheaves
of finite rank over $T$ (with an action of $G$). Similarly define
$\QCoh T$ and $\QCoh^{G}T$ replacing locally free sheaves with arbitrary
quasi-coherent sheaves. Given an $R$-linear functor $\Omega\colon\Loc^{G}R\arr\QCoh T$
a \emph{monoidal} structure on it is given by a natural transformation
$\iota_{V,W}\colon\Omega_{V}\otimes\Omega_{W}\arr\Omega_{V\otimes W}$
and an identity $1\in\Omega_{R}$ satisfying certain natural conditions.
This structure is called \emph{strong }if $\iota$ is an isomorphism
and $\Omega_{R}=\odi T1$. 

A common result of Tannaka's theory is that the category of $G$-torsors
over $T$ is equivalent to the category of strong monoidal, symmetric,
$R$-linear and exact functors $\Loc^{G}R\arr\Loc T$. This is part
of the so called {}``reconstruction problem'': reconstruct a group
$G$ from its category of representations. See \cite{Deligne1982,Rivano1972}
for the classical case when $R$ is a field and \cite{Lurie2004}
for the general one. This gives an answer to problem \ref{prob: data for general G covers introduction}
in the case of $G$-torsors. In this chapter we provide a similar
answer for $G$-covers and, more generally, for equivariant affine
maps. The idea is to extend the above correspondence to the case of
(non exact) non strong monoidal functors. In order to consider general
equivariant affine maps and not only $G$-covers, it is also necessary
to consider monoidal functors with values in the whole category of
quasi-coherent sheaves.

This functorial point of view for $G$-covers arises naturally when
trying to answer problem \ref{prob: data for general G covers introduction}.
As seen in the case of diagonalizable groups, the problem of describing
$G$-covers is equivalent to the problem of understanding the possible
algebra structures on the regular representation $R[G]$. This task
should be easier when $R[G]$ is a sum of smaller parts, as it happens
when $G$ is diagonalizable or a constant group over the complex numbers.
Therefore the first problem to solve is to determine a class of group
schemes for which this simplification is possible. Over a field $k$,
such question has already an answer: the regular representation $k[G]$
decomposes into a product of irreducible representations if and only
if the group $G$ is linearly reductive. This property is usually
taken as definition of a linearly reductive group over a field. There
is an alternative definition, which has the advantage of working over
any base scheme: a group scheme $G$ over $R$ is called \emph{linearly
reductive} if the functor of invariants $\QCoh^{G}R\arr\QCoh R$,
$\shF\longmapsto\shF^{G}$ is exact. For an introduction to this subject
see \cite{Abramovich2007}. What we are looking for is a collection
$I$ of objects in $\Loc^{G}R$ for which the $G$-equivariant maps
\[
\eta_{I,\shF}\colon\bigoplus_{V\in I}\Homsh^{G}(V,\shF)\otimes V\arr\shF
\]
are isomorphisms for all $\shF\in\QCoh T$ and all $R$-schemes $T$.
Clearly this implies that $G$ is linearly reductive and, over an
algebraically closed field $k$, that $I$ is the set of the irreducible
representations. Assume that $G$ is a linearly reductive group. The
result is:
\begin{prop*}
{[}\ref{prop:generating irreducible representations}{]} The maps
$\eta_{I,\shF}$ are isomorphisms for all $\shF\in\QCoh T$ and all
$R$-schemes $T$ if and only if, for all algebraically closed fields
$k$ and geometric points $\Spec k\arr\Spec R$, the representations
$V\otimes k$ are irreducible for all $V\in I$ and the restriction
$-\otimes k$ yields a one to one correspondence between $I$ and
the set of irreducible representations of $G\times k$ up to isomorphisms.
When $\Spec R$ is connected, the previous condition can be checked
at a fixed geometric point.
\end{prop*}
We will say that a group $G$ admitting a set $I$ as above has a
\emph{good }representation theory and that the pair $(G,I$) is a
\emph{good linearly reductive group }(abbreviated with \emph{glrg}).
For such a pair we will also write $I=I_{G}$. Notice that any $V\in I_{G}$
is not only irreducible, but also geometrically irreducible, that
is it is irreducible after base changing to all the geometric points,
and that $\Endsh^{G}(V)=\odi S$, while if $W\in I_{G}$ and $W\neq V$,
then $\Homsh^{G}(V,W)=0$. In particular over a field $k$, $G$ has
a good representation theory if and only if all irreducible representations
$V$ have trivial endomorphism rings, that is $\End^{G}(V)\simeq k$,
and in this case $I_{G}$ is the set of irreducible representations.
When the base is connected, $I_{G}$ is uniquely determined up to
tensorization by invertible sheaves (with trivial action). Other examples
of good linearly reductive groups are the diagonalizable groups over
$R=\Z$: if $M=\Hom(G,\Gm)$ and we denote by $\Z_{m}$ the representation
induced by $m\in M$, it is enough to consider the sequence $I_{G}=(\Z_{m})_{m\in M}$.
Notice that there exist linearly reductive that are not good, for
instance $\Z/p\Z$ over a field not containing a primitive $p$-root
of unity. On the other hand we prove that (étale) linearly reductive
groups have a good representation theory locally in the (étale) fppf
topology (see \ref{prop:G has locally a good representation theory}).
Moreover, any constant group $G$ has a good representation theory
over a strictly Henselian ring $R$, provided that the characteristic
of the residue field of $R$ does not divide the order of $G$ (see
\ref{thm:etale linearly reductive over sctrictly are glrg}).

When $G$ is a good linearly reductive group and the base scheme $\Spec R$
is connected, a $G$-comodule $\shF$ over an $R$-scheme $T$ which
is fppf locally the regular representation is of the form
\begin{equation}
\shF\simeq\bigoplus_{V\in I_{G}}\shF_{V}\otimes\duale V\text{ where }\shF_{V}\text{ is locally free of rank }\rk V\label{eq:decomposition fppf locally regular representation introduction}
\end{equation}
Thus $\shF$ is determined by a sequence of locally free sheaves with
prescribed ranks, namely $(\shF_{V})_{V\in I_{G}}$. Now the problem
is to understand what additional data are needed and what conditions
such data have to satisfy in oder to induce a structure of algebra
over $\shF$. A non associative ring structure on $\shF$ is given
by a collection of maps between sheaves obtained starting from the
$\shF_{V},V$ for $V\in I_{G}$, whose form depends on how the tensor
products $V\otimes W$ for $V,W\in I_{G}$ decompose into representations
in $I_{G}$. Moreover it is not difficult to convince oneself that
the conditions those maps have to satisfy in order to have a commutative
and associative algebra strongly depend on the two ways one can decompose
$(V\otimes W)\otimes Z\simeq V\otimes(W\otimes Z)$ for $V,W,Z\in I_{G}$
into representations in $I_{G}$. I have to admit that I have never
been brave enough to write down those last conditions, although this
should be an elementary task: it seems pretty clear that there is
no hope to simplify those conditions for a general group, obtaining
a really meaningful set of data. The diagonalizable case is much more
simple than the general one because tensor products of representations
are very easy.

The approach I propose to work around this situation is to associate
with a sheaf $\shF$ not only the sequence $(\shF_{V})_{V\in I_{G}}$,
but a whole functor $\Loc^{G}R\arr\Loc T$. With a $G$-comodule $\shF$
we associate the functor $\Omega^{\shF}\colon\Loc^{G}R\arr\Loc T$
given by
\[
\Omega_{V}^{\shF}=(\shF\otimes V)^{G}
\]
Notice that, with notation from (\ref{eq:decomposition fppf locally regular representation introduction}),
$\shF_{V}=\Omega_{V}^{\shF}$ for all $V\in I_{G}$. Although we do
not have a finite set of data, this approach has, at least, two advantages.
The first is that, as we will see, a structure of algebra on $\shF$
translates into natural properties on $\Omega^{\shF}$. The second
is that this point of view, without additional technicalities, allows
us to consider and describe any $G$-equivariant affine map, that
is any affine map $f\colon X\arr T$ with an action of $G$ on $X$
for which $f$ is invariant, and that the theory can be developed
for any finite, flat and finitely presented group scheme.

So assume that $G$ is a finite, flat and finitely presented group
scheme over a ring $R$. Given an $R$-scheme $T$ define $\QAdd^{G}T$
as the category of $R$-linear functors $\Loc^{G}R\arr\QCoh T$. Denote
also by $\QAdd_{R}^{G}$ (resp. $\QCoh_{R}^{G}$) the stack (not in
groupoids) whose fiber over an $R$-scheme $T$ is $\QAdd^{G}T$ (resp.
$\QCoh^{G}T$). Given $\Omega\in\QAdd^{G}T$, we will show that $\Omega_{R[G]}\in\QCoh T$
has a natural structure of $G$-comodule and the first result we will
prove is:
\begin{thm*}
{[}\ref{thm:additive functors are equivariant sheaves}{]} Given an
$R$-scheme $T$, the functors   \[   \begin{tikzpicture}[xscale=3.7,yscale=-0.6]     
\node (A0_0) at (0, 0) {$\shF_\Omega=\Omega_{R[G]}$};     
\node (A0_1) at (1, 0) {$\Omega$};     
\node (A1_0) at (0, 1) {$\QCoh^G T$};     
\node (A1_1) at (1, 1) {$\QAdd^G T$};     
\node (A2_0) at (0, 2) {$\shF$};     
\node (A2_1) at (1, 2) {$\Omega^{\shF}=(-\otimes \shF)^G$};     
\path (A0_0) edge [<-|,gray]node [auto] {$\scriptstyle{}$} (A0_1);     \path (A1_0) edge [->]node [auto] {$\scriptstyle{}$} (A1_1);     \path (A2_1) edge [<-|,gray]node [auto] {$\scriptstyle{}$} (A2_0);   \end{tikzpicture}   \] yield an equivalence between $\QCoh^{G}T$ and the full subcategory
of $\QAdd^{G}T$ of left exact functors.
\end{thm*}
The group $G$ is linearly reductive over $R$ if and only if the
functors in $\QAdd^{G}T$ are left exact for all $R$-schemes $T$
(see \ref{prop:additive functors are equivariant sheaves glrg}).
In this case we get an equivalence of stacks $\QCoh_{R}^{G}\simeq\QAdd_{R}^{G}$
and similar equivalences exist when we consider the category of finitely
presented quasi-coherent sheaves or locally free sheaves of finite
rank instead of the whole category of quasi-coherent sheaves. Note
that the functor associated with the regular representation $\odi T[G]$
is the forgetful functor $V\longmapsto V\otimes\odi T$. In the particular
case where $G$ is a good linearly reductive group, the functor $\Omega\longmapsto\Omega_{R[G]}$
has a more explicit description: given $\Omega\in\QAdd^{G}T$ there
exists a natural, $G$-equivariant isomorphism
\[
\Omega_{R[G]}\simeq\bigoplus_{V\in I_{G}}\duale V\otimes\Omega_{V}
\]
This shows how the above construction generalizes the isomorphism
(\ref{eq:decomposition fppf locally regular representation introduction}).

Now that we have a way to associate with a $G$-equivariant quasi-coherent
sheaf $\shF$ a functor $\Omega^{\shF}$, the next question is what
additional data $\Omega^{\shF}$ must have to induce a structure of
equivariant sheaf of algebras on $\shF$. The answer is a symmetric,
monoidal structure. Given an $R$-scheme $T$, denote by $\QMon^{G}T$
the category of functors $\Omega\in\QAdd^{G}T$ with a symmetric monoidal
structure and by $\QAlg^{G}T$ the category of quasi-coherent sheaves
of algebras with an action of $G$. Denote also by $\QMon_{R}^{G}$
(resp. $\QAlg_{R}^{G}$) the stack (not in groupoids) whose fiber
over an $R$-scheme $T$ is $\QMon^{G}T$ (resp. $\QAlg^{G}T$).
\begin{thm*}
{[}\ref{thm:G equivariant ring are monoidal functors}{]} Given an
$R$-scheme $T$, the functors   \[   \begin{tikzpicture}[xscale=2,yscale=-0.2]     
\node (aA0_0) at (3, 0) {};     
\node (aA0_1) at (4, 0) {};     
\node (aA1_0) at (3, 1) {$\QAlg^G T \ \ \ \ \ \ \ \ \ \ \ \ \ \ \ \ $};     
\node (aA1_1) at (4, 1) {$\ \ \ \ \ \ \ \ \ \ \ \ \ \ \ \ \ \QMon^G T$};     
\node (aA2_0) at (3, 2) {};     
\node (aA2_1) at (4, 2) {};   

\path (aA0_0) edge [->]node [auto] {$\scriptstyle{\Omega^*}$} (aA0_1);     
\path (aA2_1) edge [->]node [auto] {$\scriptstyle{*_{R[G]}}$} (aA2_0);   
\end{tikzpicture}   \]  yield an equivalence between $\QAlg^{G}T$ and the full subcategory
of $\QMon^{G}T$ of left exact functors.
\end{thm*}
When $G$ is a linearly reductive group we obtain an equivalence of
stacks $\QAlg_{R}^{G}\simeq\QMon_{R}^{G}$ and similar equivalences
are defined if we consider finitely presented quasi-coherent sheaves
or locally free sheaves of finite rank instead of all the quasi-coherent
sheaves.

Note that the regular representation $\odi T[G]$ corresponds to the
forgetful functor $V\arr V\otimes\odi T$ with the obvious monoidal
structure. Denote by $\LMon_{R}^{G}$ the substack of $\QMon_{R}^{G}$
composed of functors with values in the category of locally free sheaves
of finite rank. The answer to the initial problem \ref{prob: data for general G covers introduction}
is the following.
\begin{thm*}
{[}\ref{thm:description of GCov with monoidal functors}{]} The association
\[
\GCov\arr\LMon_{R}^{G}\comma(X\arrdi fT)\longmapsto\Omega^{f_{*}\odi X}=(f_{*}\odi X\otimes-)^{G}
\]
induces an equivalence onto the substack in groupoids of $\LMon_{R}^{G}$
of functors that, as $R$-linear functors, are fppf locally isomorphic
to the forgetful functor. If $G$ is a good linearly reductive group
and $\Spec R$ is connected, this is the substack of functors $\Omega$
such that $\rk\Omega_{V}=\rk V$ for all $V\in\Loc^{G}R$ (or all
$V\in I_{G}$).
\end{thm*}
Notice that the last part of the above Theorem is no longer true even
for linearly reductive groups without a good representation theory.
It fails in the simplest possible case: $G=\Z/3\Z$, $R=\Q$ and $T=\Spec\overline{\Q}$.
In the above correspondence the stack $\Bi G$ of $G$-torsors is
sent to the stack of symmetric, strong monoidal, $R$-linear and exact
functors (see \ref{thm:torsors are isomorphism}). We retrieve in
this way the classical Tannaka's correspondence, which was also the
starting point of the discussion about $G$-covers for general groups
$G$ at the beginning of this section.

The above Theorem shows clearly how the constructions we have made
are just a generalization of the description of $G$-covers when $G$
is a diagonalizable group (see \ref{eq:decomposition for diagonalizable groups introduction}).
If $M=\Hom(G,\Gm)$, $G$ is a good linearly reductive group with
$I_{G}=(\Z_{m})_{m\in M}$ and a functor $\Omega\in\QAdd^{G}T$ for
which $\Omega_{V}$ is locally free of rank $\rk V$ for all $V\in I_{G}$
is just given by a collection of invertible sheaves $\shL_{m}=\Omega_{\Z_{m}}$,
while a monoidal structure on $\Omega$, that is a ring structure
over $\Omega_{\Z[G]}\simeq\oplus_{m\in M}\duale{\Z_{m}}\otimes\shL_{m}$,
is just given by maps 
\[
\Omega_{\Z_{m}}\otimes\Omega_{\Z_{n}}=\shL_{m}\otimes\shL_{n}\arr\shL_{m+n}=\Omega_{\Z_{m+n}}\simeq\Omega_{\Z_{m}\otimes\Z_{n}}
\]
satisfying certain conditions.

From now on $G$ will be a linearly reductive group scheme over a
base ring $R$ and, as always, we will assume that it is flat, finite
and of finite presentation. We want to discuss some applications of
the functorial point of view introduced above.

When $G$ is a diagonalizable group, $M=\Hom(G,\Gm)$ and the sequence
$(\shL_{m},\psi_{m,n})_{m,n\in M}$ defines a $G$-cover (see \ref{eq:decomposition for diagonalizable groups introduction}),
a classical result is that this cover is a $G$-torsor if and only
the maps $\psi_{m,-m}\colon\shL_{m}\otimes\shL_{-m}\arr\shL_{0}=\odi{}$
are isomorphisms for all $m\in M$ (see \cite[Exposé VIII, Proposition 4.1 and 4.6]{Grothendieck1970}).
In this thesis we generalize this property for more general groups.
A\emph{ }linearly reductive group $G$ over an algebraically closed
field is \emph{solvable }(\emph{super solvable}) if it admits a filtration
by closed subgroups $0=H_{0}\triangleleft H_{1}\triangleleft\dots\triangleleft H_{n}=G$
such that, for all $i$, $H_{i+1}/H_{i}\simeq\mu_{p}$ for some prime
$p$ (and $H_{i}\triangleleft G$). A linearly reductive group over
a ring $R$ is solvable (super solvable) if it is so over any geometric
point of $\Spec R$. Denote also by $\LAlg_{R}^{G}$ the full substack
of $\QAlg_{R}^{G}$ of algebras that are locally free of finite rank,
which is isomorphic to $\LMon_{R}^{G}$ via the functor $\Omega^{*}$.
The result we prove is the following:
\begin{thm*}
{[}\ref{thm:torsors when omega surjective}{]} Let $G$ be a super
solvable good linearly reductive group over a ring $R$ and let $\alA\in\LAlg_{R}^{G}T$,
for an $R$-scheme $T$. Then $\alA\in\Bi G$ if and only if $\Omega_{R}^{\alA}(=\alA^{G})\simeq\odi T$
and for all $V\in I_{G}$ the maps
\begin{equation}
\Omega_{V}^{\alA}\otimes\Omega_{\duale V}^{\alA}\arr\Omega_{V\otimes\duale V}^{\alA}\arr\Omega_{R}^{\alA}\simeq\odi T\label{eq:omega torsors introduction}
\end{equation}
are surjective.
\end{thm*}
Notice that the above Theorem is no longer true if we consider solvable
groups, even in the constant case (see \ref{rem:counterexample for omega surjective implies torsors}).
The above criterion will be applied in the study of $(\mu_{3}\rtimes\Z/2\Z)$-covers. 

When $G$ is diagonalizable, we have seen that $\GCov$ is almost
never irreducible. This bad behaviour continues in the non abelian
case:
\begin{thm*}
{[}\ref{thm:GCov reducible when G not abelian}{]} If $G$ is a finite,
non abelian and linearly reductive group then $\GCov$ is reducible.
\end{thm*}
The methods used in the proof of the above Theorem neither reduce,
nor  are applicable to the diagonalizable case and they involve the
study of more general $G$-equivariant algebras than the ones inducing
$G$-covers. Moreover they allow us to construct $G$-covers outside
the main irreducible component $\stZ_{G}$, while the same problem
is more difficult in the diagonalizable case.

Another interesting question about the theory of $G$-covers, and
one of the most important from the point of view of classical algebraic
geometry, is the question about regularity of $G$-covers and, more
generally, about the preservation of geometrical properties: given
a regular (resp. normal, regular in codimension $1$) integral scheme
$Y$ and a cover $f\colon X\arr Y$, understand when $X$ is regular
(resp. normal, regular in codimension $1$). When this happens we
will call this cover regular (resp. normal, regular in codimension
$1$). Notice that a cover of a normal scheme that is regular in codimension
$1$ is normal, because a cover has Cohen-Macaulay fibers. The problem
of detecting the regularity of a cover arises together with the problem
of constructing such a cover. We have to admit that this last problem
seems very difficult to handle in this generality, but one can hope
to be able to find at least families of regular covers. Anyway in
this thesis we will concentrate only on the first problem and only
on the case of regularity in codimension $1$. We generalize what
happens in the diagonalizable case and the leading idea is the following.
If $\alA$ is a locally free algebra of finite rank over a scheme
$Y$, denote by $\hat{\tr}_{\alA}\colon\alA\arr\duale{\alA}$ the
map $x\longmapsto\tr_{\alA}(x\cdot-)$, where $\tr_{\alA}$ is the
trace map $\alA\arr\odi Y$. A classical result is that the algebra
$\alA$ is étale over $Y$ if and only if $\hat{\tr}_{\alA}$ is an
isomorphism (see \cite[Proposition 4.10]{SGA1}). The idea is that,
the less degenerate $\hat{\tr}_{\alA}$ is, the more regular the algebra
$\alA$ should be. If $f\colon\Spec\alA\arr Y$ is the associated
cover, denote by $s_{f}\in(\det\alA)^{-2}$ the determinant of $\hat{\tr}_{\alA}$,
also called the discriminant section. This section is important because
its zero locus is the complement of the locus where $f$ is étale.
If $\alA\in\GCov Y$ and $G$ has a good representation theory, given
$V\in I_{G}$ the map (\ref{eq:omega torsors introduction}) induces
a map $\Omega_{V}^{\alA}\arr\duale{(\Omega_{\duale V}^{\alA})}$ and
we will denote by $s_{f,V}\in\det(\Omega_{V}^{\alA})^{-1}\otimes\det(\Omega_{\duale V}^{\alA})^{-1}$
the section associated with its determinant. When $G$ is an étale,
good linearly reductive group the relation between the sections just
introduced is given by the following isomorphism (see \ref{prop:decomposition of discriminant})

\[
(\det\alA)^{-2}\simeq\bigotimes_{V\in I_{G}}(\det(\Omega_{V}^{f})^{-1}\otimes\det(\Omega_{\duale V}^{f})^{-1})^{\rk V}\text{ such that }s_{f}\longmapsto\bigotimes_{V\in I_{G}}s_{f,V}^{\otimes\rk V}
\]
If we denote by $Y^{(1)}$ the set of codimension $1$ points of $Y$
and by $v_{q}$ the valuation for $q\in Y^{(1)}$, the result we will
prove is the following:
\begin{thm*}
{[}\ref{thm:equivalent conditions for regularity for glrg, global version}{]}
Let $G$ be a finite and étale linearly reductive group over a ring
$R$. Let also $Y$ be an integral, noetherian and regular in codimension
$1$ (resp. normal) $R$-scheme and $f\colon X\arr Y$ be a cover
with a generically faithful action (see \ref{def: generically faithful})
of $G$ on $X$ such that $f$ is $G$-invariant and $X/G=Y$. Then
the following are equivalent:
\begin{enumerate}
\item $X$ is regular in codimension $1$ (resp. normal);
\item the geometric stabilizers of the codimension $1$ points of $X$ are
solvable and for all $q\in Y^{(1)}$ we have $v_{q}(s_{f})<\rk G$
($=\rk f$).
\end{enumerate}
In this case $f$ is generically a $G$-torsor, $f\in\stZ_{G}(Y)$
and the stabilizers of the codimension $1$ points of $X$ are cyclic.
Moreover, if $G$ has a good representation theory, the above conditions
are also equivalent to
\begin{enumerate}
\item [3)]the geometric stabilizers of the codimension $1$ points of $X$
are solvable, $f\in\GCov$ and for all $q\in Y^{(1)}$ and $V\in I_{G}$
we have $v_{q}(s_{f,V})\leq\rk V$.
\end{enumerate}
\end{thm*}
I am strongly convinced that the above statement is still true without
the hypothesis of solvability on the geometric stabilizers. Actually
I am also convinced that, with some minor modifications, the first
part of the statement continues to be true without the existence of
a generically faithful action of a group. I think that the statement
which should be true is:
\begin{conjecture*}
Let $R$ be a discrete valuation ring with residue field $k$ and
$A$ be a finite and flat $R$-algebra. Then 
\[
v_{R}(\det\hat{\tr}_{A})\geq\rk A-|\Spec A\otimes_{R}\overline{k}|
\]
and equality holds if and only if $A$ is regular, generically étale
with separable residue fields and the localizations of $A\otimes_{R}\overline{k}$
have ranks prime to the characteristic of $k$.
\end{conjecture*}
Except for the implication {}``equality $\then$ regularity'', I
am able to prove the rest of the statement. When we have a generically
faithful action of, say, a solvable group $G$ on $A$ one can argue
by induction on $\rk A=\rk G$ considering the invariant algebra for
a normal subgroup of $G$. The base case in this induction is $G=\mu_{p}$,
for some prime $p$, where the result can be easily deduced from the
theory developed for diagonalizable groups.

\section{$(\mu_{3}\rtimes\Z/2\Z)$-covers and $S_{3}$-covers.}

In the last Chapter of this thesis we will study $G$-covers for the
non abelian group scheme $G=\mu_{3}\rtimes\Z/2\Z$ and $S_{3}$-covers.
In oder to simplify the exposition we work over the ring $\stR=\Z[1/6]$.
Denote by $\sigma\in\Z/2\Z(\stR)$ the generator and consider it also
as a section of $G$ and, after choosing a transposition, as a section
of $S_{3}$. The groups $G$ and $S_{3}$ are linearly reductive over
$\stR$ and they also have a good representation theory. We can choose
$I_{G}=\{\stR,A,V\}$, where $A=\stR$ with the action induced by
the non trivial character of $\Z/2\Z$ and $V=\ind_{\mu_{3}}^{G}V_{1}$,
where $V_{1}$ is the $\mu_{3}$ representation associated with the
character $1\in\Z/3\Z$. In the second chapter we prove that, if $H$
and $H'$ are étale locally isomorphic group schemes, then $ $we
have an isomorphism $\Bi(H\rtimes\Autsh H)\simeq\Bi(H'\rtimes\Autsh H')$
(of stacks classifying fppf torsors) (see \ref{prop:bitorsors and semidirect products}).
In particular, considering $H=\mu_{3}$ and $H'=\Z/3\Z$ over $\stR$,
we obtain an isomorphism $\Bi G\simeq\Bi S_{3}$. By the general theory
of bitorsors described in the second chapter, we also obtain an isomorphism
$\GCov\simeq\RCov{S_{3}}$ over $\stR$. Thus the study of $G$-covers
coincides with the study of $S_{3}$-covers, and, due to the nature
of the isomorphism $\GCov\simeq\RCov{S_{3}}$, the problems of regularity
of covers also coincide. Anyway we will describe the structure of
$G$-equivariant algebras only, because the representation theory
of $G$ has a simpler explicit description and all the theory works
over $\Z[1/2]$, instead of $\Z[1/6]$. The groups $G$ and $S_{3}$
can be considered the simplest non abelian linearly reductive groups.
This is essentially the motivation for a detailed study of $G$-covers
and $S_{3}$-covers.

A similar analysis of $S_{3}$-covers is conducted in \cite{Easton2008},
where the author describes the data needed to build them in terms
of linear algebra. Here, using a different approach, we recover this
result and we expand it, describing particular families of $S_{3}$-covers,
characterizing the regular ones and computing the invariants of the
total space of a regular $S_{3}$-cover of a surface.

Using the theory developed above, a $G$-cover over an $\stR$-scheme
$T$ corresponds to an $\stR$-linear, symmetric and monoidal functor
$\Omega\colon\Loc^{G}\stR\arr\Loc T$ such that $\rk\Omega_{W}=\rk W$
for all $W\in I_{G}$. It is easy to deduce the data needed to build
a $G$-cover. Since $\Omega_{\stR}=\odi T$ for general reasons, we
need an invertible sheaf $\shL=\Omega_{A}$ and a locally free sheaf
$\shF=\Omega_{V}$ of rank $2$ in order to have a functor $\Omega\in\QAdd_{\stR}^{G}$.
For the monoidal structure, for all $W_{1},W_{2}\in I_{G}$ we need
maps $\Omega_{W_{1}}\otimes\Omega_{W_{2}}\arr\Omega_{W_{1}\otimes W_{2}}$.
Since we are interested in commutative algebras with unity and we
have relations $A\otimes A\simeq\stR$, $A\otimes V\simeq V$ and
$V\otimes V\simeq\stR\oplus A\oplus V$, a monoidal structure on $\Omega$
is given by maps
\[
\shL\otimes\shL\arrdi m\odi T\comma\;\shL\otimes\shF\arrdi{\alpha}\shF\comma\;\shF\otimes\shF\arrdi{(-,-)\oplus\la-,-\ra\oplus\beta}\odi T\oplus\shL\oplus\shF
\]
satisfying certain conditions, required for the associativity of $\Omega$.
As it happens in the diagonalizable case, this is the hard part. Such
conditions imply that $(-,-),\beta$ are symmetric, $\la-,-\ra$ is
antisymmetric and that $(-,-)$ is uniquely determined by the other
maps. In conclusion it turns out that a $G$-cover over $T$ is associated
with a sequence $\chi=(\shL,\shF,m,\alpha,\beta,\la-,-\ra)$ where
$\shL$ is an invertible sheaf, $\shF$ is a locally free sheaf of
rank $2$ and $m,\alpha,\beta,\la-,-\ra$ are maps
\[
\shL^{2}\arrdi m\odi T\comma\;\shL\otimes\shF\arrdi{\alpha}\shF\comma\;\Sym^{2}\shF\arrdi{\beta}\shF\comma\;\det\shF\arrdi{\la-,-\ra}\shL
\]
that satisfy certain conditions. The above association will be formulated
in terms of isomorphism of stacks (see \ref{thm:global data for Sthree}).
I do not think that further simplifications are possible in this generality.
Although the data above are directly associated to a $G$-cover, they
also correspond to an $S_{3}$-cover, as remarked above. We will identify
$\GCov$ and $\RCov{S_{3}}$ with the stack of data defined as above
and all the results cited below, including the ones regarding the
geometry of covers, continue to be true if we replace $\GCov$ by
$\RCov{S_{3}}$ and $G$-covers by $S_{3}$-covers. Anyway some general
results will be stated for both $G$ and $S_{3}$. The main idea followed
in order to get to a better understanding of $G$-covers and $S_{3}$-covers
is to look at particular loci of $\GCov$, that is to look at data
as above satisfying additional conditions. All those loci are interesting
because they will allow to understand the geometry of $\GCov$ and
$\RCov{S_{3}}$ and also to describe regular $G$-covers and $S_{3}$-covers.

It is convenient at this point to introduce more notation. Denote
by $\stC_{3}$ the stack of pairs $(\shF,\delta)$ where $\shF$ is
a locally free sheaf of rank $2$ and $\delta$ is a map $\Sym^{3}\shF\arr\det\shF$
and by $\Cov_{3}$ the stack of degree $3$ covers, also called triple
covers. It is a well known result of the theory of triple covers (see
\cite{Miranda1985,Pardini1989,Bolognesi2009}) that there exists an
isomorphism of stacks $\stC_{3}\arr\Cov_{3}$ so defined: an object
$\Phi=(\shF,\delta)\in\shC_{3}(T)$, where $T$ is an $\stR$-scheme,
induces maps $\eta_{\delta}\colon\Sym^{2}\shF\arr\odi T$ and $\beta_{\delta}\colon\Sym^{2}\shF\arr\shF$
which define an algebra structure on the sheaf $\alA_{\Phi}=\odi T\oplus\shF$.
Taking invariants by $\sigma\in\Z/2\Z$ we obtain a map $\pi\colon\GCov\arr\Cov_{3}\simeq\stC_{3}$.
Notice that the same procedure yields a map $\RCov{S_{3}}\arr\Cov_{3}$
and it is possible to prove that $\GCov$ and $\RCov{S_{3}}$ are
isomorphic over $\Cov_{3}$ (see \ref{rem:invariants by sigma for Sthree and GCov}).
The first result on the geometry of $\GCov$ we prove is the following:
\begin{thm*}
{[}\ref{thm:The-locus when omega is invertible}{]} The map $\pi\colon\GCov\arr\Cov_{3}$
restricts to an isomorphism of stacks $\stU_{\omega}\arr\Cov_{3}$,
where $\stU_{\omega}$ is the open substack of $\GCov$ where $\la-,-\ra\colon\det\shF\arr\shL$
is an isomorphism.
\end{thm*}
In particular this gives a functorial way of extending triple covers
to $G$-covers or $S_{3}$-covers. Looking at the global geometry
we prove that:
\begin{thm*}
{[}\ref{thm:description of Sthree torsors}, \ref{thm:Geometry of Sthree cov},
\ref{thm:description of the main component}{]} The stacks $\GCov$
and $\RCov{S_{3}}$ are connected, non-reduced and have two irreducible
components, the main one $\stZ_{G}$, which coincides with the zero
locus of the maps $\shL\arr\duale{\shF}\otimes\shF\arr\odi T$ and
$\shF\arr\duale{\shF}\otimes\shF\arr\odi T$ induced by $\alpha$
and $\beta$ respectively, and the closed locus of $\GCov$ where
$\beta=\la-,-\ra=0$ and $\alpha$ is fppf locally a multiple of the
identity. Moreover $\Bi G\subseteq\GCov$ is the open substack where
$\la-,-\ra$ and $m$ are isomorphisms.
\end{thm*}
For covers in $\stZ_{G}$ are possible two further simplifications
of the data associated with them, one for the whole $\stZ_{G}$ and
one that regards particular objects of $\stZ_{G}$. We want to describe
only the second simplification, because it will be the one used in
the description of regular $G$-covers. Given $\chi=(\shL,\shF,m,\alpha,\beta,\la-,-\ra)\in\GCov$
set $\shM=\shL\otimes\det\shF^{-1}$ and $\omega\in\shM$ the section
corresponding to $\la-,-\ra$. Moreover, given an $\stR$-scheme $T$
denote by $\stZ_{\omega}(T)$ the full subcategory of $\stZ_{G}(T)$
where $\odi T\arrdi{\omega}\shM$ is injective, which means that $\omega$
yields a Cartier divisor over $T$. We will prove that $\stZ_{\omega}(T)$
is isomorphic to the category whose objects are sequences $(\shM,\shF,\delta,\omega)$
where $(\shF,\delta)\in\stC_{3}(T)$, $\shM$ is an invertible sheaf
and $\omega\in\shM$ is a section such that $\odi T\arrdi{\omega}\shM$
is injective and its image contains the image of $\eta_{\delta}\colon\Sym^{2}\shF\arr\odi T$
(see \ref{thm:description of ZG when omega in Cartier}). In particular
we see that the extensions of a triple cover $(\shF,\delta)$ to a
$G$-cover in $\stZ_{\omega}(T)$ correspond bijectively to the effective
Cartier divisors contained in the locus where $\eta_{\delta}$ is
zero.

The last part of this thesis is dedicated to the study of regular
$G$-covers and $S_{3}$-covers. Notice that it is possible to apply
directly the result previously obtained on covers that are regular
in codimension $1$ for general groups (see \ref{thm:application to Sthree of theory on regular in codimension 1 covers}),
but what we get is a particular case of the description of regular
$G$-covers we want to explain. Let $Y$ be an integral, noetherian
and regular scheme. Given $\chi=(\shL,\shF,m,\alpha,\beta,\la-,-\ra)\in\GCov(Y)$
we define: $D_{m}$ and $D_{\omega}$ as the closed subschemes of
$Y$ where $m\colon\shL^{2}\arr\odi Y$ and $\la-,-\ra\colon\det\shF\arr\shL$
are zero respectively; $Y_{\alpha}$ as the vanishing locus of the
map $\alpha\colon\shL\otimes\shF\arr\shF$. Notice that we have an
inclusion $Y_{\alpha}\subseteq D_{m}$. Given $\Phi=(\shF,\delta)\in\shC_{3}(Y)$
we define: $Y_{\delta}$ and $D_{\delta}$ as the closed subschemes
of $Y$ defined by $\eta_{\delta}\colon\Sym^{2}\shF\arr\odi Y$ and
the discriminant $\Delta_{\Phi}\colon(\det\shF)^{2}\arr\odi Y$ respectively,
where the last map is induced by the determinant of $\hat{\tr}_{\alA_{\Phi}}\colon\alA_{\Phi}\arr\duale{\alA_{\Phi}}$.
Finally, given a proper, closed subscheme $Z$ of $Y$, denote by
$D(Z)$ the divisorial component of $Z$ in $Y$, that is the maximum
among the effective Cartier divisors contained in $Z$. The Theorem
we will prove is the following.
\begin{thm*}
{[}\ref{thm:regular Sthree covers first}, \ref{thm:gamma for regular Sthree covers},
\ref{thm:regular G covers and triple}{]} Let $Y$ be a regular, noetherian
and integral scheme such that $\dim Y\geq1$ and $6\in\odi Y^{*}$.
If $\chi=(\shL,\shF,m,\alpha,\beta,\la-,-\ra)\in\GCov(Y)$ then the
associated $G$-cover ($S_{3}$-cover) $X_{\chi}\arr Y$ is regular
if and only if the following conditions hold:
\begin{enumerate}
\item $D_{m},D_{\omega}$ are Cartier divisors and $D_{m}\cap D_{\omega}=\emptyset$;
\item $Y_{\alpha}=\emptyset$ or $Y_{\alpha}$ is regular of pure codimension
$2$ in $Y$;
\item $D_{\omega}$ is regular and $D_{m}$ is regular outside $Y_{\alpha}$.
\end{enumerate}
In this case the triple cover $X_{\chi}/\sigma\arr Y$ (which does
not depend on whether we see $X$ as a $G$-cover or $S_{3}$-cover)
is regular and, if $(\shF,\delta)\in\shC_{3}$ is its associated object,
we have: $D_{\omega}=D(Y_{\delta})$, $D_{\delta}=2D_{\omega}+D_{m}$
and $Y_{\delta}=D_{\omega}\sqcup Y_{\alpha}$. If $f\colon X\arr Y$
is a regular triple cover associated with $(\shF,\delta)\in\shC_{3}(Y)$,
then $Y_{\delta}=D(Y_{\delta})\sqcup Y'_{\delta}$, where $Y'_{\delta}$
is a closed subscheme of pure codimension $2$ if not empty and $D(Y_{\delta})$
is regular. Finally the maps   \[   \begin{tikzpicture}[xscale=7.0,yscale=-0.6]     \node (A0_0) at (0, 0) {$X$};     \node (A0_1) at (1, 0) {$X/\sigma$};     \node (A1_0) at (0, 1) {$\{\text{regular }G\text{-covers over }Y\}$};     \node (A2_1) at (1, 2) {$\left\{ \begin{array}{c} \text{regular triple covers }(\shF,\delta)\text{ over }Y\\ \text{such that }Y_{\delta}\text{ is regular} \end{array}\right\}$};     \node (A3_0) at (0, 3) {$\{\text{regular }S_3\text{-covers over }Y\}$};     \node (A4_0) at (0, 4) {$(\odi Y(D(Y_\delta)),\shF,\delta,1)$};     \node (A4_1) at (1, 4) {$(\shF,\delta)$};     \path (A0_0) edge [|->,gray]node [auto] {$\scriptstyle{}$} (A0_1);     \path (A3_0) edge [->]node [auto] {$\scriptstyle{}$} (A2_1);     \path (A4_1) edge [|->,gray]node [auto] {$\scriptstyle{}$} (A4_0);     \path (A1_0) edge [->]node [auto] {$\scriptstyle{}$} (A2_1);     \path (A1_0) edge [->]node[rotate=-90] [above] {$\scriptstyle{\simeq}$} (A3_0);   \end{tikzpicture}   \] are
inverses of each other.
\end{thm*}
After this classification of regular $G$-covers and $S_{3}$-covers
two questions naturally arise: how to construct them and how the cohomological
invariants of the total space related to those of the base. If $Z$
is a scheme and $\E$ is a coherent sheaf over it, we will say that
$\E$ is strongly generated if, for any closed point $q\in Z$, the
map $\Hl^{0}(Z,\E)\arr\E\otimes(\odi{Z,p}/m_{p}^{2})$ is surjective.
For the first question, we will prove the following result.
\begin{thm*}
{[}\ref{thm:construction of Sthree covers}{]} Let $k$ be an infinite
field with $\car k\neq2,3$, $Y$ be a smooth, irreducible and proper
$k$-scheme with $\dim Y\geq1$, $\shF$ be a locally free sheaf of
rank $2$ over $Y$ and set $\E=\Homsh(\Sym^{3}\shF,\det\shF)$. If
$\E\otimes\overline{k}$ is strongly generated (over $Y\times\overline{k}$)
then there exists $\delta\in\E$ such that the triple cover associated
with $(\shF,\delta)\in\shC_{3}(Y)$ extends to a $G$-cover ($S_{3}$-cover)
$X_{\delta}\arr Y$ with $X_{\delta}$ smooth and $Y_{\delta}=\emptyset$
or $\codim_{Y}Y_{\delta}=2$. Moreover, if $Y$ is geometrically connected,
then $X_{\delta}$ is geometrically connected if and only if $\det\shF\not\simeq\odi Y$
and $\Hl^{0}(Y,\shF)=0$.
\end{thm*}
When $Y$ is projective, it is possible to prove that, if $\E(-1)$
is globally generated, then $\shF$ satisfies the hypothesis of strong
generation in the above theorem (see \ref{prop:when F admits regular Sthree covers}).
For instance $\shF=\odi Y(-1)^{2}$ satisfies such hypothesis and
$\det\shF\not\simeq\odi Y$ and $\Hl^{0}(Y,\shF)=0$. Therefore
\begin{cor*}
{[}\ref{cor:Sthree regular covers over projective smooth}{]} Let
$k$ be an infinite field with $\car k\neq2,3$. Then any smooth,
projective and irreducible (resp. geometrically connected) $k$-scheme
$Y$ with $\dim Y\geq1$ has a $G$-cover ($S_{3}$-cover) $X\arr Y$
with $X$ smooth (resp. smooth and geometrically connected).
\end{cor*}
Finally, when $Y$ is a surface over an algebraically closed field,
we will compute the invariants of the total space of a regular $S_{3}$-cover
of $Y$. The result is:
\begin{thm*}
{[}\ref{thm:invariants of regular Sthree covers}{]} Let $Y$ be a
smooth, projective, integral surface over an algebraically closed
field $k$ such that $\car k\neq2,3$ and $f\colon X\arr Y$ be a
regular $S_{3}$-cover associated with $(\shF,\delta)\in\shC_{3}(Y)$.
The closed subscheme $Y_{\delta}$ of $Y$ is the disjoint union of
a smooth divisor $D$ and a finite set $Y_{0}$ of rational points
and the surface $X$ is connected if and only if $\Hl^{0}(\shF)=0$
and $\odi Y(-D)\not\simeq\det\shF$. In this case the invariants of
$X$ are given by
\begin{alignat*}{1}
K_{X}^{2}= & \;6K_{Y}^{2}+6c_{1}(\shF)^{2}-12c_{1}(\shF)K_{Y}-\frac{10}{3}D^{2}-4DK_{Y}\\
p_{g}(X)= & \; p_{g}(Y)+2h^{2}(\shF)+h^{2}(\odi Y(D)\otimes\det\shF)\\
\chi(\odi X)= & \;6\chi(\odi Y)-2c_{2}(\shF)+\frac{1}{2}(3c_{1}(\shF)^{2}-3c_{1}(\shF)K_{Y}-DK_{Y}-D^{2})\\
|Y_{0}|= & \;3c_{2}(\shF)-\frac{2}{3}D^{2}
\end{alignat*}

\end{thm*}

\section{Acknowledgments.}

The first person I would like to thank is surely my advisor Angelo
Vistoli, first of all for having proposed me the problem I will discuss
and, above all, for his continuous support and encouragement. I also
acknowledge Professors Rita Pardini, Matthieu Romagny, Diane Maclagan
and Bernd Sturmfels for the useful conversations we had and all the
suggestions I have received from them. Special thanks go to Tony Iarrobino,
who first suggested me the relation between $G$-covers and equivariant
Hilbert schemes and to the referee of my paper \cite{Tonini2011},
who helped me in the exposition, in particular in translating this
thesis from the Italian English language to something closer to English.
Finally I want to thank Mattia Talpo, John Calabrese and Dajano Tossici
for the countless times we found ourselves staring at the blackboard
trying to answer some mathematical question.

\section{Notation.}

\subsection*{General.}

A \emph{cover} is a map of schemes $f\colon X\arr Y$ which is finite,
flat and of finite presentation or, equivalently, which is affine
with $f_{*}\odi X$ locally free of finite rank. We will say that
the cover $f$ is regular (resp. regular in codimension $1$, normal,
normal crossing in codimension $1$) if the total space $X$ has the
same property. (The definition of normal crossing in codimension $1$
will be introduced later.)

If $X$ is a scheme and $p\in X$ we set $\codim_{p}X=\dim\odi{X,p}$
and we will denote by $X^{(1)}=\{p\in X\st\codim_{p}X=1\}$ the set
of codimension $1$ points of $X$.

Given $\alpha\in\N$, we will use the following convention
\[
0^{\alpha}=\begin{cases}
1 & \alpha=0\\
0 & \alpha>0
\end{cases}
\]

We denote by $\set$ the category of sets, by $\Sch/S$ the category
of schemes over a base scheme $S$ and by $\Grp$ the category of
groups. Given a (fppf) stack $\stX$ over a scheme $S$ we will denote
by $\stX^{\text{gr}}$ the associated stack of groupoids and, if $\stX$
is an algebraic stack, we denote by $|\stX|$ its associated topological
space. 

By a Henselian ring we always mean a noetherian local ring which is
Henselian. If $A$ is a local ring we will often denote by $m_{A}$
is maximal ideal. A DVR will be a local discrete valuation ring.

\subsection*{Sheaf Theory.}

Let $S$ be a scheme. We will denote by $\QCoh_{S}$, $\FCoh_{S}$,
$\Loc_{S}$ the stacks of quasi-coherent sheaves, finitely presented
quasi-coherent sheaves, locally free sheaves of finite rank over $S$
respectively. Let $\shF\in\QCoh S$. We define the functor $\WW(\shF)\colon(\Sch/S)^{\textup{op}}\arr\set$
as
\[
\WW(\shF)(U\arrdi fS)=\Hl^{0}(U,f^{*}\shF)
\]
Notice that if $\shF$ is a locally free sheaf of finite rank, then
$\WW(\shF)$ is smooth and affine over $S$. The expression $s\in\shF$
will always mean $s\in\shF(S)=\Hl^{0}(S,\shF)$. Moreover we will
denote by $V(s)$ the zero locus of $s$ in $S$, i.e. the closed
subscheme associated with the sheaf of ideals $\Ker(\odi S\arrdi s\shF)$.
Given an element $f=(a_{1},\dots,a_{r})\in\Z^{r}$ and invertible
sheaves $\shL_{1},\dots,\shL_{r}$ on a scheme we will use the notation
\[
\underline{\shL}^{f}=\bigotimes_{i}\shL_{i}^{\otimes a_{i}}\comma\ISym\underline{\shL}=\ISym(\shL_{1},\dots,\shL_{r})=\bigoplus_{g\in\Z^{r}}\underline{\shL}^{g}
\]
Notice also that, if $\shL_{i}=\odi S$ for all $i$, then there is
a canonical isomorphism $\underline{\shL}^{f}\simeq\odi{}$.

\subsection*{Representation theory.}

Let $S$ be a scheme. Given an affine group scheme $f\colon G\arr S$,
we will denote by $\odi S[G]=f_{*}\odi G$ its associated Hopf algebra
and by
\[
\Delta_{G}\colon\odi S[G]\arr\odi S[G]\otimes\odi S[G]\comma\varepsilon_{G}\colon\odi S[G]\arr\odi S\comma\sigma_{G}\colon\odi S[G]\arr\odi S[G]
\]
the co-multiplication, the co-unity and the co-inverse of $G$ respectively.
By an action of $G$ on a quasi-coherent sheaf $\shF$ over $S$ we
mean a left action of $G$ on $\WW(\shF)$, which corresponds to a
structure of right $\odi S[G]$-comodule $\shF\arr\shF\otimes\odi S[G]$.
We will often also call it a $G$-comodule structure or call $\shF$
a $G$-equivariant sheaf. By an action of $G$ on a $S$-scheme $X$
we mean a right action $X\times G\arr X$. If $X=\Spec\alA$, for
some $\odi S$-algebra $\alA$, this means that we have a $G$-comodule
structure $\alA\arr\alA\otimes\odi S[G]$ which is an algebra homomorphism,
or, equivalently, such that the multiplication $\alA\otimes\alA\arr\alA$
is $G$-equivariant and $1\in\alA^{G}$. 

Given functors $F\comma H\colon(\Sch/S)^{\text{op}}\arr\set$, left
actions of $G$ on $F$ and $H$ induce a left action on $\Homsh(F,H)$
given by   \[   \begin{tikzpicture}[xscale=5.0,yscale=-0.7]     \node (A0_0) at (0, 0) {$G\times \Homsh(F,H)$};     \node (A0_1) at (1, 0) {$\Homsh(F,H)$};     \node (A1_0) at (0, 1) {$(g,\varphi)$};     \node (A1_1) at (1, 1) {$g\varphi g^{-1}$};     \path (A0_0) edge [->]node [auto] {$\scriptstyle{}$} (A0_1);     \path (A1_0) edge [|->,gray]node [auto] {$\scriptstyle{}$} (A1_1);   \end{tikzpicture}   \] 
Let $\shF$ be a locally free sheaf of finite rank over $S$ and $\shH\in\QCoh S$
($\FCoh S$). Then $\Homsh(\shF,\shH)\in\QCoh S$ ($\FCoh S$) and
we have a natural isomorphism
\[
\WW(\Homsh(\shF,\shH))\arr\Homsh(\WW(\shF),\WW(\shH))
\]
In particular, actions of $G$ on $\shF$ and $\shH$ yield an action
of $G$ on $\Homsh(\shF,\shH)$. We denote by $\Homsh^{G}(\WW(\shF),\WW(\shH))$
(resp. $\Endsh^{G}\WW(\shF)$, $\Autsh^{G}\WW(\shF)$) the subfunctor
of $\Homsh(\WW(\shF),\WW(\shH))$ (resp. $\Endsh\WW(\shF)$, $\Autsh\WW(\shF)$)
given by the $G$-invariant elements, that are exactly the $G$-equivariant
morphisms. In particular we have 
\[
\WW(\Homsh(\shF,\shH))^{G}\simeq\Homsh(\WW(\shF),\WW(\shH))^{G}\simeq\Homsh^{G}(\WW(\shF),\WW(\shH))
\]
The subsheaf of $G$-invariants of $\Homsh(\shF,\shH)$, denoted by
$\Homsh^{G}(\shF,\shH)$, coincides with the subsheaf of morphisms
preserving the $G$-comodule structures. Finally set $\Endsh^{G}(\shF)=\Homsh^{G}(\shF,\shF)$.

We will denote by $\QCoh_{S}^{G}$, $\FCoh_{S}^{G}$, $\Loc_{S}^{G}$
the stacks over $S$ of $G$-equivariant quasi-coherent sheaves, finitely
presented quasi-coherent sheaves, locally free sheaves of finite rank
respectively. If $\shF\in\QCoh S$ we will denote by $\underline{\shF}\in\QCoh^{G}S$
the quasi-coherent sheaf $\shF$ with the trivial action of $G$.
Moreover if $\shF\in\QCoh^{G}S$ and $\delta\colon H\arr G$ is a
morphism from a group scheme $H$ over $S$ we will denote by $\R_{H}\shF\in\QCoh^{H}S$
the sheaf $\shF$ with the $H$-action induced by $\delta$.

By a subgroup scheme $H$ of a flat and finitely presented group scheme
$G$ we will always mean a subgroup which is a closed subscheme of
$G$ and it is flat and finitely presented over the base. If $N$
is an abelian group we set $\Di N=\Homsh_{\textup{groups}}(N,\Gm)$
for the diagonalizable group associated with it.

\chapter{Preliminaries on Galois covers.}

We fix a base scheme $S$ and a flat and finite group scheme $G$
finitely presented over $S$. In this chapter we want to introduce
some basic definitions about $G$-covers and prove some general results.
This is how the chapter is divided.

\emph{Section 1. }We define the notion of $G$-covers and we introduce
the stack $\GCov$ of $G$-covers. We will then prove that $\GCov$
is an algebraic stack containing $\Bi G$ as open substack and, as
examples, we will describe $\GCov$ for the groups $G=\mu_{2},\mu_{3},\alpha_{p}$.

\emph{Section 2. }We define the main irreducible component $\stZ_{G}$
of $\GCov$ as the schematic closure of $\Bi G$ in $\GCov$.

\emph{Section 3. }We show that the isomorphisms $\Bi G\simeq\Bi H$
correspond to $(G,H)$-bitorsors and we will explain how they induce
isomorphisms $\GCov\simeq\HCov$.

\section{The stack $\GCov$.}

We start defining the regular representation of a group on itself.
\begin{defn}
\label{def:regular representation}The (right) \emph{regular }action
of $G$ on itself is the action given by 
\[
G\times G\arr G\comma(x,g)\arr x\star g=g^{-1}x
\]
The \emph{regular representation of $G$ }over $S$ is the sheaf $\odi S[G]$
endowed with the co-module structure $\odi S[G]\arrdi{\mu_{G}}\odi S[G]\otimes\odi S[G]$
induced by the right regular action of $G$ on itself. By definition
$\mu_{G}$ is the composition
\[
\odi S[G]\arrdi{\Delta}\odi S[G]\otimes\odi S[G]\arrdi{\text{swap}}\odi S[G]\otimes\odi S[G]\arrdi{\id\otimes\sigma}\odi S[G]\otimes\odi S[G]
\]
\end{defn}
\begin{rem*}
We have chosen to define the regular action of $G$ on itself by $x\star g=g^{-1}x$
instead of the more usual $x\star g=xg$ because this makes computations
natural in other situations. Note that however these two actions are
isomorphic.
\end{rem*}
In what follows, we will denote by $\alA$ the regular representation.
\begin{defn}
\label{def:Gcovers}Given a scheme $T$ over $S$, a \emph{ramified
Galois cover of group $G$}, or simply a $G$\emph{-cover,} over it
is a cover $X\arrdi fT$ together with an action of $G_{T}$ on it
such that there exists an fppf covering $\{U_{i}\arr T\}$ and isomorphisms
of $G$-comodules
\[
(f_{*}\odi X)_{|U_{i}}\simeq\alA_{|U_{i}}
\]
We will call $\GCov(T)$ the groupoid of $G$-covers over $T$, where
the arrows are the $G$-equivariant isomorphisms of schemes over $T$.
\end{defn}
The $G$-covers form a stack $\GCov$ over $S$. Moreover any $G$-torsor
is a $G$-cover and more precisely we have:
\begin{prop}
\label{pro:BG open in GCov}$\Bi G$ is an open substack of $\GCov$.\end{prop}
\begin{proof}
Given a scheme $U$ over S and a $G$-cover $X=\Spec\alB$ over $U$,
$X$ is a $G$-torsor if and only if the map $G\times X\arr X\times X$
is an isomorphism. This map is induced by a map $\alB\otimes\alB\arrdi h\alB\otimes\odi{}[G_{U}]$
and so the locus over which $X$ is a $G$-torsor is given by the
vanishing of $\Coker h$, which is an open subset.
\end{proof}
In order to prove that $\GCov$ is an algebraic stack we will present
it as a quotient stack by a smooth group scheme.
\begin{prop}
The functor   \[   \begin{tikzpicture}[xscale=4.5,yscale=-0.7]     \node (A0_0) at (0, 0) {$(\Sch/S)^{\textup{op}}$};     \node (A0_1) at (1, 0) {$\set$};     \node (A1_0) at (0, 1) {$T$};     \node (A1_1) at (1, 1) {$
\left\{ \begin{array}{c}
\text{algebra structures on }\alA_{T}\\
\text{in the category of }G\text{-comodules}
\end{array}\right\}
$};     \path (A0_0) edge [->] node [auto] {$\scriptstyle{X_G}$} (A0_1);     \path (A1_0) edge [|->,gray] node [auto] {$\scriptstyle{}$} (A1_1);   \end{tikzpicture}   \] is an affine scheme finitely presented over $S$.\end{prop}
\begin{proof}
Let $T$ be a scheme over $S$. An element of $X_{G}(T)$ is given
by maps
\[
\alA_{T}\otimes\alA_{T}\arrdi m\alA_{T}\comma\odi T\arrdi e\alA_{T}
\]
for which $\alA$ becomes a sheaf of algebras with multiplication
$m$ and identity $e(1)$ and such that $\mu$ is a homomorphism of
algebras over $\odi T$. In particular $e$ has to be an isomorphism
onto $\alA^{G}=\odi T$. Therefore we have an inclusion $X_{G}\subseteq\Homsh(\WW(\alA\otimes\alA),\WW(\alA))\times\mathbb{G}_{m}$,
which turns out to be a closed immersion, since locally, after we
choose a basis of $\alA$, the above conditions translate into the
vanishing of certain polynomials.\end{proof}
\begin{prop}
$\Autsh^{G}\WW(\alA)$ is a smooth group scheme finitely presented
over $S$.\end{prop}
\begin{proof}
If $T$ is an $S$-scheme, the morphisms   \[   \begin{tikzpicture}[xscale=3,yscale=-0.6]     \node (A0_0) at (0, 0) {$\varepsilon \circ \phi$};     \node (A0_1) at (1, 0) {$\phi$};     \node (A1_0) at (0, 1) {$\duale{\odi{T}[G]}$};     \node (A1_1) at (1, 1) {$\Endsh^G (\alA\otimes\odi T)$};     \node (A2_0) at (0, 2) {$f$};     \node (A2_1) at (1, 2) {$(f \otimes \id) \circ \Delta$};     \path (A1_0) edge [->]node [auto] {$\scriptstyle{}$} (A1_1);     \path (A2_0) edge [|->,gray]node [auto] {$\scriptstyle{}$} (A2_1);     \path (A0_1) edge [|->,gray]node [auto] {$\scriptstyle{}$} (A0_0);   \end{tikzpicture}   \] 
where $\Delta$ and $\varepsilon$ are respectively the co-multiplication
and the co-unit of $\odi T[G]$, are inverses of each other. Since
\[
\WW(\duale{\odi S[G]})\simeq\Homsh(\WW(\odi S[G]),\WW(\odi S))
\]
we obtain an isomorphism $\Endsh^{G}\WW(\alA)\simeq\WW(\duale{\odi S[G]})$,
so that $\Endsh^{G}\WW(\alA)$ and its open subscheme $\Autsh^{G}\WW(\alA)$
are smooth and finitely presented over $S$.\end{proof}
\begin{rem}
$\Autsh^{G}\WW(\alA)$ acts on $X_{G}$ in the following way. Given
a scheme $T$ over $S$, a $G$-equivariant automorphism $f\colon\alA_{T}\arr\alA_{T}$
and $(m,e)\in X_{G}(T)$ we can set $f(m,e)$ for the unique structure
of sheaf of algebras on $\alA_{T}$ such that $f\colon(\alA_{T},m,e)\arr(\alA_{T},f(m,e))$
is an isomorphism of $\odi T$-algebras.\end{rem}
\begin{prop}
\label{pro:GCov is a quotient stack}The map $X_{G}\arrdi{\pi}\GCov$,
which sends a structure of algebra $\chi\in X_{G}(T)$ on $\alA_{T}$
to the cover $\Spec(\alA_{T},\chi)\arr T$ is an $\Autsh^{G}\WW(\alA)$-torsor.
In particular 
\[
\GCov\simeq[X_{G}/\Autsh^{G}\WW(\alA)]
\]
\end{prop}
\begin{proof}
Consider a cartesian diagram   \[   \begin{tikzpicture}[xscale=1.8,yscale=-1.2]     \node (A0_0) at (0, 0) {$P$};     \node (A0_1) at (1, 0) {$X_G$};     \node (A1_0) at (0, 1) {$U$};     \node (A1_1) at (1, 1) {$\GCov$};     \path (A0_0) edge [->] node [auto] {$\scriptstyle{}$} (A0_1);     \path (A1_0) edge [->] node [auto,swap] {$\scriptstyle{f}$} (A1_1);     \path (A0_1) edge [->] node [auto] {$\scriptstyle{\pi}$} (A1_1);     \path (A0_0) edge [->] node [auto] {$\scriptstyle{}$} (A1_0);   \end{tikzpicture}   \] 
where $U$ is a scheme and $f\colon Y\arr U$ is a $G$-cover. We
want to prove that $P$ is an $\Autsh^{G}\WW(\alA)$ torsor over $U$
and that the map $P\arr X_{G}$ is equivariant. Since $\pi$ is an
fppf epimorphism, we can assume that $f$ comes from $X_{G}$, i.e.
$f_{*}\odi Y=\alA_{U}$ with multiplication $m$ and neutral element
$e$. It is now easy to prove that   \[   \begin{tikzpicture}[xscale=2.1,yscale=-0.5]     \node (A0_0) at (0, 0) {$\Autsh^G \WW(\alA_U)$};     \node (A0_1) at (1, 0) {$P$};     \node (A1_0) at (0, 1) {$h$};     \node (A1_1) at (1, 1) {$h(m,e)$};     \path (A0_0) edge [->]node [auto] {$\scriptstyle{\simeq}$} (A0_1);     \path (A1_0) edge [|->,gray]node [auto] {$\scriptstyle{}$} (A1_1);   \end{tikzpicture}   \] 
is a bijection and that all the other claims hold.
\end{proof}
Using above propositions we can conclude that:
\begin{thm}
\label{cor:GCov is algebraic}The stack $\GCov$ is algebraic and
finitely presented over $S$.
\end{thm}
We want now to discuss some examples of $G$-covers. The simplest
possible case, is the trivial group $G$ over $\Z$: clearly, in this
case, the $G$-covers are the isomorphisms and $\GCov\simeq\Spec\Z$.
Probably the following examples are more interesting.
\begin{example}
$G=\mu_{2}=\Di{\Z/2\Z}$ over $\Z$. This is very classical. A $\mu_{2}$-cover
over a scheme $T$ is given by an invertible sheaf $\shL$ over $T$
with a morphism $\shL^{2}\arr\odi T$, where the induced $\mu_{2}$-cover
is $\Spec\alA$, $\alA=\odi T\oplus\shL$. In particular
\[
\RCov{\mu_{2}}\simeq[\A^{1}/\Gm]
\]
is smooth and irreducible.
\end{example}

\begin{example}
$G=\mu_{3}=\Di{\Z/3\Z}$ over $\Z$. In \cite[Lemma 6.2]{Arsie2004},
the authors prove that the data consisting of invertible sheaves $\shL_{1},\shL_{2}$
over a scheme $T$ and maps $\shL_{1}^{2}\arr\shL_{2},\shL_{2}^{2}\arr\shL_{1}$
yields a unique algebra structure on $\odi T\oplus\shL_{1}\oplus\shL_{2}$.
It is not difficult to see (and we will prove in the next chapter)that
all $\mu_{3}$-covers can be built in this way. In particular 
\[
\RCov{\mu_{3}}\simeq[\A^{2}/\mathbb{G}_{m}^{2}]
\]
is smooth and irreducible.
\end{example}
In the next chapter, we will see that those cases and the case $\mu_{2}\times\mu_{2}$
are the unique ones for which $\GCov$ has a similar description if
$G$ is diagonalizable (see \ref{pro:smooth DMCov} and \ref{cor:description of mutwotimesmutwo Cov}).
As a last example, we want to describe $\alpha_{p}$-covers. Remember
that $\alpha_{p}$ is the group scheme over $\F_{p}$ representing
the functor
\[
\alpha_{p}\colon(\Sch/\F_{p})^{op}\arr\sets\comma\alpha_{p}(X)=\{x\in\odi X\st x^{p}=0\}<\mathbb{G}_{a}(X)
\]
or, equivalently, the kernel of the Frobenius map $\mathbb{G}_{a}\arr\mathbb{G}_{a}$.
The result in this case is quite unexpected from the definition.
\begin{prop}
\label{prop:Balphap alphapCov}Let $p$ be a prime. We have an isomorphism
of $\F_{p}$-stacks 
\[
\Bi\alpha_{p}=\RCov{\alpha_{p}}\simeq[\A^{1}/\mathbb{G}_{a}]
\]
where the action of $\mathbb{G}_{a}$ on $\A^{1}$ is given by $\A^{1}\times\mathbb{G}_{a}\arr\A^{1}$,
$(x,y)\longmapsto x+y^{p}$. In particular every $\alpha_{p}$-cover
is an $\alpha_{p}$-torsor and $\RCov{\alpha_{p}}$ is smooth and
irreducible.\end{prop}
\begin{proof}
If $S$ is an $\F_{p}$-scheme an $\alpha_{p}$-action on a quasi-coherent
sheaf $\shF$ is given by a morphism $\gamma\colon\shF\arr\shF$ such
that $\gamma^{p}=0$ (see \cite[II, § 2, 2.7]{Demazure1970}). If
$(\shF,\gamma)$ and $(\shF',\gamma')$ are $\alpha_{p}$-equivariant
quasi-coherent sheaves, the representation $\shF\otimes\shF'$ is
given by $\gamma\otimes\id+\id\otimes\gamma'$. In particular if $\shF$
has an algebra structure, the multiplication is $\alpha_{p}$-equivariant
if and only if $\gamma$ is an $\odi S$-derivation. In conclusion
an $\alpha_{p}$-action on an affine scheme $X=\Spec\alA$ over $S$
is given by an $\odi S$-derivation $\partial\colon\alA\arr\alA$
such that $\partial^{p}=0$. Moreover it is easy to check that the
regular representation is $\odi S[\alpha_{p}]=\odi S[x]/(x^{p})$
with the usual derivation of polynomials.

We define the map $\phi\colon\A^{1}\arr\RCov{\alpha_{p}}$ induced
by the $\alpha_{p}$-cover over $\A^{1}=\Spec\F_{p}[z]$ given by
$A=\F_{p}[z,y]/(y^{p}-z)$ with the derivation $\partial/\partial y$.
This is an $\alpha_{p}$-torsor because the ring homomorphism 
\[
A[\alpha_{p}]=A[x]/(x^{p})\arr A\otimes_{\F_{p}[z]}A\simeq A[x]/(x^{p}-z)\text{ given by }x\longmapsto x-y
\]
is an $\alpha_{p}$-equivariant isomorphism. Moreover note that, if
$R$ is a ring and $a,b\in R$, then the $\alpha_{p}$-equivariant
isomorphisms 
\[
R[x]/(x^{p}-a)\arrdi{\psi}R[x]/(x^{p}-b)
\]
are all of the form $\psi(x)=x+u$, where $u\in R$ is such that $a=b+u^{p}$.
Therefore it remains to prove that $\phi$ is an epimorphism. The
question is local. So let $R$ be a ring and $A\in\RCov{\alpha_{p}}(R)$
with an $R$-derivation $\partial$ such that $A$, as $\alpha_{p}$-module,
is the regular representation. In particular there exists a basis
$y_{0},y_{1,}\dots,y_{p-1}$ such that $y_{0}=1$ and $\partial y_{i}=iy_{i-1}$,
where we have set $y_{-1}=0$. It is easy to prove by induction that
\[
y_{1}^{k}-y_{k}\in\la1,y_{1},\dots,y_{1}^{k-1}\ra_{R}\text{ for all }k=0,\dots,p-1
\]
In particular we can write $A\simeq R[y]/(y^{p}-f(y))$ with $\partial y=1$
and $\deg f<p$. Moreover the relation $\partial y^{p}=0=\partial f(y)$
tells us that $f\equiv b\in R$, as required.\end{proof}
\begin{example}
\label{ex: Romagny example} Let $k$ be a field of characteristic
$p>0$. We construct a cover $f\colon\A_{k}^{1}\arr\A_{k}^{1}$ and
actions of $\mu_{p}$ and $\alpha_{p}$ on $\A_{k}^{1}$ such that
$f$ is invariant, it is a torsor for both groups over $\mathbb{G}_{m,k}$,
but $f$ is not an $\alpha_{p}$-cover. This shows that for Galois
covers the acting group is not uniquely determined by the cover, as
it happens in the étale case. Moreover, the property of being a $G$-cover
is not closed in general, while this is true, as we will see, for
linearly reductive groups (see \ref{lem:G covers property closed for linearly reductive groups}).
This example has been suggested by Prof. Romagny.

As map $f$ consider the inclusion $k[x^{p}]\subseteq k[x]=A$. In
particular $A\simeq k[x^{p}][y]/(y^{p}-x^{p})$ and the action of
$\mu_{p}$ is given by setting $\deg y=1\in\Z/p\Z$. It is easy to
check by a direct computation that $f$ is a $\mu_{p}$-torsor over
$\Gm$. The right action of $\alpha_{p}$ on $\A_{k}^{1}$ is functorially
given by the expression
\[
z\star s=\frac{z}{1-sz}\text{ for }z\in\A_{k}^{1}(T)\comma s\in\alpha_{p}(T)\comma T\in\Sch/k
\]
Note that the expression $z\star s=z'$ for $z,z'\in\Gm(T)$ is equivalent
to $s=(z'-z)/zz'$ and such $s$ belongs to $\alpha_{p}(T)$ if and
only if $z^{p}=z'^{p}$, that is $f(z)=f(z')$. In particular $f$
is an $\alpha_{p}$-torsor over $\Gm$. The map $f$ is not an $\alpha_{p}$-cover,
or, equivalently, not an $\alpha_{p}$-torsor, because $0\star s=0$
for $0\in\A_{k}^{1}(T),s\in\alpha_{p}(T),T\in\Sch/k$.
\end{example}

\section{The main irreducible component $\stZ_{G}$.}

In this subsection we want to introduce what we will call the main
irreducible component of $\GCov$. In order to do that we recall what
is a schematic closure and some of its properties.
\begin{defn}
Let $\stX$ be an algebraic stack and $\stU\subseteq\stX$ be an open
substack. We will say that $\stU$ is schematically dense in $\stX$
if for any factorization
\[
\stU\arr\stZ\arrdi j\stX
\]
where $j$ is a closed immersion, $j$ is an isomorphism.
\end{defn}
Taking into account \cite[Theorem 11.10.5]{EGAIV-3} and extending
this result to algebraic stacks by taking an atlas, we get
\begin{prop}
\label{prop:properties of schematic closure}Let $\stX$ be an algebraic
stack and $\stU$ be a quasi-compact open substack. Then there exists
a minimum closed substack $\stZ$ of $\stX$ containing $\stU$ and
$\stU$ is schematically dense in $\stZ$. Moreover $\stZ$ is the
closed substack defined by the ideal $\Ker(\odi{\stX}\arr\odi{\stU})$
and $\stU$ is topologically dense in $\stZ$. Finally, if $f\colon\stX'\arr\stX$
is flat, then $\stX'\times_{\stX}\stU$ is schematically dense in
$\stX'\times_{\stX}\stZ$.\end{prop}
\begin{defn}
If $\stX$, $\stU$ and $\stZ$ are as in \ref{prop:properties of schematic closure},
we will call $\stZ$ the \emph{schematic closure }of $\stU$ in $\stX$.
\end{defn}

\begin{defn}
\label{def:the main component of GCov} We define the stack $\stZ_{G}$
as the schematic closure of $\Bi G$ in $\GCov$ and we will call
it the \emph{main irreducible component of $\GCov$.}
\end{defn}
Notice that, when the base scheme is irreducible, then $\stZ_{G}$
is an irreducible component because $\Bi G\subseteq\GCov$ is an irreducible
open substack. The formation of $\stZ_{G}$ commutes with flat base
changes of the base. Thus, if $G$ is defined over a field, $\stZ_{G}$
commutes with arbitrary base changes and it is therefore geometrically
integral.

\section{Bitorsors and Galois covers.}

It is a well known result that $G$-equivariant quasi-coherent sheaves
can also be thought of as quasi-coherent sheaves on the stack $\Bi G$.
In this section we want to use this point of view in order to show
examples of groups $G$ and $H$ for which $\GCov\simeq\HCov$. The
idea is that such an isomorphism can be defined as soon as we have
an isomorphism $\phi\colon\Bi G\arr\Bi H$, using the push-forward
$\phi_{*}$ of quasi-coherent sheaves. We will meet this situation
when we will study $(\mu_{3}\rtimes\Z/2\Z)$-covers and $S_{3}$-covers:
in this section we will prove that, over the ring $\Z[1/6]$, we have
isomorphisms
\[
\Bi(\mu_{3}\rtimes\Z/2\Z)\simeq\Bi S_{3}\text{ and }\RCov{(\mu_{3}\rtimes\Z/2\Z)}\simeq\RCov{S_{3}}
\]
The content of this section can also be found in \cite[Chapter III, Remarque 1.6.7]{Giraud1971}.
We start showing that isomorphisms $\Bi G\simeq\Bi H$ correspond
to what we will call $(G,H)$-bitorsors. This correspondence works
in great generality, that is as soon as we can talk about torsors
and we will present it from a categorical point of view. We refer
to \cite[Part 1]{Fantechi2007} for definitions and properties used
in this section. In what follows we consider given a site $\catC$,
that is a category endowed with a Grothendieck topology. We assume
that the site $\catC$ satisfies the following condition: given an
object $S\in\catC$ there exists a set of coverings $\stU$ of $S$
such that any covering of $S$ is refined by some covering in $\stU$.
This condition insures that any functor $\catC^{op}\arr\set$ can
be sheafified and it is satisfied by the site $\Sch/S$, where $S$
is a scheme, with the fppf topology or the étale topology, which is
the case in which we will apply the theory explained below.

We introduce now the concept of $(G,H)$-biactions and $(G,H)$-bitorsors.
\begin{defn}
\label{def:bitorsors}Let $G\colon\catC^{op}\arr\Grp$ be a sheaf
of groups. We will denote by $\Sh^{G}/\catC$ the fibered category
of sheaves over $\catC$ with a right $G$-action. The left (resp.
right) regular action of $G$ on itself is the left (resp. right)
action given by 
\[
G\times G\arr G\comma(g,h)\arr gh
\]
A left (resp. right) $G$-torsor is a sheaf $P\colon\catC^{op}\arr\set$
with a left action $G\times P\arr P$ (resp. right action $P\times G\arr P$)
for which $P$ is locally isomorphic to $G$ endowed with the left
(resp. right) regular action.

If $H$ is another sheaf of groups $\catC^{op}\arr\Grp$, a $(G,H)$-biaction
on a sheaf $P\colon\catC^{op}\arr\set$ is a pair $(G\times P\arrdi uP,P\times H\arrdi vP)$
where $u$ and $v$ are, respectively, a left $G$-action and a right
$H$-action on $P$, such that the following diagram is commutative.
  \[   \begin{tikzpicture}[xscale=3.0,yscale=-1.2]     \node (A0_0) at (0, 0) {$G\times P \times H$};     \node (A0_1) at (1, 0) {$P\times H$};     \node (A1_0) at (0, 1) {$G\times P$};     \node (A1_1) at (1, 1) {$P$};     \path (A0_0) edge [->]node [auto] {$\scriptstyle{u\times \id_H}$} (A0_1);     \path (A0_0) edge [->]node [auto] {$\scriptstyle{\id_G \times v}$} (A1_0);     \path (A0_1) edge [->]node [auto] {$\scriptstyle{v}$} (A1_1);     \path (A1_0) edge [->]node [auto] {$\scriptstyle{u}$} (A1_1);   \end{tikzpicture}   \] 
A $(G,H)$-bitorsor is a sheaf $P\colon\catC^{op}\arr\set$ with a
$(G,H)$-biaction for which $P$ is both a left $G$-torsor and a
right $H$-torsor. Denote by $\Bi(G,H)$ the fibered category over
$\catC$ of $(G,H)$-bitorsors.\end{defn}
\begin{rem}
Note that the right regular representation introduced above differs
from the one introduced in \ref{def:regular representation}. The
above definition will be used only in this section because it will
simplify the exposition. On the other hand, since the two actions
are isomorphic, it is clear that the results obtained below are independent
of the choice of the regular representation to use.
\end{rem}

\begin{rem}
The fibered category $\Bi(G,H)$ is a stack over $\catC$. This is
easy to prove directly, using the fact that $(\Sh/\catC)$, the fibered
category of sheaves of sets over $\catC$, is a stack (see \cite[Part 1, Example 4.11]{Fantechi2007}).
Otherwise this can be seen as consequence of the isomorphism proved
in \ref{prop:bitorsors and isomorphisms BG - BH}.
\end{rem}

\begin{rem}
Given a left $G$-action $u\colon P\times G\arr P$ and a right $H$-action
$v\colon P\times H\arr P$ , the pair $(u,v)$ is a $(G,H)$-biaction
if and only if the homomorphism $G\arr\Autsh P$ induced by $u$ factors
through $\Autsh^{H}P$, that is if $G$ acts through $H$-equivariant
isomorphisms.
\end{rem}

\begin{rem}
Let $G$ be a sheaf of groups over $\catC$. If $X\in(\Sh^{G}/\catC)$
there always exists a categorical quotient $X\arr X/G$, that is a
map through which all the $G$-invariant maps $X\arr Y$, where $Y$
is a sheaf, uniquely factor. Indeed the quotient $X/G$ is the sheafification
of the functor $\catC^{op}\arr\set$ that associates with an object
$S$ the set $X(S)/G(S)$. In particular, if $H$ is another sheaf
of groups, $T\in\catC$ and $P\in\Bi(G,H)(T)$, $X\times P$ has a
right action of $G$ given by $(x,p)g=(xg,g^{-1}p)$ and we can consider
the quotient $(X\times P)/G$, which has a right $H$-action induced
by the one of $P$.\end{rem}
\begin{prop}
\label{prop:bitorsors and isomorphisms SHG - SHH}Let $G$ and $H$
be sheaves of groups over $\catC$. Then the association   \[   \begin{tikzpicture}[xscale=5.0,yscale=-0.6]     \node (A0_0) at (0, 0) {$\Bi(G,H)$};     \node (A0_1) at (1, 0) {$\Isosh_{\catC}((\Sh^G/\catC),(\Sh^H/\catC))$};     \node (A1_0) at (0, 1) {$P$};     \node (A1_1) at (1, 1) {$(X\longmapsto (X\times P)/G)$};     \path (A0_0) edge [->]node [auto] {$\scriptstyle{\Lambda}$} (A0_1);     \path (A1_0) edge [|->,gray]node [auto] {$\scriptstyle{}$} (A1_1);   \end{tikzpicture}   \] 
is a functor of fibered categories. If $T\in\catC$, $X\in(\Sh^{G}/\catC)(T)$,
$G\times T\simeq H\times T$ and $P$ is a trivial $(G,H)$-bitorsor
over $T$, then there exists a natural isomorphism $\Lambda_{P}(X)\simeq X$
as objects of $\catC$.
\end{prop}
The proof of the above statement is not difficult and left to the
reader.
\begin{prop}
\label{prop:bitorsors and isomorphisms BG - BH}The functor $\Lambda$
of \ref{prop:bitorsors and isomorphisms SHG - SHH} induces an isomorphism
$\Bi(G,H)\arr\Isosh(\Bi G,\Bi H)$ whose inverse is given by $\phi\longmapsto\phi(G)$,
where the left $G$-action on $\phi(G)$ is given by $G\simeq\Aut^{G}G\simeq\Aut^{H}\phi(G)$.
In particular $(G,H)$-bitorsors are $H$-torsors $P$ with an isomorphism
$G\simeq\Autsh^{H}P$.\end{prop}
\begin{proof}
The only non trivial point is showing the existence of an isomorphism
$\Lambda_{\phi(G)}\simeq\phi$. This is induced by the map
\[
Q\times\phi(G)\simeq\Homsh^{G}(G,Q)\times\phi(G)\arr\phi(Q)
\]
for $Q\in(\Sh^{G}/\catC)$, functorially in $Q$.\end{proof}
\begin{cor}
Let $G$ be a sheaf of groups over $\catC$. Then the sheaves of groups
$H$ for which there exists an isomorphism $\Bi H\simeq\Bi G$ are
the sheaves $\Autsh^{G}P$ for $P\in\Bi G$.
\end{cor}
We now want to describe two examples of non trivial bitorsors.
\begin{example}
If $G$ and $H$ are sheaves of groups of $\catC$ set $P=\Isosh(G,H)$.
The maps
\[
\Autsh G\times P\arr P\comma P\times\Autsh(H)\arr P\text{, both given by }(\phi,\psi)\longmapsto\phi\circ\psi
\]
induce a $(\Autsh G,\Autsh H)$-action on $\Isosh(G,H)$ and, if $G$
and $H$ are locally isomorphic, then $\Isosh(G,H)$ is a $(\Autsh G,\Autsh H)$-bitorsor.
In particular, in this case, we obtain an isomorphism
\[
\Bi\Autsh(G)\simeq\Bi\Autsh(H)
\]

\end{example}
The second bitorsor we want to describe is a refinement of the previous
one.
\begin{prop}
\label{prop:bitorsors and semidirect products}Let $G$ and $H$ be
sheaves of groups $\catC^{op}\arr\Grp$ and set $P=G\times\Isosh(H,G)$.
The maps
\[
P\times(H\rtimes\Autsh H)\arr P\comma(G\rtimes\Autsh G)\times P\arr P\text{, both given by }(x,\phi)\cdot(y,\psi)=(x\phi(y),\phi\psi)
\]
define a $((G\rtimes\Autsh G),(H\rtimes\Autsh H))$-action on $P$
and, if $G$ and $H$ are locally isomorphic, then $P$ is a $((G\rtimes\Autsh G),(H\rtimes\Autsh H))$-bitorsor.
In particular, in this case, we have an isomorphism
\[
\Bi(G\rtimes\Autsh G)\simeq\Bi(H\rtimes\Autsh H)
\]
and, if $\Lambda_{P}\colon(\Sh^{G\rtimes\Autsh G}/\catC)\arr(\Sh^{H\rtimes\Autsh H}/\catC)$
is the functor defined in \ref{prop:bitorsors and isomorphisms SHG - SHH},
we have a canonical isomorphism
\[
(X/\Autsh G)\simeq(\Lambda_{P}(X)/\Autsh H)\text{ for all }X\in(\Sh^{G\rtimes\Autsh G}/\catC)
\]
\end{prop}
\begin{proof}
A direct computation shows that the maps in the statement yield compatible
actions. Moreover, if $\gamma\colon G\arr H$ is an isomorphism, it
is also straightforward to check that the maps $g\longmapsto g\cdot\gamma$
and $h\longmapsto\gamma\cdot h$ are equivariant isomorphisms $G\rtimes\Autsh G\arr P$
and $H\rtimes\Autsh H\arr P$ respectively. Finally consider the map
\[
\pi\colon X\times P=X\times G\times\Isosh(H,G)\arr X\text{ given by }\pi(x,g,\phi)=x(g,\id_{G})
\]
It is easy to check that $\pi(z(u,\psi))=\pi(z)(1_{G},\psi)$ and
$\pi(z(1_{H},\delta))=\pi(z)$ for all $(u,\psi)\in G\rtimes\Autsh G$
and $(1_{H},\delta)\in H\rtimes\Autsh H$. In particular $\pi$ yields
a map $(\Lambda_{P}(X)/\Autsh H)\arr(X/\Autsh G)$. This is an isomorphism
since it is so locally, i.e. when we have an isomorphism $H\arrdi{\phi}G$:
in this case the inverse is given by $x\arr(x,1_{G},\phi)$.\end{proof}
\begin{example}
\label{ex:SThree covers and the other group}Consider the group $G=\mu_{n}$
and $H=\Z/n\Z$ over $\Z[1/n]$: they are étale locally isomorphic
and therefore
\[
\Bi(\mu_{n}\rtimes(\Z/n\Z)^{*})\simeq\Bi(\Z/n\Z\rtimes(\Z/n\Z)^{*})
\]
In particular for $n=3$ we get $\Bi(\mu_{3}\rtimes\Z/2\Z)\simeq\Bi S_{3}$.
\end{example}
This is the connection with the theory of Galois covers.
\begin{thm}
\label{thm:bitorsors and GCov}Let $G$ and $H$ be flat, finite and
finitely presented group schemes over a base scheme $S$. If $P$
is a fppf $(G,H)$-bitorsor over $S$, the functor $\Lambda_{P}$
of \ref{prop:bitorsors and isomorphisms SHG - SHH} induces an isomorphism
  \[   \begin{tikzpicture}[xscale=3.0,yscale=-0.6]     \node (A1_0) at (0, 1) {$\GCov$};     \node (A1_1) at (1, 1) {$\HCov$};     \node (A2_0) at (0, 2) {$X$};     \node (A2_1) at (1, 2) {$(X\times P)/G$};     \path (A1_0) edge [->]node [auto] {$\scriptstyle{\Lambda_P}$} (A1_1);     \path (A2_0) edge [|->,gray]node [auto] {$\scriptstyle{}$} (A2_1);   \end{tikzpicture}   \] In
particular, given an $S$-scheme $Y$ and $X\in\GCov(Y)$, the cover
$X\arr Y$ (resp. the $S$-scheme $X$) is fppf locally isomorphic
to the cover $\Lambda_{P}(X)\arr Y$ (resp. the $S$-scheme $\Lambda_{P}(X)$)
and therefore they share all the properties that are local and satisfy
descent in the fppf topology. Moreover if $G$ or $H$ is étale, the
same conclusion follows for the étale topology.\end{thm}
\begin{proof}
It is enough to note that, if $X\in\GCov(Y)$ and $P$ is trivial,
a section of $P$ induces an isomorphism $\mu\colon H\arr G$ and
$\Lambda_{P}(X)$ is isomorphic to $X$ with the $H$-action obtained
from its $G$-action through $\mu$.\end{proof}
\begin{rem}
Let $G$ and $H$ be affine group schemes with an isomorphism $\phi\colon\Bi G\arr\Bi H$.
Since $\QCoh_{S}^{G}\simeq\QCoh_{\Bi G}$ and similarly for $H$,
the pushforward $\phi_{*}$ induces an isomorphism $\QCoh_{S}^{G}\simeq\QCoh_{S}^{H}$.
If $\phi$ corresponds to the $(G,H)$-bitorsor $P=\Spec\alA_{P}$
over $S$ it is possible to check that we have an isomorphism
\[
\phi_{*}\shF\simeq(\shF\otimes\alA_{P})^{G}
\]
where the $H$-action is induced by the one over $\alA_{P}$.
\end{rem}

\chapter{Galois Covers under diagonalizable group schemes.}

The aim of this chapter is to study the theory of $G$-covers in the
diagonalizable case. Let $G$ be a finite diagonalizable group scheme
over $\Z$. We now briefly summarize how this chapter is divided.

\emph{Section 1. }The stack $\GCov$ and some of its substacks, like
$\stZ_{G}$ and $\Bi G$, share a common structure, i.e. they are
all of the form $\stX_{\phi}=[\Spec\Z[T_{+}]/\shT]$, where $T_{+}$
is a finitely generated commutative monoid whose associated group
is free of finite rank, $\shT$ is a torus over $\Z$ and $\phi\colon T_{+}\arr\Z^{r}$
is an additive map, that induces the action of $\shT$ on $\Spec\Z[T_{+}]$.
The first section will be dedicated to the study of such stacks. As
we will see many facts about $\GCov$ are just applications of general
results about such stacks. For instance the existence of a special
irreducible component $\stZ_{\phi}$ of $\stX_{\phi}$ as well as
the use of $\duale T_{+}=\Hom(T_{+},\N)$ for the study of the smooth
locus of $\stZ_{\phi}$ are properties that can be stated in this
setting.

\emph{Section 2.} We will explain how $\GCov$ can be viewed as a
stack of the form $\stX_{\phi}$ and how it is related to the equivariant
Hilbert schemes. Then we will study the properties of connectedness,
irreducibility and smoothness for $\GCov$. Finally we will introduce
the stratification $\Bi G=U_{0}\subseteq U_{1}\subseteq\cdots\subseteq U_{|G|-1}=\GCov$
and we will characterize the locus $U_{1}$. 

\emph{Section 3. }We will study the locus $U_{2}$ and $G$-covers
whose total space is normal crossing in codimension $1$.

\section{\label{sec:stack Xphi}The stack $\stX_{\phi}$.}

In the following sections we will study the stack $\GCov$ when $G=\Di M$,
the diagonalizable group of a finite abelian group $M$. The structure
of this stack and of some of its substacks is somehow special and
in this section we will provide general constructions and properties
that will be used later. To a monoid map $T_{+}\arrdi{\phi}\Z^{r}$,
we will associate a stack $\stX_{\phi}$ whose objects are sequences
of invertible sheaves with additional data and we will study particular
'parametrization' of these objects, defined by a map of stacks $\stF_{\underline{\E}}\arrdi{\pi_{\underline{\E}}}\stX_{\phi}$,
where $\stF_{\underline{\E}}$ will be a 'nice' stack, for instance
smooth.

In this section we will consider given a commutative monoid $T_{+}$
together to a monoid map $\phi\colon T_{+}\arr\Z^{r}$. 
\begin{defn}
We define the stack $\stX_{\phi}$ over $\Z$ as follows.
\begin{itemize}
\item \emph{Objects. }An object over a scheme $S$ is a pair $(\underline{\shL},a)$
where:

\begin{itemize}
\item $\underline{\shL}=\shL_{1},\dots,\shL_{r}$ are invertible sheaves
on $S$;
\item $T_{+}\arrdi a\ISym\underline{\shL}$ is an additive map such that
$a(t)\in\underline{\shL}^{\phi(t)}$ for any $t\in T_{+}$.
\end{itemize}
\item \emph{Arrows. }An isomorphism $(\underline{\shL},a)\arrdi{\underline{\sigma}}(\underline{\shL}',a')$
of objects over $S$ is given by a sequence $\underline{\sigma}=\sigma_{1},\dots,\sigma_{r}$
of isomorphisms $\sigma_{i}\colon\shL_{i}\arrdi{\simeq}\shL_{i}'$
such that 
\[
\underline{\sigma}^{\phi(t)}(a(t))=a'(t)\text{ for any }t\in T_{+}
\]

\end{itemize}
\end{defn}
\begin{example}
Let $f_{1},\dots,f_{s},g_{1},\dots,g_{t}\in\Z^{r}$ and consider the
stack $\stX_{\underline{f},\underline{g}}$ of invertible sheaves
$\shL_{1},\dots,\shL_{r}$ with maps $\odi{}\arr\underline{\shL}^{f_{i}}\quad\text{and}\quad\odi{}\arrdi{\simeq}\underline{\shL}^{g_{j}}$.
If $T_{+}=\N^{s}\times\Z^{t}$ and $\phi\colon T_{+}\arr\Z^{r}$ is
the map given by the matrix $(f_{1}|\cdots|f_{s}|g_{1}|\cdots|g_{t})$
then $\stX_{\underline{f},\underline{g}}=\stX_{\phi}$.\end{example}
\begin{notation}
We set
\[
\Z[T_{+}]=\Z[x_{t}]_{t\in T_{+}}/(x_{t}x_{t'}-x_{t+t'},x_{0}-1)
\]
 and $\odi S[T_{+}]=\Z[T_{+}]\otimes_{\Z}\odi S$. The scheme $\Spec\odi S[T_{+}]$
over $S$ represents the functor that associates to any scheme $U/S$
the set of additive maps $T_{+}\arr(\odi U,\cdot)$, where $\cdot$
denotes the multiplication on $\odi U$. The group $\Di{\Z^{r}}$
acts on $\Spec\Z[T_{+}]$ by the graduation $\deg x_{t}=\phi(t)$.\end{notation}
\begin{prop}
\label{pro:atlas for the stack associated to a monoid map}Set $X=\Spec\Z[T_{+}]$.
The choice $\shL_{i}=\odi X$ and   \[   \begin{tikzpicture}[xscale=1.8,yscale=-0.6]     \node (A0_0) at (0, 0) {$\underline{\shL}^{\phi(t)}$};     \node (A0_1) at (1, 0) {$\odi{X}$};     \node (A1_0) at (0, 1) {$a(t)$};     \node (A1_1) at (1, 1) {$x_t$};     \path (A0_0) edge [->] node [auto] {$\scriptstyle{\simeq}$} (A0_1);     \path (A1_0) edge [<->,gray] node [auto] {$\scriptstyle{}$} (A1_1);   \end{tikzpicture}   \] induces
a smooth epimorphism $X\arr\stX_{\phi}$ such that $\stX_{\phi}\simeq[X/\Di{\Z^{r}}]$.
In particular $\stX_{\phi}$ is an algebraic stack.\end{prop}
\begin{proof}
It is enough to note that an object of $[X/\Di{\Z^{r}}](U)$ is given
by invertible sheaves $\shL_{1},\dots,\shL_{r}$ with a $\Di{\Z^{r}}$-equivariant
map $\Spec\ISym\underline{\shL}\arr\Spec\Z[T_{+}]$ which exactly
corresponds to an additive map $T_{+}\arr\ISym\underline{\shL}$ as
in the definition of $\stX_{\phi}$. It is easy to check that the
map $X\arr[X/\Di{\Z^{r}}]\arr\stX_{\phi}$ is the one defined in the
statement.\end{proof}
\begin{rem}
\label{rem:description of isomorphism for local objects of X phi}Given
a map $U\arrdi aX=\Spec\Z[T_{+}]$, i.e. a monoid map $T_{+}\arrdi a\odi U$,
the induced object $U\arrdi aX\arr\stX_{\phi}$ is the pair $(\underline{\shL},\tilde{a})$
where $\shL_{i}=\odi U$ and for any $t\in T_{+}$   \[   \begin{tikzpicture}[xscale=1.8,yscale=-0.5]     \node (A0_0) at (0, 0) {$\odi{U}$};     \node (A0_1) at (1, 0) {$\underline{\shL}^{\phi(t)}$};     \node (A1_0) at (0, 1) {$a(t)$};     \node (A1_1) at (1, 1) {$\tilde{a}(t)$};     \path (A0_0) edge [->]node [auto] {$\scriptstyle{\simeq}$} (A0_1);     \path (A1_0) edge [|->,gray]node [auto] {$\scriptstyle{}$} (A1_1);   \end{tikzpicture}   \]
We will denote by $a$ also the object $(\underline{\shL},\tilde{a})\in\stX_{\phi}(U)$.

Given two elements $a,b\colon T_{+}\arr\odi U\in\stX_{\phi}(U)$ we
have
\[
\Iso_{\stX_{\phi}(U)}(a,b)=\{\sigma_{1},\dots,\sigma_{r}\in\odi U^{*}\;|\;\underline{\sigma}^{\phi(t)}a(t)=b(t)\;\forall t\in T_{+}\}
\]
\end{rem}
\begin{lem}
\label{lem:morphisms of stack Xphi} Consider a commutative diagram
  \[   \begin{tikzpicture}[xscale=1.7,yscale=-1.2]     \node (A0_0) at (0, 0) {$T_+$};     \node (A0_1) at (1, 0) {$T_+'$};     \node (A1_0) at (0, 1) {$\Z^r$};     \node (A1_1) at (1, 1) {$\Z^s$};     \path (A0_0) edge [->]node [auto] {$\scriptstyle{h}$} (A0_1);     \path (A0_0) edge [->]node [auto,swap] {$\scriptstyle{\phi}$} (A1_0);     \path (A0_1) edge [->]node [auto] {$\scriptstyle{\psi}$} (A1_1);     \path (A1_0) edge [->]node [auto] {$\scriptstyle{g}$} (A1_1);   \end{tikzpicture}   \] 
where $T{}_{+},T{}_{+}'$ are commutative  monoids and $\phi\comma\psi\comma h\comma g$
are additive maps. Then we have a $2$-commutative diagram \begin{align}\label{diag:commutativity fro stack X}
\begin{tikzpicture}[xscale=3.7,yscale=-0.7]     \node (A0_0) at (0, 0) {$\Spec \Z[T_+']$};     \node (A0_1) at (1, 0) {$\Spec \Z[T_+]$};     \node (A2_0) at (0, 2) {$\stX_\psi$};     \node (A2_1) at (1, 2) {$\stX_\phi$};     \node (A3_0) at (0, 3) {$(\underline{\shL},T_+'\arrdi{a}\ISym \underline{\shL})$};     \node (A3_1) at (1, 3) {$(\underline{\shM},T_+ \arrdi{b}\ISym \underline{\shM})$};     \path (A0_0) edge [->] node [auto] {$\scriptstyle{h^*}$} (A0_1);     \path (A0_1) edge [->] node [auto] {$\scriptstyle{}$} (A2_1);     \path (A0_0) edge [->] node [auto] {$\scriptstyle{}$} (A2_0);     \path (A2_0) edge [->] node [auto] {$\scriptstyle{\Lambda}$} (A2_1);     \path (A3_0) edge [|->,gray] node [auto] {$\scriptstyle{}$} (A3_1);   \end{tikzpicture}   
\end{align} where, for $i=1,\dots,r$, $\shM_{i}=\underline{\shL}^{g(e_{i})}$
and $b$ is the unique map such that \begin{center}
  \[   \begin{tikzpicture}[xscale=1.2,yscale=-1.2]     
\node (A0_0) at (0, 0) {$T_+$};     
\node (A0_1) at (1, 0) {};     
\node (A0_2) at (2, 0) {$\ISym \underline{\shM}$};     
\node (A0_3) at (3, 0) {$\underline{\shM}^v$};     
\node (A0_4) at (4, 0) {$\underline{\shL}^{g(v)}$};     
\node (A1_0) at (0, 1) {$T_+'$};     
\node (A1_2) at (2, 1) {$\ISym \underline{\shL}$};     
\node (A1_3) at (3, 1) {$\underline{\shL}^{g(v)}$};    
\node (simeq) at (3.43, 0) {$\simeq$}; 
\path (A1_0) edge [->] node [auto,swap] {$\scriptstyle{a}$} (A1_2);     
\path (A0_3) edge [->] node [auto] {$\scriptstyle{}$} (A1_3);     
\path (A0_2) edge [->] node [auto] {$\scriptstyle{}$} (A1_2);     
\path (A0_4) edge [->] node [auto] {$\scriptstyle{\id}$} (A1_3);     
\path (A0_0) edge [->] node [auto] {$\scriptstyle{b}$} (A0_2);     
\path (A0_0) edge [->] node [auto,swap] {$\scriptstyle{h}$} (A1_0);   
\end{tikzpicture}   \] 
\par\end{center}\end{lem}
\begin{proof}
An easy computation shows that there is a canonical isomorphism $\underline{\shM}^{v}\simeq\underline{\shL}^{g(v)}$
for all $v\in\Z^{r}$ and so $b(t)$ corresponds under this isomorphism
to $a(h(t))\in\underline{\shL}^{\psi(h(t))}=\underline{\shL}^{g(\phi(t))}\simeq\underline{\shM}^{\phi(t)}$.
So the functor $\Lambda$ is well defined and we have only to check
the commutativity of the second diagram in the statement. The map
$\Spec\Z[T{}_{+}']\arr\Spec\Z[T{}_{+}]\arr\stX_{\phi}$ is given by
trivial invertible sheaves and the additive map   \[   \begin{tikzpicture}[xscale=3.6,yscale=-0.7]     \node (A0_0) at (0.3, 0) {$T_+$};     \node (A0_1) at (1, 0) {$\Z[T_+][x_1,\dots,x_r]_{\prod_{i}x_{i}}$};     \node (A0_2) at (2, 0) {$\Z[T_+'][x_1,\dots,x_r]_{\prod_{i}x_{i}}$};     \node (A1_0) at (0.3, 1) {$t$};     \node (A1_1) at (1, 1) {$x_t x^{\phi(t)}$};     \node (A1_2) at (2, 1) {$x_{h(t)} x^{\phi(t)}$};     \path (A0_0) edge [->] node [auto] {$\scriptstyle{}$} (A0_1);     \path (A1_0) edge [|->,gray] node [auto] {$\scriptstyle{}$} (A1_1);     \path (A0_1) edge [->] node [auto] {$\scriptstyle{}$} (A0_2);     \path (A1_1) edge [|->,gray] node [auto] {$\scriptstyle{}$} (A1_2);   \end{tikzpicture}   \] 
Instead the map $\Spec\Z[T{}_{+}']\arr\stX_{\psi}\arr\stX_{\phi}$
is given by trivial invertible sheaves and the map $b$ that makes
the following diagram commutative   \[   \begin{tikzpicture}[xscale=2.0,yscale=-0.7]     \node (A0_0) at (0, 0) {$T_+$};     \node (A0_2) at (2, 0) {$\Z[T_+'][x_1,\dots,x_r]_{\prod_{i}x_{i}}$};     \node (A0_3) at (3, 0) {$x^v$};     \node (A2_0) at (0, 2) {$T_+'$};     \node (A2_2) at (2, 2) {$\Z[T_+'][y_1,\dots,y_s]_{\prod_{i}y_{i}}$};     \node (A2_3) at (3, 2) {$y^{g(v)}$};     \node (A3_0) at (0, 3) {$t$};     \node (A3_2) at (2, 3) {$x_t y^{\psi(t)}$};     \path (A0_3) edge [|->,gray] node [auto] {$\scriptstyle{}$} (A2_3);     \path (A0_0) edge [->] node [auto] {$\scriptstyle{b}$} (A0_2);     \path (A0_2) edge [->] node [auto] {$\scriptstyle{}$} (A2_2);     \path (A2_0) edge [->] node [auto] {$\scriptstyle{a}$} (A2_2);     \path (A0_0) edge [->] node [auto,swap] {$\scriptstyle{h}$} (A2_0);     \path (A3_0) edge [|->,gray] node [auto] {$\scriptstyle{}$} (A3_2);   \end{tikzpicture}   \] 
Since $x_{h(t)}x^{\phi(t)}$ is sent to $x_{h(t)}y^{g(\phi(t))}=x_{h(t)}y^{\psi(h(t))}=a(h(t))$
we find again $b(t)=x_{h(t)}x^{\phi(t)}$.\end{proof}
\begin{rem}
\label{rem: description of functors of Xphi on local objects}The
functor $\stX_{\psi}\arr\stX_{\phi}$ sends an element $a\colon T_{+}'\arr\odi U\in\stX_{\psi}(U)$
to the element $a\circ h\in\stX_{\phi}(U)$. Moreover, taking into
account the description given in \ref{rem:description of isomorphism for local objects of X phi},
if $a,b\colon T_{+}'\arr\odi U\in\stX_{\psi}(U)$ we have   \[   \begin{tikzpicture}[xscale=3.0,yscale=-0.6]     \node (A0_0) at (0, 0) {$\Iso_U(a,b)$};     \node (A0_1) at (1, 0) {$\Iso_U(a\circ h,b\circ h)$};     \node (A1_0) at (0, 1) {$\underline\sigma$};     \node (A1_1) at (1, 1) {$\underline{\sigma}^{g(e_1)},\dots,\underline{\sigma}^{g(e_r)}$};     \path (A0_0) edge [->] node [auto] {$\scriptstyle{}$} (A0_1);     \path (A1_0) edge [|->,gray] node [auto] {$\scriptstyle{}$} (A1_1);   \end{tikzpicture}   \]

\end{rem}

\subsection{The main irreducible component $\stZ_{\phi}$ of $\stX_{\phi}$.}
\begin{notation}
\label{not:notation for a monoid}A monoid will be called \emph{integral}
if it satisfies the cancellation law, i.e. 
\[
\forall a,b,c\comma\quad a+b=a+c\then b=c
\]
Let $T_{+}$ be a monoid. There exists, up to a unique isomorphism,
a group $T$ (resp. integral monoid $T_{+}^{int}$) such that any
monoid map $T_{+}\arr S_{+}$, where $S_{+}$ is a group (resp. integral
monoid), factors uniquely through $T$ (resp. $T_{+}^{int}$). We
call it the \emph{associated group} (resp. \emph{associated integral
monoid}) of $T_{+}$. Notice that if $T$ is the associated group
of $T_{+}$, then $\Imm(T_{+}\arr T)$ can be chosen as the associated
integral monoid of $T_{+}$. We will continue to denote by $T$ the
associated group of $T_{+}$ and we set $T_{+}^{int}=\Imm(T_{+}\arr T)\subseteq T$.
In particular $\langle T_{+}^{int}\rangle_{\Z}=T$.

From now on $T_{+}$ will be a finitely generated monoid whose associated
group is a free $\Z$-module of finite rank. In order to simplify
notation, we will often write $\phi\colon T\arr\Z^{r}$, meaning the
extension of $\phi\colon T_{+}\arr\Z^{r}$ to $T$. Anyway, the stack
$\stX_{\phi}$ will always be the stack $\stX_{T_{+}\arr\Z^{r}}$
and when we will have to consider the stack $\stX_{T\arr\Z^{r}}$,
we will always specify a different symbol for the induced map $T\arr\Z$.\end{notation}
\begin{rem}
\label{lem:the domain monoid and group monoid associated to T +: ring}If
$D$ is a domain, then $\Spec D[T]$ is an open subscheme of $\Spec D[T_{+}]$,
while $\Spec D[T_{+}^{int}]$ is one of its irreducible components.
In particular we have\end{rem}
\begin{prop}
\label{cor:the domain monoid and group monoid associated to T +: stack}Let
$\hat{\phi}\colon T\arr\Z^{r}$ be the extension of $\phi$ and set
$\phi^{int}=\hat{\phi}_{|T_{+}^{int}}$. Then $\stB_{\phi}=\stX_{\hat{\phi}}\arr\stX_{\phi}$
is an open immersion, while $\stZ_{\phi}=\stX_{\phi^{int}}\arr\stX_{\phi}$
is a closed one. Moreover $\stZ_{\phi}$ is the reduced closed stack
associated to the closure of $\stB_{\phi}$, it is an irreducible
component of $\stX_{\phi}$ and 
\[
\stB_{\phi}\simeq[\Spec\Z[T]/\Di{\Z^{r}}]\text{ and }\stZ_{\phi}\simeq[\Spec\Z[T_{+}^{int}]/\Di{\Z^{r}}]
\]
\end{prop}
\begin{defn}
With notation above we will call respectively $\stB_{\phi}$ and $\stZ_{\phi}$
the \emph{principal open substack} and the \emph{main irreducible
component} of $\stX_{\phi}$.\end{defn}
\begin{notation}
We set
\[
\duale T_{+}=\Hom(T_{+},\N)=\{\E\in\Hom_{\text{groups}}(T,\Z)\:|\:\E(T_{+})\subseteq\N\}
\]
We will call it the \emph{dual monoid} of $T_{+}$ and we will call
its elements the\emph{ rays }for $T_{+}$. Note that $\duale T_{+}=\duale{T_{+}^{int}}$.
Given $\underline{\E}=\E^{1},\dots,\E^{s}\in\duale T_{+}$ we will
denote by $\underline{\E}$ also the induced map $T\arr\Z^{s}$. Moreover
we set
\[
\Supp\underline{\E}=\{v\in T_{+}\;|\;\exists i\;\E^{i}(v)>0\}
\]
Finally notice that the dual monoid of a group is always $0$. Therefore,
when $H$ is an abelian group, the dual $\duale H$ of $H$ will always
be the dual as $\Z$-module.\end{notation}
\begin{defn}
Given a sequence $\underline{\E}=\E^{1},\dots,\E^{s}\in\duale T_{+}$
set   \[   \begin{tikzpicture}[xscale=3.3,yscale=-0.6]     \node (A0_0) at (0, 0) {$\N^s \oplus T$};     \node (A0_1) at (1, 0) {$\Z^s \oplus \Z^r$};     \node (A1_0) at (0, 1) {$(e_i,0)$};     \node (A1_1) at (1, 1) {$(e_i,0)$};     \node (A2_0) at (0, 2) {$(0,t)$};     \node (A2_1) at (1, 2) {$\displaystyle(\underline\E(t),- \phi(t))$};     \path (A0_0) edge [->] node [auto] {$\scriptstyle{\sigma_{\underline \E}}$} (A0_1);     \path (A1_0) edge [|->,gray] node [auto] {$\scriptstyle{}$} (A1_1);     \path (A2_0) edge [|->,gray] node [auto] {$\scriptstyle{}$} (A2_1);   \end{tikzpicture}   \] where
$e_{1},\dots,e_{s}$ is the canonical basis of $\Z^{s}$. We will
call $\stF_{\underline{\E}}=\stX_{\sigma_{\underline{\E}}}$.\end{defn}
\begin{rem}
\label{rem:description of objects of FE}An object of $\stF_{\underline{\E}}$
over a scheme $U$ is given by a sequence $(\underline{\shL},\underline{\shM},\underline{z},\lambda)$
where:
\begin{itemize}
\item $\underline{\shL}=\shL_{1},\dots,\shL_{r}$ and $\underline{\shM}=(\shM_{\E})_{\E\in\underline{\E}}=\shM_{1},\dots,\shM_{s}$
are invertible sheaves on $U$;
\item $\underline{z}=(z_{\E})_{\E\in\underline{\E}}=z_{1},\dots,z_{s}$
are sections $z_{i}\in\shM_{i}$;
\item for any $t\in T$, $\lambda(t)=\lambda_{t}$ is an isomorphism $\underline{\shL}^{\phi(t)}\arrdi{\simeq}\underline{\shM}^{\underline{\E}(t)}$
additive in $t$.
\end{itemize}
An isomorphism $(\underline{\shL},\underline{\shM},\underline{z},\lambda)\arr(\underline{\shL}',\underline{\shM}',\underline{z}',\lambda')$
is a pair $(\underline{\omega},\underline{\tau})$ where $\underline{\omega}=\omega_{1},\dots,\omega_{r}\comma\underline{\tau}=\tau_{1},\dots,\tau_{s}$
are sequences of isomorphisms $\shL_{i}\arrdi{\omega_{i}}\shL_{i}'\comma\shM_{j}\arrdi{\tau_{j}}\shM_{j}'$
such that $\tau_{j}(z_{j})=z_{j}'$ and for any $t\in T$ we have
a commutative diagram   \[   \begin{tikzpicture}[xscale=1.9,yscale=-1.2]     \node (A0_0) at (0, 0) {$\underline{\shL}^{\phi(t)}$};     \node (A0_1) at (1, 0) {$\underline{\shM}^{\E(t)}$};     \node (A1_0) at (0, 1) {$\underline{\shL}'^{\phi(t)}$};     \node (A1_1) at (1, 1) {$\underline{\shM}'^{\E(t)}$};     \path (A0_0) edge [->]node [auto] {$\scriptstyle{\lambda_t}$} (A0_1);     \path (A1_0) edge [->]node [auto] {$\scriptstyle{\lambda_t'}$} (A1_1);     \path (A0_1) edge [->]node [auto] {$\scriptstyle{\underline{\tau}^{\phi(t)}}$} (A1_1);     \path (A0_0) edge [->]node [auto,swap] {$\scriptstyle{\underline{\omega}^{\phi(t)}}$} (A1_0);   \end{tikzpicture}   \] An
object over $U$ coming from the atlas $\Spec\Z[\N^{s}\oplus T]$
is a pair $(\underline{z},\lambda)$ where $\underline{z}=z_{1},\dots,z_{s}\in\odi U$
and $\lambda\colon T\arr\odi U^{*}$ is a group homomorphism. Given
$(\underline{z},\lambda),(\underline{z}',\lambda')\in\stF_{\underline{\E}}(U)$
we have 
\[
\Iso_{U}((\underline{z},\lambda),(\underline{z}',\lambda'))=\{(\underline{\omega},\underline{\tau})\in(\odi U^{*})^{r}\times(\odi U^{*})^{s}\st\tau_{i}z_{i}=z_{i}'\comma\underline{\tau}^{\underline{\E}(t)}\lambda(t)=\underline{\omega}^{\phi(t)}\lambda'(t)\}
\]
\end{rem}
\begin{defn}
Given a sequence $\underline{\E}=\E^{1},\dots,\E^{s}$ of elements
of $\duale T_{+}$ we define the map
\[
\pi_{\underline{\E}}\colon\stF_{\underline{\E}}\arr\stX_{\phi}
\]
 induced by the commutative diagram   \[   \begin{tikzpicture}[xscale=1.6,yscale=-1.5]     \node (A0_0) at (0, 0) {$\scriptstyle t$};     \node (A0_1) at (1, 0) {$T_+$};     \node (A0_3) at (3, 0) {$\Z^r$};     \node (A1_0) at (0, 1) {$\scriptstyle(\underline \E(t),-t)$};     \node (A1_1) at (1, 1) {$\N^s \oplus T$};     \node (A1_3) at (3, 1) {$\Z^s \oplus \Z^r$};     \path (A0_1) edge [->] node [auto] {$\scriptstyle{}$} (A1_1);     \path (A0_0) edge [|->,gray] node [auto] {$\scriptstyle{}$} (A1_0);     \path (A0_3) edge [right hook->] node [auto] {$\scriptstyle{}$} (A1_3);     \path (A0_1) edge [->] node [auto] {$\scriptstyle{\phi}$} (A0_3);     \path (A1_1) edge [->] node [auto,swap] {$\scriptstyle{\sigma_{\underline \E}}$} (A1_3);   \end{tikzpicture}   \] \end{defn}
\begin{rem}
\label{rem:description of piE}We can describe the functor $\pi_{\underline{\E}}$
explicitly. So suppose we have an object $\chi=(\underline{\shL},\underline{\shM},\underline{z},\lambda)\in\stF_{\underline{\E}}(U)$.
We have $\pi_{\underline{\E}}(\chi)=(\underline{\shL},a)\in\stX_{\phi}(U)$
where $a$ is given, for any $t\in T_{+}$, by   \[   \begin{tikzpicture}[xscale=2.0,yscale=-0.6]     \node (A0_0) at (0, 0) {$\underline{\shL}^{\phi(t)}$};     \node (A0_1) at (1, 0) {$\underline{\shM}^{\E(t)}$};     \node (A1_0) at (0, 1) {$a(t)$};     \node (A1_1) at (1, 1) {$\underline{z}^{\E(t)}$};     \path (A0_0) edge [->]node [auto] {$\scriptstyle{\lambda_t}$} (A0_1);     \path (A1_0) edge [|->,gray]node [auto] {$\scriptstyle{}$} (A1_1);   \end{tikzpicture}   \] 
Moreover, if $(\underline{\omega},\underline{\tau})$ is an isomorphism
in $\stF_{\underline{\E}}$, then $\pi_{\underline{\E}}(\underline{\omega},\underline{\tau})=\underline{\omega}$.

If $(\underline{z},\lambda)\in\stF_{\underline{\E}}(U)$ then $a=\pi_{\underline{\E}}(\underline{z},\lambda)\in\stX_{\phi}(U)$
is given by   \[   \begin{tikzpicture}[xscale=3.2,yscale=-0.6]     \node (A0_0) at (0, 0) {$T_+$};     \node (A0_1) at (1, 0) {$\odi{U}$};     \node (A1_0) at (0, 1) {$t$};     \node (A1_1) at (1, 1) {$ {\underline z}^{\underline \E(t)}/\lambda_t= z_1^{\E^1(t)} \cdots z_s^{\E^s(t)}/\lambda_t$};     \path (A0_0) edge [->] node [auto] {$\scriptstyle{}$} (A0_1);     \path (A1_0) edge [|->,gray] node [auto] {$\scriptstyle{}$} (A1_1);   \end{tikzpicture}   \] 
\end{rem}

\begin{rem}
\label{rem:Fdelta is an open substack of Fepsilon se delta sottosequenza di epsilon}If
$\underline{\E}=(\E^{i})_{i\in I}$ is a sequence of elements of $\duale T_{+}$,
$J\subseteq I$ and we set $\underline{\delta}=(\E^{j})_{j\in J}$
we can define a map over $\stX_{\phi}$ as   \[   \begin{tikzpicture}[xscale=2.8,yscale=-0.3]     \node (A0_0) at (0, 0) {$\stF_{\underline{\delta}}$};     \node (A0_1) at (1, 0) {$\stF_{\underline{\E}}$};     
\node (A1_3) at (2.6, 1) {$\shM_{i}'=\left\{ \begin{array}{cc} \shM_{i} & i\in J\\ \odi{} & i\notin J\end{array}\right. z_{i}'=\left\{ \begin{array}{cc} z_{i} & i\in J\\ 1 & i\notin J\end{array}\right.$};     \node (A2_0) at (0, 2) {$(\underline{\shL},\underline{\shM},\underline{z},\lambda)$};     \node (A2_1) at (1, 2) {$(\underline{\shL},\underline{\shM}',\underline{z}',\lambda)$};     \path (A0_0) edge [->]node [auto] {$\scriptstyle{\rho}$} (A0_1);     \path (A2_0) edge [|->,gray]node [auto] {$\scriptstyle{}$} (A2_1);   \end{tikzpicture}   \]  In fact $\rho$ comes from the monoid map $T\oplus\N^{I}\arr T\oplus\N^{J}$
induced by the projection. Moreover $\rho$ is an open immersion,
whose image is the open substack of $\stF_{\underline{\E}}$ of objects
$(\underline{\shL},\underline{\shM},\underline{z},\lambda)$ such
that $z_{i}$ generates $\shM_{i}$ for all $i\notin J$. We will
often consider $\stF_{\underline{\delta}}$ as an open substack of
$\stF_{\underline{\E}}$.\end{rem}
\begin{defn}
\label{def:T+epsilon sottolineato}Given a sequence $\underline{\E}=\E^{1},\dots,\E^{s}$
of elements of $\duale T_{+}$ we define
\[
T_{+}^{\underline{\E}}=T_{+}^{\E^{1},\dots,\E^{s}}=\{v\in T\:|\:\forall i\:\E^{i}(v)\geq0\}
\]
 We also consider the case $s=0$, so that $T_{+}^{\underline{\E}}=T$.
If we denote by $\hat{\phi}\colon T_{+}^{\underline{\E}}\arr\Z^{r}$
the extension of $\phi$, we also define $\stX_{\phi}^{\underline{\E}}=\stZ_{\phi}^{\underline{\E}}=\stX_{\hat{\phi}}$.\end{defn}
\begin{rem}
\label{rem:change of monoid for stack X} Assume we have a monoid
map $T_{+}\arr T_{+}'$ (compatible with $\phi$ and $\phi'$) inducing
an isomorphism on the associated groups. If $\underline{\E}=\E^{1},\dots,\E^{s}\in\duale{T'}_{+}\subseteq\duale T_{+}$,
then we have a $2$-commutative diagram   \[   \begin{tikzpicture}[xscale=1.8,yscale=-1.3]     \node (A0_0) at (0, 0) {$\stF_{\underline \E}'$};     \node (A0_1) at (1, 0) {$\stF_{\underline \E}$};     \node (A1_0) at (0, 1) {$\stX_{\phi'}$};     \node (A1_1) at (1, 1) {$\stX_\phi$};     \path (A0_0) edge [->] node [auto] {$\scriptstyle{\simeq}$} (A0_1);     \path (A0_0) edge [->] node [auto,swap] {$\scriptstyle{\pi_{\underline \E}'}$} (A1_0);     \path (A0_1) edge [->] node [auto] {$\scriptstyle{\pi_{\underline \E}}$} (A1_1);     \path (A1_0) edge [->] node [auto] {$\scriptstyle{}$} (A1_1);   \end{tikzpicture}   \] 
where $\stF_{\underline{\E}}'$ is the stack obtained from $T_{+}'$
with respect to $\underline{\E}$.\end{rem}
\begin{prop}
The map $\pi_{\underline{\E}}\colon\stF_{\underline{\E}}\arr\stX_{\phi}$
has a natural factorization
\[
\stF_{\underline{\E}}\arr\stX_{\phi}^{\underline{\E}}\arr\stZ_{\phi}\arr\stX_{\phi}
\]
\end{prop}
\begin{proof}
The factorization follows from \ref{rem:change of monoid for stack X}
taking monoid maps $T_{+}\arr T_{+}^{int}\arr T_{+}^{\underline{\E}}$.\end{proof}
\begin{rem}
This shows that $\pi_{\underline{\E}}$ has image in $\stZ_{\phi}$.We
will call with the same symbol $\pi_{\underline{\E}}$ the factorization
$\stF_{\underline{\E}}\arr\stZ_{\phi}$.
\end{rem}
We want now to show how the rays of $T_{+}$ can be used to describe
the objects of $\stZ_{\phi}$ over a field. Using notation from \ref{rem:description of isomorphism for local objects of X phi},
the result is:
\begin{thm}
\label{pro:characterization of points of Zphi}Let $k$ be a field
and $T_{+}\arrdi ak\in\stX_{\phi}(k)$. Then $a\in\stZ_{\phi}(k)$
if and only if there exists a group homomorphism $\lambda:T\arr\overline{k}^{*}$
and $\E\in\duale T_{+}$ such that 
\[
a(t)=\lambda_{t}0^{\E(t)}
\]
In particular if $\underline{\E}=\E^{1},\dots,\E^{r}$ generate $\duale T_{+}\otimes\Q$
then $\pi_{\underline{\E}}\colon\stF_{\underline{\E}}(\overline{k})\arr\stZ_{\phi}(\overline{k})$
is essentially surjective and so $\pi_{\underline{\E}}\colon|\stF_{\underline{\E}}|\arr|\stZ_{\phi}|$
is surjective. Finally, if the map $\phi\colon T\arr\Z^{r}$ is injective,
we have a one to one correspondence   \[   \begin{tikzpicture}[xscale=4.3,yscale=-0.5]     \node (A0_0) at (0, 0) {$\stZ_\phi(\overline k)/\simeq$};     \node (A0_1) at (1, 0) {$\{X\subseteq T_+\st X=\Supp \E \text{ for } \E\in\duale{T}_+\}$};     \node (A1_0) at (0, 1) {$a$};     \node (A1_1) at (1, 1) {$\{a=0\}$};     \path (A0_0) edge [->] node [auto] {$\scriptstyle{\gamma}$} (A0_1);     \path (A1_0) edge [|->,gray] node [auto] {$\scriptstyle{}$} (A1_1);   \end{tikzpicture}   \] 
In particular $|\stZ_{\phi}|=(\stZ_{\phi}(\overline{\Q})/\simeq)\bigsqcup\:[\bigsqcup_{\textup{primes }p}(\stZ_{\phi}(\overline{\F_{p}})/\simeq)]$.
\end{thm}
Before proving this Theorem we need some preliminary results, that
will be useful also later.
\begin{defn}
If $T_{+}$ is integral, $\E\in\duale T_{+}$ and $k$ is a field
we define
\[
p_{\E}=\bigoplus_{v\in T_{+},\E(v)>0}kx_{v}\subseteq k[T_{+}]
\]
If $p\in\Spec k[T_{+}]$ we set $p^{om}=\bigoplus_{x_{v}\in p}kx_{v}$.
\end{defn}
The suffix $(-)^{om}$ here stays for 'homogeneous', since, when $T_{+}=\N^{r}$
and $k[T_{+}]=k[x_{1},\dots,x_{r}]$, $p^{om}$ is a homogeneous ideal,
actually a monomial ideal.
\begin{lem}
\label{lem:properties of pE}Let $k$ be a field and assume that $T_{+}$
is integral. Then:
\begin{enumerate}
\item if $\E\in\duale T_{+}$, $p_{\E}$ is prime and $k[\{v\in T_{+}\st\E(v)=0\}]\arr k[T_{+}]\arr k[T_{+}]/p_{\E}$
is an isomorphism.
\item If $p\in\Spec k[T_{+}]$ then $p^{om}=p_{\E}$ for some $\E\in\duale T_{+}$.
\end{enumerate}
\end{lem}
\begin{proof}
$(1)$ It is obvious.

$(2)$ $p^{om}$ is a prime thanks to \cite[Proposition  1.7.12]{Kreuzer2005}
and therefore $p^{om}=p_{\E}$ for some $\E\in\duale T_{+}$ thanks
to \cite[Chapter I, Corollary 2.2.4]{Ogus2006}.\end{proof}
\begin{rem}
\label{rem:differ by torsor implies iso}If $k$ is an algebraically
closed field, $\phi\colon T\arr\Z^{r}$ is injective and $a,b\in\stX_{\phi}(k)$
differ by a torsor, i.e. there exists $\lambda\colon T_{+}\arr k^{*}$
such that $a=\lambda b$, then $a\simeq b$ in $\stZ_{\phi}(k)$.
Indeed $\lambda$ extends to a map $T\arr k^{*}$ and, since $k$
is algebraically closed, it extends again to a map $\lambda\colon\Z^{r}\arr k^{*}$.\end{rem}
\begin{proof}
(of Theorem \ref{pro:characterization of points of Zphi}) We can
assume that $k$ is algebraically closed and that $T_{+}$ is integral,
since if $a$ has an expression as in the statement then clearly $a\in\stZ_{\phi}(k)$.
Consider $p=\Ker(k[T_{+}]\arrdi ak)$. Thanks to \ref{lem:properties of pE},
we can write $p^{om}=p_{\E}$ for some $\E\in\duale T_{+}$. Set $T_{+}'=\{v\in T_{+}\st\E(v)=0\}$
and $T'=\langle T_{+}'\rangle_{\Z}$. Since $a$ maps $T'_{+}$ to
$k^{*}$, there exists an extension $\lambda\colon T'\arr k^{*}$.
On the other hand, since $k$ is algebraically closed, the inclusion
$T'\arr T$ yields a surjection
\[
\Hom(T,k^{*})\arr\Hom(T',k^{*})
\]
and so we can extend again to an element $\lambda\colon T\arr k^{*}$.
Since one has $\Supp\E=\{a=0\}$ by construction, it is easy to check
that $a(t)=\lambda_{t}0^{\E(t)}$ for all $t\in T_{+}$.

Now consider the last part of the statement and so assume $\phi\colon T\arr\Z^{r}$
injective. The map $\gamma$ is well defined thanks to above and surjective
since, given $\E\in\duale T_{+}$, one can always define $a(t)=0^{\E(t)}$.
For the injectivity, let $a,b\in\stZ_{\phi}(k)$ be such that $\{a=0\}=\{b=0\}$.
We can write $a(t)=\lambda_{t}0^{\E(t)}\comma b(t)=\mu_{t}0^{\E(t)}$,
where $\lambda,\mu\colon T\arr k^{*}$, so that $a,b$ differ by a
torsor and are therefore isomorphic thanks to \ref{rem:differ by torsor implies iso}.
Finally, since any point of $|\stZ_{\phi}|$ comes from an object
of $\stZ_{\phi}(\Z)$, we also have the last equality.
\end{proof}
In some cases the description of the objects of $\stF_{\underline{\E}}$
can be simplified, regardless of $\underline{\E}$, in the sense that
there exist a stack of reduced data $\stF_{\underline{\E}}^{\textup{red}}$,
whose objects can be described by less data, and an isomorphism $\stF_{\underline{\E}}\simeq\stF_{\underline{\E}}^{\textup{red}}$.
This kind of simplification could be very useful when we have to deal
with an explicit map of monoids $\phi\colon T_{+}\arr\Z^{r}$, as
we will see in \ref{pro:stack of reduced data for M-covers}. The
idea is that in order to define an object $(\underline{\shL},\underline{\shM},\underline{z},\lambda)\in\stF_{\underline{\E}}$,
we do not really need all the invertible sheaves $\shL_{1},\dots,\shL_{r}$,
because they are uniquely determined by a subset of them and the other
data.
\begin{defn}
\label{def:stack of reduced data}Assume $T\arrdi{\phi}\Z^{r}$ injective.
Let $V\subseteq\Z^{r}$ be a submodule with a given basis $v_{1},\dots,v_{q}$
and $\sigma\colon\Z^{r}\arr V$ be a map such that $(\id-\sigma)\Z^{r}\subseteq T$
(or equivalently $ $$\pi=\pi\circ\sigma$ where $\pi$ is the projection
$\Z^{r}\arr\Coker\phi$). Define $W=\langle(\id-\sigma)V,\sigma T\rangle\subseteq V$.
Given $\underline{\E}=\E^{1},\dots\E^{l}\in\duale T_{+}$ consider
the map   \[   \begin{tikzpicture}[xscale=2.8,yscale=-0.5]     \node (A0_0) at (0, 0) {$W\oplus \N^l$};     \node (A0_1) at (1, 0) {$\Z^q \oplus \Z^l$};     \node (A1_0) at (0, 1) {$(w,z)$};     \node (A1_1) at (1, 1) {$(-w,\underline{\E}(w)+z)$};     \path (A0_0) edge [->]node [auto] {$\scriptstyle{\psi_{\underline \E,\sigma}}$} (A0_1);     \path (A1_0) edge [|->,gray]node [auto] {$\scriptstyle{}$} (A1_1);   \end{tikzpicture}   \] We
define $\stF_{\underline{\E}}^{\textup{red},\sigma}=\stX_{\psi_{\underline{\E},\sigma}}$
and we call it the stack of reduced data of $\underline{\E}$.\end{defn}
\begin{lem}
\label{lem:for the stack of reduced data}Consider a submodule $U\subseteq\Z^{p}$,
a map $\underline{\E}\colon U\arr\Z^{l}$ and $\tau\colon\Z^{p}\arr\Z^{p}$
such that $(\id-\tau)\Z^{p}\subseteq U$. Consider the commutative
diagram   \[   \begin{tikzpicture}[xscale=1.9,yscale=-1.4]     
\node (A0_0) at (0, 0) {$(u,z)$};     
\node (A0_1) at (1, 0) {$U\oplus \N^l$};     
\node (A0_3) at (2.5, 0) {$U\oplus \N^l$};     
\node (A1_0) at (0, 1) {$(-u,\underline{\E}(u)+z)$};     
\node (A1_1) at (1, 1) {$\Z^p \oplus \Z^l$};     
\node (A1_3) at (2.5, 1) {$\Z^p \oplus \Z^l$};     
\node (A2_1) at (1, 1.4) {$(u,z)$};    
\node (A2_3) at (2.5, 1.4) {$(\tau u,\underline{\E}(u-\tau u)+z)$};     

\path (A0_1) edge [->]node [auto] {$\scriptstyle{\tau\oplus \id}$} (A0_3);     \path (A0_3) edge [->]node [auto] {$\scriptstyle{\psi}$} (A1_3);     \path (A1_1) edge [->]node [auto] {$\scriptstyle{}$} (A1_3);     \path (A2_1) edge [|->,gray]node [auto] {$\scriptstyle{}$} (A2_3);     \path (A0_0) edge [|->,gray]node [auto] {$\scriptstyle{}$} (A1_0);     \path (A0_1) edge [->]node [auto,swap] {$\scriptstyle{\psi}$} (A1_1);   \end{tikzpicture}   \] Then the induced map $\varphi\colon\stX_{\psi}\arr\stX_{\psi}$ is
isomorphic to $\id_{\stX_{\psi}}$.\end{lem}
\begin{proof}
Let $x_{1},\dots,x_{p}$ be a $\Z$-basis of $\Z^{p}$ with $a_{1},\dots,a_{k}\in\N$
such that $a_{1}x_{1},\dots,a_{k}x_{k}$ is a $\Z$-basis of $U$.
We want to define a natural isomorphism $\id_{\stX_{\psi}}\arrdi{\omega}\varphi$.
First note that it is enough to define it on the objects of $\stX_{\psi}$
coming from the atlas $\Spec\Z[U\oplus\N^{l}]$, prove the naturality
between such objects on a fixed scheme $T$ and for the restrictions.
An object coming from the atlas is of the form $(\lambda,\underline{z})$
where $\lambda\colon U\arr\odi T^{*}$ is an additive map and $\underline{z}=z_{1},\dots,z_{l}\in\odi T$.
Moreover $\varphi(\lambda,\underline{z})=(\tilde{\lambda},\underline{z})$
where $\tilde{\lambda}=\lambda\circ\tau$. Let $\underline{\eta}\in\Di{\Z^{p}}(T)$
the only elements such that $\underline{\eta}^{x_{i}}=\lambda(x_{i}-\tau x_{i})$
for $i=1,\dots,p$. These objects are well defined since $(\id-\tau)\Z^{p}\subseteq U$.
We claim that $\omega_{T,(\lambda,\underline{z})}=(\underline{\eta},\underline{1})$
is an isomorphism $(\lambda,\underline{z})\arr\varphi(\lambda,\underline{z})$
and define a natural transformation. It is an isomorphism since $1z_{j}=z_{j}$
and the condition
\[
\underline{\eta}^{-u}\underline{1}^{\underline{\E}(u)}\lambda(u)=\lambda(\tau u)\ \forall u\in U
\]
holds by construction checking it on the basis $a_{1}x_{1},\dots,a_{k}x_{k}$
of $U$ (see \ref{rem:description of isomorphism for local objects of X phi}).
It is also easy to check that this isomorphisms commute with the change
of basis. So it remains to prove that, if $(\underline{\sigma},\underline{\mu})$
is an isomorphism $(\lambda,\underline{z})\arr(\lambda',\underline{z}')$
then we have a commutative diagram   \[   \begin{tikzpicture}[xscale=2.8,yscale=-1.4]     \node (A0_0) at (0, 0) {$(\lambda,\underline z)$};     \node (A0_1) at (1, 0) {$(\lambda',\underline z')$};     \node (A1_0) at (0, 1) {$\varphi(\lambda,\underline z)$};     \node (A1_1) at (1, 1) {$\varphi(\lambda',\underline z')$};     \path (A0_0) edge [->]node [auto] {$\scriptstyle{(\underline \sigma,\underline \mu)}$} (A0_1);     \path (A0_1) edge [->]node [auto] {$\scriptstyle{\omega_{T,(\lambda',\underline z')}}$} (A1_1);     \path (A1_0) edge [->]node [auto] {$\scriptstyle{\varphi(\underline \sigma,\underline \mu)}$} (A1_1);     \path (A0_0) edge [->]node [auto,swap] {$\scriptstyle{\omega_{T,(\lambda,\underline z)}}$} (A1_0);   \end{tikzpicture}   \] 
We have $\varphi(\underline{\sigma},\underline{\mu})=(\tilde{\underline{\sigma}},\tilde{\underline{\mu}})$
with $\tilde{\underline{\mu}}=\underline{\mu}$ and $\tilde{\underline{\sigma}}^{x_{i}}=\underline{\sigma}^{\tau x_{i}}\underline{\mu}^{\underline{\E}(x_{i}-\tau x_{i})}$
(see \ref{rem: description of functors of Xphi on local objects}).
So it is easy to check that the commutativity in the second member
holds. For the first, the condition is $\tilde{\underline{\sigma}}\underline{\eta}=\underline{\eta}'\underline{\sigma}$,
which is equivalent to
\[
(\tilde{\underline{\sigma}}\underline{\eta})^{x_{i}}=\underline{\sigma}^{\tau x_{i}}\underline{\mu}^{\underline{\E}(x_{i}-\tau x_{i})}\lambda(x_{i}-\tau x_{i})=(\underline{\eta}'\underline{\sigma})^{x_{i}}=\lambda'(x_{i}-\tau x_{i})\underline{\sigma}^{x_{i}}
\]
 and to $\underline{\sigma}^{-(x_{i}-\tau x_{i})}\underline{\mu}^{\underline{\E}(x_{i}-\tau x_{i})}\lambda(x_{i}-\tau x_{i})=\lambda'(x_{i}-\tau x_{i})$
for any $i$. But, since $(\underline{\sigma},\underline{\mu})$ is
an isomorphism $(\lambda,\underline{z})\arr(\lambda',\underline{z}')$,
the condition 
\[
\underline{\sigma}^{-u}\underline{\mu}^{\underline{\E}(u)}\lambda(u)=\lambda'(u)\ \forall u\in U
\]
has to be satisfied.\end{proof}
\begin{prop}
\label{pro:isomorphism with the stack of reduced data}Assume $T\arrdi{\phi}\Z^{r}$
injective and let $\underline{\E}=\E^{1},\dots\E^{r}\in\duale T_{+}$
and $\sigma\comma V\comma v_{1},\dots,v_{q}$ be as in \ref{def:stack of reduced data}.
For appropriate choices of isomorphisms $\tilde{\lambda}$ given by
\ref{lem:morphisms of stack Xphi}, the functors   \[   \begin{tikzpicture}[xscale=5.6,yscale=-0.5]     \node (A0_0) at (0, 0) {$((\underline \shN^{\sigma e_i}\otimes \underline \shM^{\underline \E(e_i-\sigma e_i)})_{i=1,\dots,r},\underline \shM, \underline z,\tilde \lambda)$};     \node (A0_1) at (1, 0) {$(\underline \shN,\underline \shM, \underline z,\lambda)$};     \node (A1_0) at (0, 1) {$\stF_{\underline \E}$};     \node (A1_1) at (1, 1) {$\stF_{\underline \E}^{red,\sigma}$};     \node (A2_0) at (0, 2) {$(\underline \shL,\underline \shM, \underline z,\lambda)$};     \node (A2_1) at (1, 2) {$((\underline \shL^{v_i})_{i=1,\dots,q},\underline \shM, \underline z, \lambda_{|W})$};     \path (A1_1) edge [->]node [auto] {$\scriptstyle{}$} (A1_0);     \path (A1_0) edge [->]node [auto] {$\scriptstyle{}$} (A1_1);     \path (A2_0) edge [|->,gray]node [auto] {$\scriptstyle{}$} (A2_1);     \path (A0_1) edge [|->,gray]node [auto] {$\scriptstyle{}$} (A0_0);   \end{tikzpicture}   \] 
are inverses of each other.\end{prop}
\begin{proof}
Consider the commutative diagrams   \[   \begin{tikzpicture}[xscale=3.0,yscale=-1.4]     \node (A1_0) at (0, 1) {$W\oplus \N^s$};     \node (A1_1) at (1, 1) {$T\oplus \N^s$};     \node (A1_2) at (2, 1) {$T\oplus \N^s$};     \node (A1_3) at (3, 1) {$W\oplus \N^s$};     \node (A2_0) at (0, 2) {$\Z^q\oplus \Z^s$};     \node (A2_1) at (1, 2) {$\Z^r\oplus \Z^s$};     \node (A2_2) at (2, 2) {$\Z^r\oplus \Z^s$};     \node (A2_3) at (3, 2) {$\Z^q\oplus \Z^s$};     
\node (A3_2) at (2, 2.4) {$(x,y)$};     
\node (A3_3) at (3, 2.4) {$(\sigma x,\underline \E(x-\sigma x)+y)$};     
\path (A1_0) edge [right hook->]node [auto] {$\scriptstyle{}$} (A1_1);     \path (A1_3) edge [->]node [auto] {$\scriptstyle{\psi}$} (A2_3);     \path (A2_2) edge [->]node [auto] {$\scriptstyle{}$} (A2_3);     \path (A1_0) edge [->]node [auto,swap] {$\scriptstyle{\psi}$} (A2_0);     \path (A1_1) edge [->]node [auto] {$\scriptstyle{\phi_{\underline \E}}$} (A2_1);     \path (A1_2) edge [->]node [auto] {$\scriptstyle{\sigma\oplus \id}$} (A1_3);     \path (A2_0) edge [right hook->]node [auto] {$\scriptstyle{}$} (A2_1);     \path (A3_2) edge [|->,gray]node [auto] {$\scriptstyle{}$} (A3_3);     \path (A1_2) edge [->]node [auto,swap] {$\scriptstyle{\phi_{\underline \E}}$} (A2_2);   \end{tikzpicture}   \] They induce functors $\Lambda\colon\stF_{\underline{\E}}\arr\stF_{\underline{\E}}^{\textup{red},\sigma}$
and $\Delta\colon\stF_{\underline{\E}}^{\textup{red},\sigma}\arr\stF_{\underline{\E}}$
respectively, that behave as the functors of the statement thanks
to the description given in \ref{lem:morphisms of stack Xphi}. Finally,
applying \ref{lem:for the stack of reduced data}, we obtain that
$\Lambda\circ\Delta\simeq\id$ and $\Delta\circ\Lambda\simeq\id$.
\end{proof}

\subsection{Extremal rays and smooth sequences.}

We continue to use notation from \ref{not:notation for a monoid}.
We have seen that given a collection $\underline{\E}=\E^{1},\dots,\E^{r}\in\duale T_{+}$
we can associate to it a stack $\stF_{\underline{\E}}$ and a 'parametrization'
map $\stF_{\underline{\E}}\arr\stX_{\phi}$. The stack $\stF_{\underline{\E}}$
could be 'too big' if we do not make an appropriate choice of the
collection $\underline{\E}$. This happens for example if the rays
in $\underline{\E}$ are not distinct or, more generally, if a ray
in $\underline{\E}$ belongs to the submonoid generated by the other
rays in $\underline{\E}$. Thus we want to restrict our attention
to a special class of  rays, called extremal and to special sequences
of them.
\begin{defn}
An \emph{extremal} ray for $T_{+}$ is an element $\E\in\duale T_{+}$
such that
\begin{itemize}
\item $\E$ has minimal non empty support, i.e. the set $\Supp\E\subseteq T_{+}$
is minimal in
\[
(\{X\subseteq T_{+}\st X\neq\emptyset\text{ and }X=\Supp\delta\text{ for some }\delta\in\duale{T_{+}}\},\subseteq)
\]

\item $\E$ is normalized, i.e. $\E\colon T\arr\Z$ is surjective.
\end{itemize}
\end{defn}
\begin{lem}
Assume that $T_{+}$ is an integral monoid and let $v_{1},\dots,v_{l}$
be a system of generators of $T_{+}$. Then the extremal rays are
the normalized $\E\in\duale T_{+}-\{0\}$ such that $\Ker\E$ contains
$\rk T-1$ $\Q$-independent vectors among the $v_{1},\dots,v_{l}$.
In particular they are finitely many and they generate $\Q_{+}\duale T_{+}$.\end{lem}
\begin{proof}
Denote by $\Omega\subseteq\duale T_{+}$ the set of elements defined
in the statement. From \cite[Section 1.2, (9)]{Fulton1993} it follows
that $\Q_{+}\Omega=\Q_{+}\duale T_{+}$. If $\E\in\Omega$ then it
is an  extremal ray. Indeed 
\[
\emptyset\neq\Supp\E'\subseteq\Supp\E\then\exists\lambda\in\Q_{+}\text{ s.t. }\E'=\lambda\E\then\Supp\E'=\Supp\E
\]

Conversely let $\E$ be an  extremal ray and consider an expression
\[
\E=\sum_{\delta\in\Omega}\lambda_{\delta}\delta\qquad\text{with }\lambda_{\delta}\in\Q_{\geq0}
\]
 There must exists $\delta$ such that $\lambda_{\delta}\neq0$. So
\[
\Supp\delta\subseteq\Supp\E\then\Supp\delta=\Supp\E\then\exists\mu\in\Q_{+}\text{ s.t. }\E=\mu\delta\then\E=\delta
\]
\end{proof}
\begin{cor}
For an  extremal ray $\E$ and $\E'\in\duale T_{+}$ we have
\[
\Supp\E'=\Supp\E\iff\exists\lambda\in\Q_{+}\text{ s.t. }\E'=\lambda\E\iff\exists\lambda\in\N_{+}\text{ s.t. }\E'=\lambda\E
\]
\end{cor}
\begin{defn}
An element $v\in T_{+}$ is called \emph{indecomposable} if whenever
$v=v'+v''$ with $v',v''\in T_{+}$ it follows that $v'=0$ or $v''=0$.\end{defn}
\begin{prop}
$\duale T_{+}$ has a unique minimal system of generators composed
by the indecomposable elements. Moreover any  extremal ray is indecomposable.\end{prop}
\begin{proof}
The first claim of the statement follows from \cite[Chapter I, Proposition 2.1.2]{Ogus2006}
since $\duale T_{+}$ is sharp, i.e. it does not contain invertible
elements. For the second consider an extremal ray $\E$ and assume
$\E=\E'+\E''$. We have
\[
\Supp\E',\Supp\E''\subseteq\Supp\E\then\E'=\lambda\E,\E''=\mu\E\text{ with }\lambda,\mu\in\N
\]
 and so $\E=(\lambda+\mu)\E\then\lambda+\mu=1\then\lambda=0\text{ or }\mu=0\then\E'=0\text{ or }\E''=0$.\end{proof}
\begin{defn}
\label{def:definition of smooth sequence and smooth elements}A \emph{smooth
sequence} for $T_{+}$ is a sequence $\underline{\E}=\E^{1},\dots,\E^{s}\in\duale T_{+}$
for which there exist elements $v_{1},\dots,v_{s}$ in the associated
integral monoid $T_{+}^{int}$ of $T_{+}$ such that 
\[
T_{+}^{int}\cap\Ker\underline{\E}\:\text{ generates }\Ker\underline{\E}\qquad\text{and}\qquad\E^{i}(v_{j})=\delta_{i,j}\text{ for all }i,j
\]

We will also say that a ray $\E\in\duale T_{+}-\{0\}$ is \emph{smooth}
if there exists a smooth sequence as above such that $\E\in\langle\E^{1},\dots,\E^{s}\rangle_{\N}$
or, equivalently, such that $\Supp\E\subseteq\Supp\underline{\E}$.\end{defn}
\begin{rem}
\label{rem:the equivalently in definition of smooth sequence}If $T_{+}$
is integral and $\Omega$ is a system of generators, one can always
assume that $v_{i}\in\Omega.$ Moreover we also have that $\Omega\cap\Ker\underline{\E}$
generates $\Ker\underline{\E}$. 

Finally the equivalence in the last sentence of Definition \ref{def:definition of smooth sequence and smooth elements}
follows from the fact that, since $\Ker\underline{\E}$ is generated
by elements in $T_{+}^{int}$, then the inclusion of the supports
implies that $\E_{|\Ker\underline{\E}}=0$ and therefore $\E=\sum_{i}\E(v_{i})\E^{i}$.\end{rem}
\begin{lem}
\label{lem:decomposition of T+E and open smooth subscheme}Let $\underline{\E}=\E^{1},\dots,\E^{r}$
be a smooth sequence. Then
\[
T_{+}^{\underline{\E}}=\Ker\underline{\E}\oplus\langle v_{1},\dots,v_{r}\rangle_{\N}\subseteq T\text{ where }v_{1}\dots,v_{r}\in T_{+}^{int}\comma\E^{i}(v_{j})=\delta_{i,j}
\]
Moreover, if $z_{1},\dots,z_{s}\in T_{+}^{int}$ generate $T_{+}^{int}$,
then $\Z[T_{+}^{\underline{\E}}]=\Z[T_{+}^{int}]_{\prod_{\underline{\E}(z_{i})=0}x_{z_{i}}}$
so that $\Spec\Z[T_{+}^{\underline{\E}}]$ $(\stX_{\phi}^{\underline{\E}})$
is a smooth open subscheme (substack) of $\Spec\Z[T_{+}^{int}]$ $(\stZ_{\phi})$.\end{lem}
\begin{proof}
We have $T=\Ker\underline{\E}\oplus\langle v_{1},\dots,v_{r}\rangle_{\Z}$
and clearly $\Ker\underline{\E}\oplus\langle v_{1},\dots,v_{q}\rangle_{\N}\subseteq T_{+}^{\underline{\E}}$.
Conversely if $v\in T_{+}^{\underline{\E}}$ we can write 
\[
v=z+\sum_{i}\E^{i}(v)v_{i}\text{ with }z\in\Ker\underline{\E}\then v\in\Ker\underline{\E}\oplus\langle v_{1},\dots,v_{q}\rangle_{\N}
\]
In particular $\Spec\Z[T_{+}^{\underline{\E}}]\simeq\A_{\Z}^{r}\times D_{\Z}(\Ker\underline{\E})$
and so both $\Spec\Z[T_{+}^{\underline{\E}}]$ and $\stX_{\phi}^{\underline{\E}}$
are smooth. Now let 
\[
I=\{i\st\underline{\E}(z_{i})=0\}\text{ and }S_{+}=\langle T_{+}^{int},-z_{i}\text{ for }i\in I\rangle\subseteq T
\]
We need to prove that $S_{+}=T_{+}^{\underline{\E}}$. Clearly we
have the inclusion $\subseteq$. For the reverse inclusion, it is
enough to prove that $-\Ker\underline{\E}\cap T_{+}^{int}\subseteq S_{+}$.
But if $v\in\Ker\underline{\E}\cap T_{+}^{int}$ then
\[
v=\sum_{j=1}^{s}a_{j}z_{j}=\sum_{j\in I}a_{j}z_{j}\then-v\in S_{+}
\]
\end{proof}
\begin{rem}
\label{rem:subsequences of smooth sequences are smooth too}Any subsequence
of a smooth sequence is smooth too. Indeed let $\underline{\delta}=\E^{1},\dots,\E^{s}$
a subsequence of a smooth sequence $\underline{\E}=\E^{1},\dots,\E^{r}$,
with $r>s$. We have to prove that $\langle\Ker\underline{\delta}\cap T_{+}^{int}\rangle_{\Z}=\Ker\underline{\delta}$.
Take $v\in\Ker\underline{\delta}$. So
\[
v-\sum_{j=s+1}^{r}\E^{j}(v)v_{j}\in\Ker\underline{\E}=\langle\Ker\underline{\E}\cap T_{+}^{int}\rangle_{\Z}\subseteq\langle\Ker\underline{\delta}\cap T_{+}^{int}\rangle_{\Z}\then v\in\langle\Ker\underline{\delta}\cap T_{+}^{int}\rangle_{\Z}
\]
\end{rem}
\begin{prop}
\label{lem:equivalent condition for a smooth integral extremal ray}Let
$\E\in\duale T_{+}$. Then $\E$ is a smooth  extremal ray if and
only if $\E$ is a smooth sequence composed of one element, i.e. $\Ker\E\cap T_{+}^{int}$
generates $\Ker\E$ and there exists $v\in T_{+}$ such that $\E(v)=1$. 

In particular any element of a smooth sequence is a smooth  extremal
ray.\end{prop}
\begin{proof}
We can assume $T_{+}$ integral. If $\E$ is smooth and extremal,
then there exists a smooth sequence $\E^{1},\dots,\E^{q}$ such that
$\E\in\langle\E^{1},\dots,\E^{q}\rangle_{\N}$. Since $\E$ is indecomposable,
it follows that $\E=\E^{i}$ for some $i$. Conversely assume that
$\E$ is a smooth sequence. So it is smooth by definition and it is
normalized since $\E(v)=1$ for some $v$. Finally an inclusion $\Supp\delta\subseteq\Supp\E$
for $\delta\in\duale T_{+}$ means that $\delta\in\langle\E\rangle_{\N}$,
as remarked in \ref{rem:the equivalently in definition of smooth sequence},
and so $\Supp\delta=\emptyset$ or $\Supp\delta=\Supp\E$.
\end{proof}
We conclude with a lemma that will be useful later.
\begin{lem}
\label{lem:comparison smooth sequences for different monoids}Let
$T_{+}\comma T_{+}'$ be integral monoids and $h\colon T\arr T'$
be a homomorphism such that $h(T_{+})=T_{+}'$ and $\Ker h=\langle\Ker h\cap T_{+}\rangle$.
If $\underline{\E}=\E^{1},\dots\E^{r}\in\duale{T_{+}'}$ then
\[
\underline{\E}\text{ smooth sequence for }T_{+}'\iff\underline{\E}\circ h\text{ smooth sequence for }T_{+}
\]
\end{lem}
\begin{proof}
Clearly there exist $v_{i}\in T_{+}'$ such that $\E^{i}(v_{j})=\delta_{i,j}$
if and only if there exist $w_{i}\in T_{+}$ such that $\E^{i}\circ h(w_{j})=\delta_{i,j}$.
On the other hand we have a surjective morphism
\[
\Ker\underline{\E}\circ h/\langle\Ker\underline{\E}\circ h\cap T_{+}\rangle_{\Z}\arr\Ker\underline{\E}/\langle\Ker\underline{\E}\cap T_{+}'\rangle_{\Z}
\]
In order to conclude it is enough to prove that this map is injective.
So let $v\in T$ such that 
\[
h(v)=\sum_{j}a_{j}z_{j}\text{ with }a_{j}\in\Z\comma z_{j}\in T_{+}'\comma\underline{\E}(z_{j})=0
\]
Since $h(T_{+})=T_{+}'$, there exist $y_{j}\in T_{+}$ such that
$h(y_{j})=z_{j}$. In particular $y=\sum_{j}a_{j}y_{j}\in\langle\Ker\underline{\E}\circ h\cap T_{+}\rangle_{\Z}$
and 
\[
v-y\in\Ker h=\langle\Ker h\cap T_{+}\rangle\subseteq\langle\Ker\underline{\E}\circ h\cap T_{+}\rangle
\]

\end{proof}

\subsection{The smooth locus $\stZ_{\phi}^{\textup{sm}}$ of the main component
$\stZ_{\phi}$.}
\begin{lem}
\label{lem:fundamental lemma for the smooth locus of X phi} Let $\underline{\E}=\E^{1},\dots,\E^{q}$
be a smooth sequence and $\chi$ be a finite sequence of elements
of $\duale T_{+}$. Assume that all the elements of $\chi$ are distinct,
each $\E^{i}$ is an element of $\chi$ and that for any $\delta$
in $\chi$ we have
\[
\delta\in\langle\E^{1},\dots,\E^{q}\rangle_{\N}\then\exists i\;\delta=\E^{i}
\]
 As usual denote by $\pi_{\chi}$ the map $\stF_{\chi}\arr\stX_{\phi}$.
Then we have an equivalence
\[
\stF_{\underline{\E}}=\pi_{\chi}^{-1}(\stX_{\phi}^{\underline{\E}})\arrdi{\simeq}\stX_{\phi}^{\underline{\E}}
\]
\end{lem}
\begin{proof}
Set $\chi=\E^{1},\dots,\E^{q},\eta^{1},\dots,\eta^{l}=\underline{\E},\underline{\eta}$.
We first prove that $\pi_{\chi}^{-1}(\stX_{\phi}^{\underline{\E}})\subseteq\stF_{\underline{\E}}$.
Since they are open substacks, we can check this over an algebraically
closed field $k$. Let $(\underline{z},\lambda)\in\pi_{\chi}^{-1}(\stX_{\phi}^{\underline{\E}}$)
so that $a=\pi_{\chi}(\underline{z},\lambda)=\underline{z}^{\underline{\E}}/\lambda\colon T_{+}\arr k$
by \ref{rem:description of piE}. We have to prove that $z_{\eta_{j}}\neq0$.
Assume by contradiction that $z_{\eta_{j}}=0$. Since we can write
$a=b0^{\eta_{j}}$ and since $a$ extends to $T_{+}^{\underline{\E}}$
so that $a(t)\neq0$ if $t\in T_{+}\cap\Ker\underline{\E}$, we have
that $\eta_{j}$ is $0$ on $T_{+}\cap\Ker\underline{\E}$. In particular
\[
\Supp\eta^{j}\subseteq\Supp\underline{\E}\then\eta^{j}\in\langle\E^{1},\dots,\E^{q}\rangle_{\N}\then\exists i\;\eta^{j}=\E^{i}
\]

Thanks to \ref{rem:Fdelta is an open substack of Fepsilon se delta sottosequenza di epsilon},
it is enough to prove that if $\underline{\E}$ is a smooth sequence
such that $T_{+}=T_{+}^{\underline{\E}}$ then $\pi_{\underline{\E}}$
is an isomorphism. By \ref{lem:decomposition of T+E and open smooth subscheme}
we can write $T_{+}=W\oplus\N^{q}$, where $W$ is a free $\Z$-module
such that $\underline{\E}_{|W}=0$ and, if we denote by $v_{1},\dots,v_{q}$
the canonical base of $\N^{q}$, $\E^{j}(v_{i})=\delta_{i,j}$. Consider
the diagram   \[   \begin{tikzpicture}[xscale=3.0,yscale=-0.8]     
\node (A0_0) at (0, 0) {$\N^q\oplus T$};     
\node (A0_1) at (1, 0) {$T_+$};     
\node (A1_0) at (0, 1) {$\N^q \oplus W \oplus \Z^q$};     
\node (A1_1) at (1, 1) {$W \oplus \N^q$};     
\node (A1_3) at (2.5, 1.5) {$\gamma(e_i)=v_i\comma \gamma_{|W}=-\id_W\comma \gamma(v_i)=0$};     
\node (A2_3) at (2.5, 2.5) {$\delta(e_i)=\phi(v_i)\comma\delta_{|\Z^r}=\id_{\Z^r}$};     
\node (A3_0) at (0, 3) {$\Z^q\oplus\Z^r$};     
\node (A3_1) at (1, 3) {$\Z^r$};     
\node[rotate=-90] (u) at (0, 0.5) {$=$};     
\node[rotate=-90] (uu) at (1, 0.5) {$=$};     

\path (A1_0) edge [->]node [auto,swap] {$\scriptstyle{\sigma_{\underline \E}}$} (A3_0);     \path (A1_0) edge [->]node [auto] {$\scriptstyle{\gamma}$} (A1_1);     \path (A3_0) edge [->]node [auto] {$\scriptstyle{\delta}$} (A3_1);     \path (A1_1) edge [->]node [auto] {$\scriptstyle{\phi}$} (A3_1);   \end{tikzpicture}   \]  One can check directly its commutativity. In this way we get a map
$s\colon\stX_{\phi}\arr\stF_{\underline{\E}}$. Again a direct computation
on the diagrams defining $s$ and $\pi_{\underline{\E}}$ shows that
$\pi_{\underline{\E}}\circ s\simeq\id_{\stX_{\phi}}$ and that the
diagram inducing $G=s\circ\pi_{\underline{\E}}$ is   \[   \begin{tikzpicture}[xscale=3.0,yscale=-0.8]     \node (A0_0) at (0, 0) {$\N^q \oplus W \oplus \Z^q$};     \node (A0_1) at (1, 0) {$\N^q \oplus W \oplus \Z^q$};     \node (A0_3) at (2.5, 0.5) {$\alpha(e_i)=e_i-v_i,\alpha_{|W}=\id_{W}\comma \alpha_{|\Z^q}=0$};     \node (A1_3) at (2.5, 1.5) {$\beta(e_i)=\phi(v_i)\comma\beta_{|\Z^r}=\id_{\Z^r}$};     \node (A2_0) at (0, 2) {$\Z^q\oplus\Z^r$};     \node (A2_1) at (1, 2) {$\Z^q\oplus\Z^r$};     \path (A0_0) edge [->]node [auto] {$\scriptstyle{\alpha}$} (A0_1);     \path (A0_1) edge [->]node [auto] {$\scriptstyle{\sigma_{\underline \E}}$} (A2_1);     \path (A0_0) edge [->]node [auto,swap] {$\scriptstyle{\sigma_{\underline \E}}$} (A2_0);     \path (A2_0) edge [->]node [auto] {$\scriptstyle{\beta}$} (A2_1);   \end{tikzpicture}   \] We
will prove that $G\simeq\id_{\stF_{\underline{\E}}}$. An object of
$\stF_{\underline{\E}}(A)$, where $A$ is a ring, coming from the
atlas is given by $a=(\underline{z},\lambda,\underline{\mu})\colon\N^{q}\oplus W\oplus\Z^{q}\arr A$
where $\underline{z}=(a(e_{i}))_{i}=z_{1},\dots,z_{q}\in A$, $\lambda=a_{|W}\colon W\arr A^{*}$
is a homomorphism and $\underline{\mu}=(\mu(v_{i}))_{i}=\mu_{1}\dots,\mu_{q}\in A^{*}$.
Moreover $Ga=a\circ\alpha$ is $((z_{i}/\mu_{i})_{i},\lambda,\underline{1})$.
It is now easy to check that $(\underline{\mu},1)\colon Ga\arr a$
is an isomorphism and that this map defines an isomorphism $G\arr\id_{\stF_{\underline{\E}}}$.\end{proof}
\begin{cor}
\label{thm:toric open substack of Z phi via smooth sequence}If $\underline{\E}$
is a smooth sequence then $\pi_{\underline{\E}}\colon\stF_{\underline{\E}}\arr\stZ_{\phi}$
is an open immersion with image $\stX_{\phi}^{\underline{\E}}$.
\end{cor}
It turns out that if $\underline{\E}$ is a smooth sequence, then
$\stX_{\phi}^{\underline{\E}}$ has a more explicit description:
\begin{prop}
\label{pro:points of XphiE}Let $\underline{\E}=\E^{1},\dots,\E^{r}$
be a smooth sequence, $k$ be a field and $a\in\stX_{\phi}(k)$. Then
\[
a\in\stX_{\phi}^{\underline{\E}}(k)\iff\exists\E\in\langle\E^{1},\dots,\E^{r}\rangle_{\N}\comma\lambda\colon T\arr\overline{k}^{*}\text{ s.t. }a=\lambda0^{\E}
\]
Moreover if $\lambda0^{\E}\in\stX_{\phi}^{\underline{\E}}(k)$, for
some $\E\in\duale T_{+}$, $\lambda\colon T\arr\overline{k}^{*}$,
then $\E\in\langle\E^{1},\dots,\E^{r}\rangle_{\N}$.\end{prop}
\begin{proof}
We can assume $k$ algebraically closed and $T_{+}$ integral. In
this case $a\in\stX_{\phi}^{\underline{\E}}(k)$ if and only if $a\colon T_{+}\arr k$
extends to a map $\Ker\underline{\E}\oplus\N^{r}=T_{+}^{\underline{\E}}\arr k$.
So $\Leftarrow$ holds. Conversely, from \ref{pro:characterization of points of Zphi},
we can write $a=\lambda0^{\E}$ where $\lambda\colon T\arr k^{*}$
and $\E\in\duale{(T_{+}^{\underline{\E}})}$. From \ref{lem:decomposition of T+E and open smooth subscheme}
we see that $\duale{T_{+}^{\underline{\E}}}=\langle\E^{1},\dots,\E^{r}\rangle_{\N}$.
Finally, if $\lambda0^{\E}\in\stX_{\phi}^{\underline{\E}}$ for some
$\E$, then $\Supp\E\subseteq\Supp\underline{\E}$ and we are done.\end{proof}
\begin{lem}
\label{lem:fundamental lemma for all the classification for h}Let
$\underline{\E}=(\E^{i})_{i\in I}$ be a sequence of distinct smooth
extremal rays and $\Theta$ be a collection of smooth sequences with
rays in $\underline{\E}$. Set
\[
\stF_{\underline{\E}}^{\Theta}=\left\{ (\underline{\shL},\underline{\shM},\underline{z},\delta)\in\stF_{\underline{\E}}\left|\begin{array}{c}
V(z_{i_{1}})\cap\cdots\cap V(z_{i_{s}})\neq\emptyset\\
\text{iff }\exists\underline{\delta}\in\Theta\text{ s.t. }\E^{i_{1}},\dots,\E^{i_{s}}\subseteq\underline{\delta}
\end{array}\right.\right\} 
\]
Then, taking into account the identification made in \ref{rem:Fdelta is an open substack of Fepsilon se delta sottosequenza di epsilon},
we have 
\[
\stF_{\underline{\E}}^{\Theta}=\bigcup_{\underline{\delta}\in\Theta}\stF_{\underline{\delta}}
\]
\end{lem}
\begin{proof}
Let $\chi=(\underline{\shL},\underline{\shM},\underline{z},\lambda)\in\bigcup_{\underline{\delta}\in\Theta}\stF_{\underline{\delta}}(T)$,
for some scheme $T$ and let $p\in V(z_{i_{1}})\cap\cdots\cap V(z_{i_{s}})$.
This means that the pullback of $\pi_{\underline{\E}}(\chi)$ to $\overline{k(p)}$
is given by $a=b0^{\E^{i_{1}}+\cdots+\E^{i_{r}}}$ for some $b\colon T_{+}\arr\overline{k(p)}$.
By definition there exists $\underline{\delta}\in\Theta$ such that
$a\in\stF_{\underline{\delta}}(\overline{k(p)})$, i.e. $a=\mu0^{\delta}$
for some $\delta\in\langle\underline{\delta}\rangle_{\N}$, $\mu\colon T\arr\overline{k(p)}^{*}$.
So
\[
\Supp\E^{i_{j}}\subseteq\{a=0\}=\Supp\delta\subseteq\Supp\underline{\delta}\then\E^{i_{j}}\in\langle\underline{\delta}\rangle_{\N}
\]

For the other inclusion, since all the $\stF_{\underline{\delta}}$
are open substacks of $\stF_{\underline{\E}}$, we can reduce the
problem to the case of an algebraically closed field $k$. So let
$(\underline{z},\lambda)\in\stF_{\underline{\E}}^{\Theta}(k)$ and
set $J=\{i\in I\st z_{i}=0\}$. By definition of $\stF_{\underline{\E}}^{\Theta}$
there exists $\underline{\delta}\in\Theta$ such that $\underline{\eta}=(\E^{j})_{j\in J}\subseteq\underline{\delta}$
and, taking into account \ref{rem:Fdelta is an open substack of Fepsilon se delta sottosequenza di epsilon},
this means that $a\in\stF_{\underline{\eta}}(k)\subseteq\stF_{\underline{\delta}}(k)$.\end{proof}
\begin{defn}
Let $\Theta$ be a collection of smooth sequences. We define
\[
X_{\phi}^{\Theta}=\bigcup_{\underline{\delta}\in\Theta}\Spec\Z[T_{+}^{\underline{\delta}}]\subseteq\Spec\Z[T_{+}]\text{ and }\stX_{\phi}^{\Theta}=\bigcup_{\underline{\delta}\in\Theta}\stX_{\phi}^{\underline{\delta}}\subseteq\stZ_{\phi}
\]
\end{defn}
\begin{thm}
\label{pro:piE for theta isomorphism}Let $\underline{\E}=(\E^{i})_{i\in I}$
be a sequence of distinct smooth extremal rays and $\Theta$ be a
collection of smooth sequences with rays in $\underline{\E}$. Then
we have an isomorphism
\[
\stF_{\underline{\E}}^{\Theta}=\pi_{\underline{\E}}^{-1}(\stX_{\phi}^{\Theta})\arrdi{\simeq}\stX_{\phi}^{\Theta}
\]
\end{thm}
\begin{proof}
Taking into account \ref{lem:fundamental lemma for all the classification for h},
it is enough to note that
\[
\pi_{\underline{\E}}^{-1}(\stX_{\phi}^{\Theta})=\pi_{\underline{\E}}^{-1}(\bigcup_{\underline{\delta}\in\Theta}\stX_{\phi}^{\underline{\delta}})=\bigcup_{\underline{\delta}\in\Theta}\stF_{\underline{\E}\cap\underline{\delta}}=\bigcup_{\underline{\delta}\in\Theta}\stF_{\underline{\delta}}\arrdi{\simeq}\stX_{\phi}^{\Theta}
\]
\end{proof}
\begin{prop}
\label{pro:Xphitheta is a smooth toric stack}Let $\underline{\E}=(\E^{i})_{i\in I}$
be a sequence of distinct smooth extremal rays and $\Theta$ be a
collection of smooth sequences with rays in $\underline{\E}$. Then
the set
\[
\Delta^{\Theta}=\{\langle\eta_{1},\dots,\eta_{r}\rangle_{\Q_{+}}\st\exists\underline{\delta}\in\Theta\text{ s.t. }\eta_{1},\dots,\eta_{r}\subseteq\underline{\delta}\}
\]
is a toric fan in $\duale T\otimes\Q$ whose associated toric variety
over $\Z$ is $X_{\phi}^{\Theta}$. Moreover
\[
\stX_{\phi}^{\Theta}\simeq[X_{\phi}^{\Theta}/\Di{\Z^{r}}]
\]
\end{prop}
\begin{proof}
We know that if $\underline{\eta}$ is a smooth sequence then $\Spec\Z[T_{+}^{\underline{\eta}}]$
is a smooth open subset of $\Spec\Z[T_{+}^{int}]$ and it is the affine
toric variety associated to the cone $\langle\underline{\eta}\rangle_{\Q_{+}}$.
It is then easy to check that $\Delta^{\Theta}$ is a fan whose associated
toric variety is $X_{\phi}^{\Theta}$. Since $\Spec\Z[T_{+}^{\underline{\eta}}]$
is the equivariant open subset of $\Spec\Z[T_{+}^{int}]$ inducing
$\stX_{\phi}^{\underline{\eta}}$ in $\stZ_{\phi}$, then $X^{\Theta}$
is the equivariant open subset of $\Spec\Z[T_{+}^{int}]$ inducing
$\stX_{\phi}^{\Theta}$. In particular we obtain the last isomorphism.\end{proof}
\begin{lem}
\label{pro:smooth locus of Z[T +]} Assume $T_{+}$ integral and set
$\Theta$ for the set of all smooth sequences. Then $X_{\phi}^{\Theta}$
is the smooth locus of $\Spec\Z[T_{+}]$. In particular $\stZ_{\phi}^{\textup{sm}}=\stX_{\phi}^{\Theta}\simeq[X_{\phi}^{\Theta}/\Di{\Z^{r}}]$.\end{lem}
\begin{proof}
From \ref{lem:decomposition of T+E and open smooth subscheme} we
know that $\Spec\Z[T_{+}^{\underline{\E}}]$ is smooth over $\Z$
and it is an open subset of $\Spec\Z[T_{+}]$. So we focus on the
converse. Since $\Spec\Z[T_{+}]$ is flat over $\Z$, we can replace
$\Z$ by an algebraically closed field $k$. Let $p\in\Spec k[T_{+}]$
be a smooth point. In particular $p^{om}$ is smooth too. If $p^{om}=0$
then $p\in\Spec k[T]$ and we have done. So we can assume $p^{om}=p_{\E}$
for some $0\neq\E\in\duale T_{+}$ thanks to \ref{lem:properties of pE}.
We claim that there exist a smooth sequence $\E^{1},\dots,\E^{q}$
such that $\E\in\langle\E^{1},\dots,\E^{q}\rangle_{\N}$. This is
enough to conclude that $p\in\Spec k[T_{+}^{\underline{\E}}]$ . Indeed
if $x_{w}\in p$ for some $w\in\Ker\underline{\E}\cap T_{+}$ then
it belongs to $p^{om}=p_{\E}$ and so $\E(w)>0$, which is not our
case.

So assume we have $\E\in\duale T_{+}$ such that $p_{\E}$ is a regular
point. Set $W=\langle\Ker\E\cap T_{+}\rangle_{\Z}$ and $T_{+}'=T_{+}+W$.
Note that $\Spec k[T_{+}']$ is an open subset of $\Spec k[T_{+}]$
that contains $p_{\E}$. Moreover $k[T_{+}']/p_{\E}=k[W]$. Let $v_{1},\dots,v_{q}\in T_{+}$
be elements such that
\[
T_{+}'=\langle v_{1},\dots,v_{q}\rangle_{\N}+W\qquad\text{and}\qquad\E(v_{i})>0
\]
with $q$ minimal. We claim that $M=p_{\E}/p_{\E}^{2}\simeq k[W]^{q}$,
where $p_{\E}$ is thought in $k[T_{+}']$. Indeed $M$ is a $k$-vector
space over the $x_{v}$, $v\in T_{+}'$ that satisfies: $\E(v)>0$
and whenever we have $v=v'+v''$ with $v',v''\in T_{+}'$ it follows
that $\E(v')=0$ or $\E(v'')=0$. A simple computation shows that
such a $v$ must be of the form $v_{i}+W$ for some $i$. But since
we have chosen $q$ minimal we have $(v_{i}+W)\cap(v_{j}+W)=\emptyset$
if $i\neq j$. This implies that $M$ is a free $k[W]$-module with
basis $x_{v_{1}},\dots,x_{v_{q}}$. This shows that $q=\alt p_{\E}$.

Now set $V=\langle v_{1},\dots,v_{q}\rangle_{\Z}$. Since $V+W=T$,
$\rk V\leq q$ and 
\[
k[W]\simeq k[T_{+}']/p_{\E}\then\rk T=\dim k[T_{+}']=\alt p_{\E}+\dim k[W]=q+\rk W
\]
we obtain that $v_{1},\dots,v_{q}$ are independent. Let $\E^{1},\dots,\E^{q}$
given by $\E^{i}(v_{j})=\delta_{i,j}$ and $\E_{|W}^{i}=0$. In particular
$W=\Ker\underline{\E}$ and it is generated by elements in $T_{+}$.
Since $\E_{|W}=0$ we have 
\[
\E=\sum_{i=1}^{q}\E(v_{i})\E^{i}\qquad\E(v_{i})>0
\]
 Moreover since $T_{+}\subseteq T_{+}'$ and $\E^{i}\in\duale{T'}_{+}$
we get that $\E^{i}\in\duale T_{+}$, as required.\end{proof}
\begin{thm}
\label{thm:fundamental theorem for the smooth locus of ZM} If $\underline{\E}$
is a sequence of distinct indecomposable rays containing the smooth
 extremal rays then $\pi_{\underline{\E}}$ induces an equivalence
\[
\left\{ (\underline{\shL},\underline{\shM},\underline{z},\delta)\in\stF_{\underline{\E}}\left|\begin{array}{c}
V(z_{i_{1}})\cap\cdots\cap V(z_{i_{s}})=\emptyset\\
\text{if }\E^{i_{1}},\dots\E^{i_{s}}\text{ is not a}\\
\text{smooth sequence}
\end{array}\right.\right\} =\pi_{\underline{\E}}^{-1}(\stZ_{\phi}^{\textup{sm}})\arrdi{\simeq}\stZ_{\phi}^{\textup{sm}}
\]
\end{thm}
\begin{proof}
Lemma \ref{pro:smooth locus of Z[T +]} tells us that $\stZ_{\phi}^{\textup{sm}}=\stX_{\phi}^{\Theta}$,
where $\Theta$ is the collection of all smooth sequences, while \ref{lem:fundamental lemma for the smooth locus of X phi}
allows us to replace $\underline{\E}$ with the sequence of all smooth
extremal rays. Therefore it is enough to apply \ref{pro:piE for theta isomorphism}
and \ref{pro:Xphitheta is a smooth toric stack}.\end{proof}
\begin{prop}
\label{pro:equivalent condition for belonging in the smooth locus of the main irreducible component of X phi}Let
$a\colon T_{+}\arr k\in\stX_{\phi}(k)$, where $k$ is a field. Then
$a$ lies in $\stZ_{\phi}^{\textup{sm}}$ if and only if there exists
a smooth ray $\E\in\duale T_{+}$ and $\lambda\colon T\arr\overline{k}^{*}$
such that $a=\lambda0^{\E}$. \end{prop}
\begin{proof}
Apply \ref{thm:fundamental theorem for the smooth locus of ZM} and
\ref{pro:points of XphiE}.
\end{proof}

\subsection{Extension of objects from codimension $1$.}

In this subsection we want to explain how it is possible, in certain
cases, to check that an object of $\stX_{\phi}$ over a sufficiently
regular scheme $X$ comes (uniquely) from $\stF_{\underline{\E}}$
only checking what happens in codimension $1$.
\begin{notation}
Given a scheme $X$ we will denote by $\Picsh X$ the category whose
objects are invertible sheaves and whose arrows are maps between them.\end{notation}
\begin{prop}
\label{pro:the map X(Y) --> X(X) is fully faithful (equivalence) is it so between Pic}Let
$X\arrdi fY$ be a map of schemes. If $\Picsh Y\arrdi{f^{*}}\Picsh X$
is fully faithful (resp. an equivalence) then $\stX_{\phi}(Y)\arrdi{f^{*}}\stX_{\phi}(X)$
has the same property.\end{prop}
\begin{proof}
Let $(\underline{\shL},a),(\underline{\shL}',a')\in\stX_{\phi}(Y)$
and $\underline{\sigma}\colon f^{*}(\underline{\shL},a)\arr f^{*}(\underline{\shL}',a')$
be a map in $\stX_{\phi}(X)$. Any map $\sigma_{i}\colon f^{*}\shL_{i}\arr f^{*}\shL_{i}$
comes from a unique map $\tau_{i}\colon\shL_{i}\arr\shL_{i}$, i.e.
$\sigma_{i}=f^{*}\tau_{i}$. Since 
\[
f^{*}(\underline{\tau}^{\phi(t)}(a(t)))=\underline{\sigma}^{\phi(t)}(f^{*}a(t))=f^{*}(a'(t))\then\underline{\tau}^{\phi(t)}(a(t))=a'(t)
\]
$\underline{\tau}$ is a map $(\underline{\shL},a)\arr(\underline{\shL}',a')$
such that $f^{*}\underline{\tau}=\underline{\sigma}$. We can conclude
that $f^{*}\colon\stX_{\phi}(Y)\arr\stX_{\phi}(X)$ is fully faithful.

Now assume that $\Picsh Y\arrdi{f^{*}}\Picsh X$ is an equivalence.
We have to prove that $\stX_{\phi}(Y)\arrdi{f^{*}}\stX_{\phi}(X)$
is essentially surjective. So let $(\underline{\shM},b)\in\stX_{\phi}(X)$.
Since $f^{*}$ is an equivalence we can assume $\shM_{i}=f^{*}\shL_{i}$
for some invertible sheaf $\shL_{i}$ on $Y$. Since for any invertible
sheaf $\shL$ on Y one has that $\shL(Y)\simeq(f^{*}\shL)(X)$, any
section $b(t)\in\underline{\shM}^{\phi(t)}$ extends to a unique section
$a(t)\in\underline{\shL}^{\phi(t)}$. Since
\[
f^{*}(a(t)\otimes a(s))=b(t)\otimes b(s)=b(t+s)=f^{*}(a(t+s))\then a(t)\otimes a(s)=a(t+s)
\]
for any $t,s\in T_{+}$ and $a(0)=1$, it follows that $(\underline{\shL},a)\in\stX_{\phi}(Y)$
and $f^{*}(\underline{\shL},a)=(\underline{\shM},b)$.\end{proof}
\begin{cor}
\label{cor:lift when Picsh is the same}Let $X\arrdi fY$ be a map
of schemes and consider a commutative diagram    \[   \begin{tikzpicture}[xscale=1.3,yscale=-1.0]     \node (A0_0) at (0, 0) {$X$};     \node (A0_1) at (1, 0) {$\stF_{\underline \E}$};     \node (A1_0) at (0, 1) {$Y$};     \node (A1_1) at (1, 1) {$\stX_\phi$};     \path (A0_0) edge [->]node [auto] {$\scriptstyle{}$} (A0_1);     \path (A1_0) edge [->,dashed]node [auto] {$\scriptstyle{}$} (A0_1);     \path (A0_0) edge [->]node [auto,swap] {$\scriptstyle{f}$} (A1_0);     \path (A0_1) edge [->]node [auto] {$\scriptstyle{\pi_{\underline \E}}$} (A1_1);     \path (A1_0) edge [->]node [auto] {$\scriptstyle{}$} (A1_1);   \end{tikzpicture}   \] where
$\underline{\E}$ is a sequence of elements of $\duale T_{+}$. Then
if $\Picsh X\arrdi{f^{*}}\Picsh Y$ is fully faithful (resp. an equivalence)
the dashed lifting is unique (resp. exists).\end{cor}
\begin{proof}
It is enough to consider the $2$-commutative diagram   \[   \begin{tikzpicture}[xscale=2.0,yscale=-1.0]     \node (A0_0) at (0, 0) {$\stF_{\underline \E}(Y)$};     \node (A0_1) at (1, 0) {$\stF_{\underline \E}(X)$};     \node (A1_0) at (0, 1) {$\stX_\phi(Y)$};     \node (A1_1) at (1, 1) {$\stX_\phi(X)$};     \path (A0_0) edge [->]node [auto] {$\scriptstyle{f^*}$} (A0_1);     \path (A0_0) edge [->]node [auto,swap] {$\scriptstyle{\pi_{\underline \E}}$} (A1_0);     \path (A0_1) edge [->]node [auto] {$\scriptstyle{\pi_{\underline \E}}$} (A1_1);     \path (A1_0) edge [->]node [auto] {$\scriptstyle{f^*}$} (A1_1);   \end{tikzpicture}   \] and
note that $f^{*}$ is fully faithful (resp. an equivalence) in both
cases.\end{proof}
\begin{thm}
\label{thm:fundamental theorem for locally factorial schemes}Let
$X$ be a locally noetherian and locally factorial scheme, $\underline{\E}=(\E^{i})_{i\in I}$
be a sequence of distinct smooth extremal rays and $\Theta$ be a
collection of smooth sequences with rays in $\underline{\E}$. Consider
the full subcategories 
\[
\catC_{X}^{\Theta}=\left\{ (\underline{\shL},\underline{\shM},\underline{z},\delta)\in\stF_{\underline{\E}}(X)\left|\begin{array}{c}
\codim_{X}V(z_{i_{1}})\cap\cdots\cap V(z_{i_{s}})\geq2\\
\text{if }\nexists\underline{\delta}\in\Theta\text{ s.t. }\E^{i_{1}},\dots\E^{i_{s}}\subseteq\underline{\delta}
\end{array}\right.\right\} \subseteq\shF_{\underline{\E}}(X)
\]
and 
\[
\catD_{X}^{\Theta}=\left\{ \chi\in\stX_{\phi}(X)\left|\begin{array}{c}
\forall p\in X\text{ with }\codim_{p}X\leq1\\
\chi_{|\overline{k(p)}}\in\stX_{\phi}^{\Theta}
\end{array}\right.\right\} \subseteq\stX_{\phi}(X)
\]
Then $\pi_{\underline{\E}}$ induces an equivalence of categories
\[
\catC_{X}^{\Theta}=\pi_{\underline{\E}}^{-1}(\catD_{X}^{\Theta})\arrdi{\simeq}\catD_{X}^{\Theta}
\]
\end{thm}
\begin{proof}
We claim that
\[
\catC_{X}^{\Theta}=\{\chi\in\stF_{\underline{\E}}(X)\st\exists U\subseteq X\text{ open subset s.t. }\codim_{X}X-U\geq2\comma\chi_{|U}\in\stF_{\underline{\E}}^{\Theta}(U)\}
\]

$\subseteq$ Taking into account the definition of $\stF_{\underline{\E}}^{\Theta}$
in \ref{lem:fundamental lemma for all the classification for h},
it is enough to consider 
\[
U=X-\bigcup_{\nexists\underline{\delta}\in\Theta\text{ s.t. }\E^{i_{1}},\dots\E^{i_{s}}\subseteq\underline{\delta}}V(z_{i_{1}})\cap\cdots\cap V(z_{i_{s}})
\]

$\supseteq$ If $p\in V(z_{i_{1}})\cap\cdots\cap V(z_{i_{s}})$ and
$\codim_{p}X\leq1$ then $p\in U$ and again by definition of $\stF_{\underline{\E}}^{\Theta}$
there exists $\underline{\delta}\in\Theta$ such that $\E^{i_{1}},\dots,\E^{i_{s}}\subseteq\underline{\delta}$.

We also claim that 
\[
\catD_{X}^{\Theta}=\{\chi\in\stX_{\phi}(X)\st\exists U\subseteq X\text{ open subset s.t. }\codim_{X}X-U\geq2\comma\chi_{|U}\in\stX_{\phi}^{\Theta}(U)\}
\]

$\supseteq$ Such a $U$ contains all the codimension $1$ or $0$
points of $X$.

$\subseteq$ Let $\chi\in\catD_{X}^{\Theta}$ and $X\arrdi g\stX_{\phi}$
be the induced map. If $\xi$ is a generic point of $X$, we know
that $f(\xi)\in|\stX_{\phi}^{\Theta}|\subseteq|\stZ_{\phi}|$. In
particular $f(|X|)\subseteq|\stZ_{\phi}|$. Since both $X$ and $\stZ_{\phi}$
are reduced $g$ factors through a map $X\arrdi g\stZ_{\phi}$. Since
$\stX_{\phi}^{\Theta}$ is an open substack of $\stZ_{\phi}$, it
follows that $U=g^{-1}(\stX_{\phi}^{\Theta})$ is an open subscheme
of $X$, $\chi_{|U}\in\stX_{\phi}^{\Theta}(U)$ and, by definition
of $\catD_{X}^{\Theta}$, $\codim_{X}X-U\geq2$.

Taking into account \ref{pro:piE for theta isomorphism} it is clear
that $\catC_{X}^{\Theta}=\pi_{\underline{\E}}^{-1}(\catD_{X}^{\Theta})$.
We will make use of the fact that if $U\subseteq X$ is an open subscheme
such that $\codim_{X}X-U\geq2$ then the restriction yields an equivalence
$\Picsh X\simeq\Picsh U$. The map $\catC_{X}^{\Theta}\arr\catD_{X}^{\Theta}$
is essentially surjective since, given an object of $\catD_{X}^{\Theta}$,
the associated map $X\arrdi g\stX_{\phi}$ fits in a $2$-commutative
diagram   \[   \begin{tikzpicture}[xscale=1.5,yscale=-1.2]     \node (A0_0) at (0, 0) {$U$};     \node (A0_1) at (1, 0) {$\stF_{\underline \E}^\Theta\subseteq \stF_{\underline \E}$};     \node (A1_0) at (0, 1) {$X$};     \node (A1_1) at (1, 1) {$\stX_\phi$};     \path (A0_0) edge [->]node [auto] {$\scriptstyle{}$} (A0_1);     \path (A1_0) edge [->]node [auto] {$\scriptstyle{g}$} (A1_1);     \path (A0_1) edge [->]node [auto] {$\scriptstyle{\pi_{\underline \E}}$} (A1_1);     \path (A0_0) edge [->]node [auto] {$\scriptstyle{}$} (A1_0);   \end{tikzpicture}   \] 
and so lifts to a map $X\arr\stF_{\underline{\E}}$ thanks to \ref{cor:lift when Picsh is the same}.

It remains to show that $\catC_{X}^{\Theta}\arr\catD_{X}^{\Theta}$
is fully faithful. Let $\chi,\chi'\in\catC_{X}^{\Theta}$ and $U,U'$
be the open subscheme given in the definition of $\catC_{X}^{\Theta}$.
Set $V=U\cap U'$. Taking into account \ref{pro:the map X(Y) --> X(X) is fully faithful (equivalence) is it so between Pic}
and \ref{pro:piE for theta isomorphism} we have   \[   \begin{tikzpicture}[xscale=4.3,yscale=-0.9]     \node (A0_0) at (0, 0) {$\Hom_{\stF_{\underline \E}(X)}(\chi,\chi')$};     \node (A0_1) at (1, 0) {$\Hom_{\stX_\phi(X)}(\chi,\chi')$};     \node (A1_0) at (0, 1) {$\Hom_{\stF_{\underline \E}(V)}(\chi_{|V},\chi'_{|V})$};     \node (A1_1) at (1, 1) {$\Hom_{\stX_\phi(V)}(\chi_{|V},\chi'_{|V})$};     \node (A2_0) at (0, 2) {$\Hom_{\stF_{\underline \E}^\Theta(V)}(\chi_{|V},\chi'_{|V})$};     \node (A2_1) at (1, 2) {$\Hom_{\stX_\phi^\Theta(V)}(\chi_{|V},\chi'_{|V})$};     
\node[rotate=-90] (s) at (0, 0.5) {$\simeq$};
\node[rotate=-90] (ss) at (1, 0.5) {$\simeq$};
\node[rotate=-90] (s2) at (0, 1.5) {$\simeq$};
\node[rotate=-90] (ss2) at (1, 1.5) {$\simeq$};

\path (A0_0) edge [->]node [auto] {$\scriptstyle{}$} (A0_1);     \path (A1_0) edge [->]node [auto] {$\scriptstyle{}$} (A1_1);     \path (A2_0) edge [->]node [auto,swap] {$\scriptstyle{\simeq}$} (A2_1);   \end{tikzpicture}   \] 
\end{proof}

\section{Galois covers for a diagonalizable group.}

In this section we will fix a finite diagonalizable group scheme $G$
over $\Z$ and we will call $M=\Hom(G,\Gm)$ its character group.
So $M$ is a finite abelian group and $G=\Di M$. With abuse of notation
we will write $\odi U[M]=\odi U[G_{U}]$ and $\stZ_{M}=\stZ_{\Di M}$,
the main component of $\MCov$. It turns out that in this case $\Di M$-covers
have a nice and more explicit description.

In the first subsection we will show that $\MCov\simeq\stX_{\phi}$
for an explicit map $T_{+}\arrdi{\phi}\Z^{M}/\langle e_{0}\rangle$
and that this isomorphism preserves the main irreducible components
of both stacks. Moreover we will study the connection between $\MCov$
and the equivariant Hilbert schemes $\MHilb^{\underline{m}}$ and
prove some results about their geometry.

Then we will introduce an upper semicontinuous map $|\MCov|\arrdi h\N$
that yields a stratification by open substacks of $\MCov$. We will
also see that $\{h=0\}$ coincides with the open substack of $\Di M$-torsors,
while $\{h\leq1\}$ lies in the smooth locus of $\stZ_{M}$ and can
be described by a particular set of smooth  extremal rays. This will
allow us to describe normal $\Di M$-covers over a locally noetherian
and locally factorial scheme $X$ with $(\car X,|M|)=1$.

\subsection{The stack $\MCov$ and its main irreducible component $\stZ_{M}$.}

Consider a scheme $U$ and a cover $X=\Spec\alA$ on it. An action
of $\Di M$ on it consists of a decomposition 
\[
\alA=\bigoplus_{m\in M}\alA_{m}
\]
 such that $\odi U\subseteq\alA_{0}$ and the multiplication maps
$\alA_{m}\otimes\alA_{n}$ into $\alA_{m+n}$. If $X/U$ is a $\Di M$-cover
there exists an fppf covering $\{U_{i}\arr U\}$ such that $\alA_{|U_{i}}\simeq\odi{U_{i}}[M]$
as $\Di M$-comodules. This means that for any $m\in M$ we have
\[
\forall i\;(\alA_{m})_{|U_{i}}\simeq\odi{U_{i}}\then\alA_{m}\text{ invertible}
\]
Conversely any $M$-graded quasi-coherent algebra $\alA=\bigoplus_{m\in M}\alA_{m}$
with $\alA_{0}=\odi U$ and $\alA_{m}$ invertible for any $m$ yields
a $\Di M$-cover $\Spec\alA$.

So the stack $\MCov$ can be described as follows. An object of $\MCov(U)$
is given by a collection of invertible sheaves $\shL_{m}$ for $m\in M$
with maps
\[
\psi_{m,n}\colon\shL_{m}\otimes\shL_{n}\arr\shL_{m+n}
\]
and an isomorphism $\odi U\simeq\shL_{0}$ satisfying the following
relations:   \[   \begin{tikzpicture}[xscale=1.5,yscale=-1.2]     
\node (A0_1) at (1, 0.4) {$\textup{Commutativity}$};     
\node (A0_5) at (5, 0.4) {$\textup{Associativity}$};     
\node (A1_0) at (0, 1) {$\shL_m\otimes \shL_n$};     
\node (A1_2) at (2, 1) {$\shL_n \otimes \shL_m$};     
\node (A1_4) at (4, 1) {$\shL_m \otimes \shL_n \otimes \shL_t$};     
\node (A1_6) at (6, 1) {$\shL_m \otimes \shL_{n+t}$};     \node (A2_1) at (1, 2) {$\shL_{m+n}$};     
\node (A2_4) at (4, 2) {$\shL_{m+n}\otimes \shL_t$};     
\node (A2_6) at (6, 2) {$\shL_{m+n+t}$};
\node (name) at (0.2, 3.3) {$\begin{array}{c}
\textup{Neutral}\\
\textup{Element}
\end{array}
$};
\node (A3_3) at (1.2, 3) {$\shL_m$};     
\node (A3_4) at (2.7, 3) {$\shL_m \otimes \odi{U}$};     
\node (A3_5) at (4.2, 3) {$\shL_m \otimes \shL_0$};     
\node (A3_6) at (5.7, 3) {$\shL_m$};

\path (A3_4) edge [->]node [auto] {$\scriptstyle{\simeq}$} (A3_5);     \path (A1_0) edge [->]node [auto,swap] {$\scriptstyle{\psi_{m,n}}$} (A2_1);     \path (A1_6) edge [->]node [auto] {$\scriptstyle{\psi_{m,n+t}}$} (A2_6);     \path (A1_0) edge [->]node [auto] {$\scriptstyle{\simeq}$} (A1_2);     \path (A1_2) edge [->]node [auto] {$\scriptstyle{\psi_{n,m}}$} (A2_1);     \path (A3_5) edge [->]node [auto] {$\scriptstyle{\psi_{m,0}}$} (A3_6);     \path (A2_4) edge [->]node [auto] {$\scriptstyle{\psi_{m+n,t}}$} (A2_6);     \path (A1_4) edge [->]node [auto] {$\scriptstyle{\id \otimes \psi_{n,t}}$} (A1_6);     \path (A3_3) edge [->]node [auto] {$\scriptstyle{\simeq}$} (A3_4);     \path (A1_4) edge [->]node [auto,swap] {$\scriptstyle{\psi_{m,n}\otimes \id}$} (A2_4);   
\path (A3_3) edge [->,bend left=20]node [auto,swap] {$\scriptstyle{\id}$} (A3_6);\end{tikzpicture}   \] 

If we assume that $\shL_{m}=\odi Uv_{m}$, i.e. that we have sections
$v_{m}$ generating $\shL_{m}$, the maps $\psi_{m,n}$ can be thought
of as elements of $\odi U$ and the algebra structure is given by
$v_{m}v_{n}=\psi_{m,n}v_{m+n}$. In this case we can rewrite the above
conditions obtaining 
\begin{equation}
\psi_{m,n}=\psi_{n,m},\quad\psi_{m,0}=1,\quad\psi_{m,n}\psi_{m+n,t}=\psi_{n,t}\psi_{n+t,m}\label{eq:condition on psi}
\end{equation}
The functor that associates to a scheme $U$ the functions $\psi\colon M\times M\arr\odi U$
satisfying the above conditions is clearly representable by the spectrum
of the ring
\begin{equation}
R_{M}=\Z[x_{m,n}]/(x_{m,n}-x_{n,m},x_{m,0}-1,x_{m,n}x_{m+n,t}-x_{n,t}x_{n+t,m})\label{eq:writing of RM}
\end{equation}
In this way we obtain a Zariski epimorphism $\Spec R_{M}\arr\MCov$,
that we will prove to be smooth. We now want to prove that the stack
$\MCov$ is isomorphic to a stack of the form $\stX_{\phi}$. 
\begin{defn}
Define $\tilde{K}_{+}$ as the quotient monoid of $\N^{M\times M}$
by the equivalence relation generated by 
\[
e_{m,n}\sim e_{n,m},\quad e_{m,0}\sim0,\quad e_{m,n}+e_{m+n,t}\sim e_{n,t}+e_{n+t,m}
\]
Also define $\phi_{M}\colon\tilde{K}_{+}\arr\Z^{M}/\langle e_{0}\rangle$
by $\phi_{M}(e_{m,n})=e_{m}+e_{n}-e_{m+n}$. \end{defn}
\begin{prop}
$R_{M}\simeq\Z[\tilde{K}_{+}]$ and there exists an isomorphism 
\begin{equation}
\stX_{\phi_{M}}\simeq\MCov\label{eq:MCov isomorphic to Xphi}
\end{equation}
 such that $\Spec\Z[\tilde{K}_{+}]\simeq\Spec R_{M}\arr\MCov\simeq\stX_{\phi_{M}}$
is the map defined in \ref{pro:atlas for the stack associated to a monoid map}.
In particular
\[
\MCov\simeq[\Spec R_{M}/\Di{\Z^{M}/\langle e_{0}\rangle}]
\]
\end{prop}
\begin{proof}
The required isomorphism sends $(\underline{\shL},\tilde{K}_{+}\arrdi{\psi}\ISym\underline{\shL})\in\stX_{\phi_{M}}$
to the object of $\MCov$ given by invertible sheaves $(\shL'_{m}=\shL_{m}^{-1})$
and $\psi_{m,n}=\psi(e_{m,n})$.
\end{proof}
We want to prove that the isomorphism \ref{eq:MCov isomorphic to Xphi}
sends $\stZ_{\phi_{M}}$ to $\stZ_{M}$ (see def. \ref{def:the main component of GCov})
and $\stB_{\phi_{M}}$ to $\Bi\Di M$. We need the following classical
result on the structure of a $\Di M$-torsor (see \cite[Exposé VIII, Proposition 4.1 and 4.6]{Grothendieck1970}):
\begin{prop}
\label{pro:equivalent conditions for a D(M)-torsor}Let $M$ be a
finite abelian group and $P\arr U$ a $\Di M$-equivariant map. Then
$P$ is an fppf $\Di M$-torsor if and only if $P\in\MCov(U)$ and
all the multiplication maps $\psi_{m,n}$ are isomorphisms.
\end{prop}
Now consider the exact sequence   \[   \begin{tikzpicture}[xscale=1.5,yscale=-0.5]     
\node (A0_0) at (0.3, 0) {$0$};     
\node (A0_1) at (1, 0) {$K$};     \node (A0_2) at (2, 0) {$\Z^M/\langle e_0\rangle$};     \node (A0_3) at (3, 0) {$M$};     
\node (A0_4) at (3.7, 0) {$0$};     \node (A1_2) at (2, 1) {$e_m$};     \node (A1_3) at (3, 1) {$m$};     \path (A0_0) edge [->] node [auto] {$\scriptstyle{}$} (A0_1);     \path (A0_2) edge [->] node [auto] {$\scriptstyle{}$} (A0_3);     \path (A0_3) edge [->] node [auto] {$\scriptstyle{}$} (A0_4);     \path (A0_1) edge [->] node [auto] {$\scriptstyle{}$} (A0_2);     \path (A1_2) edge [|->,gray] node [auto] {$\scriptstyle{}$} (A1_3);   \end{tikzpicture}   \] 
\begin{defn}
For $m,n\in M$ we define
\[
v_{m,n}=\phi_{M}(e_{m,n})=e_{m}+e_{n}-e_{m+n}\in K
\]
 and $K_{+}$ as the submonoid of $K$ generated by the $v_{m,n}$.
We will set $x_{m,n}=x^{v_{m,n}}\in\Z[K_{+}]$ and, for $\E\in\duale K_{+}$,
$\E_{m,n}=\E(v_{m,n})$.\end{defn}
\begin{lem}
The map   \[   \begin{tikzpicture}[xscale=2.2,yscale=-0.5]     \node (A0_0) at (0, 0) {$\tilde K_+$};     \node (A0_1) at (1, 0) {$K$};     \node (A1_0) at (0, 1) {$e_{m,n}$};     \node (A1_1) at (1, 1) {$v_{m,n}$};     \path (A0_0) edge [->] node [auto] {$\scriptstyle{}$} (A0_1);     \path (A1_0) edge [|->,gray] node [auto] {$\scriptstyle{}$} (A1_1);   \end{tikzpicture}   \] 
is the associated group of $\tilde{K}_{+}$ and $K_{+}$ is its associated
integral monoid. In particular we have a $2$-cartesian diagram   \[   \begin{tikzpicture}[xscale=2.6,yscale=-1.5]     \node (A0_0) at (0, 0) {$\Spec \Z[K]$};     \node (A0_1) at (1, 0) {$\Spec \Z[K_+]$};     \node (A0_2) at (2, 0) {$\Spec R_M$};     \node (A1_0) at (0, 1) {$\Bi \Di{M}$};     \node (A1_1) at (1, 1) {$\stZ_M$};     \node (A1_2) at (2, 1) {$\MCov$};     \path (A0_0) edge [->] node [auto] {$\scriptstyle{}$} (A0_1);     \path (A0_1) edge [->] node [auto] {$\scriptstyle{}$} (A0_2);     \path (A1_0) edge [->] node [auto] {$\scriptstyle{}$} (A1_1);     \path (A0_2) edge [->] node [auto] {$\scriptstyle{}$} (A1_2);     \path (A1_1) edge [->] node [auto] {$\scriptstyle{}$} (A1_2);     \path (A0_0) edge [->] node [auto] {$\scriptstyle{}$} (A1_0);     \path (A0_1) edge [->] node [auto] {$\scriptstyle{}$} (A1_1);   \end{tikzpicture}   \] \end{lem}
\begin{proof}
Set $x=\prod_{m,n}x_{m,n}$. Since an object $\psi\in\Spec R_{M}(U)$
is a torsor if and only if $\psi_{m,n}\in\odi U^{*}$ for all $m,n$,
it follows that $(\Spec R_{M})_{x}=\Bi\Di M\times_{\MCov}\Spec R_{M}$.
We want to define an inverse to $(R_{M})_{x}\arr\Z[K]$. Consider
the algebra $S_{M}$ over $R_{M}$ induced by the atlas map $\Spec R_{M}\arr\MCov$,
i.e. 
\[
S_{M}=\bigoplus_{m\in M}R_{M}w_{m}\text{ with }w_{0}=1\comma w_{m}w_{n}=x_{m,n}w_{m+n}
\]
The algebra $(S_{M})_{x}$ is a $\Di M$-torsor over $(R_{M})_{x}$
and so $w_{m}\in(S_{M})_{x}^{*}$ for all $m$. In particular we can
define a group homomorphism   \[   \begin{tikzpicture}[xscale=2.6,yscale=-0.7]     \node (A0_0) at (0, 0) {$\Z^M/\langle e_0 \rangle$};     \node (A0_1) at (1, 0) {$(S_M)_x^*$};     \node (A1_0) at (0, 1) {$e_m$};     \node (A1_1) at (1, 1) {$w_m$};     \path (A0_0) edge [->] node [auto] {$\scriptstyle{}$} (A0_1);     \path (A1_0) edge [|->,gray] node [auto] {$\scriptstyle{}$} (A1_1);   \end{tikzpicture}   \] 
which restricts to a map $K\arr(R_{M})_{x}$ that sends $v_{m,n}$
to $x_{m,n}$. In particular the map $\tilde{K}_{+}\arr K$ defined
in the statement gives the associated group of $\tilde{K}_{+}$ and
has as image exactly $K_{+}$, which means that $K_{+}$ is the integral
monoid associated to $\tilde{K}_{+}$.

In order to conclude the proof it is enough to apply \ref{lem:the domain monoid and group monoid associated to T +: ring}
and \ref{cor:the domain monoid and group monoid associated to T +: stack}.\end{proof}
\begin{cor}
\label{cor:MCov as global quotient}The isomorphism $\stX_{\phi_{M}}\simeq\MCov$
(\ref{eq:MCov isomorphic to Xphi}) induces isomorphisms $\stB_{\phi_{M}}\simeq\Bi\Di M$
and $\stZ_{\phi_{M}}\simeq\stZ_{M}$. In particular $\stZ_{M}$ is
an irreducible component of $\MCov$ and
\[
\Bi\Di M\simeq[\Spec\Z[K]/\Di{\Z^{M}/\langle e_{0}\rangle}]\text{ and }\stZ_{M}\simeq[\Spec\Z[K_{+}]/\Di{\Z^{M}/\langle e_{0}\rangle}]
\]

\end{cor}
Note that the induced map $\phi_{M}\colon K\arr\Z^{M}/\langle e_{0}\rangle$
is just the inclusion and so it is injective. This means that any
result obtained in section \ref{sec:stack Xphi} applies naturally
in the context of $\Di M$-covers. In particular now we show how we
can describe the objects of $\stF_{\underline{\E}}$, for a sequence
of rays in $\duale{\tilde{K}}_{+}$, in a simpler way.
\begin{prop}
\label{pro:stack of reduced data for M-covers}Let $M\simeq\prod_{i=1}^{n}\Z/l_{i}\Z$
be a decomposition and let $m_{1},\dots,m_{n}$ be the associated
generators. Given $\underline{\E}=\E^{1},\dots,\E^{r}\in\duale K_{+}$
define $\stF_{\underline{\E}}^{\textup{red}}$ as the stack whose
objects over a scheme $X$ are sequences $\underline{\shL}=\shL_{1},\dots,\shL_{n},\underline{\shM}=\shM_{1},\dots,\shM_{r},\underline{z}=z_{1},\dots,z_{r},\underline{\mu}=\mu_{1},\dots,\mu_{n}$
where $\underline{\shL}\comma\underline{\shM}$ are invertible sheaves
over $X$, $z_{i}\in\shM_{i}$ and $\underline{\mu}$ are isomorphisms
\[
\mu_{i}\colon\shL_{i}^{-l_{i}}\arrdi{\simeq}\underline{\shM}^{\underline{\E}(l_{i}e_{m_{i}})}=\shM_{1}^{\E^{1}(l_{i}e_{m_{i}})}\otimes\cdots\otimes\shM_{r}^{\E^{r}(l_{i}e_{m_{i}})}
\]

Then we have an isomorphism of stacks   \[   \begin{tikzpicture}[xscale=4.7,yscale=-0.6]     \node (A0_0) at (0, 0) {$\stF_{\underline{\E}}$};     \node (A0_1) at (1, 0) {$\stF_{\underline{\E}}^{\textup{red}}$};     \node (A1_0) at (0, 1) {$(\underline \shL,\underline \shM, \underline z,\lambda)$};     \node (A1_1) at (1, 1) {$( (\shL_{m_i})_{i=1,\dots,n},\underline \shM, \underline z,(\lambda(l_ie_{m_i}))_{i=1,\dots,n})$};     \path (A0_0) edge [->]node [auto] {$\scriptstyle{}$} (A0_1);     \path (A1_0) edge [|->,gray]node [auto] {$\scriptstyle{}$} (A1_1);   \end{tikzpicture}   \] \end{prop}
\begin{proof}
We want to find $\sigma\comma V\comma v_{1},\dots,v_{q}$ as in \ref{def:stack of reduced data}
such that $\stF_{\underline{\E}}^{\textup{red},\sigma}=\stF_{\underline{\E}}^{\textup{red}}$
and that the map in the statement coincides with the one defined in
\ref{pro:isomorphism with the stack of reduced data}. Set $\delta^{i}\colon M\arr\{0,\dots,l_{i}-1\}$
as the map such that $\pi_{i}(m)=\pi_{i}(\delta_{m}^{i}m_{i})$, where
$\pi_{i}\colon M\arr\Z/l_{i}\Z$ is the projection, and think of it
also as a map $\delta^{i}\colon\Z^{M}/\langle e_{0}\rangle\arr\Z$.
Set $V=\bigoplus_{i=1}^{n}\Z e_{m_{i}}$, $v_{i}=e_{m_{i}}$ and $\sigma\colon\Z^{M}/\langle e_{0}\rangle\arr V$
as $\sigma(e_{m})=\sum_{i=1}^{n}\delta_{m}^{i}v_{i}$. Clearly $(\id-\sigma)\Z^{M}/\langle e_{0}\rangle\subseteq K$
and $(\id-\sigma)V=0$. So $W=\sigma K$. We have
\[
\sigma(v_{m,n})=\sum_{i=1}^{n}\delta_{m,n}^{i}v_{i}\in\bigoplus_{i=1}^{n}l_{i}\Z v_{i}
\]
since $\delta_{m,n}^{i}\in\{0,l_{i}\}$ for all $i$. On the other
hand $\sigma(v_{(l_{i}-1)m_{i},m_{i}})=l_{i}v_{i}$. Therefore we
have $W=\bigoplus_{i=1}^{n}l_{i}\Z v_{i}$. It is now easy to check
that all the definitions agree.
\end{proof}
We now want to express the relation between $\MCov$ and the equivariant
Hilbert scheme, that can be defined as follows. Given $\underline{m}=m_{1},\dots,m_{r}\in M$,
so that $\Di M$ acts on $\A_{\Z}^{r}=\Spec\Z[x_{1},\dots,x_{r}]$
with graduation $\deg x_{i}=m_{i}$, we define $\MHilb^{\underline{m}}\colon\Sch^{\textup{op}}\arr\set$
as the functor that associates to a scheme $Y$ the set of pairs $(X\arrdi fY,j)$
where $X\in\MCov(Y)$ and $j\colon X\arr\A_{Y}^{r}$ is an equivariant
closed immersion over $Y$. Such a pair can be also thought of as
a coherent sheaf of algebras $\alA\in\MCov(Y)$ together with a graded
surjective map $\odi Y[x_{1},\dots,x_{r}]\arr\alA$. This functor
is proved to be a scheme of finite type in \cite{Haiman2002}.
\begin{prop}
\label{pro:MHlb --> MCov has irreducible fibers}Let $\underline{m}=m_{1},\dots,m_{r}\in M$.
The forgetful map $\vartheta_{\underline{m}}\colon\MHilb^{\underline{m}}\arr\MCov$
is a smooth Zariski epimorphism onto the open substack $\MCov^{\underline{m}}$
of $\MCov$ of sheaves of algebras $\alA$ such that, for all $y\in Y$,
$\alA\otimes k(y)$ is generated in the degrees $m_{1},\dots,m_{r}$
as a $k(y)$-algebra. Moreover $\MHilb^{\underline{m}}$ is an open
subscheme of a vector bundle over $\MCov^{\underline{m}}$.\end{prop}
\begin{proof}
Let $\alA=\oplus_{m\in M}\alA_{m}\in\MCov$ and consider the map 
\[
\eta_{\alA}\colon\Sym(\alA_{m_{1}}\oplus\cdots\oplus\alA_{m_{r}})\arr\alA
\]
induced by the direct sum of the inclusions $\alA_{m_{i}}\arr\alA$.
It is easy to check that $\eta_{\alA}$ is surjective if and only
if $\alA\in\MCov^{\underline{m}}$. Therefore $\MCov^{\underline{m}}$
is an open substack of $\MCov$ and clearly contains the image of
$\vartheta_{\underline{m}}$. Consider now the cartesian diagram   \[   \begin{tikzpicture}[xscale=2.0,yscale=-1.2]     \node (A0_0) at (0, 0) {$F$};     \node (A0_1) at (1, 0) {$\MHilb^{\underline m}$};     \node (A1_0) at (0, 1) {$T$};     \node (A1_1) at (1, 1) {$\MCov^{\underline m}$};     \path (A0_0) edge [->] node [auto] {$\scriptstyle{}$} (A0_1);     \path (A1_0) edge [->] node [auto] {$\scriptstyle{\alA}$} (A1_1);     \path (A0_1) edge [->] node [auto] {$\scriptstyle{\vartheta_{\underline m}}$} (A1_1);     \path (A0_0) edge [->] node [auto] {$\scriptstyle{}$} (A1_0);   \end{tikzpicture}   \] and
let $U\arrdi{\phi}T$ be a map. The objects of $F(U)$ are pairs composed
by a graded surjection $\odi U[x_{1},\dots,x_{r}]\arr\alB$ and an
isomorphism $\alB\simeq\phi^{*}\alA$. This is equivalent to giving
a graded surjection $\odi U[x_{1},\dots,x_{r}]\arr\phi^{*}\alA$.
In this way we obtain a map
\[
F\arrdi{g_{T}}\prod_{i}\Homsh_{T}(\odi T,\alA_{m_{i}})\simeq\Spec\Sym(\bigoplus_{i}\alA_{m_{i}}^{-1})
\]
 We claim that this is an open immersion. Indeed given $(a_{i})_{i}\colon U\arr\prod_{i}\Homsh_{T}(\odi T,\alA_{m_{i}})$,
the fiber product with $F$ is the locus where the induced graded
map $\odi U[x_{1},\dots,x_{r}]\arr\alA\otimes\odi U$ is surjective,
that is an open subscheme of $U$. In particular $F$ is smooth over
$T$ and so $\vartheta_{\underline{m}}$ is smooth too. It is easy
to check that it is also a Zariski epimorphism. Finally the vector
bundle $\shN$ of the statement is defined over any $U\arr\MCov^{\underline{m}}$
given by $\alA=\bigoplus_{m}\alA_{m}$ by $\shN_{|U}=\oplus_{i}\alA_{m_{i}}^{-1}$.\end{proof}
\begin{rem}
\label{rem:relation MHilb MCov}If the sequence $\underline{m}$ contains
all elements of $M-\{0\}$, then $\MCov^{\underline{m}}=\MCov$. Therefore
in this case $\MHilb^{\underline{m}}$ is an atlas for $\MCov$.
\end{rem}

\begin{rem}
\label{rem: irreducibility and connectedness preserved by particular fppf epimorphism}
Let $X$ be a scheme, $\stX$ be an irreducible (resp. connected)
algebraic stack and $X\arrdi{\pi}\stX$ be a fppf epimorphism such
that the fiber over the generic point of $\stX$ is irreducible (resp.
such that $\pi$ is geometrically connected). Then $X$ is irreducible
(resp. connected). For the connectedness, if $X=U\cup V$, since $\pi$
is open and $|\stX|=|\pi(U)|\cup|\pi(V)|$, we have $\pi(U)\cap\pi(V)\neq\emptyset$.
In particular $U$ and $V$ meet a common fiber $Z$ of $\pi$. Since
$Z$ is connected we can conclude that $Z\cap U\cap V\neq\emptyset$.
For the irreducibility, consider a generic point $\xi\colon\Spec k\arr\stX$,
with $k$ algebraically closed, and denote by $Z$ the (topological)
image of $\Spec k\times_{\stX}X\arr X$. Note that $Z$ does not depend
on the choice of the generic point and it is irreducible by hypothesis.
If $V\subseteq X$ is a non-empty open subset of $X$, since $\pi$
is an open map, we can conclude that $V\cap Z\neq\emptyset$. Therefore
$Z$ is dense in $X$ and $X$ is irreducible.
\end{rem}

\begin{rem}
\label{rem:relation Mhilbm MCovm}The map $\vartheta_{\underline{m}}\colon\MHilb^{\underline{m}}\arr\MCov^{\underline{m}}$
of \ref{pro:MHlb --> MCov has irreducible fibers} is a smooth epimorphism
with geometrically connected and irreducible fibers. In particular,
taking into account \ref{rem: irreducibility and connectedness preserved by particular fppf epimorphism},
if $\stX$ is an algebraic stack, $\stX\arr\MCov^{\underline{m}}$
is a map and we denote by $\vartheta_{\underline{m}}^{-1}(\stX)$
the base change of $\vartheta_{\underline{m}}$ we have that: $\stX$
is connected (resp. geometrically connected, irreducible, geometrically
irreducible, smooth, reduced) if and only if $\vartheta_{\underline{m}}^{-1}(\stX)$
has the same property. The same conclusions hold if we consider the
atlas $\Spec R_{M}\arr\MCov$ instead of $\vartheta_{\underline{m}}$.

In particular, since $\Bi\Di M\subseteq\MCov^{\underline{m}}$, we
can conclude that $\vartheta_{\underline{m}}^{-1}(\stZ_{M})$ is the
main irreducible component of $\MHilb^{\underline{m}}$.
\end{rem}

We want now study some geometrical properties of the stack $\MCov$
and, therefore, of the equivariant Hilbert schemes.
\begin{rem}
\label{rem:reduced relations for RM}The ring $R_{M}$ can be written
as quotient of the ring $\Z[x_{m,n}]_{(m,n)\in J}$, where $J$ is
$\{(m,n)\in M^{2}\st m,n,m+n\neq0\}$ divided by the equivalence relation
$(m,n)\sim(n,m)$, by the ideal
\[
I=\left(\begin{array}{c}
x_{m,n}x_{m+n,t}-x_{n,t}x_{n+t,m}\text{ with }m,n,t,m+n,n+t,m+n+t\neq0\text{ and }m\neq t,\\
x_{-m,t}x_{-m+t,m}-x_{-m,s}x_{-m+s,m}\text{ with }m,s,t\neq0\text{ and distinct}
\end{array}\right)
\]
Indeed the first relations are trivial when one of $m,n,t$ is zero
or $m=t$, while if $m+n=0$ yield relations $x_{m,-m}=x_{-m,t}x_{-m+t,m}$.
Using these last relations we can remove all the variables $x_{m,n}$
with $0\in\{m,n,m+n\}$.
\end{rem}

\begin{rem}
\label{rem:N-graduation of RM}There exists a map $f\colon\tilde{K}_{+}\arr\N$
such that for any $m,n\neq0$ we have $f(e_{m,n})=1$ if $m+n\neq0$,
$f(e_{m,-m})=2$ otherwise. In particular $f(v)=0$ only if $v=0$.
Moreover $f$ induces an $\N$-graduation on both $(R_{M}\otimes A)$
and $\Z[K_{+}]\otimes A$, where $A$ is a ring, such that the degree
zero part is $A$ and that the elements $x_{m,n}$ with $m+n\neq0$
are homogeneous of degree $1$. The map $f$ is obtained as the composition
$\tilde{K}_{+}\arr K\subseteq\Z^{M}/\langle e_{0}\rangle\arrdi h\Z$,
where $h(e_{m})=1$ if $m\neq0$.
\end{rem}
One of the open problems in the theory of equivariant Hilbert schemes
is whether those schemes are connected. As said above $\MHilb^{\underline{m}}$
is connected if and only if $\MCov^{\underline{m}}$ is so. What we
can say here is:
\begin{thm}
\label{thm:Mcov geom connected}The stack $\MCov$ is connected with
geometrically connected fibers. If any non zero element of $M$ belongs
to the sequence $\underline{m}$, then $\MHilb^{\underline{m}}$ has
the same properties.\end{thm}
\begin{proof}
It is enough to prove that $\Spec R_{M}\otimes k$ is connected for
any field $k$. But $R_{M}\otimes k$ has an $\N$-graduation such
that $(R_{M}\otimes k)_{0}=k$ by \ref{rem:N-graduation of RM} and
it is a general fact that such an algebra does not contain non trivial
idempotents.
\end{proof}
We now want to discuss the problem of the reducibility of $\MCov$.
\begin{defn}
\label{def:universally reducible}Let $S$ be a scheme. An algebraic
stack $\stX$ is called \emph{universally reducible} over $S$ if,
for any base change $S'\arr S$, the stack $\stX\times_{S}S'$ is
reducible. An algebraic stack is universally reducible if it is so
over $\Z$.\end{defn}
\begin{rem}
\label{rem:universally reducible over fibers}It is easy to check
that $\stX$ is universally reducible over $S$ if and only if all
the fibers are reducible.\end{rem}
\begin{lem}
\label{lem:necessary condition MCov reducible}If there exist $m,n,t,a\in M$
such that
\begin{enumerate}
\item $m,n,t$ are distinct and not zero;
\item $a\neq0,m,n,t,m-n,n-m,n-t,t-n,m-t,2m-t,2n-t,m+n-t,m+n-2t$;
\item $2a\neq m+n-t$;
\end{enumerate}
then $\Spec R_{M}$ is universally reducible.\end{lem}
\begin{proof}
Let $k$ be a field and $I=(\underline{x}^{\alpha_{i}}-\underline{x}^{\beta_{i}})$
be an ideal of $k[x_{1},\dots,x_{r}]=k[\underline{x}]$. We will say
that $\alpha\in\N^{r}$ is transformable (with respect to $I$) if
there exists $i$ such that $\alpha_{i}\leq\alpha$ or $\beta_{i}\leq\alpha$.
Here by $\alpha\leq\beta\in\N^{r}$ we mean $\alpha_{j}\leq\beta_{j}$
for all $j$. A direct computation shows that if $\underline{x}^{\alpha}-x^{\beta}\in I$
and $\alpha\neq\beta$, then both $\alpha$ and $\beta$ are transformable. 

We will use the above notation for the ideal $I$ defining $R_{M}\otimes k$
as in \ref{rem:reduced relations for RM}. In particular the elements
$\alpha_{i},\beta_{i}\in\N^{J}$ associated to the ideal $I$ are
of the form $e_{u,v}+e_{u+v,w}$ with $u,v,u+v,w,u+v+w\neq0$. 

Set $\mu=\prod_{m,n}x_{m,n}$. Since $R_{M}\otimes k\arr k[K_{+}]\subseteq k[K]=(R_{M}\otimes k)_{\mu}$,
there exists $N>0$ such that $P=\Ker(R_{M}\otimes k\arr k[K_{+}])=\Ann\mu^{N}$.
Our strategy will be to find an element of $P$ which is not nilpotent.
Since $P$ is a minimal prime, being $\Spec k[K_{+}]$ an irreducible
component of $\Spec R_{M}\otimes k$, it follows that $R_{M}\otimes k$
is reducible. Now consider $\alpha=e_{a,m-a}+e_{m+n-t-a,t+a-m}+e_{t+a-n,n-a}$,
$\beta=e_{m+n-t-a,t+a-n}+e_{a,n-a}+e_{m-a,t+a-m}\in\N^{J}$ and $z=\underline{x}^{\alpha}-\underline{x}^{\beta}$.
We will show that $\mu z=0$, i.e. $z\in P$ and that $z$ is not
nilpotent. First of all note that $z$ is well defined since for any
$e_{u,v}$ in $\alpha$ or $\beta$ we have $u,v\neq0$ and $0\neq u+v\in\{m,n,t\}$
thanks to $1)$, $2)$. Let $S_{M}$ be the universal algebra over
$R_{M}$, i.e. $S_{M}=\bigoplus_{m\in M}R_{M}v_{m}$ with $v_{m}v_{n}=x_{m,n}v_{m+n}$
and $v_{0}=1$. By construction we have
\begin{alignat*}{1}
(v_{a}v_{m-a})(v_{m+n-t-a}v_{t+a-m})(v_{t+a-n}v_{n-a}) & =\underline{x}^{\alpha}v_{m}v_{n}v_{t}=\\
(v_{m+n-t-a}v_{t+a-n})(v_{a}v_{n-a})(v_{m-a,t+a-m}) & =\underline{x}^{\beta}v_{m}v_{n}v_{t}
\end{alignat*}
So $\underline{x}^{\alpha}x_{m,n}x_{m+n,t}v_{m+n+t}=\underline{x}^{\beta}x_{m,n}x_{m+n,t}v_{m+n+t}$
and therefore $z\mu=0$, i.e. $z\in P$.

Now we want to prove that any linear combination $\gamma=a\alpha+b\beta\in\N^{J}$
with $a,b\in\N$ is not transformable. First remember that each $e_{u,v}$
in $\gamma$ is such that $u+v\in\{m,n,t\}$. If we have $e_{u,v}+e_{u+v,w}\leq\gamma$
then there must exist $e_{i,j}\leq\gamma$ such that $i\in\{m,n,t\}$
or $j\in\{m,n,t\}$. Condition $2)$ is exactly what we need to avoid
this situation and can be written as $\{a,m-a,m+n-t-a,t+a-m,t+a-n,n-a\}\cap\{m,n,t\}=\emptyset$.

In particular, if we think of $\tilde{K}_{+}$ as a quotient of $\N^{J}$,
we have $a\alpha+b\beta=a'\alpha+b'\beta$ in $\tilde{K}_{+}$ if
and only if they are equal in $\N^{J}$. Assume for a moment that
$\alpha\neq\beta$ in $\N^{J}$. Clearly this means that $\alpha$
and $\beta$ are $\Z$-independent in $\Z^{J}$. Since any linear
combination of $\alpha$ and $\beta$ is not transformable, it follows
that $\underline{x}^{\alpha},\underline{x}^{\beta}$ are algebraically
independent over $k$ in $R_{M}\otimes k$ and, in particular, that
$z=\underline{x}^{\alpha}-\underline{x}^{\beta}$ cannot be nilpotent.
So it remains to prove that $\alpha\neq\beta$ in $\N^{J}$. Note
that for any $i\in\{m,n,t\}$ there exists only one $e_{u,v}$ in
$\alpha$ such that $u+v=i$ and the same happens for $\beta$. So,
if $\alpha=\beta$ and since $m,n,t$ are distinct, those terms have
to be equal, for instance $e_{a,m-a}=e_{m+n-t-a,t+a-n}$. But $a\neq m+n-t-a$
by $3)$, while $a\neq t+a-n$ since $t\neq n$. Therefore $\alpha\neq\beta$.\end{proof}
\begin{cor}
\label{cor:Mcov reducible}If $|M|>7$ and $M\not\simeq(\Z/2\Z)^{3}$
then $\MCov$ is universally reducible and the same holds for $\MHilb^{\underline{m}}$,
provided that $\underline{m}$ contains all elements of $M-\{0\}$.\end{cor}
\begin{proof}
We have to show that $R_{M}$ is universally reducible and so we will
apply \ref{lem:necessary condition MCov reducible}. If $M=C\times T$,
where $C$ is cyclic with $|C|\geq4$ and $T\neq0$ we can choose:
$m$ a generator of $C$, $n=3m$, $t=2m$ and $a\in T-\{0\}$. If
$M$ cannot be written as above, there are four remaining cases. $1)$
$M\simeq\Z/8\Z$: choose $m=2\comma n=4\comma t=6\comma a=1$. $2)$
$M$ cyclic with $|M|>8$ and $|M|\neq10$: choose $m=1\comma n=2\comma t=3\comma a=5$.
$3)$ $M\simeq(\Z/2\Z)^{l}$ with $l\geq4$: choose $m=e_{1}\comma n=e_{2}\comma t=e_{3}\comma a=e_{4}$.
$4)$ $M\simeq(\Z/3\Z)^{l}$ with $l\geq2$: choose $m=e_{1}\comma n=2e_{1}\comma t=e_{2}\comma a=m+t=e_{1}+e_{2}$.\end{proof}
\begin{prop}
\label{pro:smooth DMCov}$\MCov$ is smooth if and only if $\stZ_{M}$
is so. This happens if and only if $M\simeq\Z/2\Z,\Z/3\Z,\Z/2\Z\times\Z/2\Z$
and in these cases $\MCov=\stZ_{M}$. To be more precise $R_{M}=\Z[x_{m,n}]_{(m,n)\in J}$,
where $J$ is the set defined in \ref{rem:reduced relations for RM}.

In particular $\MHilb^{\underline{m}}$ is smooth and irreducible
for any sequence $\underline{m}$ if $M$ is as above. Otherwise,
if any non zero element of $M$ belongs to the sequence $\underline{m}$,
$\MHilb^{\underline{m}}$ is not smooth.\end{prop}
\begin{proof}
Let $k$ be a field. Note that 
\[
\MCov\text{ smooth}\iff R_{M}\text{ smooth}\then\stZ_{M}\text{ smooth }\then k[K_{+}]/k\text{ smooth}
\]
We first prove that if $k[K_{+}]$ is smooth then $M$ has to be one
of the groups of the statement. We have $K_{+}\simeq\N^{r}\oplus\Z^{s}$
and therefore $k[K_{+}]$ is UFD. We will consider $k[K_{+}]$ endowed
with the $\N$-graduation defined in \ref{rem:N-graduation of RM}.
Since any of the $x_{m,n}$ has degree $1$, it is irreducible and
so prime. If we have a relation $x_{m,n}x_{m+n,t}=x_{n,t}x_{n+t,m}$
with $m,n,t,m+n,n+t,m+n+t\neq0$ and $m\neq t$, then $x_{m,n}\mid x_{n,t}x_{n+t,m}$
implies that $x_{m,n}=x_{n,t}$ or $x_{m,n}=x_{n+t,m}$, which is
impossible thanks to our assumptions. We will prove that if $M$ is
not isomorphic to one of the group in the statement, then such a relation
exists. Clearly it is enough to find this relation in a subgroup of
$M$. So it is enough to consider the following cases. $1)$ $M$
cyclic with $|M|\geq5$: choose $m=n=1\comma t=2$. $2)$ $M\simeq\Z/4\Z$:
choose $m=1\comma n=2\comma t=3$. $3)$ $M\simeq(\Z/2\Z)^{3}$: choose
$m=e_{1}\comma n=e_{2}\comma t=e_{3}$. $4)$ $M\simeq(\Z/3\Z)^{2}$:
choose $m=n=e_{1}\comma t=e_{2}$.

We now want to prove that when $M$ is as in the statement, then the
ideal $I$ of \ref{rem:reduced relations for RM} is zero. If we have
a relation as in the first row, since $m\neq t$ we have $|M|\geq3$.
If $M\simeq\Z/3\Z$ then $t=2m$ and $m+t=0$. If $M\simeq(\Z/2\Z)^{2}$,
if $m,n,t$ are distinct then $m+n+t=0$, otherwise $m=n$ and $m+n=0$.
If we have a relation as in the second row, since $m,t,s$ are distinct,
we must have $M\simeq(\Z/2\Z)^{2}$. Therefore $m+t=s$ and the relation
become trivial.\end{proof}
\begin{cor}
\label{cor:description of mutwotimesmutwo Cov}The stack $\GrCov{\Z/2\Z\times\Z/2\Z}$
is isomorphic to the stack of sequences $(\shL_{i},\psi_{i})_{i=1,2,3}$,
where $\shL_{1},\shL_{2},\shL_{3}$ are invertible sheaves and $\psi_{1}\colon\shL_{2}\otimes\shL_{3}\arr\shL_{1}$,
$\psi_{2}\colon\shL_{1}\otimes\shL_{3}\arr\shL_{2}$, $\psi_{3}\colon\shL_{1}\otimes\shL_{2}\arr\shL_{3}$
are maps.\end{cor}
\begin{proof}
Set $M=(\Z/2\Z)^{2}$. Thanks to \ref{pro:smooth DMCov}, we know
that $\tilde{K}_{+}=K_{+}\simeq\N v_{e_{1},e_{2}}\oplus\N v_{e_{1},e_{1}+e_{2}}\oplus\N v_{e_{2},e_{1}+e_{2}}$.
So an object of $\MCov$ is given by invertible sheaves $\shL_{1}=\shL_{e_{1}}\comma\shL_{2}=\shL_{e_{2}}\comma\shL_{3}=\shL_{e_{1}+e_{2}}$
and maps $\psi_{1}=\psi_{e_{2},e_{1}+e_{2}}\comma\psi_{2}=\psi_{e_{1},e_{1}+e_{2}}\comma\psi_{3}=\psi_{e_{1},e_{2}}$.\end{proof}
\begin{rem}
\label{rem:MCov for M=00003DZfour}$\GrCov{\Z/4\Z}$ and $\Z/4\Z\textup{-Hilb}^{\underline{m}}$,
for any sequence $\underline{m}$, are integral and normal since one
can check directly that $R_{\Z/4\Z}=\Z[x_{1,2},x_{3,3},x_{2,3},x_{1,1}]/(x_{1,2}x_{3,3}-x_{2,3}x_{1,1})$.
I am not able to prove that $\MCov$ is irreducible when $M$ is one
of $\Z/5\Z\comma\Z/6\Z$, $\Z/7\Z\comma(\Z/2\Z)^{3}$. Anyway the
first two cases seem to be integral thanks to a computer program,
while for the last ones there are some techniques that can be used
to study this problem but they are too complicated to be explained
here.
\end{rem}

\subsection{The invariant $h\colon|\MCov|\arr\N$.}

In this subsection we investigate the local structure of a $\Di M$-cover,
especially over a local ring. In particular we will define an upper
semicontinuous map $h\colon|\MCov|\arr\N$ that measures how much
a cover fails to be a torsor: the open locus $\Bi\Di M\subseteq\MCov$
will be exactly the locus $\{h=0\}$.
\begin{notation}
Given a ring $A$, we will write $B\in\Spec R_{M}(A)$ meaning that
$B$ is an $M$-graded $A$-algebra with a given $M$-graded basis,
usually denoted by $\{v_{m}\}_{m\in M}$ with $v_{0}=1$, and a given
multiplication $\psi$ such that
\[
B=\bigoplus_{m\in M}Av_{m}\comma v_{m}v_{n}=\psi_{m,n}v_{m+n}
\]
We will also denote by $A^{*}$ the group of invertible elements of
$A$. If $f\colon X\arr Y$ is an affine map of schemes and $q\in Y$,
we will use the notation $\odi{X,q}=f_{*}\odi X\otimes_{\odi Y}\odi{Y,q}$.
In particular $X\times_{Y}\Spec\odi{Y,q}\simeq\Spec\odi{X,q}$. Notice
that, although $\odi{X,q}$ is written as a localization in a point,
this ring is not local in general.\end{notation}
\begin{lem}
\label{lem:factorization of covers through torsors on local rings}Let
$A$ be a ring and $B\in\Spec R_{M}(A)$, with graded basis $v_{m}$
and multiplication map $\psi$. Then the set
\[
H_{\psi}=H_{B/A}=\{m\in M\;|\; v_{m}\in B^{*}\}=\{m\in M\;|\;\psi_{m,-m}\in A^{*}\}
\]
 is a subgroup of $M$. Moreover if $m,n\in M$ and $h\in H_{\psi}$
then $\psi_{m,n}$ and $\psi_{m,n+h}$ differs by an element of $A^{*}$.
If $H$ is a subgroup of $H_{\psi}$ then $C=\bigoplus_{m\in H}Av_{m}$
is an element of $\Bi\Di H(A)$. Moreover if $\sigma\colon M/H\arr M$
gives representatives of $M/H$ in $M$ and we set $w_{m}=v_{\sigma(m)}$
for $m\in M/H$ we have
\[
B=\bigoplus_{m\in M/H}Cw_{m}\in\Spec R_{M/H}(C)
\]
 Finally if we denote by $\psi'$ the induced multiplication on $B$
over $C$ we have $H_{\psi'}=H_{\psi}/H$ and for any $m,n\in M$
$\psi'_{m,n}$ and $\psi_{m,n}$ differ by an element of $C^{*}$.\end{lem}
\begin{proof}
From the relations $v_{m}v_{-m}=\psi_{m,-m}$, $v_{m}^{|M|-1}=\lambda v_{-m}$,
$v_{m}^{|M|}=\lambda\psi_{m,-m}$, where $\lambda\in B$ and $v_{m}v_{n}=\psi_{m,n}v_{m+n}$
we see that $v_{m}\in B^{*}\iff\psi_{m,-m}\in A^{*}$ and that $H_{\psi}<M$.
From \ref{eq:condition on psi} we get the relations $\psi_{-h,h}=\psi_{h,u}\psi_{h+u,-h}$
and $\psi_{m,n}\psi_{m+n,h}=\psi_{n,h}\psi_{m,n+h}$. So if $h\in H$
then $\psi_{h,u}\in A^{*}$ for any $u$ and $\psi_{m,n}$ and $\psi_{m,n+h}$
differ by an element of $A^{*}$.

Now consider the second part of the statement. From \ref{pro:equivalent conditions for a D(M)-torsor}
we know that $C$ is a torsor over $A$. Since for any $m$ we have
$v_{m}=(\psi_{h,m}/v_{h})v_{\sigma(\overline{m})}$, where $h=\sigma(\overline{m})-m\in H$
we obtain the expression of $B$ as $M/H$ graded $C$-algebra and
that
\[
\psi'_{m,n}=\psi_{\sigma(m),\sigma(n)}(\psi_{h,\sigma(m)+\sigma(n)}/v_{h})\text{ where }h=\sigma(m+n)-\sigma(m)-\sigma(n)
\]
From the above equation it is easy to conclude the proof.\end{proof}
\begin{defn}
\label{def:of the maximal torsor and the subgroup associated}Given
a ring $A$ and $B\in\Spec R_{M}(A)$ we continue to use the notation
$H_{B/A}$ introduced in \ref{lem:factorization of covers through torsors on local rings}
and we will call the algebra $C$ obtained for $H=H_{B/A}$ the \emph{maximal
torsor} of the extension $B/A$. If $k$ is a field and $\E\in\duale K_{+}$
we will write $H_{\E}=H_{B/k}$, where $B$ is the algebra induced
by the multiplication $0^{\E}$. In particular
\[
H_{\E}=\{m\in M\st\E_{m,-m}=0\}
\]
Finally if $f\colon X\arr Y\in\MCov(Y)$ and $q\in Y$ we define $\shH_{f}(q)=H_{\odi{X,q}/\odi{Y,q}}$.\end{defn}
\begin{prop}
\label{pro:the H* is well defined as a map from MCov}We have a map
  \[   \begin{tikzpicture}[xscale=3.5,yscale=-0.5]     \node (A0_0) at (0, 0) {$|\MCov|$};     \node (A0_1) at (1, 0) {$\{\text{subgroups of }M\}$};     \node (A1_0) at (0, 1) {$B/k$};     \node (A1_1) at (1, 1) {$H_{B/k}$};     \path (A0_0) edge [->]node [auto] {$\scriptstyle{\shH}$} (A0_1);     \path (A1_0) edge [|->,gray]node [auto] {$\scriptstyle{}$} (A1_1);   \end{tikzpicture}   \] such
that, if $Y\arrdi u\MCov$ is given by $X\arrdi fY$, then $\shH_{f}=\shH\circ|u|$.\end{prop}
\begin{proof}
It is enough to note that if $A$ is a local ring, $B\in\MCov(A)$
is given by multiplications $\psi$ and $\pi\colon A\arr A/m_{A}\arr k$
is a morphism, where $k$ is a field, then $\psi_{m,-m}\in A^{*}\iff\pi(\psi_{m,-m})\neq0$.\end{proof}
\begin{rem}
\label{rem: H is zero implies the algebra is local}Let $(A,m_{A})$
be a local ring and $B\in\Spec R_{M}(A)$ with $M$-graded basis $\{v_{m}\}_{m\in M}$.
Then $H_{B/A}=\shH_{B/A}(m_{A})$. If $H_{B/A}=0$ then any $v_{m}$,
with $m\neq0$, is nilpotent in $B\otimes k$ and therefore $B$ is
local with maximal ideal
\[
m_{B}=m_{A}\oplus\bigoplus_{m\in M-\{0\}}Av_{m}
\]
and residue field $B/m_{B}=A/m_{A}$. In particular $m_{B}/m_{B}^{2}$
is $M$-graded.\end{rem}
\begin{lem}
\label{lem:equivalent condition for an algebra to be generated in degrees m1,...,mr}Let
$A$ be a local ring and $B=\bigoplus_{m\in M}Av_{m}\in\MCov(A)$
such that $H_{B/A}=0$. If $m_{1},\dots,m_{r}\in M$ then $B$ is
generated in degrees $m_{1},\dots,m_{r}$ as an $A$-algebra if and
only if $m_{B}=(m_{A},v_{m_{1}},\dots,v_{m_{r}})_{B}$.\end{lem}
\begin{proof}
We can write $m_{B}=m_{A}\oplus\bigoplus_{m\in M-\{0\}}Av_{m}$. Denote
$\underline{v}=v_{m_{1}},\dots,v_{m_{r}}$ and $\pi(\alpha)=\sum_{i}\alpha_{i}m_{i}$
for $\alpha\in\N^{r}$. The {}``only if'' follows since given $l\in M-\{0\}$
there exists a relation of the form $v_{l}=\mu\underline{v}^{\alpha}$
with $\mu\in A^{*}$ and $\alpha\neq0$ and so $v_{l}\in(m_{A},v_{m_{1}},\dots,v_{m_{r}})_{B}$.
For the converse note that, given $l\in M-\{0\}$, $v_{l}\in m_{B}=(m_{A},v_{m_{1}},\dots,v_{m_{r}})$
means that we have a relation $v_{l}=\lambda v_{l'}v_{m_{i}}$ for
some $i$, $\lambda\in A^{*}$ and $l'=l-m_{i}$. Moreover $v_{l}\notin A[\underline{v}]$
implies that $v_{l'}\notin A[\underline{v}]$ and $l'\neq0$. If,
by contradiction, we have such an element $l$ we can write $v_{l}=\mu v_{n_{1}}\cdots v_{n_{s}}$
with $n_{i}\in M-\{0\}$ and $s\geq|M|^{2}$. In particular there
must exist $i$ such that $m=n_{i}$ appears at least $|M|$ times
in this product. So $m_{A}\ni v_{m}^{|M|}\mid v_{l}$ and $v_{l}\in m_{A}B$,
which is not the case.
\end{proof}
Assume we have a cover $X\arrdi fY\in\MCov(Y)$. We want to define,
for any $m\in M$ a map $h_{f,m}=h_{X/Y,m}\colon Y\arr\{0,1\}$. Let
$q\in Y$ and denote by $C$ the 'maximal torsor' of $\odi{X,q}/\odi{Y,q}$
(see \ref{def:of the maximal torsor and the subgroup associated}).
Also let $p\in f^{-1}(q)$ and set $p_{C}=p\cap C$. Taking into account
\ref{rem: H is zero implies the algebra is local}, we know that $B=(\odi{X,q})_{p}=(\odi{X,q})_{p_{C}}$
and that $B\in\GrCov{M/\shH_{f}(q)}(C_{p_{C}})$ with $H_{B/C_{p_{C}}}=0$.
Moreover $B$ is local, $B/m_{B}=C_{p_{C}}/p_{C}$ and $m_{B}/m_{B}^{2}$
is $(M/\shH_{f}(q))$-graded. If we denote by $\overline{m}$ the
image of $m\in M$ in $M/\shH_{f}(q)$ and by $(m_{B}/m_{B}^{2})_{t}$
the graded pieces of $m_{B}/m_{B}^{2}$, where $t\in M/\shH_{f}(q)$,
we can define:
\begin{defn}
With notation above we set 
\[
h_{f,m}(q)=\left\{ \begin{array}{cc}
0 & \text{ if }m\in\shH_{f}(q)\\
\dim_{C_{p_{C}}/p_{C}}(m_{B}/m_{B}^{2})_{\overline{m}} & \text{otherwise}
\end{array}\right.
\]
We also set 
\[
h_{f}(q)=\dim_{C_{p_{C}}/p_{C}}(m_{B}/m_{B}^{2})-\dim_{C_{p_{C}}/p_{C}}(m_{B}/m_{B}^{2})_{0}=(\sum_{m\in M}h_{f,m}(q))/|\shH_{f}(q)|
\]
If $\E\in\duale K_{+}$ we set $h_{\E,m}=h_{f,m}\comma h_{\E}=h_{f}\in\N$
where $f$ is the cover $\Spec A\arr\Spec k$ and $A$ is the algebra
given by multiplication $0^{\E}$ over some field $k$.
\end{defn}
The following lemma shows that the value of $h_{f,m}(q)$ does not
depend on the choice of the point $p\in X$ over $q\in Y$. 
\begin{lem}
\label{lem:showing that hm and h are well defined}Let $(A,m_{A})$
be a local ring, $B\in\MCov(A)$ given by the multiplication $\psi$
and $t\in M$. Set also $h_{B/A,t}=h_{B/A,t}(m_{A})$, for some choice
of a prime of $B$ over $m_{A}$. Then $h_{B/A,t}=1$ if and only
if the following conditions are satisfied:
\begin{itemize}
\item $t\notin H_{B/A}$;
\item for all $u,n\in M-H_{B/A}$ such that $u+n\equiv t\text{ mod }H_{B/A}$
we have $\psi_{u,n}\notin A^{*}$.
\end{itemize}
\end{lem}
\begin{proof}
Let $C$ be the maximal torsor of the extension $B/A$ and $p$ be
a maximal prime of $B$. We use notation from \ref{lem:factorization of covers through torsors on local rings}.
For any $l\in M-H_{B/A}$ we have a surjective map
\[
k(p)=(m_{B_{p}}/pC_{p})_{\overline{l}}\arr(m_{B_{p}}/m_{B_{p}}^{2})_{\overline{l}}
\]
and so $\dim_{k(p)}(m_{B_{p}}/m_{B_{p}}^{2})_{\overline{l}}\in\{0,1\}$,
where $\overline{l}$ is the image of $l$ under the projection $M\arr M/H_{A/B}$.
If we prove the last part of the statement clearly we will also have
that $h_{B/A,t}$ is well defined. If $t\in H_{B/A}$ then $h_{B/A,t}=0$,
while if there exist $u,n$ as in the statement such that $\psi_{u,n}\in A^{*}$,
then $w_{\overline{t}}\in C_{p}^{*}w_{\overline{u}}w_{\overline{n}}\subseteq m_{B_{p}}^{2}$
and again $h_{B/A,t}=0$. On the other hand if $h_{B/A,t}=0$ and
$t\notin H_{B/A}$ then $w_{\overline{t}}\in m_{B_{p}}^{2}$ and therefore
we have an expression
\[
w_{\overline{t}}=bx+\sum_{\overline{u},\overline{n}\neq0}b_{\overline{u},\overline{n}}w_{\overline{u}}w_{\overline{n}}\text{ with }b,b_{\overline{u},\overline{n}}\in B_{p},x\in m_{C_{p}}
\]
The second sum splits as a sum of products of the form $c_{s,\overline{u},\overline{n}}w_{s}w_{\overline{u}}w_{\overline{n}}$
with $s+\overline{u}+\overline{n}=\overline{t}$ and $c_{s,\overline{u},\overline{n}}\in C_{p}$.
Since $C_{p}$ is local, one of these monomials generates $C_{p}w_{\overline{t}}$.
In this case, if $s+\overline{u}=0$ then $\overline{u}\in H_{B_{p}/C_{p}}=0$
which is not the case. So we have an expression 
\[
w_{\overline{t}}=\lambda w_{\overline{u}}w_{\overline{n}}=\lambda\psi'_{\overline{u},\overline{n}}w_{\overline{t}}\then\psi'_{\overline{u},\overline{n}}\in C_{p}^{*}
\]
where $\overline{u},\overline{n}\neq0$ and $\overline{u}+\overline{n}=\overline{t}$.
Since $\psi'_{\overline{u},\overline{n}}$ and $\psi_{u,n}$ differs
by an element of $C^{*}$ thanks to \ref{lem:factorization of covers through torsors on local rings},
it follows that $\psi_{u,n}\in A^{*}$.\end{proof}
\begin{prop}
We have maps   \[   \begin{tikzpicture}[xscale=3.0,yscale=-0.5]     \node (A0_0) at (0, 0) {$|\MCov|$};     \node (A0_1) at (1, 0) {$\{0,1\}$};     \node (A0_2) at (2, 0) {$|\MCov|$};     \node (A0_3) at (3, 0) {$\N$};     \node (A1_0) at (0, 1) {$B/k$};     \node (A1_1) at (1, 1) {$h_{B/k,m}$};     \node (A1_2) at (2, 1) {$B/k$};     \node (A1_3) at (3, 1) {$h_{B/k}$};     \path (A0_0) edge [->]node [auto] {$\scriptstyle{h_m}$} (A0_1);     \path (A1_0) edge [|->,gray]node [auto] {$\scriptstyle{}$} (A1_1);     \path (A0_2) edge [->]node [auto] {$\scriptstyle{h}$} (A0_3);     \path (A1_2) edge [|->,gray]node [auto] {$\scriptstyle{}$} (A1_3);   \end{tikzpicture}   \] 
such that, if $Y\arrdi u\MCov$ is given by $X\arrdi fY$, then $h_{f,m}=h_{m}\circ|u|$
and $h_{f}=h\circ|u|$.\end{prop}
\begin{proof}
Taking into account \ref{lem:showing that hm and h are well defined}
and \ref{pro:the H* is well defined as a map from MCov}, it is enough
to note that if $A$ is a local ring, $B\in\MCov(A)$ is given by
multiplications $\psi$ and $\pi\colon A\arr A/m_{A}\arr k$ is a
morphism, where $k$ is a field, then $\psi_{u,v}\in A^{*}\iff\pi(\psi_{u,v})\neq0$
and $H_{B/A}=H_{B\otimes_{A}k/k}$.\end{proof}
\begin{cor}
\label{cor:An algebra with H =00003D 0 is generated in the degrees where h=00003D1}Under
the hypothesis of \ref{lem:equivalent condition for an algebra to be generated in degrees m1,...,mr},
$\{m\in M\;|\; h_{B/A,m}=1\}$ is the minimum among the subsets $Q$
of $M$ such that $B$ is generated as an $A$-algebra in the degrees
$Q$. In particular $B$ is generated in $h_{B/A}$ degrees.\end{cor}
\begin{prop}
Let $(A,m_{A})$ be a local ring, $B\in\MCov(A)$ and $C$ the maximal
torsor of $B/A$. Then
\[
h_{B/A}(m_{A})=\dim_{k(p)}\Omega_{B/C}\otimes_{B}k(p)
\]
for any maximal prime $p$ of $B$. In particular if $(|H_{B/A}|,\car A/m_{A})=1$
we also have $h_{B/A}(m_{A})=\dim_{k(p)}\Omega_{B/A}\otimes_{B}k(p)$
for any maximal prime $p$ of $B$.\end{prop}
\begin{proof}
If $A$ is any ring and $B\in\MCov(A)$ is given by basis $\{v_{m}\}_{m\in M}$
and multiplication $\psi$ one sees from the universal property that
\[
\Omega_{B/A}=B^{M}/\langle e_{0},v_{n}e_{m}+v_{m}e_{n}-\psi_{m,n}e_{m+n}\rangle
\]
Now consider $B\in\GrCov{M/H}(C)$, where $H=H_{B/A}$ and let $p$
be a maximal prime of $B$. Following the notation of \ref{lem:factorization of covers through torsors on local rings},
we have that $w_{m}\in p$ for any $m\in M/H-\{0\}$ and $\psi_{m,n}'\in p\iff\psi_{m,n}\in m_{A}$.
So $\Omega_{B/C}\otimes_{B}k(p)$ is free on the $e_{m}$ for $m\in M/H-\{0\}$
such that for any $u,n\in M/H-\{0\}$, $u+n=m$ implies $\psi_{u,n}\notin A^{*}$,
that are exactly $h_{B/A}(m_{A})$ thanks to \ref{lem:showing that hm and h are well defined}.\end{proof}
\begin{cor}
The function $h$ is upper semicontinuous.\end{cor}
\begin{proof}
Let $X\arrdi fY$ be a $\Di M$-cover and $q\in Y$. Set $r=h_{f}(q)$
and $H=\shH_{f}(q)$. We can assume that $Y=\Spec A$, $X=\Spec B$
with graded basis $\{v_{m}\}_{m\in M}$ and multiplication $\psi$
and that $\psi_{m,-m}\in A^{*}$ for any $m\in H$. Set $C=A[v_{m}]_{m\in H}$.
The ring $C_{q}$ is the maximal torsor of $B_{q}/A_{q}$ and so,
if $p\in X$ is a point over $q$, we have $r=\dim_{k(p)}\Omega_{B/C}\otimes_{B}k(p)$.
Finally let $U\subseteq X$ be an open neighborhood of $p$ such that
$\dim_{k(p')}\Omega_{B/C}\otimes_{B}k(p')\leq r$ for any $p'\in U$
and $V=f(U)$. We want to prove that $h\leq r$ on $V$. Indeed given
$q'=f(p')\in V$, if $D$ is the maximal torsor of $B_{q'}/A_{q'}$,
we have $C_{q'}\subseteq D\subseteq B_{q'}$. So 
\[
h_{f}(q')=\dim_{k(p')}\Omega_{B_{q'}/D}\otimes_{B_{q'}}k(p')\leq\dim_{k(p')}\Omega_{B_{q'}/C_{q'}}\otimes_{B_{q'}}k(p')\leq r
\]
\end{proof}
\begin{rem}
\label{rem:The 0 section and the 0 algebras}The $0$ section $R_{M}\arr\Z$,
i.e. the map that sends any $x_{m,n}$ with $m,n\neq0$ to zero, induces
a closed immersion
\[
\Picsh^{|M|-1}\simeq\Bi\shT=[\Spec\Z/\shT]\subseteq[\Spec R_{M}/\shT]\simeq\MCov
\]
where $\shT=\Di{\Z^{M}/\langle e_{0}\rangle}$.\end{rem}
\begin{prop}
The following results hold:
\begin{enumerate}
\item $\{h=0\}=|\Bi\Di M|$;
\item $\{h\geq|M|\}=\emptyset$;
\item $\{h=|M|-1\}=|\Bi\Di{\Z^{M}/\langle e_{0}\rangle}|$ (see \ref{rem:The 0 section and the 0 algebras})
\end{enumerate}
\end{prop}
\begin{proof}
If $X\arrdi fY$ is a $\Di M$-torsor, clearly $h_{f}=0$. So $1)$
and $2)$ follow from \ref{cor:An algebra with H =00003D 0 is generated in the degrees where h=00003D1}.
Finally, if $B\in\MCov(k)$ with multiplication $\psi$, $h_{B/k}=|M|-1$
if and only if $H_{B/k}=0$ and $h_{B/k,m}=1$ for all $m\in M-\{0\}$.
This means that $\psi_{m,n}=0$ for any $m,n\neq0$ by \ref{lem:showing that hm and h are well defined}.
\end{proof}
In particular, setting $U_{i}=\{h\leq i\}$, we obtain a stratification
$\Bi\Di M=U_{0}\subseteq U_{1}\subseteq\cdots\subseteq U_{|M|-1}=\MCov$
of $\MCov$ by open substacks.

\subsection{The locus $h\leq1$.}

In this subsection we want to describe $\Di M$-covers with $h\leq1$.
This means that 'up to torsors' we have a graded $M$-algebra generated
over the base ring in one degree. We will see that $\{h\leq1\}$ is
a smooth open substack of $\stZ_{M}$ determined by a special class
of explicit smooth  extremal rays of $K_{+}$. This will allow us
to give a description of normal $\Di M$-covers over locally noetherian
and locally factorial scheme $X$ with $(\car X,|M|)=1$. Such a description,
when $X$ is a smooth algebraic variety over an algebraic closed field
$k$ was already given in \cite[Theorem 2.1, Corollary 3.1]{Pardini1991}.
\begin{notation}
Given $\E\in\duale K_{+}$ we will write $\E_{m,n}=\E(v_{m,n})$.
Since $K\otimes\Q\simeq\Q^{M}/\langle e_{0}\rangle$ we will also
write $\E_{m}=\E(e_{m})\in\Q$, so that $\E_{m,n}=\E_{m}+\E_{n}-\E_{m+n}$.
When we will have to consider different abelian groups, we will write
$K_{+M}$, $K_{M}$ instead of, respectively, $K_{+}$, $K$, in order
to avoid confusion. Given a group homomorphism $\eta\colon M\arr N$
we will denote by $\eta_{*}\colon K_{M}\arr K_{N}$ the homomorphism
such that $\eta_{*}(v_{m,n})=v_{\eta(m),\eta(n)}$ for all $m,n\in M$,
where $K_{M}$ is the group associated to $K_{+}$,\end{notation}
\begin{rem}
Let $A$ be a ring and consider a sequence $\underline{\E}=\E^{1},\dots,\E^{r}\in\duale K_{+}$.
An element of $\stF_{\underline{\E}}(A)$ coming from the atlas (see
\ref{rem:description of objects of FE}) is given by a pair $(\underline{z},\lambda)$
where $\underline{z}=z_{1},\dots,z_{r}\in A$ and $\lambda\colon K\arr A^{*}$.
The image of this object under $\pi_{\underline{\E}}$ is the algebra
whose multiplication is given by $\psi_{m,n}=\lambda_{m,n}^{-1}z_{1}^{\E_{m,n}^{1}}\cdots z_{r}^{\E_{m,n}^{r}}$.\end{rem}
\begin{lem}
\label{lem:comparison smooth sequences for M covers}Let $\eta\colon M\arr N$
be a surjective morphism and $\underline{\E}$ be a sequence in $\duale{(K_{+N})}$.
Then $\underline{\E}$ is a smooth sequence for $N$ if and only if
$\underline{\E}\circ\eta_{*}$ is a smooth sequence for $M$.\end{lem}
\begin{proof}
We want to apply \ref{lem:comparison smooth sequences for different monoids}.
Therefore we have to prove that $\eta_{*}(K_{+M})=K_{+N}$, which
is clear, and that $\Ker\eta_{*}=\langle\Ker\eta_{*}\cap K_{+N}\rangle$.
Consider the map $f\colon\Z^{M}/\langle e_{0}\rangle\arr\Z^{N}/\langle e_{0}\rangle$
given by $f(e_{m})=e_{\eta(m)}$ and set $H=\Ker\eta$. Clearly $f_{|K_{M}}=\eta_{*}$.
It is easy to check that $G=\langle v_{m,n}\text{ for }m\in H\rangle_{\Z}\subseteq\Ker\eta^{*}\subseteq\Ker f$
and that $\Ker f/\Ker\eta_{*}\simeq H$. So in order to conclude,
it is enough to note that the map $H\arr\Ker f/G$ sending $h$ to
$e_{h}$ is a surjective group homomorphism since we have relations
$e_{h}+e_{h'}-e_{h+h'}=v_{h,h'}$ and $e_{m+h}-e_{m}=e_{h}-v_{m,h}$
for $m\in M$ and $h,h'\in H$.\end{proof}
\begin{prop}
\label{pro:Rita's smooth integral extremal rays}Let $\eta\colon M\arr\Z/l\Z$
be a surjective homomorphism with $l>1$. Then 
\[
\E^{\eta}(v_{m,n})=\left\{ \begin{array}{cc}
0 & \text{if }\eta(m)+\eta(n)<l\\
1 & \text{otherwise}
\end{array}\right.
\]
defines a smooth  extremal ray for $K_{+}$.\end{prop}
\begin{proof}
$\E^{\eta}\in\duale K_{+}$ because, if $\sigma\colon\Z/l\Z\arr\N$
is the obvious section, $\E^{\eta}$ is the restriction of the map
$\Z^{M}/\langle e_{0}\rangle\arr\Z$ sending $e_{m}$ to $\sigma(\eta(m))$.
In order to conclude the proof, we will apply \ref{lem:comparison smooth sequences for M covers}
and \ref{lem:equivalent condition for a smooth integral extremal ray}.
Set $N=\Z/l\Z$. One clearly has $\E^{\eta}=\E^{\id}\circ\eta_{*}$
and so we can assume $M=\Z/l\Z$ and $\eta=\id$. In this case one
can check that $v_{1,1},v_{1,2},\dots,v_{1,l-1}$ is a $\Z$-base
of $K$ such that $\E^{\eta}(v_{1,j})=0$ if $j<l-1$, $\E^{\eta}(v_{1,l-1})=1$.
\end{proof}
Those particular rays have been already defined in \cite[Equation 2.2]{Pardini1991}.
\begin{notation}
\label{not:ZME}If $\phi\colon\tilde{K}_{+}\arr\Z^{M}/\langle e_{0}\rangle$
is the usual map we set $\stZ_{M}^{\underline{\E}}=\stX_{\phi}^{\underline{\E}}$
(see definition \ref{def:T+epsilon sottolineato}) for any sequence
$\underline{\E}$ of elements of $\duale K_{+}$. Remember that if
$\underline{\E}$ is a smooth sequence then $\stZ_{M}^{\underline{\E}}$
is a smooth open subset of $\stZ_{M}$ (see \ref{thm:toric open substack of Z phi via smooth sequence})
and its points have the description given in \ref{pro:points of XphiE}.

Set $\Phi_{M}$ for the union over all $d>1$ of the sets of surjective
maps $M\arr\Z/d\Z$.\end{notation}
\begin{thm}
\label{thm:fundamental theorem for hleqone}Let $\underline{\E}=(\E^{\eta})_{\eta\in\Phi_{M}}$.
We have
\[
\big\{ h\leq1\big\}=\bigcup_{\eta\in\Phi_{M}}\stZ_{M}^{\E^{\eta}}
\]
In particular $\{h\leq1\}\subseteq\stZ_{M}^{\textup{sm}}$ and $\pi_{\underline{\E}}$
induces an equivalence of categories
\[
\{(\underline{\shL},\underline{\shM},\underline{z},\lambda)\in\stF_{\underline{\E}}\;|\; V(z_{\eta})\cap V(z_{\mu})=\emptyset\text{ if }\eta\neq\mu\}=\pi_{\underline{\E}}^{-1}(\{h\leq1\})\arrdi{\simeq}\{h\leq1\}
\]
\end{thm}
\begin{proof}
The last part of the statement follows from the first one just applying
\ref{pro:piE for theta isomorphism} with $\Theta=\{(\E^{\eta})\}_{\eta\in\Phi_{M}}$.
Let $k$ be an algebraically closed field and $B\in\MCov(k)$ with
graded basis $\{v_{m}\}_{m\in M}$ and multiplication $\psi$.

$\supseteq$. Assume $B\in\stZ_{M}^{\E^{\eta}}(k)$. If B is a torsor
we will have $h_{B/k}=0$. Otherwise we can write $\psi=\xi0^{\E^{\eta}}$
for some $\xi\colon K\arr k^{*}$. Replacing $\Spec k$ by a geometrical
point of the maximal torsor of $B/k$, we can assume that $M=\Z/d\Z$
and $\eta=\id$. In particular $H_{B/k}=0$ and, from the definition
of $\E^{\id}$, we get $B\simeq k[x]/(x^{d})$. So $h_{B/k}=\dim_{k}m_{B}/m_{B}^{2}=1$.

$\subseteq$. Assume $h_{B/k}=1$. Set $C$ for the maximal torsor
of $B/k$ (see \ref{def:of the maximal torsor and the subgroup associated}),
$H=H_{B/k}$ and $l=|M/H|$. The equality $h_{B/k}=1$ means that
there exists a unique $\overline{r}\in M/H$ (where $r\in M$) such
that $h_{B/k,r}=1$ and so $C_{q}[v_{r}]=B_{q}\simeq C_{q}[x]/(x^{l})$
for all (maximal) primes $q$ of $C$. In particular $B=C[v_{r}]\simeq C[x]/(x^{l})$
and $\overline{r}$ generates $M/H$. Let $\eta\colon M\arr M/H\simeq\Z/l\Z$
be the projection. We want to prove that $B\in\stZ_{M}^{\E^{\eta}}$.
Replacing $k$ by a geometrical point of some fppf extension of $k$,
we can assume $C=k[H]$, i.e. $v_{h}v_{h'}=v_{h+h'}$ if $h,h'\in H$.
Finally the elements $v_{h}v_{r}^{i}$ for $h\in H$ and $0\leq i<l$
define an $M$-graded basis of $B/k$ whose associated multiplication
is $0^{\E^{\eta}}$.\end{proof}
\begin{thm}
\label{thm:for hleqone}Let $\underline{\E}=(\E^{\eta})_{\eta\in\Phi_{M}}$
and let $X$ be a locally noetherian and locally factorial scheme.
Consider the full subcategories 
\[
\catC_{X}^{1}=\{(\underline{\shL},\underline{\shM},\underline{z},\lambda)\in\stF_{\underline{\E}}(X)\;|\;\codim_{X}V(z_{\eta})\cap V(z_{\mu})\geq2\text{ if }\eta\neq\mu\}\subseteq\shF_{\underline{\E}}(X)
\]
and 
\[
\catD_{X}^{1}=\{Y\arrdi fX\in\MCov(X)\st h_{f}(p)\leq1\ \forall p\in X\text{ with }\codim_{p}X\leq1\}\subseteq\MCov(X)
\]
Then $\pi_{\underline{\E}}$ induces an equivalence of categories
\[
\catD_{X}^{1}=\pi_{\underline{\E}}^{-1}(\catC_{X}^{1})\arrdi{\simeq}\catC_{X}^{1}
\]
\end{thm}
\begin{proof}
Apply \ref{thm:fundamental theorem for locally factorial schemes}
with $\Theta=\{(\E^{\eta})\}_{\eta\in\Phi_{M}}$.\end{proof}
\begin{thm}
\label{thm:regular in codimension 1 covers}Let $\underline{\E}=(\E^{\eta})_{\eta\in\Phi_{M}}$
and let $X$ be a locally noetherian and locally factorial scheme
without isolated points and $(\car X,|M|)=1$, i.e. $1/|M|\in\odi X(X)$.
Consider the full subcategories 
\[
Reg_{X}^{1}=\{Y/X\in\MCov(X)\st Y\text{ regular in codimension }1\}\subseteq\MCov(X)
\]
and
\[
\widetilde{Reg}_{X}^{1}=\left\{ (\underline{\shL},\underline{\shM},\underline{z},\lambda)\in\stF_{\underline{\E}}(X)\left|\begin{array}{c}
\forall\E\neq\delta\in\underline{\E}\;\codim_{X}V(z_{\E})\cap V(z_{\delta})\geq2\\
\forall\E\in\underline{\E}\forall p\in X^{(1)}\; v_{p}(z_{\E})\leq1
\end{array}\right.\right\} \subseteq\shF_{\underline{\E}}(X)
\]
Then we have an equivalence of categories
\[
\widetilde{Reg}_{X}^{1}=\pi_{\underline{\E}}^{-1}(Reg_{X}^{1})\arrdi{\simeq}Reg_{X}^{1}
\]
\end{thm}
\begin{proof}
We will make use of \ref{thm:for hleqone}. If $Y\arrdi fX\in Reg_{X}^{1}$,
$p\in Y^{(1)}$ and $q=f(p)$ then $h_{f}(q)\leq\dim_{k(p)}m_{p}/m_{p}^{2}=1$.
So $Reg_{X}^{1}\subseteq\catD_{X}^{1}$. So we have only to check
that $\widetilde{Reg}_{X}^{1}=\pi_{\underline{\E}}^{-1}(Reg_{X}^{1})\subseteq\catC_{X}^{1}$.
Since $X$ is a disjoint union of positive dimensional, integral connected
components, we can assume that $X=\Spec R$, where $R$ is a discrete
valuation ring. Let $\chi\in\catC_{X}^{1}$, $ $$A/R\in\catD_{X}^{1}$
the associated covers, $H=H_{A/R}$ and $C$ be the maximal torsor
of $A/R$. We have to prove that $\chi\in\widetilde{Reg}_{X}^{1}$
if and only if $A$ is regular in codimension $1$. Since $D_{R}(H)$
is etale over $R$ so is also $\Spec C$. It is so easy to check that,
replacing $R$ by a localization of $C$ and $M$ with $M/H$, we
can assume that $H=0$. Since $\chi\in\catC_{X}^{1}$, the multiplication
of $A$ over $R$ is of the form $\psi=\mu z^{r\E^{\phi}}$, where
$\mu\colon K\arr R^{*}$ is an $M$-torsor, $z$ is a parameter of
$A$, $\phi\colon M\arr\Z/l\Z$ is an isomorphism and $r=v_{R}(z_{\E^{\phi}})$.
Moreover $v_{R}(z_{\E^{\psi}})=0$ if $\psi\neq\phi$. Replacing $M$
by $\Z/l\Z$ through $\phi$ we can assume $\phi=\id$. Finally, since
$\mu$ induces an (fppf) torsor which is etale over $R$, replacing
$R$ by an etale neighborhood, we can assume $\mu=1$. After these
reductions we have $A=R[X]/(X^{|M|}-z^{r})$ which is regular in codimension
$1$ if and only if $r=1$.\end{proof}
\begin{rem}
In the theorem above one can replace the condition 'regular in codimension
$1$' in the definition of $Reg_{X}^{1}$ with 'normal' thanks to
Serre's conditions, since all the fibers involved are Gorenstein.
Moreover note that a locally noetherian and locally factorial scheme
$X$ is a disjoint union of integral connected components. Therefore
an isolated point is just a connected component which is $\Spec k$,
for a field $k$. We want to avoid this situation because regularity
in codimension $1$ for a cover over a field is an empty condition.
\end{rem}

\begin{rem}
Theorem \ref{thm:regular in codimension 1 covers} is a rewriting
of Theorem $2.1$ and Corollary $3.1$ of \cite{Pardini1991} extended
to locally noetherian and locally factorial schemes without isolated
points, where an object of $\stF_{\underline{\E}}(X)$ is called a
building data.
\end{rem}

\section{The locus $h\leq2$.}

In this section we want to give a characterization of the open substack
$\{h\leq2\}\subseteq\MCov$ as done in \ref{thm:for hleqone} for
$\{h\leq1\}$. The general problem we want to solve can be stated
as follows.
\begin{problem}
\label{prob: general problem for hleqtwo}Find a sequence of smooth
 extremal rays $\underline{\E}$ for $M$ and a collection $\Theta$
of smooth sequences with rays in $\underline{\E}$ such that (see
\ref{not:ZME})
\[
\{h\leq2\}=\bigcup_{\underline{\delta}\in\Theta}\stZ_{M}^{\underline{\delta}}
\]
or, equivalently, such that, for any algebraically closed field $k$,
the algebras $A\in\MCov(k)$ with $h_{A/k}\leq2$ are exactly the
algebras associated to a multiplication of the form $\psi=\omega0^{\E}$
where $\omega\colon K\arr k^{*}$ is a group homomorphism and $\E\in\langle\underline{\delta}\rangle_{\N}$
for some $\underline{\delta}\in\Theta$.
\end{problem}
For example in the case $h\leq1$ the analogous problem is solved
taking $\underline{\E}=(\E^{\phi})_{\phi\in\Phi_{M}}$ and $\Theta=\{(\E)\text{ for }\E\in\underline{\E}\}$
(see \ref{thm:fundamental theorem for hleqone}). Once we have found
a pair $\underline{\E},\Theta$ as in \ref{prob: general problem for hleqtwo}
we can formally apply theorems \ref{pro:piE for theta isomorphism}
and \ref{thm:fundamental theorem for locally factorial schemes}.
This is done in theorems \ref{thm:fundamental thm for hleqtwo} and
\ref{thm:fundamental thm locally factoria hleqtwo}.

Similarly to what happens in the case $h\leq1$, we can restrict our
attention to the case when $M$ is generated by two elements $m,n$
and the first problem to solve is to describe $M$-graded algebras
$A$ over a field $k$ generated in these degrees $m,n$ (see \ref{not:starting from a cover, m.n generate M}).
This is done associating with $A$ an invariant $\overline{q}_{A}\in\N$
(see \ref{pro:Universal algebras generated by m,n with overline q fixed})
and this solution also suggests how to proceed for the next problem,
i.e. find the sequence $\underline{\E}$ of problem \ref{prob: general problem for hleqtwo}.

When $M$ is any finite abelian group, it turns out that the extremal
rays $\E$ for $M$ such that $h_{\E}=2$ correspond to particular
sequences of the form $\chi=(r,\alpha,N,\overline{q},\phi)$, where
$r,\alpha,N,\overline{q}\in\N$ and $\phi$ is a surjective map from
$M$ to a group $M_{r,\alpha,N}$ generated by two elements (see \ref{def: M r alpha N}).
The sequence of smooth  extremal rays {}``needed'' to describe the
substack $\{h\leq2\}$ is composed by the {}``old'' rays $(\E^{\eta})_{\eta\in\Phi_{M}}$
and by these new rays. Finally the smooth sequences in the family
$\Theta$ of problem \ref{prob: general problem for hleqtwo} will
all be given by elements of the dual basis of particular $\Z$-basis
of $K$ (see \ref{lem:lambda,delta for overlineq}).

In the last subsection we will see (Theorem \ref{thm:NC in codimension one})
that the normal crossing in codimension $1$ $\Di M$-covers of a
locally noetherian and locally factorial scheme with no isolated points
and with $(\car X,|M|)=1$ can be described in the spirit of classification
\ref{thm:regular in codimension 1 covers} and extending this result.
\begin{notation}
If $m\in M$ we will denote by $o(m)$ the order of $m$ in the group
$M$.
\end{notation}

\subsection{Good sequences.}

In this subsection we provide some general technical results in order
to work with $M$-graded algebras over local rings. So we will consider
given a local ring $D$, a sequence $\underline{m}=m_{1},\dots,m_{r}\in M$
and $C\in\MCov(D)$ generated in degrees $m_{1},\dots,m_{r}$. Since
$\Pic(D)$=0 for any $u\in M$ we have $C_{u}\simeq D$. Given $u\in M$,
we will call $v_{u}$ a generator of $C_{u}$ and we will also use
the abbreviation $v_{i}=v_{m_{i}}$. Moreover, if $\underline{A}=(A_{1},\dots,A_{r})\in\N^{r}$
we will also write
\[
v^{\underline{A}}=v_{1}^{A_{1}}\cdots v_{r}^{A_{r}}
\]

\begin{defn}
A sequence for $u\in M$ is a sequence $\underline{A}\in\N^{r}$ such
that $A_{1}m_{1}+\cdots+A_{r}m_{r}=u$. Such a sequence will be called
\emph{good }if the map $C_{m_{1}}^{A_{1}}\otimes\cdots\otimes C_{m_{r}}^{A_{r}}\arr C_{u}$
is surjective, i.e. $v^{\underline{A}}$ generates $C_{u}$. If $r=2$
we will talk about pairs instead of sequences.\end{defn}
\begin{rem}
Any $u\in M$ admits a good sequence since, otherwise, we will have
$C_{u}=(D[v_{1},\dots,v_{r}])_{u}\subseteq m_{D}C_{u}$. If $\underline{A}$
is a good sequence and $\underline{B}\leq\underline{A}$, then also
$\underline{B}$ is a good sequence.\end{rem}
\begin{lem}
\label{lem:fundamental local rings and generators}Let $\underline{A}$,
$\underline{B}$ be two sequences for some element of $M$ and assume
that $\underline{A}$ is good. Set $\underline{E}=\min(\underline{A},\underline{B})=(\min(A_{1},B_{1}),\dots,\min(A_{r},B_{r}))$
and take $\lambda\in D$. Then 
\[
v^{\underline{B}}=\lambda v^{\underline{A}}\then v^{\underline{B}-\underline{E}}=\lambda v^{\underline{A}-\underline{E}}
\]
\end{lem}
\begin{proof}
Clearly we have $v^{\underline{E}}(v^{\underline{B}-\underline{E}}-\lambda v^{\underline{A}-\underline{E}})=0$.
On the other hand, since $\underline{A}-\underline{E}$ is a good
sequence, there exists $\mu\in D$ such that $v^{\underline{B}-\underline{E}}=\mu v^{\underline{A}-\underline{E}}$.
Since $\underline{A}$ is a good sequence, substituting we get $v^{\underline{A}}(\mu-\lambda)=0\then\mu=\lambda$.
\end{proof}

\subsection{\label{sub:algebras gen in two degrees}$M$-graded algebras generated
in two degrees.}
\begin{defn}
\label{def: M r alpha N}Given $0\leq\alpha<N$ and $r>0$ we set
\[
M_{r,\alpha,N}=\Z^{2}/\langle(r,-\alpha),(0,N)\rangle
\]
\end{defn}
\begin{prop}
\label{pro:description group generated by two elements}A finite abelian
group $M$ with two marked elements $m,n\in M$ generating it is canonically
isomorphic to $(M_{r,\alpha,N},e_{1},e_{2})$ where $r=\min\{s>0\st sm\in\langle n\rangle\}$,
$rm=\alpha n$ and $N=o(n)$. Moreover we have: $|M|=Nr$, $o(m)=rN/(\alpha,N)$
and 
\[
m,n\neq0\text{ and }m\neq n\iff N>1\text{ and }(r>1\text{ or }\alpha>1)
\]
\end{prop}
\begin{proof}
We have
\[
0\arr\Z^{2}\arrdi{\scriptsize{\left(\begin{array}{cc}
r & 0\\
-\alpha & N
\end{array}\right)}}\Z^{2}\arr M_{r,\alpha,N}\arr0\text{ exact}\then|M_{r,\alpha,N}|=\left|\det\left(\begin{array}{cc}
r & 0\\
-\alpha & N
\end{array}\right)\right|=rN
\]
and clearly $e_{1},e_{2}$ generate $M$. Moreover $M_{r,\alpha,N}/\langle e_{2}\rangle\simeq\Z/r\Z$
and therefore $r$ is the minimum such that $re_{1}\in\langle e_{2}\rangle$.
Finally it is easy to check that $N=o(e_{2})$. If now $M,r,\alpha,N$
are as in the statement, there exists a unique map $ $$M_{r,\alpha,N}\arr M$
sending $e_{1},e_{2}$ to $m,n$. This map is an isomorphism since
it is clearly surjective and $|M|=o(m)o(n)/|\langle m\rangle\cap\langle n\rangle|=o(n)r=|M_{r,\alpha,N}|$.
The last equivalence in the statement is now easy to prove.\end{proof}
\begin{notation}
\label{not:for alpha,N,r e M}In this subsection we will fix a finite
abelian group $M$ generated by two elements $0\neq m,n\in M$ such
that $m\neq n$. Up to isomorphism, this means $M=M_{r,\alpha,N}$
with $m=e_{1}\comma n=e_{2}$ and with the conditions $0\leq\alpha<N\comma r>0\comma N>1\comma(r>1\text{ or }\alpha>1)$. 

We will write $d_{q}$ the only integer $0<d_{q}\leq N$ such that
$qrm+d_{q}n=0$, for $q\in\Z$, or, equivalently, $d_{q}\equiv-q\alpha\text{ mod }(N)$.
\end{notation}

\begin{problem}
\label{not:starting from a cover, m.n generate M}Let $k$ be a field.
We want to describe, up to isomorphism, algebras $A\in\MCov(k)$ such
that $A$ is generated in degrees $m,n$ and $H_{A/k}=0$. Thanks
to \ref{cor:An algebra with H =00003D 0 is generated in the degrees where h=00003D1},
this is equivalent to asking for an algebra $A$ such that $H_{A/k}=0$
and
\[
\{l\in M\st h_{A/k,l}=1\}\subseteq\{m,n\}
\]
The solution of this problem is contained in \ref{pro:Universal algebras generated by m,n with overline q fixed}. 
\end{problem}
In this subsection we will fix an algebra $A$ as in \ref{not:starting from a cover, m.n generate M},
we will consider given a graded basis $\{v_{l}\}_{l\in M}$ of $A$
and we will denote by $\psi$ the associated multiplication. Note
that $H_{A/k}=0$ means $v_{m},v_{n}\notin A^{*}$.
\begin{defn}
Define
\[
z=\min\{h>0\;|\;\exists i\in\N\comma\lambda\in k\text{ such that }v_{m}^{h}=\lambda v_{n}^{i}\text{ and }hm=in\}
\]
\[
x=\min\{h>0\;|\;\exists i\in\N\comma\mu\in k\text{ such that }v_{n}^{h}=\mu v_{m}^{i}\text{ and }hn=im\}
\]
Denote by $0\leq y<o(n)$, $0\leq w<o(m)$ the elements such that
$zm=yn$, $xn=wm$, by $\lambda,\mu\in k$ the elements such that
$v_{m}^{z}=\lambda v_{n}^{y}$, $v_{n}^{x}=\mu v_{m}^{w}$, with the
convention that $\lambda=0$ if $v_{n}^{y}=0$ and $\mu=0$ if $v_{m}^{w}=0$.
Finally set $\overline{q}=z/r$ and define the map of sets   \[   \begin{tikzpicture}[xscale=3.5,yscale=-0.6]     \node (A0_0) at (0, 0) {$\{0,1,\dots,z-1\}$};     \node (A0_1) at (1, 0) {$\{0,1,\dots,o(n)\}$};     \node (A1_0) at (0, 1) {$c$};     \node (A1_1) at (1, 1) {$\min\{d\in\N \ | \ v_{m}^{c}v_{n}^{d}=0\}$};     \path (A0_0) edge [->] node [auto] {$\scriptstyle{f}$} (A0_1);     \path (A1_0) edge [|->,gray] node [auto] {$\scriptstyle{}$} (A1_1);   \end{tikzpicture}   \] 
We will also write $\overline{q}_{A}\comma z_{A}\comma x_{A}\comma y_{A}\comma w_{A}\comma\lambda_{A}\comma\mu_{A}\comma f_{A}$
if necessary.
\end{defn}
We will see that $A$ is uniquely determined by $\overline{q}$ and
$\lambda$ up to isomorphism.
\begin{lem}
\label{lem:good pair for A<z}Given $l\in M$ there exists a unique
good pair $(a,b)$ for $l$ with $0\leq a<z$. Moreover $0\leq b<f(a)$.\end{lem}
\begin{proof}
\emph{Existence. }We know that there exists a good pair $(a,b)$ for
$l$ and we can assume that $a$ is minimum. If $a\geq z$ we can
write $v_{m}^{a}v_{n}^{b}=\lambda v_{m}^{a-z}v_{n}^{b+y}$. Therefore
$\lambda\neq0$ and $(a-z,b+y)$ is a good pair for l, contradicting
the minimality of $a$. Finally $v_{m}^{a}v_{n}^{b}\neq0$ means $b<f(a)$.

\emph{Uniqueness. }Let $(a,b)$, $(a',b')$ be two good pairs for
$l$ and assume $0\leq a<a'<z$. So there exists $\omega\in k^{*}$
such that
\[
v_{m}^{a}v_{n}^{b}=\omega v_{m}^{a'}v_{n}^{b'}\then v_{n}^{b}=\omega v_{m}^{a'-a}v_{n}^{b'}
\]
If $b\geq b'$ then $a'-a\geq z$ by definition of $z$, while if
$b<b'$ then $v_{n}$ is invertible.\end{proof}
\begin{defn}
\label{not:definition of E delta starting from the cover}Given $l\in M$
we will write the associated good pair as $(\E_{l},\delta_{l})$ with
$\E_{l}<z$. We will consider $\E,\delta$ as maps $\Z^{M}/\langle e_{0}\rangle\arr\Z$
and, if necessary, we will also write $\E^{A},\delta^{A}$.\end{defn}
\begin{notation}
Up to isomorphism, we can change the given basis to 
\[
v_{l}=v_{m}^{\E_{l}}v_{n}^{\delta_{l}}
\]
so that the multiplication $\psi$ is given by
\begin{equation}
v_{a}v_{b}=v_{m}^{\E_{a}+\E_{b}}v_{n}^{\delta_{a}+\delta_{b}}=\psi_{a,b}v_{m}^{\E_{a+b}}v_{n}^{\delta_{a+b}}=\psi_{a,b}v_{a+b}\label{eq:multiplications psi for m,n first}
\end{equation}
\end{notation}
\begin{cor}
$f$ is a decreasing function and
\begin{equation}
f(0)+\cdots+f(z-1)=|M|\label{eq:the sum of f's is |M|}
\end{equation}
\end{cor}
\begin{proof}
If $(a,b)$ is a pair such that $0\leq a<z$ and $0\leq b<f(a)$ then
$v_{m}^{a}v_{n}^{b}\neq0$, i.e. $(a,b)$ is a good pair for $am+bn$.
So
\[
\sum_{c=0}^{z-1}f(c)=|\{(a,b)\;|\;0\leq a<z,\;0\leq b<f(a)\}|=|M|
\]
\end{proof}
\begin{rem}
\label{rem:some remarks on f}The following pairs are good:
\[
(z-1)m:(z-1,0),\;(x-1)n:(0,x-1),\; zm=yn:(0,y)\comma xn=wm:(w,0)
\]
i.e. $v_{m}^{z-1},v_{n}^{x-1},v_{n}^{y},v_{m}^{w}\neq0$. In particular
$f(0)\geq x,y+1$ and $f(c)>0$ for any $c$. Indeed
\begin{alignat*}{3}
v_{m}^{z-1}=\omega v_{m}^{a}v_{n}^{b} & \then & v_{m}^{z-1-a}=\omega v_{n}^{b} & \then & a=z-1,\; b=0\\
v_{m}^{z}=\omega v_{m}^{a}v_{n}^{b} & \then & v_{m}^{z-a}=\omega v_{n}^{b} & \then & a=0,\; b=y
\end{alignat*}
where $(a,b)$ are good pairs for the given elements and, by symmetry,
we get the result.
\end{rem}

\begin{rem}
\label{rem:lambda is 0 iff mu is 0}If $\lambda\neq0$ or $\mu\neq0$
then $x=y$, $z=w$ and $\lambda\mu=1$. Assume for example $\lambda\neq0$.
If $y=0$ then $v_{m}^{z}=\lambda\neq0$ and so $v_{m}$ is invertible.
So $y>0$ and, since $v_{n}^{y}=\lambda^{-1}v_{m}^{z}$, we also have
$y\geq x$. Now
\[
0\neq v_{m}^{z}=\lambda v_{n}^{y}=\lambda\mu v_{n}^{y-x}v_{m}^{w}
\]
So $\mu\neq0$ and $(y-x,w)$ is a good pair. As before $w\geq z$
and therefore
\[
\lambda\mu v_{n}^{y-x}v_{m}^{w-z}=1\then y=x,\; w=x\text{ and }\lambda\mu=1
\]
\end{rem}
\begin{lem}
\label{lem:computation of E and delta}Let $a,b\in M$. We have:
\begin{itemize}
\item Assume $\E_{a,b}>0$. If $\delta_{a,b}\leq0$ then $\E_{a,b}\geq z\comma\delta_{a,b}\geq-y$.
Moreover $\psi_{a,b}\neq0\iff\lambda\neq0,\E_{a,b}=z,\delta_{a,b}=-y(=-x)$
and in this case $\psi_{a,b}=\lambda$.
\item Assume $\E_{a,b}<0$. Then $\E_{a,b}\geq-w,\delta_{a,b}\geq x$. Moreover
$\psi_{a,b}\neq0\iff\mu\neq0,\E_{a,b}=-w(=-z),\delta_{a,b}=x$ and
in this case $\psi_{a,b}=\mu$.
\item Assume $\E_{a,b}=0$. Then we have $\delta_{a,b}=0$ and $\psi_{a,b}=1$
or $\delta_{a,b}\geq o(n)$ and $\psi_{a,b}=0$.
\end{itemize}
\end{lem}
\begin{proof}
Set $\psi=\psi_{a,b}$. We start with the case $\E_{a,b}>0$. From
\ref{eq:multiplications psi for m,n first} we get
\[
v_{m}^{\E_{a,b}}v_{n}^{\delta_{a}+\delta_{b}}=\psi v_{n}^{\delta_{a+b}}
\]
If $\delta_{a,b}>0$ then $v_{m}^{\E_{a,b}}v_{n}^{\delta_{a,b}}=\psi$
and so $\psi=0$ since $v_{m}\notin A^{*}$. If $\delta_{a,b}\leq0$
we instead have $v_{m}^{\E_{a,b}}=\psi v_{n}^{-\delta_{a,b}}$ and
so $\E_{a,b}\geq z$. If $-\delta_{a,b}<y$ then $(0,-\delta_{a,b})$
is good. So we can write
\[
v_{m}^{\E_{a,b}-z}\lambda v_{n}^{y+\delta_{a,b}}=\psi\then\psi=0
\]
since $v_{n}$ is not invertible. If $\delta_{a,b}\leq-y$ we have
\[
0\leq\E_{a,b}-z<z,\;0\leq-\delta_{a,b}-y<f(0),\;(\E_{a,b}-z)m=(-\delta_{a,b}-y)n,\; v_{m}^{\E_{a,b}-z}\lambda=\psi v_{n}^{-\delta_{a,b}-y}
\]
and so both $(\E_{a,b}-z,0)$ and $(0,-\delta_{a,b}-y)$ are good
pair for the same element of $M$. Therefore we must have $\E_{a,b}=z$,
$\delta_{a,b}=-y$ and $\psi=\lambda$.

Now assume $\E_{a,b}=0$. If $\delta_{a,b}<0$ then $v_{n}^{-\delta_{a,b}}\psi=1$
which is impossible. So $\delta_{a,b}\geq0$. If $\delta_{a,b}=0$
clearly $\psi=1$. If $\delta_{a,b}>0$ then $v_{n}^{\delta_{a,b}}=\psi$
and so $\psi=0$ and $\delta_{a,b}\geq o(n)$.

Finally assume $\E_{a,b}<0$. From \ref{eq:multiplications psi for m,n first}
we get
\[
v_{n}^{\delta_{a}+\delta_{b}}=\psi v_{m}^{-\E_{a,b}}v_{n}^{\delta_{a+b}}
\]
We must have $\delta_{a,b}>0$ since $v_{m}$ is not invertible. So
$v_{n}^{\delta_{a,b}}=\psi v_{m}^{-\E_{a,b}}$ and $\delta_{a,b}\geq x$,
from which
\[
v_{n}^{\delta_{a,b}-x}\mu v_{m}^{w}=\psi v_{m}^{-\E_{a,b}}
\]
Note that, since $0\leq-\E_{a,b}\leq\E_{a+b}<z$, $(-\E_{a,b},0)$
is a good pair. If $w>-\E_{a,b}$ then $\psi=0$. So assume $w\leq-\E_{a,b}$.
Arguing as above we must have $\delta_{a,b}=x$, $\E_{a,b}=-w$ and
$\psi=\mu$.
\end{proof}

\begin{lem}
\label{lem:From lambda not zero to zero A'}Define
\[
A'=k[s,t]/(s^{z},s^{c}t^{f(c)}\:\text{for}\:0\leq c<z)
\]
Then $A'\in\MCov(k)$ with graduation $\deg s=m$, $\deg t=n$ and
it satisfies the requests of \ref{not:starting from a cover, m.n generate M},
i.e. $A'$ is generated in degrees $m,n$ and $H_{A'/k}=0$. Moreover
we have
\[
\overline{q}_{A'}=\overline{q}_{A}\comma z_{A'}=z_{A},\; y_{A'}=y_{A},\;\E^{A'}=\E^{A},\;\delta^{A'}=\delta^{A},\;\lambda_{A'}=\mu_{A'}=0\comma f_{A'}=f_{A}
\]
\end{lem}
\begin{proof}
Clearly the elements $s^{c}t^{d}$ for $0\leq c<z$, $0\leq d<f(c)$
generates $A'$ as a $k$-space. Since they are $\sum_{c=0}^{z-1}f(c)=|M|$
and they all have different degrees, it is enough to prove that any
of them are non-zero. So let $(c',d')$ a pair as always. It is enough
to show that $B=k[s,t]/(s^{c'+1},t^{d'+1})\arr A'/(s^{c'+1},t^{d'+1})$
is an isomorphism. But $c'<z$ implies that $s^{z}=0$ in $B$. If
$c'<c$ then $s^{c}t^{f(c)}=0$ in $B$ and finally if $c'\geq c$
then $d'+1\leq f(c')\leq f(c)$ and so $s^{c}t^{f(c)}=0$ in B.

The algebra $A'$ is clearly generated in degrees $m,n$ and $H_{A'/k}=0$
since $s^{z}=t^{f(0)}=0$ and $z,f(0)>0$. Moreover $s^{z}=0t^{y}$
implies that $z'=z_{A'}\leq z$. Assume by contradiction $z'<z$.
From $0\neq s^{z'}=\lambda't^{y'}$ we know that $t^{y'}\neq0$ so
that $y'<f(0)$. Therefore $(\E_{z'm},\delta_{z'm})=(z',0)=(0,y')$
and so $z'=0$, which is a contradiction. Then $z'=z$, $y_{A'}=y'=y$.
Also $s^{z}=0t^{y}$ and $t^{y}\neq0$ imply $\lambda_{A'}=0$ and,
thanks to \ref{rem:lambda is 0 iff mu is 0}, $\mu_{A'}=0$. Finally
by construction we also have $\E^{A'}=\E$, $\delta^{A'}=\delta$
and $f_{A'}=f$.\end{proof}
\begin{lem}
\label{lem:overline q from a cover}We have
\[
{\displaystyle d_{\overline{q}}=\max_{1\leq q\leq\overline{q}}d_{q}}
\]
\end{lem}
\begin{proof}
Thanks to \ref{lem:From lambda not zero to zero A'} we can assume
$\lambda=0$ and, therefore, $\mu=0$. So $v_{n}^{x}=0$, $v_{n}^{x-1}\neq0$
and $v_{n}^{y}\neq0$ imply $y<x=f(0)$. Let $1\leq q<\overline{q}$
and $l=qr$. We have $(\E_{l},\delta_{l})=(qr,0)$. If $N-d_{q}<x=f(0)$
then we will also have $(\E_{l},\delta_{l})=(0,N-d_{q})$ and so $q=0$,
which is not the case. So $N-d_{q}\geq x>y=N-d_{\overline{q}}\then d_{q}<d_{\overline{q}}$.\end{proof}
\begin{lem}
\label{lem:From a cover we get overline q}Define $\hat{q}$ as the
only integers $0\leq\hat{q}<\overline{q}$ such that
\[
d_{\hat{q}}=\min_{0\leq q<\overline{q}}d_{q}
\]
If $\lambda=0$ we have $d_{\hat{q}}\leq x=f(0)\text{ and }f(c)=\left\{ \begin{array}{cc}
x & \text{if }0\leq c<\hat{q}r\\
d_{\hat{q}} & \text{if }\hat{q}r\leq c<z
\end{array}\right.$\end{lem}
\begin{proof}
We want first prove that $f(c)=\min(x,d_{q}\;\text{for}\;0\leq qr\leq c)$.
Clearly we have the inequality $\leq$ since $v_{n}^{x}=v_{m}^{qr}v_{n}^{d_{q}}=0$.
Set $d=f(c)$ and let $(a,b)$ a good pair for $cm+dn$, so that $v_{m}^{c}v_{n}^{d}=0v_{m}^{a}v_{n}^{b}$.
We cannot have $b\geq d$ since otherwise $v_{m}^{c}=0$ implies $c\geq z$.
If $a\geq c$ then $v_{n}^{d}=0$ and so $d=f(c)\geq x$. Conversely
if $a<c$ then $0\leq c-a=qr\leq c<z$ and $0<d-b=d_{q}\leq d=f(c)$.

We are now ready to prove the expression of $f$. Note that the pairs
$(qr,d_{q}-1)$, with $0\leq q<\overline{q}$, are all the possible
pairs for $-n$. So there exists a unique $0\leq\tilde{q}<\overline{q}$
such that $(\tilde{q}r,d_{\tilde{q}}-1)$ is good. In particular if
$0\leq q\neq\tilde{q}<\overline{q}$ we have an expression 
\[
v_{m}^{qr}v_{n}^{d_{q}-1}=0v_{m}^{\tilde{q}}v_{n}^{d_{\tilde{q}}-1}\then\left\{ \begin{array}{ccccc}
q<\tilde{q} & \then & v_{n}^{d_{q}-1}=0 & \then & d_{q}\geq x\\
q>\tilde{q} & \then & d_{q}>d_{\tilde{q}}
\end{array}\right.
\]
Since $v_{n}^{d_{\tilde{q}}-1}\neq0$ we must have $d_{\hat{q}}\leq x$.
This shows that $\tilde{q}=\hat{q}$ and the expression of $f$. Finally
If $\overline{q}>1$ then $\hat{q}>0$ and so $d_{\hat{q}}\leq x=f(0)$
since $f$ is a decreasing function. If $\overline{q}=1$ then $\hat{q}=0$
and so $N=d_{\hat{q}}=f(0)\leq x\leq N$.\end{proof}
\begin{defn}
We will continue to use notation from \ref{lem:From a cover we get overline q}
for $\hat{q}$ and we will also write $\hat{q}_{A}$ if necessary.
\end{defn}

\subsection{\label{sub:the invariant overlineq}The invariant $\overline{q}$.}
\begin{lem}
\label{lem:fundamental for the case m,n}Let $\beta,N\in\N$, with
$N>1$, and define $d_{q}^{\beta}=d_{q}$, for $q\in\Z$, the only
integer $0<d_{q}\leq N$ such that $d_{q}\equiv q\beta$ mod $N$.
Set
\[
\Omega_{\beta,N}=\{0<q\leq o(\beta,\Z/N\Z)=N/(N,\beta)\;|\; d_{q'}<d_{q}\text{ for any }0<q'<q\},
\]
set $q_{n}$ for the $n$-th element of it and denote by $0\leq\hat{q}<q_{n}$
the only number such that 
\[
d_{\hat{q}}=\min_{0\leq q<q_{n}}d_{q}
\]
Then we have relations $\hat{q}N+q_{n}d_{\hat{q}}-\hat{q}d_{q_{n}}=N$
and, if $n>1$, $q_{n}=q_{n-1}+\hat{q}$, $d_{q_{n}}=d_{q_{n-1}}+d_{\hat{q}}$
and $d_{q_{n-1}}+d_{q}>N$ for $q<\hat{q}$.\end{lem}
\begin{proof}
First of all note that all is defined also in the extremal case $\beta=0$.
In this case $\Omega_{\beta,N}=\{1\}$. Assume first $n>1$. Set $\tilde{q}=q_{n}-q_{n-1}$
so that $d_{q_{n}}=d_{q_{n-1}}+d_{\tilde{q}}$ since $d_{q_{n}}>d_{q_{n-1}}$.
Assume by contradiction that $\tilde{q}\neq\hat{q}$. Since $\tilde{q}<q_{n}$
we have $d_{\hat{q}}<d_{\tilde{q}}$. Let also $q'=q_{n}-\hat{q}$
and, as above, we can write $d_{q_{n}}=d_{q'}+d_{\hat{q}}$. Now
\[
d_{q_{n}}-d_{q'}=d_{\hat{q}}<d_{\tilde{q}}=d_{q_{n}}-d_{q_{n-1}}\then d_{q_{n-1}}<d_{q'}
\]
Since $q_{n-1}\in\Omega_{\beta,N}$ we must have $q'>q_{n-1}$, which
is a contradiction because otherwise, being $q'<q_{n}$, we must have
$q'=q_{n}$. So $\tilde{q}=\hat{q}$. For the last relation note that,
since $q_{n}$ is the first $q>q_{n-1}$ such that $d_{q}>d_{q_{n-1}}$,
then $\hat{q}$ is the first such that $d_{q_{n-1}}+d_{\hat{q}}\leq N$.

Now consider the first relation. We need to do induction on all the
$\beta$. So we will write $d_{q}^{\beta}$ and $q_{n}^{\beta}$ in
order to remember that those numbers depend on to $\beta$. The induction
statement on $1\leq q<N$ is: for any $0\leq\beta<N$ and for any
$n$ such that $q_{n}^{\beta}\leq q$ the required formula holds.
The base step is $q=1$. In this case we have $n=1$, $q_{1}=1$,
$\hat{q}=0$, $d_{0}=N$ and the formula can be proven directly. For
the induction step we can assume $q>1$ and $n>1$. We will write
$\hat{q}_{n}^{\beta}$ for the $\hat{q}$ associated to $n$ and $\beta$.
First of all note that, by the relations proved above, we can write
\[
\hat{q}_{n}^{\beta}N+q_{n}^{\beta}d_{\hat{q}_{n}^{\beta}}^{\beta}-\hat{q}_{n}^{\beta}d_{q_{n}^{\beta}}^{\beta}=\hat{q}_{n}^{\beta}N+q_{n-1}^{\beta}d_{\hat{q}_{n}^{\beta}}^{\beta}-\hat{q}_{n}^{\beta}d_{q_{n-1}^{\beta}}^{\beta}
\]
and so we have to prove that the second member equals $N$. If $\hat{q}_{n}^{\beta}\leq q_{n-1}^{\beta}$
then $\hat{q}_{n-1}^{\beta}=\hat{q}_{n}^{\beta}$ and the formula
is true by induction on $q-1\geq q_{n-1}^{\beta}$. So assume $\hat{q}_{n}^{\beta}>q_{n-1}^{\beta}$
and set $\alpha=N-\beta$. Clearly we will have
\[
o=o(\alpha,\Z/N\Z)=o(\beta,\Z/N\Z)\text{ and }d_{q}^{\beta}+d_{q}^{\alpha}=N\text{ for any }0<q<o
\]
Moreover
\[
d_{\hat{q}_{n}^{\beta}}^{\beta}<d_{q}^{\beta}\text{ for any }0<q<q_{n}^{\beta}\then d_{\hat{q}_{n}^{\beta}}^{\alpha}>d_{q}^{\alpha}\text{ for any }0<q<\hat{q}_{n}^{\beta}\then\exists l\text{ s.t. }q_{l}^{\alpha}=\hat{q}_{n}^{\beta}
\]
and
\[
d_{q_{n-1}^{\beta}}^{\beta}\geq d_{q}^{\beta}\text{ for any }0<q<q_{n}^{\beta}\then d_{q_{n-1}^{\beta}}^{\alpha}\leq d_{q}^{\alpha}\text{ for any }0\leq q<q_{l}^{\alpha}=\hat{q}_{n}^{\beta}\then\hat{q}_{l}^{\alpha}=q_{n-1}^{\beta}
\]
Using induction on $q_{l}^{\alpha}=\hat{q}_{n}^{\beta}<q_{n}^{\beta}\leq q$
we can finally write
\begin{alignat*}{1}
N= & \hat{q}_{l}^{\alpha}N+q_{l}^{\alpha}d_{\hat{q}_{l}^{\alpha}}^{\alpha}-\hat{q}_{l}^{\alpha}d_{q_{l}^{\alpha}}^{\alpha}=q_{n-1}^{\beta}N+\hat{q}_{n}^{\beta}d_{q_{n-1}^{\beta}}^{\alpha}-q_{n-1}^{\beta}d_{\hat{q}_{n}^{\beta}}^{\alpha}\\
= & q_{n-1}^{\beta}N+\hat{q}_{n}^{\beta}(N-d_{q_{n-1}^{\beta}}^{\beta})-q_{n-1}^{\beta}(N-d_{\hat{q}_{n}^{\beta}}^{\beta})=\hat{q}_{n}^{\beta}N+q_{n-1}^{\beta}d_{\hat{q}_{n}^{\beta}}^{\beta}-\hat{q}_{n}^{\beta}d_{q_{n-1}^{\beta}}^{\beta}
\end{alignat*}

\end{proof}
We continue to keep notation from \ref{not:for alpha,N,r e M}. With
$d_{q}$ we will always mean $d_{q}^{N-\alpha}$ as in \ref{lem:fundamental for the case m,n}.
Lemma \ref{lem:overline q from a cover} can be restated as:
\begin{prop}
\label{pro:overlineqA is in omeganminusalpha,n}Let $A$ be an algebra
as in \ref{not:starting from a cover, m.n generate M}. Then $\overline{q}_{A}\in\Omega_{N-\alpha,N}$.
\end{prop}
So given an algebra $A$ as in \ref{not:starting from a cover, m.n generate M}
we can associate to it the number $\overline{q}_{A}\in\Omega_{N-\alpha,N}$.
Conversely we will see that any $\overline{q}\in\Omega_{N-\alpha,N}$
admits an algebra $A$ as in \ref{not:starting from a cover, m.n generate M}
such that $\overline{q}=\overline{q}_{A}$. It turns out that all
the objects $z_{A}$, $y_{A}$, $f_{A}$, $\E^{A}$, $\delta^{A}$,
$\hat{q}_{A}$ and, if $\lambda_{A}=0$, $x_{A}$, $w_{A}$ associated
to $A$ only depend on $\overline{q}_{A}$. Therefore in this subsection,
given $\overline{q}\in\Omega_{N-\alpha,N}$, we will see how to define
such objects independently from an algebra $A$.

In this subsection we will consider given an element $\overline{q}\in\Omega_{N-\alpha,N}$.
\begin{defn}
\label{not:alpha, N, r, overq, q', hatq, f}Set $\hat{q}$ for the
only integer $0\leq\hat{q}<\overline{q}$ such that $d_{\hat{q}}=\min_{0\leq q<\overline{q}}d_{q}$,
$q'=\overline{q}-\hat{q}$, $z=\overline{q}r$, $y=N-d_{\overline{q}}$,
\[
x=\left\{ \begin{array}{cc}
N-d_{q'} & \text{if }\overline{q}>1\\
N & \text{if }\overline{q}=1
\end{array}\right.\comma w=\left\{ \begin{array}{cc}
q'r & \text{if }\overline{q}>1\\
0 & \text{if }\overline{q}=1
\end{array}\right.\comma f(c)=\left\{ \begin{array}{cc}
x & \text{if }0\leq c<\hat{q}r\\
d_{\hat{q}} & \text{if }\hat{q}r\leq c<z
\end{array}\right.
\]
We will also write $\hat{q}_{\overline{q}}\comma q_{\overline{q}}'\comma z_{\overline{q}}\comma x_{\overline{q}}\comma f_{\overline{q}}\comma y_{\overline{q}}\comma w_{\overline{q}}$
if necessary.\end{defn}
\begin{rem}
Using notation from \ref{lem:fundamental for the case m,n} we have
$\overline{q}=q_{n}$ for some $n$ and, if $n>1$, i.e. $\overline{q}>1$,
$q_{n-1}=q'$. Note that $zm=yn$, $wm=xn$, $y<x$, $w<z$. Moreover,
from \ref{lem:fundamental for the case m,n} and from a direct computation
if $\overline{q}=1$, we obtain $zx-yw=|M|$. Finally if $\overline{q}>1$
one has relations $\hat{q}r=z-w\text{ and }d_{\hat{q}}=x-y$.\end{rem}
\begin{lem}
\label{lem:definition of epsilon delta}We have that:
\begin{enumerate}
\item $f$ is a decreasing function and $\sum_{c=0}^{z-1}f(c)=|M|$;
\item \label{enu:writing of the form Am+Bn}any element $t\in M$ can be
uniquely written as 
\[
t=Am+Bn\text{ with }0\leq A<z,0\leq B<f(A)
\]

\end{enumerate}
\end{lem}
\begin{proof}
$1)$ If $\overline{q}=1$ it is enough to note that $\hat{q}=0$,
$d_{0}=N$ and $Nr=|M|$. So assume $\overline{q}>1$. We have $x=N-d_{q'}\geq d_{\hat{q}}$
since $d_{\overline{q}}=d_{q'}+d_{\hat{q}}$ and 
\[
\sum_{c=0}^{z-1}f(c)=\hat{q}rx+(\overline{q}r-\hat{q}r)d_{\hat{q}}=(z-w)x+w(x-y)=zx-wy=|M|
\]

$2)$ First of all note that the expressions of the form $Am+Bn$
with $0\leq A<z$, $0\leq B<f(A)$ are $\sum_{c=0}^{z-1}f(c)=|M|$.
So it is enough to prove that they are all distinct. Assume we have
expressions $Am+Bn=A'm+B'n\text{ with }0\leq A'\leq A<z,0\leq B<f(A),0\leq B'<f(A')$.

$A'=B'=0$, i.e. $Am+Bn=0$. If $A=0$ then $B=0$ since $f(0)=x\leq N$.
If $A>0$, we can write $A=qr$ for some $0<q<\overline{q}$. In particular
$\overline{q}>1$ and $B=d_{q}<f(A)$. If $q<\hat{q}$ then $f(A)=x=N-d_{q'}>d_{q}$
contradicting \ref{lem:fundamental for the case m,n}, while if $q\geq\hat{q}$
then $f(A)=d_{\hat{q}}\leq d_{q}$.

$A'=B=0$, i.e. $Am=B'n$. If $A=0$ then $B'=0$ as above. If $A>0$
we can write $A=qr$ for some $0<q<\overline{q}$. Again $\overline{q}>1$.
In particular $B'=N-d_{q}<f(0)=x=N-d_{q'}$ and so $d_{q'}<d_{q}$,
while $d_{q'}=\max_{0<q<\overline{q}}d_{q}$.

\emph{General case}. We can write $(A-A')m+Bn=B'n$ and we can reduce
the problem to the previous cases since if $B\geq B'$ then $B-B'\leq B<f(A)\leq f(A-A')$,
while if $B<B'$ then $B'-B\leq B'<f(A')\leq f(0)$.\end{proof}
\begin{defn}
Given $l\in M$ we set $(\E_{l},\delta_{l})$ the unique pair for
$l$ such that $0\leq\E_{t}<z\comma0\leq\delta_{t}<f(\E_{t})$ and
we will consider $\E\comma\delta$ as maps $\Z^{M}/\langle e_{0}\rangle\arr\Z$.
We will also write $\E^{\overline{q}}\comma\delta^{\overline{q}}$
if necessary.\end{defn}
\begin{prop}
\label{pro:from an algebra to overline q}Let $A$ be an algebra as
in \ref{not:starting from a cover, m.n generate M}. Then
\[
z_{A}=z_{\overline{q}_{A}}\comma y_{A}=y_{\overline{q}_{A}}\comma\hat{q}_{A}=\hat{q}_{\overline{q}_{A}}\comma\E^{A}=\E^{\overline{q}_{A}}\comma\delta^{A}=\delta^{\overline{q}_{A}}\comma f_{A}=f_{\overline{q}_{A}}
\]
and, if $\lambda_{A}=0$, then $x_{A}=x_{\overline{q}_{A}}\comma w_{A}=w_{\overline{q}_{A}}$.\end{prop}
\begin{proof}
Set $\overline{q}=\overline{q}_{A}$. Then $z_{A}=\overline{q}r=z_{\overline{q}}$
and $z_{A}m=y_{A}n=y_{\overline{q}}n$ implies $y_{A}=y_{\overline{q}}$.
Also $\hat{q}_{A}=\hat{q}_{\overline{q}}$ by definition. Taking into
account \ref{lem:From lambda not zero to zero A'} we can now assume
$\lambda_{A}=0$. We claim that all the remaining equalities follow
from $x_{A}=x_{\overline{q}}$. Indeed clearly $w_{A}=w_{\overline{q}}$.
Also by definition of $f_{\overline{q}}$ and thanks to \ref{lem:From a cover we get overline q}
we will have $f_{A}=f_{\overline{q}}$ and therefore $\E^{A}=\E^{\overline{q}}\comma\delta^{A}=\delta^{\overline{q}}$,
that conclude the proof.

We now show that $x_{A}=x_{\overline{q}}$. If $\overline{q}=1$ then
$\hat{q}=0$ and so, from \ref{lem:From a cover we get overline q},
we have $d_{\hat{q}}=N=x_{A}=x_{1}$. If $\overline{q}>1$, by definition
of $f_{\overline{q}}$ and thanks to \ref{lem:definition of epsilon delta}
and \ref{lem:From a cover we get overline q}, we can write
\[
|M|=\sum_{c=0}^{z_{\overline{q}}-1}f_{\overline{q}}(c)=r\hat{q}_{\overline{q}}x_{\overline{q}}+(z_{\overline{q}}-\hat{q}_{\overline{q}}r)d_{\hat{q}_{\overline{q}}}=\sum_{c=0}^{z_{A}-1}f_{A}(c)=r\hat{q}_{A}x_{A}+(z_{A}-\hat{q}_{A}r)d_{\hat{q}_{A}}
\]
and so $x_{A}=x_{\overline{q}}$.\end{proof}
\begin{defn}
\label{def:universal algebra}Define the $M$-graded $\Z[a,b]$-algebra
\[
A^{\overline{q}}=\Z[a,b][s,t]/(s^{z}-at^{y},t^{x}-bs^{w},s^{\hat{q}r}t^{d_{\hat{q}}}-a^{\gamma}b)\text{ where }\gamma=\left\{ \begin{array}{cc}
0 & \text{if }\overline{q}=1\\
1 & \text{if }\overline{q}>1
\end{array}\right.
\]
with $M$-graduation $\deg s=m$, $\deg t=n$. If are given elements
$a_{0},b_{0}$ of a ring $C$ we will also write $A_{a_{0},b_{0}}^{\overline{q}}=A^{\overline{q}}\otimes_{\Z[a,b]}C$,
where $\Z[a,b]\arr C$ sends $a,b$ to $a_{0},b_{0}$.\end{defn}
\begin{prop}
\label{pro:general algebra for m,n 2}$A^{\overline{q}}\in\MCov(\Z[a,b])$,
it is generated in degrees $m,n$ and $\{v_{l}=s^{\E_{l}}t^{\delta_{l}}\}_{l\in M}$
is an $M$-graded basis for it.\end{prop}
\begin{proof}
We have to prove that, for any $l\in M$, $(A^{\overline{q}})_{l}=\Z[a,b]v_{l}$
and we can check this over a field $k$, i.e. considering $A=A_{a,b}^{\overline{q}}$
with $a,b\in k$. We first consider the case $a,b\in k^{*}$, so that
$s,t\in A^{*}$. Let $\pi\colon\Z^{2}\arr M$ the map such that $\pi(e_{1})=m\comma\pi(e_{2})=n$.
The set $T=\{(a,b)\in\Ker\pi\st s^{a}t^{b}\in k^{*}\}$ is a subgroup
of $\Ker\pi$ such that $(z,-y),(-w,x)\in T$. Since $\det\left(\begin{array}{cc}
z & -w\\
-y & x
\end{array}\right)=zx-wy=|M|$ we can conclude that $T=\Ker\pi$. Therefore $v_{l}$ generate $(A^{\overline{q}})_{l}$
since for any $c,d\in\N$ we have $s^{c}t^{d}/v_{cm+dn}\in k^{*}$
and $0\neq v_{l}\in A^{*}$.

Now assume $a=0$. If $\overline{q}=1$ then $\hat{q}=w=0$, $d_{\hat{q}}=x=N$
and so $A=k[s,t]/(s^{z},t^{N}-b)$ satisfies the requests. If $\overline{q}>1$
it is easy to see that $v_{l}$ generates $A_{l}$. On the other hand
$\dim_{k}A=|\{(A,B)\st0\leq A<z,0\leq B<x,A\leq\hat{q}r\text{ or }B\leq d_{\hat{q}}\}|=zx-(z-\hat{q}r)(x-d_{\hat{q}})=zx-yw=|M|$.
The case $b=0$ is similar.\end{proof}
\begin{thm}
\label{pro:Universal algebras generated by m,n with overline q fixed}Let
$k$ be a field. If $\overline{q}\in\Omega_{N-\alpha,N}$ and $\lambda\in k$,
with $\lambda=0$ if $\overline{q}=N/(\alpha,N)$, then
\[
A_{\overline{q},\lambda}=k[s,t]/(s^{z_{\overline{q}}}-\lambda t^{y_{\overline{q}}},t^{x_{\overline{q}}},s^{\hat{q}_{\overline{q}}r}t^{d_{\hat{q}_{\overline{q}}}})
\]
is an algebra as in \ref{not:starting from a cover, m.n generate M}
with $\overline{q}_{A_{\overline{q},\lambda}}=\overline{q}$ and $\lambda_{A_{\overline{q},\lambda}}=\lambda$.
Conversely, if $A$ is an algebra as in \ref{not:starting from a cover, m.n generate M}
then $\overline{q}_{A}\in\Omega_{N-\alpha,N}$, $\lambda_{A}\in k$,
$\lambda_{A}=0$ if $\overline{q}_{A}=N/(\alpha,N)$ and $A\simeq A_{\overline{q}_{A},\lambda_{A}}$.\end{thm}
\begin{proof}
Consider $A=A_{\overline{q},\lambda}$, which is just $A_{\lambda,0}^{\overline{q}}$.
Clearly $t\notin A^{*}$. On the other hand $s\notin A^{*}$ since
$y=0\iff z=o(m)\iff\overline{q}=N/(\alpha,N)$. Therefore $H_{A/k}=0$
and $A$ is an algebra as in \ref{not:starting from a cover, m.n generate M}.
Moreover clearly $\overline{q}_{A}\leq\overline{q}$. If by contradiction
this inequality is strict, we will have a relation $s^{qr}=\omega t^{y'}$
with $0\leq q<\overline{q}$. Since $s^{qr}=v_{qrm}\neq0$ we will
have that $t^{y'}\neq0$ and $y'<x$, a contradiction thanks to \ref{lem:definition of epsilon delta}.
In particular $\lambda=\lambda_{A}$.

Now let $A$ be as in \ref{not:starting from a cover, m.n generate M}
and set $\overline{q}=\overline{q}_{A}$, $\lambda=\lambda_{A}$.
We already know that $\overline{q}\in\Omega_{N-\alpha,N}$ (see \ref{pro:overlineqA is in omeganminusalpha,n}).
We claim that the map $A_{\overline{q},\lambda}\arr A$ sending $s,t$
to $v_{m},v_{n}$ is well defined and so an isomorphism. Indeed we
have $v_{m}^{z}=\lambda v_{n}^{y}$ by definition and, thanks to \ref{pro:from an algebra to overline q},
we have $v_{m}^{\hat{q}r}v_{n}^{d_{\hat{q}}}=0$ since $d_{\hat{q}}=f_{A}(\hat{q}r)$
and $v_{n}^{x}=0$ since $f_{A}(0)=x$. Finally if $\overline{q}=N/(\alpha,N)$
then $y=y_{A}=0$ and $z=o(m)$, so that $\lambda_{A}=v_{m}^{o(m)}=0$.\end{proof}
\begin{cor}
If $k$ is an algebraically closed field then, up to graded isomorphism,
the algebras as in \ref{not:starting from a cover, m.n generate M}
are exactly $A_{\overline{q},1}$ if $\overline{q}\in\Omega_{N-\alpha,N}-\{N/(\alpha,N)\}$
and $A_{\overline{q},0}$ if $\overline{q}\in\Omega_{N-\alpha,N}$.\end{cor}
\begin{proof}
Clearly the algebras above cannot be isomorphic. Conversely if $\lambda\in k^{*}$
(and $\overline{q}<N/(\alpha,N)$) the transformation $t\arr\sqrt[y]{\lambda}t$
with $y=y_{\overline{q}}$ yields an isomorphism $A_{\overline{q},\lambda}\simeq A_{\overline{q},1}$. 
\end{proof}

\subsection{\label{sub:smooth integral rays for hleqtwo}Smooth  extremal rays
for $h\leq2$.}

In this subsection we continue to keep notation from \ref{not:for alpha,N,r e M},
i.e. $M=M_{r,\alpha,N}$ and we will considered given an element $\overline{q}\in\Omega_{N-\alpha,N}$.
\begin{rem}
We have $z=1\iff\overline{q}=r=1$ and $x=1\iff\overline{q}=N$. Indeed
the first relation is clear, while for the second one note that, by
definition of $x$ and since $N>1$, we have $x=1\iff d_{q'}=N-1\iff\overline{q}=N/(\alpha,N),(\alpha,N)=1$.\end{rem}
\begin{lem}
\label{lem:lambda,delta for overlineq}The vectors of $K_{+}$
\begin{equation}
\begin{array}{lc}
v_{cm,dn} & 0<c<z,0<d<f(c)\\
v_{m,im} & 0<i<z-1\\
v_{n,jn} & 0<j<x-1\\
v_{m,(z-1)m} & \text{if }z>1\\
v_{n,(x-1)n} & \text{if }x>1
\end{array}\label{eq:base for the case m,n}
\end{equation}
form a basis of $K$. Assume $\overline{q}r\neq1$ and $\overline{q}\neq N$,
i.e. $z,x>1$, and denote by $\Lambda,\Delta$ the last two terms
of the dual basis of \ref{eq:base for the case m,n}. Then $\Lambda,\Delta\in\duale K_{+}$
and they form a smooth sequence. Moreover $\Lambda=1/|M|(x\E+w\delta)$,
$\Delta=1/|M|(y\E+z\delta)$ and
\[
\Lambda_{m,-m}=\Delta_{n,-n}=1\comma\Lambda_{n,-n}=\left\{ \begin{array}{cc}
0 & \text{if }\overline{q}=1\\
1 & \text{otherwise}
\end{array}\right.\comma\Delta_{m,-m}=\left\{ \begin{array}{cc}
0 & \text{if }\overline{q}=N/(\alpha,N)\\
1 & \text{otherwise}
\end{array}\right.
\]
\end{lem}
\begin{proof}
Note that we cannot have $z=x=1$ since otherwise $|M|=f(0)=x=1$,
i.e. $M=0$. The vectors of (\ref{eq:base for the case m,n}) are
at most $\rk K$ since
\[
\sum_{c=1}^{z-1}(f(c)-1)+z-2+x-2+2=\sum_{c=0}^{z-1}(f(c)-1)+z-1=|M|-z+z-1=|M|-1=\rk K
\]
If $z=1$ then (\ref{eq:base for the case m,n}) is $v_{n,n},\dots,v_{n,(x-1)n}$.
So $x=|M|=N$, i.e. $n$ generates $M$, and (\ref{eq:base for the case m,n})
is a base of $K$. In the same way if $x=1$, then $m$ generates
$M$ and (\ref{eq:base for the case m,n}) is a base of $K$.

So we can assume that $z,x>1$. The functions $\E$ and $\delta$
define a map $\Z^{M}/\langle e_{0}\rangle\arrdi{(\E,\delta)}\Z^{2}$.
Denote by $K'$ the subgroup of $K$ generated by the vectors in (\ref{eq:base for the case m,n}),
except for the last two lines. We claim that $(\E,\delta)_{|K'}=0$.
This follows by a direct computation just observing that if we have
an expression $Am+Bn$ as in \ref{lem:definition of epsilon delta},
\ref{enu:writing of the form Am+Bn}) then $(\E,\delta)(e_{Am+Bn})=(A,B)$.
Consider the diagram   \[   \begin{tikzpicture}[xscale=2.2,yscale=-1.2]     
\node (A0_0) at (0.1, 0.6) {$\scriptstyle{\sigma(e_1)=v_{m,(z-1)m}},$};     
\node (A0_1) at (1.1, 0.6) {$\scriptstyle{\sigma(e_2)=v_{n,(x-1)n}}$};     
\node (A1_0) at (0, 1) {$\Z^2$};     
\node (A1_1) at (1, 1) {$K/K'$};     
\node (A1_2) at (2, 1) {$\Z^M/\langle e_0, K'\rangle$};
\node (A1_3) at (3, 1) {$\Z^2$};     
\node (A1_4) at (4, 1) {$\Z^M/\langle e_0, K'\rangle$};
\node (A1_5) at (5, 1) {$M$};     
\node (A2_3) at (3.2, 1.4) {$\scriptstyle{\tau(e_1)=e_m},$};     
\node (A2_4) at (3.8, 1.4) {$\scriptstyle{\tau(e_2)=e_n}$};     
\node (A2_5) at (4.7, 1.4) {$\scriptstyle{p(e_l)=l}$};     

\path (A1_4) edge [->]node [auto,swap] {$\scriptstyle{p}$} (A1_5);     \path (A1_0) edge [->]node [auto] {$\scriptstyle{\sigma}$} (A1_1);     \path (A1_0) edge [->,bend left=35]node [auto] {$\scriptstyle{U}$} (A1_3);     \path (A1_1) edge [right hook->]node [auto] {$\scriptstyle{}$} (A1_2);     \path (A1_2) edge [->]node [auto] {$\scriptstyle{(\E,\delta)}$} (A1_3);     \path (A1_3) edge [->,bend right=45]node [auto] {$\scriptstyle{\pi}$} (A1_5);     \path (A1_3) edge [->]node [auto,swap] {$\scriptstyle{\tau}$} (A1_4);   \end{tikzpicture}   \]  We have $(\E,\delta)(v_{m,(z-1)m})=(z,-y)$ since $y<x=f(0)$ and
$(\E,\delta)(v_{n,(x-1)n})=(-w,x)$ since $w<z$. So $|\det U|=zx-yw=|M|$
and, since $\pi\circ U=0$, $U$ is an isomorphism onto $\Ker\pi$.
Moreover $\tau^{-1}=(\E,\delta)$ since $e_{l}\equiv\E_{l}e_{m}+\delta_{l}e_{n}\text{ mod }K'$.
It follows that $\sigma$ is an isomorphism and so (\ref{eq:base for the case m,n})
is a basis of $K$.

Consider now the second part of the statement. Clearly $\Lambda,\Delta\in\langle\E,\delta\rangle_{\Q}$.
Therefore we have
\[
\Lambda=a\E+b\delta,\left\{ \begin{array}{c}
\Lambda(v_{m,(z-1)m})=1=az-yb\\
\Lambda(v_{n,(x-1)n})=0=xb-aw
\end{array}\then\left\{ \begin{array}{c}
a=x/|M|\\
b=w/|M|
\end{array}\right.\right.
\]
and the analogous relation for $\Delta$ follows in the same way.
Now note that, thanks to \ref{pro:Universal algebras generated by m,n with overline q fixed}
and \ref{pro:from an algebra to overline q}, we have that $\E=\E^{A}\comma\delta=\delta^{A}$
for an algebra $A$ as in \ref{not:starting from a cover, m.n generate M}
with $\overline{q}_{A}=\overline{q}$, $\lambda_{A}=0$ and sharing
the same invariants of $\overline{q}$. So we can apply \ref{lem:computation of E and delta}.
We want to prove that $\Lambda,\Delta\in\duale K_{+}$ so that they
form a smooth sequence by construction. Assume first that $\E_{a,b}>0$.
Clearly $\Lambda_{a,b},\Delta_{a,b}\geq0$ if $\delta_{a,b}\geq0$.
On the other hand if $\delta_{a,b}<0$ we know that $\E_{a,b}\geq z$
and $\delta_{a,b}\geq-y$ and so
\[
|M|\Lambda_{a,b}=x\E_{a,b}+w\delta_{a,b}\geq xz-yw=|M|\text{ and }|M|\Delta_{a,b}=y\E_{a,b}+z\delta_{a,b}\geq yz-zy=0
\]
The other cases follows in the same way. It remains to prove the last
relations. Since $-n=\hat{q}rm+(d_{\hat{q}}-1)n$, we have $\E_{n,-n}=\hat{q}r$
and $\delta_{n,-n}=d_{\hat{q}}$. Using the relation $zx-wy=|M|$
the values of $\Lambda_{n,-n}\comma\Delta_{n,-n}$ can be checked
by a direct computation. Similarly, considering the relations $-m=(\hat{q}r-1)m+d_{\hat{q}}n$
if $1<\overline{q}$, $-m=(r-1)m+(N-\alpha)n$ if $\overline{q}=1$
and $\alpha\neq0$, $-m=(r-1)m$ if $\alpha=0$, we can compute the
values of $\Lambda_{m,-m}$ and $\Delta_{m,-m}$.\end{proof}
\begin{prop}
\label{pro:general algebra for m,n}The multiplication of $A^{\overline{q}}$
(see \ref{def:universal algebra}) with respect to the basis $v_{l}=v_{m}^{\E_{l}}v_{n}^{\delta_{l}}$
is: $a^{\E^{\phi}}$ if $\overline{q}=N$, where $\phi\colon M\arrdi{\simeq}\Z/|M|\Z,$
$\phi(m)=1$; $b^{\E^{\eta}}$ if $\overline{q}r=1$, where $\eta\colon M\arrdi{\simeq}\Z/|M|\Z,$
$\phi(n)=1$; $a^{\Lambda}b^{\Delta}$ if $\overline{q}r\neq1\comma\overline{q}\neq N$,
where $\Lambda,\Delta$ are the rays defined in \ref{lem:lambda,delta for overlineq}.\end{prop}
\begin{proof}
In the proof of \ref{pro:general algebra for m,n 2} we have seen
that if $x=1$ $(\overline{q}=N)$, then $M=\langle m\rangle$ and
$A^{\overline{q}}=\Z[a,b][s]/(s^{|M|}-a)$, while if $z=1$ $(\overline{q}r=1)$
then $M=\langle n\rangle$ and $A^{\overline{q}}=\Z[a,b][t]/(t^{|M|}-b)$.
So we can assume $x,z>1$. Let $B$ the $\Di M$-cover over $\Z[a,b]$
given by multiplication $\psi=a^{\Lambda}b^{\Delta}$ and denote by
$\{\omega_{l}\}_{l\in M}$ a graded basis (inducing $\psi)$. By definition
of $\Lambda,\Delta$ we have $\omega_{l}=\omega_{m}^{\E_{l}}\omega_{n}^{\delta_{l}}$
for any $l\in M$ and $\psi_{m,(z-1)m}=a,\;\psi_{n,(x-1)n}=b$. Therefore
\[
\omega_{m}^{z}=\omega_{m}\omega_{(z-1)m}=a\omega_{zm}=a\omega_{yn}=a\omega_{n}^{y}\comma\omega_{n}^{x}=\omega_{n}\omega_{(x-1)n}=b\omega_{xn}=b\omega_{wm}=b\omega_{m}^{w}
\]
and, checking both cases $\overline{q}=1$ and $\overline{q}>1$,
$\omega_{m}^{\hat{q}r}\omega_{n}^{d_{\hat{q}}}=\omega_{-n}\omega_{n}=a^{\Lambda_{n,.n}}b^{\Delta_{n,.n}}=a^{\gamma}b$.
In particular we have an isomorphism $A^{\overline{q}}\arr B$ sending
$v_{m},v_{n}$ to $\omega_{m},\omega_{n}$.\end{proof}
\begin{notation}
From now on M will be any finite abelian group. If $\phi\colon M\arr M_{r,\alpha,N}$
is a surjective map, $r,\alpha,N$ satisfy the conditions of \ref{not:for alpha,N,r e M},
$\overline{q}\in\Omega_{N-\alpha,N}$ with $\overline{q}r\neq1\comma\overline{q}\neq N$
then we set $\Lambda^{r,\alpha,N,\overline{q},\phi}=\Lambda\circ\phi_{*}\comma\Delta^{r,\alpha,N,\overline{q},\phi}=\Delta\circ\phi_{*}$,
where $\Lambda,\Delta$ are the rays defined in \ref{lem:lambda,delta for overlineq}
with respect to $r,\alpha,N,\overline{q}$. If $\phi=\id$ we will
omit it.\end{notation}
\begin{defn}
Set
\[
\Sigma_{M}=\left\{ (r,\alpha,N,\overline{q},\phi)\left|\begin{array}{c}
0\leq\alpha<N\comma r>0\comma N>1\comma(r>1\text{ or }\alpha>1)\\
\overline{q}\in\Omega_{N-\alpha,N}\comma\overline{q}r\neq1\comma\overline{q}\alpha\not\equiv1\text{ mod }N\\
\overline{q}\neq N/(\alpha,N)\comma\phi\colon M\arr M_{r,\alpha,N}\text{ surjective}
\end{array}\right.\right\} 
\]
and $\Delta^{*}\colon\Sigma_{M}\arr\{\text{smooth extremal rays of }M\}$.\end{defn}
\begin{rem}
Since $e_{2},e_{1}$ generate $M_{r,\alpha,N}$, there exist unique
$\duale r,\duale{\alpha},\duale N$ with an isomorphism $\duale{(-)}\colon M_{r,\alpha,N}\arr M_{\duale r,\duale{\alpha},\duale N}$
sending $e_{2},e_{1}$ to $e_{1},e_{2}$. One can check that $\duale r=(\alpha,N)$,
$\duale N=rN/(\alpha,N)$ and $\duale{\alpha}=\tilde{q}r$, where
$\tilde{q}$ is the only integer $0\leq\tilde{q}<N/(\alpha,N)$ such
that $\tilde{q}\alpha\equiv(\alpha,N)\text{ mod }N$. 

If $A$ is an algebra as in \ref{not:starting from a cover, m.n generate M}
for $M_{r,\alpha,N}$, then, through $\duale{(-)}$, $A$ can be thought
of as a $M_{\duale r,\duale{\alpha},\duale N}$-cover, that we will
denote by $\duale A$, and $\duale A$ is an algebra as in \ref{not:starting from a cover, m.n generate M}
with respect to $M_{\duale r,\duale{\alpha},\duale N}$, with $\overline{q}_{\duale A}=x_{A}/(\alpha,N)$,
$\lambda_{\duale A}=\mu_{A}$. We can define a bijection $\duale{(-)}\colon\Omega_{N-\alpha,N}-\{N/(N,\alpha)\}\arr\Omega_{\duale N-\duale{\alpha},\duale N}-\{\duale N/(\duale{\alpha},\duale N)\}$
in the following way. Given $\overline{q}$ take an algebra $A$ as
in \ref{not:starting from a cover, m.n generate M} for $M_{r,\alpha,N}$
with $\overline{q}_{A}=\overline{q}$ and $\lambda_{A}\neq0$, which
exists thanks to \ref{pro:Universal algebras generated by m,n with overline q fixed},
and set $\duale{\overline{q}}=\overline{q}_{\duale A}$. Taking into
account \ref{rem:lambda is 0 iff mu is 0} and \ref{pro:from an algebra to overline q},
$\duale{\overline{q}}=y_{\overline{q}}/(\alpha,N)$ since $x_{A}=y_{A}=y_{\overline{q}}$
and $\duale{(-)}$ is well defined and bijective since $\lambda_{\duale A}=\mu_{A}=\lambda_{A}^{-1}$.
Note that the condition $\overline{q}\alpha\equiv1\text{ mod }N$
is equivalent to $\duale r=1$ and $\duale{\overline{q}}=1$

Finally if $\phi\colon M\arr M_{r,\alpha,N}$ is a surjective morphism
then we set $\duale{\phi}=\duale{(-)}\circ\phi\colon M\arr M_{\duale r,\duale{\alpha},\duale N}$.
Note that in any case we have the relation $\duale{\duale{(-)}}=\id$.
In particular, since $\duale 1=\alpha/\duale r$, $\overline{q}=\duale{\alpha}/r$
is the dual of $1\in\Omega_{\duale N-\duale{\alpha},\duale N}$.\end{rem}
\begin{prop}
\label{pro:trivial ray for lambda,delta}Let $r,\alpha,N$ be as in
\ref{not:for alpha,N,r e M}, $\overline{q}\in\Omega_{N-\alpha,N}$
with $\overline{q}r\neq1\comma\overline{q}\neq N$ and $\phi\colon M\arr M_{r,\alpha,N}$
be a surjective map. Set $\chi=(r,\alpha,N,\overline{q},\phi)$. Then
\begin{enumerate}
\item $\overline{q}=N/(\alpha,N)$: $\Delta^{\chi}=\E^{\xi}$, $\xi\colon M\arrdi{\phi}M_{r,\alpha,N}\arr M_{r,\alpha,N}/\langle m\rangle\simeq\langle n\rangle\simeq\Z/(\alpha,N)\Z$;
$\overline{q}\alpha\equiv1\text{ mod }N$: $\Delta^{\chi}=\E^{\zeta}$,
$\zeta\colon M\arrdi{\phi}M_{r,\alpha,N}=\langle e_{1}\rangle$;
\item $\overline{q}=1$: $\Lambda^{\chi}=\E^{\omega}$, $\omega\colon M\arrdi{\phi}M_{r,\alpha,N}\arr M_{r,\alpha,N}/\langle n\rangle=\langle m\rangle\simeq\Z/r\Z$;\newline$w_{\overline{q}}=1$:
$\Lambda^{\chi}=\E^{\theta},\theta\colon M\arrdi{\phi}M_{r,\alpha,N}=\langle e_{2}\rangle$;
\item $\overline{q}>1$ and $w_{\overline{q}}\neq1$: $\Lambda^{\chi}=\Delta^{r,\alpha,N,\overline{q}-\hat{q},\phi}$.
\end{enumerate}
In particular in the first two cases we have $h_{\Lambda^{\chi}}=h_{\Delta^{\chi}}=1$.\end{prop}
\begin{proof}
We can assume $M=M_{r,\alpha,N}$ and $\phi=\id$. The algebra associated
to $0^{\Lambda^{\chi}}$, $0^{\Delta^{\chi}}$ are respectively $C_{\overline{q}}=k[s,t]/(s^{z},t^{x}-s^{w},s^{\hat{q}r}t^{d_{\hat{q}}}-0^{\gamma})$,
$B_{\overline{q}}=k[s,t]/(s^{z}-t^{y},t^{x},s^{\hat{q}r}t^{d_{\hat{q}}})$
by \ref{pro:general algebra for m,n}. 

$1)$ If $\overline{q}=N/(\alpha,N)$, then $z=o(m)$, $y=0$, $d_{\hat{q}}=(\alpha,N)$
and so $B_{\overline{q}}=k[s,t]/(s^{o(m)}-1,t^{(\alpha,N)})$, the
algebra associated to $0^{\E^{\xi}}$. If $\overline{q}\alpha\equiv1\text{ mod }N$
then $\duale r=(\alpha,N)=1$ and $\overline{q}=\duale{\alpha}/r$,
i.e. $\duale{\overline{q}}=1$. So $y=1$ and $B_{\overline{q}}\simeq k[s]/(s^{|M|})$,
the algebra associated to $0^{\E^{\gamma}}$. 

$2)$ If $\overline{q}=1$, then $z=r$, $\hat{q}=w=0$, $x=d_{\hat{q}}=N$
and so $C_{1}=k[s,t](t^{n}-1,s^{r})$, the algebra associated to $0^{\E^{\omega}}$.
If $w=1$ then $\overline{q}>1$ and so $C_{\overline{q}}=k[t]/(t^{|M|}),$
the algebra associated to $0^{\E^{\theta}}$.

$3)$ If $\overline{q}>1$ then $H_{C_{\overline{q}}}=0$ and so $C_{\overline{q}}$
is an algebra as in \ref{not:starting from a cover, m.n generate M}.
An easy computation shows that $z_{C_{\overline{q}}}=w>1$, so that
$\overline{q}_{C_{\overline{q}}}=\overline{q}-\hat{q}$ and $\lambda_{\overline{q}}=1$.
Therefore $\Lambda^{\chi}=\Delta^{r,\alpha,N,\overline{q}-\hat{q}}$
by \ref{pro:Universal algebras generated by m,n with overline q fixed}. \end{proof}
\begin{prop}
\label{pro:classification sm int ray htwo}$\duale{\Sigma}_{M}=\Sigma_{M}$
and we have a bijection 
\[
\Delta^{*}\colon\Sigma_{M}/\duale{(-)}\arr\{\text{smooth extremal rays }\E\text{ with }h_{\E}=2\}
\]
\end{prop}
\begin{proof}
$\duale{\Sigma}_{M}\subseteq\Sigma_{M}$ since $\overline{q}\alpha\not\equiv1\text{ mod }N$
is equivalent to $\duale{\overline{q}}\duale r\neq1$. Now let $\E$
be a smooth extremal ray such that $h_{\E}=2$ and $A$ the associated
algebra over some field $k$. We can assume $H_{A/k}=H_{\E}=0$. The
relation $h_{\E}=2$ means that there exist $0\neq m,n\in M$, $m\neq n$
such that $ $$A$ is generated in degrees $m,n$. So $M=M_{r,\alpha,N}$
as in \ref{not:for alpha,N,r e M} and $A$ is an algebra as in \ref{not:starting from a cover, m.n generate M}.
By \ref{pro:Universal algebras generated by m,n with overline q fixed}
and \ref{pro:trivial ray for lambda,delta} we can conclude that there
exist $\chi\in\Sigma_{M}$ such that $\E=\Delta^{\chi}$.

Now let $\chi=(r,\alpha,N,\overline{q},\phi)\in\Sigma_{M}$. We have
to prove that $h_{\Delta^{\chi}}=2$ and, since $M_{r,\alpha,N}\neq0$,
assume by contradiction that $h_{\Delta^{\chi}}=1$. We can assume
$M=M_{r,\alpha,N}$ and $\phi=\id$. Note that $h_{\Delta^{\chi}}=1$
means that the associated algebra $B$ is generated in degree $m$
or $n$. If $A$ is an algebra as in \ref{not:starting from a cover, m.n generate M},
then $A$ is generated in degree $n$ if and only if $z=1$, that
means $\overline{q}r=1$. So $B$ is generated in degree $m$, i.e.
$\duale B$ is generated in degree $e_{2}\in M_{\duale r,\duale{\alpha},\duale N}$,
which is equivalent to $1=z_{\duale B}=\duale{\overline{q}}\duale r=1$,
and, as we have seen, to $\overline{q}\alpha\equiv1\text{ mod N}$. 

Now let $\chi'=(r',\alpha',N',\overline{q}',\phi')\in\Sigma_{M}$
such that $\E=\Delta^{\chi}=\Delta^{\chi'}$. Again we can assume
$H_{\E}=0$ and take $B,B'$ the algebras associated respectively
to $\chi,\chi'$. By definition of $\Delta_{*}$, $\phi,\phi'$ are
isomorphisms. If $g=\phi'\circ\phi^{-1}\colon M_{r,\alpha,N}\arr M_{r',\alpha',N'}$
then we have a graded isomorphism $p\colon B\arr B'$ such that $p(B_{l})=B'_{g(l)}$.
Therefore $g(\{e_{1},e_{2}\})=\{e_{1},e_{2}\}$, i.e. $g=\id$ or
$g=\duale{(-)}$. It is now easy to show that $\chi'=\chi$ or $\chi'=\duale{\chi}$.\end{proof}
\begin{notation}
We set $\Phi_{M}=\{\phi\colon M\arr\Z/l\Z\st l>1\comma\phi\text{ surjective}\}$,
$\Theta_{M}^{2}=\{\E^{\phi}\}_{\phi\in\Phi_{M}}\cup\{(\Lambda^{\chi},\Delta^{\chi})\}_{\chi\in\overline{\Sigma}_{M}}$,
where $\overline{\Sigma}_{M}$ is the set of sequences $(r,\alpha,N,\overline{q},\phi)$
where $r,\alpha,N\in\N$ satisfy $0\leq\alpha<N,r>0,r>1\text{ or }\alpha>1$,
$\overline{q}\in\Omega_{N-\alpha,N}$ satisfy $\overline{q}r\neq1$,
$\overline{q}\neq N$ and $\phi\colon M\arr M_{r,\alpha,N}$ is a
surjective map. Finally set $\underline{\E}=(\E^{\phi},\Delta^{\chi})_{\phi\in\Phi_{M},\chi\in\Sigma_{M}/\duale{(-)}}$.\end{notation}
\begin{thm}
\label{thm:fundamental thm for hleqtwo}Let $M$ be a finite abelian
group. Then 
\[
\{h\leq2\}=(\bigcup_{\phi\in\Phi_{M}}\stZ_{M}^{\E^{\phi}})\bigcup(\bigcup_{(\Lambda,\Delta)\in\Theta_{M}^{2}}\stZ_{M}^{\Lambda,\Delta})
\]
In particular $\{h\leq2\}\subseteq\stZ_{M}^{\textup{sm}}$. Moreover
$\pi_{\underline{\E}}\colon\stF_{\underline{\E}}\arr\MCov$ induces
an equivalence of categories
\[
\left\{ (\underline{\shL},\underline{\shM},\underline{z},\lambda)\in\stF_{\underline{\E}}\left|\begin{array}{c}
V(z_{\E^{1}})\cap\cdots\cap V(z_{\E^{r}})\neq\emptyset\text{ iff}\\
r=1\text{ or }(r=2\text{ and }(\E^{1},\E^{2})\in\Theta_{M}^{2})
\end{array}\right.\right\} =\pi_{\underline{\E}}^{-1}(h\leq2)\arrdi{\simeq}\{h\leq2\}
\]
\end{thm}
\begin{proof}
The expression of $\{h\leq2\}$ follows from \ref{pro:Universal algebras generated by m,n with overline q fixed}
and \ref{pro:general algebra for m,n}. Taking into account \ref{pro:classification sm int ray htwo},
the last part instead follows from \ref{pro:piE for theta isomorphism}
taking $\Theta=\Theta_{M}^{2}$.
\end{proof}
In \cite{Maclagan2002} the authors prove that the toric Hilbert schemes
associated to a polynomial algebra in two variables are smooth and
irreducible. The same result is true more generally for multigraded
Hilbert schemes, as proved later in \cite{Maclagan2010}. Here we
obtain an alternative proof in the particular case of equivariant
Hilbert schemes:
\begin{cor}
\label{cor:MHilbmn smooth and irreducible}If $M$ is a finite abelian
group and $m,n\in M$ then $\MHilb^{m,n}$ is smooth and irreducible.\end{cor}
\begin{proof}
Taking into account the diagram in \ref{rem:relation Mhilbm MCovm}
it is enough to note that $\MCov^{m,n}\subseteq\{h\leq2\}\subseteq\stZ_{M}^{\textup{sm}}$.\end{proof}
\begin{prop}
\label{pro:when sigmaM is empty}$\Sigma_{M}=\emptyset$ if and only
if $M\simeq(\Z/2\Z)^{l}$ or $M\simeq(\Z/3\Z)^{l}$.\end{prop}
\begin{proof}
For the only if, note that if $\phi\colon M\arr\Z/l\Z$ with $l>3$
is surjective, then, taking $m=l-1\comma n=1\in\Z/l\Z$, we have $\Z/l\Z\simeq M_{1,l-1,l}$
and $(1,l-1,l,2,\phi)\in\Sigma_{M}$.

For the converse set $M=(\Z/p\Z)^{l}$, where $p=2,3$ and, by contradiction,
assume we have $(r,\alpha,N,\overline{q},\phi)\in\Sigma_{M}$. In
particular $\phi$ is a surjective map $M\arr M_{r,\alpha,N}$. If
$e_{1},e_{2}\in M_{r,\alpha,N}$ are $\F_{p}$-independent then $M_{r,\alpha,N}=\langle e_{1}\rangle\times\langle e_{2}\rangle$,
$\alpha=0$, $\Omega_{N-\alpha,N}=\{1\}$ and therefore $\overline{q}=1=N/(\alpha,N)$,
which implies that $\chi\notin\Sigma_{M}$. On the other hand, if
$M_{1,\alpha,p}\simeq\Z/p\Z$, the only extremal rays for $\Z/p\Z$
are $\E^{\id}$ and, if $p=3$, $\E^{-\id}$ since $K_{+\Z/p\Z}\simeq\N^{p-1}$
by \ref{pro:smooth DMCov}.\end{proof}
\begin{thm}
\label{thm:fundamental thm locally factoria hleqtwo}Let $M$ be a
finite abelian group and $X$ be a locally noetherian and locally
factorial scheme. Consider the full subcategories 
\[
\catC_{X}^{2}=\left\{ (\underline{\shL},\underline{\shM},\underline{z},\delta)\in\stF_{\underline{\E}}(X)\left|\begin{array}{c}
\codim_{X}V(z_{i_{1}})\cap\cdots\cap V(z_{i_{s}})\geq2\\
\text{if }\nexists\underline{\delta}\in\Theta_{M}^{2}\text{ s.t. }\E^{i_{1}},\dots\E^{i_{s}}\subseteq\underline{\delta}
\end{array}\right.\right\} \subseteq\shF_{\underline{\E}}(X)
\]
and
\[
\catD_{X}^{2}=\{Y\arrdi fX\in\MCov(X)\st h_{f}(p)\leq2\ \forall p\in X\text{ with }\codim_{p}X\leq1\}\subseteq\MCov(X)
\]
Then $\pi_{\underline{\E}}$ induces an equivalence of categories
\[
\catC_{X}^{2}=\pi_{\underline{\E}}^{-1}(\catD_{X}^{2})\arrdi{\simeq}\catD_{X}^{2}
\]
\end{thm}
\begin{proof}
Apply \ref{thm:fundamental theorem for locally factorial schemes}
with $\Theta=\Theta_{M}^{2}$.\end{proof}
\begin{rem}
In general $\{h\leq3\}$ does not belong to the smooth locus on $\stZ_{M}$.
For example, if $M=\Z/4\Z$, $\MCov=\{h\leq3\}$ is integral but not
smooth by \ref{pro:smooth DMCov} and \ref{rem:MCov for M=00003DZfour}.
\end{rem}

\subsection{Normal crossing in codimension 1.}

In this subsection we want describe, in the spirit of classification
\ref{thm:regular in codimension 1 covers}, normal crossing in codimension
$1$ covers of a locally noetherian and locally factorial scheme with
no isolated points and with $(\car X,|M|)=1$.
\begin{defn}
\label{def:normal crossing codimesion one}A scheme $X$ is normal
crossing in codimension $1$ if for any codimension $1$ point $p\in X$
there exists a local and etale map $\widehat{\odi{}}_{X,p}\arr R$,
where $R$ is $k[[x]]$ or $k[[s,t]]/(st)$ for some field $k$ and
$ $$\widehat{\odi{}}_{X,p}$ denote the completion of $\odi{X,p}$.\end{defn}
\begin{rem}
If $X$ is locally of finite type over a perfect field $k$, one can
show that the above condition is equivalent to having an open subset
$U\subseteq X$ such that $\codim_{X}X-U\geq2$ and there exists an
etale coverings $\{U_{i}\arr U\}$ with etale maps $U_{i}\arr\Spec k[x_{1},\dots,x_{n_{i}}]/(x_{1}\cdots x_{r_{i}})$
for any $i$. Anyway we will not use this property.\end{rem}
\begin{notation}
In this subsection we will consider a field $k$ and we will set $A=k[[s,t]]/(st)$.
Given an element $\xi\in\Aut_{k}k[[x]]$ we will write $\xi_{x}=\xi(x)$
so that, if $p\in k[[x]]$ then $\xi(p)(x)=p(\xi_{x})$. We will call
$I\in\Aut_{k}k[[s,t]]$ the unique map such that $I(s)=t\comma I(t)=s$.
Given $B\in k^{*}$ we will denote by $\underline{B}$ the automorphism
of $k[[x]]$ such that $\underline{B}_{x}=Bx$.

Finally, given $f\in k[[x_{1},\dots,x_{n}]]$ and $g\in k[x_{1},\dots,x_{n}]$
the notation $f=g+\cdots$ will mean $f\equiv g\text{ mod }(x_{1},\dots,x_{r})^{\deg g+1}$.
\end{notation}
The first problem to deal with is to describe the action on $A$ of
a finite group $M$ and check when $A$ is a $\Di M$-cover over $A^{M}$,
assuming to have the $|M|$-roots of unity in $k$. We start collecting
some general facts about $A$.
\begin{prop}
We have:
\begin{enumerate}
\item $A=k\oplus sk[[s]]\oplus tk[[t]]$
\item Given $f,g\in A-\{0\}$ then $fg=0$ if and only if $f\in sk[[s]],g\in tk[[t]]$
or vice versa.
\item Any automorphism in $\Aut_{k}A$ is of the form $(\xi,\eta)$ or $I(\xi,\eta)$
where $\xi,\eta\in\Aut_{k}k[[x]]$ and $(\xi,\eta)(f(s,t))=f(\xi_{s},\eta_{t})$.
\item If $\xi\in\Aut_{k}k[[x]]$ has finite order then $\xi=\underline{B}$
where $B$ is a root of unity in $k$. In particular if $(\xi,\eta)\in\Aut_{k}A$
has finite order then $\xi=\underline{B}\comma\eta=\underline{C}$
where $B\comma C$ are roots of unity in $k$.
\item Let $f\in k[[x]]-\{0\}$, $B,C$ roots of unity in $k$. Then $f(Bx)=Cf(x)$
if and only if $C=B^{r}$ for some $r>0$ and, if we choose the minimum
$ $$r$, $f\in x^{r}k[[x^{o(B)}]]$.
\end{enumerate}
\end{prop}
\begin{proof}
$1)$ is straightforward and $2)$ follows easily expressing $f$
and $g$ as in $1)$. For $3)$ note that if $\theta\in\Aut_{k}A$
then $\theta(s)\theta(t)=0$ and apply $2)$. Finally $4)$ and $5)$
can be shown looking at the coefficients of $\xi_{x}$ and of $f$.\end{proof}
\begin{lem}
\label{lem:NC excluding stupid subgroups}If $M<\Aut_{k}A$ is a finite
subgroup containing only automorphisms of the form $(\xi,\eta)$ then
$A^{M}\simeq A$.\end{lem}
\begin{proof}
It is easy to show that $A^{M}\simeq k[[s^{a},t^{b}]]/(s^{a}t^{b})\simeq A$
where $a=\text{lcm}\{i\st\exists(\underline{A},\underline{B})\in M\text{ s.t. }\ord A=i\}$
and $b=\text{lcm}\{i\st\exists(\underline{A},\underline{B})\in M\text{ s.t. }\ord B=i\}$.
\end{proof}
Since we are interested in covers of regular in codimension $1$ schemes
(and $A$ is clearly not regular) we can focus on subgroups $M<\Aut_{k}A$
containing some $I(\xi,\eta)$.
\begin{lem}
\label{lem:classification of actions for NC}Let $M<\Aut_{k}A$ be
a finite abelian group and assume that $(\car k,|M|)=1$ and that
there exists $I(\xi,\eta)\in M$. Then, up to equivariant automorphisms,
we have $M=\langle I(\id,\underline{B})\rangle$ or, if $M$ is not
cyclic, $M=\langle(\underline{C},\underline{C})\rangle\times\langle I\rangle$
where $\underline{B}\comma\underline{C}$ are roots of unity and $o(C)$
is even.\end{lem}
\begin{proof}
The existence of an element of the form $I(\xi,\eta)$ in $M$ implies
that $s$ and $t$ cannot be homogeneous in $m_{A}/m_{A}^{2}$, that
$2\mid|M|$ and therefore that $\car k\neq2$. 

Applying the exact functor $\Hom_{k}^{M}(m_{A}/m_{A}^{2},-)$, we
get that the surjection $m_{A}\arr m_{A}/m_{A}^{2}$ has a $k$-linear
and $M$-equivariant section. This means that there exists $x,y\in m_{A}$
such that $m_{A}=(x,y)$ and $M$ acts on $x\comma y$ with characters
$\chi,\zeta$. In this way we get an action of $M$ on $k[[X,Y]]$
and an equivariant surjective map $\phi\colon k[[X,Y]]\arr A$. Moreover
$\Ker\phi=(h)$, where $h=fg$ and $f\comma g\in k[[X,Y]]$ are such
that $\phi(f)=s$, $\phi(g)=t$. We can write $f=aX+bY+\cdots\comma g=cX+dY+\cdots$
with $ad-bc\neq0$. Since $ax+by=s$ in $m_{A}/m_{A}^{2}$ and $s$
is not homogeneous there, we have $a,b\neq0$. Similarly we get $c,d\neq0$.
In particular, up to normalize $f\comma g\comma x$ we can assume
$b=c=d=1$. Now $h=aX^{2}+(a+1)XY+Y^{2}+\cdots$ and applying Weierstrass
preparation theorem \cite[Theorem 9.2]{Lang2002}, there exists a
unique $\tilde{h}\in(h)$ such that $(\tilde{h})=(h)$ and $\tilde{h}=\psi_{0}(X)+\psi_{1}(X)Y+Y^{2}$.
The uniqueness of $\tilde{h}$ and the $M$-invariance of $(h)$ yield
the relations $m(\tilde{h})=\eta(m)^{2}\tilde{h}$,
\begin{equation}
m(\psi_{0})=\psi_{0}(\chi(m)X)=\eta(m)^{2}\psi_{0}\comma m(\psi_{1})=\psi_{1}(\chi(m)X)=\eta(m)\psi_{1}\label{eq:Relation for describe A NC}
\end{equation}
for any $m\in M$. Moreover $\tilde{h}=\mu h$ where $\mu\in k[[X,Y]]^{*}$
and, since the coefficient of $Y^{2}$ in both $h$ and $\tilde{h}$
is $1$, we also have $\mu(0)=1$. In particular $\psi_{0}=aX^{2}+\cdots\text{ and }\psi_{1}=(a+1)X+\cdots$
and so $(a+1)(\chi-\zeta)=0$ by \ref{eq:Relation for describe A NC}.
Since $s$ is not homogeneous in $m_{A}/m_{A}^{2}$, $\chi\neq\eta$
and $a=-1$. Since $\car k\neq2$ we can write $\tilde{h}=(Y+\psi_{1}/2)^{2}-(\psi_{1}^{2}/4-\psi_{0})=y'^{2}-z'$.
Note that $y',z'$ are homogeneous thanks to \ref{eq:Relation for describe A NC}.
Moreover, by Hensel's lemma, we can write $z'=X^{2}+\cdots=X^{2}q^{2}$
for an homogeneous $q\in k[[x]]$ with $q(0)=1$. So $x'=xq$ is homogeneous
and $\tilde{h}=y'^{2}-x'^{2}$. This means that we can assume $s=x-y$,
$t=x+y$. In particular $\chi^{2}=\eta^{2}$ and $M$ acts on $s,t$
as
\[
m(s)=\frac{\chi+\zeta}{2}(m)s+\frac{\chi-\zeta}{2}(m)t\qquad m(t)=\frac{\chi-\zeta}{2}(m)s+\frac{\chi+\zeta}{2}(m)t
\]
Consider the exact sequence
\begin{equation}
0\arr H\arr M\arrdi{\chi/\eta}\{-1,1\}\arr0\label{eq:exact sequence for NC}
\end{equation}
If $M$ is cyclic, say $M=\langle m\rangle$, we have $\chi(m)=-\eta(m)$
and so $m=I(\underline{B},\underline{B})$, where $B=(\chi(m)-\eta(m))/2$
is a root of unity. Up to normalize $s$ we can write $m=I(\id,\underline{B})$.

Now assume that $M$ is not cyclic. The group $H$ acts on $s$ and
$t$ with the character $\chi_{|H}=\zeta_{|H}$ and this yields an
injective homomorphism $\chi_{|H}\colon H\arr\{\text{roots of unity of }k\}$.
So $H=\langle(\underline{C},\underline{C})\rangle$ for some root
of unity $C$. The extension \ref{eq:exact sequence for NC} corresponds
to an element of $\Ext^{1}(\Z/2\Z,H)\simeq H/2H$ that differs to
the sequence $0\arr H\arr\Z/2o(C)\Z\arr\{-1,1\}\arr0$. So $H/2H\simeq\Z/2\Z$,
$o(C)$ is even and the sequence \ref{eq:exact sequence for NC} splits.
We can conclude that $M=\langle(\underline{C},\underline{C})\rangle\times\langle m\rangle$,
where $m=I(\underline{D},\underline{D})$ for some root of unity $D$
and $o(m)=2$. Normalizing $s$ we can write $m=I(\id,\underline{D})=I$.\end{proof}
\begin{prop}
\label{pro:NC complete description of invariants, algebras multiplication}Let
$M<\Aut_{k}A$ be a finite abelian group such that $(\car k,|M|)=1$
and that there exists $I(\xi,\eta)\in M$. Also assume that $k$ contains
the $|M|$-roots of unity. Then $A^{M}\simeq k[[z]]$, $A\in\MCov(A^{M})$
and only the following possibilities happen: there exists a row of
table \ref{tab:table for NC} such that $M\simeq H$ is generated
by $m,n$, $H\simeq M_{r,\alpha,N}$, $A\simeq B$ as $M$-covers,
where $\deg U=m\comma\deg V=n$ and $A$ over $A^{M}$ is given by
multiplication $z^{\E}$. Moreover all the rays of the form $\Delta^{*}$
in the table satisfy $h_{\Delta^{*}}=2$.

\begin{table}
\caption{\label{tab:table for NC}}
\centering

\begin{tabular}{|c|c|c|c|}
\hline 
$H$ & $m,n,r,\alpha,N,\overline{q}$ & $B$ & $\E$\tabularnewline
\hline 
\hline 
$\Z/2\Z$ & $1,1,1,1,2,1$ & $\frac{k[[z]][U]}{(U^{2}-z^{2})}$ & $2\E^{id}$\tabularnewline
\hline 
$(\Z/2\Z)^{2}$ & $(1,0),(0,1),2,0,2,1$ & $\frac{k[[z]][U,V]}{(U^{2}-z,V^{2}-z)}$ & $\E^{\pr_{1}}+\E^{\pr_{2}}$\tabularnewline
\hline 
$\begin{array}{c}
\Z/2l\Z\times\Z/2\Z\\
l>1
\end{array}$ & $(1,0),(1,1),2,2,2l,1$ & $\frac{k[[z]][U,V]}{(U^{2}-V^{2},V^{2l}-z)}$ & $\Delta^{2,2,2l,1}$\tabularnewline
\hline 
$\Z/4l\Z$ & $1,2l+1,1,2l+1,4l,2$ & $\frac{k[[z]][U,V]}{(U^{2}-V^{2},V^{2l+1}-zU,UV^{2l-1}-z)}$ & $\Delta^{1,2l+1,4l,2}$\tabularnewline
\hline 
$\begin{array}{c}
\Z/2l\Z\\
l>1\text{ odd}
\end{array}$ & $1,l+1,2,2,l,1$ & $\frac{k[[z]][U,V]}{(U^{2}-V^{2},V^{l}-z)}$ & $\Delta^{2,2,l,1}$\tabularnewline
\hline 
\end{tabular}

\end{table}
\end{prop}
\begin{proof}
We can reduce the problem to the actions obtained in \ref{lem:classification of actions for NC}.
We first consider the cyclic case, i.e. $M=\langle I(\id,\underline{B})\rangle\simeq\Z/2l\Z$
where $l=o(B)$. There exists $E$ such that $E^{2}=B$. Given $0\leq r<|M|=2l$,
we want to compute $A_{r}=\{a\in A\st I(\id,\underline{B})a=E^{r}a\}$.
The condition $a=c+f(s)+g(t)\in A_{r}$ holds if and only if $a=0$
when $r>0$, $f(t)=E^{r}g(t)$ and $g(Bs)=E^{r}f(s)$. Moreover $f(t)=E^{-r}g(Bt)=E^{-2r}f(Bt)\then f(Bt)=B^{r}f(t)$.
If we denote by $\delta_{r}$ the only integer such that $0\leq\delta_{r}<l$
and $\delta_{r}\equiv r\text{ mod }l$, we have that, up to constants,
$A^{r}$ is given by elements of the form $E^{r}f(s)+f(t)$ for $f\in X^{\delta_{r}}k[[X^{l}]]$.
Call $\beta=s^{l}+t^{l}\in A_{0}=A^{M}$ and $v_{r}=E^{r}s^{\delta_{r}}+t^{\delta_{r}}$,
$v_{0}=1$. We claim that $A^{M}=A_{0}=k[[\beta]]$ and $v_{r}$ freely
generates $A_{r}$ as an $A_{0}$ module. The first equality holds
since $A_{0}$ is a domain and we have relations
\[
\sum_{n\geq1}a_{n}s^{nl}+\sum_{n\geq1}a_{n}t^{nl}=\sum_{n\geq1}a_{n}(s^{l}+t^{l})^{n}=\sum_{n\geq1}a_{n}\beta^{n}
\]
while the second claim come from the relation
\[
E^{r}s^{\delta_{r}}(c+h(s))+t^{\delta_{r}}(c+h(t))=(E^{r}s^{\delta_{r}}+t^{\delta_{r}})(c+h(s)+h(t))\text{ for }h\in X^{l}k[[X^{l}]]
\]
 and the fact that $v_{r}$ is not a zero divisor in $A$.

So $A\in\MCov(k[[\beta]])$ and it is generated by $v_{1}=Es+t$ and
$v_{l+1}=-Es+t$ and so in degrees $1$ and $l+1$. If $l=1$, so
that $M\simeq\Z/2\Z$, $B=1$, $E=-1$ and $v_{1}^{2}=\beta^{2}$.
This means that $A\simeq k[[\beta]][U]/(U^{2}-\beta^{2})$ and its
multiplication over $k[[\beta]]$ is given by $\beta^{2\E^{\id}}$.
This is the first row. Assume $l>1$ and set $m=1\comma n=l+1$. Note
that $0\neq m\neq n$ and that $M\simeq M_{r,\alpha,N}$ for some
$r\comma\alpha\comma N$ that we are going to compute.

$l$ odd. We have $r=\alpha=2$ and $N=l$ since $\langle l+1\rangle=\langle2\rangle\subseteq\Z/2l\Z$.
Consider $\overline{q}=1\in\Omega_{N,N-\alpha}$ and the associated
numbers are $z=r=2\comma y=\alpha=2\comma\hat{q}=0\comma d_{\hat{q}}=x=N=l\comma w=0$.
Since $v_{1}^{z}=v_{l+1}^{y}$ and $v_{l+1}^{l}=\beta$, we will have
$A\simeq_{k[[\beta]]}A_{\lambda,\mu}^{1}$ where $\lambda,\mu=1,\beta\in k[[\beta]]$
(see \ref{def:universal algebra}) and therefore the multiplication
is $\beta^{\Delta^{2,2,l,1}}$ by \ref{pro:general algebra for m,n}.
This is the fifth row.

$l$ even. We have $r=1\comma\alpha=l+1\comma N=2l$ since $\langle l+1\rangle=\Z/2l\Z$.
Since $d_{1}=l-1\equiv-\alpha$ and $d_{2}=2l-2\equiv2(-\alpha)$
modulo $2l$ we can consider $\overline{q}=2\in\Omega_{N-\alpha,N}$.
The associated numbers are $z=y=2\comma\hat{q}=1\comma d_{\hat{q}}=l-1\comma x=N-(d_{\overline{q}}-d_{\hat{q}})=l+1\comma w=1\equiv xn=(l+1)^{2}\text{ mod }2l$.
Since $v_{1}^{z}=v_{l+1}^{y}$, $v_{l+1}^{x}=\beta v_{1}$ and $v_{1}^{\hat{q}r}v_{l+1}^{d_{\hat{q}}}=\beta$,
we will have $A\simeq_{k[[\beta]]}A_{\lambda,\mu}^{2}$ where $\lambda,\mu=1,\beta\in k[[\beta]]$
whose multiplication is $\beta^{\Delta_{1,l+1,2l,2}}$. This is the
fourth row.

Now consider the case $M=\langle(\underline{C},\underline{C})\rangle\times\langle I\rangle$
with $o(C)=l$ even. Set $\beta=s^{l}+t^{l}$, $v_{1,0}=s+t$ and
$v_{1,1}=-s+t$. Note that $v_{r,i}$ is homogeneous of degree $(r,i)$.
Set $m=(1,0)\comma n=(1,1)$. They are generators of $M$ and so $M\simeq M_{r,\alpha,N}$
for some $r,\alpha,N$. We have $N=o(n)=l$, $r>1$ since $\langle n\rangle\neq M$
and so $r=2$ since $2m=2n$. If $l=2$ we get $\alpha=0$ and if
$l>2$ we get $\alpha=2$. Choose $\overline{q}=1$ so that the associated
numbers are $z=2\comma y=\alpha\comma\hat{q}=0\comma d_{\hat{q}}=x=N=l\comma w=0$.
As done above, it is easy to see that $A^{M}=k[[\beta]]$. We first
consider the case $l=2$. Since $v_{1,0}^{2}=\beta\comma v_{1,1}^{2}=\beta$
we get a surjection $A_{\beta,\beta}^{1}\arr A$ which is an isomorphism
by dimesion. From the expression of $A_{\beta,\beta}^{1}$ we can
deduce directly that the multiplication is $\beta^{\E^{\pr_{1}}+\E^{\pr_{2}}}$,
where $\pr_{i}\colon(\Z/2\Z)^{2}\arr\Z/2\Z$ are the two projections.
This is the second row.

Now assume $l>2$. Since $v_{1,0}^{2}=v_{1,1}^{2}$ and $v_{1,1}^{l}=\beta$
and arguing as above we get $A\simeq_{k[[\beta]]}A_{\lambda,\mu}^{1}$
where $\lambda,\mu=1,\beta\in k[[\beta]]$ and the multiplication
$\beta^{\Delta^{2,2,l,1}}$. This is the third row.

Finally the last sentence is clear by definition of $\Sigma_{M}$
and \ref{pro:classification sm int ray htwo}.\end{proof}
\begin{rem}
\label{rem:Y/X NC then X has a k}If $X$ is a locally noetherian
integral scheme and there exists a $\Di M$-cover $Y/X$ such that
$Y$ is normal crossing in codimension $1$, then $X$ is defined
over a field. Indeed if $\car\odi X(X)=p$ then $\F_{p}\subseteq\odi X(X)$.
Otherwise $\Z\subseteq\odi X(X)$ and we have to prove that any prime
number $q\in\Z$ is invertible. We can assume $X=\Spec R$, where
$R$ is a local noetherian domain. If $\dim R=0$ then $R$ is a field,
otherwise, since $\alt(q)\leq1$, we can assume $\dim R=1$ and $R$
complete. By definition of normal crossing in codimension $1$, if
$Y=\Spec S$ and $p\in Y$ is over $m_{R}$ we have a flat and local
map $R\arr S\arr S_{p}\arr B$, such that $B$ contains a field $k$.
The prime $q$ is a non zero divisor in $R$ and therefore in $B$.
In particular $0\neq q\in k^{*}\subseteq B^{*}$ and $q\in R^{*}$.\end{rem}
\begin{thm}
\label{thm:NC in codimension one}Let $M$ be a finite abelian group,
$X$ be a locally noetherian and locally factorial scheme with no
isolated points and $(\car X,|M|)=1$. Consider the full subcategory
\[
NC_{X}^{1}=\{Y/X\in\MCov(X)\st Y\text{ is normal crossing in codimension }1\}\subseteq\MCov(X)
\]
Then $NC_{X}^{1}\neq\emptyset$ if and only if each connected component
of $X$ is defined over a field. In this case define
\[
\underline{\E}=\left(\begin{array}{c}
\E^{\phi}\text{ for }\phi\colon M\arr\Z/l\Z\text{ surjective with }l\geq1,\\
\Delta^{2,2,l,1,\phi}\text{ for }\phi\colon M\arr M_{2,2,l}\text{ surjective with }l\geq3,\\
\Delta^{1,2l+1,4l,2,\phi}\text{ for }\phi\colon M\arr M_{1,2l+1,4l}\text{ surjective with }l\geq1
\end{array}\right)
\]
and $\catC_{NC,X}^{1}$ as the full subcategory of $\stF_{\underline{\E}}(X)$
of objects $(\underline{\shL},\underline{\shM},\underline{z},\lambda)$
such that:
\begin{enumerate}
\item for all $\E\neq\delta\in\underline{\E}$, $\codim V(z_{\E})\cap V(z_{\delta})\geq2$
except for the case where $\E=\E^{\phi}\comma\delta=\E^{\psi}$   $$ \begin{tikzpicture}[xscale=2.0,yscale=-0.5]     \node (A0_2) at (2, 0) {$\Z/2\Z$};     \node (A1_0) at (0, 1) {$M$};     \node (A1_1) at (1, 1) {$(\Z/2\Z)^2$};     \node (A2_2) at (2, 2) {$\Z/2\Z$};     \path (A1_0) edge [->,bend left=65]node [auto,swap] {$\scriptstyle{\psi}$} (A2_2);     \path (A1_0) edge [->>]node [auto] {$\scriptstyle{}$} (A1_1);     \path (A1_1) edge [->]node [auto,swap] {$\scriptstyle{\pr_2}$} (A2_2);     \path (A1_1) edge [->]node [auto] {$\scriptstyle{\pr_1}$} (A0_2);     \path (A1_0) edge [->,bend right=65]node [auto] {$\scriptstyle{\phi}$} (A0_2);   \end{tikzpicture}  $$
in which $v_{p}(z_{\E^{\phi}})=v_{p}(z_{\E^{\psi}})=1$ if $p\in Y^{(1)}\cap V(z_{\E^{\phi}})\cap V(z_{\E^{\psi}})$;
\item for all $\E\in\underline{\E}$ and $p\in Y^{(1)}$ $v_{p}(z_{\E})\leq2$
and $v_{p}(z_{\E})=2$ if and only if $\E=\E^{\phi}$ where $\phi\colon M\arr\Z/2\Z$
is surjective.
\end{enumerate}
Then we have an equivalence of categories
\[
\catC_{NC,X}^{1}=\pi_{\underline{\E}}^{-1}(NC_{X}^{1})\arrdi{\simeq}NC_{X}^{1}
\]
\end{thm}
\begin{proof}
The first claim comes from \ref{rem:Y/X NC then X has a k}. We will
make use of \ref{thm:fundamental thm locally factoria hleqtwo}. If
$Y/X\in NC_{Y}^{1}$ and $p\in Y^{(1)}$ we have $h_{Y/X}(p)\leq\dim_{k(p)}m_{p}/m_{p}^{2}\leq2$
since etale maps preserve tangent spaces and $\dim m_{A}/m_{A}^{2}\leq2$.
So $NC_{X}^{1}\subseteq\catD_{X}^{2}$.

Let $\underline{\delta}$ be the sequence of smooth extremal rays
used in \ref{thm:fundamental thm locally factoria hleqtwo}. We know
that $\pi_{\underline{\delta}}^{-1}(NC_{X}^{1})\subseteq\catC_{X}^{2}$.
So we have only to prove that $\pi_{\underline{\delta}}^{-1}(NC_{X}^{1})\subseteq\stF_{\underline{\E}}(X)\subseteq\stF_{\underline{\delta}}(X)$
and that any element $Y\in NC_{X}^{1}$ locally, in codimension $1$,
satisfies the requests of the theorem. Since $X$ is a disjoint union
of positive dimensional, integral connected components, we can assume
that $X=\Spec R$, where $R$ is a complete discrete valuation ring.
Since $R$ contains a field, then $R\simeq k[[x]]$ . Let $\chi\in\pi_{\underline{\E}}^{-1}(\catD_{X}^{2})$
and $D$ the associated $M$-cover over $R$. Let $C$ be the maximal
torsor of $D/R$ and $H=H_{D/R}$. Note that, for any maximal ideal
$q$ of $C$ we have $C_{q}\simeq k(q)[[x]]$ since $C/R$ is etale.
Moreover $\Spec D\in NC_{X}^{1}$ for $M$ if and only if for any
maximal prime $p$ of $D$ $\Spec D_{p}\in NC_{\Spec C_{q}}^{1}$
for $M/H$, where $q=C\cap p$. In the same way $\chi\in\catC_{NC,X}^{1}$
for $M$ if and only if, for any maximal prime $q$ of C, $\chi_{|\Spec C_{q}}\in\catC_{NC,\Spec C_{q}}^{1}$
for $M/H$. We can therefore reduce the problem to the case $H_{D/R}=0$.
We can also assume that $k$ contains the $|M|$-roots of unity.

First assume that $\Spec D\in NC_{Y}^{1}$. If $D$ is regular, the
conclusion comes from \ref{thm:regular in codimension 1 covers}.
So assume $D$ not regular and denote by $\mu\colon R=k[[x]]\arr D$
the associated map. We know that $D/m_{A}=k$. By Cohen's structure
theorem we can write $D=k[[y]]/I$ in such a way that $\mu_{|k}=\id_{k}$.
By definition, since $D$ is local and complete, there exists an etale
extension $D\arr B=L[[s,t]]/(st)$. Using the properties of complete
rings, $B/D$ is finite and so $B\simeq D\otimes_{k}L$. Replacing
the base $R$ by $R\otimes_{k}L$ we can assume that $D\simeq k[[s,t]]/(st)$.
The function $\mu_{|k}\colon k\arr D$ extends to a map $\nu\colon D\arr D$
sending $s,t$ to itselves. This map is clearly surjective. Since
$\Spec D$ contains $3$ points, $\nu$ induces a closed immersion
$\Spec D\arr\Spec D$ which is a bijection. Since $D$ is reduced
$\nu$ is an isomorphism. This shows that we can write $D=A=k[[s,t]]/(st)$
in such a way that $\mu_{|k}=\id_{k}$. So $\Di M\simeq\underline{M}$
acts as a subgroup of $\Aut_{k}A$ such that $A^{M}\simeq k[[z]]$.
In particular, by \ref{lem:NC excluding stupid subgroups}, there
exists $I(\xi,\eta)\in M$. Up to equivariant isomorphisms the possibilities
allowed are described in \ref{pro:NC complete description of invariants, algebras multiplication}
and coincides with the ones of the statement. So $\chi\in\catC_{NC,X}^{1}$.

Now assume that $\chi\in\catC_{NC,X}^{1}$. By definition of $\pi_{\underline{\E}}$
the multiplication that defines $D$ over $R$ is something of the
form $\psi=\lambda z^{\E}$, where $\lambda$ is an $M$-torsor and
$\E$ is one of the ray of table \ref{tab:table for NC}. The case
$\E=\E^{\phi}$ comes from \ref{thm:regular in codimension 1 covers}.
Since, in our hypothesis, an $M$-torsor (in the fppf meaning) is
also an etale torsor, replacing the base $R$ by an etale neighborhood
(that maintains the form $k[[x]]$), we can assume $\lambda=1$. In
this case, thanks to \ref{lem:classification of actions for NC} and
\ref{pro:NC complete description of invariants, algebras multiplication},
we can conclude that $A\simeq k[[s,t]]/(st)$ as required.\end{proof}
\begin{cor}
Let $X$ be a locally noetherian and regular in codimension $1$ (normal)
scheme with no isolated points, $M$ be a finite abelian group with
$(\car X,|M|)=1$ and $|M|$ odd. If $Y/X$ is a $\Di M$-cover and
$Y$ is normal crossing in codimension $1$ then $Y$ is regular in
codimension $1$ (normal).\end{cor}
\begin{proof}
Since $Y/X$ has Cohen-Macaulay fibers it is enough to prove that
$Y$ is regular in codimension $1$ by Serre's criterion. So we can
assume $X=\Spec R$, where $R$ is a discrete valuation ring, and
apply \ref{thm:regular in codimension 1 covers} just observing that
$\widetilde{Reg}_{X}^{1}=\catC_{NC,X}^{1}$.\end{proof}
\begin{rem}
We keep notation from \ref{thm:NC in codimension one} and set $\underline{\delta}=(\E^{\eta},\eta\colon M\arr\Z/d\Z\text{ surjective },d>1)$.
We have that $\pi_{\underline{\delta}}^{-1}(NC_{X}^{1})=\catC_{NC,X}^{1}\cap\stF_{\underline{\delta}}$,
i.e. the covers $Y/X\in NC_{X}^{1}$ writable only with the rays in
$\underline{\delta}$, has the same expression of $\catC_{NC,X}^{1}$
but with object in $\stF_{\underline{\delta}}$. Therefore the multiplications
that yield a not smooth but with normal crossing in codimension $1$
covers are only $\E^{\phi}+\E^{\psi}$, where $\phi\comma\psi$ are
morphism as in $1)$, and $\E^{2\phi}$, where $\phi\colon M\arr\Z/2\Z$
is surjective. This result can also be found in \cite[Theorem 1.9]{Alexeev2011}.
In particular, if $M=(\Z/2\Z)^{r}$, where $\underline{\delta}=\underline{\E}$
thanks to \ref{pro:when sigmaM is empty}, these are the only possibilities.
\end{rem}

\chapter{\label{chap:Functorial-1}Equivariant affine maps and monoidality.}

The aim of this chapter is the study of $G$-covers for general groups,
but with particular attention to the linearly reductive case. We now
briefly summarize how this chapter is divided.

\emph{Section 1. }We will introduce the definition of linearly reductive
groups and study their representation theory. Looking for an analogous
behaviour to the representation theory for groups over a field, we
will introduce the notion of good linearly reductive groups (briefly
glrg). We will then focus on linearly reductive groups over strictly
Henselian rings and their action on finite algebras. The last part
will be dedicated to the study of induction of equivariant algebras
from a subgroup.

\emph{Section 2. }We will prove the equivalence between the category
of $G$-equivariant quasi-coherent sheaves of algebras over a scheme
$T$ and the category of linear, left exact, symmetric monoidal functors
$\Loc^{G}R\arr\Loc T$. The first step will be to establish a correspondence
between $G$-equivariant quasi-coherent sheaves and functors as above,
but without any monoidal structure and then describe how the properties
of commutativity, associativity and existence of a unity translate
into properties of the associated functor. We will then determine
what functors correspond to $G$-covers and $G$-torsors and, when
$G$ is a super solvable glrg, we will also describe a simpler criterion
to distinguish $G$-torsors among $G$-equivariant algebras.

\emph{Section 3. }In this section we will prove that $\GCov$ is reducible
if $G$ is a linearly reductive and non abelian group. The proof is
based on the use of what we will call rank functions, that allow us
to distinguish $\GCov$ inside $\LAlg_{R}^{G}$ and their behaviour
under induction from a subgroup.

\emph{Section 4. }This section is dedicated to the problem of regular
in codimension $1$ $G$-covers. We will describe such covers using
the trace map associated with an algebra and we will also discuss
a possible extension of the results to the non equivariant case.

In this chapter, we will often prove statements valid over any scheme
and, in order to simplify the reading, the letter $T$, if not stated
otherwise, will denote a scheme over the given base.

\section{\label{sec:Preliminaries on linearly reductive groups}Preliminaries
on linearly reductive groups.}

In this section we will study the representation theory of finite,
linearly reductive groups. In particular we will introduce the notion
of good linearly reductive groups (glrg). This class of groups has
a very special representation theory, very close to the one of usual
linearly reductive groups over an algebraically closed field.

We will then focus on groups over strictly Henselian rings, where
their structure is simpler and finally we will discuss the properties
of induction and state some useful results.

We will consider given a base scheme $S$ and a flat, finite and finitely
presented group scheme $G$ over $S$.

\subsection{Representation theory of linearly reductive groups.}

As the section name suggests, in this section we will introduce the
notion of linearly reductive groups and discuss their representation
theory. In particular we will define the notion of good linearly reductive
group (briefly glrg): these are the groups admitting a set of geometrically
irreducible representations with which is possible to describe any
equivariant quasi-coherent sheaf, analogously to what happens over
an algebraically closed field. Let $G$ be a linearly reductive group.
We will prove that if $G$ is defined over an algebraically closed
field or if it is diagonalizable then it is a glrg. Moreover we will
show that $G$ is always fppf locally a glrg and, if $G$ is étale,
also étale locally. In particular any étale (and therefore constant),
finite linearly reductive group defined over a strictly Henselian
ring is a glrg.

In what follows $G$ will be a flat, finite and finitely presented
group scheme over the given base. Before dealing with linearly reductive
groups, we prove the following propositions, which will be very useful.
\begin{prop}
\label{prop:the structure map are the invariants}Let $X=\Spec\alA$
be an affine $S$-scheme with a (right) action of $G$ and $\shF$
be a quasi-coherent sheaf over $S$. Then we have a $G$-equivariant
isomorphism
\[
\phi\colon\Homsh_{S}(X,\WW(\shF))\arr\WW(\shF\otimes\alA)
\]
If $X=G$, with the regular action on itself, we have vertical isomorphisms
  \[   \begin{tikzpicture}[xscale=3.7,yscale=-0.8]     
\node (A0_0) at (0, 0) {$\Homsh^G(G,\WW(\shF))$};     
\node (A0_1) at (1, 0) {$\Homsh(G,\WW(\shF))$};     
\node (A1_0) at (0, 1) {$\WW(\shF)$};     
\node (A1_1) at (1, 1) {$\WW(\shF\otimes\odi{S}[G])$};     
\node[rotate=-90] (s) at (0, 0.5) {$\simeq$};     
\node[rotate=-90] (ss) at (1, 0.5) {$\simeq$};     
\path (A0_0) edge [->]node [auto] {$\scriptstyle{}$} (A0_1);     \path (A1_0) edge [->]node [auto,swap] {$\scriptstyle{\WW(\nu)}$} (A1_1);   \end{tikzpicture}   \]  where $\nu\colon\shF\arr\shF\otimes\odi S[G]$ is the structure map.
In particular $\nu$ yields an isomorphism $\shF\simeq(\shF\otimes\odi S[G])^{G}$
of sheaves (without actions).\end{prop}
\begin{proof}
Notice that we will only use that $G$ is an affine scheme. Let $\pi\colon X\arr S$
be the structure morphism. Note that if $U$ is an $S$-scheme, then
$\WW(\shF)\times U\simeq\WW(\shF\otimes\odi U)$ as sheaves over $U$
and, since $\pi_{*}\pi^{*}\shF\simeq\shF\otimes\alA$, we have that
\[
(\WW(F)\times U)(X\times U)=\Hl^{0}(X\times U,\pi_{U}^{*}(\shF\otimes\odi U))=\Hl^{0}(U,\shF\otimes\alA\otimes\odi U)=\WW(\shF\otimes\alA)(U)
\]
where $\pi_{U}$ is the base change of $\pi$ to $U$. In particular,
by Yoneda's lemma, the natural transformation $\phi\colon\Homsh_{S}(X,\WW(\shF))\arr\WW(\shF\otimes\alA)$
given by 
\[
\phi_{U}(X\times U\arrdi{\delta}\WW(\shF)\times U)=\delta(\id_{X\times U})
\]
is an isomorphism. We have to show that $\phi$ is $G$-equivariant
and we can assume that $S=\Spec R$, for some ring $R$. Denote by
$ $$\xi\colon\shF\otimes\alA\arr\shF\otimes\alA\otimes R[G]$ the
action of $G$ on $\shF\otimes\alA$ induced by $\phi$. We want to
prove that $\xi$ coincides with the structure morphism of the tensor
product of representations $\shF\otimes\alA$. In general if $M$
is an $R$-module with an action of $G$ then the multiplication by
$\id_{G}$ on $\WW(M)(G)=M\otimes R[G]$ yields the structure map
$M\arr M\otimes R[G]\arrdi{\id_{G}\cdot-}M\otimes R[G]$ of $M$.
In particular, by definition we have 
\[
\xi(x)=\phi_{G}(\id_{G}\cdot\phi_{G}^{-1}(x\otimes1))=[\id_{G}\cdot\phi_{G}^{-1}(x\otimes1)](\id_{X\times G})\text{ for }x\in\shF\otimes\alA
\]
Given $\delta\colon X\times G\arr\WW(F)\times G$ we have
\[
(\id_{G}\cdot\delta)(\id_{X\times G})=\id_{G}\cdot(\delta(\id_{X\times G}\cdot\id_{G}))
\]
Moreover $(\id_{X\times G}\cdot\id_{G})\colon X\times G\arr X\times G$
is given by the (right) action of $G$ on $X$, i.e. it is the $\Spec$
of the structure map $\mu\colon\alA\otimes\odi S[G]\arr\alA\otimes\odi S[G]$.
On the other hand, given $z\in\shF\otimes R[G]\otimes\alA=\Hom(X\times G,\WW(\shF)\times G)$
then $\id_{G}\cdot z=(\overline{\nu}\otimes\id_{\alA})(z)$, where
$\overline{v}\colon\shF\otimes R[G]\arr\shF\otimes R[G]$ is the structure
map, i.e. the $R[G]$-linear map such that $\overline{\nu}(x\otimes1)=\nu(x)$.
Finally, by definition of the Yoneda's isomorphism, we have 
\[
\phi_{G}^{-1}(x\otimes1)(U\arrdi{\alpha}X\times G)=[\WW(\shF)(\alpha)](x\otimes1)\in(\WW(\shF)\times G)(U)=\WW(\shF)(U)
\]
In conclusion $\phi_{G}^{-1}(x\otimes1)(\Spec\mu)=(\id_{\shF}\otimes\mu)(x\otimes1)$.
Putting everything together we get that $\xi$ is the composition
\[
\shF\otimes\alA\otimes\odi S[G]\arrdi{\id_{\shF}\otimes\mu}\shF\otimes\alA\otimes\odi S[G]\simeq\shF\otimes\odi S[G]\otimes\alA\arrdi{\overline{\nu}\otimes\id_{\alA}}\shF\otimes\odi S[G]\otimes\alA\simeq\shF\otimes\alA\otimes\odi S[G]
\]
which induces the classical co-module structure on the tensor product
$\shF\otimes\alA$.

Now let $\alA=\odi S[G]$. The map   \[   \begin{tikzpicture}[xscale=2.8,yscale=-0.6]     \node (A0_0) at (0, 0) {$\Homsh^G(G,\WW(\shF))$};     \node (A0_1) at (1, 0) {$\WW(\shF)$};     \node (A1_0) at (0, 1) {$\psi$};     \node (A1_1) at (1, 1) {$\psi(1)$};     \path (A0_0) edge [->]node [auto] {$\scriptstyle{\eta}$} (A0_1);     \path (A1_0) edge [|->,gray]node [auto] {$\scriptstyle{}$} (A1_1);   \end{tikzpicture}   \] 
is an isomorphism. The composition
\[
\WW(F)\arrdi{\eta^{-1}}\Homsh^{G}(G,\WW(\shF))\subseteq\Homsh(G,\WW(\shF))\arrdi{\phi}\WW(\shF\otimes\odi S[G])
\]
yields a map $\omega\colon\shF\arr\shF\otimes\odi S[G]$ and we have
to prove that $\omega=\nu$. Again we can assume that $S=\Spec R$,
for a ring $R$. Given $x\in\WW(\shF)(R)=\shF$ we have
\[
\phi_{R}(\eta_{R}^{-1}(x))=[\eta_{R}^{-1}(x)](\id_{G})=\id_{G}\cdot(x\otimes1)=\overline{\nu}(x\otimes1)=\nu(x)
\]
as required.\end{proof}
\begin{lem}
\label{lem:presentation of G equivariant co modules}Let $R$ be a
ring and $M\in\QCoh^{G}R$. Then there exists a $G$-equivariant presentation
\[
(\duale{R[G]})^{\oplus J}\arr(\duale{R[G]})^{\oplus I}\arr M\arr0
\]
If $M$ is finitely presented, then we can choose $I$ and $J$ finite.\end{lem}
\begin{proof}
From \ref{prop:the structure map are the invariants} we have an isomorphism
\[
M\arrdi{\phi}\Hom^{G}(\duale{R[G]},M)\simeq(R[G]\otimes M)^{G}
\]
and it is easy to check that $\phi_{m}(\varepsilon_{G})=m$, where
$m\in M$. Then
\[
\bigoplus_{m\in M}\phi_{m}\colon(\duale{R[G]})^{\oplus M}\arr M
\]
is $G$-equivariant and surjective. If $M$ is finitely presented,
obviously we can assume $I$ finite. Let
\[
K=\Ker((\duale{R[G]})^{I}\arr M)
\]
Since $R[G]$ is locally free, $K$ is locally finitely presented
and therefore finitely presented.\end{proof}
\begin{defn}
The group scheme $G$ is linearly reductive over $S$ if the functor
of invariants
\[
(-)^{G}\colon\QCoh^{G}S\arr\QCoh S
\]
is exact.
\end{defn}
From now on we will assume that $G$ is linearly reductive. Remember
that this condition is stable under base change and is local in the
fppf topology ( see \cite[Proposition 2.6]{Abramovich2007}).

The following lemmas are crucial in the study of the representation
theory of linearly reductive groups over general schemes and they
explain how invariants behave for such groups.
\begin{lem}
\label{lem:constant coherent sheaves go out invariants}Let $\shF\in\QCoh^{G}S$
and $\shH\in\QCoh S$. Then the natural map
\[
\shF^{G}\otimes\shH\arr(\shF\otimes\underline{\shH})^{G}
\]
is an isomorphism.\end{lem}
\begin{proof}
We can assume that $S=\Spec A$, for some ring $A$, and that $\shF=\widetilde{M}$,
$\shH=\widetilde{N}$, for some $A$-modules $M$, $N$. The natural
map $M^{G}\otimes_{A}N\arr(M\otimes_{A}N)^{G}$ is an isomorphism,
bacause it is so when $N$ is free and in general taking a presentation
of $N$, taking into account that $(-)^{G}$ is an exact functor.\end{proof}
\begin{lem}
\label{lem:base change of invariants}Let $\shF\in\FCoh^{G}S$. Then
the map $\shF^{G}\arr\shF$ splits locally, $\shF^{G}\in\FCoh S$
and the natural map
\[
\WW(\shF^{G})\arr\WW(\shF)^{G}
\]
is an isomorphism. In particular if $\shF$ is locally free so is
$\shF^{G}$.\end{lem}
\begin{proof}
The second map in the statement is an isomorphism thanks to \ref{lem:constant coherent sheaves go out invariants}.
For the local splitting and the finite presentation of $\shF^{G}$,
we can assume $S=\Spec A$, $\shF=\widetilde{M}$ where $A$ is a
ring and $M$ a finitely presented $A$-module. Moreover, because
$(-)^{G}$ is invariant by any base change, we can also assume that
$A$ is noetherian. In this case the splitting follows from \cite[Theorem 7.14]{Matsumura1989}.
\end{proof}
We now recover the property that usually is used as definition of
linearly reductive groups over a field:
\begin{lem}
Assume $S=\Spec k$, where $k$ is a field. Then any finite dimensional
representation of $G$ is a direct sum of irreducible representations.\end{lem}
\begin{proof}
Any $G$-equivariant injection $V\arr W$ of representations has a
$G$-equivariant section since the maps
\[
\Hom(W,V)\arr\Hom(V,V)\comma\Hom^{G}(W,V)\arr\Hom^{G}(V,V)
\]
are surjective.\end{proof}
\begin{defn}
Given a finite, linearly reductive group $G$ over a field we denote
by $I_{G}$ a set of representatives of the irreducible representations
of $G$. We will often say that $I_{G}$ is the {}``set'' of irreducible
representations of $G$ or refer to such a {}``set'', always meaning
that we are choosing a set of representatives.
\end{defn}
We want now to find sheaves over $S$ that play the role of the irreducible
representations over a field.
\begin{defn}
\label{def: evaluation for coherent equivariant sheaves}Given $\shF\in\QCoh_{R}^{G}$
and $V\in\Loc^{G}S$ we define
\[
\theta_{V}^{\shF}\colon\duale V\otimes\Homsh^{G}(\duale V,\shF)\arr\shF\comma\theta_{V}^{\shF}(x\otimes\psi)=\psi(x)
\]
Note that this map is $G$-equivariant.\end{defn}
\begin{rem}
If $V\in\Loc^{G}S$ we have $G$-equivariant morphisms and commutative
diagrams   \[   \begin{tikzpicture}[xscale=4.1,yscale=-0.5]     
\node[rotate=-90] (s) at (0, 1.9) {$\simeq$};  
\node[rotate=-90] (ss) at (1.07, 1.9) {$\simeq$}; 
\node (p) at (-0.5, 4) {$\psi\otimes x\otimes v$};
\node (pp) at (-0.5, 0) {$(u\arr\psi(u)x)\otimes v$};
\node (A0_1) at (1, 0) {$\phi\otimes v$};     
\node (A1_0) at (0, 1) {$\Homsh^G(V,\shF)\otimes V$};     
\node (A1_1) at (1, 1) {$\Homsh(V,\shF)\otimes V$};     
\node (A1_2) at (2, 1) {$\phi(v)$};     
\node (A2_2) at (2, 2) {$\shF$};     
\node (A3_0) at (0, 3) {$\ \ \ (\duale{V}\otimes \shF)^G\otimes V$};     
\node (A3_1) at (1.07, 3) {$\duale{V}\otimes \shF\otimes V$};     
\node (A4_1) at (1.07, 4) {$\psi\otimes x\otimes v$};     
\node (A3_2) at (2, 3) {$\psi(v)x$};     
\path (p) edge [->]node [auto] {$\scriptstyle{}$} (pp); 
\path (A0_1) edge [->,bend right=25]node [auto] {$\scriptstyle{}$} (A1_2);     \path (A3_1) edge [->,bend left=15]node [auto] {$\scriptstyle{}$} (A2_2);     \path (A1_0) edge [right hook->]node [auto] {$\scriptstyle{}$} (A1_1);     \path (A4_1) edge [->,bend left=25]node [auto] {$\scriptstyle{}$} (A3_2);     \path (A1_1) edge [->,bend right=15]node [auto] {$\scriptstyle{}$} (A2_2);     \path (A3_0) edge [right hook->]node [auto] {$\scriptstyle{}$} (A3_1);   \end{tikzpicture}   \] Given a collection $I\subseteq\Loc^{G}S$ we have a natural, $G$-equivariant
morphism
\[
\eta_{I,\shF}=\bigoplus_{V\in I}\theta_{\duale V}^{\shF}\colon\bigoplus_{V\in I}\Homsh^{G}(V,\shF)\otimes V\arr\shF\qquad\forall\:\shF\in\QCoh^{G}S
\]
\end{rem}
\begin{prop}
\label{prop:generating irreducible representations}Let $I$ be a
collection of elements of $\Loc^{G}S$. The following are equivalent:
\begin{enumerate}
\item \label{enu:generating irreducible representations natural maps 1}the
natural maps
\[
\eta_{I,\shF}\colon\bigoplus_{V\in I}\Homsh^{G}(V,\shF)\otimes V\arr\shF\qquad\forall\:\shF\in\FCoh^{G}S
\]
are isomorphisms;
\item \label{enu:generating irreducible representations natural maps quasi-coherent}same
as \ref{enu:generating irreducible representations natural maps 1},
but for any $\shF\in\QCoh^{G}T$ and any $S$-scheme $T$.
\item \label{enu:generating irreducible representations for any 2}for any
algebraically closed field $k$ and geometric point $\Spec k\arr S$
the map
\[
I\arrdi{-\otimes k}I_{G_{k}}
\]
is well defined and bijective;
\item \label{enu:generating irreducible representations is S 3}If $S$
is connected, same as \ref{enu:generating irreducible representations for any 2}
but for just one geometric point.
\end{enumerate}
If $S=\Spec k$, then $I_{G}$ satisfies the above conditions if and
only if $\End^{G}(V)\simeq k$ for any $V\in I_{G}$.\end{prop}
\begin{proof}
Assume first that $S=\Spec k$, where $k$ is a field. Given $V,W\in I_{G}$,
we have that any equivariant map $V\arr W$ is either $0$ or an isomorphism.
So $\eta_{I_{G},V}$ is an isomorphism if and only if $\End^{G}(V)\simeq k$.
Conversely, if this holds for any irreducible representation of $G$,
then $\eta_{I_{G},*}$ is an isomorphism for all finite representations
since these are direct sums of irreducible representations.

Now consider the general case and given a geometric point $\Spec k\arr S$
denote $I_{k}=\{V\otimes k\st V\in I\}$. Since $G$ is linearly reductive
and thanks to \ref{lem:base change of invariants}, we can conclude
that $\eta_{I,\shF}\otimes k=\eta_{I_{k},\shF\otimes k}$ for any
$\shF\in\FCoh^{G}S$ and that $\Homsh^{G}(V,\shF)$ is locally free
if $\shF$ is so.

\ref{enu:generating irreducible representations natural maps quasi-coherent}$\then$\ref{enu:generating irreducible representations natural maps 1}
and \ref{enu:generating irreducible representations for any 2}$\then$\ref{enu:generating irreducible representations is S 3}.
Obvious.

\ref{enu:generating irreducible representations natural maps 1}$\then$\ref{enu:generating irreducible representations for any 2}.
If $V\in I$, since $\eta_{I,V}$ is an isomorphism and $\Homsh^{G}(V,V)\otimes V\arr V$
is surjective, we have that $\Homsh^{G}(V,W)=0$ if $V\neq W\in I$
and that $\Homsh^{G}(V,V)$ is an invertible sheaf. In particular
$\End^{G_{k}}(V\otimes k)\simeq k$ and $V\otimes k$ is therefore
irreducible for any geometric point $\Spec k\arr S$, so that $I_{k}\subseteq I_{G_{k}}$.
For the converse let $W\in I_{G_{k}}$. Since $\duale W\simeq(\duale W\otimes k[G])^{G}\simeq\Hom^{G}(W,k[G])$
by \ref{prop:the structure map are the invariants}, there exists
a $G$-equivariant nonzero map $W\arr k[G]\simeq\odi S[G]\otimes k$,
which is injective because $W$ is irreducible. On the other hand
the only irreducible representations of $G_{k}$ appearing in $k[G]$
are the ones in $I_{k}$ since $\eta_{I,\odi S[G]}\otimes k=\eta_{I_{k},k[G]}$.

\ref{enu:generating irreducible representations for any 2}$\then$\ref{enu:generating irreducible representations natural maps quasi-coherent}
We can assume that $T=\Spec R$, where $R$ is a ring. If $M\in\QCoh^{G}R$
is locally free of finite rank then $\eta_{I,M}\otimes k=\eta_{I_{G_{k}},M\otimes k}$
is an isomorphism for any geometric point $\Spec k\arr T$. Since
both source and target of the map $\eta_{I,M}$ are locally free thanks
to \ref{lem:base change of invariants}, we can conclude that it is
an isomorphism. In particular $\eta_{I,\duale{R[G]}}$ is an isomorphism.
If $M\in\QCoh^{G}R$, thanks to \ref{lem:presentation of G equivariant co modules}
we have a $G$-equivariant presentation $V_{1}\arr V_{0}\arr M\arr0$,
where the $V_{i}$ are a direct sum of copies of $\duale{R[G]}$.
Since $\Hom^{G}(V,-)\otimes V$ is exact when $V$ is locally free,
we have a commutative diagram  \[   \begin{tikzpicture}[xscale=4.5,yscale=-1.1]     
\node (A0_0) at (0, 0) {$\displaystyle{\bigoplus_{V\in I}\Homsh^{G}(V, V_1)\otimes V}$};     
\node (A0_1) at (1, 0) {$\displaystyle{\bigoplus_{V\in I}\Homsh^{G}(V, V_0)\otimes V}$};     
\node (A0_2) at (2, 0) {$\displaystyle{\bigoplus_{V\in I}\Homsh^{G}(V, M)\otimes V}$};     
\node (A0_3) at (2.7, 0) {$0$};     
\node (A1_3) at (2.7, 1) {$0$}; 
\node (A1_0) at (0, 1) {$V_1$};     
\node (A1_1) at (1, 1) {$V_0$};     
\node (A1_2) at (2, 1) {$M$};     
    
\path (A0_1) edge [->]node [auto] {$\scriptstyle{\eta_{I,V_0}}$} (A1_1);     \path (A0_0) edge [->]node [auto] {$\scriptstyle{}$} (A0_1);     \path (A0_1) edge [->]node [auto] {$\scriptstyle{}$} (A0_2);     \path (A1_0) edge [->]node [auto] {$\scriptstyle{}$} (A1_1);     \path (A1_1) edge [->]node [auto] {$\scriptstyle{}$} (A1_2);     \path (A0_2) edge [->]node [auto] {$\scriptstyle{\eta_{I,M}}$} (A1_2);     \path (A1_2) edge [->]node [auto] {$\scriptstyle{}$} (A1_3);     \path (A0_2) edge [->]node [auto] {$\scriptstyle{}$} (A0_3);     \path (A0_0) edge [->]node [auto] {$\scriptstyle{\eta_{I,V_1}}$} (A1_0);   \end{tikzpicture}   \] The $\eta_{I,V_{i}}$ are isomorphisms by additivity and therefore
we can conclude that $\eta_{I,M}$ is an isomorphism as well.

\ref{enu:generating irreducible representations is S 3}$\then$\ref{enu:generating irreducible representations for any 2}
Let $\Spec k_{0}\arr S$ be the given geometric point. For $V,W\in I$
we have that $\Homsh^{G}(V,W)$ are locally free and checking the
rank on $k_{0}$, we can conclude that $\Homsh^{G}(V,W)=0$ if $V\neq W$
and that $\Homsh^{G}(V,V)$ is invertible. In particular $I_{k}\subseteq I_{G_{k}}$
for any geometric point and therefore $\eta_{I_{k},k[G]}$ is injective
since $\eta_{I_{G_{k}},k[G]}$ is so. But 
\begin{alignat*}{1}
\dim_{k}\bigoplus_{V\in I}\Hom^{G}(V\otimes k,k[G])\otimes V\otimes k & =\dim_{k_{0}}\bigoplus_{V\in I}\Hom^{G}(V\otimes k_{0},k[G])\otimes V\otimes k_{0}\\
 & =\dim_{k_{0}}k_{0}[G]=\dim_{k}k[G]
\end{alignat*}
so that $\eta_{I_{k},k[G]}$ is an isomorphism and $I_{k}=I_{G_{k}}$.\end{proof}
\begin{prop}
\label{prop:behaviour of irreducible representations for general glrg}Let
$I$ be a collection of elements of $\Loc^{G}S$ satisfying the conditions
in \ref{prop:generating irreducible representations}. Then
\[
\Homsh^{G}(V,W)=0\comma\Homsh^{G}(V,V)=\odi S\id_{V}\ \text{ for all \ensuremath{V\neq W\in I}}
\]
If $S$ is connected then $I$ is uniquely determined up to tensorization
by invertible sheaves (with trivial actions).\end{prop}
\begin{proof}
Assume $S$ connected. If we tensor the sheaves in $I$ by invertible
sheaves, we do not change their restrictions to the geometric points.
Conversely let $I'$ be another collection satisfying the condition
in \ref{prop:generating irreducible representations}. Given $W\in I'$,
there exists $V\in I$ such that $\Homsh^{G}(V,W)\neq0$. The sheaf
$\Homsh^{G}(V,W)$ is locally free thanks to \ref{lem:base change of invariants}.
Changing the base to all the geometric points of $S$, we see that
the map $\Homsh^{G}(V,W)\otimes V\arr W$ has to be surjective, that
$\Homsh^{G}(V,W)$ has rank $1$ and that $\rk V=\rk W$. In this
way we see that $V$ and $W$ differ by an invertible sheaf and that
$V$ is uniquely determined by $W$.

Now consider the locally free sheaf $\Homsh^{G}(V,W)$ for $V,W\in I$.
If $V\neq W$ then this sheaf is $0$ because $I\arr I_{G_{k}}$ is
bijective and therefore $V\otimes k\not\simeq W\otimes k$ for any
geometric point $\Spec k\arr S$. Finally, if $V=W$, we see that
$\id_{V}$ generates $\Homsh^{G}(V,V)$ in any geometric point.\end{proof}
\begin{defn}
We will say that $G$ has a \emph{good} \emph{representation theory}
over $S$ if it admits a collection $I$ as in \ref{prop:generating irreducible representations}.
We will briefly call a \emph{glrg }(good linearly reductive group)
a pair $(G,I_{G})$ where $G$ is a finite, flat, finitely presented
and linearly reductive group scheme over $S$ and $I_{G}$ is a collection
of elements as in \ref{prop:generating irreducible representations}.
We will simply write $G$ if this will not lead to confusion. If $T\arr S$
is a map, then $G_{T}=G\times_{S}T$ with the collection of the pullbacks
of the sheaves in $I_{G}$ is a glrg and we will always consider $G_{T}$
as a glrg with this particular collection.
\end{defn}
Note that if $G$ is a glrg then any $V\in I_{G}$ is not only an
irreducible representation, but a geometrically irreducible one. We
now show two examples of glrg's.
\begin{example}
Assume $S=\Spec k$, where $k$ is a field. Then $G$ has a good representation
theory if and only if $\End^{G}(V)\simeq k$ for all the irreducible
representations of $G$ and in this case, up to isomorphism, the only
collection $I$ satisfying \ref{prop:generating irreducible representations}
is $I_{G}$, the set of irreducible representations of $G$. In particular
any linearly reductive group $G$ over an algebraically closed field
is a glrg. Indeed if $W$ is an irreducible representation of $G$
then $\Hom^{G}(V,W)\neq0$ for some $V\in I$, since $\eta_{I,W}$
is an isomorphism. So $W\simeq V$.
\end{example}

\begin{example}
If $G$ is a diagonalizable group with group of characters $M=\Hom_{\text{grp}}(G,\Gm)$,
then $G$ has a good representation theory over $\Spec\Z$ and we
can choose as $I_{G}$ the set of representations $\Z_{m}$ given
by $\Z_{m}\arr\Z_{m}\otimes\Z[G]$, $1\arr1\otimes m$.
\end{example}
The definition of good linearly reductive group is just what we need
in order to have a representation theory for which coherent sheaves
with an action of $G$ are just, functorially, a collection of coherent
sheaves. The correct statement, which easily follows from the definition
of glrg and from \ref{prop:behaviour of irreducible representations for general glrg},
is the following:
\begin{prop}
\label{prop:equivariant coherent sheaves are collection of coherent sheves for glrg}If
$G$ is a glrg, then the functors below are quasi-inverse equivalences
of categories   \[   \begin{tikzpicture}[xscale=3.6,yscale=-0.6]     \node (A0_0) at (0, 0) {$\displaystyle \bigoplus_{V\in I_G}\duale V\otimes\shF_V$};     \node (A0_1) at (1, 0) {$(\shF_V)_{V\in I_G}$};     \node (A1_0) at (0, 1) {$\QCoh^G_S$};     \node (A1_1) at (1, 1) {$\QCoh_S^{I_G}$};     \node (A2_0) at (0, 2) {$\shF$};     \node (A2_1) at (1, 2) {$((V\otimes \shF)^G)_{V \in I_G}$};     \path (A1_0) edge [->]node [auto] {$\scriptstyle{}$} (A1_1);     \path (A2_0) edge [|->,gray]node [auto] {$\scriptstyle{}$} (A2_1);     \path (A0_1) edge [|->,gray]node [auto] {$\scriptstyle{}$} (A0_0);   \end{tikzpicture}   \] The
same statement holds if we replace $\QCoh$ by $\FCoh$ or $\Loc$.\end{prop}
\begin{example}
When $G$ is a finite diagonalizable group with group of characters
$M=\Hom(G,\Gm)$ and $R=\Z$, since $I_{G}$ is in bijection with
$M$, we retrieve the classical equivalence between $\QCoh_{R}^{G}$
and the stack of $M$-graded quasi-coherent sheaves.
\end{example}
The following result extends the usual result for linearly reductive
groups over an algebraically closed field.
\begin{prop}
\label{prop:decomposition of OG}If $G$ is a glrg, then we have an
isomorphism
\[
\odi S[G]\simeq\bigoplus_{V\in I_{G}}\duale{\underline{V}}\otimes V
\]
\end{prop}
\begin{proof}
By \ref{prop:the structure map are the invariants}, we have
\[
\Homsh^{G}(V,\odi S[G])\simeq(\duale V\otimes\odi S[G])^{G}\simeq\duale V
\]

\end{proof}
We state here the subsequent lemma, although we will use it in the
following sections.
\begin{lem}
\label{lem:co unit description}Given $V\in\Loc^{G}R$, the composition
\[
\duale V\otimes V\arrdi{\simeq}\Hom^{G}(V,R[G])\otimes V\arrdi{\theta_{\duale V}^{R[G]}}R[G]\arrdi{\varepsilon_{G}}R
\]
is the evaluation $e_{V}\colon\duale V\otimes V\arr R$, $e_{V}(\phi\otimes v)=\phi(v)$.
In particular if $G$ is a glrg we have
\[
\varepsilon_{G}=\bigoplus_{V\in I_{G}}e_{V}\colon\bigoplus_{V\in I_{G}}\duale V\otimes V\simeq R[G]\arr R
\]
\end{lem}
\begin{proof}
The statement is local on $R$, so we can assume that $R[G]$ is free
with basis $\{w_{k}\}_{k}$. If $\mu\colon V\arr V\otimes R[G]$ is
the structure map of $V$, then the structure map $\nu\colon\duale V\arr\duale V\otimes R[G]$
has the expression
\[
\nu(\phi)=\sum_{k}(\phi\otimes w_{k}^{*}\circ\mu)\otimes w_{k}
\]
The composition in the statement can be written as 
\[
f_{V}\colon\duale V\otimes V\arrdi{\nu\otimes\id}(\duale V\otimes R[G])\otimes V\simeq(\duale V\otimes V)\otimes R[G]\arrdi{e_{V}\otimes\varepsilon_{G}}R
\]
So
\[
f_{v}(\phi\otimes v)=e_{V}\otimes\varepsilon_{G}(\sum_{k}(\phi\otimes w_{k}^{*}\circ\mu)\otimes v\otimes w_{k})=\sum_{k}\varepsilon_{G}(w_{k})\phi\otimes w_{k}^{*}(\mu(v))\in R
\]
Moreover we can write
\[
\mu(v)=\sum_{k}v_{k}\otimes w_{k}\text{ and }v=\sum_{k}\varepsilon_{G}(w_{k})v_{k}
\]
and therefore $\phi\otimes w_{k}^{*}(\mu(v))=\phi(v_{k})$ and $f_{v}(\phi\otimes v)=\phi(v)$.
\end{proof}
We do not know an explicit characterization of glrg's among the linearly
reductive groups. On the other hand we are going to prove that any
finite, flat and finitely presented linearly reductive group $G$
is fppf locally a glrg. So, up to fppf base change, we can always
assume that we have a collection $I_{G}$ of geometrically irreducible
representations and therefore a simpler representation theory. If
moreover $G$ is étale, we will show that $G$ is also étale locally
a glrg. In particular we will conclude that if $G$ is étale and defined
over a strictly Henselian ring then it is a glrg.
\begin{lem}
\label{lem:deformation of locally free sheaves}Let $\stX$ be a proper
and flat algebraic stack over a noetherian local ring $R$. Denote
by $k$ the residue field of $R$ and consider a locally free sheaf
$V_{0}$ of rank $n$ over $\stX\times k$. If $\Hl^{2}(\stX\times k,\Endsh(V_{0}))=0$,
then there exists a locally free sheaf of rank $n$ over $\stX\times\widehat{R}$
lifting $V_{0}$, where $\widehat{R}$ is the completion of $R$.\end{lem}
\begin{proof}
Taking into account Grothendieck's existence theorem for proper stacks,
we can assume that $R$ is an Artinian ring (so that $\widehat{R}\simeq R$)
and that we have a lifting $\overline{V}$ of $V_{0}$ over $\stX\times(R/I)$,
where $I$ an ideal of $R$ such that $I^{2}=0$. Define the stack
$\stY$ over the fppf site $\stX_{\text{fppf}}$ of $\stX$ whose
objects over $\Spec B\arr\stX$ are locally free sheaves $N$ of rank
$n$ over $B$ with an isomorphism $\phi\colon N\otimes(B/IB)\arr\overline{V}\otimes(B/IB)$.
A section of $\stY\arr\stX_{\text{fppf}}$ yields a lifting of $\overline{V}$
on $\stX$. We are going to prove that $\stY$ is a gerbe over $\stX_{\text{fppf}}$
banded by the sheaf of abelian groups $\pi_{*}\Endsh(V_{0})$, where
$\pi\colon\stX\times k\arr\stX$ is the obvious closed immersion.
Since $\Hl^{2}(\stX,\pi_{*}\Endsh(V_{0}))=\Hl^{2}(\stX\times k,\Endsh(V_{0}))=0$
parametrizes those gerbes (see \cite[Chapter IV, §3, Section 3.4]{Giraud1971}),
we can then conclude that $\stY\arr\stX_{\text{fppf}}$ is a trivial
gerbe, which means that it has a section as required.

I claim that $\overline{V}$ is trivial in the fppf topology of $\stX$,
which implies that $\stY\arr\stX_{\text{fppf}}$ has local sections.
Indeed if $B$ is a ring and $P\arr\Spec B/IB$ is a $\Gl_{n}$-torsor
then by standard deformation theory it extends to a smooth map $Q\arr\Spec B$.
In particular, if we base change to $Q$, we can conclude that $P$
over $Q\times(B/IB)$ has a section, which means that it is trivial.

I also claim that two objects of $\stY$ over the same object of $\stX_{\text{fppf}}$
are locally isomorphic. Replacing again locally free sheaves by $\Gl_{n}$-torsors,
given $\Gl_{n}$-torsors $P,Q$ over $\Spec B$, we have to show that
an equivariant isomorphism $P\times(B/IB)\arr Q\times(B/IB)$ locally
extends to an equivariant isomorphism $P\arr Q$. In particular we
can assume that $P$ and $Q$ are both trivial and in this case the
above property follows because $\Gl_{n}(B)\arr\Gl_{n}(B/IB)$ is surjective,
being $\Gl_{n}$ smooth.

The previous two claims show that $\stY\arr\stX_{\text{fppf}}$ is
a gerbe. We have now to check the banding and therefore to compute
the automorphism group of an object $(N,\phi)\in\stY$ over a ring
$B$. The group $\Aut(\chi)$ consists of the automorphism $N\arrdi{\lambda}N$
inducing the identity on $N/IN$. It is easy to check that the map
\[
\Hom_{B}(N,IN)\arr\Aut\chi\comma\delta\longmapsto\id_{N}+\delta
\]
is an isomorphism of groups. Since $IN=I\otimes_{R}N$ we have
\[
\Hom_{B}(N,IN)=I\otimes\End_{B}(N)\simeq I/I^{2}\otimes\End_{B}(N)\simeq\End_{B/m_{R}B}(M\otimes(B/m_{R}B))
\]
\end{proof}
\begin{lem}
\label{lem:lifting representation on henselian rings}Assume that
$S=\Spec R$, where $R$ is a Henselian ring with residue field $k$
and let $V$ be a representation of $G$ over $\overline{k}$. If
$V$ is defined over $k$ then it lifts to $R$.\end{lem}
\begin{proof}
Since $G$ is finitely presented, we can assume that $R$ is the Henselization
of a scheme of finite type over $\Z$. Since $G$ is linearly reductive,
we have that $\Hl^{2}(\Bi(G\times k),-)=0$ and, thinking $G$-representations
as sheaves over $\Bi G$ and using \ref{lem:deformation of locally free sheaves},
we obtain a lifting of $V$ to a representation over the completion
$\widehat{R}$. We can then conclude using Artin approximation theorem
over $R$.\end{proof}
\begin{prop}
\label{prop:G has locally a good representation theory}There exists
an fppf coverings $\stU=\{U_{i}\arr S\}_{i\in I}$ such that $G\times_{S}U_{i}$
is a glrg over $U_{i}$. If $G$ is étale there exists an étale covering
with the same property.\end{prop}
\begin{proof}
We first deal with the case $S=\Spec k$, where $k$ is a field. An
irreducible representation $V$ of $G_{\overline{k}}$ is given by
a group homomorphism $G_{\overline{k}}\arr\Gl(V)$. Such a morphism
is defined over a finite extension $L/k$. Since $I_{G_{\overline{k}}}$
is finite we get our extension. Now assume that $G$ is étale. If
$k$ is perfect we already have our result. So assume $\car k=p>0$.
After passing to a separable extension of $k$ we can assume $G$
constant of order prime to $p$. So $G$ is defined over $\F_{p}$,
which is perfect and again we have our claim.

Now return to the general case. Since $G$ is finitely presented,
we can assume $S$ to be of finite type over $\Z$. Let $p\in S$
and $L/k(p)$ an extension such that $G_{L}$ is a glrg and $L/k(p)$
is separable if $G$ is étale. There exists a flat finitely presented
map $h\colon U\arr S$ such that $f^{-1}(p)\simeq\Spec L$. If $L/k$
is separable we can restrict $U$ and assume $h$ to be étale. This
shows that we can assume that $G_{k(p)}$ is a glrg. Now let $R$
be the Henselization of $\odi{S,p}$. From \ref{lem:lifting representation on henselian rings}
any $G_{k(p)}$ representation lifts to $R$ and, since $R$ is a
direct limit of algebras whose spectrum is étale over $S$, we get
the required result.
\end{proof}
Putting together \ref{lem:lifting representation on henselian rings}
and \ref{prop:G has locally a good representation theory} we get:
\begin{thm}
\label{thm:etale linearly reductive over sctrictly are glrg}A constant
linearly reductive group over a strict Henselian ring has a good representation
theory.
\end{thm}

\subsection{Linearly reductive groups over strictly Henselian rings.}

In this section we will study the structure and the actions of a flat,
finite and finitely presented linearly reductive group $G$ in the
special case when the base scheme is the spectrum of a strictly Henselian
ring. In particular we will describe the decomposition into unions
of connected components of a finite scheme with an action of $G$
and the structure of the connected component of $G$ containing the
identity.

Through this subsection we will assume $S=\Spec R$, where $R$ is
a strictly Henselian ring. Again $G$ will be a finite, flat, linearly
reductive group over $R$. We start with:
\begin{lem}
\label{cor:base change of local ring is local for strictly henselian ring}If
$A,B$ are local $R$-algebras with $A$ finite, then $A\otimes_{R}B$
is local.\end{lem}
\begin{proof}
Set $k_{A}\comma k_{B}$ for their residue fields. Since $A\otimes_{R}B$
is finite over $B$ it is enough to note that $k_{A}\otimes_{k_{R}}k_{B}$
is local since $k_{A}/k_{R}$ is purely inseparable.\end{proof}
\begin{prop}
We have an exact sequence 
\[
0\arr G_{1}\arr G\arr\underline{G}\arr0
\]
where $G_{1}$ is the connected component of $G$ containing $1$
and $\underline{G}$ is a constant group. Moreover, if $p$ is the
characteristic of the residue field of $R$, then $p\nmid|\underline{G}|$,
$G_{1}$ is diagonalizable and its group of characters $\Homsh(G_{1},\Gm)$
is a $p$-group. The decomposition of $G$ into connected components
is of the form 
\[
G=\bigsqcup_{i\in\underline{G}}G_{i}
\]
If $G$ acts on a finite $R$-scheme $X$, then it acts on the connected
components of $X$ and this action factors through $\underline{G}$.
Moreover the stabilizers of the connected components of $X$ are union
of connected components of $G$.\end{prop}
\begin{proof}
Let $G=\sqcup_{i\in I}G_{i}$ and $X=\sqcup_{j\in J}X_{j}$ be the
decomposition into connected components of $G$ and $X$ respectively
and let $\mu\colon X\times G\arr X$ the action of $G$ on $X$. Since
$X_{j}\times G_{i}$ is connected, there exists a unique $k_{j,i}\in J$
such that $\mu(X_{j}\times G_{i})\subseteq X_{k_{j,i}}$. Assume now
that $X=G$ with the regular representation. Since $\mu$ is isomorphic
to the projection $G\times G\arr G$, it is flat and finite, so $\mu(G_{j}\times G_{i})=G_{k_{j,i}}$
is a connected component of $G$. Define a product on $I$ by $i\cdot j=k_{i,j}$.
It is easy to check that $I$ is a group, whose neutral element $1\in I$
is the index of the connected component of the identity. Set $\underline{G}=I$.
The map $G\arr\underline{G}$ is surjective and the kernel is exactly
$G_{1}$. Since both $G_{1}$ and $\underline{G}$ are linearly reductive,
we can conclude that $G_{1}$ is diagonalizable and that $p\nmid|\underline{G}|$
by \cite[Lemma 2.20]{Abramovich2007}. Set $M=\Homsh(G_{1},\Gm)$
and $k$ for the residue field of $R$. If $\Z/q\Z<M$, then we have
a surjective morphism $G_{1}\arr\mu_{q,R}$. So $\mu_{q,R}$ has to
be connected and, since it is finite and flat, $\mu_{q,k}$ is connected
as well. But if $q\neq p$ then $\mu_{q,k}\simeq\Z/q\Z$. Therefore
$M$ is a $p$-group.

Now return to the general case, i.e. when $X$ is a finite $R$-scheme.
Since $\mu$ is an action, then $k_{*,*}$ defines an action of $\underline{G}$
on $J$. Moreover if $g\in G_{i}(T)$ we have that $(X_{j}\times T)g\subseteq X_{k_{j,i}}$
and therefore this is an equality since $J$ is finite. In particular
\[
\Stab X_{j}=\bigsqcup_{i\in\underline{G}\st k_{j,i}=j}G_{i}
\]
\end{proof}
\begin{notation}
We will continue to denote by $G_{1}$ the connected component of
$G$, by $\underline{G}$ the constant group $G/G_{1}$ and by $M=\Homsh(G_{1},\Gm)$
the group of characters of $G_{1}$. Given an index $i\in\underline{G}$
we will also denote by $G_{i}$ the connected component of $G$ corresponding
to such index.\end{notation}
\begin{cor}
\label{cor:well split over algebraically closed field}If $R=k$ is
an algebraically closed field, then $G\arr\underline{G}$ has a unique
section. In particular
\[
G\simeq G_{1}\ltimes\underline{G}
\]
\end{cor}
\begin{proof}
Set $p=\car k$. If $p=0$ then $G=\underline{G}$. So assume $p\neq0$
and let $G_{i}$ be a connected component of $G$. If we prove that
$|G_{i}(k)|=1$ then $G(k)\arr\underline{G}$ is an isomorphism (of
constant groups) and the section is unique. Since $k$ is algebraically
closed, we have $G_{i}(k)\neq\emptyset$. In particular $G_{i}\simeq G_{1}$.
But
\[
G_{1}(k)=\Hom_{\Grp}(M,k^{*})=0
\]
since $M$ is a $p$-group. 
\end{proof}
Now we want to study the open subgroups of $G$.
\begin{rem}
If $G$ is linearly reductive as we are assuming, then a subgroup
scheme is again a finite, flat and of finite presentation linearly
reductive group scheme (see \cite[Proposition 2.7]{Abramovich2007}).\end{rem}
\begin{prop}
\label{prop:decomposition of G into H torsors}Let $H$ be an open
and closed subgroup of $G$ and set 
\[
H^{i}=\bigsqcup_{j\in\underline{H}i}G_{j}\text{ where }i\in\underline{G}
\]
The schemes $H^{i}$ are stable under the right action of $H$ on
$G$ and they are fppf $H$-torsors. Moreover if $g\in H^{i}(T)$,
then $H^{i}\times T=(H\times T)g$.\end{prop}
\begin{proof}
If $h\in\underline{H}$ and $j\in\underline{H}i$ then $G_{j}\star G_{h}=G_{j\star h}=G_{h^{-1}j}\subseteq H^{i}$,
where $\star$ denote the regular representation, so $H^{i}$ is $H$-stable.
Since $G$ is flat and finite, $H^{i}$ has section in the fppf topology,
so we have to prove only the last claim, since the multiplication
by $g$ $H\times T\arr(H\times T)g$ is $H$-equivariant. Let $g\in H^{i}(T)$.
We can assume that $g\in G_{j}(T)$ for $j\in\underline{H}i$. In
this case it is enough to note that $(G_{h}\times T)g=G_{hj}\times T$.
\end{proof}
We state the following lemma here, although it will be used in the
following sections.
\begin{lem}
\label{lem:invariants by Gone sends local to local}Let $X$ be a
finite $R$-scheme with an action of $G$. Then $X/G_{1}$ has the
same connected components as $X$.\end{lem}
\begin{proof}
We have to prove that if $(A,m_{A})$ is a local and finite $R$-algebra
with an action of a diagonalizable group $D(H)$, then $A_{0}$ is
local. So we have to prove that any $x\in A_{0}-m_{A}\cap A_{0}$
is invertible in $A_{0}$. Since $x\notin m_{A}$ there exists $y\in A$
such that $xy=1$. Writing $y$ with respect to the decomposition
$A=\bigoplus_{h\in H}A_{h}$ we get
\[
y=\sum_{h\in H}y_{h}\then1=xy=\sum_{h\in H}xy_{h}\then xy_{0}=1
\]

\end{proof}

\subsection{Induction and $G$-equivariant algebras.}

One of the key points in the study of $G$-covers in the following
sections is the fact that each such cover, locally (at least on a
strict Henselization), can be described from an $H$-cover, where
$H$ is a proper subgroup of $G$, having some extra properties. Algebraically,
this procedure is obtained through an induction from $H$ to $G$.
So in this section we will introduce the concept of induction from
a subgroup, state some of its properties and then we will focus on
induction of algebras.

Throughout this section we will assume $S=\Spec R$, where $R$ is
a ring and $G$ will be as always a finite, flat and finitely presented
group scheme over $R$.
\begin{rem}
Let $H$ be a subgroup of $G$ and $F\colon(\Sch/S)^{op}\arr\set$
be a functor with a left action of $H$. Regarding $G$ as a $H$-space
via the restriction of the regular representation, we define
\[
\ind_{H}^{G}F=\Homsh^{H}(G,F)
\]
We endow $\ind_{H}^{G}F$ with the following left action of $G$.
The group $G$ acts on the right on itself through the product $G\times G\arrdi mG$
and, considering the trivial action of $G$ on $F$, we get a left
action of $G$ on $\Homsh(G,F)$ that restricts to a left action of
$G$ on $\ind_{H}^{G}F$.

Concretely, given $f\colon G\arr F\in\Homsh(G,F)$ we have that 
\[
f\in\ind_{H}^{G}F=\Homsh^{H}(G,F)\iff f(hg)=hf(g)\text{ for all }h\in H
\]
and if $g\in G$ then
\[
(g\star f)(t)=f(tg)
\]
\end{rem}
\begin{defn}
If $H$ is a subgroup scheme of $G$ and $\shF\in\FCoh^{H}$ we have
(see \ref{prop:the structure map are the invariants}) 
\[
\WW((\shF\otimes\odi{}[G])^{H})\simeq\Homsh^{H}(G,\WW(\shF))=\ind_{H}^{G}\WW(\shF)
\]
So we can define
\[
\ind_{H}^{G}\shF=(\shF\otimes\odi{}[G])^{H}\in\FCoh^{G}
\]
with the action given by the isomorphism $\WW(\ind_{H}^{G}\shF)\simeq\ind_{H}^{G}\WW(\shF)$.
\end{defn}
The following is a well known property of adjunction between induction
and restriction.
\begin{prop}
\cite[section 3.3]{Jantzen2003} If $H$ is a flat subgroup scheme
of $G$ and $V\in\FCoh^{G}$, $W\in\FCoh^{H}$, we have an isomorphism
\[
\Homsh^{H}(\R_{H}V,W)\simeq\Homsh^{G}(V,\ind_{H}^{G}W)
\]

\end{prop}
We now pass to the study of induction of finite algebras with an action
of $G$. From now to the end of the section $G$ will be assumed linearly
reductive.
\begin{defn}
We will denote by $\CAlg^{G}R$ the category of finite $R$-algebras
$A$ with a left action of $G$ on them, or, equivalently, a right
action of $G$ on $\Spec A$.\end{defn}
\begin{lem}
\label{lem:A is induction of the localization}If $R$ is strictly
Henselian, $H$ is an open and closed subgroup of $G$ and $A\in\CAlg^{H}R$
then 
\[
\ind_{H}^{G}A\simeq\prod_{i\in\underline{G}/\underline{H}}B_{i}
\]
as rings, where the $B_{i}$ are fppf locally isomorphic to $A$.
More precisely, if $R'$ is an $R$-algebra and $g\in G_{i}(R')$
then we have an induced isomorphism   \[   \begin{tikzpicture}[xscale=3.3,yscale=-1.2]     \node (A0_0) at (0, 0) {$\ind^G_H A\otimes R'$};     \node (A0_1) at (1, 0) {$\ind^G_H A\otimes R'$};     \node (A1_0) at (0, 1) {$B_i\otimes R'$};     \node (A1_1) at (1, 1) {$A\otimes R'$};     \path (A0_0) edge [->]node [auto] {$\scriptstyle{g}$} (A0_1);     \path (A0_0) edge [->]node [auto] {$\scriptstyle{}$} (A1_0);     \path (A0_1) edge [->]node [auto] {$\scriptstyle{}$} (A1_1);     \path (A1_0) edge [->,dashed]node [auto] {$\scriptstyle{\simeq}$} (A1_1);   \end{tikzpicture}   \] \end{lem}
\begin{proof}
We will make use of \ref{prop:decomposition of G into H torsors}.
The inclusions $H^{i}\arr G$ induce an isomorphism of functors
\[
\Homsh^{H}(G,\WW(A))\arr\prod_{i}\Homsh^{H}(H^{i},\WW(A))
\]
So we can set $B_{i}$ for the coherent algebra such that $\WW(B_{i})\simeq\Homsh^{H}(H^{i},\WW(A))$.
Since $H^{i}$ is an fppf $H$-torsor, $B_{i}$ is fppf locally isomorphic
to $A$. For the last claim, note that $H^{i}\times R'=(H\times R')g$
and therefore it is enough to apply $\Homsh^{H}(-,\WW(A))$ to the
commutative diagram of $H$-spaces   \[   \begin{tikzpicture}[xscale=0.9,yscale=-1.0]     \node (A0_0) at (0, 0) {$H$};     \node (A0_2) at (2, 0) {$G$};     \node (A0_3) at (3, 0) {$x$};     \node (A1_0) at (0, 1) {$Hg$};     \node (A1_2) at (2, 1) {$G$};     \node (A1_3) at (3, 1) {$xg$};     \path (A1_0) edge [->]node [auto] {$\scriptstyle{}$} (A1_2);     \path (A0_0) edge [->]node [auto] {$\scriptstyle{}$} (A1_0);     \path (A0_3) edge [|->,gray]node [auto] {$\scriptstyle{}$} (A1_3);     \path (A0_2) edge [->]node [auto] {$\scriptstyle{}$} (A1_2);     \path (A0_0) edge [->]node [auto] {$\scriptstyle{}$} (A0_2);   \end{tikzpicture}   \] 

\end{proof}
\begin{lem}
\label{lem:G acts transetively on the maximal ideals of fibers}Let
$R$ be a local ring and $A\in\CAlg^{G}R$ such that $A^{G}=R$. If
$G$ is constant then it acts transitively on the maximal ideals of
$A$.\end{lem}
\begin{proof}
Let $p,q\in\Spec A$ be closed points and assume by contradiction
that for any $g\in G$, $q\neq g(p)$. In particular we cannot have
$q\subseteq\cup_{g\in G}g(p)$ and therefore there exists $x\in q$
such that $g(x)\notin p$ for any $g\in G$. But
\[
\prod_{g\in G}g(x)\in q\cap A^{G}=q\cap R=m_{R}\subseteq p\then\exists g\in G\st g(x)\in p
\]

\end{proof}
The following proposition is one of the key points in the study of
the structure of covers and we will use it many times in the following
sections. It roughly means that the whole algebra (over which $G$
acts) can be recovered from a local algebra (over which acts a particular
subgroup of $G$) through induction. In particular it allows us to
reduce problems to local algebras, when we have to deal with properties
that behave well under induction.
\begin{prop}
\label{prop:induction from a localization on henselian ring}Assume
that $R$ is strictly Henselian and let $A\in\CAlg^{G}R$ be such
that $A^{G}=R$ and $p\in\Spec A$ be a closed point. Denote by $H_{p}$
the stabilizer of the connected component $\Spec A_{p}$ of $\Spec A$.
Then we have a $G$-equivariant isomorphism
\[
A\arr\ind_{H_{p}}^{G}A_{p}
\]
\end{prop}
\begin{proof}
Set $H=H_{p}$. The map $A\arr A_{p}$ is $H$-equivariant and therefore
we get a map $A\arrdi{\psi}\ind_{H}^{G}A_{p}$. Write $X_{q}=\Spec A_{q}$
for a closed point $q$ of $\Spec A$. Those are the connected components
of $X=\Spec A$. Let also $Y=\Spec(\ind_{H}^{G}A_{p})$, $\ind_{H}^{G}A_{p}=\prod_{i\in\underline{G}/\underline{H}}B_{i}$
and $Y_{i}=\Spec B_{i}$. Assume $X_{p}G_{i}=X_{q}$, where $i\in\underline{G}$.
Since $Y_{1}$ is mapped to $X_{p}$ and $ $$Y_{1}G_{i}=Y_{i}$ we
have a decomposition   \[   \begin{tikzpicture}[xscale=2.4,yscale=-1.0]     \node (A0_0) at (0, 0) {$A$};     \node (A0_1) at (1, 0) {$\ind^G_H A_p$};     \node (A1_0) at (0, 1) {$A_q$};     \node (A1_1) at (1, 1) {$B_i$};     \path (A0_0) edge [->]node [auto] {$\scriptstyle{\psi}$} (A0_1);     \path (A1_0) edge [->]node [auto] {$\scriptstyle{\psi_q}$} (A1_1);     \path (A0_1) edge [->]node [auto] {$\scriptstyle{}$} (A1_1);     \path (A0_0) edge [->]node [auto] {$\scriptstyle{}$} (A1_0);   \end{tikzpicture}   \] 
We have to prove that all the maps $\psi_{q}$ are isomorphisms and
that $G$ acts transitively on the connected components of $X$.

If $X_{p}G_{i}=X_{q}$, $R'$ is an fppf $R$-algebra and $g\in G(R')$
we have a commutative diagram   \[   \begin{tikzpicture}[xscale=2.9,yscale=-1.0]     \node (A0_1) at (1, 0) {$A\otimes R'$};     \node (A0_2) at (2, 0) {$A\otimes R'$};     \node (A1_1) at (1, 1) {$\ind^G_H A_p \otimes R'$};     \node (A1_2) at (2, 1) {$\ind^G_H A_p \otimes R'$};     \node (A2_0) at (0, 2) {$A_q \otimes R'$};     \node (A2_1) at (1, 2) {$B_i \otimes R'$};     \node (A2_2) at (2, 2) {$A_p \otimes R'$};     \path (A2_1) edge [->]node [auto] {$\scriptstyle{u}$} (A2_2);     \path (A0_1) edge [->]node [auto] {$\scriptstyle{g}$} (A0_2);     \path (A0_2) edge [->]node [auto] {$\scriptstyle{\psi \otimes R'}$} (A1_2);     \path (A0_1) edge [->,bend left=20]node [auto] {$\scriptstyle{}$} (A2_0);     \path (A1_1) edge [->]node [auto] {$\scriptstyle{g}$} (A1_2);     \path (A2_0) edge [->]node [auto] {$\scriptstyle{\psi_q \otimes R'}$} (A2_1);     \path (A1_1) edge [->]node [auto] {$\scriptstyle{}$} (A2_1);     \path (A0_1) edge [->]node [auto] {$\scriptstyle{\psi \otimes R'}$} (A1_1);     \path (A1_2) edge [->]node [auto] {$\scriptstyle{}$} (A2_2);   \end{tikzpicture}   \] Since
$G$ permutes the connected components of $X$, thanks to \ref{lem:A is induction of the localization},
the composition $u\circ(\psi_{q}\otimes R')$ is an isomorphism. Since
also $u$ is an isomorphism we can conclude that $\psi_{q}\otimes R'$
and therefore $\psi_{q}$ is an isomorphism. 

It remains to prove that $G$ acts transitively on the connected components
of $X$. Since $Z=X/G_{1}$ has the same connected components as $X$
for \ref{lem:invariants by Gone sends local to local}, $\underline{G}$
acts on $Z$ and $Z/G=\Spec R$, we can assume $G=\underline{G}$.
In this case the conclusion follows from \ref{lem:G acts transetively on the maximal ideals of fibers}.\end{proof}

\section{Equivariant sheaves and functors.}

Given a glrg $G$ over a ring $R$, proposition \ref{prop:generating irreducible representations}
tells us that a $G$-equivariant quasi-coherent sheaf $\shF$ over
an $R$-scheme $T$ is determined by a collection of quasi-coherent
sheaves on $T$ indexed by $I_{G}$, namely $\{(V\otimes\shF)^{G}\}_{V\in I_{G}}$.
Since we are mainly interested in affine maps of schemes, it is natural
to ask what additional structure a collection of sheaves as above
must have in order to correspond to a quasi-coherent sheaf of algebras.
We will answer this question but, in order to do that, it will be
convenient to associate to a sheaf $\shF$ not only a collection,
but a whole functor $\Omega^{\shF}=(-\otimes\shF)^{G}$ from the category
of locally free and finite $G$-representations $\Loc^{G}R$ to the
category of quasi-coherent sheaves. This has the advantage of making
sense for any finite, flat and finitely presented group scheme $G$.
The functor $\Omega^{\shF}$ is left exact and $R$-linear. We will
show that a structure of sheaf of algebras on $\shF$ corresponds
to a structure of monoidal functor on $\Omega^{\shF}$ and we will
conclude that the category of $G$-equivariant quasi-coherent sheaves
of algebras is equivalent to the category of left exact and $R$-linear
monoidal functors $\Loc^{G}R\arr\QCoh T$. When $G$ is linearly reductive,
any $R$-linear functor $\Loc^{G}R\arr\QCoh T$ is automatically exact,
and the above correspondences hold if we consider finitely presented
quasi-coherent sheaves or locally free sheaves of finite ranks instead
of all the quasi-coherent sheaves.

In the last two sections we will consider the case of $G$-torsors
and we will prove that, in the association above, they correspond
to left exact strong monoidal functors. This result is already proved
in \cite{Lurie2004}, and comes from a more general statement. On
the other hand the proof we present here is more elementary. We will
also prove a stronger result when $G$ is a super-solvable glrg (see
\ref{def: super solvable groups}), always in terms of functors.

In what follows we will consider given a flat, finite and finitely
presented group scheme $G$ over the base scheme $S$. We will also
assume that $S$ is affine, namely $S=\Spec R$, where $R$ is a ring.

\subsection{Linear functors and equivariant quasi-coherent sheaves.}

In this section we will show how we can pass from a $G$-equivariant
quasi-coherent sheaf on an $R$-scheme $T$ to a functor $\Loc^{G}R\arr\QCoh T$
and conversely.

We start defining the stack of $R$-linear functors $\Loc^{G}R\arr\QCoh(-)$.
\begin{defn}
Given an $R$-scheme $T$ we define $\QAdd^{G}T$ as the category
whose objects are $R$-linear functors 
\[
\Omega\colon\Loc^{G}R\arr\QCoh T
\]

We will denote by $\QAdd_{R}^{G}$ the stack over $\Sch/R$ whose
fibers are the categories $\QAdd^{G}T$. We define the categories
$\LAdd^{G}T$, $\CAdd^{G}T$ and the stacks $\LAdd_{R}^{G}$, $\CAdd_{R}^{G}$
replacing $\QCoh T$ by $\Loc T$, $\FCoh T$ respectively in the
above definition.
\end{defn}
The motivation of the notation $\QAdd^{G}$ is that $\text{Add}$
stands for additive functors, while $\text{Q}$ recall quasi-coherent
sheaves.

Since we have to deal with additive categories that are not abelian,
namely $\Loc^{G}R$, we specify here what we mean by (left) exact
functors.
\begin{defn}
An additive functor $F\colon\alA\arr\alB$ between additive categories
is (left, right) exact if it sends short exact sequences to (left,
right) short exact sequences.\end{defn}
\begin{rem}
Notice that, if $\alA$ is not abelian, the definition above does
not imply that an exact functor sends long exact sequences to long
exact sequences.
\end{rem}
We first state the main Theorem of this section.
\begin{thm}
\label{thm:additive functors are equivariant sheaves}Given an $R$-scheme
$T$, we have functors  \[   \begin{tikzpicture}[xscale=3.7,yscale=-0.6]     
\node (A0_0) at (0, 0) {$\shF_\Omega=\Omega_{R[G]}$};     
\node (A0_1) at (1, 0) {$\Omega$};     
\node (A1_0) at (0, 1) {$\QCoh^G T$};     
\node (A1_1) at (1, 1) {$\QAdd^G T$};     
\node (A2_0) at (0, 2) {$\shF$};     
\node (A2_1) at (1, 2) {$\Omega^{\shF}=(-\otimes \shF)^G$};     
\path (A0_0) edge [<-|,gray]node [auto] {$\scriptstyle{}$} (A0_1);     \path (A1_0) edge [->]node [auto] {$\scriptstyle{}$} (A1_1);     \path (A2_1) edge [<-|,gray]node [auto] {$\scriptstyle{}$} (A2_0);   \end{tikzpicture}   \] Moreover $\Omega^{\shF}$ is always left exact, there exist a natural
isomorphism $\shF\arr\Omega_{R[G]}^{\shF}$ and a natural transformation
$\Omega\arr\Omega^{\Omega_{R[G]}}$ which is an isomorphism if and
only if $\Omega$ is left exact. In particular $\Omega^{*}$ is an
equivalence onto the full subcategory of $\QAdd^{G}T$ of left exact
functors.\end{thm}
\begin{rem}
It is part of the statement of the Theorem that for each $\Omega\in\QAdd^{G}T$
there exists a natural action of $G$ on the quasi-coherent sheaf
$\Omega_{R[G]}$. Moreover we have to warn the reader that the functor
$\Omega^{*}$ does not extend to a map of stacks, because if $\shF\in\QCoh^{G}T$
and $f\colon T'\arr T$ is a base change, then the natural map $f^{*}(\shF\otimes V)^{G}\arr((f^{*}\shF)\otimes V)^{G}$
is not an isomorphism in general. However, assuming Theorem \ref{thm:additive functors are equivariant sheaves},
we can prove the following.\end{rem}
\begin{prop}
\label{prop:additive functors are equivariant sheaves glrg}The following
conditions are equivalent:
\begin{enumerate}
\item $G$ is linearly reductive over $R$;
\item the functor of invariants $(-)^{G}\colon\Loc^{G}R\arr\QCoh R$ is
exact;
\item all the $R$-linear functors $\Omega\colon\Loc^{G}R\arr\QCoh R$ are
left exact
\end{enumerate}
In this case all the $R$-linear functors $\Loc^{G}R\arr\QCoh T$
are exact and the maps defined in \ref{thm:additive functors are equivariant sheaves}
yield isomorphisms of stacks
\[
\QCoh_{R}^{G}\simeq\QAdd_{R}^{G}\qquad\Loc_{R}^{G}\simeq\LAdd_{R}^{G}\qquad\FCoh_{R}^{G}\simeq\CAdd_{R}^{G}
\]
\end{prop}
\begin{proof}
We first prove that if all functors in $\QAdd^{G}T$ are left exact,
then $(-)^{G}\colon\QCoh^{G}T\arr\QCoh T$ is exact. In particular
we will have implications $3)\then1)\then2)$. Given a surjection
$\phi\colon\shF\arr\shF'$ in $\QCoh^{G}R$ we define the functor
\[
\Omega\colon\Loc^{G}R\arr\QCoh T\comma\Omega_{V}=\Coker((\shF\otimes V)^{G}\arrdi{(\phi\otimes\id_{V})^{G}}(\shF'\otimes V)^{G})
\]
From \ref{prop:the structure map are the invariants} we see that
$\Omega_{R[G]}=0$ and from \ref{thm:additive functors are equivariant sheaves}
we can conclude that $\Omega=0$. In particular $\Omega_{R}=\Coker(\shF^{G}\arr\shF'^{G})=0$.

Now assume that $(-)^{G}\colon\Loc^{G}R\arr\QCoh R$ is exact. We
want to prove that any $\Omega\in\QAdd^{G}T$ is exact, showing, in
particular, implication $2)\then3)$. It is enough to prove that any
short exact sequence in $\Loc^{G}R$ has a $G$-equivariant splitting.
Consider a short exact sequence in $\Loc^{G}R$ 
\[
0\arr V'\arr V\arr V''\arr0
\]
This is a split sequence in $\Loc R$. In particular $\Hom^{G}(V'',-)=(-)^{G}\circ\Hom(V'',-)$
maintains the exactness of such sequence. Therefore the map 
\[
\Hom^{G}(V'',V)\arr\Hom^{G}(V'',V'')
\]
is surjective and a lifting of $\id_{V''}$ yields the required section.\end{proof}
\begin{rem}
Theorem \ref{thm:additive functors are equivariant sheaves} is no
longer true if $S$ is not affine. For instance let $S$ be a proper
scheme over $k$ such that $\Hl^{0}(\odi S)=k$ and consider $G=1$
and the $\odi S$-linear functor
\[
\Omega=\Hl^{0}(-)\otimes_{k}\odi S\colon\Loc^{G}S=\Loc S\arr\QCoh S
\]
If $\Omega\simeq(\Omega_{\odi S[G]}\otimes-)^{G}=\id_{\Loc S}$ it
will follow that any locally free sheaf is free. When $G$ is linearly
reductive, the right class of functors to consider for a general base
scheme $S$ is the one of functors $\Loc_{S}^{G}\arr\QCoh_{S}$. This
works also in general, for non linearly reductive groups, if we restrict
those stacks to the fppf site of $S$. Indeed we have to warn the
reader that in general, if $\shF\in\QCoh S$, the functor $(\shF\otimes-)^{G}$
does not yield a map of stacks $\Loc_{S}^{G}\arr\QCoh_{S}$, even
when $S$ is affine, because the invariant functor $(-)^{G}$ does
not commute with arbitrary base changes. Anyway in this exposition
we have preferred to avoid technicalities and , for instance, consider
the simplest case $S$ affine.
\end{rem}

\begin{rem}
\label{rem: additive functors uniquely determined on irreducible representations}
When $G$ is a glrg, theorem \ref{thm:additive functors are equivariant sheaves}
and \ref{prop:equivariant coherent sheaves are collection of coherent sheves for glrg}
say that, in order to define an $R$-linear functor $\Omega\colon\Loc^{G}R\arr\QCoh T$,
it is enough to give quasi-coherent sheaves $(\shF_{V})_{V\in I_{G}}$.
We can then set 
\[
\Omega_{W}=\bigoplus_{V\in I_{G}}\Homsh^{G}(V,W)\otimes\shF_{V}
\]

\end{rem}
Before proving theorem \ref{thm:additive functors are equivariant sheaves}
we need some preliminary lemmas, which will be useful also in other
situations.
\begin{lem}
\label{lem:equivalence among Coh T and additive functors Loc R to Coh T}Given
an $R$-scheme T we have equivalences of categories   \[   \begin{tikzpicture}[xscale=5.7,yscale=-0.6]     \node (A0_0) at (0, 0) {$F(R)$};     \node (A0_1) at (1, 0) {$F$};     \node (A1_0) at (0, 1) {$\QCoh T$};     \node (A1_1) at (1, 1) {$\{R\text{-linear functors }\Loc R\arr\QCoh T\}$};     \node (A2_0) at (0, 2) {$\shF$};     \node (A2_1) at (1, 2) {$-\otimes_{R} \shF$};     \path (A1_0) edge [->]node [auto] {$\scriptstyle{}$} (A1_1);     \path (A2_0) edge [|->,gray]node [auto] {$\scriptstyle{}$} (A2_1);     \path (A0_1) edge [|->,gray]node [auto] {$\scriptstyle{}$} (A0_0);   \end{tikzpicture}   \] \end{lem}
\begin{proof}
Clearly $R\otimes_{R}\shF\simeq\shF$. On the other hand, since $F$
is $R$-linear, we can define   \[   \begin{tikzpicture}[xscale=3.0,yscale=-0.5]     \node (A0_0) at (0, 0) {$V\otimes F(R)$};     \node (A0_1) at (1, 0) {$F(V)$};     \node (A1_0) at (0, 1) {$v\otimes x$};     \node (A1_1) at (1, 1) {$F_v(x)$};     \path (A0_0) edge [->]node [auto] {$\scriptstyle{\gamma_{F,V}}$} (A0_1);     \path (A1_0) edge [|->,gray]node [auto] {$\scriptstyle{}$} (A1_1);   \end{tikzpicture}   \] where
$F_{v}=F(R\arrdi vV)\colon F(R)\arr F(V)$. It is straightforward
to check that the maps $\gamma_{F,*}-\otimes F(R)\arr F$ are natural
in $F$. So it remains to prove that it is an isomorphism. By additivity
of $F$, $\gamma_{F,V}$ is an isomorphism when $V$ is free. Now
let $V\in\Loc R$ and consider a presentation $V_{1}\arr V_{0}\arr V$
with $V_{1},V_{0}$ free. We have a commutative diagram   \[   \begin{tikzpicture}[xscale=2.5,yscale=-1.1]     \node (A0_0) at (0, 0) {$V_1\otimes F(R)$};     \node (A0_1) at (1, 0) {$V_0\otimes F(R)$};     \node (A0_2) at (2, 0) {$V\otimes F(R)$};     \node (A0_3) at (3, 0) {$0$};     \node (A1_0) at (0, 1) {$F(V_1)$};     \node (A1_1) at (1, 1) {$F(V_0)$};     \node (A1_2) at (2, 1) {$F(V)$};     \node (A1_3) at (3, 1) {$0$};     \path (A0_1) edge [->]node [auto] {$\scriptstyle{\gamma_{F,V_0}}$} (A1_1);     \path (A0_0) edge [->]node [auto] {$\scriptstyle{}$} (A0_1);     \path (A0_1) edge [->]node [auto] {$\scriptstyle{}$} (A0_2);     \path (A1_0) edge [->]node [auto] {$\scriptstyle{}$} (A1_1);     \path (A1_1) edge [->]node [auto] {$\scriptstyle{}$} (A1_2);     \path (A0_2) edge [->]node [auto] {$\scriptstyle{\gamma_{F,V}}$} (A1_2);     \path (A1_2) edge [->]node [auto] {$\scriptstyle{}$} (A1_3);     \path (A0_2) edge [->]node [auto] {$\scriptstyle{}$} (A0_3);     \path (A0_0) edge [->]node [auto] {$\scriptstyle{\gamma_{F,V_1}}$} (A1_0);   \end{tikzpicture}   \] Since
$V$ is projective, both rows are exact and since $V_{1},V_{0}$ are
free we can conclude that $\gamma_{F,V}$ is an isomorphism.\end{proof}
\begin{cor}
\label{cor:additivity over locally free sheaves}Let $\Omega\in\QAdd^{G}T$.
Then there exists a unique natural transformation
\[
\gamma_{V,W}\colon V\otimes\Omega_{W}\arr\Omega_{\underline{V}\otimes W}\text{ for }V\in\Loc R\comma W\in\Loc^{G}R
\]
such that $\gamma_{R,W}=\id_{\Omega_{W}}$ and it is an isomorphism.
Moreover $\gamma$ is natural also in $\Omega\in\QAdd^{G}T$.\end{cor}
\begin{proof}
The functors $V\longmapsto V\otimes\Omega_{W}$ and $V\longmapsto\Omega_{\underline{V}\otimes W}$
from $\Loc R$ to $\QCoh T$ coincides on $V=R$. So $\id_{\Omega_{W}}$
extends to a unique natural transformation $\gamma_{-,W}\colon-\otimes\Omega_{W}\arr\Omega_{\underline{-}\otimes W}$,
which is an isomorphism. The naturality with respect to $W\in\Loc^{G}R$
and $\Omega\in\QAdd^{G}T$ follows by a similar trick.

\end{proof}
We are now ready to define the action of $G$ on $\Omega_{R[G]}$
for any $\Omega\in\QAdd^{G}T$.
\begin{lem}
The co-multiplication 
\[
R[G]\arrdi{\Delta_{G}}R[G]\otimes\underline{R[G]}
\]
is $G$-equivariant and, given $\Omega\in\QAdd^{G}T$, the map
\[
\Omega_{R[G]}\arrdi{\Omega_{\Delta_{G}}}\Omega_{R[G]\otimes\underline{R[G]}}\simeq\Omega_{R[G]}\otimes R[G]
\]
defines an action of $G$ on $\Omega_{R[G]}$.\end{lem}
\begin{proof}
The map $\Delta_{G}$ is $G$-equivariant since $(gh)\star t=t^{-1}gh=(g\star t)h$
for any $h,g,t\in G$, where $\star$ denotes the regular action of
$G$ on itself (see \ref{def:regular representation} for the convention
used). Instead the commutative diagrams that $\Omega_{\Delta_{G}}$
has to satisfy in order to be an action come from the following commutative
diagrams of $G$-equivariant maps, after applying the functor $\Omega$.
  \[   \begin{tikzpicture}[xscale=3.6,yscale=-1.2]     
\node (A0_0) at (0, 0) {$R[G]$};     
\node (A0_1) at (1, 0) {$R[G]\otimes \underline{R[G]}$};     
\node (A0_3) at (2.5, 0) {$R[G]\otimes \underline{R[G]}$};     
\node (A1_0) at (0, 1) {$R[G]\otimes \underline{R[G]}$};     
\node (A1_1) at (1, 1) {$R[G]\otimes \underline{R[G]} \otimes \underline{R[G]}$};     
\node (A1_2) at (1.8, 1) {$R[G]$};     
\node (A1_3) at (2.5, 1) {$R[G]$};    
\path (A1_2) edge [->]node [auto] {$\scriptstyle{\Delta_G}$} (A0_3);     \path (A0_0) edge [->]node [auto] {$\scriptstyle{\Delta_G}$} (A0_1);     \path (A1_0) edge [->]node [auto] {$\scriptstyle{\Delta_G \otimes \id}$} (A1_1);     \path (A0_3) edge [->]node [auto] {$\scriptstyle{\id\otimes \varepsilon}$} (A1_3);     \path (A0_0) edge [->]node [auto] {$\scriptstyle{\Delta_G}$} (A1_0);     \path (A0_1) edge [->]node [auto] {$\scriptstyle{\id\otimes \Delta_G}$} (A1_1);     \path (A1_2) edge [->]node [auto] {$\scriptstyle{\id}$} (A1_3);   \end{tikzpicture}   \] \end{proof}
\begin{lem}
\label{lem:defining natural transformations on RG}Let $\Omega,\Gamma\in\QAdd^{G}T$,
with $\Gamma$ left exact. Then the map   \[   \begin{tikzpicture}[xscale=4.0,yscale=-0.6]     \node (A0_0) at (0, 0) {$\Hom_{\QAdd^{G}T}(\Omega,\Gamma)$};     \node (A0_1) at (1, 0) {$\Hom_T(\Omega_{R[G]},\Gamma_{R[G]})$};     \node (A1_0) at (0, 1) {$\sigma$};     \node (A1_1) at (1, 1) {$\sigma_{R[G]}$};     \path (A0_0) edge [->]node [auto] {$\scriptstyle{}$} (A0_1);     \path (A1_0) edge [|->,gray]node [auto] {$\scriptstyle{}$} (A1_1);   \end{tikzpicture}   \] is
injective and its image is composed of the morphisms $\delta\colon\Omega_{R[G]}\arr\Gamma_{R[G]}$
such that, for any $u\in\End^{G}(R[G])$, the following diagram is
commutative   \[   \begin{tikzpicture}[xscale=1.8,yscale=-1.0]     \node (A0_0) at (0, 0) {$\Omega_{R[G]}$};     \node (A0_1) at (1, 0) {$\Gamma_{R[G]}$};     \node (A1_0) at (0, 1) {$\Omega_{R[G]}$};     \node (A1_1) at (1, 1) {$\Gamma_{R[G]}$};     \path (A0_0) edge [->]node [auto] {$\scriptstyle{\delta}$} (A0_1);     \path (A1_0) edge [->]node [auto] {$\scriptstyle{\delta}$} (A1_1);     \path (A0_1) edge [->]node [auto] {$\scriptstyle{\Gamma_u}$} (A1_1);     \path (A0_0) edge [->]node [auto] {$\scriptstyle{\Omega_u}$} (A1_0);   \end{tikzpicture}   \] Moreover
$\sigma\colon\Omega\arr\Gamma$ is an isomorphism if and only if $\sigma_{R[G]}$
is an isomorphism and $\Omega$ is left exact.\end{lem}
\begin{proof}
Denote by $M$ the set of maps $\delta$ as in the statement. Clearly,
if $\sigma\colon\Omega\arr\Gamma$ is a natural transformation, then
$\sigma_{R[G]}\in M$. With all the $V\in\Loc^{G}R$ we associate
an exact sequence
\[
\duale{V_{1}}\arr\duale{V_{0}}\arr\duale V\arr0
\]
as in \ref{lem:presentation of G equivariant co modules}. In particular
the $V_{i}$ are direct sums of copies of the regular representation
$R[G]$. Since $V$ is locally free, the dual of the above sequence
is still exact and can be decomposed in two short exact sequences
in $\Loc^{G}R$. In particular, since $\Gamma$ is left exact, the
sequence
\[
0\arr\Gamma_{V}\arr\Gamma_{V_{0}}\arr\Gamma_{V_{1}}
\]
is exact too.

Note also that, thanks to the additivity and $R$-linearity of $\Omega$
and $\Gamma$, a $\delta\in M$ extends uniquely to a natural transformation
$\delta_{*}\colon\Omega\arr\Gamma$ if we restrict those functors
to the full subcategory of $\Loc^{G}R$ of sheaves which are direct
sums of copies of the regular representation. In particular, given
$\delta\in M$ and $V\in\Loc^{G}R$ there exists a unique $\delta_{V}$
making the following diagram commutative   \[   \begin{tikzpicture}[xscale=1.8,yscale=-1.0]     \node (A0_0) at (0, 0) {$0$};     \node (A0_1) at (1, 0) {$\Omega_{V}$};     \node (A0_2) at (2, 0) {$\Omega_{V_0}$};     \node (A0_3) at (3, 0) {$\Omega_{V_1}$};     \node (A1_0) at (0, 1) {$0$};     \node (A1_1) at (1, 1) {$\Gamma_{V}$};     \node (A1_2) at (2, 1) {$\Gamma_{V_0}$};     \node (A1_3) at (3, 1) {$\Gamma_{V_1}$};     \path (A0_0) edge [->]node [auto] {$\scriptstyle{}$} (A0_1);     \path (A0_2) edge [->]node [auto] {$\scriptstyle{}$} (A0_3);     \path (A1_0) edge [->]node [auto] {$\scriptstyle{}$} (A1_1);     \path (A0_3) edge [->]node [auto] {$\scriptstyle{\delta_{V_1}}$} (A1_3);     \path (A0_1) edge [->]node [auto] {$\scriptstyle{}$} (A0_2);     \path (A1_1) edge [->]node [auto] {$\scriptstyle{}$} (A1_2);     \path (A0_2) edge [->]node [auto] {$\scriptstyle{\delta_{V_0}}$} (A1_2);     \path (A1_2) edge [->]node [auto] {$\scriptstyle{}$} (A1_3);     \path (A0_1) edge [->]node [auto] {$\scriptstyle{\delta_V}$} (A1_1);   \end{tikzpicture}   \] Here
we use that the second row is exact. This shows that the map $*_{R[G]}$
in the statement is injective and tells how to extend a $\delta\in M$
to a map $\delta_{*}\colon\Omega\arr\Gamma$. In order to prove that
such map is natural and does not depend on the choice of the exact
sequence associated with $V\in\Loc^{G}R$, it is enough to note that
every map $f\colon V\arr W$, where $W\in\Loc^{G}R$, extends to a
diagram of $G$-equivariant maps   \[   \begin{tikzpicture}[xscale=1.8,yscale=-1.0]     \node (A0_0) at (0, 0) {$0$};     \node (A0_1) at (1, 0) {$V$};     \node (A0_2) at (2, 0) {$V_0$};     \node (A0_3) at (3, 0) {$V_1$};     \node (A1_0) at (0, 1) {$0$};     \node (A1_1) at (1, 1) {$W$};     \node (A1_2) at (2, 1) {$W_0$};     \node (A1_3) at (3, 1) {$W_1$};     \path (A0_0) edge [->]node [auto] {$\scriptstyle{}$} (A0_1);     \path (A0_2) edge [->]node [auto] {$\scriptstyle{}$} (A0_3);     \path (A1_0) edge [->]node [auto] {$\scriptstyle{}$} (A1_1);     \path (A0_3) edge [->]node [auto] {$\scriptstyle{f_2}$} (A1_3);     \path (A0_1) edge [->]node [auto] {$\scriptstyle{}$} (A0_2);     \path (A1_1) edge [->]node [auto] {$\scriptstyle{}$} (A1_2);     \path (A0_2) edge [->]node [auto] {$\scriptstyle{f_1}$} (A1_2);     \path (A1_2) edge [->]node [auto] {$\scriptstyle{}$} (A1_3);     \path (A0_1) edge [->]node [auto] {$\scriptstyle{f}$} (A1_1);   \end{tikzpicture}   \] 
Indeed it is enough to take the dual sequences and note that $\Hom^{G}(\duale{R[G]},-)\simeq(R[G]\otimes-)^{G}$
is just the forgetful functor $\QCoh^{G}R\arr\QCoh R$ by \ref{prop:the structure map are the invariants}.
The last claim follows easily from the diagram above.\end{proof}
\begin{prop}
\label{prop:natural map Omega in OmegaF}The composition
\[
\eta_{V}\colon V\arrdi{\mu_{V}}V\otimes R[G]\arrdi{\id\otimes\sigma_{G}}V\otimes R[G]\arrdi{\text{swap}}R[G]\otimes\underline{V}\text{ for }V\in\Loc^{G}R
\]
defines a natural transformation $\eta\colon\id_{\Loc^{G}R}\arr R[G]\otimes\underline{-}$
of functors $\Loc^{G}R\arr\Loc^{G}R$ and $\eta_{R[G]}=\Delta_{G}$.
Given an $R$-scheme $T$ this map induces a natural transformation
\[
\Omega_{V}\arr(\Omega_{R[G]}\otimes V)^{G}
\]
of functors $\Loc^{G}R\times\QAdd^{G}T\arr\Coh T$ which is an isomorphism
if $V=R[G]$ or $\Omega$ is left exact. Moreover the induced map
\[
\theta_{V}\colon\Omega_{V}\otimes\duale V\arr\Omega_{R[G]}
\]
is $G$-equivariant and it is given by
\[
\Omega_{V}\otimes\duale V\simeq\Omega_{V}\otimes\Hom^{G}(V,R[G])\arrdi{x\otimes\phi\arr\Omega_{\phi}(x)}\Omega_{R[G]}
\]
\end{prop}
\begin{proof}
The first claim is a classical result, taking into account the particular
comodule structure we have put on $R[G]$. Given $\Omega\in\QAdd^{G}T$
and applying it on $\eta_{V}$ for any $V\in\Loc^{G}R$, we get a
natural map $\delta_{V}\colon\Omega_{V}\arr\Omega_{R[G]}\otimes V$
such that $\delta_{R[G]}=\Omega_{\Delta_{G}}$, the comodule structure
of $\Omega_{R[G]}$. If $V=R[G]$ we have a factorization 
\[
\Omega_{R[G]}\arrdi{\Omega_{\Delta_{G}}}(\Omega_{R[G]}\otimes R[G])^{G}\arr\Omega_{R[G]}\otimes R[G]
\]
and, since $(\Omega_{R[G]}\otimes-)^{G}\arr\Omega_{R[G]}\otimes-$
is a natural transformation of left exact functors, by \ref{lem:defining natural transformations on RG}
it follows that $\delta_{*}$ factors through a natural transformation
$\Omega\arr(\Omega_{R[G]}\otimes-)^{G}$.

The map $\theta_{V}$ in the statement can be obtained applying $\Omega$
to the map $\gamma_{V}\colon V\otimes\duale{\underline{V}}\arr R[G]$
induced by $\eta_{V}$. We have to prove that the composition
\[
f_{V}\colon V\otimes\Hom^{G}(V,R[G])\arrdi{\id\otimes\duale{\varepsilon_{G}}}V\otimes\duale V\arrdi{\gamma_{V}}R[G]
\]
is just the evaluation $x\otimes\phi\longmapsto\phi(x)$. By construction
we have $\gamma_{V}(x\otimes\psi)=m_{R[G]}\circ(\id\otimes\psi)\circ\eta_{V}(x)$.
In particular
\[
f_{V}(x\otimes\phi)=m_{R[G]}\circ[\id\otimes(\varepsilon_{G}\circ\phi)]\circ\eta_{V}(x)=m_{R[G]}\circ(\id\otimes\varepsilon_{G})\circ(\id\otimes\phi)\circ\eta_{V}(x)
\]
Since $\eta_{*}$ is natural, we have $\id\otimes\phi\circ\eta_{V}=\eta_{R[G]}\circ\phi=\Delta_{G}\circ\phi$
and, since $m_{R[G]}\circ(\id\otimes\varepsilon_{G})\circ\Delta_{G}=\id$,
that
\[
f_{V}(x\otimes\phi)=m_{R[G]}\circ(\id\otimes\varepsilon_{G})\circ\Delta_{G}(\phi(x))=\phi(x)
\]

\end{proof}

\begin{proof}
\emph{(of Theorem \ref{thm:additive functors are equivariant sheaves})}
The left exactness of $\Omega^{\shF}$ follows from the fact that
any short exact sequence in $\Loc^{G}R$ is locally split in $\Loc R$,
so that $(-\otimes\shF)$ is exact here and the fact that $(-)^{G}$
is left exact. 

Let now $\shF\in\QCoh_{R}^{G}$ with structure map $\shF\arrdi{\mu}\shF\otimes R[G]$.
Thanks to \ref{prop:the structure map are the invariants} we have
an isomorphism $\shF\arrdi{\mu}(\shF\otimes R[G])^{G}=\Omega_{R[G]}^{\shF}$
and we want to prove that it is $G$-equivariant. This is equivalent
to requiring that the dashed map $\alpha$ making the following diagram
commutative is just $\mu$.   \[   \begin{tikzpicture}[xscale=4.1,yscale=-1.2]     \node (A0_0) at (0, 0) {$\shF$};     \node (A0_1) at (1, 0) {$(\shF\otimes R[G])^G$};     \node (A0_2) at (2, 0) {$\shF\otimes R[G]$};     \node (A1_1) at (1, 1) {$(\shF\otimes R[G]\otimes\underline{R[G]})^G$};     \node (A1_2) at (2, 1) {$\shF\otimes R[G]\otimes\underline{R[G]}$};     \node (A2_0) at (0, 2) {$\shF\otimes\underline{R[G]}$};     \node (A2_1) at (1, 2) {$(\shF\otimes R[G])^G\otimes\underline{R[G]}$};     \node (A2_2) at (2, 2) {$\shF\otimes R[G]\otimes\underline{R[G]}$};     \path (A0_1) edge [->]node [auto] {$\scriptstyle{}$} (A1_1);     \path (A0_0) edge [->]node [auto] {$\scriptstyle{\mu}$} (A0_1);     \path (A0_1) edge [right hook->]node [auto] {$\scriptstyle{}$} (A0_2);     \path (A2_1) edge [right hook->]node [auto] {$\scriptstyle{}$} (A2_2);     \path (A1_1) edge [right hook->]node [auto] {$\scriptstyle{}$} (A1_2);     \path (A0_2) edge [->]node [auto] {$\scriptstyle{\id\otimes\Delta_G}$} (A1_2);     \path (A1_1) edge [->]node [auto] {$\scriptstyle{}$} (A2_1);     \path (A0_0) edge [->,dashed]node [auto] {$\scriptstyle{\alpha}$} (A2_0);     \path (A2_0) edge [->]node [auto] {$\scriptstyle{\mu\otimes\id}$} (A2_1);     \path (A1_2) edge [->]node [auto] {$\scriptstyle{\id}$} (A2_2);   \end{tikzpicture}   \] 
Note that $\mu\otimes\id\circ\alpha=\id\otimes\Delta_{G}\circ\mu=\mu\otimes\id\circ\mu$
and that $\mu\otimes\id$ is injective. We can therefore conclude
that $\alpha=\mu$.

The natural transformation $\Omega\arr(\Omega_{R[G]}\otimes-)^{G}$
and all the other claims are in \ref{prop:natural map Omega in OmegaF}.
\end{proof}
We want now to give a different description of the functor $\Omega\arr\Omega_{R[G]}=\shF_{\Omega}$
of Theorem \ref{thm:additive functors are equivariant sheaves} in
the particular case when the group $G$ is a glrg.
\begin{prop}
\label{prop:from functors to sheaves for linearly reductive groups}Assume
that $G$ is a glrg. Given $\Omega\in\QAdd_{R}^{G}$ the isomorphisms
(see \ref{prop:decomposition of OG}) 
\[
R[G]\simeq\bigoplus_{V\in I_{G}}\duale{\underline{V}}\otimes V\text{ and }\Omega_{R[G]}\arr\bigoplus_{V\in I_{G}}\duale V\otimes\Omega_{V}
\]
are $G$-equivariant and the last one defines a natural isomorphism
$(-)_{R[G]}\arr\bigoplus_{V\in I_{G}}\duale V\otimes(-)_{V}$ of functors
$\QAdd_{R}^{G}\arr\QCoh_{R}^{G}$.\end{prop}
\begin{proof}
We can assume that $\Omega=\Omega^{\shF}=(\shF\otimes-)^{G}$ for
some $\shF\in\QCoh_{R}^{G}$. The map $\mu\colon\shF\arr(\shF\otimes R[G])^{G}$
is a $G$-equivariant isomorphism, where $\mu$ is comodule structure
on $\shF$ and its inverse is the restriction of $\id\otimes\varepsilon_{G}$,
which is therefore $G$-equivariant. Thanks to \ref{lem:co unit description}
and using its notation, we have a commutative diagram   \[   \begin{tikzpicture}[xscale=4.0,yscale=-0.7]     \node (A0_0) at (0, 0) {$(R[G]\otimes\shF)^G$};     \node (A0_1) at (1, 0) {$R[G]\otimes\shF$};     \node (A0_2) at (2, 0) {$\shF$};     \node (A2_0) at (0, 2) {$\displaystyle\bigoplus_{V\in I_{G}}\duale V\otimes(V\otimes \shF)^G$};     \node (A2_1) at (1, 2) {$\displaystyle\bigoplus_{V\in I_{G}}\duale V\otimes V\otimes \shF$};     \node (A2_2) at (2, 2) {$\shF$};     \path (A0_0) edge [->]node [auto] {$\scriptstyle{}$} (A0_1);     \path (A0_1) edge [->]node [auto] {$\scriptstyle{\varepsilon_G\otimes \id_\shF}$} (A0_2);     \path (A2_1) edge [->]node [auto] {$\scriptstyle{\bigoplus_{V\in I_{G}}e_V\otimes \id_\shF}$} (A2_2);     \path (A0_2) edge [->]node [auto] {$\scriptstyle{\id_\shF}$} (A2_2);     \path (A0_1) edge [->]node [auto] {$\scriptstyle{}$} (A2_1);     \path (A0_0) edge [->]node [auto] {$\scriptstyle{}$} (A2_0);     \path (A2_0) edge [->]node [auto] {$\scriptstyle{}$} (A2_1);   \end{tikzpicture}   \] and
we have to prove that the first vertical map is $G$-equivariant.
But this is true because in each row the composition of the maps is
a $G$-equivariant isomorphism.
\end{proof}

\subsection{Lax monoidal functors and equivariant quasi-coherent sheaves of algebras.}

In this section we want to use the association described above in
order to describe the quasi-coherent sheaves of algebras that have
an action of $G$ on it. We will see that a (non associative) ring
structure on a sheaf $\shF\in\QCoh^{G}$, its possible commutativity
and associativity translate as natural properties of the functor $\Omega^{\shF}$.
For instance we will show that a (lax) symmetric monoidal structure
over $\Omega^{\shF}$ corresponds to a structure of associative and
commutative sheaf of algebras on $\shF$.

We start setting up some definitions:
\begin{defn}
\label{def: pseudo monoidal commutativi associative}Given an $R$-scheme
$T$, a \emph{pseudo monoidal }functor $\Omega\colon\Loc^{G}R\arr\QCoh T$
is an  $R$-linear functor together with a natural transformation
\[
\iota_{V,W}^{\Omega}\colon\Omega_{V}\otimes\Omega_{W}\arr\Omega_{V\otimes W}\text{ for any }V,W\in\Loc^{G}R
\]
A pseudo monoidal functor $\Omega\colon\Loc^{G}R\arr\QCoh T$ 
\begin{enumerate}
\item \emph{\label{enu:commutative for functors}}is \emph{symmetric} (commutative)
if for any $V,W\in\Loc^{G}R$ the following diagram is commutative$$
\begin{tikzpicture}[xscale=2.7,yscale=-1.2]     \node (A0_1) at (1, 0) {$\Omega_V\otimes \Omega_W$};     \node (A0_2) at (2, 0) {$\Omega_{V\otimes W}$};     \node (A1_1) at (1, 1) {$\Omega_W\otimes \Omega_V$};     \node (A1_2) at (2, 1) {$\Omega_{W\otimes V}$};     \path (A0_1) edge [->]node [auto] {$\scriptstyle{\iota_{V,W}^\Omega}$} (A0_2);     \path (A0_2) edge [->]node [auto] {$\scriptstyle{}$} (A1_2);     \path (A1_1) edge [->]node [auto] {$\scriptstyle{\iota_{W,V}^\Omega}$} (A1_2);     \path (A0_1) edge [->]node [auto] {$\scriptstyle{}$} (A1_1);   \end{tikzpicture}
$$where the vertical arrows are the obvious isomorphisms;
\item \emph{\label{enu:associative for functors}}is \emph{associative}
if for any $V,W,Z\in\Loc^{G}R$ the following diagram is commutative
$$
 \begin{tikzpicture}[xscale=3.1,yscale=-1.2]     \node (A0_0) at (0, 0) {$\Omega_V\otimes \Omega_W\otimes \Omega_Z$};     \node (A0_1) at (1, 0) {$\Omega_{V\otimes W}\otimes \Omega_Z$};     \node (A1_0) at (0, 1) {$\Omega_V\otimes \Omega_{W\otimes Z}$};     \node (A1_1) at (1, 1) {$\Omega_{V\otimes W\otimes Z}$};     \path (A0_0) edge [->]node [auto] {$\scriptstyle{\iota_{V,W}^\Omega \otimes \id}$} (A0_1);     \path (A1_0) edge [->]node [auto] {$\scriptstyle{\iota_{V,W\otimes Z}^\Omega}$} (A1_1);     \path (A0_0) edge [->]node [auto] {$\scriptstyle{\id\otimes \iota_{W,Z}^\Omega}$} (A1_0);     \path (A0_1) edge [->]node [auto] {$\scriptstyle{\iota_{V\otimes W,Z}^\Omega}$} (A1_1);   \end{tikzpicture}
$$
\end{enumerate}
A unity for $\Omega$ is an element $1\in\Omega_{R}$ such that, for
any $V\in\Loc^{G}R$, the following diagram is commutative   \[   \begin{tikzpicture}[xscale=2.4,yscale=-0.8]     \node (A0_1) at (1, 0) {$\Omega_R \otimes \Omega_V$};     \node (A0_2) at (2, 0) {$\Omega_{R\otimes V}$};     \node (A1_0) at (0, 1) {$\Omega_V$};     \node (A1_3) at (3, 1) {$\Omega_V$};     \node (A2_1) at (1, 2) {$\Omega_V\otimes \Omega_R$};     \node (A2_2) at (2, 2) {$\Omega_{V\otimes R}$};     \path (A1_0) edge [->]node [auto,swap] {$\scriptstyle{\id\otimes 1}$} (A2_1);     \path (A1_0) edge [->]node [auto] {$\scriptstyle{1\otimes\id}$} (A0_1);     \path (A1_0) edge [->]node [auto] {$\scriptstyle{\id}$} (A1_3);     \path (A0_1) edge [->]node [auto] {$\scriptstyle{\iota_{R,V}^\Omega}$} (A0_2);     \path (A2_2) edge [->]node [auto] {$\scriptstyle{}$} (A1_3);     \path (A0_2) edge [->]node [auto] {$\scriptstyle{}$} (A1_3);     \path (A2_1) edge [->]node [auto] {$\scriptstyle{\iota_{V,R}^\Omega}$} (A2_2);   \end{tikzpicture}   \] A
\emph{lax monoidal} functor $\Omega\colon\Loc^{G}R\arr\QCoh T$ is
a pseudo monoidal functor that is associative and has a unity $1$.
\end{defn}

\begin{defn}
Given an $R$-scheme $T$ we define the categories
\begin{itemize}
\item $\QRings T$, whose objects are $\alA\in\QCoh T$ with a map $m\colon\alA\otimes\alA\arr\alA$,
called the multiplication;
\item $\QRings^{G}T$, whose objects are $\alA\in\QCoh^{G}T$ with a $G$-equivariant
map $m\colon\alA\otimes\alA\arr\alA$;
\item $\QAlg^{G}T$, whose objects are quasi-coherent sheaves $\alA\in\QRings^{G}T$
of commutative and associative algebras with a unity $1\in\alA^{G}$;
\item $\QPMon^{G}T$, whose objects are pseudo-monoidal functors $\Omega\colon\Loc^{G}R\arr\QCoh T$.
\item $\QMon^{G}T$, whose objects are commutative lax monoidal functors
$\Omega\colon\Loc^{G}R\arr\QCoh T$. Here we require that the morphisms
maintain the unities.
\item $\Aff^{G}T$, whose objects are affine schemes $X\arrdi fT$ with
an action of $G$ on $X$ such that $f$ is $G$-invariant;
\item $\ffpSch^{G}T$ , the full subcategory of $\Aff^{G}T$ of finite and
finitely presented maps;
\item $\Cov^{G}$, the full subcategory of $\Aff^{G}T$ of covers.
\end{itemize}
Replacing $\QCoh$ with $\Loc$, $\FCoh$ we also define $\LRings T$,
$\LRings^{G}T$, $\LAlg^{G}T$, $\LPMon^{G}T$, $\LMon^{G}T$ and
$\CRings T$, $\CRings^{G}T$, $\CAlg^{G}T$, $\CPMon^{G}T$, $\CMon^{G}T$
respectively

We define the stacks $\textup{HRings}_{R}$, $\textup{HRings}_{R}^{G}$,
$\textup{HAlg}_{R}^{G}$, $\textup{HPMon}_{R}^{G}$, $\textup{HMon}_{R}^{G}$
whose fibers over an $R$-scheme $T$ is $\textup{HRings}T$, $\textup{HRings}^{G}T$,
$\textup{HAlg}^{G}T$, $\textup{HPMon}^{G}T$, $\textup{HMon}^{G}T$
respectively, where $H$ can be $\textup{Q}$, $\textup{C}$ or $\textup{L}$.
We also define $\Aff_{R}^{G}$, $\ffpSch_{R}^{G}$, $\Cov_{R}^{G}$
as the stacks whose fibers over an $R$-scheme $T$ are respectively
$\Aff^{G}T$, $\ffpSch^{G}T$, $\Cov^{G}T$.\end{defn}
\begin{rem}
The functors $\Spec\colon\QAlg_{R}^{G}\arr\Aff_{R}^{G}$ and the push
forward $\Aff_{R}^{G}\arr\QAlg_{R}^{G}$ are each other's quasi-inverse
and restrict to isomorphisms $\CAlg_{R}^{G}\simeq\ffpSch_{R}^{G}$
and $\LAlg_{R}^{G}\simeq\Cov^{G}$. Indeed, by \cite[Proposition 1.47]{EGAIV-1},
a finite quasi-coherent algebra is finitely presented as a module
if and only if it is so as an algebra.
\end{rem}

\begin{rem}
The categories $\QAlg^{G}T$, $\QMon^{G}T$ are (not full) subcategories
of $\QRings^{G}T$, $\QPMon^{G}T$ respectively, because a unity for
a ring or for a lax monoidal functor is unique.
\end{rem}
The following is another application of \ref{lem:equivalence among Coh T and additive functors Loc R to Coh T}.
\begin{lem}
\label{lem:monoidality respects triviality}Given $\Omega\in\QPMon_{R}^{G}$
and $V,W\in\Loc R,$ $V',W'\in\Loc^{G}R$ we have a commutative diagram
  \[   \begin{tikzpicture}[xscale=4.5,yscale=-1.2]     \node (A0_0) at (0, 0) {$\Omega_{\underline V \otimes V'}\otimes \Omega_{\underline W \otimes W'}$};     
\node (A0_1) at (1, 0) {$V \otimes \Omega_{V'} \otimes W \otimes \Omega_{W'}$};     
\node (A0_2) at (2, 0) {$ V \otimes  W\otimes\Omega_{V'}  \otimes \Omega_{W'}$};     
\node (A1_0) at (0, 1) {$\Omega_{\underline V \otimes V' \otimes \underline W \otimes W'}$};     \node (A1_1) at (1, 1) {$\Omega_{\underline V \otimes \underline W \otimes V'  \otimes W'}$};     \node (A1_2) at (2, 1) {$ V \otimes  W \otimes\Omega_{ V'  \otimes W'}$};     \path (A0_0) edge [->]node [auto] {$\scriptstyle{}$} (A0_1);     \path (A1_0) edge [->]node [auto] {$\scriptstyle{}$} (A1_1);     \path (A0_2) edge [->]node [auto] {$\scriptstyle{\id\otimes\iota_{V',W'}}$} (A1_2);     \path (A1_1) edge [->]node [auto] {$\scriptstyle{}$} (A1_2);     \path (A0_0) edge [->]node [auto] {$\scriptstyle{\iota_{\underline V\otimes V',\underline W\otimes W'}}$} (A1_0);     \path (A0_1) edge [->]node [auto] {$\scriptstyle{}$} (A0_2);   \end{tikzpicture}   \] 
\end{lem}
The following proposition describes how much data is needed to define
a pseudo monoidal functor when the group $G$ is a glrg.
\begin{prop}
\label{prop:essential data for a pseudo monoidal functor}Assume that
$G$ is glrg and define the stack $\stY$ whose objects are $(\shA_{V},\iota_{V,W})_{V,W\in I_{G}}$
where $\alA_{V}\in\QCoh_{R}$ and $\iota_{V,W}$ is a map 
\[
\iota_{V,W}\colon\alA_{V}\otimes\alA_{W}\arr\bigoplus_{\Delta\in I_{G}}\Hom^{G}(\Delta,V\otimes W)\otimes\alA_{\Delta}
\]
Then the functor   \[   \begin{tikzpicture}[xscale=6.8,yscale=-0.6]     \node (A0_0) at (0, 0) {$\QPMon^G_R$};     \node (A0_1) at (1, 0) {$\stY$};     \node (A1_0) at (0, 1) {$(\Omega,\iota^\Omega)$};     \node (A1_1) at (1, 1) {$\displaystyle(\Omega_V,\Omega_V\otimes\Omega_W\arrdi{\iota^\Omega_{V,W}}\Omega_{V\otimes W}\simeq\bigoplus_{\Delta\in I_G}\Hom^G(\Delta,V\otimes W)\otimes \Omega_\Delta)_{V,W\in I_G})$};     \path (A0_0) edge [->]node [auto] {$\scriptstyle{}$} (A0_1);     \path (A1_0) edge [|->,gray]node [auto] {$\scriptstyle{}$} (A1_1);   \end{tikzpicture}   \] 
is an equivalence.\end{prop}
\begin{proof}
By \ref{rem: additive functors uniquely determined on irreducible representations}
and \ref{lem:monoidality respects triviality}, we see that the map
in the statement is fully faithful. We have only to show that it is
essentially surjective. For simplicity, given $\Delta,W\in\Loc^{G}R$
we will write $W_{\Delta}=\Hom^{G}(\Delta,W)$. Let $\chi=(\alA_{V},\iota_{V,W})_{V,W\in I_{G}}\in\stY$.
By \ref{rem: additive functors uniquely determined on irreducible representations},
there exists $\Omega\in\QAdd_{R}^{G}$ such that $\Omega_{V}\simeq\alA_{V}$
and it is given by 
\[
\Omega_{W}=\bigoplus_{\Delta\in I_{G}}W_{\Delta}\otimes\alA_{\Delta}
\]
By definition, the map $\iota_{V,W}$ yield maps $\iota_{V,W}\colon\Omega_{V}\otimes\Omega_{W}\arr\Omega_{V\otimes W}$
for any $V,W\in I_{G}$. Given $\Lambda,\Gamma\in\Loc^{G}R$ we define
$\iota_{\Lambda,\Delta}^{\Omega}$ as
\[
\Omega_{\Lambda}\otimes\Omega_{\Gamma}\simeq\bigoplus_{V,W\in I_{G}}\Lambda_{V}\otimes\Gamma_{W}\otimes\Omega_{V}\otimes\Omega_{W}\arrdi{\id\otimes\iota_{V,W}}\bigoplus_{V,W\in I_{G}}\Lambda_{V}\otimes\Gamma_{W}\otimes\Omega_{V\otimes W}\simeq\Omega_{\Lambda\otimes\Gamma}
\]
where the last isomorphism is induced by $\bigoplus_{V,W\in I_{G}}\Lambda_{V}\otimes\Gamma_{W}\otimes V\otimes W\simeq\Lambda\otimes\Gamma$.
It is easy to check that $\iota^{\Omega}$ is a natural transformation
and that $(\Omega,\iota^{\Omega})\in\LPMon_{R}^{G}$ is mapped to
our starting object $\chi\in\stY$.
\end{proof}
Given $\alA\in\QRings^{G}T$ with multiplication $m$, we endow $\Omega^{\alA}$
with the pseudo monoidal structure 
\[
(V\otimes\alA)^{G}\otimes(W\otimes\alA)^{G}\arr(V\otimes W\otimes\alA\otimes\alA)^{G}\arrdi{(\id\otimes m)^{G}}(V\otimes W\otimes\alA)^{G}
\]
Conversely given $\Omega\in\QPMon^{G}T$, we define the multiplication
on $\shF_{\Omega}=\Omega_{R[G]}$ by 
\[
\Omega_{R[G]}\otimes\Omega_{R[G]}\arrdi{\iota_{R[G],R[G]}^{\Omega}}\Omega_{R[G]\otimes R[G]}\arrdi{\Omega_{m_{G}}}\Omega_{R[G]}
\]
We will denote $\alA_{\Omega}$ the sheaf $\shF_{\Omega}$ together
with the multiplication map.

The following is the main Theorem of this section.
\begin{thm}
\label{thm:G equivariant ring are monoidal functors}Given an $R$-scheme
$T$, the functors $\Omega^{*}$ and $\shF_{*}$ of Theorem \ref{thm:additive functors are equivariant sheaves}
extend to functors   \[   \begin{tikzpicture}[xscale=2,yscale=-0.2]     
\node (A0_0) at (0, 0) {};     
\node (A0_1) at (1, 0) {};     
\node (A1_0) at (0, 1) {$\QRings^G T \ \ \ \ \ \ \ \ \ \ \ \ \ \ \ \ \ $};     
\node (A1_1) at (1, 1) {$\ \ \ \ \ \ \ \ \ \ \ \ \ \ \ \ \ \ \QPMon^G T$};     
\node (A2_0) at (0, 2) {};     
\node (A2_1) at (1, 2) {};     

\node (aA0_0) at (3, 0) {};     
\node (aA0_1) at (4, 0) {};     
\node (aA1_0) at (3, 1) {$\QAlg^G T \ \ \ \ \ \ \ \ \ \ \ \ \ \ \ \ $};     
\node (aA1_1) at (4, 1) {$\ \ \ \ \ \ \ \ \ \ \ \ \ \ \ \ \ \QMon^G T$};     
\node (aA2_0) at (3, 2) {};     
\node (aA2_1) at (4, 2) {};   

\path (A0_0) edge [->]node [auto] {$\scriptstyle{\Omega^*}$} (A0_1);     
\path (A2_1) edge [->]node [auto] {$\scriptstyle{\alA_*}$} (A2_0);   

\path (aA0_0) edge [->]node [auto] {$\scriptstyle{\Omega^*}$} (aA0_1);     
\path (aA2_1) edge [->]node [auto] {$\scriptstyle{\alA_*}$} (aA2_0);   
\end{tikzpicture}   \] Moreover there exist a natural isomorphism $\alA\arr\Omega_{R[G]}^{\alA}$
and a natural transformation $\Omega\arr\Omega^{\Omega_{R[G]}}$ which
is an isomorphism if and only if $\Omega$ is left exact. In particular
$\Omega^{*}$ is an equivalence onto the full subcategory of $\QPMon^{G}T$
($\QMon^{G}T)$ of left exact functors. If $G$ is linearly reductive
the above functors define isomorphisms of stacks
\[
\QRings_{R}^{G}\simeq\QPMon_{R}^{G}\comma\CRings_{R}^{G}\simeq\CPMon_{R}^{G},\;\LRings_{R}^{G}\simeq\LPMon_{R}^{G}
\]
\[
\QAlg_{R}^{G}\simeq\QMon_{R}^{G}\comma\CAlg_{R}^{G}\simeq\CMon_{R}^{G},\;\LAlg_{R}^{G}\simeq\LMon_{R}^{G}
\]

\end{thm}
This theorem will be proved at the end of this section, because we
need to collect several lemmas before.
\begin{rem}
Given an $R$-scheme $T$ we have a functor   \[   \begin{tikzpicture}[xscale=4.0,yscale=-0.6]     \node (A0_0) at (0, 0) {$\QPMon^G T \times\LRings^G R$};     \node (A0_1) at (1, 0) {$\QRings T$};     \node (A1_0) at (0, 1) {$(\Omega,(A,m))$};     \node (A1_1) at (1, 1) {$(\Omega_A,\Omega_m \circ \iota_{A, A}^\Omega)$};     \path (A0_0) edge [->]node [auto] {$\scriptstyle{}$} (A0_1);     \path (A1_0) edge [|->,gray]node [auto] {$\scriptstyle{}$} (A1_1);   \end{tikzpicture}   \] 
\end{rem}
The following lemma shows that the functor $\alA_{*}\colon\QAdd^{G}T\arr\QRings^{G}T$
is well defined.
\begin{lem}
\label{lem:multiplication of AE is G equivariant}If $\Omega\in\QPMon_{R}^{G}$
then $\alA_{\Omega}=\Omega_{R[G]}\in\QRings_{R}^{G}$, i.e. the multiplication
$\alA_{\Omega}\otimes\alA_{\Omega}\arr\alA_{\Omega}$ is $G$-equivariant.\end{lem}
\begin{proof}
Set $A=R[G]$ , $\Delta=\Delta_{G}\colon A\arr A\otimes R[G]$ and
$m=m_{G}\colon A\otimes A\arr A$. We claim that the diagrams in   \[   \begin{tikzpicture}[xscale=3.0,yscale=-1.2]     
\node (A0_0) at (0, 0) {$\Omega_A\otimes\Omega_A$};     
\node (A0_2) at (2, 0) {$\Omega_{A\otimes A}$};     
\node (A0_3) at (3, 0) {$\Omega_A$};     
\node (A1_0) at (0, 1) {$\Omega_{A\otimes \underline{R[G]}}\otimes \Omega_{A\otimes \underline{R[G]}}$};     
\node (A1_2) at (2, 1) {$\Omega_{A\otimes \underline{R[G]}\otimes A\otimes \underline{R[G]}}$};     
\node (A2_0) at (0, 2) {$\Omega_A\otimes\Omega_A \otimes R[G]\otimes R[G]$};     
\node (A2_2) at (2, 2) {$\Omega_{A\otimes A\otimes \underline{R[G]\otimes  R[G]}}$};     
\node (A3_1) at (1, 3) {$\Omega_{A\otimes A} \otimes R[G]\otimes R[G]$};     
\node (A3_2) at (2, 3) {$\Omega_{A\otimes A\otimes \underline{R[G]}}$};     
\node (A3_3) at (3, 3) {$\Omega_{A \otimes \underline{R[G]}}$};     
\node (A4_0) at (0, 4) {$\Omega_A\otimes\Omega_A \otimes R[G]$};     
\node (A4_1) at (1, 4) {$\Omega_{A\otimes A} \otimes R[G]$};     
\node (A4_3) at (3, 4) {$\Omega_A \otimes R[G]$};     \path (A4_0) edge [->]node [auto,swap] {$\scriptstyle{\iota_{A,A}\otimes \id}$} (A4_1);     \path (A2_2) edge [->]node [auto] {$\scriptstyle{}$} (A3_1);     \path (A3_2) edge [->]node [auto] {$\scriptstyle{}$} (A4_1);     \path (A1_0) edge [->]node [auto] {$\scriptstyle{\iota_{A\otimes R[G],A\otimes R[G]}}$} (A1_2);     \path (A2_0) edge [->]node [auto] {$\scriptstyle{\iota_{A,A}\otimes \id}$} (A3_1);     \path (A3_1) edge [->]node [auto,swap] {$\scriptstyle{\id\otimes m_G}$} (A4_1);     \path (A4_1) edge [->]node [auto] {$\scriptstyle{\Omega_m \otimes \id}$} (A4_3);     \path (A0_2) edge [->]node [auto] {$\scriptstyle{\Omega_{\Delta\otimes\Delta}}$} (A1_2);     \path (A2_0) edge [->]node [auto,swap] {$\scriptstyle{\id\otimes m_G}$} (A4_0);     \path (A0_3) edge [->]node [auto] {$\scriptstyle{\Omega_\Delta}$} (A3_3);     \path (A2_2) edge [->]node [auto] {$\scriptstyle{\Omega_{\id\otimes m_G}}$} (A3_2);     \path (A1_0) edge [->]node [auto] {$\scriptstyle{}$} (A2_0);     \path (A0_0) edge [->]node [auto] {$\scriptstyle{\iota_{A,A}}$} (A0_2);     \path (A3_3) edge [->]node [auto] {$\scriptstyle{}$} (A4_3);     \path (A0_0) edge [->]node [auto] {$\scriptstyle{\Omega_\Delta\otimes \Omega_\Delta}$} (A1_0);     \path (A0_2) edge [->]node [auto] {$\scriptstyle{\Omega_m}$} (A0_3);     \path (A3_2) edge [->]node [auto] {$\scriptstyle{\Omega_{m\otimes \id}}$} (A3_3);     \path (A1_2) edge [->]node [auto] {$\scriptstyle{}$} (A2_2);   \end{tikzpicture}   \] are commutative. Note that the outer diagram is the one required for
the $G$-equivariancy of the multiplication $\Omega_{A}\otimes\Omega_{A}\arr\Omega_{A}$.
The pentagonal diagram is commutative thanks to \ref{lem:monoidality respects triviality}.
The only non trivially commutative diagram left is the upper right
rectangle. This is commutative because it is obtained applying $\Omega$
to the diagram   \[   \begin{tikzpicture}[xscale=3.9,yscale=-1.2]     \node (A0_0) at (0, 0) {$A\otimes A$};     \node (A0_1) at (1, 0) {$A\otimes \underline{R[G]}\otimes A \otimes \underline{R[G]}$};     \node (A0_2) at (2, 0) {$A\otimes A\otimes \underline{R[G]\otimes R[G]}$};     \node (A1_0) at (0, 1) {$A$};     \node (A1_2) at (2, 1) {$A\otimes \underline{R[G]}$};     \path (A0_0) edge [->]node [auto] {$\scriptstyle{\Delta\otimes\Delta}$} (A0_1);     \path (A1_0) edge [->]node [auto] {$\scriptstyle{\Delta}$} (A1_2);     \path (A0_0) edge [->]node [auto] {$\scriptstyle{m}$} (A1_0);     \path (A0_1) edge [->]node [auto] {$\scriptstyle{}$} (A0_2);     \path (A0_2) edge [->]node [auto] {$\scriptstyle{m\otimes m_G}$} (A1_2);   \end{tikzpicture}   \] which
is commutative since $\Delta$ is a map of rings.
\end{proof}
We have now to deal with how the properties of being commutative,
associative or having a unity translate in the context of functors.
\begin{rem}
\label{lem:characterization of muliplication by invariant hom}If
$\Omega\in\QPMon^{G}T$ and $V,W\in\Loc^{G}R$ we have a commutative
diagram   \[   \begin{tikzpicture}[xscale=5,yscale=-1.2]     
\node (A0_0) at (0, 0) {$\duale V\otimes \Omega_V\otimes \duale W\otimes \Omega_W$};     
\node (A0_1) at (1, 0) {$\Omega_{R[G]}\otimes\Omega_{R[G]}$};     
\node (A1_0) at (0, 1) {$\duale{V\otimes W} \otimes \Omega_{V\otimes W}$};     
\node (A1_1) at (1, 1) {$\Omega_{R[G]}$};     
\path (A0_0) edge [->]node [auto] {$\scriptstyle{\theta_V\otimes \theta_W}$} (A0_1);     
\path (A0_0) edge [->]node [auto] {$\scriptstyle{}$} (A1_0);     
\path (A0_1) edge [->]node [auto] {$\scriptstyle{m}$} (A1_1);     
\path (A1_0) edge [->]node [auto] {$\scriptstyle{\theta_{V\otimes W}}$} (A1_1);   \end{tikzpicture}   \] where $\theta_{*}$ are the evaluation maps defined in \ref{prop:natural map Omega in OmegaF}
and $m$ is the multiplication.
\end{rem}

\begin{rem}
\label{lem:trivial torsor functor}Given an $R$-scheme $T$, the
natural isomorphisms (see \ref{prop:the structure map are the invariants})
\[
V\otimes\odi T\simeq(V\otimes\odi T[G])^{G}\simeq\Homsh^{G}(\duale V,\odi T[G])\text{ for }V\in\Loc^{G}R
\]
are monoidal. 
\end{rem}
The following lemmas show that the functors $\alA_{*}$ and $\Omega^{*}$
are well defined on $\QMon^{G}T$ and $\QAlg^{G}T$ respectively.
\begin{lem}
\label{lem:Associativity for V,W,Z}Let $\Omega\in\QPMon^{G}T$, set
$\alA=\Omega_{R[G]}$ with multiplication $m$ and let $V,W,Z\in\Loc^{G}R$.
Using notations from \ref{prop:natural map Omega in OmegaF}, the
commutativity of the diagram \ref{enu:commutative for functors} (resp.
\ref{enu:associative for functors}) in definition \ref{def: pseudo monoidal commutativi associative}
implies the commutativity between $x,y$ (resp. associativity among
$x,y,z$) for sections $x\in\Imm\theta_{V}$, $y\in\Imm\theta_{W}$
(resp. and $z\in\Imm\theta_{Z}$). In particular if $\Omega$ is symmetric
(associative) then $\alA_{\Omega}$ is commutative (associative).
The converses to the previous statements hold if $\Omega$ is left
exact.\end{lem}
\begin{proof}
Denote by $\iota_{V,W}\colon\Omega_{V}\otimes\Omega_{W}\arr\Omega_{V\otimes W}$
the monoidal structure on $\Omega$ and by $\text{ex}\colon A\otimes B\arr B\otimes A$
the exchange map. Let $v\in\Omega_{V}$, $w\in\Omega_{W}$, $z\in\Omega_{Z}$
and set also
\[
\xi=\iota_{V,W\otimes Z}(v\otimes\iota_{W,Z}(w\otimes z))\comma\zeta=\iota_{V\otimes W,Z}(\iota_{V,W}(v\otimes w)\otimes z)
\]
\[
\eta=\Omega_{\text{ex}}(\iota_{V,W}(v\otimes w))\comma\mu=\iota_{W,V}(w\otimes v)
\]
If $\alpha\in\duale V,\beta\in\duale W\comma\gamma\in\duale Z$ set
\[
x=\theta_{V}(v\otimes\alpha)(\theta_{W}(w\otimes\beta)\theta(z\otimes\gamma))\comma y=(\theta_{V}(v\otimes\alpha)\theta_{W}(w\otimes\beta))\theta(z\otimes\gamma)
\]
\[
a=\theta_{V}(v\otimes\alpha)\theta_{W}(w\otimes\beta)\comma b=\theta_{W}(w\otimes\beta)\theta_{V}(v\otimes\alpha)
\]
Thanks to \ref{lem:characterization of muliplication by invariant hom},
we see that
\[
x=\theta_{V\otimes W\otimes Z}(\xi\otimes\alpha\otimes\beta\otimes\gamma)\comma y=\theta_{V\otimes W\otimes Z}(\zeta\otimes\alpha\otimes\beta\otimes\gamma)
\]
\[
a=\theta_{W\otimes V}(\eta\otimes\beta\otimes\alpha)\comma b=\theta_{W\otimes V}(\mu\otimes\beta\otimes\alpha)
\]
The commutativity of the diagrams \ref{enu:commutative for functors}
and \ref{enu:associative for functors} coincide with the equalities
$\eta=\mu$ and $\xi=\zeta$ for any $v,w,z$ respectively, which
imply the equalities $a=b$ and $x=y$ for any $v,w,z,\alpha,\beta,\gamma$.
So the first claim holds. For the converse, it is enough to show that
if $U\in\Loc^{G}R$ and $u\in\Omega_{U}$ then
\[
\theta_{U}(u\otimes\delta)=0\;\forall\delta\in\duale U\then u=0
\]
But this is the injectivity of the induced map $\Omega_{U}\arr\Omega_{R[G]}\otimes U$,
which comes from \ref{prop:natural map Omega in OmegaF} since $\Omega$
if left exact.

For the last claims, it is enough to note that $\theta_{R[G]}$ is
surjective. Indeed taking the element $\phi\in\duale{R[G]}$ corresponding
to $\id_{R[G]}\in\End^{G}R[G]$ we have
\[
\theta_{R[G]}(x\otimes\phi)=x\text{ for }x\in\Omega_{R[G]}
\]
thanks to \ref{prop:natural map Omega in OmegaF}.\end{proof}
\begin{lem}
\label{lem:the natural transformation between monoidal functor is monoidal}If
$\Omega\in\QPMon^{G}T$, then the natural transformation $\Omega\arr\Omega^{\alA_{\Omega}}$
defined in \ref{prop:natural map Omega in OmegaF} is monoidal.\end{lem}
\begin{proof}
Set $\Gamma=\Omega^{\alA_{\Omega}}$. The monoidality of the map $\Omega\arr\Gamma$
is expressed by the equalities of the two maps 
\[
\Omega_{V}\otimes\Omega_{W}\arr\Gamma_{V}\otimes\Gamma_{W}\arr\Gamma_{V\otimes W}\comma\Omega_{V}\otimes\Omega_{W}\arr\Omega_{V\otimes W}\arr\Gamma_{V\otimes W}
\]
for any $V,W\in\Loc^{G}R$. Since $-\otimes U\colon\Loc^{G}R\arr\Loc^{G}R$
is an exact functor for any $U\in\Loc^{G}R$, by \ref{lem:defining natural transformations on RG}
we have only to check that the above maps coincide when $V=W=R[G]$.
In this case, by applying $\Omega$ and thanks to \ref{lem:monoidality respects triviality},
we reduce to the problem of the commutativity of the following diagram
  \[   \begin{tikzpicture}[xscale=4.3,yscale=-1.2]     \node (A0_0) at (0, 0) {$A\otimes A$};     \node (A0_1) at (1, 0) {$A\otimes \underline A \otimes A \otimes \underline A\simeq A\otimes A\otimes \underline A\otimes \underline A$};     \node (A0_2) at (2, 0) {$A \otimes \underline A\otimes \underline A$};     \node (A1_0) at (0, 1) {$A\otimes \underline A$};     \node (A1_2) at (2, 1) {$A\otimes \underline A\otimes \underline A$};     \path (A0_0) edge [->]node [auto] {$\scriptstyle{\Delta_G\otimes\Delta_G}$} (A0_1);     \path (A1_0) edge [->]node [auto] {$\scriptstyle{\Delta_G\otimes\id}$} (A1_2);     \path (A0_0) edge [->]node [auto] {$\scriptstyle{\omega}$} (A1_0);     \path (A0_1) edge [->]node [auto] {$\scriptstyle{m_A\otimes\id}$} (A0_2);     \path (A0_2) edge [->]node [auto] {$\scriptstyle{\id\otimes\omega}$} (A1_2);   \end{tikzpicture}   \] where
$A=R[G]$ and $\omega\colon A\otimes A\arrdi{\simeq}A\otimes\underline{A}$
is the tensor product of $\Delta_{G}$ and $\id\otimes1$. But the
commutativity of such diagram can be checked directly taking spectra
and using the functorial point of view.\end{proof}
\begin{lem}
\label{lem:preservation of unities}Let $\Omega\in\QPMon^{G}T$ and
$1\in\Omega_{R}$. If $1$ is a unity for $\Omega$ then it is also
a unity for $\Omega_{R[G]}$. The converse holds is $\Omega$ is left
exact.\end{lem}
\begin{proof}
If $1$ is a unity for $\Omega$ we will have a commutative diagram
  \[   \begin{tikzpicture}[xscale=2.6,yscale=-0.5]     \node (A0_1) at (1, 0) {$\Omega_R \otimes \Omega_{R[G]}$};     \node (A0_2) at (2, 0) {$\Omega_{R\otimes R[G]}$};     \node (A1_0) at (0, 1) {$\Omega_{R[G]}$};     \node (A1_3) at (3, 1) {$\Omega_{R[G]}$};     \node (A2_1) at (1, 2) {$\Omega_{R[G]}\otimes \Omega_{R[G]}$};     \node (A2_2) at (2, 2) {$\Omega_{R[G]\otimes R[G]}$};     \path (A1_0) edge [->]node [auto] {$\scriptstyle{}$} (A2_1);     \path (A0_1) edge [->]node [auto] {$\scriptstyle{}$} (A0_2);     \path (A0_1) edge [->]node [auto] {$\scriptstyle{}$} (A2_1);     \path (A2_2) edge [->]node [auto] {$\scriptstyle{}$} (A1_3);     \path (A0_2) edge [->]node [auto] {$\scriptstyle{}$} (A2_2);     \path (A0_2) edge [->]node [auto] {$\scriptstyle{}$} (A1_3);     \path (A1_0) edge [->]node [auto] {$\scriptstyle{1\otimes \id}$} (A0_1);     \path (A2_1) edge [->]node [auto] {$\scriptstyle{}$} (A2_2);   \end{tikzpicture}   \] 
and so $1$ is a left unity for $\Omega_{R[G]}$. Similarly it is
also a right unity. Conversely, if $\Omega$ is left exact, the result
follows easily because the isomorphism $ $$\Omega_{V}\simeq\Homsh^{G}(\duale V,\Omega_{R[G]})$
is monoidal thanks to \ref{lem:the natural transformation between monoidal functor is monoidal}.
\end{proof}
We are finally ready to prove Theorem \ref{thm:G equivariant ring are monoidal functors}\emph{.}
\begin{proof}
(\emph{of Theorem }\ref{thm:G equivariant ring are monoidal functors}\emph{)
}The functors of \ref{thm:additive functors are equivariant sheaves}
are well defined over $\QRings^{G}T$ and $\QPMon^{G}T$ thanks to
\ref{lem:multiplication of AE is G equivariant}. We claim that they
are well defined also over $\QAlg^{G}T$ and $\QMon^{G}T$. For the
unities, using their uniqueness, it is enough to apply \ref{lem:preservation of unities}
and note that the natural transformation $\Omega\arr\Omega^{\alA_{\Omega}}$
over $R\in\Loc^{G}R$ is just $\Omega$ applied to the inclusion $R\arr R[G]$.
Associativity and commutativity instead come from \ref{lem:Associativity for V,W,Z}.

The natural transformation $\Omega\arr\Omega^{\alA_{\Omega}}$ is
the one defined in \ref{prop:natural map Omega in OmegaF}, which
is monoidal thanks to \ref{lem:the natural transformation between monoidal functor is monoidal}.
It remains to prove that if $\alA\in\QRings^{G}T$ then the comodule
map $\alA\arr\Omega_{R[G]}^{\alA}=(\alA\otimes R[G])^{G}$ is a map
of rings. But the commutative diagram expressing this fact is exactly
the diagram expressing the $G$-equivariance of the multiplication
$\alA\otimes\alA\arr\alA$, since, by construction, the ring structure
on $(\alA\otimes R[G])^{G}$ is the one as subring of $\alA\otimes R[G]$.
\end{proof}

\subsection{Ramified Galois covers and the forgetful functor.}

We have seen that a quasi-coherent sheaf of (commutative, associative
and with unity) algebras, or, equivalently, an affine map, with an
action of $G$ corresponds to a monoidal functor. In this subsection
we want to describe the functors associated with $G$-covers. The
theorem we want to prove is the following.
\begin{thm}
\label{thm:description of GCov with monoidal functors}The map of
stacks   \[   \begin{tikzpicture}[xscale=2.6,yscale=-0.6]     
\node (A0_0) at (0, 0) {$\GCov$};     
\node (A0_1) at (1, 0) {$\LMon^G_R$};     
\node (A1_00) at (0, 0.87) {$X\arrdi{f}T$};     
\node (A1_0) at (0, 1) {$\ \ \ \ \ \ \ \ \ \ \ \ \ $};     
\node (A1_1) at (1, 1) {$\Omega^{f_*\odi X}$};     
\path (A0_0) edge [->]node [auto] {$\scriptstyle{}$} (A0_1);     \path (A1_0) edge [|->,gray]node [auto] {$\scriptstyle{}$} (A1_1);   \end{tikzpicture}   \] is well defined and yields an isomorphism between $\GCov$ and the
substack in groupoids of $\LMon_{R}^{G}$ of functors $\Omega$ that,
in $\LAdd_{R}^{G}$, are fppf locally isomorphic to the forgetful
functor. When $G$ is a glrg over a connected scheme, this is also
the substack in groupoids of $\Omega\in\LMon_{R}^{G}$ such that $\rk\Omega_{V}=\rk V$
for all the representations $V\in I_{G}$.\end{thm}
\begin{rem}
In general it is not true that, if $\alA\in\LAlg_{R}^{G}T$ is such
that $\rk\Omega_{V}^{\alA}=\rk V$ for all $V\in\Loc^{G}R$, then
$\alA\in\GCov(T)$, even for linearly reductive groups. A counterexample
with $G=\Z/3\Z$, $R=\Q$ and $T=\Spec k$, where $k=\overline{\Q}$,
is $\alA=k[x,y]/(x,y)^{2}$ with the action of $\mu_{3}\simeq G\times k$
given by graduation $\deg x=\deg y=1\in\Z/3\Z$. Denote by $k_{i}$
the irreducible $\mu_{3}$-representation over $k$ induced by $i\in\Z/3\Z=\Hom(\mu_{3},\Gm)$.
Note that $\Z/3\Z$ has only one non trivial irreducible representation
$W$ over $\Q$ and it satisfies $W\otimes k=k_{1}\oplus k_{2}$.
Therefore $\Z/3\Z$ is not a glrg. The functor $\delta\colon\Loc^{\mu_{3}}k\arr\Loc k$
associated with $\alA\in\LAlg^{\mu_{3}}k$ is simply given by $\delta_{k_{0}}=k$,
$\delta_{k_{1}}=k^{2}$, $\delta_{k_{2}}=0$. In particular $\alA\notin\GCov(k)$
by \ref{thm:description of GCov with monoidal functors}. $ $ On
the other hand, since $G$-representations over $\Q$ decompose into
irreducible representations, it is easy to check that $\Omega^{\alA}\colon\Loc^{G}\Q\arr\Loc k$,
which is nothing else that $\Omega_{V}^{\alA}=\delta_{V\otimes k}$,
satisfies $\rk\Omega_{V}^{\alA}=\rk V$ for all $V\in\Loc^{G}\Q$.
\end{rem}

\begin{rem}
\label{rem:regular representation and forgetful functor}Thanks to
\ref{prop:the structure map are the invariants}, the functor $\Omega^{\odi T[G]}\in\QAdd^{G}T$
associated with the regular representation is just the forgetful functor
\[
\Loc^{G}R\ni V\longmapsto V\otimes\odi T\in\Loc T
\]
\end{rem}
\begin{proof}
(\emph{of Theorem }\ref{thm:description of GCov with monoidal functors})
We will make use of \ref{thm:G equivariant ring are monoidal functors}.
Let $X\arrdi fT\in\GCov(T)$ and set $\alA=f_{*}\odi X\in\LAlg^{G}T$.
Since $\Omega^{\odi T[G]}$ is the forgetful functor and taking invariants
behaves well under flat base changes, we have that $\Omega^{\alA}\colon\Loc^{G}R\arr\QCoh T$
is fppf locally the forgetful functor. In particular $\Omega^{\alA}$
is exact and $\Omega^{\alA}\in\LMon^{G}T$, that is $\Omega^{\alA}$
has image in $\Loc T$. If $T'\arrdi hT$ is any base change, then
$h^{*}\circ\Omega^{\alA}$ is still exact because exact sequences
in $\Loc T$ split locally, and therefore
\[
h^{*}\circ\Omega^{\alA}\simeq\Omega^{h^{*}\Omega_{R[G]}^{\alA}}\simeq\Omega^{h^{*}\alA}
\]
So the map in the statement is well defined and the first equivalence
is clear from \ref{thm:additive functors are equivariant sheaves}.
Assume now that $G$ is a glrg. We have to show that $\Omega\in\LAdd_{R}^{G}$
is locally the forgetful functor if and only if $\rk\Omega_{V}=\rk V$
for all $V\in I_{G}$. This is clear from \ref{prop:equivariant coherent sheaves are collection of coherent sheves for glrg}
and \ref{prop:decomposition of OG}. 
\end{proof}

\subsection{Strong monoidal functors and $G$-torsors.}

After the description of $G$-covers in terms of functors, it arises
naturally the question of what kind of functors correspond to $G$-torsors.
We will show that the answer is strong monoidal functors. Notice that,
over a field, this is a classical result of the Tannakian theory (see
\cite{Deligne1982,Rivano1972}). Moreover such a result has already
been proved in \cite{Lurie2004}, as a particular case of a more general
theory. In this subsection we want to give a more elementary proof,
based on the results obtained in the previous sections.

Notice also that the equivalence between $G$-torsors and strong monoidal
functors, in the diagonalizable case, is another well known result
(see \ref{pro:equivalent conditions for a D(M)-torsor}) and it does
not require the machinery developed here or the Tannakian theory.

Thanks to \ref{lem:trivial torsor functor} and \ref{rem:regular representation and forgetful functor},
we have a description of the trivial $G$-torsor:
\begin{prop}
\label{prop:functor of trivial torsor}The functor associated with
the trivial $G$-torsor $\odi T[G]$ is just the forgetful functor
\[
\Loc^{G}R\ni V\longmapsto V\otimes\odi T\in\Loc T
\]
with the usual monoidal structure.
\end{prop}
For general $G$-torsors we need the following definition.
\begin{defn}
Given an $R$-scheme define $\LSMon^{G}T$ as the full subcategory
of $\LMon^{G}T$ of objects $\Omega$ that are left exact, \emph{strong}
monoidal, i.e. such that for any $V,W\in\Loc^{G}R$ the map
\[
\iota_{V,W}^{\Omega}\colon\Omega_{V}\otimes\Omega_{W}\arr\Omega_{V\otimes W}
\]
is an isomorphism, and such that the map $\odi T\arr\Omega_{R}$ is
injective. Define also $\LSMon_{R}^{G}$ as the full subcategory of
$\LMon_{R}^{G}$ whose fibers over an $R$-scheme $T$ are $\LSMon^{G}T$.\end{defn}
\begin{thm}
\label{thm:torsors are isomorphism} $\LSMon_{R}^{G}$ is a substack
of $\LMon_{R}^{G}$ and the functors   \[   \begin{tikzpicture}[xscale=3.7,yscale=-0.6]     \node (A0_0) at (0, 0) {$\Spec \Omega_{R[G]}$};     \node (A0_1) at (1, 0) {$\Omega$};     \node (A1_0) at (0, 1) {$\Bi_R G$};     \node (A1_1) at (1, 1) {$\LSMon^G_R$};     \node (A2_0) at (0, 2) {$ X\arrdi{f} T$};     \node (A2_1) at (1, 2) {$ (- \otimes f_*\odi{X})^G$};     \path (A1_0) edge [->]node [auto] {$\scriptstyle{}$} (A1_1);     \path (A2_0) edge [|->,gray]node [auto] {$\scriptstyle{}$} (A2_1);     \path (A0_1) edge [|->,gray]node [auto] {$\scriptstyle{}$} (A0_0);   \end{tikzpicture}   \] 
are well defined and they are each other's inverse.
\end{thm}
We will prove Theorem above after the following lemma.
\begin{lem}
\label{lem:left exactness for strong monoidal is exactness}Let $\Omega\in\LMon^{G}T$
be a strong monoidal functor. Then $\Omega_{R}=\odi T$ and 
\[
\Omega\text{ left exact}\iff\Omega\text{ exact}\iff\Supp\Omega_{R[G]}=T
\]
In particular $\LSMon_{R}^{G}$ is a substack of $\LMon_{R}^{G}$.\end{lem}
\begin{proof}
Since $\Omega_{R}\otimes\Omega_{R}\simeq\Omega_{R}$ and $\odi T\subseteq\Omega_{R}$,
we can conclude that $\rk\Omega_{R}=1$. But $\Omega_{R}$ has a structure
of $\odi T$ algebra induced by the multiplication $R\otimes R\arr R$.
So $\Spec\Omega_{R}\arr T$ is a degree one cover, which is therefore
an isomorphism.

For the equivalences, consider a short exact sequence in $\Loc^{G}R$
\[
\shV_{*}\colon\quad0\arr V'\arr V\arr V''\arr0
\]
Since there exists a natural isomorphism $U\otimes R[G]\simeq\underline{U}\otimes R[G]$
for $U\in\Loc^{G}R$ (see \cite[Part I, Example 3.7]{Jantzen2003}),
we see that $\shV_{*}\otimes R[G]$ is a splitting sequence in $\Loc^{G}R$.
In particular $\Omega_{\shV_{*}\otimes R[G]}$ is exact. Moreover
$\Omega_{\shV_{*}\otimes R[G]}\simeq\Omega_{\shV_{*}}\otimes\Omega_{R[G]}$.
If $\Supp\Omega_{R[G]}=T$, the functor $-\otimes\Omega_{R[G]}$ is
faithful exact since $\Omega_{R[G]}$ is locally free, and therefore
$\Omega$ is exact. Conversely, if $\Omega$ is left exact we have
$\odi T=\Omega_{R}\subseteq\Omega_{R[G]}$.

For the final statement, we have to show that the subcategory $\LSMon_{R}^{G}\subseteq\LMon_{R}^{G}$
is preserved by the pullback. This follows because $\Omega_{R}=\odi T$
and the pullback of an exact sequence of locally free of finite rank
sheaves is still exact.
\end{proof}

\begin{proof}
(\emph{of Theorem }\ref{thm:torsors are isomorphism}) $\LSMon_{R}^{G}$
is a substack of $\LMon_{R}^{G}$ thanks to \ref{lem:left exactness for strong monoidal is exactness}.
Since $G$ is flat, finite and of finite presentation, the push forward
functor $\Bi_{R}G\arr\LAlg_{R}^{G}$ is fully faithful with essential
image the full subcategory of algebras $\alA$ for which there exist
$G$-equivariant isomorphisms of algebras $\alA\simeq\odi{}[G]$ locally
in the fppf topology. In what follows we identify $\Bi_{R}G$ with
this stack. Since taking invariants commutes with flat base change,
given $\alA\in\Bi_{R}G(T)$, $\Omega^{\alA}$ is locally isomorphic
to the forgetful functor $\Loc^{G}R\arr\Loc T$, which is strong monoidal
and left exact. Thanks to \ref{thm:G equivariant ring are monoidal functors},
$\Omega_{*}\colon\Bi_{R}G\arr\LSMon_{R}^{G}$ is fully faithful and
we have only to prove that, if $\Omega\in\LSMon^{G}T$, then $\Omega_{R[G]}\in\Bi_{R}G(T)$.
Note that $f\colon\Spec\Omega_{R[G]}\arr T$ is faithfully flat and
finitely presented since $\Omega_{R[G]}$ is locally free and $\odi T\subseteq\Omega_{R[G]}$.
So it clearly has sections in the fppf topology. We therefore need
to show that the map 
\[
\rho\colon\Omega_{R[G]}\otimes\Omega_{R[G]}\arr\Omega_{R[G]}\otimes R[G]\text{ given by }\rho(x\otimes y)=\mu(x)(y\otimes1)
\]
is an isomorphism, where $\mu$ is the comodule structure on $\Omega_{R[G]}$.
Set $A=R[G]$ and consider the map 
\[
\omega\colon A\otimes A\arrdi{\Delta_{G}\otimes(\id_{A}\otimes1)}(A\otimes\underline{A})\otimes(A\otimes\underline{A})\simeq A\otimes A\otimes\underline{A}\otimes\underline{A}\arrdi{m_{A}\otimes m_{A}}A\otimes\underline{A}
\]
where $m_{A}$ is the multiplication. The map $\omega$ is a $G$-equivariant
isomorphism because it corresponds to $G\times G\ni(g,h)\arr(gh,g)\in G\times G$.
Moreover it is easy to check that we have a commutative diagram   \[   \begin{tikzpicture}[xscale=3.2,yscale=-1.0]     \node (A0_0) at (0, 0) {$\Omega_{A}\otimes \Omega_{A}$};     \node (A0_1) at (1, 0) {$\Omega_{A}\otimes R[G]$};     \node (A1_0) at (0, 1) {$\Omega_{A\otimes A}$};     \node (A1_1) at (1, 1) {$\Omega_{A\otimes \underline{R[G]}}$};     \path (A0_0) edge [->]node [auto] {$\scriptstyle{\rho}$} (A0_1);     \path (A0_0) edge [->]node [auto] {$\scriptstyle{}$} (A1_0);     \path (A0_1) edge [<-]node[rotate=-90] [above] {$\scriptstyle{\simeq}$} (A1_1);     \path (A1_0) edge [->]node [auto] {$\scriptstyle{\Omega_{\omega}}$} (A1_1);   \end{tikzpicture}   \] 
Since $\Omega$ is strong monoidal we get the result. %

\end{proof}

\subsection{Super solvable groups and $G$-torsors.}

Where $G$ is a diagonalizable group and $\Omega\in\LMon^{G}$ we
know that $\Omega$ corresponds to a $G$-torsor if and only if the
maps
\[
\Omega_{m}\otimes\Omega_{n}\arr\Omega_{m+n}\quad\forall m,n\in\Hom(G,\Gm)
\]
 are isomorphisms and $\Omega_{0}=\odi{}$. Here $\Omega_{m}=\Omega_{V_{m}}$,
where $V_{m}$ is the one dimensional representation associated to
$m\in\Hom(G,\Gm)$. On the other hand this condition is also equivalent
to require that $\Omega_{0}=\odi{}$ and that the maps
\[
\Omega_{m}\otimes\Omega_{-m}\arr\Omega_{0}=\odi{}\quad\forall m\in\Hom(G,\Gm)
\]
are surjective (and therefore isomorphisms). We want to generalize
this kind of statement for a larger class of groups, namely super
solvable groups (see \ref{def: super solvable groups} for the definition).

In this section we will assume that $G$ is a glrg and we continue
to work on a base ring $R$.
\begin{defn}
\label{def: super solvable groups} We will say that a group scheme
$G$ over an algebraically closed field is \emph{super solvable} if
there exists a filtration by closed subgroups
\[
1=H_{0}<H_{1}<\dots<H_{r}=G
\]
such that $H_{i}\triangleleft G$ and $H_{i+1}/H_{i}\simeq\mu_{p}$
for some prime $p$ and for all $i$. 

A finite, flat and finitely presented group scheme $G$ over a base
$S$ will be called super solvable if it is so over any geometric
point.\end{defn}
\begin{rem}
In our hypothesis, if $G$ is constant over an algebraically closed
field $k$, then it is super solvable according to the above definition
if and only if it is so as abstract group. Indeed, since $G$ is linearly
reductive, we will have $\car k\nmid|G|$, and $\mu_{q}\simeq\Z/q\Z$
if $\car k\nmid q$.
\end{rem}

\begin{rem}
\label{Rem: action of G on diagonalizable normal subgroup}Assume
that $R$ is strictly Henselian. If $H$ is an open and closed normal
subgroup of $G$ which is diagonalizable, then the conjugacy yields
an action of $\underline{G}/\underline{H}$ on $\Hom(H,\Gm)$. In
particular if $H=G_{1}$ we get an action of $\underline{G}$ on the
group $M=$$\Hom(G_{1},\Gm)$. Indeed $G$ acts by conjugacy on $H$
and, since $H$ is abelian, it induces an action of $\underline{G}/\underline{H}=G/H$
on $H$ and therefore on $\Hom(H,\Gm)$. \end{rem}
\begin{notation}
In the situation of remark \ref{Rem: action of G on diagonalizable normal subgroup}
we will consider $\Hom(H,\Gm)$ and, in particular, $\Hom(G_{1},\Gm)$,
endowed by the left action of $\underline{G}$ defined above.
\end{notation}
The following remark gives a concrete description of what a super
solvable group is over an algebraically closed field.
\begin{rem}
Assume that $R=k$ is an algebraically closed field. Then $G$ is
super solvable if and only if $\underline{G}$ is super solvable and
there exists a filtration by subgroups 
\[
0=H_{0}<H_{1}<\cdots<H_{r}=M=\Hom(G_{1},\Gm)
\]
such that each $H_{i}$ is $\underline{G}$-stable and $H_{i+1}/H_{i}$
is cyclic of prime order. \end{rem}
\begin{notation}
Given a group $G$ over a scheme $S$ and a character $\chi\in\Hom(G,\Gm)$
we will denote by $V_{\chi}$ the representation of $G$ on $\odi S$
induced by such character.

Given $\alA\in\CAlg^{G}T$ and a representation $V\in I_{G}$ we set
\[
\omega_{V}^{\alA}\colon\Omega_{V}^{\alA}\otimes\Omega_{\duale V}^{\alA}\arr\Omega_{V\otimes\duale V}^{\alA}\arr\Omega_{R}^{\alA}=\alA^{G}
\]
We will also write simply $\omega_{V}$ instead of $\omega_{V}^{\alA}$
if this will not lead to confusion.\end{notation}
\begin{thm}
\label{thm:torsors when omega surjective}Let $G$ be a super solvable
glrg and let $\alA\in\LAlg^{G}T$. Then $\alA\in\Bi G$ if and only
if $\alA^{G}=\Omega_{R}^{\alA}\simeq\odi T$ and for any representation
$V\in I_{G}$ the map
\[
\omega_{V}^{\alA}\colon\Omega_{V}^{\alA}\otimes\Omega_{\duale V}^{\alA}\arr\Omega_{V\otimes\duale V}^{\alA}\arr\Omega_{R}^{\alA}\simeq\odi T
\]
is surjective.
\end{thm}
Before proving the above Theorem, we need some preliminary results.
\begin{lem}
\label{lem:super solvable property}Let $G$ be a constant super solvable
group, $H$ be a subgroup and $k$ be an algebraically closed field
such that $\car k\nmid|G|$. If $V^{H}\neq0$ for all the irreducible
representations $V$ of $G$ over $k$ then $H=0$.\end{lem}
\begin{proof}
We will argue by induction on $|G|$. If $G=0$ there is nothing to
prove. So assume $G\neq0$. If $K\neq0$ is a normal subgroup of $G$
and $\phi\colon G\arr G/K$ is the projection, then $\phi(H)<G/K$
satisfies the inductive hypothesis and therefore $\phi(H)=0$, i.e.
$H\subseteq K$. In particular we can choose $K$ to be cyclic since
$G$ is super solvable and we can conclude that $H$ is normal and
abelian in $G$. Let $W$ be an irreducible $H$-representation. Given
a system $\shR$ of representatives of $G/H$ we can write
\[
R_{H}\ind_{H}^{G}W=\bigoplus_{g\in\shR}W_{g}
\]
where $W_{g}$ is the representation of $H$ given by $W$ and the
action $h\star x=ghg^{-1}x$. By hypothesis, we know that $(\ind_{H}^{G}W)^{H}\neq0$.
So there exist $g\in\shR$, $x\in W_{g}$ such that $h\star x=ghg^{-1}x=x$
for any $h\in H$. Since $H$ is normal we can conclude that $W^{H}\neq0$.
So $H$ has only the trivial representation and therefore $H=0$.\end{proof}
\begin{lem}
\label{lem:orbits of G in M for super solvable}Let $M$ be an abelian
$p$-group, for a prime $p$, and $G$ be a constant group acting
on $M$. Assume that there exists a filtration 
\[
0=H_{0}<H_{1}<\cdots<H_{r}=M
\]
by $G$-stable subgroup such that $H_{i+1}/H_{i}\simeq\Z/p\Z$. Then
for any proper subgroup $H$ of $M$ there exists a $G$-orbit in
$M-H$.\end{lem}
\begin{proof}
We can assume that $H$ has index $p$. In particular $pM\subseteq H$.
The action of $G$ on $M$ induces an action of $G$ on $M/pM$ that
has a filtration like the one of $M$. Since $H/pM$ is a proper subgroup
of $M/pM$ we can assume that $pM=0$, i.e. $M$ is a finite $\F_{p}$
vector space. Choosing a basis $e_{1},\dots,e_{r}$ according to the
given filtration, we can assume that $G$ acts by triangular matrices.
Let $H$ be a subspace of $M$ and assume that any $G$-orbit has
an element in $H$. We have to prove that $H=M$. Set $e_{0}=0$ and
assume by induction that $e_{0},\dots,e_{j-1}\in H$ for $j\leq n$.
We have that $Ge_{j}\cap H\neq\emptyset$. Since $G$ acts by triangular
matrices we can write
\[
H\ni g(e_{j})=\lambda e_{j}+x\text{ with }\lambda\in\F_{p}^{*}\comma x\in<e_{1},\dots,e_{j-1}>_{\F_{p}}\subseteq H
\]
So $\lambda e_{j}\in H$ and $e_{j}\in H$.
\end{proof}
In what follows $G$ is still our glrg over the base ring $R$.
\begin{lem}
\label{lem:rk V invertible for irreducible representations}If $V\in I_{G}$
then $\rk V\in R^{*}$.\end{lem}
\begin{proof}
Consider $W=\Hom(V,V)$ and note that, by \ref{prop:behaviour of irreducible representations for general glrg},
we have $W^{G}=R\cdot\id_{V}$. Since $G$ is a glrg, there exists
a $G$-equivariant map $\phi\colon W\arr R$ such that $\phi(\id_{V})=1$.
On the other hand any $G$-equivariant map $V\otimes\duale V\arr R$
is of the form $\lambda e_{V}$ for $\lambda\in R$, where $e_{V}$
is the evaluation $e_{V}(v\otimes\psi)=\psi(v)$. If $r=\rk V$ it
is easy to check that $e_{V}(\id_{V})=r$. So $1=\lambda e_{V}(\id_{V})=\lambda r$.\end{proof}
\begin{lem}
\label{lem:omega and the projection}Let $\alA\in\LAlg^{G}T$ and
$V\in I_{G}$. Then 
\[
\pi(\theta_{V}(\phi\otimes x)\theta_{\duale V}(v\otimes y))=\frac{\phi(v)}{\rk V}\omega_{V}(x\otimes y)
\]
where $\phi\in\duale V,v\in V,x\in\Omega_{V}^{\alA},y\in\Omega_{\duale V}^{\alA},$
$\theta_{-}\colon\duale{(-)}\otimes\Hom^{G}(\duale{(-)},\alA)\arr\alA$
is the evaluation and $\pi\colon\alA\arr\alA^{G}$ is the projection
according to the $G$-equivariant decomposition of $\alA$.\end{lem}
\begin{proof}
For any $W\in\Loc^{G}R$, the map
\[
\duale W\otimes\Omega_{W}=\duale W\otimes\Hom^{G}(\duale W,\alA)\arr\alA\arr\alA^{G}
\]
 is non zero only on the factor $\duale{(W^{G})}\otimes\Omega_{W^{G}}$.
Let $W=\duale V\otimes V$ and remember that, by \ref{prop:behaviour of irreducible representations for general glrg},
we have $W^{G}=R\id_{V}$. Under the isomorphism $W\simeq\duale W$,
$\id_{V}$ is sent to the evaluation $e_{V}\colon\duale V\otimes V\arr R$,
while $\phi\otimes v$ to the map $\psi$ given by $\psi(\delta\otimes z)=\delta(v)\phi(z)$.
The equivariant section of $R\arrdi{1\arr e_{V}}\duale W$ is given
by $(\phi\arr\phi(\id_{V})/\rk V)$ so the component of $\psi$ in
$(\duale W)^{G}$ is
\[
e_{V}\psi(\id_{V})/\rk V=e_{V}\phi(v)/\rk V
\]
By definition $\omega_{V}(x\otimes y)=x\otimes y(e_{V})$ and taking
into account \ref{lem:characterization of muliplication by invariant hom}
we have
\[
\pi(\theta_{V}(\phi\otimes x)\theta_{\duale V}(v\otimes y))=x\otimes y(e_{V})\phi(v)/\rk V=\omega_{V}(x\otimes y)\phi(v)/\rk V
\]

\end{proof}
During the proof of Theorem \ref{thm:torsors when omega surjective},
we will reduce to consider local algebras. The following lemma explains
what happens in this situation.
\begin{lem}
\label{lem:what happens on local rings}Assume that $R$ is strictly
Henselian and let $A\in\CAlg^{G}R$ be a local $R$-algebra such that
$A^{G}=R$. Then
\begin{itemize}
\item If $G=\underline{G}$ then the maximal ideal of $A$ is
\[
m_{R}\oplus\bigoplus_{R\neq V\in I_{G}}\duale V\otimes\Omega_{V}^{A}
\]
and for any $R\neq V\in I_{G}$ the map $\omega_{V}^{A}$ is not surjective.
\item If $G=G_{1}$ then
\[
H=\{m\in M\st(\duale{V_{m}}\otimes\Omega_{V_{m}}^{A})\cap A^{*}\neq\emptyset\}=\{m\in M\st\omega_{V_{m}}\text{ is surjective}\}
\]
is a subgroup of $M$, and the subalgebra $B=\oplus_{m\in H}\duale{V_{m}}\otimes\Omega_{V_{m}}\subseteq A$
is a $D(H)$-torsor.
\end{itemize}
\end{lem}
\begin{proof}
Set $\Omega=\Omega^{A}$, $\omega=\omega^{A}$ and $k=R/m_{R}$. Assume
$G=\underline{G}$ and let $m_{A}$ be the maximal ideal of $A$.
Since $m_{A}$ is stable under the action of $\underline{G}$, it
can be written as
\[
m_{A}=\bigoplus_{V\in I_{G}}\duale V\otimes\Gamma_{V}
\]
where $\Gamma_{V}\subseteq\Omega_{V}$. In particular 
\[
L=A/m_{A}=\bigoplus_{V\in I_{G}}\duale V\otimes(\Omega_{V}/\Gamma_{V})
\]
$G$ acts on $L$ and $L^{G}=\Omega_{R}/\Gamma_{R}=k$. Therefore
$L/k$ is separable, i.e. $L=k$ and by dimension we get the first
equality. Taking into account \ref{lem:omega and the projection}
we also have that if $R\neq V\in I_{G}$ then $\omega_{V}$ is not
surjective.

Now assume $G=G_{1}=D(M)$ and set $\Omega_{m}=\duale{V_{m}}\otimes\Omega_{V_{m}}$.
Note that if $\Omega_{m}\otimes\Omega_{-m}\arr R$ is surjective then
$\Omega_{m}\cap A^{*}\neq\emptyset$ since $R$ is local. Conversely
if $x\in\Omega_{m}\cap A^{*}$ let $\lambda=x^{|M|}\in R$. If $\lambda\in m_{R}$
then $x\in m_{A}$, which is not the case. Since $x^{|M|-1}\in\Omega_{-m}$
and therefore $\omega_{m}(x\otimes x^{|M|-1})=x^{|M|}\in R^{*}$ we
have that $\omega_{m}$ is surjective. Finally if $x\in\Omega_{m}\cap A^{*}$
and $y\in\Omega_{n}\cap A^{*}$ then $xy\in\Omega_{m+n}\cap A^{*}$.
So $H$ is a subgroup and $B$ is a $D(H)$-torsor thanks to \ref{pro:equivalent conditions for a D(M)-torsor}.
\end{proof}
The following two lemmas describe how the associated functors $\Omega^{*}$
change when making an induction or taking invariants.
\begin{lem}
\label{prop:E of an induction is the restriction}Let $H$ be a subgroup
scheme of $G$ and assume they are both glrg. If $\alA\in\CAlg^{H}T$,
then 
\[
\ind_{H}^{G}\alA\simeq(\alA\otimes R[G])^{H}\in\CAlg^{G}T\text{ and }\;\Omega^{\ind_{H}^{G}\alA}=\Omega^{\alA}\circ\R_{H}\colon\Loc^{G}R\arrdi{\R_{H}}\Loc^{H}R\arrdi{\Omega^{\alA}}\FCoh T
\]
\end{lem}
\begin{proof}
We have 
\[
\Omega_{V}^{\ind_{H}^{G}\alA}=\Homsh^{G}(\duale V,\ind_{H}^{G}\alA)\simeq\Homsh^{H}(\duale{(\R_{H}V}),\alA)=\Omega_{\R_{H}V}^{\alA}
\]
So $\ind_{H}^{G}\alA\in\CAlg^{G}T$ and it is a subring $\Omega_{R[G]}^{\ind_{H}^{G}\alA}\simeq(\alA\otimes R[G])^{H}\subseteq\alA\otimes R[G]$.\end{proof}
\begin{lem}
\label{lem:functor of algebra of invariants by a normal subgroup}Let
$K$ be a normal subgroup scheme of $G$ and $\alA\in\CAlg^{G}T$.
Then $\alA^{K}\in\CAlg^{G/K}T$ and its associated functor is 
\[
\Loc^{G/K}R\arrdi{\text{restriction}}\Loc^{G}R\arrdi{\Omega^{\alA}}\FCoh T
\]
\end{lem}
\begin{proof}
Note that if $F\colon(\Sch/R)^{op}\arr\set$ is a functor with an
action of $G$, then $F^{K}$ is stable under the action of $G$ and
therefore the map $G\arr\Autsh F^{K}$ factors through $G/K\arr\Autsh F^{K}$.
So $\alA^{K}\in\CAlg^{G}T$ and $\alA^{K}\in\CAlg^{G/K}T$. Now note
that if $V\in\Loc^{G/K}R$ then
\[
\Homsh^{G}(\R_{G}V,\alA)\simeq\Homsh^{G}(\R_{G}V,\alA^{K})=\Homsh^{G}(\R_{G}V,\R_{G}\alA^{K})
\]
since $K$ acts trivially on $R_{G}V$. So we have to prove that the
natural map 
\[
\Homsh^{G/K}(V,W)\text{\ensuremath{\arr\Homsh^{G}}(\ensuremath{\R_{G}}V,\ensuremath{\R_{G}}W) for }V\in\Loc^{G/K}T\comma W\in\QCoh^{G/K}T
\]
is an isomorphism. In order to do that, note that $\Homsh(\R_{G}V,\R_{G}W)=\R_{G}\Homsh(V,W)$
and that in general $(\R_{G}U)^{G}=U^{G/K}$ for all $G/K$-modules
$U$.\end{proof}
\begin{lem}
\label{lem:when omega is surjective in the induction}Let $R'$ be
a local $R$-algebra, $H$ be a subgroup scheme of $G$ with a good
representation theory and $\alB\in\CAlg^{H}R'$. If we set $\alA=\ind_{H}^{G}\alB$
and we take $V\in\Loc^{G}R$ then
\[
\omega_{V}^{\alA}\text{ surjective }\iff\exists\Delta\in I_{H}\text{ s.t. }\Hom^{H}(\Delta,V)\neq0\text{ and }\omega_{\Delta}^{\alB}\text{ is surjective}
\]
\end{lem}
\begin{proof}
Let $\Omega=\Omega^{\alA}$ and $\delta=\Omega^{\alB}$. By \ref{prop:E of an induction is the restriction},
we know that $\Omega=\delta\circ\R_{H}$. Denote by $\iota^{\Omega}$,
$\iota^{\delta}$ the natural transformations that define the monoidal
structures of $\Omega$ and $\delta$ respectively. Let $V\in I_{G}$.
If we set $V_{\Delta}=\Hom^{H}(\Delta,V)$ for $\Delta\in\Loc^{H}R$
we have
\[
V\simeq\bigoplus_{\Delta\in I_{H}}V_{\Delta}\otimes\Delta\text{ and }\duale V\simeq\bigoplus_{\Delta\in I_{H}}\duale{V_{\duale{\Delta}}}\otimes\Delta
\]
The map $\iota_{V,\duale V}^{\Omega}\colon\Omega_{V}\otimes\Omega_{\duale V}\arr\Omega_{V\otimes\duale V}$
factors through
\[
\id\otimes\iota_{\Delta,\duale{\Lambda}}^{\delta}\colon\duale{V_{\Delta}}\otimes V_{\Lambda}\otimes\delta_{\Delta}\otimes\delta_{\duale{\Lambda}}\arr\duale{V_{\Delta}}\otimes V_{\Lambda}\otimes\delta_{\Delta\otimes\duale{\Lambda}}
\]
Now call $e_{W}\colon W\otimes\duale W\arr R$ the evaluation map
for any $R$-module $W$. The map $e_{V}$ sends any $(V_{\Delta}\otimes\Delta)\otimes\duale{(V_{\Lambda}\otimes\Lambda)}$
with $\Delta\neq\Lambda$ to $0$ and restricts to
\[
(V_{\Delta}\otimes\Delta)\otimes\duale{(V_{\Delta}\otimes\Delta)}\simeq V_{\Delta}\otimes\duale{V_{\Delta}}\otimes\Delta\otimes\duale{\Delta}\arrdi{e_{V_{\Delta}}\otimes e_{\Delta}}R
\]
So the composition $\omega_{V}^{\alA}=\Omega_{e_{V}}\circ\iota_{V,\duale V}^{\Omega}\colon\Omega_{V}\otimes\Omega_{\duale V}\arr R$
is $0$ on $\duale{V_{\Delta}}\otimes V_{\Lambda}\otimes\delta_{\Delta}\otimes\delta_{\duale{\Lambda}}$
if $\Lambda\neq\Delta$, is $e_{V}\otimes\omega_{\Delta}^{\alB}$
if $\Lambda=\Delta$. In particular
\[
\Imm(\omega_{V}^{\alA})=\sum_{\Delta\st V_{\Delta}\neq0}\Imm(\omega_{\Delta}^{\alB})
\]
Since $R'$ is local we get the required result.\end{proof}
\begin{lem}
\label{lem:descend of torsors along induction}Let $H$ be a subgroup
scheme of $G$ and $\alB\in\LAlg^{H}T$. Then
\[
\ind_{H}^{G}\alB\in\Bi G\iff\alB\in\Bi H
\]
\end{lem}
\begin{proof}
If $\alB\in\Bi H$ it is enough to note that $\ind_{H}^{G}\odi{}[H]\simeq\odi{}[G]$.
So assume $\alA=\ind_{H}^{G}\alB\in\Bi G$. Since $\alB$ is locally
free, the condition of being a $H$-torsor is open, and we can assume
$R=k$ and $T=\Spec R$, where $k$ is an algebraically closed field.
If $\delta=\Omega^{\alB}$, from \ref{prop:E of an induction is the restriction},
we know that $\Omega^{\alA}=\delta\circ\R_{H}$. By \ref{thm:torsors are isomorphism},
we have to prove that for any $V,W\in\Loc^{H}k$ the map $i_{V,W}^{\delta}\colon\delta_{V}\otimes\delta_{W}\arr\delta_{V\otimes W}$
is an isomorphism. Note that if $V=V_{1}\oplus V_{2}$ in $\Loc^{H}k$,
then $\iota_{V,W}^{\delta}$ is an isomorphism if and only if $\iota_{V_{1},W}^{\delta}$
and $\iota_{V_{2},W}^{\delta}$ are so. Therefore we must find a $H$-representation
$V$, containing all the irreducible representations of $H$ and such
that $\iota_{V,V}^{\delta}$ is an isomorphism. I claim that $V=\R_{H}k[G]$
satisfies the request. Indeed $\iota_{V,V}^{\delta}$ is an isomorphism
since $\alA$ is a $G$-torsor and $\Omega^{\alA}=\delta\circ\R_{H}$.
Moreover, since we have a $H$-equivariant surjective map $\R_{H}k[G]\arr k[H]$,
we have that any irreducible representation of $H$ is a quotient
and therefore a subrepresentation of $\R_{H}R[G]$.\end{proof}
\begin{lem}
\label{lem:induction from diagonalizable groups} Assume that $R$
is strictly Henselian. Let also $H$ be an open and closed normal
subgroup of $G$ which is diagonalizable. Then, if $m\in\Hom(H,\Gm)$
we have a $H$-equivariant isomorphism 
\[
\R_{H}\ind_{H}^{G}V_{m}\simeq\bigoplus_{g\in\underline{G}/\underline{H}}V_{g(m)}
\]
where the action of $\underline{G}/\underline{H}$ on $\Hom(H,\Gm)$
is the one given in \ref{Rem: action of G on diagonalizable normal subgroup}.
Moreover $\ind_{H}^{G}V_{m}\simeq\ind_{H}^{G}V_{n}$ if and only if
there exists $g\in\underline{G}/\underline{H}$ such that $g(m)=n$.\end{lem}
\begin{proof}
Set $\odi m=\WW(V_{m})$. By \ref{prop:decomposition of G into H torsors},
we have a decomposition of $G$ into $H$-torsors
\[
G=\bigsqcup_{i\in\underline{G}/\underline{H}}H^{i}
\]
In particular we have
\[
\ind_{H}^{G}\odi m=\Homsh^{H}(G,\odi m)\simeq\prod_{i\in\underline{G}/\underline{H}}\Homsh^{H}(H^{i},\odi m)
\]
Since $H^{i}$ is a $H$-torsor for all $i$, we see that $U_{i}=\Homsh^{H}(H^{i},\odi m)$
is an invertible sheaf on $R$ and therefore $U_{i}\simeq\WW(R)$.
Now consider the right action of $G$ on itself given by the multiplication.
Each $H^{i}$ is $H$ invariant since $H$ is normal, so the above
decomposition is $H$-equivariant. In particular $U_{i}\simeq\odi n$
for some $n\in\Hom(H,\Gm)$. In order to compute $n$, we can assume
that $R=k$ is an algebraically closed field, so that $H^{i}=Hi$,
where we think of $i$ as an element of $H^{i}(k)$. Given $f\colon Hi\arr\odi m\in\Homsh^{H}(Hi,\odi m)$
and $h\in H$ we have
\[
(h\star f)(ti)=m(t)f(ih)=m(t)f(ihi^{-1}i)=m(ihi^{-1})f(ti)\then h\star f=i^{-1}(m)(h)f
\]
and therefore that $n=i^{-1}(m)$.

For the last claim, a $G$-equivariant isomorphism between $\ind_{H}^{G}V_{m}$
and $\ind_{H}^{G}V_{n}$ is also $H$-equivariant and therefore $g(m)=n$
for some $g\in\underline{G}/\underline{H}$. Conversely note that
two representations of $G$ are isomorphic if and only if they are
so on the algebraic closure of the residue field of $R$, because
the restriction $I_{G}\arr I_{G_{k}}$ is an isomorphism by hypothesis
and $R$ is local. So we can again assume $R=k$ algebraically closed.
Given $g\in G(k)\simeq\underline{G}$, we claim that
\[
\psi\colon\ind_{H}^{G}\odi m=\Homsh^{H}(G,\odi m)\arr\Homsh^{H}(G,\odi{g(m)})=\ind_{H}^{G}\odi n\comma\psi(f)(u)=f(g^{-1}u)
\]
is $G$-equivariant isomorphism. It is well defined since
\[
\psi(f)(hu)=f(g^{-1}hu)=f(g^{-1}hgg^{-1}u)=m(g^{-1}hg)f(g^{-1}u)=g(m)(h)\psi(f)(u)
\]
It is $G$-equivariant since
\[
(u\cdot\psi(f))(v)=\psi(f)(vu)=f(g^{-1}vu)=(u\cdot f)(g^{-1}v)=\psi(u\cdot f)(v)
\]

\end{proof}

\begin{proof}
(\emph{of theorem} \ref{thm:torsors when omega surjective}) If $\alA\in\Bi G$
then, by \ref{thm:torsors are isomorphism}, all the maps $\omega_{V}$
are surjective since $\Omega_{V\otimes\duale V}\arr\Omega_{R}$ is
surjective. Conversely, since both conditions in the statement are
open conditions, we can assume that $T=\Spec R$, that $R=k$ is an
algebraically closed field and replace $\alA$ by a finite $k$-algebra
$A$. By \ref{lem:A is induction of the localization}, we can write
$A\simeq\ind_{H}^{G}A_{p}$ where $p$ is a closed point of $\Spec A$
and $H$ is the stabilizer of $\Spec A_{p}$. In particular $A_{p}\in\LAlg^{H}k$
and therefore comes from a functor $\delta\colon\Loc^{H}k\arr\Vect_{k}\in\LMon^{H}k$.
We want to prove first that $H=G_{1}$. We set $V_{\Delta}=\Hom^{H}(\Delta,V)$
for $V\in\Loc^{G}R$ and $\Delta\in\Loc^{H}R$. We will use \ref{lem:super solvable property}
showing that for any $V\in I_{\underline{G}}$ we have $V^{\underline{H}}\neq0$,
so that $\underline{H}=1$. By \ref{lem:when omega is surjective in the induction},
since $\omega_{V}^{A}$ is surjective, there exists $\Delta\in I_{H}$
such that $V_{\Delta}\neq0$ and $\omega_{\Delta}^{A_{p}}$ is surjective.
We will show that $\Delta=R$. Since $R_{H}R_{G}V=R_{H}R_{\underline{H}}V$,
we have that $V_{\Lambda}=0$ if $\Lambda\in I_{H}-I_{\underline{H}}$.
So $\Delta\in I_{\underline{H}}$. Since $A_{p}$ is local, we have
that $A_{p}^{G_{1}}\in\LAlg^{\underline{H}}R$ is local thanks to
\ref{lem:invariants by Gone sends local to local}. From \ref{lem:functor of algebra of invariants by a normal subgroup}
and \ref{lem:what happens on local rings} follows that if $\Lambda\in I_{\underline{H}}$
then $\omega_{\Lambda}^{A_{p}^{G_{1}}}=\omega_{\Lambda}^{A_{p}}$
and that $\omega_{\Lambda}^{A_{p}^{G_{1}}}=0$ if $\Lambda\neq R$.
Since $\omega_{\Delta}^{A_{p}}\neq0$ and $\Delta\in I_{\underline{H}}$
we can conclude that $\Delta=R$ and therefore $V_{\Delta}=V^{H}=V^{\underline{H}}\neq0$,
as required.

We want now to prove that $A_{p}$ is a $G_{1}=D(M)$-torsor. By \ref{lem:what happens on local rings},
the set $Q=\{m\in M\st\omega_{V_{m}}\text{ is surjective}\}$ is a
subgroup of $M$ and, if we prove that $Q=M$, we will have that $A_{p}$
is a $G_{1}$-torsor and therefore that $A$ is a $G$-torsor. We
will use \ref{lem:orbits of G in M for super solvable} and \ref{lem:induction from diagonalizable groups}.
Given $m\in M$ we have shown that there must exist $n\in N$ such
that 
\[
V_{n}\subseteq\R_{G_{1}}\ind_{G_{1}}^{G}V_{m}=\bigoplus_{g\in\underline{G}}V_{g(m)}\text{ and }\omega_{V_{n}}\text{ is surjective}
\]
So given $m\in M$ there exists $g\in\underline{G}$ such that $g(m)\in Q$
and therefore $Q=M$ as required.\end{proof}
\begin{rem}
\label{rem:counterexample for omega surjective implies torsors}We
want to show now an example of a constant solvable group and of a
group $G$ such that $\underline{G}$ is super solvable for which
\ref{thm:torsors when omega surjective} does not apply.

Let $k$ be an algebraically closed field and set 
\[
G=(\mu_{2}\times\mu_{2})\ltimes\Z/3\Z
\]
where the action of $\Z/3\Z$ on $\Aut(\mu_{2}\times\mu_{2})\simeq\Aut(\Z/2\Z\times\Z/2\Z)\simeq\GL_{2}\F_{2}$
is given by the order $3$ matrix
\[
A=\left(\begin{array}{cc}
0 & 1\\
1 & 1
\end{array}\right)
\]
If $\car k\neq2$, then $G$ is constant and solvable, if $\car k=2$
then $\underline{G}=\Z/3\Z$ which is super solvable.

Set $K=(\Z/2\Z)^{2}$ with $\F_{2}$ basis $e_{1}=(1,0)$, $e_{2}=(0,1)$,
$H=D(K)$. Since $Ae_{1}=e_{2}$, $A^{2}e_{1}=e_{1}+e_{2}$, we see
that $\Z/3\Z$ permutes the $3$ subgroups of index $2$ of $K$.
We now describe the irreducible representations of $G$. We claim
that they are
\[
U=\ind_{H}^{G}V_{e_{1}}\comma k\comma V_{\chi}\comma V_{\chi^{2}}
\]
where $\chi\colon G\arr\Z/3\Z\arr k^{*}$ is a non trivial character.
We will make use of \ref{lem:induction from diagonalizable groups}.
If $V\in I_{G}$, there exists $u\in K$ such that
\[
V\subseteq\ind_{H}^{G}V_{u}=V_{u}\oplus V_{Au}\oplus V_{A^{2}u}
\]
If two among $V_{u},V_{Au},V_{A^{2}u}$ are isomorphic then $u=Au=A^{2}u$
and therefore $u=0$. In this case 
\[
\ind_{H}^{G}V_{0}=k\oplus V_{\chi}\oplus V_{\chi^{2}}
\]
So assume $u\neq0$. Since $K-\{0\}$ is a $\Z/3\Z$-orbit, $\ind_{H}^{G}V_{e_{1}}\simeq\ind_{H}^{G}V_{e_{2}}\simeq\ind_{H}^{G}V_{e_{1}+e_{2}}$.
So we have to prove that $U=\ind_{H}^{G}V_{e_{1}}$ is irreducible.
If it is not so, it will contain an irreducible representation of
dimension $1$ whose restriction to $H$ is not trivial. So there
must exist a character $\eta\colon G\arr\Gm$ such that $\eta_{|H}$
is not trivial. But if $\zeta\in H=\Homsh(K,\Gm)$, we denote by $\xi\in G(k)$
the generator of $\Z/3\Z$ and $\chi=\zeta\circ A$ we have
\[
A^{2}=A+\id\then\xi\chi\xi^{-1}\chi^{-1}=(\chi\circ A)\chi^{-1}=\zeta
\]
So $\eta_{|H}(\zeta)=\eta(\xi)\eta(\chi)\eta(\xi)^{-1}\eta(\chi)^{-1}=1$.

We are ready to explain the counterexample to \ref{thm:torsors when omega surjective}
for the above group $G$. Consider
\[
B=k[x,y]/(x^{2}-1,y^{2})\in\LAlg^{H}k\text{ and }A=\ind_{H}^{G}B
\]
where the action of $H$ on $G$ is given by $\deg x=e_{1}$, $\deg y=e_{2}$.
We want to show that $\omega_{V}^{A}$ is surjective for any $V\in I_{G}$
but $A$ is not a $G$-torsor. Taking into account \ref{lem:descend of torsors along induction},
$A$ is not a $G$-torsor because $B$ is not a $H$-torsor. Finally,
taking into account \ref{lem:when omega is surjective in the induction},
$\omega_{V}^{A}$ is surjective for any $V\in I_{G}$ because $\R_{H}V$
contains either the trivial representation $k$ or $V_{e_{1}}$ and
$\omega_{k}^{B}$, $\omega_{V_{e_{1}}}^{B}$ are surjective by construction
of $B$.\end{rem}

\section{Reducibility of $\GCov$ for non abelian linearly reductive groups.}

We know that, when $G$ is diagonalizable, except for some few cases
of lower rank, $\GCov$ is reducible (see \ref{cor:Mcov reducible}).
The goal of this section is to extend this bad behavior also to all
non abelian, linearly reductive groups $G$. The method we will use
does not apply and does not reduce to the diagonalizable case. When
$G$ is a glrg over a connected base, we will study the stacks in
groupoids $(\LAlg_{R}^{G})^{\text{gr}}$ and $(\LMon_{R}^{G})^{\text{gr}}$
and we will decompose them in a disjoint union of stacks parametrized
by functions $I_{G}\arr\N$, called rank functions. The stack $\GCov$
will correspond to the rank function $f_{V}=\rk V$. The result about
reducibility of $\GCov$ is then obtained looking at the behaviour
of the rank functions under induction from a subgroup.

We start stating the Theorem we will prove at the end of this section.
\begin{thm}
\label{thm:GCov reducible when G not abelian}If $G$ is a finite,
non abelian and linearly reductive group then $\GCov$ is reducible.
If $G$ is defined over a connected scheme, then $\GCov$ is also
universally reducible.
\end{thm}
Remember that universally reducible means reducible after any base
change (see \ref{def:universally reducible}). Note that, if we do
not assume that the base $S$ is connected, we can not conclude that
$\GCov$ is universally reducible, since one can always take $G$
as disjoint union of $\mu_{2}$ and $S_{3}$ over $\Spec\Q\sqcup\Spec\Q$.
On the other hand what happens when the base is not connected is clear
from the following Proposition.
\begin{prop}
\label{prop:the locus where G abelian open and closed}If $G$ is
a linearly reductive group over a scheme $S$, then the locus of $S$
where $G$ is abelian is open and closed in $S$.\end{prop}
\begin{proof}
Denote by $Z$ this locus. Topologically, $|Z|$ is closed in $S$,
because it is the locus where the maps $G\times G\arr G$ given by
$(g,h)\longmapsto gh$ and $(g,h)\longmapsto hg$ coincide and $G$
is flat and proper. We have to prove that, given an algebraically
closed field $k$ and a map $\Spec k\arrdi pS$ such that $G_{k}=G\times k$
is abelian, there exists a fppf neighborhood of $S$ around $p$ where
$G$ is abelian. By \cite[Theorem 2.19]{Abramovich2007}, we can assume
that $G=\Delta\ltimes H$, where $\Delta$ is diagonalizable and $H$
is constant. If $G_{k}$ is abelian, then $H$ is abelian, the map
$H\arr\Aut\Delta\simeq\Aut(\Hom(\Delta,\Gm))^{op}$ is trivial and
therefore $G\simeq\Delta\times H$ is abelian.
\end{proof}
From now on, except for the proof of Theorem \ref{thm:GCov reducible when G not abelian},
$G$ will be a linearly reductive group over a ring $R$ with connected
spectrum. It will be clear that this is not a necessary condition,
but we want to avoid technicalities like considering multivalued rank
functions for a locally free sheaf. In particular any $V\in\Loc^{G}R$
has a well defined rank. As mention above, one of the main ingredient
in the proof of Theorem \ref{thm:GCov reducible when G not abelian}
is the theory of rank functions, that we are going to introduce.
\begin{defn}
Assume that $G$ is a glrg. We will say that $\Omega\in\LAdd_{R}^{G}$
($\alA\in\Loc_{R}^{G}$) has \emph{equivariant constant rank }(or
is of equivariant constant rank) if for any $V\in\Loc^{G}R$ the locally
free sheaf $\Omega_{V}$ ($\Omega_{V}^{\alA}=(V\otimes\alA)^{G}$)
has constant rank.

Given $\alA\in\Loc_{R}^{G}$ or $\Omega\in\LAdd_{R}^{G}$ of equivariant
constant rank we define the rank functions $\rk^{\alA}\colon I_{G}\arr\N\comma\rk^{\Omega}\colon I_{G}\arr\N$
as
\[
\rk_{V}^{\Omega}=\rk\Omega_{V}\comma\rk_{V}^{\alA}=\rk_{V}^{\Omega^{\alA}}=\rk(V\otimes\alA)^{G}
\]
Given $f\colon I_{G}\arr\N$ we will denote by $\Loc_{R,f}^{G}$,
$\LRings_{R,f}^{G}$, $\LAlg_{R,f}^{G}$, $\LAdd_{R,f}^{G}$, $\LPMon_{R,f}^{G}$,
$\LMon_{R,f}^{G}$ the full substack of $\Loc_{R}^{G}$, $\LRings_{R}^{G}$,
$\LAlg_{R}^{G}$, $\LAdd_{R}^{G}$, $\LPMon_{R}^{G}$, $\LMon_{R}^{G}$
of objects $\chi$ of equivariant constant rank such that $\rk^{\chi}=f$
respectively.

Given $f\colon I_{G}\arr\N$ we will still call $f$ the extension
$f\colon\Loc^{G}R\arr\N$ given by
\[
f_{U}=\sum_{V\in I_{G}}\rk(\Hom^{G}(V,U))f_{V}
\]
In particular if $\Omega\in\LAdd_{R,f}^{G}$ we will have $\rk\Omega_{U}=f_{U}$
for any $U\in\Loc^{G}R$.\end{defn}
\begin{rem}
\label{rem:description of GCov with monoidal functors}Theorem \ref{thm:description of GCov with monoidal functors}
says that $\Omega^{*}$ induces an isomorphism $\GCov\simeq(\LMon_{R,f}^{G})^{\text{gr}}$,
where $f\colon I_{G}\arr\N$ is the rank function given by $f_{V}=\rk V$.
\end{rem}
The following Theorem shows how $\LAlg_{R}^{G}$ can be described
in terms of the rank functions.
\begin{thm}
\label{thm:Description of LAlgGRf when G is a glrg}Assume that $G$
is a glrg. Then
\[
(\LAlg_{R}^{G})^{\text{gr}}=\bigsqcup_{f\in\N^{I_{G}}}(\LAlg_{R,f}^{G})^{\text{gr}}
\]
Given $f\colon I_{G}\arr\N$, let $\delta\in\LAdd_{R,f}^{G}R$ be
the $R$-linear functor such that $\delta_{V}=R^{f_{V}}$ for $V\in I_{G}$
and set 
\[
X=\prod_{V,W\in I_{G}}\underline{\M}_{f_{V}f_{W},f_{V\otimes W}}\comma\underline{\GL}_{f}=\prod_{V\in I_{G}}\underline{\GL}_{f_{V}}
\]
Then we have a cartesian diagram   \[   \begin{tikzpicture}[xscale=3.0,yscale=-1.2]     \node (A0_0) at (0, 0) {$X$};     \node (A0_1) at (1, 0) {$\Spec R$};     \node (A1_0) at (0, 1) {$(\LPMon^G_{R,f})^{\text{gr}}$};     \node (A1_1) at (1, 1) {$(\LAdd^G_{R,f})^{\text{gr}}$};     \path (A0_0) edge [->]node [auto] {$\scriptstyle{}$} (A0_1);     \path (A1_0) edge [->]node [auto] {$\scriptstyle{}$} (A1_1);     \path (A0_1) edge [->]node [auto] {$\scriptstyle{\delta}$} (A1_1);     \path (A0_0) edge [->]node [auto] {$\scriptstyle{}$} (A1_0);   \end{tikzpicture}   \] 
where the vertical maps are $\underline{\GL}_{f}$-torsors. In particular
\[
(\LAdd_{R,f}^{G})^{\text{gr}}\simeq\Bi(\underline{\GL}_{f})\comma(\LPMon_{R,f}^{G})^{\text{gr}}\simeq[X/\underline{\GL}_{f}]
\]
Moreover the map $(\LMon_{R,f}^{G})^{\text{gr}}\arr(\LPMon_{R,f}^{G})^{\text{gr}}$
is an immersion. In particular all the stacks $\Loc_{R,f}^{G}$, $\LRings_{R,f}^{G}$,
$\LAlg_{R,f}^{G}$, $\LAdd_{R,f}^{G}$, $\LPMon_{R,f}^{G}$, $\LMon_{R,f}^{G}$
are algebraic stacks of finite presentation over $R$.\end{thm}
\begin{proof}
The first claim holds since, given $\Omega\in\LMon_{R}^{G}T$ for
some scheme $T$ and $V\in I_{G}$, then $\rk_{V}^{\Omega}=\rk\Omega_{V}$
is constant on the connected components of $T$. 

Now let $f\in\N^{I_{G}}$ and $\delta\in\LAdd_{R,f}^{G}$ be the $R$-linear
functor given by $\delta_{V}=R^{f_{V}}$, which exists thanks to \ref{rem: additive functors uniquely determined on irreducible representations}.
By \ref{rem: additive functors uniquely determined on irreducible representations},
we clearly have that $(\LAdd_{R,f}^{G})^{\text{gr}}\simeq\Bi(\underline{\GL}_{f})$.
Now consider the forgetful map
\[
(\LPMon_{R,f}^{G})^{\text{gr}}\arr(\LAdd_{R,f}^{G})^{\text{gr}}\text{ and }Z=(\LPMon_{R,f}^{G})^{\text{gr}}\times_{(\LAdd_{R,f}^{G})^{\text{gr}}}\Spec R
\]
$Z$ is the functor that associates to an $R$-scheme $T$ all the
possible pseudo monoidal structures on $\delta\otimes\odi T$. By
\ref{prop:essential data for a pseudo monoidal functor} we have that
$Z=X$. Now we have to verify that $(\LMon_{R,f}^{G})^{\text{gr}}\arr(\LPMon_{R,f}^{G})^{\text{gr}}$
is an immersion. First, note that this map is fully faithful, because
an object in $\LPMon_{R}^{G}$ has at most one unity and the isomorphisms
must preserve them. So $Z=(\LMon_{R,f}^{G})^{\text{gr}}\times_{(\LPMon_{R,f}^{G})^{\text{gr}}}X$
is a subfunctor of $X$, namely the subfunctor of the pseudo-monoidal
structures $\iota_{V,W}$ that satisfy commutativity, associativity
and has a unity. We have to show that $Z\arr X$ is a finitely presented
immersion. We first consider the associativity. Given $V,W,Z$ and
$\iota_{V,W}\in X(T)$ there are two way of forming a map
\[
\delta_{V}\otimes\delta_{W}\otimes\delta_{Z}\otimes\odi T\arr\delta_{V\otimes W\otimes Z}
\]
Taking the difference we get a map
\[
q\colon X\arr Y=\prod_{V,W,Z}\Homsh(\delta_{V}\otimes\delta_{W}\otimes\delta_{Z},\delta_{V\otimes W\otimes Z})
\]
By functoriality, this map is a morphism of scheme, so the locus of
$X$ of the $\iota_{V,W}$ that are associative is $q^{-1}(0)$, which
is a closed subscheme of $X$. Moreover, it is easy to see that the
map $q$ involves only a finite number of matrices defined over $R$.
Therefore it is defined over some noetherian subring of $R$. In particular
the locus $q^{-1}(0)$ is defined by a finite set of equations.

We can argue similarly for the commutativity. Now we have to deal
with the unity. We have to describe the locus of $X$ of the $\iota_{V,W}$
such that there exists $x\in\delta_{R}\otimes T$ such that $\iota_{R,V}(x\otimes v)=v$
for any $V\in I_{G}$, $v\in V$. Consider the induced linear map
\[
\zeta_{\iota}\colon\delta_{R}\otimes\odi T\arr\prod_{V\in I_{G}}\End(\delta_{V})\otimes\odi T
\]
The unities of $(\delta\otimes\odi T,\iota)$ are the elements $x\in\delta_{R}\otimes\odi T$
such that $\zeta_{\iota}(x)=(\id_{\delta_{V}})_{V\in I_{G}}$. Since
we know that that the unities are unique, we can first impose the
condition that $\zeta_{\iota}$ is injective after any base change.
If we regard $\zeta_{\iota}$ as a matrix over $\odi T$, this is
the locus where we have inverted the maximal minors of $\zeta_{\iota}$.
We can now assume that there exists a maximal minor of $\zeta_{\iota}$
which is invertible over $T$. This means that we can write 
\[
\prod_{V\in I_{G}}\End(\delta_{V})\otimes\odi T=\Imm(\zeta_{\iota})\oplus\shF
\]
If $(\id_{V})_{V\in I_{G}}=z\oplus z'$, then the locus where a unity
exists is exactly the zero locus of $z'\in\shF$.
\end{proof}
After having discussed the rank functions, we come back to our initial
goal, the reducibility of $\GCov$. 
\begin{lem}
\label{lem:G covers property closed for linearly reductive groups}
The stack $\GCov$ is open and closed in $(\LAlg_{R}^{G})^{\text{gr}}$.
In particular $\stZ_{G}$ is the schematic closure of $\Bi G$ in
$(\LAlg_{R}^{G})^{\text{gr}}$.\end{lem}
\begin{proof}
Follows from \ref{prop:G has locally a good representation theory},
\ref{thm:Description of LAlgGRf when G is a glrg} and \ref{thm:description of GCov with monoidal functors}.
\end{proof}
The following proposition is the key in the proof of the reducibility
of $\GCov$.
\begin{prop}
\label{prop:ind B in Z if B in Z}Let $H$ be an open and closed subgroup
scheme of $G$. Then if $\alB\in\LAlg_{R}^{H}$, we have 
\[
\ind_{H}^{G}\alB\in\stZ_{G}\iff\alB\in\stZ_{H}\comma\ind_{H}^{G}\alB\in\Bi G\iff\alB\in\Bi H
\]
In particular we have cartesian diagrams   \[   \begin{tikzpicture}[xscale=1.8,yscale=-1.0]     \node (A0_0) at (0, 0) {$\Bi H$};     \node (A0_1) at (1, 0) {$\stZ_H$};     \node (A0_2) at (2, 0) {$\HCov$};     \node (A1_0) at (0, 1) {$\Bi G$};     \node (A1_1) at (1, 1) {$\stZ_G$};     \node (A1_2) at (2, 1) {$\GCov$};     \path (A0_0) edge [->]node [auto] {$\scriptstyle{}$} (A0_1);     \path (A0_1) edge [->]node [auto] {$\scriptstyle{}$} (A1_1);     \path (A1_0) edge [->]node [auto] {$\scriptstyle{}$} (A1_1);     \path (A1_1) edge [->]node [auto] {$\scriptstyle{}$} (A1_2);     \path (A0_2) edge [->]node [auto] {$\scriptstyle{\ind^G_H}$} (A1_2);     \path (A0_0) edge [->]node [auto] {$\scriptstyle{}$} (A1_0);     \path (A0_1) edge [->]node [auto] {$\scriptstyle{}$} (A0_2);   \end{tikzpicture}   \] \end{prop}
\begin{proof}
Consider the fiber products   \[   \begin{tikzpicture}[xscale=1.8,yscale=-1.0]     \node (A0_0) at (0, 0) {$\stY$};     \node (A0_1) at (1, 0) {$\stZ$};     \node (A0_2) at (2, 0) {$(\LAlg^H_R)^{\text{gr}}$};     \node (A1_0) at (0, 1) {$\Bi G$};     \node (A1_1) at (1, 1) {$\stZ_G$};     \node (A1_2) at (2, 1) {$(\LAlg^G_R)^{\text{gr}}$};     \path (A0_0) edge [->]node [auto] {$\scriptstyle{}$} (A0_1);     \path (A0_1) edge [->]node [auto] {$\scriptstyle{}$} (A0_2);     \path (A1_0) edge [->]node [auto] {$\scriptstyle{}$} (A1_1);     \path (A1_1) edge [->]node [auto] {$\scriptstyle{}$} (A1_2);     \path (A0_2) edge [->]node [auto] {$\scriptstyle{\ind^G_H}$} (A1_2);     \path (A0_0) edge [->]node [auto] {$\scriptstyle{}$} (A1_0);     \path (A0_1) edge [->]node [auto] {$\scriptstyle{}$} (A1_1);   \end{tikzpicture}   \] 
It easy to show that $\stZ$ ($\stY$) is the substack of $(\LAlg_{R}^{H})^{\text{gr}}$
of algebras $\alB$ such that $\ind_{H}^{G}\alB\in\stZ_{G}$($\Bi G$).
In particular by \ref{lem:descend of torsors along induction}, $\stY=\Bi H$
and $\stZ_{H}$ is a closed subscheme of $\stZ$. We have to prove
that $\stZ_{H}=\stZ$. We claim that, if $S$ is a noetherian scheme
and $\alB\in\LAlg_{R}^{H}(S)$ then $\alB\in\stZ_{H}(S)$ if and only
if for any strictly Henselian ring $C$ and map $\Spec C\arr S$ we
have $\alB\otimes C\in\stZ_{H}(C)$. Indeed consider the base change
$S'=S\times_{(\LAlg_{R}^{H})^{\text{gr}}}\stZ_{H}\arr S$. This is
a closed immersion and by hypothesis its base change to any strict
Henselization of a localization of $S$ is an isomorphism. But this
implies that $S'=S$ and therefore $\alB\in\stZ_{H}(S)$.

We are now going to prove that $\stZ=\stZ_{H}$. This will conclude
the proof since $\ind_{H}^{G}$ sends $\HCov$ to $\GCov$, because
$\ind_{H}^{G}\odi{}[H]\simeq\odi{}[G]$. If $S\arrdi{\alB}\stZ$ is
an fppf atlas, then this is equivalent to $\alB\in\stZ_{H}(S)$. By
the remark above we have to show that if $\alB\in\LAlg_{R}^{H}C$,
where $C$ is a strictly Henselian ring, such that $A=\ind_{H}^{G}\alB\in\stZ_{G}(C)$
then $\alB\in\stZ_{H}(C)$. Note that if $X\arr\Spec C$ is an fppf
map we can always assume to have a section. Indeed $\alB\in\stZ_{H}(C)$
if and only if $\alB\otimes\odi X\in\stZ_{H}(X)$ and again we can
restrict to the strictly Henselian ring mapping to $X$. Let $Z$
be an fppf atlas of $\stZ_{G}$. By remark above we can assume that
$A$ comes from the atlas $Z$. Let $p\in Z$ the image of the closed
point of $\Spec C$ under the map $\Spec C\arrdi AZ$ and let $D$
be the strict Henselization of $\odi{Z,p}$. The universal object
of $Z$ induces an $A_{D}\in\stZ_{G}(D)$ and since $\Spec D\arr Z$
is flat, the open subset of $\Spec D$ where $A_{D}$ is a $G$-torsor
in schematically dense in $\Spec D$. Since $\Spec C\arr Z$ factors
through $D$ we have a map $ $$D\arr C$ such that $A_{D}\otimes C\simeq A$.
By hypothesis $A=\ind_{H}^{G}B$ and, since $C$ is strictly Henselian,
by \ref{prop:induction from a localization on henselian ring} we
can assume $B$ to be local. If $p\in\Spec A_{D}$ is the image of
the closed point of $\Spec B$ under the map $\Spec B\arr\Spec A\arr\Spec A_{D}$
by \ref{prop:induction from a localization on henselian ring} we
can write $A_{D}\simeq\ind_{H_{p}}^{G}(A_{D})_{p}$ where $H_{p}$
is the stabilizer of the connected component $\Spec(A_{D})_{p}$ in
$\Spec A_{D}$. Let $V\subseteq\Spec D$ be the open locus where $(A_{D})_{p}$
is a $H_{p}$-torsor. This is exactly the locus where $A_{D}$ is
a $G$-torsor thanks to \ref{lem:descend of torsors along induction}.
So $V$ is schematically dense in $\Spec D$. Since $(A_{D})_{p}\otimes C$
is local by \ref{cor:base change of local ring is local for strictly henselian ring}
and it is a factor of $A\simeq A_{D}\otimes C$, $\Spec(A_{D})_{p}\otimes C$
is the connected component of $\Spec A$ containing $\Spec B$. We
are in the situation   \[   \begin{tikzpicture}[xscale=3.6,yscale=-1.2]     
\node (A0_0) at (0.3, 0) {$\Spec B$};     
\node (A0_2) at (2, 0) {$\Spec (A_D)_p\otimes C$};     
\node (A1_0) at (0.3, 1) {$\Spec A$};     \node (A1_1) at (1, 1) {$\Spec A_D\otimes C$};     \node (A1_2) at (2, 1) {$\Spec \ind^G_{H_p\otimes C}(A_D)_p\otimes C$};     \path (A1_0) edge [->,bend right=48]node [auto] {$\scriptstyle{\beta}$} (A1_2);     \path (A1_0) edge [->]node [auto] {$\scriptstyle{}$} (A1_1);     \path (A0_2) edge [->]node [auto] {$\scriptstyle{j}$} (A1_2);     \path (A1_1) edge [->]node [auto] {$\scriptstyle{}$} (A1_2);     \path (A0_0) edge [->]node [auto] {$\scriptstyle{\alpha}$} (A0_2);     \path (A0_0) edge [->]node [auto] {$\scriptstyle{i}$} (A1_0);   \end{tikzpicture}   \] Since $G$ permutes the connected component of $\Spec A_{D}$, we
see that $H_{p}\otimes C$ is the stabilizer of $\Spec(A_{D})_{p}\otimes C$
and therefore $H\subseteq H_{p}\otimes C$. Since $H$ is open and
closed in $G$, we have that $B$ is a factor of $A$ and therefore
is a localization of $A$. It follows that $\alpha$ is an isomorphism
and, since $i,j,\beta$ are $H$-equivariant, that is $H$-equivariant.
So we have a $H$-equivariant isomorphism $B\simeq(A_{D})_{p}\otimes C$.
We are going to prove that $H=H_{p}\otimes C$. Since $-\otimes_{D}C$
maintains the connected components of $G$, there exists an open and
closed subgroup $H'\subseteq H_{p}$ such that $H'\otimes C=H$. In
particular we have
\[
(A_{D})_{p}^{H'}\otimes C\simeq((A_{D})_{p}\otimes C)^{H}\simeq B^{H}\simeq(\ind_{H}^{G}B)^{G}=A^{G}=C
\]
 since $A\in\stZ_{G}(C)$. Since $(A_{D})_{p}^{H'}$ is a locally
free algebra over $D$, it follows that $(A_{D})_{p}^{H'}=D$. If
$q\in V$, i.e. $(A_{D})_{p}$ is a a $H_{p}$-torsor over $q\in\Spec D$,
then the base change to $\overline{k(q)}$ of $\Spec(A_{D})_{p}$
is $H_{p}$ and $H_{p}/H'\simeq\Spec\overline{k(q)}$. So $H_{p}$
and $H'$ has the same connected components and therefore $H_{p}=H'$
and $H_{p}\otimes C=H$. It remains to prove that $(A_{D})_{p}\in\stZ_{H_{p}}(D)$.
We have cartesian diagrams   \[   \begin{tikzpicture}[xscale=1.8,yscale=-1.0]     \node (A0_0) at (0, 0) {$V$};     \node (A0_1) at (1, 0) {$Z$};     \node (A0_2) at (2, 0) {$\Spec D$};     \node (A1_0) at (0, 1) {$\Bi H_p$};     \node (A1_1) at (1, 1) {$\stZ_{H_p}$};     \node (A1_2) at (2, 1) {$(\LAlg^{H_p}_R)^{\text{gr}}$};     \path (A0_0) edge [->]node [auto] {$\scriptstyle{}$} (A0_1);     \path (A0_1) edge [->]node [auto] {$\scriptstyle{}$} (A0_2);     \path (A1_0) edge [->]node [auto] {$\scriptstyle{}$} (A1_1);     \path (A1_1) edge [->]node [auto] {$\scriptstyle{}$} (A1_2);     \path (A0_2) edge [->]node [auto] {$\scriptstyle{(A_D)_p}$} (A1_2);     \path (A0_0) edge [->]node [auto] {$\scriptstyle{}$} (A1_0);     \path (A0_1) edge [->]node [auto] {$\scriptstyle{}$} (A1_1);   \end{tikzpicture}   \] 
Since $V$ is schematically dense in $\Spec D$, it follows that $Z=\Spec D$
as required.\end{proof}
\begin{lem}
\label{lem:solvable if all subgroups are abelian}\cite{Miller1903}
A constant group whose proper subgroups are abelian is solvable.%

\end{lem}
We are ready for the proof of Theorem \ref{thm:GCov reducible when G not abelian}.
The idea is that we can reduce to a minimal non abelian subgroup of
$G$ and assume that it is constant. By lemma \ref{lem:solvable if all subgroups are abelian}
$G$ is solvable. In this case we will find a subgroup $H$ of $G$
and an algebra $\alA\in\LAlg^{H}k$ such that $\alA\notin\HCov(k)$
and $\ind_{H}^{G}\alA\in\GCov(k)$. Up to some details, if $\GCov$
is irreducible, then $|\GCov|=|\stZ_{G}|$ and therefore $\alA\in\stZ_{H}(k)\subseteq\HCov(k)$
by \ref{prop:ind B in Z if B in Z}, obtaining a contradiction.
\begin{proof}
(\emph{of Theorem }\ref{thm:GCov reducible when G not abelian}) If
the base scheme is not connected, then clearly $\GCov$ is reducible.
By \ref{rem:universally reducible over fibers} and \ref{prop:the locus where G abelian open and closed},
we can assume that $S=\Spec k$, where $k$ is a field and, since
in this case $\stZ_{G}$ is geometrically integral, we can also assume
$k=\overline{k}$. In particular $G$ is a glrg. Let $H$ be an open
and closed subgroup of $G$. We claim that if one of the following
statement is fulfilled, then $\GCov$ is reducible:
\begin{enumerate}
\item $\HCov$ is reducible
\item there exists $f\colon I_{H}\arr\N$ whose extension $f\colon\Loc^{H}k\arr\N$
is such that $f_{\R_{H}V}=\rk V$ for any $V\in I_{G}$ and there
exists $\Delta\in I_{H}$ such that $f_{\Delta}\neq\rk\Delta$
\end{enumerate}
Note that $\GCov$ is irreducible if and only if $\stZ_{G}(k)=\GCov(k)$.
Assume that $\HCov$ is reducible and, by contradiction, that $\GCov$
is irreducible. If $B\in\HCov(k)$ then $\ind_{H}^{G}B\in\GCov(k)=\stZ_{G}(k)$
and so $B\in\stZ_{H}(k)$. Therefore $\HCov$ is irreducible.

Now let $f\colon I_{H}\arr\N$ as in $2)$ and let $\delta\in\LAdd^{H}k$
be the unique $R$-linear functor such that $\delta_{\Delta}=k^{f_{\Delta}}$
for any $\Delta\in I_{H}$. Note that by hypothesis $f_{R}=1$. Consider
\[
F=\bigoplus_{R\neq\Delta\in I_{H}}\duale{\Delta}\otimes\delta_{\Delta}\comma B=k\oplus F
\]
If we set $F^{2}=0$ we obtain a structure of algebra on $B$ such
that $B\in\LAlg_{k,f}^{H}k$. We claim that $A=\ind_{H}^{G}B\in\GCov(k)$.
Indeed $\Omega^{B}=\delta$, $\Omega^{A}=\Omega^{B}\circ\R_{H}$ and
therefore
\[
\rk\Omega_{V}^{A}=\rk\Omega_{\R_{H}V}^{B}=f_{\R_{H}V}=\rk V
\]
We also claim that $A\notin\stZ_{G}(k)$, that implies that $\GCov$
is reducible. Indeed
\[
A=\ind_{H}^{G}B\in\stZ_{G}(k)\then B\in\stZ_{H}(k)\then B\in\HCov(k)
\]
by \ref{prop:ind B in Z if B in Z}, which is not the case because
there exists by hypothesis $\Delta\in I_{H}$ such that $\rk\Omega_{\Delta}^{B}=f_{\Delta}\neq\rk\Delta$.

We return now to the original statement. And we argue by induction
on the rank of $G$. As base case we take the case in which there
exists a normal and abelian subgroup $H$ of $G$ such that $G/H\simeq\Z/p\Z$
for some prime $p$. We first show how to reduce to this case. Since
$G$ is non abelian, we have that $\underline{G}$ is a non trivial
group. We start reducing to the case where $\underline{G}$ is solvable.
If $\underline{G}$ is abelian we are already in this case. Otherwise
take a minimal non abelian subgroup $K$ of $\underline{G}$. All
the proper subgroups of $K$ are abelian and therefore $K$ is solvable
thanks to \ref{lem:solvable if all subgroups are abelian}. If we
call $\phi\colon G\arr\underline{G}$, then $G'=\phi^{-1}(K)$ is
a non abelian open and closed subgroup of $G$ such that $\underline{G'}\simeq K$
is solvable and we can therefore reduce to it. Now, if $\underline{G}$
is solvable, it has a surjective map $\underline{G}\arr\Z/p\Z$. So
there exists an open and closed normal subgroup $H$ of $G$ such
that $G/H\simeq\Z/p\Z$. If $H$ is non abelian we can lower the rank
and since $\underline{H}<\underline{G}$ is solvable, we can continue
the induction reaching the claimed base case.

So assume to have a surjection $G\arr\Z/p\Z$ for some prime $p$
such that the kernel $H$ is abelian. In particular $H$ is diagonalizable
and set $N=\Hom(H,\Gm)$. We will construct an $f\colon I_{H}\arr\N$
as in $2)$. The group $G/H\simeq\Z/p\Z$ acts on $H$ and on $N=\Hom(H,\Gm)$
by conjugation as explained in \ref{Rem: action of G on diagonalizable normal subgroup}.
Let $\shR$ be a set of representatives of $N/(\Z/p\Z)$. Note that,
since $p$ is prime, an element $n\in N$ is fixed or its orbit $o(n)$
has order $p$. We claim that if $V\in I_{G}$ there exists a unique
$m\in\shR$ such that 
\[
\R_{H}V=V_{m}^{\rk V}\text{ with }|o(m)|=1\text{ or }V=\ind_{H}^{G}V_{m}\text{ with }|o(m)|=p
\]
Indeed there exists $m\in N$ such that $V\subseteq\ind_{H}^{G}V_{m}$.
Remember that, by \ref{lem:induction from diagonalizable groups},
given $n,n'\in N$ we have
\[
\R_{H}\ind_{H}^{G}V_{n}=\bigoplus_{g\in\Z/p\Z}V_{g(n)}\text{ and (}\ind_{H}^{G}V_{n}\simeq\ind_{H}^{G}V_{n'}\iff n'\in o(n))
\]
So we can assume $m\in\shR$. Moreover such an $m$ is unique since
if $V\subseteq\ind_{H}^{G}V_{m'}$, $\R_{H}V$ is the sum of some
$V_{n}$ that are in the orbits of both $m$ and $m'$. In particular,
if $|o(m)|=1$, then $\ind_{H}^{G}V_{m}=V_{m}^{p}$ and therefore
$\R_{H}V=V_{m}^{\rk V}$. So assume $|o(m)|=p$. Given $W\in\Loc^{G}k$
$(\Loc^{H}k)$ and $g\in G(k)$ call $W_{g}$ the representation of
$G$ $(H)$ that has $W$ as underlying vector space, while the action
of $G$ $(H)$ is given by $t\star x=(g^{-1}tg)x$. Note that by definition
$(V_{n})_{g}=V_{g(n)}$. In particular the multiplication by $g^{-1}$
on $V$ yields a $G$-equivariant isomorphism $V\simeq V_{g}$ and
therefore $V_{n}\subseteq\R_{H}V$ implies that $V_{g(n)}\subseteq\R_{H}V$.
Since $|o(m)|=p$ we can conclude that $V=\ind_{H}^{G}V_{m}$. Define
\[
f_{V_{n}}=\left\{ \begin{array}{cc}
|o(n)| & \text{if }n\in\shR\\
0 & \text{otherwise}
\end{array}\right.
\]
We claim that $f$ satisfies the requests of $2)$. Indeed if $V\in I_{G}$
and there exists $m\in\shR$ such that $V=V_{m}^{\rk V}$ with $|o(m)|=1$
then $f_{\R_{H}V}=\rk Vf_{V_{m}}=\rk V$. Otherwise there exists $m\in\shR$
with $|o(m)|=p$ such that
\[
V=\ind_{H}^{G}V_{m}\then f_{\R_{H}V}=\sum_{g\in\Z/p\Z}f_{V_{g(m)}}=p=\rk V
\]
Finally note that if $n\in\shR$ is such that $|o(n)|=p$ then $f_{V_{n}}=p\neq1=\rk V_{n}$.
So we have to show that such an $n$ exists. If by contradiction this
is false, then the actions of $\Z/p\Z$ on $N$ and $H$, as well
as the action of $G$ on $H$ by conjugation are trivial. So $H$
commutes with all the elements of $G$. Let $g\in G(k)\simeq\underline{G}$
be not in $H$. Any element of $G(T)$ can be written as $hg^{i}$
with $h\in H(T)$ and $0\leq i<p$. It is straightforward to check
that two such elements commute and that therefore $G$ is abelian,
which is not the case.
\end{proof}

\section{\label{sub:Regularity in codimension one}Regularity in codimension
$1$.}

In this section we want to address the following question: given a
discrete valuation ring $R$ and $A\in\LAlg^{G}R$, what are the conditions
that ensure that $A$ is regular? We will see that one of those conditions
will be that $A\in\GCov(R)$. This problem translates in the following,
more geometrical problem: given a normal, noetherian scheme $S$ and
a $G$-cover $X\arr S$, what are the conditions that ensure that
$X$ is normal too? The idea is to look at the map $\widehat{\tr}\colon A\arr\duale A$
induced by the trace map: $\widehat{\tr}$ is an isomorphism if and
only if $A$ is étale and the less degenerate $\widehat{\tr}$ is,
the more regular the algebra $A$ should be. We will explain what
this 'less' degenerate means. At the end, we will also discuss a possible
extension to covers without an action of a group.

In this section we fix a (étale) locally constant and finite group
scheme $G$ over a ring $R$ such that $\rk G\in R^{*}$. This means
exactly that $G$ is a finite and étale linearly reductive group over
$R$. We require this last condition because we want $G$-torsors
to be regular (over regular base).
\begin{notation}
If $\shP$ is a locally free sheaf and $\eta\colon\shP\otimes\shP\arr\odi T$
is a map, we will denote by $\widehat{\eta}\colon\shP\arr\duale{\shP}$
the associated map. If $\shP$ is also an algebra and $\phi\in\duale{\shP}$
we will also set $\widehat{\phi}=\widehat{\eta}$ where 
\[
\eta\colon\shP\otimes\shP\arrdi m\shP\arrdi{\phi}\odi T
\]
where $m$ is the multiplication of $\shP$. Given a basis $\beta=\{x_{1},\dots,x_{s}\}$
of $\shP$, the matrix associated with $\eta$ is $(\eta(x_{i}\otimes x_{j}))_{i,j}$,
which is also the matrix representing $\widehat{\eta}$ with respect
to the basis $\beta$ and its dual.\end{notation}
\begin{defn}
Let $\alA\in\LAlg^{G}T$. We define 
\[
\tr_{\alA/\odi T}\colon\alA\arr\odi T\text{ and }\widehat{\tr}_{\alA/\odi T}\colon\alA\arr\duale{\alA}
\]
the trace map and its associated map respectively. We also set 
\[
\shQ^{\alA/\odi T}=\Coker\widehat{\tr}_{\alA/\odi T}\text{ and }e^{\alA/\odi T}=\l(\shQ^{\alA/\odi T})
\]
where $\l$ is the length, and, if $G$ is a glrg, $\alA^{G}=\odi T$
and $V\in\Loc^{G}R$, 
\[
\shQ_{V}^{\alA/\odi T}=\Coker(\Omega_{V}^{\alA}\arrdi{\xi_{V}}\duale{(\Omega_{\duale V}^{\alA})})\text{ and }e_{V}^{\alA/\odi T}=\l(\shQ_{V}^{\alA/\odi T})
\]
where $\xi_{V}$ is the map induced by $\omega_{V}\colon\Omega_{V}\otimes\Omega_{\duale V}\arr\Omega_{V\otimes\duale V}\arr\Omega_{R}=\alA^{G}=\odi T$.
We will also omit the superscript $*/\odi T$ when it will be clear
what is the base scheme.\end{defn}
\begin{notation}
If $R$ is a local ring with residue field $k$ and $Q$ is an $R$-module,
we will say that $Q$ is defined over the closed point of $R$ if
$m_{R}Q=0$. This condition is equivalent to the fact that the map
$Q\arr Q\otimes k$ is an isomorphism or that $Q=i_{*}Q'$, where
$i\colon\Spec k\arr\Spec R$ is the closed point, for some $k$-vector
space $Q'$.
\end{notation}
The Theorem we will prove is the following:
\begin{thm}
\label{thm:equivalent conditions for regularity for glrg}Let $R$
be a DVR and $G$ be a finite and étale linearly reductive group scheme
over $R$. Let also $A\in\LAlg^{G}R$ be such that $A^{G}=R$ and
that the action of $G$ on $A$ is generically faithful over $R$
(see \ref{def: generically faithful}). Then $A$ is regular if and
only if the geometric stabilizers of $A$ are solvable and one of
the following conditions holds:
\begin{enumerate}
\item \label{enu:eA less that rank A}$e^{A}<\rk A$;
\item \label{enu:QA defined over the field}$\shQ^{A}$ is defined over
the closed point of $R$.
\end{enumerate}
In this case $A$ is generically a $G$-torsor, $A\in\stZ_{G}(R)$
and, given a closed point $p$ of $A$, its geometric stabilizer $H_{p}$
is cyclic and we have
\[
e^{A}=\rk A-|\Spec(A\otimes_{R}\overline{k})|=\rk G(1-\frac{1}{\rk H_{p}})
\]
If $G$ is a glrg, then conditions \ref{enu:eA less that rank A}
and \ref{enu:QA defined over the field} can be replaced respectively
by
\begin{enumerate}
\item [3)]$e_{V}^{A}\leq\rk V$ for all $V\in I_{G}$;
\item [4)]$\shQ_{V}^{A}$ is defined over the closed point of $R$ for
all $V\in I_{G}$.
\end{enumerate}
\end{thm}
We will define what a faithful action and a geometric stabilizer means
in this context. The proof of the above Theorem is at the end of the
section, because we prefer to collect first the necessary lemmas and
definitions. Anyway, before doing so, we want to state a global version
of Theorem above.
\begin{defn}
Let $S$ be a scheme, $Y$ be an $S$-scheme and $f\colon X\arr Y$
be a cover. We define the section $s_{f}\in(\det f_{*}\odi X)^{-2}$
as the section induced by the determinant of the trace map
\[
\widehat{\tr}_{f_{*}\odi X/\odi Y}\colon f_{*}\odi X\arr\duale{f_{*}\odi X}
\]
If $G$ acts on $X$ and $f$ is $G$-invariant set $\Omega_{V}^{f}=(f_{*}\odi X\otimes V)^{G}$
for $V\in\Loc^{G}S$. Moreover if $G$ is a glrg over $S$ and $f\in\GCov$
then, for any $V\in I_{G}$, since $\rk V=\rk\Omega_{V}^{f}$, we
define $s_{f,V}\in\det(\Omega_{V}^{f})^{-1}\otimes\det(\Omega_{\duale V}^{f})^{-1}$
as the section induced by the determinant of
\[
\Omega_{V}^{f}\arr\duale{\Omega_{\duale V}^{f}}\;\Longleftarrow\;\Omega_{V}^{f}\otimes\Omega_{\duale V}^{f}\arr\Omega_{V\otimes\duale V}^{f}\arr\Omega_{\odi S}^{f}=(f_{*}\odi X)^{G}=\odi Y
\]

\end{defn}
The following Proposition, proved in \ref{lem:trace for ring with group action},
explains the relations among the sections just introduced.
\begin{prop}
\label{prop:decomposition of discriminant} Assume that $G$ is a
glrg over $S$ and let $Y$ be an $S$-scheme and $f\colon X\arr Y\in\GCov$.
Then there exists an isomorphism
\[
(\det f_{*}\odi X)^{-2}\simeq\bigotimes_{V\in I_{G}}(\det(\Omega_{V}^{f})^{-1}\otimes\det(\Omega_{\duale V}^{f})^{-1})^{\rk V}\text{ such that }s_{f}\longmapsto\bigotimes_{V\in I_{G}}s_{f,V}^{\otimes\rk V}
\]

\end{prop}
Given a regular in codimension $1$ scheme $Y$ and a codimension
$1$ point $q$ of $Y$ we denote by $v_{q}$ the discrete valuation
associated with $\odi{Y,q}$.
\begin{thm}
\label{thm:equivalent conditions for regularity for glrg, global version}Let
$S$ be a scheme and $G$ be a finite and étale linearly reductive
group over $S$. Let also $Y$ be an integral, noetherian and regular
in codimension $1$ (resp. normal) $S$-scheme and $f\colon X\arr Y$
be a cover with a generically faithful action of $G$ on $X$ such
that $f$ is $G$-invariant and $X/G=Y$. Then the following are equivalent:
\begin{enumerate}
\item $X$ is regular in codimension $1$ (resp. normal);
\item the geometric stabilizers of the codimension $1$ points of $X$ are
solvable and for all $q\in Y^{(1)}$ we have $v_{q}(s_{f})<\rk G$.
\end{enumerate}
In this case $f$ is generically a $G$-torsor, $X\in\stZ_{G}(Y)$
and the stabilizers of the codimension $1$ points of $X$ are cyclic.
Moreover if $G$ is a glrg over $S$ the above conditions are also
equivalent to
\begin{enumerate}
\item [3)]the geometric stabilizers of the codimension $1$ points of $X$
are solvable, $f\in\GCov$ and for all $q\in Y^{(1)}$ and $V\in I_{G}$
we have $v_{q}(s_{f,V})\leq\rk V$.
\end{enumerate}
\end{thm}
Note that the geometric stabilizers are automatically solvable if
$G$ has this property. We will show how to obtain the above theorem
as corollary of Theorem \ref{thm:equivalent conditions for regularity for glrg}
as soon as we have introduced the definitions of faithful action and
geometric stabilizer for an étale group scheme.
\begin{defn}
\label{def: generically faithful} By a faithful action of a group
scheme $G$ on a scheme $X$ we mean an action such that the associated
morphism of functors $G\arr\Autsh X$ is injective. When both $G$
and $X$ are defined over a scheme $S$, we will say that the action
of $G$ on $X$ is generically faithful over $S$ if it is faithful
over a dense open subset of $S$. We will often omit to specify the
base scheme $S$ when it will be clear from the context.\end{defn}
\begin{rem}
If $G$ and $X$ are covers of a scheme $S$, then the locus in $S$
where $G$ acts faithfully on $X$ is open. Moreover, if $G$ is constant
and $S$ is integral then the action of $G$ on $X$ is generically
faithful if and only if the map of sets $G\arr\Aut X$ is injective.
Indeed, if $f\colon X\arr S$ is the structure morphism, $\Autsh X$
is a locally closed subscheme of the vector bundle $\Endsh_{S}(f_{*}\odi X)$.
In particular the kernel $H$ of the map $G\arr\Autsh X$ is a closed
subscheme of $G$, so that $H\arr S$ is a finite group scheme. In
particular the locus where the zero section $S\arr H$ is an isomorphism,
which is open, is the locus where $G$ acts faithfully on $X$. When
$G$ is constant, $S$ is integral and we write $X=\Spec\alA$ and
$k(S)$ for the field of fractions of $S$, the action is generically
faithful over $S$ if and only if $G\times k(S)\arr\Autsh(\alA\otimes k(S))$
is injective, which is equivalent to the injectivity of the map of
sets $G\arr\Aut(\alA\otimes k(S)),$ because the maps $\Aut(\alA\otimes k(S))\arr\Aut(\alA\otimes\odi U)$
are injective for all $k(S)$-schemes $U$. Finally, since $\Aut\alA\arr\Aut(\alA\otimes k(S))$
is injective, we can also conclude that $G\arr\Aut\alA$ is injective
if and only if $G\arr\Aut(\alA\otimes k(S))$ is so.\end{rem}
\begin{lem}
\label{lem:generically G torsor implies faithful action}Assume that
$R$ is reduced and let $A\in\LAlg^{G}R$. Then $A$ is generically
a $G$-torsor if and only if it is generically étale, the action of
$G$ is generically faithful and $A^{G}=R$. In this case $\rk A=|G|$
and the action of $G$ is faithful.\end{lem}
\begin{proof}
Assume that $A$ is generically a $G$-torsor. Since $G$ is étale
it is generically étale. Moreover, since $R$ is reduced and thanks
to \ref{lem:G covers property closed for linearly reductive groups},
we have $A\in\stZ_{G}(R)\subseteq\GCov(R)$ and therefore $A^{G}=R$.
To prove the faithfulness, we can assume that $A$ is the regular
representation, since the injectivity of $G\arr\Autsh X$, where $X=\Spec A$,
is an fppf local condition. The result then follows from the fact
that for the regular representation the map $G\arr\Autsh_{R}\WW(R[G])$
is injective.

For the converse we can reduce to the case where $R$ is an algebraically
closed field by looking at the generic points. In particular $G$
is constant and the claim is a classical result in the theory of étale
Galois covers.
\end{proof}
We now introduce the concept of geometric stabilizer of a point.
\begin{defn}
\label{def: geometric stabilizer}Let $\alA\in\LAlg_{R}^{G}T$ such
that $\alA^{G}=\odi T$ and $p\in\Spec\alA$. We define the geometric
stabilizer $H_{p}$ of $p$ as 
\[
H_{p}=\{g\in G_{\overline{k(p)}}\st g(\overline{p})=\overline{p}\}
\]
where $\overline{p}\in\Spec\alA\otimes\overline{k(p)}$ is the point
given by $\alA\otimes_{\odi T}\overline{k(p)}\arr\overline{k(p)}\otimes_{\odi T}\overline{k(p)}\arr\overline{k(p)}$.\end{defn}
\begin{prop}
The geometric stabilizer is invariant by base change, i.e. if we have
a cartesian diagram   \[   \begin{tikzpicture}[xscale=2.1,yscale=-0.5]     \node (A0_0) at (0, 0) {$p'$};     \node (A0_1) at (1, 0) {$p$};     \node (A1_0) at (0, 1) {$\Spec \alA'$};     \node (A1_1) at (1, 1) {$\Spec \alA$};     \node (A3_0) at (0, 3) {$T'$};     \node (A3_1) at (1, 3) {$T$};     \path (A0_0) edge [|->,gray]node [auto] {$\scriptstyle{}$} (A0_1);     \path (A1_0) edge [->]node [auto] {$\scriptstyle{}$} (A3_0);     \path (A1_0) edge [->]node [auto] {$\scriptstyle{}$} (A1_1);     \path (A3_0) edge [->]node [auto] {$\scriptstyle{}$} (A3_1);     \path (A1_1) edge [->]node [auto] {$\scriptstyle{}$} (A3_1);   \end{tikzpicture}   \] 
then $H_{p'}\simeq H_{p}$ under the map $G_{\overline{k(p')}}\arr G_{\overline{k(p)}}$.
Moreover if $G$ is constant then the image of $H_{p}$ under the
map $G_{\overline{k(p)}}\arr G$ is
\[
\{g\in G\st g(p)=p\text{ and the induced map }k(p)\arr k(p)\text{ is the identity}\}
\]
\end{prop}
\begin{proof}
Assume $G$ constant and let $K_{p}$ be the group defined in the
last part of the proposition. We need to prove that $K_{p}$ is invariant
by base change since $K_{\overline{p}}=H_{p}$ where $\overline{p}\in\Spec\alA\otimes\overline{k(p)}$
is as in \ref{def: geometric stabilizer} (also if $G$ is not necessarily
constant). If $g\in G$, we have a commutative diagram   \[   \begin{tikzpicture}[xscale=2.7,yscale=-1.0]     \node (A0_0) at (0, 0) {$k(p)$};     \node (A0_1) at (1, 0) {$k(p)\otimes_{\odi{T}} \odi{T'}$};     \node (A0_2) at (2, 0) {$k(p')$};     \node (A1_0) at (0, 1) {$k(g(p))$};     \node (A1_1) at (1, 1) {$k(g(p))\otimes_{\odi{T}} \odi{T'}$};     \node (A1_2) at (2, 1) {$k(g(p'))$};     \path (A0_0) edge [->]node [auto] {$\scriptstyle{}$} (A0_1);     \path (A0_1) edge [->]node [auto] {$\scriptstyle{}$} (A0_2);     \path (A1_0) edge [->]node [auto] {$\scriptstyle{}$} (A1_1);     \path (A0_2) edge [->]node [auto] {$\scriptstyle{\beta}$} (A1_2);     \path (A1_1) edge [->]node [auto] {$\scriptstyle{}$} (A1_2);     \path (A0_0) edge [->]node [auto] {$\scriptstyle{\alpha}$} (A1_0);     \path (A0_1) edge [->]node [auto] {$\scriptstyle{\alpha\otimes \id}$} (A1_1);   \end{tikzpicture}   \] 
where $\alpha,\beta$ are the maps induced by $g\in\Aut\alA,\Aut\alA'$
respectively. If $g\in K_{p'}$ then $g(p)=p$, $\beta=\id$, $\alpha=\beta_{|k(p)}=\id$
and therefore $g\in K_{p}$. Conversely if $g\in K_{p}$, so that
$\alpha=\id$, then $\alpha\otimes\id=\id$. In particular $g(p')=p'$
and $\beta=\id$, i.e. $g\in H_{p'}$.
\end{proof}

\begin{proof}
(\emph{proof of Theorem }\ref{thm:equivalent conditions for regularity for glrg, global version}
\emph{assuming Theorem }\ref{thm:equivalent conditions for regularity for glrg}).
First note that if $Y$ is normal then $X$ is regular in codimension
$1$ if and only if it is normal, because $f$ has Cohen-Macaulay
fibers. All the statements in the Theorem are local in the codimension
$1$ points of $X$ and $Y$. Therefore we can assume that $Y$ is
the spectrum of a DVR $R$ and that $X=\Spec A$, where $A\in\LAlg^{G}R$.
In order to conclude it is enough to note that, in this case, $e^{A}=v_{R}(s_{f})$
and, if $G$ is a glrg and $A\in\GCov$, $e_{V}^{A}=v_{R}(s_{f,V})$.
\end{proof}
Now that we have collected all the needed definitions, we can start
proving all the lemmas required for the proof of Theorem \ref{thm:equivalent conditions for regularity for glrg}.
\begin{notation}
Given $\alA\in\LAlg^{G}T$ we will denote by $\shP^{\alA/\odi T}=\Ker\tr^{\alA/\odi T}$
and by $\sigma^{\alA/\odi T}\colon\shP^{\alA/\odi T}\otimes\shP^{\alA/\odi T}\arr\odi T$
the restriction of $\tr^{\alA/\odi T}\circ m_{\alA}$. Again we will
simply write $\shP^{\alA}$ and $\sigma^{\alA}$ when $T$ will be
given.
\end{notation}
We now show how to describe the trace map in the case of an algebra
with an action of a glrg.
\begin{lem}
\label{lem:trace for ring with group action}Let $\alA\in\LAlg^{G}T$.
Then $\tr^{\alA}\colon\alA\arr\odi T$ is $G$-equivariant. If we
assume that $G$ is a glrg, that $\alA^{G}=\odi T$ and that $\rk\alA=\rk G$,
then we also have
\[
\shP^{\alA}=\Ker\tr^{\alA}=\bigoplus_{R\neq V\in I_{G}}\duale V\otimes\Omega_{V}^{\alA}\text{ and }\shQ^{\alA}=\bigoplus_{V\in I_{G}}\duale V\otimes\shQ_{V}^{\alA}
\]
Moreover Proposition \ref{prop:decomposition of discriminant} is
true.\end{lem}
\begin{proof}
In order to prove that $\tr^{\alA}$ is $G$-equivariant we can work
locally in the étale topology and assume $T$ affine and $G$ constant.
In this case it is clear that $\tr^{\alA}(g(x))=x$ for any $g\in G$
and $x\in\alA$. Now assume that $G$ is a glrg. We will have that
\[
\Ker\tr^{\alA}=\bigoplus_{V\in I_{G}}\duale V\otimes\Gamma_{V}\text{ with }\Gamma_{V}\subseteq\Omega_{V}^{\alA}
\]
Note that since $G$ is linearly reductive, $\rk\alA=\rk G$ is invertible
in $\odi T^{*}$ and therefore $\tr^{\alA}\colon\alA\arr\odi T$ is
surjective. So
\[
\odi T=\bigoplus_{V\in I_{G}}\duale V\otimes(\Omega_{V}^{\alA}/\Gamma_{V})
\]
is a $G$-equivariant decomposition and therefore $\Gamma_{V}=\Omega_{V}^{\alA}$.
Let $V,W\in I_{G}$. By construction the product of elements of $\duale V\otimes\Omega_{V}$
and $\duale W\otimes\Omega_{W}$ lies in $\Ker\tr^{\alA}$, i.e. has
no component in $\alA^{G}=\Omega_{R}^{\alA}$, except for the case
when $W=\duale V$. So the trace map $\alA\arr\duale{\alA}$ is the
direct sum of the maps induced by $\delta_{V}\colon\duale V\otimes\Omega_{V}\otimes V\otimes\Omega_{\duale V}\arr\alA\otimes\alA\arr\alA\arrdi{\tr_{\alA}}\odi T$.
We have seen that $\tr_{\alA}=(\rk G)\pi$, where $\pi$ is the projection
according to the $G$-equivariant decomposition of $\alA$. By \ref{lem:omega and the projection},
the map $\delta_{V}$ is given by
\[
\duale V\otimes\Omega_{V}\otimes V\otimes\Omega_{\duale V}\simeq\duale V\otimes V\otimes\Omega_{V}\otimes\Omega_{\duale V}\arrdi{u(e_{\duale V}\otimes\omega_{V})}\Omega_{R}
\]
where $e_{*}\colon(*)\otimes\duale{(*)}\arr R$ is the evaluation
map and $u=\rk G/\rk V$, which is invertible by \ref{lem:rk V invertible for irreducible representations}.
So the map induced by the above morphism is exactly $u(\id_{\duale V}\otimes\xi_{V})\colon\duale V\otimes\Omega_{V}^{\alA}\arr\duale V\otimes\duale{(\Omega_{\duale V}^{\alA})}$,
as required. For the last claim, if $\alA\in\GCov$, it is enough
to note that $\det\xi_{V}$ induces the section $s_{f,V}$, where
$f$ is the map $\Spec\alA\arr T$. 
\end{proof}
One of the key points in the proof of Theorem \ref{thm:equivalent conditions for regularity for glrg}
is the local case. We are going now to focus on it.
\begin{lem}
\label{lem:equivalent conditions for local ring regarding regularity}Let
$R$ be a strictly Henselian DVR, $A\in\LAlg^{G}R$ such that $A$
is local, $A^{G}=R$ and $\rk A=|G|$. Then
\[
\shQ^{A}=\Coker(\shP^{A}\arrdi{\widehat{\sigma^{A}}}\duale{(\shP^{A})})\comma\Imm\widehat{\sigma^{A}}\subseteq m_{R}\duale{(\shP^{A})}\comma m_{A}=m_{R}\oplus\shP^{A}\comma\tr_{A/R}(m_{A})\subseteq m_{R}
\]
Moreover the following conditions are equivalent
\begin{enumerate}
\item $\widehat{\sigma^{A}}$ is surjective onto $m_{R}\duale{(\shP^{A})}$;
\item $\shQ^{A}$ is defined over the closed point of $R$;
\item $e^{A}=\rk A-1$
\item $e^{A}<\rk A$.
\end{enumerate}
\end{lem}
\begin{proof}
Since $G$ is linearly reductive, $\rk A=\rk G\in R^{*}$and the trace
map $\tr^{A}$ is surjective and since $G$ is étale and $R$ strictly
Henselian, $G$ is constant. In particular $\shP^{A}$ is free and
$A=R\oplus\shP^{A}$. Moreover $m_{A}=\shP^{A}\oplus m_{R}$ thanks
to \ref{lem:trace for ring with group action} and \ref{lem:what happens on local rings}
and therefore $\tr^{A}(m_{A})\subseteq m_{R}$. This shows that $\widehat{\sigma^{A}}$
has image in $\Hom(\shP^{A},m_{R})=m_{R}\duale{(\shP^{A})}$. Since
$\widehat{\tr}^{A}\colon A\arr\duale A$ is the sum of the multiplication
$R\arrdi{\rk A}\duale R$ and the map $\widehat{\sigma^{A}}$, we
have that $\shQ^{A}=\Coker\widehat{\sigma^{A}}$ and that $\shQ^{A}$
is defined over the closed point of $R$ if and only if $m_{R}\duale{(\shP^{A})}\subseteq\Imm\sigma_{A}$,
i.e. the equivalence between $1)$ and $2)$ holds. In this case $\shQ^{A}=\duale{(\shP^{A})}/m_{R}\duale{(\shP^{A})}\simeq(R/m_{R})^{|G|-1}$,
which shows $2)\then3)\then4)$. So it remains to prove that $4)\then1)$.
Let $\pi\in R$ be a uniformizer. Since $\widehat{\sigma^{A}}$ has
image in $\pi\duale{(\shP^{A})}$ we can set $u=\widehat{\sigma^{A}}/\pi\colon\shP^{A}\arr\duale{(\shP^{A})}$
and we have to prove that $u$ is an isomorphism. Note that, by construction,
$e^{A}=v_{R}(\det\widehat{\sigma^{A}})$ and so
\[
0\leq v_{R}(\det u)=e^{A}-\rk A+1<1\then\det u\in R^{*}\then u\text{ isomorphism}
\]
\end{proof}
\begin{lem}
\label{lem:orthogonal for bilinear map}Let $R$ be a DVR, $P$ be
a free $R$-module and $\eta\colon P\otimes P\arr R$ be an $R$-linear
map. Then if $Q\subseteq P$ is a saturated submodule and we assume
that both $\widehat{\eta}$ and $\widehat{\eta_{|Q\otimes Q}}$ are
injective then $Q^{\perp}=\{x\in P\st\eta(x\otimes y)=0\text{ for all }y\in Q\}$
is saturated too and $Q\oplus Q^{\perp}=P$.%
\end{lem}
\begin{proof}
It easy to check that $Q^{\perp}$ is a submodule of $P$ which is
saturated. Moreover
\[
Q\cap Q^{\perp}=\Ker(\widehat{\eta_{|Q\otimes Q}})=0\then Q\oplus Q^{\perp}\subseteq P
\]
Since both $Q$ and $Q^{\perp}$ are saturated, we have only to prove
that $\rk Q^{\perp}=\rk P-\rk Q$. Consider the diagram   \[   \begin{tikzpicture}[xscale=1.5,yscale=-1.0]     \node (A0_0) at (0, 0) {$0$};     \node (A0_1) at (1, 0) {$Q^\perp$};     \node (A0_2) at (2, 0) {$P$};     \node (A0_3) at (3, 0) {$P/Q^\perp$};     \node (A0_4) at (4, 0) {$0$};     \node (A1_0) at (0, 1) {$0$};     \node (A1_1) at (1, 1) {$\duale{P/Q}$};     \node (A1_2) at (2, 1) {$\duale P$};     \node (A1_3) at (3, 1) {$\duale Q$};     \node (A1_4) at (4, 1) {$0$};     \path (A0_1) edge [->,dashed]node [auto] {$\scriptstyle{\alpha}$} (A1_1);     \path (A0_0) edge [->]node [auto] {$\scriptstyle{}$} (A0_1);     \path (A0_1) edge [->]node [auto] {$\scriptstyle{}$} (A0_2);     \path (A1_0) edge [->]node [auto] {$\scriptstyle{}$} (A1_1);     \path (A0_3) edge [->,dashed]node [auto] {$\scriptstyle{\beta}$} (A1_3);     \path (A1_1) edge [->]node [auto] {$\scriptstyle{}$} (A1_2);     \path (A0_3) edge [->]node [auto] {$\scriptstyle{}$} (A0_4);     \path (A0_2) edge [->]node [auto] {$\scriptstyle{\widehat \eta}$} (A1_2);     \path (A1_2) edge [->]node [auto] {$\scriptstyle{}$} (A1_3);     \path (A0_2) edge [->]node [auto] {$\scriptstyle{}$} (A0_3);     \path (A1_3) edge [->]node [auto] {$\scriptstyle{}$} (A1_4);   \end{tikzpicture}   \] 
The first row is exact since $Q^{\perp}$ is saturated, the second
one because $Q$ is so. Note that $Q^{\perp}=\Ker(P\arrdi{\widehat{\eta}}\duale P\arr\duale Q)$
and therefore the maps $\beta$ and $\alpha$ are well defined and
$\beta$ is injective. By the snake lemma we get that $\rk\Coker\beta=0$
and the desired equality. %
\end{proof}
\begin{rem}
\label{rem: generically etale means eA less infty of rk QA zero}If
$R$ is a DVR and $A\in\LAlg^{G}R$ then $A$ is generically étale
if and only if $e^{A}<\infty$ or $\rk\shQ^{A}=0$. Indeed those conditions
are all equivalent to the condition that $\widehat{\tr}_{A\otimes k(R)}$
is an isomorphism.
\end{rem}
The following lemma is the hard part in the proof of Theorem \ref{thm:equivalent conditions for regularity for glrg}.
\begin{lem}
\label{lem:regularity for local rings with action of solvable groups}Assume
that $G$ is a solvable group and let $R$ be a strictly Henselian
DVR. Let also $A\in\LAlg^{G}R$ be a local algebra such that $A^{G}=R$
and that the action of $G$ on $A$ is generically faithful. Then
$A$ is a DVR if one of the following conditions holds:
\begin{itemize}
\item $\shQ^{A}$ is defined over the closed point of $R$;
\item $e^{A}<\rk A$.
\end{itemize}
\end{lem}
\begin{proof}
Note that if one of the above conditions is satisfied, then $A$ is
generically étale and so generically a $G$-torsor thanks to \ref{rem: generically etale means eA less infty of rk QA zero}
and \ref{lem:generically G torsor implies faithful action}. In particular
$\rk A=\rk G$ and by \ref{lem:equivalent conditions for local ring regarding regularity}
the two conditions in the statement are equivalent. We will make use
of \ref{lem:what happens on local rings} and we will argue by induction
on $\rk G$. Note that $G$ is constant and that by \ref{thm:etale linearly reductive over sctrictly are glrg}
$G$ is a glrg. We first consider the base case $G\simeq\Z/p\Z$,
where $p$ is a prime. There exists a basis $\{v_{i}\}_{i\in G^{*}}$
of $A$ such that $v_{i}v_{j}=\psi_{i,j}v_{i+j}$, with $\psi_{i,j}\in R$.
Set $e_{i,j}=v_{R}(\psi_{i,j})$, where $v_{R}$ is the valuation
of $R$. The associativity conditions yield relations 
\begin{equation}
\psi_{n,t}\psi_{n+t,s}=\psi_{t,s}\psi_{t+s,n}\then e_{n,t}+e_{n+t,s}=e_{t,s}+e_{t+s,n}\then e_{n,-n}=e_{-n,s}+e_{s-n,n}\label{eq:for the regularity of local algebra when abelian}
\end{equation}
From \ref{lem:trace for ring with group action}, we have 
\[
e^{A}=\sum_{i\neq0}e_{i,-i}
\]
Since $A$ is local we have that $e_{i,-i}>0$ for all $i\neq0$.
In particular $e^{A}<p$ implies that $e_{i,-i}=1$ for all $i\neq0$.
From (\ref{eq:for the regularity of local algebra when abelian})
we see that $e_{n,t}\in\{1,0\}$ for all $n,t$. Note that $m_{A}/m_{A}^{2}$
is $G$-equivariant and by contradiction assume that $\dim_{k}m_{A}/m_{A}^{2}>1$.
Since $e_{n,-n}=1$, $v_{n}v_{-n}$ generates the maximal ideal $m_{R}$
of $R$ and therefore $m_{R}A\subseteq m_{A}^{2}$ and $(m_{A}/m_{A}^{2})_{0}=0$.
Moreover $\dim_{k}(m_{A}/m_{A}^{2})_{n}\in\{1,0\}$ for all $n$,
because, if $n\neq0$, $v_{n}\in m_{A}$ and $m_{R}v_{n}\subseteq m_{A}^{2}$.
So there exists $n\neq t\in G^{*}$ such that $v_{n},v_{t}\notin m_{A}^{2}$
and $n,t\neq0$. Note that if $u+s=n$, with $u,s\neq0$, then $e_{u,s}=1$
since otherwise
\[
v_{u}v_{s}=\psi_{u,s}v_{n}\text{ with }\psi_{u,s}\in R^{*}\then v_{n}\in m_{A}^{2}
\]
and similarly if $u+s=t$. So set $u=n-t$ and consider the relation
\[
e_{-u,t+u}+e_{t,u}=e_{u,-u}=1
\]
obtained from (\ref{eq:for the regularity of local algebra when abelian}).
We have that $-u,t+u=n,t,u$ are all non zero and by the remark above
we get
\[
e_{-u,t+u}=e_{t,u}=1\then2=1
\]
which is impossible.

We now come back to the general case. If $G$ is simple, then we are
in the base case. Otherwise take a normal subgroup $0\neq H\neq G$
of $G$ and let $B=A^{H}$.

We want first to prove that $B$ is a DVR using the inductive hypothesis
on $B/R$. Since $H$ is normal, we have that $B\in\LAlg^{G/H}R$
and that it is generically a $G/H$-torsor. From \ref{lem:functor of algebra of invariants by a normal subgroup}
and \ref{lem:trace for ring with group action} we have that
\[
\shP^{B}=\bigoplus_{R\neq W\in I_{G/H}}\duale W\otimes\Omega_{W}^{A}\subseteq\shP^{A}\cap B=\bigoplus_{R\neq V\in I_{G}}(\duale V)^{H}\otimes\Omega_{V}^{A}
\]
Since $B=R\oplus\shP^{B}\subseteq R\oplus\shP^{A}\cap B=B$ we can
conclude that $\shP^{B}=\shP^{A}\cap B$. Moreover $B$ is local.
Indeed if $x\in B\cap A^{*}$ and $y\in A$ is such that $xy=1$ we
will have
\[
1=\prod_{h\in H}h(xy)=x^{|H|}(\prod_{h\in H}h(y))\then x\in B^{*}
\]
We can apply \ref{lem:orthogonal for bilinear map}, since $\shP^{B}$
is saturated in $\shP^{A}$ and both $\widehat{\sigma^{A}}$ and $\widehat{\sigma^{B}}$
are injective since $A$ and $B$ are generically étale. So $\shP^{A}=\shP^{B}\oplus(\shP^{B})^{\perp}$.
Moreover $\widehat{\sigma^{B}}$ and $\widehat{\sigma^{A}}_{|\shP^{B}}$
differ only by the multiplication of $\rk A/\rk B=|G|/|H|\in R^{*}$
since $\shP^{B}=\shP^{A}\cap B$. We can conclude that $\shQ^{B/R}$
is a submodule of $\shQ^{A/R}$ and it is therefore defined over the
closed point of $R$. We can now apply the inductive hypothesis on
$B/R$ and conclude that $B$ is a DVR.

We are going now to apply inductive hypothesis on $A/B$. Note that
$B$ is strictly Henselian since $B/m_{B}=R/m_{R}$ and $B/R$ is
finite. We clearly have that $A$ is generically a $H$-torsor on
$B$. Since $k(B)\simeq B\otimes_{R}k(R)$, $A\otimes_{R}k(R)$ is
free over $k(B)$ and $A\subseteq A\otimes_{R}k(R)$, we see that
$A$ is a $B$-module without torsion and therefore free. This means
that $A\in\LAlg^{H}B$. In order to apply the inductive hypothesis
on $A$ and conclude that it is a DVR, we have to show that the image
of the map $\widehat{\tr}_{A/B}\colon A\arr\Hom_{B}(A,B)$ contains
$m_{B}\Hom_{B}(A,B)$, thanks to \ref{lem:equivalent conditions for local ring regarding regularity}.
Since $A$ is free over $B$ and $B$ is free over $R$ we have the
relations
\[
\tr_{B/R}\circ\tr_{A/B}=\tr_{A/R}\then\psi\circ\widehat{\tr}_{A/B}=\widehat{\tr}_{A/R}
\]
where $\psi\colon\Hom_{B}(A,B)\arr\Hom_{R}(A,R)$ is the map induced
by $\tr_{B/R}\colon B\arr R$. We start proving that $\psi$ is injective.
Let $\phi\colon A\arr B$ be such that $\psi(\phi)=0$. This means
that $\Imm\phi\subseteq\shP^{B}=\Ker\tr_{B/R}$. If $\Imm\phi\neq0$,
since it is an ideal of $B$, we will have $\Imm\phi=m_{B}^{t}$ for
some $t$. In particular $\Imm\phi\cap R\neq0$, while we know that
$\shP^{B}\cap R=0$. So it remains to prove that if $y\in m_{B}$
and $\phi\in\Hom_{B}(A,B)$ then $\xi=\psi(y\phi)\in\Imm\widehat{\tr}_{A/R}$.
Remember that $m_{R}\Hom_{R}(A,R)$ is contained in $\Imm\widehat{\tr}_{A/R}$
since $\shQ^{\alA}$ is defined over the closed point of $R$. Let
$\pi\in R$ be an uniformizing element. We have
\[
\forall x\in A\;\xi(x)=\psi(y\phi)(x)=\tr_{B/R}(y\phi(x))\in m_{R}\then\xi=\pi(\xi/\pi)\in m_{R}\Hom_{R}(A,R)
\]
since $\tr_{B/R}(m_{B})\subseteq m_{R}$, thanks to \ref{lem:equivalent conditions for local ring regarding regularity}.\end{proof}
\begin{lem}
\label{lem:passing to the strict henselization}Assume that $R$ is
a DVR, $A\in\LAlg^{G}R$ and call $R^{sh}$ the strict Henselization
of $R$. Then $A$ is regular (generically a $G$-torsor) if and only
if $A\otimes_{R}R^{sh}$ is so.\end{lem}
\begin{proof}
$A\otimes_{R}R^{sh}$ is faithfully flat over $A$ and it is a direct
limit of étale extensions of $A$. In particular $\Spec A\times_{R}R^{sh}\arr\Spec A$
is surjective and the dimensions of the tangent spaces remain constant.
So $A$ is regular if and only if $A\otimes_{R}R^{sh}$ is so. Since
$R^{sh}$ is a domain and the condition of being a $G$-torsor is
open, we get that $A$ is generically a $G$-torsor if and only if
$A\otimes_{R}R^{sh}$ is so.\end{proof}
\begin{lem}
\label{lem:regular implies cyclic stabilizer}Assume that $R$ is
a DVR and let $A\in\LAlg^{G}R$ such that $A^{G}=R$ and that the
action of $G$ is generically faithful. If $A$ is regular then it
is generically a $G$-torsor and the geometric stabilizer of a closed
point of $A$ is cyclic.\end{lem}
\begin{proof}
Thanks to \ref{lem:passing to the strict henselization} we can pass
to the strict Henselization of $R$ and assume that $G$ is constant.
Let $p\in\Spec A$ be a closed point, $k$ be the residue field of
$R$ and $H$ be the stabilizer of $\Spec A_{p}$ in $\Spec A$. By
\ref{prop:induction from a localization on henselian ring}, we know
that $\ind_{H}^{G}A_{p}\simeq A$. In particular $A_{p}\in\LAlg^{H}R$,
it is a DVR and $A_{p}^{H}=R$. From \ref{lem:what happens on local rings}
we see that $k(p)=k$ and therefore $H$ is the geometric stabilizer
of $p$ and it is given by $H=\{g\in G\st g(p)=p\}$. Moreover $A_{p}\otimes k(R)$
is a field such that $(A_{p}\otimes k(R))^{H}=k(R)$ and therefore
a separable extension of $k(R)$. So $A\otimes k(R)/k(R)$ is étale
and therefore a $G$-torsor thanks to \ref{lem:generically G torsor implies faithful action}.
We can conclude that $A_{p}$ is generically a $H$-torsor by \ref{prop:ind B in Z if B in Z}
and, again by \ref{lem:generically G torsor implies faithful action},
that $H\subseteq\Aut_{R}A_{p}$. Consider the map
\[
H\arr\Aut_{k}m_{p}/m_{p}^{2}\simeq k^{*}
\]
We will prove that it is injective, which implies that $H$ is cyclic.
Let $h\in H\subseteq\Aut_{R}A_{p}$ such that $h_{|m_{p}/m_{p}^{2}}=\id$
and $\pi$ be a uniformizing of $A_{p}$. We first show that $h(\pi)=\pi$.
If $h(\pi)\neq\pi$, we can write 
\[
h(\pi)=\pi+u\pi^{k}\text{ mod }m_{p}^{k+1}\text{ with }u\in R^{*}\comma k>1
\]
An easy induction shows that $h^{n}(\pi)=\pi+nu\pi^{k}\text{ mod }m_{p}^{k+1}$.
If $n$ is the order of $h$, so that $h^{n}(\pi)=\pi$, we will have
$nu=0$ and therefore $u=0$ since $\car k\nmid n$. So $h_{|m_{p}}=\id$.
If $a\in A_{p}^{*}$ then there exists $r\in R$ such that $a-r\in m_{p}$
and we have
\[
h(a)=h(a-r)+h(r)=a-r+r=a
\]

\end{proof}
We are now ready for the proof of \ref{thm:equivalent conditions for regularity for glrg}.
\begin{proof}
(\emph{of Theorem} \ref{thm:equivalent conditions for regularity for glrg})
We first consider the last part of the statement, i.e. the case when
$G$ is a glrg. From \ref{lem:trace for ring with group action} we
have that 
\[
\shQ^{A}=\bigoplus_{V\in I_{G}}\duale V\otimes\shQ_{V}^{A}\text{ and }e^{A}=\sum_{V\in I_{G}}\rk V\cdot e_{V}^{A}
\]
In particular $2)\iff4)$ and, thanks to \ref{rem: generically etale means eA less infty of rk QA zero}
and \ref{lem:generically G torsor implies faithful action}, we also
have that any of the conditions $1),2),3),4)$ implies that $A$ is
generically a $G$-torsor and $A\in\GCov R$. In particular $\rk A=\rk G\in R^{*}$,
$\shQ_{R}^{A}=0$ and $\rk\Omega_{V}^{A}=\rk V$. The existence of
a surjective map $\duale{\Omega_{\duale V}^{A}}\arr\shQ_{V}^{A}$
tells us that $4)\then3)$, while $\shQ_{R}^{A}=0$ tells us that
$3)\then1)$.

We now consider the general case. Note that the length does not change
when passing to the strict Henselization of $R$ and an $R$-module
is schematically supported on the closed point of $R$ if and only
if it is so on the strict Henselization of $R$. Thanks to \ref{lem:passing to the strict henselization},
we can assume that $R$ is strictly Henselian and therefore that $G$
is constant. We want now to reduce to the case in which $A$ is local.
Denote by $k$ the residue field of $R$ and let $p\in\Spec A$ be
a closed point and $H$ be the stabilizer of the connected component
$\Spec A_{p}$ of $\Spec A$. By \ref{prop:induction from a localization on henselian ring},
we know that $\ind_{H}^{G}A_{p}\simeq A$. In particular $A_{p}^{H}=R$,
$H$ is the geometric stabilizer $H_{p}$ of $p$ since $k(p)=k$
by \ref{lem:what happens on local rings}, $A_{p}\in\LAlg^{H}R$ and
$A$ is generically a $G$-torsor if and only if $A_{p}$ is generically
a $H$-torsor. Since $R$ is strictly Henselian we have that 
\[
A=\prod_{q\in\Spec_{m}A}A_{q}\comma\shQ^{A}=\bigoplus_{q\in\Spec_{m}A}\shQ^{A_{q}}\comma e^{A}=\sum_{q\in\Spec_{m}A}e^{A_{q}}=|\Spec_{m}A|e^{A_{p}}
\]
where the last equality holds since $A_{q}\simeq A_{p}$ for any $q\in\Spec_{m}A$,
thanks to \ref{lem:G acts transetively on the maximal ideals of fibers}.
Note also that 
\[
|G|/|H|=|\Spec_{m}A|=|\Spec A\otimes_{R}k|=|\Spec A\otimes_{R}\overline{k}|
\]
since $k(q)=k$ for any $q\in\Spec_{m}A$. We can therefore assume
$A$ to be local and generically a $G$-torsor and that $G$ is solvable.
Lemma \ref{lem:regularity for local rings with action of solvable groups}
assure us that if the conditions in the statement are fulfilled then
$A$ is a DVR and $e^{A}=\rk A-1$. So assume that $A$ is regular
and therefore that $G$ is cyclic, thanks to \ref{lem:regular implies cyclic stabilizer}.
In this case 
\[
A=\bigoplus_{n\in G^{*}}Rv_{n}\comma v_{n}v_{t}=\psi_{n,t}v_{n+t'}\text{ with }\psi_{n,t}\in R-\{0\}\text{ and }m_{A}=\bigoplus_{n\neq0}Rv_{n}
\]
From \ref{thm:regular in codimension 1 covers}, we can write $\psi_{n,t}=\lambda_{n.t}z^{\E_{n,t}}$
with $\lambda_{n,t}\in R^{*}$, $z\in m_{R}-m_{R}^{2}$ and $\E$
is as in \ref{pro:Rita's smooth integral extremal rays}. In particular
$\E_{m,-m}=1$ for all $0\neq m\in G^{*}$. If $V_{m}$ is the one
dimensional representation associated with $m\in G^{*}$ we see that
$\shQ_{V_{m}}^{A}=R/(z^{\E_{m,-m}})$. In particular condition $4)$,
and therefore all the others thanks to the initial discussion about
glrg, are satisfied.\end{proof}
\begin{rem}
In \ref{thm:equivalent conditions for regularity for glrg} the conditions
\ref{enu:eA less that rank A} $e^{A}<\rk A$ and \ref{enu:QA defined over the field}
$\shQ^{A}$ defined over the closed point of $R$ do not depend on
the action of $G$. This suggests the following question: given a
finite and flat algebra $A$ over a DVR, is it true that $A$ is a
generically étale regular algebra if and only if it satisfies \ref{enu:eA less that rank A}
or \ref{enu:QA defined over the field}?

The answer to that question is negative for condition \ref{enu:eA less that rank A}.
Indeed take a not regular local $R$-algebra $B$, a local regular
one $C$ with $e^{C}=\rk C-1$ and define $A=B\times C^{r}$. The
algebra $A$ is not regular but
\[
e^{A}=e^{B}+r\rk C-r<\rk B+r\rk C=\rk A\iff r>e^{B}-\rk B
\]
Anyway this is not a satisfying answer, since it is clear that in
this general situation a right condition on $e^{A}$ has to take in
consideration the value of $|\Spec A\otimes_{R}\overline{k}|$, where
$k$ is the residue field of $R$. So one can change condition \ref{enu:eA less that rank A}
with the more strong $e^{A}\leq\rk A-|\Spec A\otimes_{R}\overline{k}|$.

In general also the converse is false, if we do not assume some condition
on $\rk A$. Consider for example a DVR $R$ with uniformizer $\pi$
such that $0\neq2\in m_{R}$ and consider
\[
f(x)=x^{2}-2x-\pi\text{ and }A=R[x]/(f)
\]
I claim that $A$ is regular, generically étale over $R$ with residue
field $R/m_{R}$ but $e^{A}\geq\rk A=2$. It is generically étale
since $A\otimes k(R)$ is either $k(R)^{2}$ or a field, in which
case is a separable extension of $k(R)$ because $f$ is separable.
It is local because $A\otimes k=k[x]/(x^{2})$ and it is regular because
its maximal ideal $(\pi,x)$ is clearly generated by $x$. On the
other hand
\[
\tr_{A}x=2\comma\tr_{A}x^{2}=4+2\pi\then\det\widehat{\tr}_{A}=4(1+\pi)\then e^{A}=v_{R}(\det\widehat{\tr}^{A})=2v_{R}(2)\geq2
\]

There are other possible remarks in this direction and, at the end,
a very reasonable conjecture is the following\end{rem}
\begin{conjecture}
Let $R$ be a DVR with residue field $k$ and $A$ be a finite and
flat $R$-algebra. Then
\[
e^{A}\geq\rk A-|\Spec A\otimes_{R}\overline{k}|
\]
and the following conditions are equivalent:
\begin{enumerate}
\item $A$ is regular, generically étale with separable residue fields and
the localizations of $A\otimes_{R}\overline{k}$ have ranks prime
to the characteristic $\car k$;
\item the equality holds in the inequality above;
\item $\shQ^{A}$ is defined over the closed point of $R$.
\end{enumerate}
\end{conjecture}
Currently I am able to prove the inequality above, the equivalence
of $2)$ and $3)$, the implication $1)\then2)$ and $2)\then1)$,
except for the regularity of $A$.

\chapter{$(\mu_{3}\rtimes\Z/2\Z)$-covers and $S_{3}$-covers.}

In this chapter we describe Galois covers for the group $G=\mu_{3}\ltimes\Z/2\Z$
defined over the ring $\stR=\Z[1/2]$ and for the group $S_{3}$ defined
over $\stR_{3}=\Z[1/6]$. This is a summary of the division of the
chapter.

\emph{Section 1. }This section is dedicated to the description of
the representation theory of the group $G$ over the ring $\stR$.
We will show that $G$ is a good linearly reductive group, that $\Bi G\simeq\Bi S_{3}$
over $\stR_{3}$ and we will describe the geometrically irreducible
representations of $G$ and their tensor products.

\emph{Section 2. }We will describe the global data needed to define
a $G$-cover. Such data will be given in terms of linear algebra,
that is as a collection of locally free sheaves and maps between them
satisfying the commutativity of certain diagrams. We will express
those conditions in terms of local equations.

\emph{Section 3. }In this section we describe the geometry of $\GCov$
and $\RCov{S_{3}}$ and families of $G$-covers with additional properties.
We will prove that $\GCov$ ($\RCov{S_{3}}$ over $\stR_{3}$) has
exactly two irreducible components and we will describe them in terms
of vanishing of maps between coherent sheaves. Such results will require
the study of particular open substacks of $\GCov$ and in particular
of $\Bi G$.

\emph{Section 4. }We will give a characterization of regular $G$-covers
and regular $S_{3}$-covers in terms of properties of closed subschemes
of the base associated with the data defining them. In particular
we will prove an equivalence between regular $G$-covers, regular
$S_{3}$-covers and regular triple covers satisfying a codimension
$2$ condition. We will then show how it is possible to construct
such covers and we will compute the invariants of the total space
of a regular $S_{3}$-cover of a smooth surface over an algebraically
closed field.

\section{Preliminaries and notation.}

In this chapter we will work over the ring $\shR$ of integers with
$2$ inverted, that is $\stR=\Z[1/2]$. Sometimes we will also need
to have $3$ invertible in the base scheme and we will denote by $\stR_{3}$
the ring of integers with $6$ inverted, that is $\stR_{3}=\Z[1/6]$.
In all the chapter the symbol $G$ will denote the group scheme $G=\mu_{3}\rtimes\Z/2\Z$
defined over $\stR$, where the action of $\Z/2\Z$ on $\mu_{3}$
or, equivalently, $\Z/3\Z$ is given by the inversion. Note that,
in this case, $\mu_{2}\simeq\Z/2\Z$. Moreover, over $\stR_{3}[x]/(x^{2}+x+1)$,
we have $G\simeq S_{3}$. In particular, by \ref{ex:SThree covers and the other group}
and \ref{thm:bitorsors and GCov}, we have isomorphisms
\[
\Bi G\simeq\Bi S_{3}\text{ and }\GCov\simeq\RCov{S_{3}}
\]
over the ring $\stR_{3}$. Therefore the study of $G$-covers coincides
with the study of $S_{3}$-covers over $\stR_{3}$. We have preferred
to study $G$-covers directly, because the representation theory of
$G$ has a simpler description and $G$, as we will see below, is
a glrg over $\stR$, while $S_{3}$ is not even linearly reductive
over this base. On the other hand we will remark the results for $G$-covers
that can be traduced in results for $S_{3}$-covers.

We want to describe the representation theory of $G$. Set $\sigma\in\Z/2\Z(\stR)$
for the generator of $\Z/2\Z$. We will also think of $\sigma$ as
an element of $G(\stR)$ and as a given transposition of $S_{3}(\stR)$.
Set also $V_{0}=\shR,V_{1},V_{2}$ for the representations of $\mu_{3}$
corresponding to its characters in $\Z/3\Z$. Moreover consider the
set $I_{G}$ of $G$-representations

\[
\shR\comma A=V_{\chi}\comma V=\ind_{\mu_{3}}^{G}V_{1}
\]
where $\chi\colon G\arr\Gm$ is induced by the non trivial character
of $\Z/2\Z$. Since $2$ is invertible, it is easy to check that $(G,I_{G})$
is a good linearly reductive group over $\shR$: it is linearly reductive
because extension of a diagonalizable group and an étale constant
group of order invertible in $\stR$ and it has a good representation
theory because the representations in $I_{G}$ restrict to the irreducible
representations of $S_{3}$ over the geometric point $\Spec\overline{\Q}\arr\Spec\stR$
(see \ref{prop:generating irreducible representations}). We will
consider the following basis $1\in\shR$, $1_{A}\in A$ and $v_{1},v_{2}\in V$
such that $v_{i}\in V_{i}$. Moreover since $\sigma$ exchanges $V_{1}$
and $V_{2}$, we will also assume that $\sigma(v_{1})=v_{2},\sigma(v_{2})=v_{1}$.
Now we describe the tensor products of the representations in $I_{G}$.
We have
\[
A\otimes A\simeq\stR\comma1_{A}\otimes1_{A}\arr1\text{ and }A\otimes V\simeq V\comma1_{A}\otimes v_{1}\arr-v_{1},1_{A}\otimes v_{2}\arr v_{2}
\]
and, if we set $v_{ij}=v_{i}\otimes v_{j}\in V\otimes V$, 
\[
\stR\oplus A\oplus V\simeq V\otimes V\comma1\arr v_{12}+v_{21},1_{A}\arr v_{21}-v_{12},v_{1}\arr v_{22},v_{2}\arr v_{11}
\]
Finally note that the $G$-equivariant projection $V\otimes V\arr\stR\comma v_{ij}\arr1-\delta_{ij}$,
where $\delta_{ij}$ is the Kronecker symbol, yields an isomorphism
\[
V\simeq\duale V\comma v_{1}\arr v_{2}^{*},v_{2}\arr v_{1}^{*}
\]

Since we will have to deal with locally free sheaves of rank two,
we recall here the following fact about them.
\begin{rem}
If $\shF$ is a locally free sheaf of rank $2$ over a scheme $T$,
the canonical map $\shF\otimes\shF\arr\det\shF$ induces an isomorphism
\[
\shF\simeq\duale{\shF}\otimes\det\shF
\]
If $y,z$ is a basis of $\shF$, then the above map is given by $y\arr-z^{*}\otimes(y\wedge z)$,
$z\arr y^{*}\otimes(y\wedge z)$.
\end{rem}
In this chapter, we will often prove statements valid over any scheme
and, in order to simplify the reading, the letter $T$, if not stated
otherwise, will denote a scheme over the given base, that is $\stR$
or $\stR_{3}$.

\section{\label{sec:Global description of Sthree covers}Global description
of $(\mu_{3}\rtimes\Z/2\Z)$-covers.}

In this section we want to describe the data needed to define a $G$-cover
over any $\stR$-scheme. The idea will be to use the results of the
previous chapter and describe particular lax, symmetric monoidal functors
$\Loc^{G}\stR\arr\Loc T$. We first introduce such data and then we
will show their relationship with $G$-covers. We remark here that
the global description obtained here, although with a different notation,
has already been introduced in \cite{Easton2008}.
\begin{notation}
In this section only, by an $\odi T$-algebra or a sheaf of $\odi T$-algebras
we will mean a locally free sheaf of (non associative) rings $\alA$
over $T$ with a unity $1\in\alA$. 
\end{notation}
We define the stack $\stY$ over $\stR$ whose objects are sequences
$\chi=(\shL,\shF,m,\alpha,\beta,\la-,-\ra)$ where: $\shL$ is an
invertible sheaf, $\shF$ is a rank $2$ locally free sheaf and $m,\alpha,\beta,\la-,-\ra$
are maps
\[
\shL^{2}\arrdi m\odi T\comma\shL\otimes\shF\arrdi{\alpha}\shF\comma\Sym^{2}\shF\arrdi{\beta}\shF\comma\det\shF\arrdi{\la-,-\ra}\shL
\]
With an object $\chi\in\stY$ as above we associate the map $(-,-)_{\chi}\colon\shF\otimes\shF\arr\odi T$
given by
\[
(-,-)_{\chi}\colon\shF\otimes\shF\simeq\duale{\shF}\otimes\det\shF\otimes\shF\arrdi{\id\otimes\la-,-\ra\otimes\id}\duale{\shF}\otimes\shL\otimes\shF\arrdi{\id\otimes\alpha}\duale{\shF}\otimes\shF\arr\odi T
\]
where we are using the canonical isomorphism $\shF\simeq\duale{\shF}\otimes\det\shF$.
Notice that, although we are using the symbol $(-,-)$ of a symmetric
product, $(-,-)_{\chi}$ is not necessarily symmetric. Moreover we
also associate with $\chi$ the maps $\gamma_{\chi},\gamma_{\chi}'\colon\shF\otimes\shF\arr\odi T\oplus\shL$
given by
\[
\gamma_{\chi}=(-,-)_{\chi}+\la-,-\ra\comma\gamma'_{\chi}=(-,-)_{\chi}-\la-,-\ra
\]
When $\chi$ is given, we will simply write $(-,-),\gamma,\gamma'$
or $\chi=(\shL,\shF,m,\alpha,\beta,(-,-),\la-,-\ra)\in\stY$. Moreover
we set

\[
\alA_{\chi}=\odi T\oplus\shL\oplus\shF_{1}\oplus\shF_{2}\text{ with }\shF_{1}=\shF_{2}=\shF
\]

\begin{prop}
\label{prop:Achi is G-equivariant}Given $\chi\in\stY$ as above,
the sheaf $\alA_{\chi}$ has a unique $G$-comodule structure such
that $\shF_{0}=\odi T\oplus\shL,\shF_{1},\shF_{2}$ define the $\mu_{3}$-action
and $\sigma$ acts as $-\id_{\shL}$ on $\shL$ and induces $\id_{\shF}\colon\shF_{1}\arr\shF_{2}$.
\end{prop}
This Proposition will be proved in the next section. We endow $\alA_{\chi}$
with a structure of sheaf of $\odi T$-algebras given by the maps
\[
\shL^{2}\arrdi m\odi T\comma\shF_{1}\otimes\shL\simeq\shL\otimes\shF_{1}\arrdi{\alpha}\shF_{1}\comma\shF_{2}\otimes\shL\simeq\shL\otimes\shF_{2}\arrdi{-\alpha}\shF_{2}
\]
\[
\shF_{1}\otimes\shF_{1}\arrdi{\beta}\shF_{2}\comma\shF_{2}\otimes\shF_{2}\arrdi{\beta}\shF_{1}\comma\shF_{1}\otimes\shF_{2}\arrdi{\gamma_{\chi}}\odi T\oplus\shL\comma\shF_{2}\otimes\shF_{1}\arrdi{\gamma'_{\chi}}\odi T\oplus\shL
\]
We are implicitly assuming that the maps $\odi T\otimes\alA_{\chi},\alA_{\chi}\otimes\odi T\arr\alA_{\chi}$
are just the usual isomorphisms, or, in other words, that $1\in\odi T$
is a unity for $\alA_{\chi}$.

We want now to give a list of equations involving the maps $m,\alpha,\beta,\la-,-\ra$,
which we will show are the relationships needed for the associativity
of $\alA_{\chi}$. Such equations will be 'local' relations and therefore
we introduce the following notation:
\begin{notation}
\label{not: local notation for Sthree}When we fix a generator $t$
of $\shL$, the maps $m,\alpha,\beta,\la-,-\ra$ will be thought of
as: $m\in\odi T$, given by $m(t\otimes t)$; $\alpha\colon\shF\arr\shF$,
given by {}``$\alpha(u)=\alpha(t\otimes u)$''; $\la-,-\ra\colon\det\shF\arr\odi S$,
given by {}``$\la u,v\ra=\la u,v\ra t$''. When we will say that
some particular relation among the maps $m,\alpha,\beta,\la-,-\ra$
locally holds, this will always mean that such relation holds as soon
as basis $t$ and $y,z$ of, respectively, $\shL$ and $\shF$ are
given.

The equations are:
\end{notation}
\begin{align}
\alpha^{2} & =m\id_{\shF}\label{eq:ass m,alpha}\\
\la\alpha(u),v\ra & =\la\alpha(v),u\ra\label{eq:ass alpha,gamma}\\
\alpha(\beta(u\otimes v)) & =-\beta(u\otimes\alpha(v))\label{eq:ass alpha,beta}\\
\beta(\beta(u\otimes v)\otimes w) & =\la\alpha(w),v\ra u-\la w,v\ra\alpha(u)\label{eq:ass gamma,beta}\\
\la u,\beta(v\otimes w)\ra & =\la w,\beta(u\otimes v)\ra\label{eq:ass beta,beta}
\end{align}
Moreover note that, locally, we also have the relation 
\begin{equation}
(x,y)_{\chi}=\la\alpha(y),x\ra\label{eq:ass alpha, gamma first}
\end{equation}
Taking into account the definition of $\LRings_{\stR}^{G}$ given
in \ref{def: pseudo monoidal commutativi associative}, we will prove
the following Theorem.
\begin{thm}
\label{thm:global data for Sthree}The map of stacks   \[   \begin{tikzpicture}[xscale=3.7,yscale=-0.6]     \node (A0_0) at (0, 0) {$\stY$};     \node (A0_1) at (1, 0) {$\LRings^G_\stR$};     \node (A1_0) at (0, 1) {$\chi=(\shL,\shF,m,\alpha,\beta, \la-,-\ra)$};     \node (A1_1) at (1, 1) {$\alA_\chi$};     \path (A0_0) edge [->]node [auto] {$\scriptstyle{}$} (A0_1);     \path (A1_0) edge [|->,gray]node [auto] {$\scriptstyle{}$} (A1_1);   \end{tikzpicture}   \] is
well defined and induces an isomorphism between the substack of $\stY$
of objects that locally satisfy the relations (\ref{eq:ass m,alpha}),
(\ref{eq:ass alpha,gamma}), (\ref{eq:ass alpha,beta}), (\ref{eq:ass gamma,beta}),
(\ref{eq:ass beta,beta}) and $\GCov$ (where a cover is thought of
as its corresponding sheaf of algebras).
\end{thm}
The next section is devoted to the proof of the above Theorem.
\begin{notation}
Assuming Theorem \ref{thm:global data for Sthree}, we will identify
the stack $\stY$ with $\GCov$ and, over $\stR_{3}$, with $\RCov{S_{3}}$.
Therefore we will also use the expression $\chi=(\shL,\shF,m,\alpha,\beta,\la-,-\ra)\in\GCov$
and similarly for $S_{3}$. 
\end{notation}

\subsection{From functors to algebras.}

The goal of this section is to deduce what data is needed to define
a $G$-cover and to prove Theorem \ref{thm:global data for Sthree}.
Consider the map $f\colon\Loc^{G}\stR\arr\N$ given by $f_{W}=\rk W$.
By Theorem \ref{thm:description of GCov with monoidal functors} and
remark \ref{rem:description of GCov with monoidal functors}, we need
to describe the symmetric lax monoidal functors $\Omega\colon\Loc^{G}\shR\arr\Loc T$
such that $\rk^{\Omega}=f$, i.e. $(\LMon_{\stR,f}^{G})^{\text{gr}}$.
We will proceed in the following way. We will identify $\stY$ with
a closed substack of $(\LPMon_{\stR,f}^{G})^{\text{gr}}$ in such
a way that, for any $\chi\in\stY$, the algebra $\alA_{\chi}$, as
defined in the previous section, is isomorphic to the algebra $\chi_{R[G]}$,
where $\chi$ is thought of as a pseudo-monoidal functor. For convenience,
set $\stX$ for the full substack of $(\LPMon_{\stR,f}^{G})^{\text{gr}}$
of functors $\Omega$ such that $\Omega_{\stR}=\odi T$ and $1\in\Omega_{\stR}$
is a unity for $\Omega$.

By \ref{prop:additive functors are equivariant sheaves glrg}, an
$\stR$-linear functor $\Omega$ such that $\rk^{\Omega}=f$ and that
$\Omega_{\stR}=\odi T$ is just given by 
\[
\Omega_{A}=\shL\comma\Omega_{V}=\shF
\]
where $\shL$ is an invertible sheaf and $\shF$ is a rank $2$ locally
free sheaf, both over $T$. The corresponding $G$-equivariant sheaf
is
\[
\alA_{\Omega}=\duale{\shR}\otimes\odi T\oplus\duale A\otimes\shL\oplus\duale V\otimes\shF
\]
A pseudo-monoidal structure on $\Omega$ for which $1\in\odi T=\Omega_{\stR}$
is a unity is given by maps
\[
\shL\otimes\shL\arrdi m\odi T\comma\shL\otimes\shF\arrdi{\alpha}\shF\comma\shF\otimes\shL\arrdi{\hat{\alpha}}\shF\comma\shF\otimes\shF\arrdi{\eta_{1}\oplus\eta_{2}\oplus\beta}\odi T\oplus\shL\oplus\shF
\]
The stack $\stX$ can be therefore thought of as the stack whose objects
are sequences $(\shL,\shF,m,\alpha,\hat{\alpha},\eta_{*},\beta)$
as above. In particular $\stY$ can be embedded in $\stX$ $ $by
sending $\chi=(\shL,\shF,m,\alpha,\beta,\la-,-\ra)\in\stY$ to the
sequence $(\shL,\shF,m,\alpha,\hat{\alpha},(-,-)_{\chi},\la-,-\ra,\beta)$,
where $\hat{\alpha}$ is obtained from $\alpha$ exchanging the factors
in the source. 

Given $\Omega\in\stX$, we want now to describe the algebra $\alA_{\Omega}$.
It will be convenient to introduce the following notation:
\[
\shL_{1}=\duale A\otimes\shL\comma\shQ_{0}=\odi T\oplus\shL_{1}\comma\shQ_{1}=\duale{V_{2}}\otimes\shF\comma\shQ_{2}=\duale{V_{1}}\otimes\shF
\]
In particular 
\[
\alA_{\Omega}=\odi T\oplus\shL_{1}\oplus\shQ_{1}\oplus\shQ_{2}
\]
where $\shQ_{0},\shQ_{1},\shQ_{2}$ are the sheaves induced by the
$\mu_{3}$ action on $\alA_{\Omega}$, while $\sigma$ is $-\id_{\shL_{1}}$
on $\shL_{1}$ and induces the isomorphism $\shQ_{1}\arr\shQ_{2}$,
$v_{2}^{*}\otimes u\arr v_{1}^{*}\otimes u$. We want to describe
the multiplication on $\alA_{\Omega}$ starting from the maps $m,\alpha,\hat{\alpha},\eta_{*},\beta$.
The sheaf $\shQ_{0}$ is a $\Z/2\Z$-cover and the associated map
\[
\shL_{1}\otimes\shL_{1}\arrdi{\mu}\odi T
\]
is just given by $\mu(1_{A}^{*}\otimes x\otimes1_{A}^{*}\otimes y)=\xi(x\otimes y)$.
Then we have the maps
\[
\shL_{1}\otimes\shQ_{1}\arrdi{\zeta}\shQ_{1}\comma\shL_{1}\otimes\shQ_{2}\arrdi{\zeta'}\shQ_{2}
\]
that are given by
\[
\zeta(1_{A}^{*}\otimes x\otimes v_{2}^{*}\otimes y)=v_{2}^{*}\alpha(x\otimes y)\comma\zeta'(1_{A}^{*}\otimes x\otimes v_{1}^{*}\otimes y)=-v_{1}^{*}\alpha(x\otimes y)
\]
The multiplications of $\shQ_{1}$ and $\shQ_{2}$ with $\shL$ are
just obtained exchanging the factors above and replacing $\alpha$
by $\hat{\alpha}$. Finally we have maps
\[
\shQ_{1}\otimes\shQ_{1}\arrdi{\lambda}\shQ_{2}\comma\shQ_{2}\otimes\shQ_{2}\arrdi{\lambda'}\shQ_{1}\comma\shQ_{1}\otimes\shQ_{2}\arrdi{\delta}\odi T\oplus\shL_{1}\comma\shQ_{2}\otimes\shQ_{1}\arrdi{\delta'}\odi T\oplus\shL_{1}
\]
that are given by
\[
\lambda(v_{2}^{*}\otimes x\otimes v_{2}^{*}\otimes y)=v_{1}^{*}\beta(x\otimes y)\comma\lambda'(v_{1}^{*}\otimes x\otimes v_{1}^{*}\otimes y)=v_{2}^{*}\beta(x\otimes y)
\]
\[
\delta(v_{2}^{*}\otimes x\otimes v_{1}^{*}\otimes y)=\eta_{1}(x\otimes y)+1_{A}^{*}\eta_{2}(x\otimes y)\comma\delta'(v_{1}^{*}\otimes x\otimes v_{2}^{*}\otimes y)=\eta_{1}(x\otimes y)-1_{A}^{*}\eta_{2}(x\otimes y)
\]
Now define 
\[
\alB_{\Omega}=\odi T\oplus\shL\oplus\shF_{1}\oplus\shF_{2}\text{ with }\shF_{1}=\shF_{2}=\shF\text{ and }\shF_{0}=\odi T\oplus\shL
\]
The isomorphisms $\stR\simeq\duale A$, $\stR\simeq\duale{V_{1}}$,
$\stR\simeq\duale{V_{2}}$ induce an isomorphism $\shB_{\Omega}\simeq\alA_{\Omega}$
of coherent sheaves. The $G$-comodule structure inherited by $\alB_{\Omega}$
is given by: $\shF_{0},\shF_{1},\shF_{2}$ yields the $\mu_{3}$-action,
while $\sigma\in\Z/2\Z$ acts as $-\id_{\shL}$ over $\shL$ and induces
$\id_{\shF}\colon\shF_{1}=\shF\arr\shF=\shF_{2}$ over $\shF_{1}$.
The $\odi T$-algebra structure inherited by $\alB_{\Omega}$ is given
by
\[
\shL^{2}\arrdi m\odi T\comma\shL\otimes\shF_{1}\arrdi{\alpha}\shF_{1}\comma\shL\otimes\shF_{2}\arrdi{-\alpha}\shF_{2}\comma\shF_{1}\otimes\shL\arrdi{\hat{\alpha}}\shF_{1}\comma\shF_{2}\otimes\shL\arrdi{-\hat{\alpha}}\shF_{2}
\]
\[
\shF_{1}\otimes\shF_{1}\arrdi{\beta}\shF_{2}\comma\shF_{2}\otimes\shF_{2}\arrdi{\beta}\shF_{1}\comma\shF_{1}\otimes\shF_{2}\arrdi{\gamma}\odi T\oplus\shL\comma\shF_{2}\otimes\shF_{1}\arrdi{\gamma'}\odi T\oplus\shL
\]
where $\gamma=\eta_{1}+\eta_{2}$ and $\gamma'=\eta_{1}-\eta_{2}$.
This shows that given $\chi\in\stY\subseteq\stX$ we have that $\alA_{\chi}$,
as defined in the previous section, coincides with $\alB_{\chi}$
in $\LRings_{\stR}^{G}$. In particular this proves \ref{prop:Achi is G-equivariant}.

\textbf{Commutativity conditions.} Given $\Omega\in\stX$, we want
to read the symmetry of $\Omega$ on the associated sequence $(\shL,\shF,m,\alpha,\hat{\alpha},\eta_{*},\beta)$.
Since $\shL$ is invertible there are no conditions for the commutativity
of the first map $m$. The map $\hat{\alpha}\colon\shF\otimes\shL\arr\shF$
is clearly obtained from $\alpha$ with an exchange. Finally note
that the exchange isomorphism $V\otimes V\simeq V\otimes V$ is the
identity on both $R$ and $V$, while it is minus the identity on
$A$. So the symmetry for the map $\Omega_{V}\otimes\Omega_{V}\arr\Omega_{V\otimes V}$
is equivalent to the symmetry of $\eta_{1}$ and $\beta$ and to the
antisymmetry of $\eta_{2}$. When $\Omega\in\stX$ is symmetric we
will use the notation
\[
(-,-)=\eta_{1}\colon\Sym^{2}\shF\arr\odi T\comma\la-,-\ra=\eta_{2}\colon\det\shF\arr\shL
\]
The associated sequence will be $(\shL,\shF,m,\alpha,\beta,(-,-),\la-,-\ra)$,
where we omit $\hat{\alpha}$ because it is determined by $\alpha$.

\textbf{Associativity conditions. }Given a symmetric $\Omega\in\stX$,
we want to read the associativity conditions on the associated sequence
$(\shL,\shF,m,\alpha,\beta,(-,-),\la-,-\ra)$. We can (and it is also
convenient to) understand such conditions working directly on the
algebra $\alB_{\Omega}$. Actually, we will proceed by listing some
diagrams that must commute when $\alB_{\Omega}$ is associative and
then we will show that their commutativity is indeed enough to imply
the associativity of $\alB_{\Omega}$.

We will make use of the notation introduced in \ref{not: local notation for Sthree}.
\begin{itemize}
\item   \[   \begin{tikzpicture}[xscale=2.8,yscale=-1.0]     \node (A0_0) at (0, 0) {$\shL\otimes\shL\otimes \shF_1$};     \node (A0_1) at (1, 0) {$\odi{S}\otimes \shF_1$};     \node (A1_0) at (0, 1) {$\shL\otimes \shF_1$};     \node (A1_1) at (1, 1) {$\shF_1$};     \path (A0_0) edge [->]node [auto] {$\scriptstyle{m\otimes\id}$} (A0_1);     \path (A0_0) edge [->]node [auto] {$\scriptstyle{\id\otimes\alpha}$} (A1_0);     \path (A0_1) edge [->]node [auto] {$\scriptstyle{\id}$} (A1_1);     \path (A1_0) edge [->]node [auto] {$\scriptstyle{\alpha}$} (A1_1);   \end{tikzpicture}   \] Locally
we obtain the condition (\ref{eq:ass m,alpha}).
\item   \[   \begin{tikzpicture}[xscale=3.2,yscale=-0.7]     
\node (A0_0) at (0, 0) {$\shF_1 \otimes \shF_2 \otimes \shL$};     
\node (A0_1) at (1, 0) {$(\odi{S}\oplus\shL)\otimes \shL$};     
\node (A1_1) at (1, 1) {$\shL\otimes\shL\oplus\shL$};     
\node[rotate=-90] (A1_3) at (1, 0.5) {$\simeq$};     
\node (A3_0) at (0, 3) {$\shF_1\otimes \shF_2$};     
\node (A3_1) at (1, 3) {$\odi{S}\oplus\shL$};     
\path (A0_0) edge [->]node [auto] {$\scriptstyle{\gamma\otimes \id}$} (A0_1);     \path (A0_0) edge [->]node [auto] {$\scriptstyle{\id\otimes -\alpha}$} (A3_0);     \path (A3_0) edge [->]node [auto] {$\scriptstyle{\gamma}$} (A3_1);     \path (A1_1) edge [->]node [auto] {$\scriptstyle{m\oplus\id}$} (A3_1);   \end{tikzpicture}   \]  The commutativity of this diagram is locally equivalent to $(u,\alpha(v))=-m\la u,v\ra$,
$(u,v)=-\la u,\alpha(v)\ra$ and, assuming (\ref{eq:ass m,alpha}),
to (\ref{eq:ass alpha,gamma}) and $(u,v)=\la\alpha(v),u\ra$.
\item   \[   \begin{tikzpicture}[xscale=3.2,yscale=-1.2]     \node (A0_0) at (0, 0) {$\shF_1\otimes\shF_1\otimes\shL$};     \node (A0_1) at (1, 0) {$\shF_2\otimes\shL$};     \node (A1_0) at (0, 1) {$\shF_1\otimes\shF_1$};     \node (A1_1) at (1, 1) {$\shF_2$};     \path (A0_0) edge [->]node [auto] {$\scriptstyle{\beta\otimes\id}$} (A0_1);     \path (A0_0) edge [->]node [auto] {$\scriptstyle{\id\otimes\alpha}$} (A1_0);     \path (A0_1) edge [->]node [auto] {$\scriptstyle{-\alpha}$} (A1_1);     \path (A1_0) edge [->]node [auto] {$\scriptstyle{\beta}$} (A1_1);   \end{tikzpicture}   \] The
commutativity of this diagram is locally equivalent to (\ref{eq:ass alpha,beta}).
\item   \[   \begin{tikzpicture}[xscale=3.2,yscale=-0.7]     
\node (A0_0) at (0, 0) {$\shF_2 \otimes \shF_2\otimes\shF_1$};     
\node (A0_1) at (1, 0) {$\shF_1\otimes\shF_1$};     
\node[rotate=-90] (A1_2) at (0, 2.5) {$\simeq$};     
\node (A2_0) at (0, 2) {$\shF_2 \otimes (\odi{S}\oplus\shL)$};     
\node (A3_0) at (0, 3) {$\shF_2 \oplus \shF_2\otimes \shL$};     
\node (A3_1) at (1, 3) {$\shF_2$};     
\path (A0_0) edge [->]node [auto] {$\scriptstyle{\beta\otimes\id}$} (A0_1);     
\path (A3_0) edge [->]node [auto] {$\scriptstyle{\id\oplus(-\alpha)}$} (A3_1);     
\path (A0_1) edge [->]node [auto] {$\scriptstyle{\beta}$} (A3_1);     
\path (A0_0) edge [->]node [auto] {$\scriptstyle{\id\otimes\gamma'}$} (A2_0);   
\end{tikzpicture}   \] The commutativity of this diagram, assuming that $(u,v)=\la\alpha(v),u\ra$,
is locally equivalent to (\ref{eq:ass gamma,beta}).
\item   \[   \begin{tikzpicture}[xscale=3.2,yscale=-1.2]     \node (A0_0) at (0, 0) {$\shF_1\otimes\shF_1\otimes\shF_1$};     \node (A0_1) at (1, 0) {$\shF_1\otimes\shF_2$};     \node (A1_0) at (0, 1) {$\shF_2\otimes\shF_1$};     \node (A1_1) at (1, 1) {$\odi{S}\oplus\shL$};     \path (A0_0) edge [->]node [auto] {$\scriptstyle{\id\otimes\beta}$} (A0_1);     \path (A0_0) edge [->]node [auto] {$\scriptstyle{\beta\otimes\id}$} (A1_0);     \path (A0_1) edge [->]node [auto] {$\scriptstyle{\gamma}$} (A1_1);     \path (A1_0) edge [->]node [auto] {$\scriptstyle{\gamma'}$} (A1_1);   \end{tikzpicture}   \] Since
$\gamma'(u\otimes v)=\gamma(v\otimes u)$, the commutativity of this
diagram is locally equivalent to (\ref{eq:ass beta,beta}) and the
analogous one for $(-,-)$, which however follows from (\ref{eq:ass m,alpha}),
(\ref{eq:ass alpha,gamma}), (\ref{eq:ass alpha,beta}) and (\ref{eq:ass beta,beta}),
assuming $(u,v)=\la\alpha(v),u\ra$. Indeed
\[
\begin{alignedat}{1}(w,\beta(u\otimes v)) & =\la\alpha(\beta(u\otimes v)),w\ra=-\la\beta(\alpha(u)\otimes v),w\ra=\la w,\beta(\alpha(u)\otimes v)\ra\\
 & =\la w,\beta(v\otimes\alpha(u))\ra=\la\alpha(u),\beta(w\otimes v)\ra=(u,\beta(v\otimes w))
\end{alignedat}
\]

\end{itemize}
In order to prove that the associativity conditions we have introduced
are enough to deduce the associativity of $\alB_{\Omega}$, we need
the following remark.
\begin{rem}
\label{lem:essential associative conditions}Let $A$ be a commutative
(but not necessary associative) ring and $x,y,z\in A$. If 
\[
(xy)z=x(yz)\text{ and }(yx)z=y(xz)
\]
then all the permutations of $x,y,z$ satisfy associativity. Indeed
\[
y(zx)=(yx)z=x(yz)=(yz)x\comma z(xy)=(yx)z=y(xz)=(zx)y
\]
\[
(zy)x=x(yz)=(xy)z=z(yx)\comma(xz)y=y(zx)=(yz)x=x(zy)
\]
\end{rem}
\begin{proof}
(\emph{of Theorem }\ref{thm:global data for Sthree}) The map in the
statement is the composition $\stY\hookrightarrow(\LPMon_{\stR,f}^{G})^{\text{gr}}\arrdi{\alA^{*}}(\LRings_{\stR}^{G})^{\text{gr}}$
and it is fully faithful. Call $\overline{\stY}$ the full substack
of $\stY$ satisfying the conditions of the statement. By Theorem
\ref{thm:description of GCov with monoidal functors} and remark \ref{rem:description of GCov with monoidal functors},
we need to prove that $(\LMon_{\stR,f}^{G})^{\text{gr}}=\overline{\stY}$.

$(\LMon_{\stR,f}^{G})^{\text{gr}}\subseteq\overline{\stY}$. Let $\Omega=(\shL,\shF,m,\alpha,\beta,(-,-),\la-,-\ra)\in(\LMon_{\stR,f}^{G})^{\text{gr}}$,
where we are using the description for symmetric functors, and $\chi=(\shL,\shF,m,\alpha,\beta,\la-,-\ra)\in\stY$.
One of the associativity conditions, namely the local conditions $(u,v)=\la\alpha(v),u\ra$,
tells us that $(-,-)=(-,-)_{\chi}$, thanks to (\ref{eq:ass alpha, gamma first}).
In particular $\Omega=\chi\in\stY$. In order to conclude that $\Omega\in\overline{\stY}$,
it is enough to note that the conditions in the statement are all
associativity conditions for $\Omega$, which are satisfied because
$\Omega$ is associative.

$\overline{\stY}\subseteq(\LMon_{\stR,f}^{G})^{\text{gr}}$. Let $\chi=(\shL,\shF,m,\alpha,\beta,\la-,-\ra)\in\overline{\stY}\subseteq\stY$.
The relations (\ref{eq:ass alpha,gamma}) and (\ref{eq:ass alpha, gamma first})
imply that $\alA_{\chi}$ is a commutative $\odi T$-algebra. We need
to show that $\alA_{\chi}$ is associative and we will use \ref{lem:essential associative conditions}.
Given $A,B,C\in\{\odi S,\shL,\shF_{1},\shF_{2}\}$ we will say that
$(A,B,C)$ holds if $a(bc)=(ab)c$ for all $a\in A,b\in B,c\in C$.
Since $\sigma\in\Z/2\Z$ induces a ring automorphism of $\alA_{\chi}$,
if $(A,B,C)$ holds then $(\sigma(A),\sigma(B),\sigma(C))$ holds
and, moreover, if also $(B,A,C)$ holds then all the permutations
of $(A,B,C)$ and $(\sigma(A),\sigma(B),\sigma(C))$ hold.

Clearly $(\shL,\shL,\shL)$ holds. Condition (\ref{eq:ass m,alpha})
insures that $(\shL,\shL,\shF_{1})$, $(\shL,\shL,\shF_{2})$ and
all their permutations hold. Condition (\ref{eq:ass alpha,gamma})
says that all the permutations of $(\shF_{1},\shF_{2},\shL)$ hold,
while condition (\ref{eq:ass alpha,beta}) tells us that all the permutations
of $(\shF_{1},\shF_{1},\shL)$ and $(\shF_{2},\shF_{2},\shL)$ hold.
The relation (\ref{eq:ass gamma,beta}) implies that $(\shF_{1},\shF_{1},\shF_{2})$,
$(\shF_{2},\shF_{2},\shF_{1})$ and all their permutations hold. Finally
(\ref{eq:ass beta,beta}) says that $(\shF_{1},\shF_{1},\shF_{1})$
and $(\shF_{2},\shF_{2},\shF_{2})$ hold. It is now easy to check
that we have obtained all the possible triples.
\end{proof}

\subsection{Local analysis.}

Let $\chi=(\shL,\shF,m,\alpha,\beta,\la-,-\ra)\in\stY$ and assume
that $t\in\shL$ is a generator and that $y,z$ is a basis of $\shF$.
The aim of this subsection is to translate conditions (\ref{eq:ass m,alpha}),
(\ref{eq:ass alpha,gamma}), (\ref{eq:ass alpha,beta}), (\ref{eq:ass gamma,beta})
and (\ref{eq:ass beta,beta}), writing all the maps $\alpha,\beta,\la-,-\ra$
with respect to the given basis. In particular we will use notation
from \ref{not: local notation for Sthree}, so that $m\in\odi T$,
$\alpha$ is a map $\shF\arr\shF$ and $\la-,-\ra\colon\det\shF\arr\odi T$.
\begin{notation}
Write
\[
\beta(y^{2})=ay+bz\comma\beta(yz)=cy+dz\comma\beta(z^{2})=ey+df\comma\la y,z\ra=\omega\comma\alpha=\left(\begin{array}{cc}
A & B\\
C & D
\end{array}\right)
\]

\end{notation}
We start computing useful relations in order to impose the associativity
conditions
\[
(y,y)=-C\omega\comma(y,z)=-D\omega\comma(z,y)=A\omega\comma(z,z)=B\omega
\]
\[
\begin{array}{ll}
\beta(\beta(y^{2})y)=(a^{2}+bc)y+b(a+d)z & \beta(\beta(z^{2})y)=(ea+fc)y+(eb+fd)z\\
\beta(\beta(y^{2})z)=(ac+be)y+(ad+bf)z & \beta(\beta(z^{2})z)=e(c+f)y+(ed+f^{2})z
\end{array}
\]
\[
\beta(\beta(yz)y)=\beta(\beta(zy)y)=c(a+d)y+(cb+d^{2})z
\]
\[
\beta(\beta(yz)z)=\beta(\beta(zy)z)=(c^{2}+de)y+d(c+f)z
\]
\[
\begin{array}{ll}
\la y,\beta(y^{2})\ra=b\omega & \la y,\beta(z^{2})\ra=f\omega\\
\la z,\beta(y^{2})\ra=-a\omega & \la z,\beta(z^{2})\ra=-e\omega\\
\la y,\beta(yz)\ra=\la y,\beta(zy)\ra=d\omega & \la z,\beta(yz)\ra=\la z,\beta(zy)\ra=-c\omega
\end{array}
\]
If we set $\Gamma(u,v,w)=\la\alpha(w),v\ra u-\la w,v\ra\alpha(u)$
we have
\[
\begin{array}{ll}
\Gamma(y,y,y)=-C\omega y & \Gamma(z,y,y)=-C\omega z\\
\Gamma(y,y,z)=\omega(A-D)y+\omega Cz & \Gamma(z,y,z)=B\omega y\\
\Gamma(y,z,y)=-\omega Cz & \Gamma(z,z,y)=-B\omega y+\omega(A-D)z\\
\Gamma(y,z,z)=B\omega y & \Gamma(z,z,z)=B\omega z
\end{array}
\]
\[
\begin{array}{rcl}
\alpha(\beta(y^{2}))+\beta(\alpha(y)y) & = & (2aA+bB+cC)y+(C(a+d)+b(A+D))z\\
\alpha(\beta(yz))+\beta(\alpha(y)z) & = & (2cA+dB+eC)y+(C(c+f)+d(A+D))z\\
\alpha(\beta(zy))+\beta(\alpha(z)y) & = & (B(a+d)+c(A+D))y+(2dD+bB+cC)z\\
\alpha(\beta(z^{2}))+\beta(\alpha(z)z) & = & (B(c+f)+e(A+D))y+(2fD+eC+dB)z
\end{array}
\]
From the above relations we can conclude:
\begin{lem}
\label{lem:local conditions for the map for Sthree}The object $\chi=(\shL,\shF,m,\alpha,\beta,\la-,-\ra)\in\stY$
belongs to $\RCov G$ if and only if the following relations hold.
\begin{equation}
\begin{array}{ccc}
(\ref{eq:ass m,alpha}) & \iff & m=A^{2}+BC\comma(A-D)(A+D)=B(A+D)=C(A+D)=0\\
(\ref{eq:ass alpha,gamma}) & \iff & \omega(A+D)=0\\
(\ref{eq:ass alpha,beta}) & \iff & \left\{ \begin{array}{c}
(2aA+bB+cC)=(2cA+dB+eC)=0\\
C(a+d)+b(A+D)=C(c+f)+d(A+D)=0\\
B(a+d)+c(A+D)=B(c+f)+e(A+D)=0\\
a(A+D)-D(a+d)=c(A+D)-D(c+f)=0
\end{array}\right.\\
(\ref{eq:ass gamma,beta}) & \iff & \left\{ \begin{array}{c}
a^{2}+bc=-\omega C\comma ac+be=\omega(A-D)\comma c^{2}+de=B\omega\\
(a-d)(a+d)=b(a+d)=c(a+d)=0\\
(c-f)(c+f)=d(c+f)=e(c+f)=0\\
a(a+d)+b(c+f)=e(a+d)+c(c+f)=0
\end{array}\right.\\
(\ref{eq:ass beta,beta}) & \iff & \omega(a+d)=\omega(c+f)=0
\end{array}\label{eq:loc com and ass conditions}
\end{equation}
\end{lem}
\begin{rem}
\label{rem: cover with non zero trace}It is not true in general that
some of $(a+d)$, $(c+f)$, $(A+D)$ is $0$. For instance over the
ring $k[x]/(x^{2})$ we have a $G$-cover given by
\[
a=b=c=d=e=f=\omega=A=B=C=D=x\comma m=0
\]
On the other hand if we are on a reduced ring $R$ then $(a+d)=(c+f)=0$.
Indeed over every domain we have
\[
(a+d)\neq0\then a=d,c=b=0\text{ and }a(a+d)+b(c+f)=0\then a=0
\]
so $a+d=2a=0$. The same argument shows that $c+f=0$.\end{rem}
\begin{notation}
\label{not: associated parameters for chi and Sthree}Given $\chi=(\shL,\shF,m,\alpha,\beta,\la-,-\ra)\in\stY$,
a basis $y,z$ of $\shF$ and a generator $t\in\shL$, we will denote
by 
\[
a_{\chi},b_{\chi},c_{\chi},d_{\chi},e_{\chi},f_{\chi},\omega_{\chi},A_{\chi},B_{\chi},C_{\chi},D_{\chi},m_{\chi}
\]
the data associated with $\chi$ as above. We will always omit the
$-_{\chi}$ if this will not lead to confusion.
\end{notation}

\section{Geometry of $(\mu_{3}\rtimes\Z/2\Z)$-Cov and $\RCov{S_{3}}$.}

The aim of this section is to describe the geometry of the stack $\GCov$
and, as a consequence, of $\RCov{S_{3}}$. Clearly this is related
to the problem of the description of $G$-covers: we will individuate
three smooth open substacks of $\GCov$, that is families of $G$-covers
with a certain global or local property. In some cases this will allow
us to describe $G$-covers using less data than needed to build a
general $G$-cover, that is the objects $\shL,\shF,m,\alpha,\beta,\la-,-\ra$.
For instance when $\la-,-\ra\colon\det\shF\arr\shL$ is an isomorphism,
we will show that the data $\shF,\beta$ determine all the others.
This is the first locus we will describe, denoted by $\stU_{\omega}$.
The other two will be the locus $\stU_{\alpha}$ where $\alpha$ is
never a multiple of the identity and the locus $\stU_{\beta}$ where
$\beta$ is never zero. Although in those cases we do not have a global
description of $G$-covers belonging to these families, what we have
is a local description that will be extremely useful also in the next
sections, for instance because it will turn out that any $G$-cover
between smooth varieties belongs to the locus where $\beta$ is never
zero.

Unluckily, the three loci described above do not cover $\GCov$, actually
we will see that they are contained in the main irreducible component
$\stZ_{G}$. On the other hand they almost cover $\stZ_{G}$: the
complement of their union in $\stZ_{G}$ is formed by the {}``zero
covers'', that is the $G$-covers where $m=\alpha=\beta=\la-,-\ra=0$,
and therefore, topologically, we have missed only {}``one point'',
actually one point for each characteristic. The stack $\stZ_{G}$
is an irreducible component of $\GCov$ and Theorem \ref{thm:GCov reducible when G not abelian}
tells us that $\GCov$ is reducible. We will show more: the complementary
$\stZ_{2}$ of the union of $\stU_{\omega}$, $\stU_{\alpha}$ and
$\stU_{\beta}$ in the whole $\GCov$ is another irreducible component.
Therefore $\GCov$ and, over $\stR_{3}$, $\RCov{S_{3}}$ have exactly
two irreducible components. Although the covers in $\stZ_{2}$ are
highly degenerate, the stack $\stZ_{2}$ has a very simple description
and a very simple geometry, for instance it is smooth, while $\stZ_{G}$
is not. Over a field, $\stZ_{2}$ is topologically composed of two
points and $(\GCov)-\stZ_{G}\simeq\Bi\Gl_{2}$. The last subsection
is dedicated to the study of the irreducible component $\stZ_{G}$
and we will show that, in this case, the maps $m,\alpha,\beta,\la-,-\ra$
are uniquely determined by $\beta$ and $\la-,-\ra$.

\subsection{Triple covers and the locus where $\la-,-\ra\colon\det\shF\arr\shL$
is an isomorphism.}

In this subsection we want to describe a smooth open substack of $\GCov$,
more precisely the locus $\stU_{\omega}$ of objects $\chi=(\shL,\shF,m,\alpha,\beta,\la-,-\ra)$
such that $\la-,-\ra\colon\det\shF\arr\shL$ is an isomorphism. 
\begin{defn}
\label{def:definition of Cthree}Define $\shC_{3}$ as the stack whose
objects are pairs $(\shF,\delta)$ where $\shF$ is a locally free
sheaf of rank $2$ and $\delta$ is a map
\[
\delta\colon\Sym^{3}\shF\arr\det\shF
\]

\end{defn}
Notice that $\stC_{3}$ is a smooth stack over $\Z$ because it is
a vector bundle over $\Bi\Gl_{2}$, which is the stack of rank $2$
locally free sheaves. We will show that $\stU_{\omega}$ is isomorphic
to the stack $\shC_{3}$. This also explains the reason of the section
name: it is a classical result (see \cite{Miranda1985,Bolognesi2009,Pardini1989})
that, over $\stR_{3}$, the stack $\shC_{3}$ is isomorphic to the
stack $\Cov_{3}$ of degree $3$ covers. We will show that, in this
case, the map $\Cov_{3}\simeq\shC_{3}\arr\stU_{\omega}$ is a section
of the map $\GCov\arr\Cov_{3}$, obtained taking invariants by $\sigma\in\Z/2\Z\subseteq G$.

We now give an alternative description of $\shC_{3}$ and we will
need the following notation.
\begin{notation}
\label{not: trace of a map} Given locally free sheaves $\shN$ and
$\shF$ over $T$ and a map $\zeta\colon\shN\otimes\shF\arr\shF$
we will call trace of $\zeta$ the composition $\tr\zeta\colon\shN\arr\duale{\shF}\otimes\shF\arr\odi T$.
We will also denote by $\Homsh_{\tr=0}(\shN\otimes\shF,\shF)$ the
subsheaf of $\Homsh(\shN\otimes\shF,\shF)$ of maps whose trace is
$0$. Notice that, when $\shN=\odi T$ and $\shF$ is free, the trace
we have just defined is the usual trace of a associated matrix.
\end{notation}

\begin{notation}
\label{not: trace of beta}Given $\beta\colon\Sym^{2}\shF\arr\shF$
we will use the notation 
\[
\tr\beta=\tr(\shF\otimes\shF\arr\Sym^{2}\shF\arrdi{\beta}\shF)\colon\shF\arr\odi T
\]
If $y,z$ is a basis of $\shF$ and $\beta(y^{2})=ay+bz$, $\beta(yz)=cy+dz$,
$\beta(z^{2})=ey+fz$, then $(\tr\beta)(y)=a+d$, $(\tr\beta)(z)=c+f$.
\end{notation}
It is easy to check (see also \cite{Miranda1985,Bolognesi2009}) that
if $\beta\colon\Sym^{2}\shF\arr\shF$ is a map such that $\tr\beta=0$
there exists a dashed map $\delta$:  \[   \begin{tikzpicture}[xscale=2.4,yscale=-1.0]     \node (A0_0) at (0, 0) {$\Sym^2\shF\otimes \shF$};     \node (A0_1) at (1, 0) {$\shF\otimes\shF$};     \node (A1_0) at (0, 1) {$\Sym^3 \shF$};     \node (A1_1) at (1, 1) {$\det \shF$};     \path (A0_0) edge [->]node [auto] {$\scriptstyle{\beta\otimes \id}$} (A0_1);     \path (A0_0) edge [->]node [auto] {$\scriptstyle{}$} (A1_0);     \path (A0_1) edge [->]node [auto] {$\scriptstyle{}$} (A1_1);     \path (A1_0) edge [->,dashed]node [auto] {$\scriptstyle{\delta}$} (A1_1);   \end{tikzpicture}   \] 
This association yields an isomorphism \begin{align}\label{iso:delta beta}
 \begin{tikzpicture}[xscale=5.0,yscale=-0.8]     \node (A0_0) at (0, 0) {$\Homsh_{\tr = 0}(\Sym^2 \shF,\shF)$};     \node (A0_1) at (1, 0) {$\Homsh(\Sym^3 \shF,\det\shF)$};     \node (A1_0) at (0, 1) {$\left(\begin{array}{ccc} a & c & e\\ b & -a & -c \end{array}\right)$};     \node (A1_1) at (1, 1) {$\left(\begin{array}{cccc} -b & a & c & e\end{array}\right)$};     \path (A0_0) edge [->]node [auto] {$\scriptstyle{}$} (A0_1);     \path (A1_0) edge [|->,gray]node [auto] {$\scriptstyle{}$} (A1_1);   \end{tikzpicture}
\end{align}where the last row describes how this map behaves if a basis $y,z$
of $\shF$ is chosen, where we have considered $y^{3},y^{2}z,yz^{2},z^{3}$
as basis of $\Sym^{3}\shF$. So $\shC_{3}$ can also be described
as the stack of pairs $(\shF,\beta)$ where $\shF$ is a locally free
sheaf of rank $2$ and $\beta\colon\Sym^{2}\shF\arr\shF$ is a map
such that $\tr\beta=0$.
\begin{notation}
We will denote the correspondence (\ref{iso:delta beta}) by $\beta\longmapsto\delta_{\beta}$
and $\delta\longmapsto\beta_{\delta}$.
\end{notation}
Now let $(\shF,\delta)\in\shC_{3}$. Define $\eta_{\delta}$ as the
map
\begin{equation}
\Sym^{2}\shF\arrdi u\Lambda^{2}\Sym^{2}\shF\otimes\duale{\Lambda^{2}\shF}\arrdi v\odi S\label{eq:etabeta}
\end{equation}
where $v$ is induced by $\Lambda^{2}\beta_{\delta}\colon\Lambda^{2}\Sym^{2}\shF\arr\Lambda^{2}\shF$
and $u$ is induced by   \[   \begin{tikzpicture}[xscale=4.7,yscale=-0.6]     \node (A0_0) at (0, 0) {$\Lambda^2\shF\otimes \Sym^2 \shF$};     \node (A0_1) at (1, 0) {$\Lambda^2 \Sym^2 \shF$};     \node (A1_0) at (0, 1) {$(x_1\wedge x_2)\otimes x_3 x_4$};     \node (A1_1) at (1, 1) {$-x_1x_3\wedge  x_2 x_4-x_1 x_4\wedge x_2 x_3$};     \path (A0_0) edge [->]node [auto] {$\scriptstyle{}$} (A0_1);     \path (A1_0) edge [|->,gray]node [auto] {$\scriptstyle{}$} (A1_1);   \end{tikzpicture}   \] 
and $\alpha_{\delta}\colon\det\shF\otimes\shF\arr\shF$ as the map
induced by
\[
\det\shF\otimes\duale{\shF}\otimes\shF\arrdi{\simeq}\shF\otimes\shF\arrdi{\eta_{\delta}/2}\odi S
\]
using the canonical isomorphism $\det\shF\otimes\duale{\shF}\simeq\shF$.
Finally define $m_{\delta}\colon(\det\shF)^{2}\arr\odi T$ as minus
the map induced by
\[
(\det\shF)^{2}\otimes\det\shF\simeq\det(\det\shF\otimes\shF)\arrdi{\det\alpha_{\delta}}\det\shF
\]

Denote by $\Cov_{3}$ the stack of degree $3$ covers, or, equivalently,
the stack of locally free sheaves of $\odi T$algebras of rank $3$.
Denote also by $\stU_{\omega}$ the open substack of $\GCov$ of objects
$\chi=(\shL,\shF,m,\alpha,\beta,\la-,-\ra)$ such that $\la-,-\ra\colon\det\shF\arr\shL$
is an isomorphism. The theorem we want to prove is:
\begin{thm}
\label{thm:The-locus when omega is invertible}The maps of stacks
  \[   \begin{tikzpicture}[xscale=4.0,yscale=-0.6]     
\node (A0_0) at (0, 0) {$(\shF,\delta)$};
\node (A0_1) at (1, 0) {$(\det \shF,\shF,m_\delta,\alpha_\delta,\beta_\delta,\id_{\det \shF})$};     
\node (A1_0) at (0, 1) {$\shC_3$};     
\node (A1_1) at (1, 1) {$\stU_\omega$};     
\node (A2_0) at (0, 2) {$(\shF,\delta_\beta)$};     
\node (A2_1) at (1, 2) {$(\shL,\shF,m,\alpha,\beta,\la-,-\ra)$};     
\path (A0_0) edge [|->,gray]node [auto] {$\scriptstyle{}$} (A0_1);          
\path (A1_0) edge [->]node [auto] {$\scriptstyle{\Lambda}$} (A1_1);     
\path (A2_1) edge [|->,gray]node [auto] {$\scriptstyle{}$} (A2_0);   
\end{tikzpicture}   \] are well defined and they are inverses of each other. In particular
$\stU_{\omega}$ is a smooth open substack of $\GCov$. Moreover,
over $\stR_{3}$, the composition $\Cov_{3}\simeq\shC_{3}\arrdi{\Lambda}\GCov$
is a section of the map $\GCov\arr\Cov_{3}$ obtained by taking invariants
by $\sigma\in\Z/2\Z$ and the same result hold if we replace $\GCov$
by $\RCov{S_{3}}$.
\end{thm}
We will prove this theorem at the end of this section, after collecting
some useful remarks.
\begin{rem}
If $y,z$ is a basis of $\shF$, we identify $\det\shF\simeq\odi T$
using the generator $y\wedge z\in\det\shF$ and we write $\delta$
as
\begin{equation}
\delta(y^{3})=-b\comma\delta(y^{2}z)=a\comma\delta(yz^{2})=c\comma\delta(z^{3})=e\label{eq:expression of delta}
\end{equation}
then we have expressions
\begin{equation}
\eta_{\delta}(y^{2})=2(a^{2}+bc)\comma\eta_{\delta}(yz)=ac+be\comma\eta_{\delta}(z^{2})=2(c^{2}-ae)\label{eq:expression of eta delta}
\end{equation}
\[
2\alpha_{\delta}(y)=\eta_{\delta}(yz)y-\eta_{\delta}(y^{2})z\comma2\alpha_{\delta}(z)=\eta_{\delta}(z^{2})y-\eta_{\delta}(yz)z
\]
In particular, if $\chi=(\shL,\shF,m,\alpha,\beta_{\delta},\la-,-\ra)\in\GCov$,
then, by (\ref{eq:ass alpha, gamma first}) and (\ref{eq:loc com and ass conditions}),
we have
\begin{equation}
\eta_{\delta}=2(-,-)_{\chi}\label{eq:eta delta is two times symmetric product}
\end{equation}

\end{rem}

We now want to show the relationship between $\shC_{3}$ and $\Cov_{3}$.
The reader can refer to \cite{Miranda1985,Bolognesi2009} for details
and proofs.
\begin{rem}
\label{rem:triple covers and invariants by sigma}If $\Phi=(\shF,\delta)\in\shC_{3}$
and we set $\alA_{\Phi}=\odi T\oplus\shF$, we can endow $\alA_{\Phi}$
by a structure of $\odi T$-algebras given by
\[
\Sym^{2}\shF\arrdi{\eta_{\delta}+\beta_{\delta}}\alA_{\Phi}
\]
This association defines a map of stacks $\shC_{3}\arr\Cov_{3}$.
This map is an isomorphism if $3$ is inverted in the base scheme.
Indeed if $\alA\in\Cov_{3}$, $ $the trace map $\tr_{\alA/\odi T}\colon\alA\arr\odi T$
is surjective and we can write $\alA=\odi S\oplus\shF$, where $\shF=\ker\tr_{\alA}$.
The multiplication of $\alA$ induces a map $\beta\colon\Sym^{2}\shF\arr\shF$
such that $\tr\beta=0$ and therefore a $\delta\colon\Sym^{3}\shF\arr\det\shF$
such that $\beta_{\delta}=\beta$.

Now let $\chi=(\shL,\shF,m,\alpha,\beta,\la-,-\ra)\in\GCov$. It's
easy to see that 
\[
\alA_{\chi}^{\sigma}=\{a\oplus0\oplus x_{1}\oplus x_{2}\st a\in\odi T\comma x_{1}=x_{2}\in\shF\}
\]
where $\sigma\in\Z/2\Z\subseteq G$. The map   \[   \begin{tikzpicture}[xscale=3.0,yscale=-0.5]     \node (A0_0) at (0, 0) {$\odi{T}\oplus\shF$};     \node (A0_1) at (1, 0) {$\alA^\sigma$};     \node (A1_0) at (0, 1) {$a\oplus x$};     \node (A1_1) at (1, 1) {$a\oplus 0\oplus x\oplus x$};     \path (A0_0) edge [->]node [auto] {$\scriptstyle{}$} (A0_1);     \path (A1_0) edge [|->,gray]node [auto] {$\scriptstyle{}$} (A1_1);   \end{tikzpicture}   \] 
is an isomorphism of $\odi S$-modules and the induced algebra structure
on $\odi T\oplus\shF$ is given by
\[
\beta\colon\Sym^{2}\shF\arr\shF\text{ and }2(-,-)\colon\Sym^{2}\shF\arr\odi T
\]
Notice that it is not true in general that $\shF=\ker\tr_{\alA^{\sigma}}$,
also if $3$ is inverted: this equality holds if and only if $\tr\beta=0$,
for instance over any reduced scheme or, as we will see in the next
sections, over the principal irreducible component $\stZ_{G}\subseteq\GCov$.
In this case $(\shF,\delta_{\beta})\in\shC_{3}$ and $\alA_{\chi}^{\sigma}\simeq\alA_{(\shF,\delta_{\beta})}$.
Moreover we get a well defined map of stacks $\{\tr\beta=0\}\arr\shC_{3}$,
where $\{\tr\beta=0\}$ is the closed substack of $\GCov$ where $\tr\beta=0$.
When $3$ is inverted, such map, composed by $\shC_{3}\arr\Cov_{3}$,
extends to a map of stacks $\GCov\arr\Cov_{3}$, by taking invariants
by $\sigma$.
\end{rem}

\begin{rem}
\label{rem:invariants by sigma for Sthree and GCov}Proposition \ref{prop:bitorsors and semidirect products}
tells us that, over $\stR_{3}$, the isomorphism $\GCov\simeq\RCov{S_{3}}$
preserves the quotient by $\sigma\in\Z/2\Z$, that is we have a commutative
diagram   \[   \begin{tikzpicture}[xscale=0.9,yscale=-1.0]     \node (A0_0) at (0, 0) {$\GCov$};     \node (A0_2) at (2, 0) {$\RCov{S_3}$};     \node (A0_4) at (4, 0) {$X$};     \node (A0_5) at (5, 0) {$\alA$};     \node (A1_1) at (1, 1) {$\Cov_3$};     \node (A1_4) at (4, 1) {$X/\sigma$};     \node (A1_5) at (5, 1) {$\alA^\sigma$};     \path (A0_4) edge [|->,gray]node [auto] {$\scriptstyle{}$} (A1_4);     \path (A0_0) edge [->]node [auto] {$\scriptstyle{}$} (A1_1);     \path (A0_5) edge [|->,gray]node [auto] {$\scriptstyle{}$} (A1_5);     \path (A0_2) edge [->]node [auto] {$\scriptstyle{}$} (A1_1);     \path (A0_0) edge [->]node [auto] {$\scriptstyle{\simeq}$} (A0_2);   \end{tikzpicture}   \] 
\end{rem}
We are ready for the proof of the main theorem of this subsection.
\begin{proof}
(\emph{of Theorem} \ref{thm:The-locus when omega is invertible})
We need to prove that $\Lambda$ is well defined. Let $\Phi=(\shF,\delta)\in\shC_{3}$.
We have that $\chi=\Lambda(\Phi)\in\stY$ and we have to prove that
$\chi$ satisfies the conditions of \ref{lem:local conditions for the map for Sthree}.
We can therefore work locally and assume we have a basis $y,z$ of
$\shF$. If we write $\delta$ as in (\ref{eq:expression of delta}),
the parameters associated to $\chi$ (see \ref{not: associated parameters for chi and Sthree})
are 
\[
a,b,c,d=-a,e,f=-c,\omega=1,A=-D=(ac+be)/2\comma B=c^{2}-ae\comma C=-a^{2}-bc
\]
It is easy to check that all the conditions in \ref{lem:local conditions for the map for Sthree}
are satisfied. Moreover 
\[
m_{\delta}=A^{2}+BC=-(AD-BC)=-\det\alpha
\]
So $\Lambda(\Phi)\in\GCov$ and, by definition, $\Lambda(\Phi)\in\stU_{\omega}$. 

Conversely, if $\chi=(\shL,\shF,m,\alpha,\beta,\la-,-\ra)\in\stU_{\omega}$,
taking into account relations (\ref{lem:local conditions for the map for Sthree})
and the fact that locally $\omega$ is invertible, we have that $\tr\beta=0$.
So $(\shF,\delta_{\beta})\in\shC_{3}$. Denote by $\Delta$ the map
$\Delta\colon\stU_{\omega}\arr\shC_{3}$ defined in the statement.
Clearly $\Delta\circ\Lambda\simeq\id$. For the converse, consider
\[
(\la-,-\ra,\id_{\shF})\colon\Lambda\circ\Delta(\chi)=(\det\shF,\shF,m_{\delta},\alpha_{\delta},\beta,\id_{\det\shF})\arr(\shL,\shF,m,\alpha,\beta,\la-,-\ra)=\chi\text{ where }\delta=\delta_{\beta}
\]
In order to conclude that the association above defines an isomorphism
$\Lambda\circ\Delta\simeq\id$, it is enough to prove that it is well
defined. The only non trivial condition to check is the commutativity
of the following diagram.   \[   \begin{tikzpicture}[xscale=2.1,yscale=-1.2]     \node (A0_0) at (0, 0) {$\det\shF\otimes\shF$};     \node (A0_1) at (1, 0) {$\shF$};     \node (A1_0) at (0, 1) {$\shL\otimes\shF$};     \node (A1_1) at (1, 1) {$\shF$};     \path (A0_0) edge [->]node [auto] {$\scriptstyle{\alpha_\delta}$} (A0_1);     \path (A0_0) edge [->]node [auto,swap] {$\scriptstyle{\la-,-\ra\otimes\id_\shF}$} (A1_0);     \path (A0_1) edge [->]node [auto] {$\scriptstyle{\id_\shF}$} (A1_1);     \path (A1_0) edge [->]node [auto] {$\scriptstyle{\alpha}$} (A1_1);   \end{tikzpicture}   \] Working
locally, the commutativity of the above diagram is equivalent to the
condition $\alpha_{\delta}=\omega\alpha$, which can be easily verified.

Now assume we are over $\stR_{3}$. The map $\GCov\arr\Cov_{3}\simeq\shC_{3}$
extends the map $\stU_{\omega}\arr\shC_{3}$ defined in the statement.
Therefore $\Cov_{3}\simeq\shC_{3}\arr\stU_{\omega}\subseteq\GCov$
is a section of such map.
\end{proof}

\subsection{The locus where $\alpha\colon\shL\otimes\shF\arr\shF$ is never a
multiple of the identity.}

Define $\stU_{\alpha}$ as the full substack of $\GCov$ of objects
$\chi=(\shL,\shF,m,\alpha,\beta,\la-,-\ra)$ such that $\alpha\colon\shL\otimes\shF\arr\shF$
is never a multiple of the identity, i.e. such that $\alpha$ is not
a multiple of the identity after some base change. We want to prove
the following:
\begin{thm}
\label{thm:not degenerate locus of S3}Let $R=\stR[m,a,b]$. Then
\[
(R,R^{2},m,\alpha,\beta,\la-,-\ra)
\]
where
\[
\alpha=\left(\begin{array}{cc}
0 & m\\
1 & 0
\end{array}\right);\begin{array}{l}
\beta(e_{1}^{2})=ae_{1}+be_{2}\\
\beta(e_{1}e_{2})=-mbe_{1}-ae_{2};\begin{array}{l}
\la e_{1},e_{1}\ra=\la e_{2},e_{2}\ra=0\\
\la e_{1},e_{2}\ra=-\la e_{2},e_{1}\ra=mb^{2}-a^{2}
\end{array}\\
\beta(e_{2}^{2})=mae_{1}+mbe_{2}
\end{array}
\]
is an object of $\RCov G(R)$. The induced map $\A^{3}\arr\RCov G$
is a smooth Zariski epimorphism onto $\stU_{\alpha}$. In particular
$\stU_{\alpha}$ is a smooth open substack of $\GCov$.
\end{thm}
Before proving this Theorem we need two lemmas.
\begin{lem}
\label{lem:the not degenerate locus is open}Let $\shF$ be a locally
free sheaf of rank $2$, $\shL$ be an invertible sheaf, both over
$T$ and $\alpha\colon\shL\otimes\shF\arr\shF$ be a map. Let also
$k$ be a field, $\Spec k\arr T$ be a map and $p\in T$ the induced
point. If $\alpha\otimes k$ is not a multiple of the identity, then
there exists a Zariski open neighborhood $V$ of $p$ in $T$ and
$y\in\shF_{|V}$ such that $\shL_{|V}=\odi Vt$ and $y,\alpha(t\otimes y)$
is a basis of $\shF_{|V}$.\end{lem}
\begin{proof}
If the statement is true when $T=\Spec k'$, for some field $k'$,
then it follows in general by Nakayama's lemma. So assume that $T=\Spec k$
and, by contradiction, that such a basis does not exist. It is easy
to deduce that any vector of $\shF$ is an eigenvector for $\alpha$.
By standard linear algebra we can conclude that $\alpha$ is a multiple
of the identity.\end{proof}
\begin{lem}
\label{lem:associated parameters for alpha nowhere multiple of the identity}Let
$\chi=(\shL,\shF,m,\alpha,\beta,\la-,-\ra)\in\stY$ and $y\in\shF$
be such that $\shL=\odi T$ and $y,z=\alpha(y)$ is a basis of $\shF$.
Then $\chi\in\GCov$ if and only if the associated parameters (see
\ref{not: associated parameters for chi and Sthree}) of $\chi$ with
respect to the basis $y,z$ are 
\[
a,b,c=-mb,d=-a,e=ma,f=mb,\omega=mb^{2}-a^{2},A=D=0,B=m,C=1
\]
In this case $\chi\in\stU_{\alpha}$.\end{lem}
\begin{proof}
First of all note that, if the associated parameters of $\chi$ are
as above, then they satisfy equations (\ref{eq:loc com and ass conditions}).
Therefore $\chi\in\GCov$ and, by definition, $\alpha$ is nowhere
a multiple of the identity, i.e. $\chi\in\stU_{\alpha}$. Consider
now the inverse implication and denote by $a,b,c,d,e,f,\omega,A,B,C,D$
the parameters associated to $\chi$ with respect to the basis $y,z$
of $\shF$. By definition of $y,z$ we have $A=0$ and $C=1$. Moreover
$m=A^{2}+BC=B$ and
\[
C(A+D)=0\then D=0
\]
\[
b(A+D)+C(a+d)=d(A+D)+C(c+f)=0\then d=-a\comma f=-c
\]
\[
(2aA+bB+cC)=(2cA+dB+eC)=0\then c=-mb,e=ma
\]
\[
a^{2}+bc=-\omega C\then\omega=mb^{2}-a^{2}
\]

\end{proof}

\begin{proof}
(\emph{of theorem }\ref{thm:not degenerate locus of S3}). By \ref{lem:associated parameters for alpha nowhere multiple of the identity}
and \ref{lem:the not degenerate locus is open}, $\stU_{\alpha}$
is an open substack of $\GCov$, $\chi\in\stU_{\alpha}(R)$ and the
induced map $\pi\colon\A^{3}\arr\stU_{\alpha}$ is a Zariski epimorphism.
It remains to prove that $\pi$ is smooth. Let $T\arrdi{\chi}\stU_{\alpha}$
be a map and consider the fiber product $Z=T\times_{\stU_{\alpha}}\A^{3}$.
We have to show that $Z$ is smooth over $T$. In order to do that,
since $\pi$ is a Zariski epimorphism, we can assume to have $(m,a,b)\in\odi T$
such that $\pi(m,a,b)=\chi$. Let $V$ be a $T$-scheme. An element
of the set $ $$Z(V)$ is a sequence $\Phi=(m',a',b',\lambda,u,v,w,z)\in\odi V^{8}$
such that, if we set $\psi_{\Phi}=\left(\begin{array}{cc}
u & v\\
w & z
\end{array}\right)$, then $\psi_{\Phi}\in\Gl_{2,V}$, $\lambda\in\odi V^{*}$ and $(\lambda,\psi_{\Phi})$
is an isomorphism $\pi(m',a',b')\arr\pi(m,a,b)$. We claim that the
map of $T$-schemes 
\[
i\colon Z\arr\A_{T}^{2}\times\Gm\comma i(m',a',b',\lambda,u,v,w,z)=(\lambda,u,w)
\]
is an open immersion. If we set 
\[
v(u,w)=\lambda mw\comma z(u,w)=\lambda w
\]
the condition $\psi_{\Phi}^{-1}\circ\alpha\circ(\lambda\otimes\psi_{\Phi})(e_{1})=e_{2}$
is equivalent to $v=v(u,w)$, $z=z(u,w)$. Since $\lambda,\psi_{\Phi}$
determine $m',a',b'$ and $\lambda,u,w$ determine $\lambda,\psi_{\Phi}$,
we can conclude that $i$ is a monomorphism. Define $U\subseteq\A_{T}^{2}\times\Gm$
as the open subscheme where $uz(u,w)-v(u,w)w$ is invertible. This
is just the expression of $\det\psi_{\Phi}$. Therefore $i(Z)\subseteq U$.
Consider now $\xi=(\lambda,u,w)\in U$ and define $\psi_{\xi}=\left(\begin{array}{cc}
u & v(u,w)\\
w & z(u,w)
\end{array}\right)$. Note that by construction $\psi_{\xi}\in\Gl_{2,V}$. In particular
there exists $\chi'=(\odi V,\odi V^{2},m',\alpha',\beta',\la-,-\ra')\in\stU_{\alpha}(V)$
such that $(\lambda,\psi_{\xi})\colon\chi'\arr\pi(m,a,b)$ is an isomorphism.
Since by construction
\[
\alpha'(e_{1})=\psi_{\Phi}^{-1}\circ\alpha\circ(\lambda\otimes\psi_{\Phi})(e_{1})=e_{2}
\]
from \ref{lem:associated parameters for alpha nowhere multiple of the identity}
we see that there exists $m',a',b'\in\odi V$ such that $\pi(m',a',b')=\chi'$.
In particular $\Phi=(m',a',b',\lambda,u,v(u,w),w,z(u,w))\in Z(V)$
and $i(\Phi)=\xi$.
\end{proof}

\subsection{The locus where $\beta\colon\Sym^{2}\shF\arr\shF$ is never zero.}

In this subsection we work over $\stR_{3}=\Z[1/6]$. Define $\stU_{\beta}$
as the full substack of $\GCov$ of objects $\chi=(\shL,\shF,m,\alpha,\beta,\la-,-\ra)$
such that $\beta\colon\Sym^{2}\shF\arr\shF$ is never zero, i.e. such
that $\beta$ is not zero after some base change. We want to prove
the following:
\begin{thm}
\label{thm:the locus where beta is not zero}Let $R=\stR_{3}[\omega,A,C]$.
Then 
\[
(R,R^{2},m,\alpha,\beta,\la-,-\ra)
\]
where
\[
\alpha=\left(\begin{array}{cc}
A & \omega C^{2}\\
C & -A
\end{array}\right);\begin{array}{l}
\beta(e_{1}^{2})=e_{2}\comma\beta(e_{1}e_{2})=-\omega Ce_{1}\comma\beta(e_{2}^{2})=2\omega Ae_{1}+\omega Ce_{2}\\
\la e_{1},e_{2}\ra=\omega\comma m=A^{2}+\omega C^{3}
\end{array}
\]
is an object of $\RCov G(R)$. The associated map $\A^{3}\arr\RCov G$
is a smooth zariski epimorphism onto $\stU_{\beta}$. In particular
$\stU_{\beta}$ is a smooth open substack of $\GCov$.
\end{thm}
Before proving this theorem we need two lemmas.
\begin{lem}
\label{lem:local basis for beta nowhere zero}Let $\chi=(\shL,\shF,m,\alpha,\beta,\la-,-\ra)\in\GCov$
(resp. $\beta\colon\Sym^{2}\shF\arr\shF$), $k$ be a field, $\Spec k\arr T$
be a map and $p\in T$ the induced point. If $\beta\otimes k\neq0$
(resp. $\beta\otimes k\neq0$, $(\tr\beta)\otimes k=0$ (see \ref{not: trace of beta}))
then there exists a Zariski open neighborhood $V$ of $p$ in $T$
and $y\in\shF_{|V}$ such that $y,\beta(y^{2})$ is a basis of $\shF_{|V}$.\end{lem}
\begin{proof}
If the statement is true when $T=\Spec k'$, for some field $k'$,
then it follows in general by Nakayama's lemma. So assume that $T=\Spec k$
and, by contradiction, that such a basis does not exist. Notice that
if $\chi\in\GCov$ is given, then $(\tr\beta)\otimes k=0$ thanks
to \ref{rem: cover with non zero trace}. Choosing a basis of $\shF$
we can write 
\[
\beta=\left(\begin{array}{ccc}
a & c & e\\
b & -a & -c
\end{array}\right)
\]
The condition that $y,\beta(y^{2})$ are dependent for all $y\in\shF$
is equivalent to
\[
bu^{3}-3au^{2}v-3cuv^{2}-ev^{3}=0\qquad\forall u,v\in k
\]
In particular, choosing $(u,v)\in\{(1,0),(0,1),(1,1),(1,-1))$, we
see that $b=e=3a=3c=0$ and therefore $\beta=0$, since $3$ is invertible.\end{proof}
\begin{lem}
\label{lem:associated parameters for beta nowhere zero}Let $\chi=(\shL,\shF,m,\alpha,\beta,\la-,-\ra)\in\stY$
and $y\in\shF$ be such that $\shL=\odi T$ and $y,z=\beta(y^{2})$
is a basis of $\shF$. Then $\chi\in\GCov$ if and only if the associated
parameters (see \ref{not: associated parameters for chi and Sthree})
of $\chi$ with respect to the basis $y,z$ are
\[
a=0,b=1,c=-\omega C,d=0,e=2\omega A,f=\omega C,\omega,A,B=\omega C^{2},C,D=-A
\]
In this case $\chi\in\stU_{\beta}$. \end{lem}
\begin{proof}
First of all, it is easy to check that, if the associated parameters
of $\chi$ are the ones listed in the statement, then they satisfy
equations (\ref{eq:loc com and ass conditions}). Therefore $\chi\in\GCov$
and, since $\beta(y^{2})\neq0$ after all base changes, $\chi\in\stU_{\beta}$.

Assume now that $\chi\in\GCov$. By definition of the basis $y,z$,
we have $a=0$ and $b=1$. Using relations (\ref{eq:loc com and ass conditions}),
we also have
\[
b(a+d)=a(a+d)+b(c+f)=0\then d=-a=0\comma f=-c
\]
\[
b(A+D)+C(a+d)=0\then D=-A
\]
\[
a^{2}+bc=-\omega C\comma ac+be=2\omega A\then c=-\omega C\comma e=2\omega A
\]
\[
2aA+bB+cC=0\then B=\omega C^{2}
\]

\end{proof}

\begin{proof}
(\emph{of theorem }\ref{thm:the locus where beta is not zero}). From
\ref{lem:local basis for beta nowhere zero} and \ref{lem:associated parameters for beta nowhere zero}
we see that $\stU_{\beta}$ is an open substack of $\GCov$, that
$(R,R^{2},m,\alpha,\beta,\la-,-\ra)\in\stU_{\alpha}(R)$ and that
its induced map $\pi\colon\A^{3}\arr\stU_{\beta}$ is a Zariski epimorphism.
It remains to prove that $\pi$ is smooth. Let $T\arrdi{\chi}\stU_{\beta}$
be a map and consider the fiber product $Z=T\times_{\stU_{\beta}}\A^{3}$.
We have to show that $Z$ is smooth over $T$. In order to do that,
since $\pi$ is a Zariski epimorphism, we can assume to have $(\omega,A,C)\in\odi T$
such that $\pi(\omega,A,C)=\chi$. Let $V$ be a $T$-scheme. An element
of the set $ $$Z(V)$ is a sequence $\Phi=(\omega',A',B',\lambda,u,v,w,z)\in\odi V^{8}$
such that, if we set $\psi_{\Phi}=\left(\begin{array}{cc}
u & v\\
w & z
\end{array}\right)$, then $\psi_{\Phi}\in\Gl_{2,V}$, $\lambda\in\odi V^{*}$ and $(\lambda,\psi_{\Phi})$
is an isomorphism $\pi(\omega',A',C')\arr\pi(\omega,A,C)$. We claim
that the map of $T$-schemes 
\[
i\colon Z\arr\A_{T}^{2}\times\Gm\comma i(\omega',A',C',\lambda,u,v,w,z)=(\lambda,u,w)
\]
is an open immersion. If we set 
\[
v(u,w)=2\omega w(Aw-Cu)\comma z(u,w)=u^{2}+\omega Cw^{2}
\]
the condition $\psi_{\Phi}^{-1}\circ\beta\circ(\Sym^{2}\psi_{\Phi})(e_{1}^{2})=e_{2}$
is equivalent to $v=v(u,w)$, $z=z(u,w)$. Since $\lambda,\psi_{\Phi}$
determine $\omega',A',C'$ and $\lambda,u,w$ determine $\lambda,\psi_{\Phi}$,
we can conclude that $i$ is a monomorphism. Define $U\subseteq\A_{T}^{2}\times\Gm$
as the open subscheme where $uz(u,w)-v(u,w)w$ is invertible. This
is just the expression of $\det\psi_{\Phi}$. Therefore $i(Z)\subseteq U$.
Consider now $\xi=(\lambda,u,w)\in U$ and define $\psi_{\xi}=\left(\begin{array}{cc}
u & v(u,w)\\
w & z(u,w)
\end{array}\right)$. Note that by construction $\psi_{\xi}\in\Gl_{2,V}$. In particular
there exists $\chi'=(\odi V,\odi V^{2},m',\alpha',\beta',\la-,-\ra')\in\stU_{\beta}(V)$
such that $(\lambda,\psi_{\xi})\colon\chi'\arr\pi(\omega,A,C)$ is
an isomorphism. Since by construction
\[
\beta'(e_{1}^{2})=\psi_{\Phi}^{-1}\circ\beta\circ(\Sym^{2}\psi_{\Phi})(e_{1}^{2})=e_{2}
\]
from \ref{lem:associated parameters for beta nowhere zero} we see
that there exists $\omega',A',C'\in\odi V$ such that $\pi(\omega',A',C')=\chi'$.
In particular $\Phi=(\omega',A',C',\lambda,u,v(u,w),w,z(u,w))\in Z(V)$
and $i(\Phi)=\xi$.
\end{proof}

\subsection{The regular representation and the stack of torsors $\Bi(\mu_{3}\rtimes\Z/2\Z)$.}

We want to describe the regular representation $\stR[G]$, as an algebra,
and the stack $\Bi G$ of $G$-torsors. By \ref{prop:functor of trivial torsor},
the $G$-algebra $\stR[G]$ is associated with the forgetful functor
$\Omega\colon\Loc^{G}\stR\arr\Loc\stR$. By the theory of representation
of $G$ and taking into account how we have associated with a functor
the object of $\GCov$, it is easy to deduce that the sequence $\chi=(\shL,\shF,m,\alpha,\beta,\la-,-\ra)\in\GCov(\stR)$
associated with $\Omega$ is given by
\[
\shL=A,\shF=V;\begin{array}{l}
\alpha(1_{A}\otimes v_{1})=-v_{1}\\
\alpha(1_{A}\otimes v_{2})=v_{2}
\end{array};\begin{array}{l}
\beta(v_{1}^{2})=v_{2}\comma\beta(v_{1}v_{2})=0\comma\beta(v_{2}^{2})=v_{1}\\
\la v_{1},v_{2}\ra=(-1/2)1_{A}\comma m(1_{A}\otimes1_{A})=1
\end{array}
\]
where $I_{G}=\{\stR,A,V\}$.
\begin{defn}
\label{def:discriminant maps for Sthree} Given $\Phi=(\shF,\delta)\in\shC_{3}$,
we define the discriminant map $\Delta_{\Phi}\colon(\det\shF)^{2}\arr\odi T$
as the determinant of the map $ $$\shF\arr\duale{\shF}$ induced
by $\eta_{\delta}\colon\Sym^{2}\shF\arr\odi T$. Given $\chi=(\shL,\shF,m,\alpha,\beta,\la-,-\ra)\in\GCov$
we define the map $\Delta_{\chi}\colon(\det\shF)^{2}\arr\odi T$ as
the determinant of the map $ $$\shF\arr\duale{\shF}$ induced by
$(-,-)_{\chi}\colon\Sym^{2}\shF\arr\odi T$.\end{defn}
\begin{rem}
\label{rem:general discriminant}The map $\Delta_{\chi}$ coincides
with minus the composition
\[
(\det\shF)^{2}\arrdi{\la-,-\ra^{\otimes2}}\shL^{2}\arrdi m\odi T
\]
Moreover, if $\tr\beta=0$, then $\Delta_{(\shF,\delta_{\beta})}=4\Delta_{\chi}$
thanks to \ref{eq:eta delta is two times symmetric product}. For
the first claim, we can argue locally, i.e. choosing a basis $y,z$
of $\shF$, setting $\shL=\odi T$ and considering the parameters
associated with $\chi$. In this case
\[
\Delta_{\chi}=(y,y)(z,z)-(y,z)^{2}=-BC\omega^{2}-A^{2}\omega^{2}=-\omega^{2}m
\]
\end{rem}
\begin{thm}
\label{thm:description of Sthree torsors}An object $\chi=(\shL,\shF,m,\alpha,\beta,\la-,-\ra)\in\GCov$
(resp. $\in\RCov{S_{3}}$ over $\stR_{3}$) corresponds to a $G$-torsor
(resp. $S_{3}$-torsor) if and only if the maps
\[
m\colon\shL^{2}\arr\odi T\comma\la-,-\ra\colon\det\shF\arr\shL
\]
are isomorphisms, or, equivalently, $\Delta_{\chi}\colon(\det\shF)^{2}\arr\odi T$
is an isomorphism. In this case: $\alpha$ is an isomorphism, $\beta,(-,-)$
are surjective and $\tr\beta=\tr\alpha=0$. Moreover $\Bi G\subseteq\stU_{\omega},\stU_{\alpha}$
and, over $\stR_{3}$, $\Bi G\subseteq\stU_{\beta}$. Finally the
map $\Lambda$ of Theorem \ref{thm:The-locus when omega is invertible}
is an isomorphism from the full substack of $\shC_{3}$ of objects
$\Phi$ such that $\Delta_{\Phi}$ is an isomorphism to $\Bi G$.\end{thm}
\begin{proof}
The claims about $S_{3}$ follows from the same claims about $G$
because, over $\stR_{3}$, the isomorphism $\GCov\simeq\RCov{S_{3}}$
preserves the torsors. Let $\Omega$ be the functor associated to
$\chi$. Since $G$ is super solvable, by Theorem \ref{thm:torsors when omega surjective},
$\chi$ corresponds to a $G$-torsor if and only if $m\colon\shL^{2}=\Omega_{A}\otimes\Omega_{A}\arr\odi T$
and $(-,-)\colon\shF\otimes\shF=\Omega_{V}\otimes\Omega_{\duale V}\arr\odi T$
are surjective. The first condition says that $m$ is an isomorphism
and, in particular, that $\alpha$ is an isomorphism. Moreover it
is easy to check, locally, that in this case $(-,-)$ is surjective
if and only $\la-,-\ra$ is an isomorphism. By definition of $\Delta_{\chi}$,
this map is an isomorphism if and only if both $m$ and $\la-,-\ra$
are isomorphisms. Except for the last sentence of the statement, all
the other claimed properties follow by checking them on $\stR[G]$.

Since $\Bi G\subseteq\stU_{\omega}$ and $\Lambda$ is an isomorphism,
we get an isomorphism $\Lambda^{-1}(\Bi G)\arr\Bi G$ and, by \ref{rem:general discriminant},
$\Lambda^{-1}(\Bi G)$ is the substack of $\shC_{3}$ of objects $\Phi$
such that $\Delta_{\Phi}$ is an isomorphism.\end{proof}
\begin{rem}
\label{rem: fake and true discriminant}Let $\Phi=(\shF,\delta)\in\shC_{3}$,
$\alA_{\Phi}=\odi T\oplus\shF$ be the algebra associated with $\Phi$
(see \ref{rem:triple covers and invariants by sigma}) and assume
to work over $\stR_{3}$. Then $\det\alA_{\Phi}\simeq\det\shF$ and
the determinant of the map $\alA_{\Phi}\arr\duale{\alA_{\Phi}}$ induced
by the trace map $\tr_{\alA_{\Phi}}\colon\alA_{\Phi}\arr\odi T$ coincides
with $\Delta_{\Phi}$. In particular $\alA_{\Phi}$ is étale if and
only if $\Delta_{\Phi}$ is an isomorphism.\end{rem}
\begin{cor}
Set $\shF=\stR^{2}$ with basis $e_{1},e_{2}$ and consider $\delta\colon\Sym^{3}\shF\arr\det\shF$
given by $\delta(e_{2}^{3})=-\delta(e_{1}^{3})=1$ and $\delta(e_{1}e_{2}^{2})=\delta(e_{1}^{2}e_{2})=0$.
Then
\[
G\simeq\Autsh_{\shC_{3}}(\shF,\delta)
\]
Assume now that the base scheme is $\stR_{3}$. Then the map $\Bi G\arr\Cov_{3}$,
obtained by taking invariants by $\sigma\in\Z/2\Z$, is an isomorphism
onto the locus $\text{Et}_{3}$ of étale degree $3$ covers. In particular
\[
G\simeq\Autsh_{\Cov_{3}}(\stR_{3}[t]/(t^{3}-1))
\]
\end{cor}
\begin{proof}
For the first claim, it is enough to note that $\Lambda(\shF,\delta)$
induces the $G$-cover $\stR[G]$. For the second one, assume we work
over $\stR_{3}$ and let $\Phi=(\shF,\delta)\in\shC_{3}$. Thanks
to \ref{rem: fake and true discriminant}, the map $\Lambda$ of Theorem
\ref{thm:The-locus when omega is invertible} yields an isomorphism
$\text{Et}_{3}\arr\Bi G$, whose inverse is the map $\Bi G\arr\Cov_{3}$
of the statement. In particular $G\simeq\Autsh_{\Cov_{3}}\stR_{3}[G]^{\sigma}$
and it is easy to check that $\stR[G]^{\sigma}\simeq\stR[t]/(t^{3}-1)$.\end{proof}
\begin{rem}
The above corollary gives an alternative proof of the fact that $\Bi G\simeq\Bi S_{3}$
over $\stR_{3}$ (see \ref{ex:SThree covers and the other group}).
Indeed $\Bi G\simeq\text{Et}_{3}$ and it is a classical result that
$\text{Et}_{3}\simeq\Bi S_{3}$. Moreover we see that the $S_{3}$-torsor
$P$ corresponding to $\stR_{3}[t]/(t^{3}-1)\in\text{Et}_{3}$ is
a $(G,S_{3})$-torsor over $\stR_{3}$ (see \ref{def:bitorsors}).
\end{rem}

\subsection{Irreducible components of $(\mu_{3}\rtimes\Z/2\Z)$-Cov and $\RCov{S_{3}}$.}

In this subsection we want to prove that $\GCov$ and, over $\stR_{3}$,
$\RCov{S_{3}}$ have exactly two irreducible components. Moreover
we will show that they are universally reducible and nonreduced. We
will also describe the irreducible component of $\GCov$ that is not
the principal one, that is $\stZ_{G}$. 
\begin{defn}
Let $\stX\arr T$ be an algebraic stack over $T$. We will say that
$\stX$ is \emph{universally not reduced }over $T$ if for any base
change $T'\arr T$ the stack $\stX\times_{T}T'$ is not reduced.
\end{defn}
The theorem we want to prove is:
\begin{thm}
\label{thm:Geometry of Sthree cov}Let $R=\stR[a,b,c,d,e,f,\omega,A,B,C,D]/(\text{relations})$
where the relations are the ones given in (\ref{eq:loc com and ass conditions}).
Then $\chi=(R,R^{2},m,\alpha,\beta,\la-,-\ra)$ where 
\[
\alpha=\left(\begin{array}{cc}
A & B\\
C & D
\end{array}\right);\begin{array}{l}
e_{1}^{2}=ae_{1}+be_{2}\comma e_{1}e_{2}=ce_{1}+de_{2}\comma e_{2}^{2}=ee_{1}+fe_{2}\\
\la e_{1},e_{2}\ra=\omega\comma m=A^{2}+BC
\end{array}
\]
is an object of $\RCov G(R)$ and the associated map $\Spec R\arr\RCov G$
is a $(\Gm\times\Gl_{2})$-torsor. In particular
\[
\RCov G\simeq[\Spec R/(\Gm\times\Gl_{2})]
\]
The stack $\GCov$ (resp. $\RCov{S_{3}}$) is geometrically connected,
universally not reduced and universally reducible over $\stR$ (resp.
$\stR_{3}$) and it has two irreducible components that are geometrically
irreducible.

The minimal primes of $R$ are
\[
P_{1}=(a+d,c+f,A+D)\comma P_{2}=(a,b,c,d,e,f,\omega,B,C,A-D)
\]
and the irreducible components of $\RCov G$ are $\stZ_{G}$ and 
\[
\stZ=\{\beta=\la-,-\ra=0\text{ and }\alpha\text{ is (fppf) locally a multiple of the identity}\}
\]
Moreover we have isomorphisms   \[   \begin{tikzpicture}[xscale=4.5,yscale=-0.6]     
\node (A0_0) at (0.18, 0) {$[\A^{1}/\Gm]\times\Bi\Gl_{2}$};     
\node (A2_1) at (0.53, 0) {$\simeq$};     
\node (A0_1) at (1, 0) {$\{(\shL,\shF,\mu)\st\mu\colon\shL\arr\odi{}\}$};     
\node (A0_2) at (2, 0) {$\stZ$};     
\node (A1_1) at (1, 1) {$(\shL,\shF,\mu)$};     
\node (A1_2) at (2, 1) {$(\shL,\shF,\mu^2,\mu\otimes \id_\shF,0,0)$};     

\path (A0_1) edge [->]node [auto] {$\scriptstyle{\simeq}$} (A0_2);     \path (A1_1) edge [|->,gray]node [auto] {$\scriptstyle{}$} (A1_2);   \end{tikzpicture}   \] 
\end{thm}
Before proving this theorem, we need some lemmas.
\begin{lem}
\label{lem:alpha locally id}Let $\shL,\shF$ be respectively an invertible
sheaf and a rank $2$ locally free sheaf. Then 
\[
\Homsh(\shL,\odi S)\arrdi{-\otimes\shF}\Homsh(\shL\otimes\shF,\shF)
\]
is a an isomorphism onto the locus of maps $\alpha\colon\shL\otimes\shF\arr\shF$
that are (fppf) locally a multiple of the identity.\end{lem}
\begin{proof}
Clearly the map $-\otimes\shF$ is injective and has the right image.
So we have to prove that given $\alpha\colon\shL\otimes\shF\arr\shF$
which is locally a multiple of the identity, there exists $\shL\arrdi{\mu}\odi S$
such that $\alpha=\mu\otimes\id$. Set 
\[
\mu\colon\shL\arr\shL\otimes\shF\otimes\duale{\shF}\arrdi{\alpha\otimes\id}\shF\otimes\duale{\shF}\arr\odi S
\]
We want to prove that $\alpha=(\mu/2)\otimes\id$. We can assume $\shL=\odi S$,
$\shF=\odi S^{2}$ and $\alpha=\lambda\id$ whit $\lambda\in\odi S$.
It is easy to check that $\mu$ is the multiplication for $2\lambda$,
so that $\alpha=(\mu/2)\otimes\id$.\end{proof}
\begin{lem}
\label{lem:three dimensional local k algebras}Let $k$ be a field.
Then, up to isomorphism, the only local $k$-algebras $A$ with $\dim_{k}A=3$
and $\dim_{k}m_{A}=2$ are
\[
A=k[x]/(x^{3})\text{ and }A=k[x,y]/(x^{2},xy,y^{2})
\]
\end{lem}
\begin{proof}
Let $W=\Ann m_{A}$. We have $0\subsetneq W\subseteq m_{A}$. Assume
first that $\dim_{k}W=1$ and let $x\in m_{A}-W$. We want to prove
that $1,x,x^{2}$ is a basis of $A$. In this case we will have $A\simeq k[X]/(X^{3})$.
Consider an expression
\[
a+bx+cx^{2}=0
\]
Since $x\in m_{A}$, we have $a=0$. In particular $(b+cx)\in m_{A}$
because it is a zero divisor and again we can conclude that $b=0$.
Finally $x\notin W$ implies $x^{2}\neq0$ and therefore $c=0$. Assume
now $\dim_{k}W=2$, i.e. $W=\Ann m_{A}$. If $x,y$ is a basis of
$m_{A}$ then we have a surjective map
\[
k[X,Y]/(X^{2},XY,Y^{2})\arr A
\]
which is an isomorphism by dimension.\end{proof}
\begin{lem}
\label{Lem: the degenerate locus outside U}Let $R=\stR[b,\omega,A]/(Ab,A\omega)$.
Then 
\[
(R,R^{2},m,\alpha,\beta,\la-,-\ra)
\]
 where 
\[
\grave{\alpha=\left(\begin{array}{cc}
A & 0\\
0 & A
\end{array}\right)};\begin{array}{l}
\beta(e_{1}^{2})=be_{2}\comma\beta(e_{1}e_{2})=\beta(e_{2}^{2})=0\\
\la e_{1},e_{2}\ra=\omega\comma m=A^{2}
\end{array}
\]
is an object of $\RCov G(R)$. The induced map $\Spec R\arr\RCov G$
is topologically surjective onto the locus $\RCov{|G}|-|\stU_{\alpha}|$.\end{lem}
\begin{proof}
Set $\chi$ for the object defined in the statement. A direct computation
on the relations (\ref{eq:loc com and ass conditions}) shows that
$\chi\in\GCov(R)$. On the other hand, since $\alpha$ is globally
a multiple of the identity, we also have that the image of $\Spec R$
in $|\GCov|$ lies outside $|\stU_{\alpha}|$.

Now we have to prove that the map in the statement is topologically
surjective. Let $k$ be a field and $\chi=(\shL,\shF,m,\alpha,\beta,\la-,-\ra)\in\RCov G(k)$
such that $\chi\notin\stU_{\alpha}$. Consider the $k$-algebra $\alA^{\sigma}\simeq k\oplus\shF=\stB$,
where $\sigma\in\Z/2\Z$ and where the multiplication is given by
$\Sym^{2}\shF\arrdi{2(-,-)_{\chi}+\beta}k\oplus\shF$. The vector
space $\shF\subseteq\stB$ is an ideal because $\alpha=\lambda\id$
for some $\lambda\in k$ and $(u,v)_{\chi}=\la\alpha(v),u\ra=\lambda\la v,u\ra=0$,
because it is both symmetric and antisymmetric. Therefore $\shF$
is a maximal ideal of $\stB$ and $\shB/\shF=k$. In particular, from
\ref{lem:three dimensional local k algebras}, $\stB$ is isomorphic
to either $k[x]/(x^{3})$ or $k[x,y]/(x^{2},xy,y^{2})$. So there
exists a basis $y,z$ of $\shF$ such that $\beta(y^{2})=by$, $\beta(yz)=\beta(z^{2})=0$
and therefore the parameters associates with $\chi$ with respect
to this basis satisfy 
\[
a=c=d=e=f=B=C=0,D=A=\lambda
\]

\end{proof}

\begin{proof}
(\emph{of Theorem} \ref{thm:Geometry of Sthree cov}) The results
about $\RCov{S_{3}}$ follow from the same results about $\GCov$.
Note that $\Spec R$ represents the functor of $G$-equivariant structures
of commutative, associative $\stR$-algebras over $\stR[G]$. In particular
the group $\shH=\Autsh_{G,1}\stR[G]$ of the $G$-equivariant isomorphisms
of $G$-comodules preserving $1\in\stR$ acts on $\Spec R$ and it
is easy to verify that $\GCov$ is the quotient stack of $\Spec R$
by this group. Finally the representation theory of $G$ tell us that
$\shH\simeq\Gm\times\Gl_{2}$ and therefore that $\RCov G\simeq[\Spec R/\Gm\times\Gl_{2}]$.

If $\stR'$ is an $\stR$-algebra, thanks to \ref{rem: cover with non zero trace},
we know that $a+d,c+f$ belongs to all the prime ideals of $R\otimes_{\stR}\stR'$,
but $a+d,c+f\neq0$ in $R\otimes_{\stR}\stR'$. Therefore $R$ and
$\RCov G$ are universally not reduced. Since all the relations in
(\ref{eq:loc com and ass conditions}) are homogeneous, $R$ is a
$\N$-graded $\stR$-algebra such that $R_{0}=\stR$. In particular
$\Spec\stR$ and therefore $\GCov$ are geometrically connected.

We now focus on the irreducible components of $\GCov$. Let $\stR'$
be a domain over $\stR$, $R'=R\otimes_{\stR}\stR'$ and continue
to denote by $P_{1},P_{2}$ the ideal in the statement of $R'$. Notice
that $R'/P_{2}=\stR'[A]$. In particular $P_{2}$ is prime, $A+D\notin P_{2}$
and therefore $P_{1}\subsetneq P_{2}$. Now let $P$ be a prime ideal
of $R'$. We want to show that $P_{1}\subseteq P$ or $P_{2}\subseteq P$.
If $A+D\in P$ then $P_{1}\subseteq P$. If $A+D\notin P$, then,
taking into accounts (\ref{eq:loc com and ass conditions}) and the
fact that $a+d,c+f\in P$, it is easy to check that $P_{2}\subseteq P$.
Since $R'_{\omega}\neq0$, $\omega$ is not nilpotent, so there exists
a minimal prime $P'$ such that $\omega\notin P'$. In particular
$P_{2}\subsetneq P'$ and therefore $P_{1}\subseteq P'$. If by contradiction
$P'$ is the only minimal prime, i.e. $\Spec R'$ is irreducible,
then $P'\subseteq P_{2}$ and therefore $P_{1}\subseteq P_{2}$, which
is not the case. In particular there exists a minimal prime $P''$
such that $P''\subseteq P_{2}$. Again if $P_{1}\subseteq P''$ we
find a contradiction, so $P_{2}\subseteq P''$ and $P_{2}=P''$ is
a minimal prime. So $\Spec R$ is reducible and, by \ref{rem: irreducibility and connectedness preserved by particular fppf epimorphism},
the same conclusion holds for $\RCov G$. Moreover, having considered
an arbitrary base change to a domain, it also follows that $\Spec(R/P_{2})$
and the closed substack of $\GCov$ induced, which is $\stZ$, are
geometrically an irreducible component. We want now to show that $\GCov$
has exactly $2$ irreducible components, namely $\stZ_{G}$ and $\stZ$.
In particular, by \ref{rem: irreducibility and connectedness preserved by particular fppf epimorphism},
it follows that $P'$ and $P_{2}$ are the only minimal primes of
$R$. Moreover, since $P'$ is the only minimal prime over $P_{1}$,
we can also conclude that $\sqrt{P_{1}}=P'$. 

Let $\stZ_{2}$ be the closed substack $\stZ$ defined in the statement
and $\stZ_{1}$ the closed substack where $\alpha$ vanishes. If $R'$
is the algebra defined in \ref{Lem: the degenerate locus outside U},
we see that $|\stZ_{1}|$ and $|\stZ_{2}|$ are the image of, respectively,
$\Spec(R'/(A))$ and $\Spec(R'/(b,\omega))$ under the map $\Spec R'\arr\GCov$.
In particular $|\stZ_{1}|$ and $|\stZ_{2}|$ are irreducible and
$|\RCov G|=|\stU_{\alpha}|\cup|\stZ_{1}|\cup|\stZ_{2}|$. Denote by
$\stU_{3}$ the open locus where $3$ is invertible, i.e. $\stU_{3}=\Spec\stR_{3}\times_{\stR}\GCov$
and by $\stU_{\beta}$ the locus where $\beta$ is never $0$ and
$3$ is invertible. Although we are working on $\stR$, $\stU_{\beta}\subseteq\stU_{3}$
is exactly the stack considered in \ref{thm:the locus where beta is not zero}.
Since $\Bi G\subseteq\stU_{\alpha}$, $\emptyset\neq\Bi G\cap\stU_{3}\subseteq\stU_{\beta}$
and both $\stU_{\alpha}$ and $\stU_{\beta}$ are irreducible thanks
to \ref{thm:not degenerate locus of S3} and \ref{thm:the locus where beta is not zero},
we can conclude that $\overline{|\stU_{\alpha}|}=\overline{|\stU_{\beta}|}=|\stZ_{G}|$.
On the other hand $|\stU_{\beta}|\cap|\stZ_{1}|\neq\emptyset$, because
it contains the algebra locally given by $\alpha=\la-,-\ra=0$ and
$a=c=d=e=f=0$, $b=1$, which is well defined thanks to \ref{Lem: the degenerate locus outside U}.
Therefore $|\stZ_{1}|\subseteq\overline{|\stU_{\beta}|}=|\stZ_{G}|$.
In particular $|\GCov|=|\stZ_{G}|\cup|\stZ|$ and, because it is reducible,
$\stZ_{G}$ and $\stZ$ are the only irreducible components of $\GCov$.

The last isomorphisms follows from \ref{lem:alpha locally id}. In
order to prove that $\stZ_{G}$ is geometrically irreducible, it is
enough to prove that, if $k$ is a field over $\stR$, then $|\stZ_{G}\times k|\cap|\stZ\times k|\subseteq|\stZ_{G\times k}|$.
The stacks $\stZ_{G}\times k$ and $\stZ\times k$ are induced by
$\Spec(R/\sqrt{P_{1}})\otimes k$ and $\Spec(R/P_{1})\otimes k$ respectively,
whose intersection is the point where $m=\alpha=\beta=\la-,-\ra=0$.
But $\Spec R\otimes k$ is a cone with this point as vertex and therefore
any irreducible component of $\Spec R\otimes k$ must contain it.

Finally the fact that $P_{1}$ is a prime can be checked using the
software Macaulay$2$: if 
\[
I=(2aA+bB+cC,2cA-aB+eC,a^{2}+bc+\omega C,ac+be-2\omega A,c^{2}-ae-B\omega)
\]
as ideal of $R'=\Z[a,b,c,e,\omega,A,B,C]$, we have $R/P_{1}\simeq(R'/I)_{2}$,
where $(-)_{2}$ denotes the localization by $2\in\Z$, and Macaulay$2$
tells us that $I$ is a prime ideal of $R'$.

\end{proof}

\subsection{The main irreducible components $\stZ_{(\mu_{3}\rtimes\Z/2\Z)}$
and $\stZ_{S_{3}}$.}

In this subsection we want to give a more precise description of the
irreducible component $\stZ_{G}$ of $\GCov$ and, consequently, of
$\stZ_{S_{3}}\subseteq\RCov{S_{3}}$ over $\stR_{3}$. In particular
we will have to consider maps $\shL\otimes\shF\arr\shF$ whose trace
is zero (see \ref{not: trace of a map}) and we first want to describe
them.
\begin{rem}
Let $\shN$ and $\shF$ be respectively an invertible sheaf and a
locally free sheaf of rank $2$. Given a map $\alpha\colon\shN\otimes\shF\arr\shF$,
we have a factorization   \[   \begin{tikzpicture}[xscale=1.8,yscale=-0.5]     \node (A0_0) at (0, 0) {$s\otimes p \otimes q$};     \node (A0_2) at (2, 0) {$\alpha(s\otimes q)p-\alpha(s\otimes p)q$};     \node (A1_0) at (0, 1) {$\shN\otimes\shF\otimes \shF$};     \node (A1_2) at (2, 1) {$\Sym^2 \shF$};     \node (A3_1) at (1, 3) {$\shN\otimes \det \shF$};     \path (A1_0) edge [->]node [auto] {$\scriptstyle{}$} (A1_2);     \path (A3_1) edge [->,dashed]node [auto] {$\scriptstyle{}$} (A1_2);     \path (A1_0) edge [->]node [auto] {$\scriptstyle{}$} (A3_1);     \path (A0_0) edge [|->,gray]node [auto] {$\scriptstyle{}$} (A0_2);   \end{tikzpicture}   \] 
that yields an isomorphism   \begin{align}\label{iso:primo alpha}
  \begin{tikzpicture}[xscale=4.6,yscale=-0.7]     \node (A0_0) at (0, 0) {$\Homsh_{\tr=0}(\shN\otimes\shF,\shF)$};     \node (A0_1) at (1, 0) {$\Homsh(\shN\otimes \det \shF,\Sym^2\shF)$};     \node (A1_0) at (0, 1) {$\left(\begin{array}{cc} u & v\\ w & -u \end{array}\right)$};     \node (A1_1) at (1, 1) {$(v\ \ \ \ \ -2u\ \   -w)$};     \path (A0_0) edge [->]node [auto] {$\scriptstyle{}$} (A0_1);     \path (A1_0) edge [|->,gray]node [auto] {$\scriptstyle{}$} (A1_1);   \end{tikzpicture}
\end{align}where the last row describes the behaviour of the map if $\shN=\odi T$
and a basis $y,z$ of $\shF$ is given.

We want to introduce also a different description. In order to do
that we note that the pairing $\Sym^{2}\shF\otimes\Sym^{2}(\duale{\shF})\arr\odi T$
defined by $uv\otimes\xi\eta\longmapsto\xi(u)\eta(v)+\eta(u)\xi(v)$
induces an isomorphism   \[   \begin{tikzpicture}[xscale=4.1,yscale=-0.6]     \node (A0_0) at (0, 0) {$\Sym^2 \duale \shF$};     \node (A0_1) at (1, 0) {$\duale{(\Sym^2 \shF)}$};     \node (A1_0) at (0, 1) {$(y^*)^2, y^*z^*,(z^*)^2$};     \node (A1_1) at (1, 1) {$2(y^2)^*,(yz)^*,2(z^2)^*$};     \path (A0_0) edge [->]node [auto] {$\scriptstyle{}$} (A0_1);     \path (A1_0) edge [|->,gray]node [auto] {$\scriptstyle{}$} (A1_1);   \end{tikzpicture}   \] where,
again, the last row expresses the behaviour of the map if a basis
$y,z$ of $\shF$ is given. The composition

\[
\Sym^{2}\shF\simeq\Sym^{2}(\duale{\shF}\otimes\det\shF)\simeq\Sym^{2}\duale{\shF}\otimes(\det\shF)^{2}\simeq\duale{(\Sym^{2}\shF)}\otimes(\det\shF)^{2}
\]
yields an isomorphism \begin{align}\label{iso:terzo alpha}
\begin{tikzpicture}[xscale=5,yscale=-0.6]     \node (A0_0) at (0, 0) {$\Homsh(\shN \otimes (\det \shF)^2,\Sym^2\shF)$};     \node (A0_1) at (1, 0) {$\Homsh(\Sym^2\shF,\shN^{-1})$};     \node (A1_0) at (0, 1) {$(u\ \ \ \ \ v\ \  \ \ \  w)$};     \node (A1_1) at (1, 1) {$(2w \ \ -v\ \ \ \ \   2u)$};     \path (A0_0) edge [->]node [auto] {$\scriptstyle{\check{-}}$} (A0_1);     \path (A1_0) edge [|->,gray]node [auto] {$\scriptstyle{}$} (A1_1);   \end{tikzpicture}  
\end{align}where the last row describes the behaviour of the map if $\shN=\odi T$
and a basis $y,z$ of $\shF$ is given.\end{rem}
\begin{notation}
\label{not: associated check map}We continue to keep notation introduced
above: given $\zeta\colon\shN\otimes\det\shF^{2}\arr\Sym^{2}\shF$
we will denote by $\check{\zeta}\colon\Sym^{2}\shF\arr\shN^{-1}$
the associated map. Moreover we will also denote by $\hat{-}$ the
inverse of $\check{-}$: given $\eta\colon\Sym^{2}\shF\arr\shN^{-1}$
the associated map will be $\hat{\eta}\colon\shN\otimes(\det\shF)^{2}\arr\Sym^{2}\shF$.
\end{notation}
Remarks above motivate the introduction of the following stack. Define
the stack $\stX$ whose objects are sequences $(\shM,\shF,\delta,\zeta,\omega)$
where $\shM$ is an invertible sheaf, $\shF$ is a locally free sheaf
of rank $2$, $\omega$ is a section of $\shM$ and $\delta,\zeta$
are maps
\[
\delta\colon\Sym^{3}\shF\arr\det\shF\comma\zeta\colon(\det\shF)^{2}\otimes\shM\arr\Sym^{2}\shF
\]
satisfying the following conditions:
\begin{enumerate}
\item the composition
\begin{equation}
(\det\shF)^{2}\otimes\shM\otimes\shF\arrdi{\zeta\otimes\id}\Sym^{2}\shF\otimes\shF\arr\Sym^{3}\shF\arrdi{\delta}\det\shF\label{eq:first condition on M,F,delta,zeta,omega}
\end{equation}
is zero;
\item the composition 
\begin{equation}
\Sym^{2}\shF\arrdi{\check{\zeta}}\shM^{-1}\arrdi{\duale{\omega}}\odi S\label{eq:second condition on M,F,delta,zeta,omega}
\end{equation}
 coincides with $\eta_{\delta}$ (see \ref{not: associated check map}
and (\ref{eq:etabeta})).
\end{enumerate}
Given an object $\chi=(\shM,\shF,\delta,\zeta,\omega)\in\stX$ we
define $\shL_{\chi}=\shM\otimes\det\shF$, $\alpha_{\chi}\colon\shL_{\chi}\otimes\shF\arr\shF$
the map obtained from $\zeta$ using (\ref{iso:primo alpha}), $m_{\chi}=-\det\alpha_{\chi}\colon\shL_{\chi}^{2}\arr\odi T$,
$\beta_{\chi}=\beta_{\delta}\colon\Sym^{2}\shF\arr\shF$ (see (\ref{iso:delta beta}))
and finally $\la-,-\ra_{\chi}=\omega\otimes\id_{\det\shF}\colon\det\shF\arr\shM\otimes\det\shF=\shL_{\chi}$.

Remembering the notation introduced in \ref{not: trace of beta},
we want to prove the following Theorem.
\begin{thm}
\label{thm:description of the main component}The main irreducible
component $\stZ_{G}$ of $\GCov$ is the closed locus of $\GCov$
of objects $(\shL,\shF,m,\alpha,\beta,\la-,-\ra)$ such that
\[
\tr\alpha\colon\shL\arr\odi T\text{ and }\tr\beta\colon\shF\arr\odi T
\]
vanish. Moreover we have an isomorphism of stacks   \[   \begin{tikzpicture}[xscale=4.5,yscale=-0.6]     \node (A0_0) at (0, 0) {$\stX$};     \node (A0_1) at (1, 0) {$\stZ_G$};     \node (A1_0) at (0, 1) {$\chi=(\shM,\shF,\delta,\zeta,\omega)$};     \node (A1_1) at (1, 1) {$(\shL_{\chi},\shF,m_\chi, \alpha_{\chi},\beta_{\chi},\la-,-\ra_\chi)$};     \path (A0_0) edge [->]node [auto] {$\scriptstyle{}$} (A0_1);     \path (A1_0) edge [|->,gray]node [auto] {$\scriptstyle{}$} (A1_1);   \end{tikzpicture}   \] %

\end{thm}
Notice that, again, the above result continue to hold if we replace
$\stZ_{G}$ by $\stZ_{S_{3}}$ and we assume to work over $\stR_{3}$.
Before proving this Theorem we want to give an explicit description
of the objects of $\stX$ (and a posteriori of $\stZ_{G}$) such that
$\omega\in\shM$ is an effective Cartier divisor, i.e. the map $\odi T\arrdi{\omega}\shM$
is injective. This will be helpful when we will have to study $G$-covers
whose total space is regular.

Given a scheme $T$, denote by $\stZ_{\omega}(T)$ the category whose
objects are sequences $(\shM,\shF,\delta,\omega)$, where $\shM$
is an invertible sheaf, $\shF$ is a rank $2$ locally free sheaf,
$\delta$ is a map $\delta\colon\Sym^{3}\shF\arr\det\shF$ and $\omega$
is a section of $\shM$ such that $\odi S\arrdi{\omega}\shM$ is injective
and such that the zero locus of $\eta_{\delta}\colon\Sym^{2}\shF\arr\odi S$
contains the zero locus of $\omega$, or, equivalently, such that
\[
\Imm\eta_{\delta}\subseteq\Imm(\shM^{-1}\arrdi{\duale{\omega}}\odi T)
\]

With an object $\chi=(\shM,\shF,\delta,\omega)\in\stZ_{\omega}(T)$
we associate the unique map $\zeta_{\chi}\colon(\det\shF)^{2}\otimes\shM\arr\Sym^{2}\shF$
such that the associated map $\check{\zeta_{\chi}}$ (see \ref{not: associated check map})
is a factorization as in   \[   \begin{tikzpicture}[xscale=2.0,yscale=-1.2]     \node (A0_0) at (0, 0) {$\Sym^2\shF$};     \node (A0_1) at (1, 0) {$\shM^{-1}$};     \node (A1_1) at (1, 1) {$\odi{S}$};     \path (A0_0) edge [->,dashed]node [auto] {$\scriptstyle{\check{\zeta}_\chi}$} (A0_1);     \path (A0_0) edge [->]node [auto] {$\scriptstyle{\eta_\delta}$} (A1_1);     \path (A0_1) edge [->]node [auto] {$\scriptstyle{\duale \omega}$} (A1_1);   \end{tikzpicture}   \] 
\begin{thm}
\label{thm:description of ZG when omega in Cartier}Let $T$ be a
scheme. The association   \[   \begin{tikzpicture}[xscale=3.8,yscale=-0.6]     \node (A0_0) at (0, 0) {$\stZ_{\omega}(T)$};     \node (A0_1) at (1, 0) {$\stZ_{G}(T)$};     \node (A1_0) at (0, 1) {$\chi=(\shM,\shF,\delta,\omega)$};     \node (A1_1) at (1, 1) {$(\shM,\shF,\delta,\zeta_\chi,\omega)$};     \path (A0_0) edge [->]node [auto] {$\scriptstyle{}$} (A0_1);     \path (A1_0) edge [|->,gray]node [auto] {$\scriptstyle{}$} (A1_1);   \end{tikzpicture}   \] 
is an equivalence of categories onto the full subcategory of $\stZ_{G}(S)$
where $\omega\colon\odi S\arr\shM$ is injective.
\end{thm}
By Theorem \ref{thm:description of the main component}, we can define
a forgetful map $\Delta\colon\stZ_{G}\arr\shC_{3}$ (see \ref{def:definition of Cthree}),
that, when $3$ is invertible, is the restriction of the map $\GCov\arr\Cov_{3}$
obtained by taking invariants by $\sigma\in\Z/2\Z$. We have already
seen (\ref{thm:The-locus when omega is invertible}) that $\Delta$
has a section. On the other hand, if $\chi=(\shF,\delta)\in\shC_{3}(T)$
and we denote by $\stZ_{\chi}$ the fiber of $\stZ_{\omega}(T)\arr\shC_{3}(T)$,
i.e. the subcategory of $\stZ_{\omega}(T)$ of objects $\phi\in\stZ_{\omega}(T)$
such that $\Delta(\phi)=\chi$ and morphisms $\Phi\arrdi{\psi}\Phi'$
such that $\Delta(\psi)=\id_{\chi}$ , from \ref{thm:description of ZG when omega in Cartier},
it easy to deduce the following:
\begin{cor}
\label{cor:extensions of Cthree objects with Domega cartier}Let $T$
be a scheme and $\chi=(\shF,\delta)\in\shC_{3}(T)$. Then $\stZ_{\chi}$
is a set, i.e. there exists at most one isomorphism between two of
its objects, and the following maps   \[   \begin{tikzpicture}[xscale=7.0,yscale=-0.8]     \node (A0_0) at (0, 0) {$\Imm(\shM^{-1}\arrdi \omega \odi T)$};     \node (A0_1) at (1, 0) {$(\shM,\shF,\delta,\omega)$};     \node (A1_0) at (0, 1) {$\left\{ \begin{array}{c} \text{invertible sheaves of ideals }\shN\subseteq\odi T\\ \text{such that }\Imm\eta_{\delta}\subseteq\shN \end{array}\right\} $};     \node (A1_1) at (1, 1) {$\stZ_\chi$};     \node (A2_0) at (0, 2) {$\shN$};     \node (A2_1) at (1, 2) {$(\shN^{-1},\shF,\delta,1)$};     \path (A1_0) edge [->]node [auto] {$\scriptstyle{}$} (A1_1);     \path (A2_0) edge [|->,gray]node [auto] {$\scriptstyle{}$} (A2_1);     \path (A0_1) edge [|->,gray]node [auto] {$\scriptstyle{}$} (A0_0);   \end{tikzpicture}   \] 
are inverses of each other.
\end{cor}
We will prove Theorems \ref{thm:description of the main component}
and\emph{ }\ref{thm:description of ZG when omega in Cartier} just
after the following lemma.
\begin{lem}
\label{lem:for the description of Z Sthree}Let $\chi=(\shL,\shF,m,\alpha,\beta,\la-,-\ra)\in\stY$
such that $\tr\alpha=\tr\beta=0$ and set $\shM=\shL\otimes\det\shF^{-1}$.
Let also $\zeta\colon\shM\otimes(\det\shF)^{2}\arr\Sym^{2}\shF$ the
map associated with $\alpha$ through the isomorphism (\ref{iso:primo alpha})
and $\delta=\delta_{\beta}\colon\Sym^{3}\shF\arr\det\shF$ (see (\ref{iso:delta beta})).
If $\shL=\odi T$, $y,z$ is a basis of $\shF$ and we use notation
from \ref{not: associated parameters for chi and Sthree}, we have
equivalences
\[
\text{the map (}\ref{eq:first condition on M,F,delta,zeta,omega}\text{) is zero}\iff\beta\circ\zeta=0\iff\left\{ \begin{array}{c}
2aA+bB+cC=0\\
2cA+eC-aB=0
\end{array}\right.
\]
\[
\text{the map (}\ref{eq:second condition on M,F,delta,zeta,omega}\text{) coincides with }\eta_{\delta}\iff\left\{ \begin{array}{c}
a^{2}+bc=-\omega C\\
ac+be=2\omega A\\
c^{2}-ae=B\omega
\end{array}\right.
\]
\end{lem}
\begin{proof}
The conditions $\tr\alpha=\tr\beta=0$ means that $a+d=c+f=A+D=0$.
Note that we have expressions 
\[
\zeta=By^{2}-2Ayz-Cz^{2}\comma\check{\zeta}=-2C(y^{2})^{*}+2A(yz)^{*}+2B(z^{2})^{*}
\]
thanks to (\ref{iso:primo alpha}) and (\ref{iso:terzo alpha}). In
particular
\[
\beta(\zeta)=(aB-2cA-eC)y+(2aA+bB+cC)z
\]
 and, by definition of $\delta_{\beta}$, the composition \ref{eq:first condition on M,F,delta,zeta,omega}
is given by 
\[
(\beta(\zeta)\wedge y)y^{*}+(\beta(\zeta)\wedge z)z^{*}=-(2aA+bB+cC)y^{*}+(aB-2cA-eC)z^{*}
\]
Therefore the first equivalence is clear. For the second one, note
that the map \ref{eq:second condition on M,F,delta,zeta,omega} is
just $\omega\check{\zeta}$. Therefore the last equivalence easily
follows taking into account the expression of $\eta_{\delta}$ given
in \ref{eq:expression of eta delta}.
\end{proof}

\begin{proof}
(\emph{of theorem \ref{thm:description of the main component})} The
result follows easily from \ref{thm:Geometry of Sthree cov}, \ref{lem:for the description of Z Sthree}
and the local conditions (\ref{eq:loc com and ass conditions}), taking
into account that $\tr\alpha=\tr\beta=0$ means that, locally, the
relations $a+d=c+f=A+D=0$ hold.\end{proof}
\begin{notation}
We identify the stack $\stX$ with the stack $\stZ_{G}$. Given $\chi=(\shM,\shF,\delta,\zeta,\omega)\in\stZ_{G}$
we will continue to denote by $\shL_{\chi},m_{\chi},\alpha_{\chi},\beta_{\chi},\la-,-\ra_{\chi},(-,-)_{\chi}$
the objects associated with $\chi$. Moreover we will often omit the
$(-)_{\chi}$ if this will not lead to confusion. \end{notation}
\begin{rem}
Given $\chi=(\shM,\shF,\delta,\zeta,\omega)\in\stZ_{G}$, we want
to show an alternative way of retrieving the map $m_{\chi}\colon\shL_{\chi}^{2}=\shM^{2}\otimes(\det\shF)^{2}\arr\odi S$,
which will be useful in the next chapter. Indeed it is easy to check
locally that the following composition is just $-4m_{\chi}$. 
\[
\shM^{2}\otimes(\det\shF)^{2}\arrdi{\id_{\shM}\otimes\zeta}\shM\otimes\Sym^{2}\shF\arrdi{\id_{\shM}\otimes\check{\zeta}}\shM\otimes\shM^{-1}\simeq\odi T
\]
\end{rem}
\begin{proof}
(\emph{of Theorem }\ref{thm:description of ZG when omega in Cartier})
If $(\shM,\shF,\delta,\zeta,\omega)\in\stZ_{G}(T)$ is such that $\omega$
is injective, then $\chi=(\shM,\shF,\delta,\omega)\in\stZ_{\omega}(T)$
and $\zeta=\zeta_{\chi}$, because by definition $\duale{\omega}\circ\check{\zeta}=\eta_{\delta}$.
Conversely, given $\chi=(\shM,\shF,\delta,\omega)\in\stZ_{\omega}(T)$,
we need to prove that $\Phi=(\chi,\zeta_{\chi})\in\stZ_{G}(T)$. By
construction the condition $\duale{\omega}\circ\check{\zeta}_{\chi}=\eta_{\delta}$
is satisfied. Given $\beta=\beta_{\delta}$ and taking into account
\ref{lem:for the description of Z Sthree}, it remains to prove that
$\beta\circ\zeta_{\chi}=0$. Note that the composition 
\[
(\det\shF)^{2}\arrdi{\omega\otimes\id_{(\det\shF)^{2}}}\shM\otimes(\det\shF)^{2}\arrdi{\zeta_{\chi}}\Sym^{2}\shF
\]
is just $\hat{\eta_{\delta}}$ (see \ref{not: associated check map}).
Since $\odi T\arrdi{\omega}\shM$ is injective, we only need to show
that $\beta\circ\hat{\eta_{\delta}}=0$. Taking into account the local
descriptions (\ref{iso:terzo alpha}) and \ref{eq:expression of eta delta},
locally we have:
\begin{alignat*}{1}
\beta(\hat{\eta_{\delta}}) & =(c^{2}-ae)(ay+bz)-(ac+be)(cy-az)+(a^{2}+bc)(ey-cz)\\
 & =[(c^{2}-ae)a-(ac+be)c+(a^{2}+bc)e]y+\\
 & [(c^{2}-ae)b+(ac+be)a-(a^{2}+bc)c]z=0
\end{alignat*}

\end{proof}

\section{Normal and Regular $(\mu_{3}\rtimes\Z/2\Z)$-covers and $S_{3}$-covers.}

In this section we want to study the following problem: given a regular
in codimension $1$ (resp. normal, regular) integral, noetherian scheme
$Y$ describe the $G$-covers $X\arr Y$ such that $X$ is regular
in codimension $1$ (resp. normal, regular). Since, as we will see,
this problem is related to the same problem for degree $3$ covers,
we will simply call them triple covers. Moreover, because we want
that $G$-torsors of regular schemes are regular $G$-covers, in this
chapter we assume that $G$ is étale, that is we assume that the base
ring is $\stR_{3}$. In other words we will work with schemes $Y$
such that $6\in\odi Y^{*}$. Notice that, under this assumption, the
above problems for $G$-covers are equivalent to the same problems
for $S_{3}$-covers.
\begin{rem}
\label{rem: regular SThree if and only if regular G covers}The isomorphism
$\GCov\simeq\RCov{S_{3}}$ preserves the regularity of covers, that
is if $X\arr Y$ is a $G$-cover and $X'\arr Y$ is the associated
$S_{3}$-cover then, if $Y$ is regular in codimension $1$ (resp.
normal, regular) then $X$ has the same property if and only if $X'$
has it. Moreover, if $Y$ is defined over a scheme $S$ then $X$
is smooth (geometrically connected) over $S$ if and only if $X'$
is so. Indeed $G$ and $S_{3}$ are étale locally isomorphic over
$\stR_{3}$ and the same conclusion holds for $X$ and $X'$ (see
\ref{thm:bitorsors and GCov}). Moreover the properties of being regular
in codimension $1$, normal, regular, smooth and geometrically connected
are all local and they all satisfy descent in the étale topology.
\end{rem}
We continue to keep notation for which $\GCov$ and $\RCov{S_{3}}$
are identified with the stack of data $\chi=(\shL,\shF,m,\alpha,\beta,\la-,-\ra)$
describing them. In particular anything that is defined starting from
$\chi$ is automatically associated with the corresponding $G$-cover
and $S_{3}$-cover. Moreover, as in the other sections of this chapter,
we continue to use $\GCov$ instead of $\RCov{S_{3}}$, but here this
is really just a notation for the stack of data $\chi$. We introduce
some loci associated with a $G$-cover or an $S_{3}$-cover.
\begin{defn}
\label{def:closed subscheme associated to Gcovers first}Given $\chi=(\shL,\shF,m,\alpha,\beta,\la-,-\ra)\in\GCov(Y)$
we define:
\begin{itemize}
\item $D_{m}$ as the closed subscheme of $Y$ defined by $m\colon\shL^{2}\arr\odi Y$;
\item $D_{\omega}$ as the closed subscheme of $Y$ defined by $\la-,-\ra\colon\det\shF\arr\shL$;
\item $Y_{\alpha}$ as the zero locus of the map $\alpha\colon\shL\otimes\shF\arr\shF$,
that coincides with the closed subscheme of $Y$ defined by $\shL\otimes\shF\otimes\duale{\shF}\arrdi{\alpha\otimes\id}\shF\otimes\duale{\shF}\arr\odi Y$.
\end{itemize}
\end{defn}
Notice that we have an inclusion $Y_{\alpha}\subseteq D_{m}$.
\begin{notation}
Let $Y$ be a scheme. Given $\chi=(\shL,\shF,m,\alpha,\beta,\la-,-\ra)\in\GCov(Y)$
we continue to keep notation introduced in section \ref{sec:Global description of Sthree covers}.
In particular we have the map $(-,-)_{\chi}\colon\Sym^{2}\shF\arr\odi Y$
and $\alA_{\chi}$ will denote the associated $G$-equivariant algebra.
We will also denote by $X_{\chi}=\Spec\alA_{\chi}$ and by $f_{\chi}\colon X_{\chi}\arr Y$
the associated $G$-cover. When $\chi\in\stZ_{G}(Y)$ we will also
write $\chi=(\shM,\shF,\delta,\zeta,\omega)$ as in \ref{thm:description of the main component}
and denote by $\shL_{\chi},m_{\chi},\alpha_{\chi},\beta_{\chi},\la-,-\ra_{\chi}$
its associated objects. When we have also that $\odi Y\arrdi{\omega}\shM$,
or, equivalently $\la-,-\ra_{\chi}$, is injective, i.e. $\chi\in\stZ_{\omega}(Y)$,
we will write $\chi=(\shM,\shF,\delta,\omega)$ as in \ref{thm:description of ZG when omega in Cartier}
and denote by $\zeta_{\chi}$ its associated object. We will often
omit the $(-)_{\chi}$ if this will not lead to confusion. When $\tr\beta=0$,
(condition that holds as soon as $Y$ is reduced) we will denote by
$\delta=\delta_{\beta}$ and $\eta_{\delta}$ the maps introduced
in \ref{iso:delta beta} and \ref{eq:etabeta} respectively. Remember
that $\eta_{\delta}=2(-,-)_{\chi}$ (see \ref{eq:expression of eta delta}).
\end{notation}
Since we will work mostly over regular local rings, we also introduce
the following notation and remarks.
\begin{notation}
\label{not: regular Sthree covers} Let $(R,m_{R},k)$ be a local
ring and $\chi=(\shL,\shF,m,\alpha,\beta,\la-,-\ra)\in\RCov G(R)$,
so that $\alA=\alA_{\chi}=R\oplus\shL\oplus\shF_{1}\oplus\shF_{2}$
is its associated algebra. We will denote by $t$ a generator of $\shL$
and by $y,z$ a basis of $\shF$ and we will use the parameters associated
with $\chi$ as in \ref{not: associated parameters for chi and Sthree}.
We set $\alA_{0}=R\oplus\shL$, which is an algebra such that $t^{2}=m$,
and $y_{1},z_{1}$ and $y_{2},z_{2}$ the basis of $\shF_{1}$ and
$\shF_{2}$ respectively equal to $y,z\in\shF$. The structure of
algebra of $\alA$ is given by 
\[
ty_{1}=Ay_{1}+Cz_{1}\comma tz_{1}=By_{1}+Dz_{1}\comma ty_{2}=-Ay_{2}-Cz_{2}\comma tz_{2}=-By_{1}-Dz_{1}
\]
\[
y_{1}^{2}=ay_{2}+bz_{2}\comma y_{1}z_{1}=cy_{2}-az_{2}\comma z_{1}^{2}=ey_{2}-cz_{2}\comma y_{2}^{2}=ay_{1}+bz_{1}\comma y_{2}z_{2}=cy_{1}-az_{1}\comma z_{2}^{2}=ey_{1}-cz_{1}
\]
\[
y_{1}y_{2}=-C\omega\comma y_{1}z_{2}=A\omega+\omega t\comma z_{1}z_{2}=B\omega
\]
In particular the schemes $D_{m}$, $D_{\omega}$, $Y_{\alpha}$ are
the $\Spec$ of, respectively, $R/(m)$, $R/(\omega)$, $R/(A,B,C,D)$.
\end{notation}

\subsection{Normal and regular in codimension $1$ $(\mu_{3}\rtimes\Z/2\Z)$-covers.}

In this subsection we want to describe $G$-covers of regular in codimension
$1$ (resp. normal) schemes whose total space is regular in codimension
$1$ (resp. normal) and we will apply the general theory developed
in section \ref{sub:Regularity in codimension one}. In particular
the following result is a direct corollary of \ref{thm:equivalent conditions for regularity for glrg, global version}.
Such result can be recovered by the results proved in the following
sections, where we will describe regular $G$-covers.
\begin{thm}
\label{thm:application to Sthree of theory on regular in codimension 1 covers}Let
$Y$ be an integral, noetherian and regular in codimension $1$ (resp.
normal) scheme such that $\dim Y\geq1$ and $6\in\odi Y^{*}$, let
$\chi=(\shL,\shF,m,\alpha,\beta,\la-,-\ra)\in\GCov(Y)$ and denote
by $f\colon X\arr Y$ the associated $G$-cover (resp. $S_{3}$-cover).
Then $X$ is regular in codimension $1$ (resp. normal) if and only
if 
\[
\codim D_{m}\cap D_{\omega}\geq2\text{ and }D_{m},D_{\omega}\text{ are regular in codimension }1\text{ Cartier divisors}
\]
In this case $f\colon X\arr Y$ is generically a $G$-torsor (resp.
$S_{3}$-torsor).\end{thm}
\begin{proof}
Notice that it is enough to prove the statement only for the group
$G$. Remember that $G$ is a glrg over $\stR$ and $I_{G}=\{\stR,A,V\}$.
We want to apply Theorem \ref{thm:equivalent conditions for regularity for glrg, global version}.
Denote by $\Omega$ the functor associated with $\chi$. Note that
the action of $G$ on $X_{\chi}$ is generically faithful because
we are assuming $\chi\in\GCov$. We have to consider the maps $\shL=\Omega_{A}\arr\duale{\Omega_{\duale A}}=\duale{\shL}$and
$\shF=\Omega_{V}\arr\duale{\Omega_{\duale V}}=\duale{\shF}$ induced
by $m\colon\shL^{2}\arr\odi Y$ and $(-,-)_{\chi}\colon\shF\otimes\shF\arr\odi Y$
and their determinants $s_{f,A}=m\in\shL^{-2}$ and $s_{f,V}\in(\det\shF)^{-2}$.
Theorem \ref{thm:equivalent conditions for regularity for glrg, global version}
tells us that $X_{\chi}$ is regular in codimension $1$ (resp. normal)
if and only if $v_{p}(s_{f,A})\leq1$ and $v_{p}(s_{f,V})\leq2$ for
all $p\in Y^{(1)}$. We are using the convention for which $v_{p}(0)=\infty$,
so that the previous conditions also implies that $s_{f,A},s_{f,V}\neq0$.
Notice that $s_{f,V}=s_{\omega}^{2}\otimes s_{f,A}\in(\det\shF^{-1}\otimes\shL)^{2}\otimes\shL^{-2}$,
where $s_{\omega}=\la-,-\ra_{\chi}\in\det\shF^{-1}\otimes\shL$ and
that the zero loci of $s_{f,A}$ and $s_{\omega}$ are respectively
$D_{m}$ and $D_{\omega}$. Therefore $v_{p}(s_{f,A})\leq1$ and $v_{p}(s_{f,V})\leq2$
for all $p\in Y^{(1)}$ means that $D_{m}\cap D_{\omega}\cap Y^{(1)}=\emptyset$
and that $D_{m}$ and $D_{\omega}$ are regular over the codimension
$1$ points of $Y$. In this case $f$ is generically a $G$-torsor
again thanks to \ref{thm:equivalent conditions for regularity for glrg, global version}.
\end{proof}

\subsection{Regular $(\mu_{3}\rtimes\Z/2\Z)$-covers and $S_{3}$-covers.}

In this subsection we are mainly interested in regular $G$-covers
and $S_{3}$-covers. The Theorem we will prove is the following.
\begin{thm}
\label{thm:regular Sthree covers first}Let $Y$ be a regular, noetherian
and integral scheme such that $\dim Y\geq1$ and $6\in\odi Y^{*}$
and let $\chi=(\shL,\shF,m,\alpha,\beta,\la-,-\ra)\in\GCov(Y)$. Then
its associated $G$-cover ($S_{3}$-cover) is regular if and only
if the following conditions hold:
\begin{enumerate}
\item $D_{m},D_{\omega}$ are Cartier divisors and $D_{m}\cap D_{\omega}=\emptyset$;
\item $Y_{\alpha}=\emptyset$ or $Y_{\alpha}$ is regular and has pure codimension
$2$ in $Y$;
\item $D_{\omega}$ is regular and $D_{m}$ is regular outside $Y_{\alpha}$.
\end{enumerate}
Moreover in this case $\chi\in\stU_{\beta}$, the locus where $\beta$
is never zero.
\end{thm}
In the following lemmas we will work over a regular local ring $R$
and we will consider given an object $\chi=(\shL,\shF,m,\alpha,\beta,\la-,-\ra)\in\GCov(R)$.
\begin{lem}
\label{lem:split when omega alpha are not zero}The map $(-,-)\colon\Sym^{2}\shF\arr R$
is surjective if and only if $\omega$ is invertible and $\alpha\otimes k\neq0$.
Moreover in this case $\chi$ lies in the locus where $\alpha$ is
never a multiple of the identity and $\Spec\alA\arr\Spec\alA_{0}$
is étale.\end{lem}
\begin{proof}
We have that $(y,y)=-C\omega$, $(y,z)=(z,y)=A\omega$, $(z,z)=B\omega$.
So the first claim holds. Assume now by contradiction that $\alpha\otimes k$
is a multiple of the identity. So over $k$ we have $B=C=0$ and $A=D\neq0$.
On the other hand from (\ref{eq:loc com and ass conditions}) we see
that $\omega(A+D)=0$ and so $A=D=0$. In particular we can consider
$y,z=\alpha(y)$ as basis of $\shF$. So
\[
y_{1}^{3}=y_{1}(ay_{2}+bz_{2})=-\omega(a-bt)\in\alA^{*}\then\alA\in\Bi\mu_{3}(\alA_{0})
\]
since $\omega=mb^{2}-a^{2}=-(a-bt)(a+bt)\in\alA^{*}$.\end{proof}
\begin{lem}
\label{lem:Split of A where m is invertible}Assume that $m$ is invertible,
that $\alpha$ is never a multiple of the identity and that a $\lambda\in R$
such that $\lambda^{2}=m$ is given. Then $\alA_{\chi}\simeq\alB^{2}$
as $R$-algebras, where $\alB=R\oplus Rv_{1}\oplus Rv_{2}\in\RCov{\mu_{3}}(R)$
satisfies: $v_{1}^{2}=(a+\lambda b)v_{2}$, $v_{2}^{2}=(a-\lambda b)v_{1}$.
Moreover we have $v_{1}v_{2}=-\omega$, $v_{1}^{3}=-(a+\lambda b)\omega$,
$v_{2}^{3}=-(a-\lambda b)\omega$.\end{lem}
\begin{proof}
We consider a basis of $\shF$ of the form $y,z=\alpha(y)$. Set $u=\frac{1}{2}+\frac{t}{2\lambda}\in\alA_{0}$
and note that it is an idempotent. Since $u$ is $\mu_{3}$-invariant,
$\sigma\in\Z/2\Z\subseteq G$ yields an isomorphism of $R$-algebras
$\alB=\alA/u\alA\simeq\alA/(1-u)\alA$, since $\sigma(u)=1-u$. In
particular $\alA\simeq\alB^{2}$. Since $\alA_{0}/u\alA_{0}\simeq R$,
we have that $\alB\in\RCov{\mu_{3}}(R)$ and its decomposition is
given by 
\[
\alB=R\oplus\shF_{1}/u\shF_{1}\oplus\shF_{2}/u\shF_{2}
\]
An easy computation shows that $uz_{1}=\lambda uy_{1}$ and that $2uz_{1}=\lambda y_{1}+z_{1}$.
In particular $\shF_{1}/u\shF_{1}$ is generated by $v_{1}=y_{1}$
and $z_{1}=-\lambda v_{1}$. Similarly we get that $v_{2}=y_{2}$
generates $\shF_{2}/u\shF_{2}$ and $z_{2}=\lambda v_{2}$. It is
now easy to deduce the desired relations.\end{proof}
\begin{lem}
\label{lem:regularity of codimension one type}Assume that $\dim R\geq1$
and that $m$ is invertible or the map $(-,-)\colon\shF\otimes\shF\arr\odi S$
is surjective. Then $\alA_{\chi}$ is regular if and only if one of
the following conditions holds: both $\omega$ and $m$ are invertible;
$m$ is invertible and $\omega\in m_{R}-m_{R}^{2}$; $\omega$ is
invertible and $m\in m_{R}-m_{R}^{2}$. In all of those cases $\alpha\otimes k$
is not a multiple of the identity and $\beta\otimes k\neq0$.\end{lem}
\begin{proof}
We first want to show that we can assume that $\alpha\otimes k$ is
not a multiple of the identity. Since $\omega\tr\alpha=0$ and, if
$\alA_{\chi}$ is regular then $\alA_{\chi}\in\stZ_{G}(R)$ thanks
to \ref{lem:generically G torsor implies faithful action}, we can
conclude that $\tr\alpha=0$, taking into account that $R$ is a domain
and \ref{thm:description of the main component}. On the other hand,
since $\tr\alpha=A+D$, then $\alpha\otimes k$ is a multiple of the
identity if and only if $\alpha\otimes k=0$, that cannot happen if
$m$ is invertible or $(-,-)$ is surjective.

\emph{Case}: $m$ \emph{invertible}. Without loss of generality, we
can assume $m=\lambda^{2}$ for some $\lambda\in R$. From \ref{lem:Split of A where m is invertible},
$\alA_{\chi}\simeq\alB^{2}$ and so $\alA_{\chi}$ is regular if and
only if $\alB$ is so. If $\omega$ is invertible then $\alB$ is
étalè over $R$, otherwise $\alB$ is regular if and only if
\[
-(a-\lambda b)\omega=(a-\lambda b)^{2}(a+\lambda b)=(a-\lambda b)(a+\lambda b)^{2}\in m_{R}-m_{R}^{2}
\]
This happens if and only if only one between $(a-\lambda b),(a+\lambda b)$
is in $m_{R}-m_{R}^{2}$, which is equivalent to $\omega\in m_{R}-m_{R}^{2}$.

\emph{Case}: $(-,-)$ \emph{surjective}. By \ref{lem:split when omega alpha are not zero},
$\Spec\alA_{\chi}\arr\Spec\alA_{0}$ is étalè. Therefore $\alA_{\chi}$
is regular if and only if $\alA_{0}$ is so, which is equivalent to
the condition: $m$ invertible or $m\in m_{R}-m_{R}^{2}$.

Finally note that, since $\alpha\otimes k$ is not a multiple of the
identity, $\omega=mb^{2}-a^{2}$ and so
\[
\beta\otimes k=0\iff a,b\in m_{R}\then\omega\in m_{R}^{2}
\]
\end{proof}
\begin{lem}
\label{lem:regularity outside codimension one}Assume that $\dim R\geq1$,
that $m$ is not invertible and that $(-,-)\colon\Sym^{2}\shF\arr R$
is not surjective. If $\alA_{\chi}$ is regular, then $\beta\otimes k\neq0$,
$\alpha\otimes k=0$ and $\omega$ is invertible. If such conditions
are satisfied and we choose a basis $y,z=\beta(y^{2})$ of $\shF$
as in \ref{thm:the locus where beta is not zero}, then $\alA_{\chi}$
is regular if and only if $A,C$ are independent in $m_{R}/m_{R}^{2}$.\end{lem}
\begin{proof}
Note that $\alA_{\chi}$ is local with maximal ideal $m_{\alA_{\chi}}=m_{R}\oplus\shL\oplus\shF_{1}\oplus\shF_{2}$.
Set $\alA=\alA_{\chi}$, $\shT=m_{\alA}/m_{\alA}^{2}$, $\alA_{k}=\alA\otimes k$,
$\shT_{k}=m_{\alA_{k}}/m_{\alA_{k}}^{2}$. We claim that there exists
an exact sequence 
\[
0\arr m_{R}/(m_{R}\cap m_{\alA}^{2})\arr\shT\arr\shT_{k}\arr0
\]
Indeed $\Ker(\shT\arr\shT_{k})=m_{R}\alA/(m_{R}\alA\cap m_{\alA}^{2})=Q$
and $m_{R}/(m_{R}\cap m_{\alA}^{2})\hookrightarrow Q$. On the other
this map is surjective since $m_{R}\shL\oplus m_{R}\shF_{1}\oplus m_{R}\shF_{2}\subseteq m_{\alA}^{2}\cap m_{R}\alA$.
We also have $m_{R}\cap m_{\alA}^{2}=(m_{R}^{2},m,\Imm(-,-)_{\chi})$
because $m_{\alA}^{2}$ is a sub $G$-comodule of $\alA$. Since $\dim\alA=\dim R$,
if we denote by $W$ the $k$-vector subspace of $m_{R}/m_{R}^{2}$
generated by $m$ and $\Imm(-,-)_{\chi}$, we can conclude that 
\[
\alA\text{ regular}\iff\dim_{k}\shT_{k}\leq\dim_{k}W
\]
Notice that 
\[
\shT_{k}=Q_{0}\oplus Q_{1}\oplus Q_{2}\text{ with }Q_{0}\simeq k/(\omega)\text{ and }Q_{1}\simeq Q_{2}\simeq\shF\otimes k/(\Imm(\alpha\otimes k),\Imm(\beta\otimes k))
\]
Indeed $\shT_{k}$ is a $G$-comodule and also a $\mu_{3}$-comodule,
thus the decomposition with the $Q_{i}$. Moreover $\sigma\in\Z/2\Z$
yields an isomorphism $Q_{1}\simeq Q_{2}$ and the relation $y_{1}z_{2}-z_{1}y_{2}=2\omega t$
implies that $Q_{0}=\shL/(\omega t)$. The last isomorphism follows
because $m_{\alA_{k}}^{2}\cap(\shF_{1}\otimes k)=(\shF_{2}^{2},\shL\shF_{1})$. 

Assume that $\beta\otimes k=0$. In particular $\Imm(-,-)_{\chi}\subseteq m_{R}^{2}$
and therefore $W=\la m\ra$. On the other hand, since $m\in m_{R}$,
$\alpha\otimes k$ is not surjective and therefore $Q_{1},Q_{2}\neq0$.
In conclusion $\dim_{k}W\leq1$, while $\dim_{k}\shT_{k}\geq2$ and
we see that $\alA$ cannot be regular if $\beta\otimes k=0$.

We now assume that $\beta\otimes k\neq0$ and, thanks to \ref{lem:local basis for beta nowhere zero},
we choose a basis $y,z=\beta(y^{2})$ for $\shF$ as in the statement.
Moreover, thanks to \ref{lem:associated parameters for beta nowhere zero},
the parameters associated with $\chi$ satisfy $a=0$, $b=1$, $c=-\omega C$,
$e=2\omega A$, $B=\omega C^{2}$, $D=-A$ and $m=A^{2}+\omega C^{3}$.
Since $\beta(y^{2})=z$ we see that $\shQ_{1}\simeq k/(A,\omega C)$.
Moreover by definition $W=\la A^{2},\omega A,\omega C\ra$. If $\omega\in m_{R}$,
it follows that $A\in m_{R}$, since $m\in m_{R}$. In this case $\dim_{k}W\leq1$,
while $\dim_{k}\shT_{k}\geq2$. In particular if $\alA$ is regular
then $\omega$ is invertible. By the hypothesis on $(-,-)$ and thanks
to \ref{lem:split when omega alpha are not zero}, it also follows
that $\alpha\otimes k=0$. If we assume that $\omega$ is invertible
and that $\alpha\otimes k=0$, we get $\dim_{k}\shT_{k}=2$ and $W=\la A,C\ra$.
In this case $\alA$ is regular if and only if $\dim_{k}W\geq2$,
which exactly means that $A,C$ are independent in $m_{R}/m_{R}^{2}$.
\end{proof}

\begin{proof}
(\emph{of Theorem }\ref{thm:regular Sthree covers first}) It is enough
to prove the statement for the group $G$. Denote by $f_{\chi}\colon X_{\chi}\arr Y$
the $G$-cover associated with $\chi$. Notice that $D_{m},D_{\omega}$
are Cartier divisor if and only if $f_{\chi}$ generically $G$-torsor
thanks to \ref{thm:description of Sthree torsors}, condition true
when $X_{\chi}$ is regular, by \ref{lem:regular implies cyclic stabilizer}.
In what follows we therefore assume this condition. In particular
$Y_{\alpha}\subsetneq Y$. Since $\dim Y\geq1$, we can reduce the
problem to the case when $Y=\Spec R$, where $R$ is a local, regular
ring with $\dim R\geq1$. In particular we use notation from \ref{not: regular Sthree covers}.
We also set $I_{\alpha}=(A,B,C,D)$. The conditions in the statement
become:
\begin{enumerate}
\item $m,\omega\neq0$; $m\notin m_{R}$ or $\omega\notin m_{R}$;
\item if $I_{\alpha}\neq R$ then $R/I_{\alpha}$ is regular and $\alt I_{\alpha}=2$;
\item $\omega\notin m_{R}^{2}$; $m\notin m_{R}^{2}$ if $I_{\alpha}=R$.
\end{enumerate}
We split the proof in two parts, according to lemmas \ref{lem:regularity of codimension one type}
and \ref{lem:regularity outside codimension one}.

$m$ \emph{invertible or} $(-,-)$\emph{ surjective}. Note that both
conditions imply that $I_{\alpha}=R$ by \ref{lem:split when omega alpha are not zero}
and the result easily follows from \ref{lem:regularity of codimension one type}.

$m$ \emph{not invertible and} $(-,-)$\emph{ not surjective}. If
$\alA_{\chi}$ is regular the result easily follows from \ref{lem:regularity outside codimension one}.
Taking into account the same lemma, we have only to show that conditions
$1)$, $2)$, $3)$ imply that $\omega$ is invertible and that $\beta\otimes k\neq0$.
The first one is clear by $1)$, since $m\in m_{R}$. In particular
$(-,-)$ not surjective implies that $\alpha\otimes k=0$, by \ref{lem:split when omega alpha are not zero}.
Assume by contradiction that $\beta\otimes k=0$. By the relations
\ref{eq:loc com and ass conditions}, we see that $A=-D,B,C\in m_{R}^{2}$.
In particular $I_{\alpha}\subseteq m_{R}^{2}$ and therefore $R/I_{\alpha}$
is not regular, against condition $2)$.
\end{proof}

\subsection{Regular $(\mu_{3}\rtimes\Z/2\Z)$-covers, $S_{3}$-covers and triple
covers.}

In this subsection we want to compare regular $G$-covers and regular
$S_{3}$-covers with regular triple covers. In particular we will
show that any regular $G$-cover ($S_{3}$-covers) induces a regular
triple cover, by taking invariants by $\sigma\in\Z/2\Z$. Conversely
a regular triple cover can be extended to a $G$-cover ($S_{3}$-cover),
provided that a certain codimension $2$ condition is fulfilled. We
will also show how it is possible to construct regular $G$-covers
($S_{3}$-cover) of a smooth variety. We start introducing some loci
associated to a triple cover.
\begin{defn}
\label{def:closed subscheme associated to Gcovers second}Let $Y$
be a scheme and $\Phi=(\shF,\delta)\in\shC_{3}(Y)$. We define
\begin{itemize}
\item $Y_{\delta}$ as the closed subscheme of $Y$ defined by $\eta_{\delta}\colon\Sym^{2}\shF\arr\odi Y$;
\item $D_{\delta}$ as the closed subscheme of $Y$ defined by the map $\Delta_{\Phi}\colon(\det\shF)^{2}\arr\odi Y$
introduced in \ref{def:discriminant maps for Sthree}.
\end{itemize}
When $Y$ is integral we denote by $\Cov_{3}^{\text{nd}}(Y)$ the
full subcategory of $\Cov_{3}(Y)$ of objects $(\shF,\delta)$ such
that $Y_{\delta}$ is a proper subscheme of $Y$, or, equivalently,
such that $\eta_{\delta}\neq0$.
\end{defn}
The suffix $\text{nd}$ in the symbol $\Cov_{3}^{\text{nd}}(Y)$ stands
for 'not degenerate'. Indeed, for a triple cover $f\colon X\arr Y$
associated with $(\shF,\delta)\in\Cov_{3}$, the closed subscheme
of $Y$ where $\eta_{\delta}=0$ coincides, topologically, with the
locus where $f$ has triple points (see \ref{lem:three dimensional local k algebras}).
Moreover, the complementary of the locus $D_{\delta}$ is the étale
locus of $f$ (see \ref{rem: fake and true discriminant}).

We now want to state three Theorems we want to prove in this section.
In oder to do so we have to introduce the divisorial component of
a subscheme:
\begin{defn}
If $Y$ is a locally factorial, noetherian and integral scheme and
$Z$ is a proper closed subscheme, there exists a maximum among the
effective Cartier divisors contained in $Z$. Such divisor will be
denoted by $D(Z)$ and called the \emph{divisorial component} of $Z$
in $Y$.
\end{defn}
The first Theorem we will prove shows an alternative description of
regular $G$-covers, starting from the more simple description of
$\stZ_{\omega}$ (see \ref{thm:description of ZG when omega in Cartier}).
\begin{thm}
\label{thm:gamma for regular Sthree covers}Let $Y$ be a regular,
noetherian and integral scheme such that $\dim Y\geq1$ and $6\in\odi Y^{*}$.
Then the association   \[   \begin{tikzpicture}[xscale=3.6,yscale=-0.6]     \node (A0_0) at (0, 0) {$\Cov_{3}^{\text{nd}}(Y)$};     \node (A0_1) at (1, 0) {$\stZ_{\omega}(Y)$};     \node (A1_0) at (0, 1) {$(\shF,\delta)$};     \node (A1_1) at (1, 1) {$(\odi{Y}(D(Y_\delta)),\shF,\delta,1)$};     \path (A0_0) edge [->]node [auto] {$\scriptstyle{\Gamma}$} (A0_1);     \path (A1_0) edge [|->,gray]node [auto] {$\scriptstyle{}$} (A1_1);   \end{tikzpicture}   \] 
is a fully faithful section of the projection $\RCov G(Y)\arr\Cov_{3}(Y)$
and all the regular $G$-covers are in the essential image of $\Gamma$.
Moreover, if $\chi=\Gamma(\shF,\delta)$, then $D=D_{\delta}-2D(Y_{\delta})$
is a Cartier divisor if $D_{\delta}$ is so and the associated $G$-cover
$X_{\chi}\arr Y$ is regular if and only if $Y_{\delta}$ is regular,
$D_{\delta}$ is a Cartier divisor, $D\cap D(Y_{\delta})=\emptyset$,
$Y_{\delta}\cap D$ is empty or has pure codimension $2$ and $D$
is regular outside $Y_{\delta}$. In this case $D_{\omega}=D(Y_{\delta})$,
$D=D_{m}$, $Y_{\alpha}=Y_{\delta}\cap D_{m}$ and $Y_{\delta}=D_{\omega}\sqcup Y_{\alpha}$.
\end{thm}
The following Theorem shows how the correspondence among regular $G$-covers,
regular $S_{3}$-covers and regular triple covers behaves.
\begin{thm}
\label{thm:regular G covers and triple}Let $Y$ be a regular, noetherian
and integral scheme such that $\dim Y\geq1$ and $6\in\odi Y^{*}$.
If $f\colon X\arr Y$ is a regular triple cover, then $Y_{\delta}=D(Y_{\delta})\sqcup Y'_{\delta}$,
where $Y'_{\delta}$ is a closed subscheme of pure codimension $2$
if not empty and $D(Y_{\delta})$ is regular. Moreover $f$ extends
to a regular $G$-cover ($S_{3}$-covers) if and only if $Y_{\delta}'$
is regular. More precisely the maps $\GCov,\RCov{S_{3}}\arr\Cov_{3}$
obtained by quotienting by $\sigma\in\Z/2\Z$ induce isomorphisms
  \[   \begin{tikzpicture}[xscale=7.0,yscale=-0.5]     \node (A0_0) at (0, 0) {$X$};     \node (A0_1) at (1, 0) {$X/\sigma$};     \node (A1_0) at (0, 1) {$\{\text{regular }G\text{-covers over }Y\}$};     \node (A2_1) at (1, 2) {$\left\{ \begin{array}{c} \text{regular triple covers }(\shF,\delta)\text{ over }Y\\ \text{such that }Y_{\delta}\text{ is regular} \end{array}\right\}$};     \node (A3_0) at (0, 3) {$\{\text{regular }S_3\text{-covers over }Y\}$};     \path (A0_0) edge [|->,gray]node [auto] {$\scriptstyle{}$} (A0_1);     \path (A3_0) edge [->]node [auto] {$\scriptstyle{}$} (A2_1);     \path (A1_0) edge [->]node [auto] {$\scriptstyle{}$} (A2_1);     \path (A1_0) edge [->]node[rotate=-90] [above] {$\scriptstyle{\simeq}$} (A3_0);   \end{tikzpicture}   \] Moreover
the inverse of the upper morphism is the functor $\Gamma$ introduced
in \ref{thm:gamma for regular Sthree covers}.
\end{thm}
We will also prove an existence theorem for regular $G$-covers and
$S_{3}$-covers, but we need the following definition first. 
\begin{defn}
Let $k$ be a field, $Y$ be a $k$-scheme and $\E$ be a quasi-coherent
sheaf over $Y$. The sheaf $\E$ is called \emph{strongly generated
}if, for any closed point $q\in Y$, the map
\[
\Hl^{0}(Y,\E)\arr\E\otimes(\odi{Y,p}/m_{p}^{2})
\]
is surjective. The sheaf $\E$ is called \emph{geometrically strongly
generated }if it is strongly generated over the geometric fiber $Y\times\overline{k}$.
\end{defn}
The last Theorem we want to prove is the following: 
\begin{thm}
\label{thm:construction of Sthree covers}Let $k$ be an infinite
field with $\car k\neq2,3$, $Y$ be a smooth, irreducible and proper
$k$-scheme with $\dim Y\geq1$ and $\shF$ be a locally free sheaf
of rank $2$ over $Y$. If $\E=\Homsh(\Sym^{3}\shF,\det\shF)$ is
geometrically strongly generated then there exists $\delta\in\E$
such that the triple cover associated with $(\shF,\delta)\in\shC_{3}(Y)$
extends to a $G$-cover ($S_{3}$-cover) $X_{\delta}\arr Y$ with
$X_{\delta}$ smooth and $Y_{\delta}=\emptyset$ or $\codim_{Y}Y_{\delta}=2$.
Moreover, if $Y$ is geometrically connected, then $X_{\delta}$ is
geometrically connected if and only if $\det\shF\not\simeq\odi Y$
and $\Hl^{0}(Y,\shF)=0$.
\end{thm}
The following Proposition shows that the hypothesis of the above Theorem
can be easily satisfied.
\begin{prop}
\label{prop:when F admits regular Sthree covers}Let $Y$ be a projective
scheme over a field and $\E$ be a coherent sheaf over $Y$. Then
\[
\E(-1)\text{ globally generated }\then\E\text{ geometrically strongly generated}
\]

\end{prop}
In particular if $Y$ and $\shF$ are as in Theorem \ref{thm:construction of Sthree covers},
with $Y$ projective, then $\shF(-n)$ satisfies its hypothesis for
$n\gg0$. Moreover considering $\shF=\odi Y(-1)^{2}$, so that $\E=\duale{(\Sym^{3}\shF)}\otimes\det\shF\simeq\odi Y(1)^{3}$,
we obtain:
\begin{cor}
\label{cor:Sthree regular covers over projective smooth}Let $k$
be an infinite field with $\car k\neq2,3$. Then any smooth, projective
and irreducible (resp. geometrically connected) $k$-scheme $Y$ with
$\dim Y\geq1$ has a $G$-cover ($S_{3}$-cover) $X\arr Y$ with $X$
smooth (resp. smooth and geometrically connected).
\end{cor}
The rest of this section is dedicated to the proofs of what claimed
above. Note that all the results for $S_{3}$-covers are just a consequence
of the same results for $G$-covers, thanks to \ref{rem: regular SThree if and only if regular G covers}
and  \ref{rem:invariants by sigma for Sthree and GCov}. Therefore
we will focus only on $G$-covers. We want now to argue why the concept
of divisorial component introduced above is well defined.
\begin{rem}
Let $Y$ be a locally factorial, noetherian and integral scheme and
$Z$ be a proper closed subscheme, defined by the sheaf of ideals
$\shI$. We want to show that the divisorial component of $Z$ in
$Y$ exists, or, equivalently, that the set of invertible sheaves
of ideals of $\odi Y$ containing $\shI$ has a minimum $\shL^{\shI}$.
Moreover we prove that, if $p\in Y$ and $\shI_{p}=(f_{1},\dots,f_{r})\subseteq\odi{Y,p}$,
then $\shL_{p}^{\shI}=(\gcd(f_{1},\dots,f_{r}))$. The key point for
proving this fact is that if $\shJ$, $\shN$ are sheaves of ideals
of $Y$, with $\shN$ invertible, then
\[
\shJ\subseteq\shN\iff\shJ_{q}\subseteq\shN_{q}\text{ for all }q\in Y^{(1)}
\]
In particular it is easy to check that $D(Z)$ is induced by the Weil
divisor
\[
\sum_{q\in Y^{(1)}}\dim_{k(q)}(\odi{Y,q}/\shI_{q})q
\]

\end{rem}

\begin{rem}
\label{rem: notation for Ydelta Ddelta}Let $\Phi=(\shF,\delta)\in\shC_{3}(Y)$.
Its associated algebra is $\alA_{\Phi}=\odi Y\oplus\shF$ with multiplication
$\beta=\beta_{\delta}\colon\Sym^{2}\shF\arr\shF$ and $\eta_{\delta}\colon\Sym^{2}\shF\arr\odi Y$
(see \ref{rem:triple covers and invariants by sigma}). The map 
\[
(\det\shF)^{2}\arrdi{\hat{\eta}_{\delta}}\Sym^{2}\shF\arrdi{\eta_{\delta}}\odi Y
\]
coincides with the map $\Delta_{\Phi}\colon(\det\shF)^{2}\arr\odi Y$
defining $D_{\delta}$. Indeed $\tr_{\alA_{\Phi}}(u\cdot_{\alA_{\Phi}}v)=3\eta_{\delta}(uv)$
for all $u,v\in\shF$ and, if $y,z$ is a basis of $\shF$, then
\[
\eta_{\delta}(\hat{\eta_{\delta}})=\eta_{\delta}(y^{2})\eta_{\delta}(z^{2})-\eta_{\delta}(yz)^{2}
\]
Since $Y_{\delta}$ is defined by the ideal $(\eta_{\delta}(y^{2}),\eta_{\delta}(yz),\eta_{\delta}(z^{2}))$
we see that $Y_{\delta}\subseteq D_{\delta}$ and, if $Y$ is locally
factorial, noetherian and integral and $D_{\delta}$ is a Cartier
divisor, then $D_{\delta}-2D(Y_{\delta})$ is an effective Cartier
divisor. 

Assume now we have an extension $(\shL,\shF,m,\alpha,\beta_{\delta},\la-,-\ra)\in\GCov$
of $(\shF,\delta)$, with its associated parameters. Since $\eta_{\delta}=2(-,-)_{\chi}$
(see \ref{eq:eta delta is two times symmetric product}) we have that
$\eta_{\delta}(y^{2})=-2C\omega$, $\eta_{\delta}(yz)=2A\omega$,
$\eta_{\delta}(z^{2})=2B\omega$. Since $A\omega=D\omega$, we have
that $Y_{\alpha},D_{\omega}\subseteq Y_{\delta}$, that $|Y_{\delta}|=|Y_{\alpha}|\cup|D_{\omega}|$
and, if $Y$ is regular, noetherian and integral and $D_{\omega}$
is a Cartier divisor, that $D_{\omega}\subseteq D(Y_{\delta})$. Moreover
by \ref{rem:general discriminant}, we also obtain
\[
\eta_{\delta}(\hat{\eta_{\delta}})=-4\omega^{2}m
\]
Therefore $D_{m},D_{\omega}\subseteq D_{\delta}$, $|D_{\delta}|=|D_{\omega}|\cup|D_{m}|$,
$D_{\delta}$ is a Cartier divisor if and only if $D_{m}$ and $D_{\omega}$
are so and in this case $D_{\delta}=D_{m}+2D_{\omega}$.%
\end{rem}
\begin{proof}
(\emph{of Theorem }\ref{thm:gamma for regular Sthree covers}) By
definition of $\Cov_{3}^{\text{nd}}(Y)$ and thanks to \ref{cor:extensions of Cthree objects with Domega cartier}
we see that $\Gamma$ is well defined and fully faithful. By \ref{thm:regular Sthree covers first}
we also see that all the regular $G$-covers of $Y$ belong to $\stZ_{\omega}(Y)$.
Now let $\chi=(\shM,\shF,\delta,\omega)\in\stZ_{\omega}(Y)$. Notice
that, if $X_{\chi}$ is regular, then $D_{\delta}$ is Cartier and
$Y_{\delta}\subseteq D_{\delta}$ by \ref{rem: notation for Ydelta Ddelta}.
In particular $(\shF,\delta)\in\Cov_{3}^{\text{nd}}(Y)$. Therefore
assume that this condition holds and notice that $\Gamma(\shF,\delta)\simeq\chi$
is equivalent to $D_{\omega}=D(Y_{\delta})$. Assume that $X_{\chi}$
is regular or that: $\chi=\Gamma(\shF,\delta)$ and the conditions
in the statement are satisfied. By \ref{thm:regular Sthree covers first}
and \ref{rem: notation for Ydelta Ddelta} we can conclude noting
that: $D_{m},D_{\omega},D_{\delta}$ are Cartier divisors and $D=D_{m}=D_{\delta}-2D_{\omega}$;
$D_{m}\cap D_{\omega}=\emptyset$; $Y_{\delta}=D_{\omega}\sqcup Y_{\alpha}$
and $Y_{\delta}\cap D_{m}=Y_{\alpha}$ is regular of pure codimension
$2$ if not empty, which also implies that $D_{\omega}=D(Y_{\delta})$.\end{proof}
\begin{rem}
\label{rem:local regular conditions triple covers}Let $R$ be a regular
ring with $\dim R\geq1$, $r(t)=t^{3}+gt+h\in R[t]$ and set $\Delta_{r}=4g^{3}+27h^{2}$.
Adapting \cite[Lemma 5.1]{Miranda1985} to our situation we have that
$R[t]/(r(t))$ is regular if and only if either:
\[
g\in m_{R},h\notin m_{R}^{2}\text{ or }g\notin m_{R},\Delta_{r}\notin m_{R}^{2}
\]
\end{rem}
\begin{proof}
(\emph{of Theorem }\ref{thm:regular G covers and triple}) We can
reduce the problem to the case where $Y=\Spec R$, where $R$ is a
local, regular ring. Let $\Phi=(\shF,\delta)\in\shC_{3}(R)$ be a
triple cover, $\chi=(\shM,\shF,\delta,\omega)\in\stZ_{\omega}(R)$
an extension of $\Phi$. Notice that any object of $\shC_{3}$ has
an extension to $\stZ_{\omega}(R)$. If $\Phi$ corresponds to a regular
triple cover then it has Gorenstein fibers and, by \cite[Theorem 3.1]{Bolognesi2009},
$\delta$ is never zero. On the other hand, if $X_{\chi}$ is regular,
then $\chi\in\stU_{\beta}$, i.e. $\delta$ is never zero, by \ref{thm:regular Sthree covers first}.
Therefore we can assume that $\delta$ is never zero and choose a
basis $y,z=\beta(y^{2})$ of $\shF$. In particular, from \ref{lem:associated parameters for beta nowhere zero}
and \ref{rem: notation for Ydelta Ddelta}, we see that 
\[
\beta(yz)=-\omega Cy\comma\beta(z^{2})=2\omega Ay+\omega Cz\comma\eta_{\delta}(y^{2})=-2\omega C\comma\eta_{\delta}(yz)=2\omega A\comma\eta_{\delta}(z^{2})=2\omega^{2}C^{2}
\]
and $m=A^{2}+\omega C^{3}$. If $\alA_{\Phi}=R\oplus\shF$ is the
algebra associated to $\Phi$ and we set 
\[
r(t)=t^{3}+3\omega Ct-2\omega A
\]
a direct computation shows that $r(y)=0$ and therefore that $\alA_{\Phi}\simeq R[t]/(r(t))$.
Moreover the discriminant is 
\[
\Delta_{r}=4\cdot27\omega^{3}C^{3}+4\cdot27\omega^{2}A^{2}=4\cdot27\omega^{2}m
\]
 and defines the locus $D_{\delta}$ by \ref{lem:associated parameters for beta nowhere zero}.
Notice that if $\eta_{\delta}=0$, then $r(t)=t^{3}$ and $\alA_{\phi}$
is not regular, while if $X_{\chi}$ is regular then $\eta_{\delta}\neq0$
thanks to \ref{thm:gamma for regular Sthree covers}. We can therefore
assume that $\Phi\in\Cov_{3}^{\text{nd}}(R)$. We split the proof
in two cases and in both we will use \ref{rem:local regular conditions triple covers}
and \ref{rem: notation for Ydelta Ddelta}. In particular set $g=3\omega C$
and $h=-2\omega A$.

$\alA_{\Phi}$ \emph{regular and} $\chi=\Gamma(\shF,\delta)$. We
will use that $\Delta_{r}\in m_{R}^{2}$ implies that $g\in m_{R}$
and $h\notin m_{R}^{2}$. By definition we have that $D_{\omega}=D(Y_{\delta})$,
i.e. that $\gcd(A,C)=1$. Notice that we cannot have $m,\omega\in m_{R}$
or $\omega\in m_{R}^{2}$, since otherwise $\Delta_{r},h\in m_{R}^{2}$.
Therefore $D_{\omega}$ is regular and $D_{m}\cap D_{\omega}=\emptyset$.
In particular, since $|Y_{\delta}|=|D_{\omega}|\cup|Y_{\alpha}|$
and $Y_{\alpha}\subseteq D_{m}$, we have $Y_{\delta}=D_{\omega}\sqcup Y_{\alpha}$.
If $m\in m_{R}^{2}$, then $\Delta_{r}\in m_{R}^{2}$. In particular
$g\in m_{R}$ and therefore $h\in m_{R}$, which shows that $D_{m}$
is regular outside $Y_{\delta}$. We have to show that $Y_{\alpha}$
has pure codimension $2$ if not empty. So assume that $A,C\in m_{R}$.
In particular $g,h\in m_{R}$, so $\Delta_{r}\in m_{R}^{2}$ and therefore
$h\notin m_{R}^{2}$. Moreover $m\in m_{R}$ and, as we have seen,
$\omega\notin m_{R}$. We can conclude that $A\in m_{R}-m_{R}^{2}$.
In particular $R/(A)$ is a regular domain of dimension $\dim R-1$.
If $\alt(A,C)\neq2$, we must have that $C\in(A)$, that cannot happen
since $\gcd(A,C)=1$. In conclusion if $\alA_{\Phi}$ is regular then
all the conditions required by the regularity of $X_{\chi}$ (see
\ref{thm:gamma for regular Sthree covers}) are satisfied, except
for the regularity of $Y_{\alpha}$.

$X_{\chi}$ \emph{regular}. We will make use of \ref{thm:gamma for regular Sthree covers}.
In particular from it we see that $\Gamma(\shF,\delta)=\chi$ and
we have to prove that $\alA_{\Phi}$ is regular. Assume by contradiction
that this is false. Therefore we must have $g\in m_{R},h\in m_{R}^{2}$
or $g\notin m_{R},\Delta_{r}\in m_{R}^{2}$. If the last condition
is satisfied, then $\omega,C\notin m_{R}$ and therefore $m\in m_{R}^{2}$.
Since $(A,C)=R$, in this case $D_{m}$ is not regular. So assume
$g\in m_{R},h\in m_{R}^{2}$. Note that in particular $\Delta_{R}\in m_{R}^{3}$.
In particular, if $\omega\in m_{R}$, then $m\notin m_{R}$ and therefore
$\omega\in m_{R}^{2}$, while $D_{\omega}$ is regular. So $\omega\notin m_{R}$
and therefore $C\in m_{R}$ and $A\in m_{R}^{2}$. But this cannot
happen because $A,C$ are independent in $m_{R}/m_{R}^{2}$.\end{proof}
\begin{lem}
\label{lem:key lemma for constructing Sthree covers}Let $k$ be an
algebraically closed field, let $(R,m_{R},k)$ be a regular local
ring with $\dim R\geq1$ and denote by $\overline{(-)}\colon R\arr\overline{R}=R/m_{R}^{2}$
the projection. Set $\shF=R^{2}$, $\overline{\shF}=\overline{R}^{2}$
and given $\delta\colon\Sym^{3}\shF\arr\det\shF$ denote by $\alA_{\delta}$
the $S_{3}$-cover over $R$ obtained as in \ref{thm:gamma for regular Sthree covers}.
Finally set 
\[
\overline{V}=\Hom_{\overline{R}}(\Sym^{3}\overline{\shF},\det\overline{\shF})\comma\overline{W}=\{\gamma\in\overline{V}\st\exists\Sym^{3}\shF\arrdi{\delta}\det\shF\text{ s.t. }\overline{\delta}=\gamma\text{ and }\alA_{\delta}\text{ not regular}\}
\]
Then $\overline{W}$ is a closed subscheme of the $k$-vector space
$\overline{V}$ with $\codim_{\overline{V}}\overline{W}\geq\dim R+1$.\end{lem}
\begin{proof}
Since $\delta$ varies in the arguments below, if $\delta=(-a,b,c,e)$
with respect to some basis of $\shF$, all the parameters associated
with $\delta$ (and its associated $S_{3}$-cover) will be thought
of as (polynomial) functions in $a,b,c,e$. Moreover we fix a $k$-basis
$1,\overline{x}_{1},\dots,\overline{x}_{s}$ of $\overline{R}$, where
$x_{1},\dots,x_{r}\in m_{R}$ form a basis of $m_{R}/m_{R}^{2}$.
Moreover $r\in R$ will be denoted by
\[
r=r_{0}\cdot1+r_{1}\overline{x}_{1}+\cdots r_{s}\overline{x}_{s}\text{ where }s=\dim R\comma r_{i}\in k
\]
Notice that, if $\delta,\delta'\in\Hom(\Sym^{3}\shF,\det\shF)$ are
such that $\overline{\delta}$ and $\overline{\delta}'$ differs by
an automorphism of $\overline{\shF}$, i.e. there exists $\phi\in\Gl(\overline{\shF})$
such that $(\det\phi)\delta'=\delta\circ\Sym^{3}\phi$, then $\alA_{\delta}$
is regular if and only if $\alA_{\delta'}$ is so. Indeed the previous
condition means that $\alA_{\delta}\otimes\overline{R}\simeq\alA_{\delta'}\otimes\overline{R}$.
Given $y\in\overline{\shF}$ set
\[
\overline{V}_{y}=\{\gamma\in\overline{V}\st y,\beta_{\gamma}(y^{2})\text{ is a basis of }\overline{\shF}\}
\]
Notice that the $\overline{V}_{y}$ are open subsets of $\overline{V}$,
not empty if $y_{0}\neq0$, and that, thanks to \ref{lem:local basis for beta nowhere zero},
their union covers $\overline{V}-\{\gamma\in\overline{V}\st\beta_{\gamma}\otimes k=0\}$.
Since $\{\gamma\in\overline{V}\st\beta_{\gamma}\otimes k=0\}\subseteq\overline{W}$
has codimension $4$ in $\overline{V}$, it is enough to prove that
\[
\codim_{\overline{V}_{y}}W\cap\overline{V}_{y}\geq s+1
\]
for all $y\in\overline{\shF}$ such that $y_{0}\neq0$. Consider now
the map $p\colon\overline{V}_{y}\arr\overline{R}^{2}$ that sends
$\gamma$ to the parameters $c,e$ associated with $\gamma$ with
respect the basis $y,\beta_{\gamma}(y^{2})$. Note that a $\gamma\in\overline{V}_{y}$
differs by an automorphism of $\overline{\shF}$ from $(-1,0,p(\gamma))$.
In particular if we set
\[
\overline{W}'=\{(\overline{c},\overline{e})\in\overline{R}^{2}\st(-1,0,\overline{c},\overline{e})\in\overline{W}\}
\]
then $p^{-1}(\overline{W}')=\overline{W}\cap\overline{V}_{y}$. Now
denote by $U$ the open subspace of $\overline{\shF}$ of $z$ such
that $y,z$ is a basis of $\overline{\shF}$. The map 
\[
V_{y}\arr U\times\overline{R}^{2}\text{ given by }\gamma\longmapsto(\beta_{\gamma}(y^{2}),p(\gamma))
\]
is an isomorphism whose inverse sends $z\in U,u,v\in\overline{R}$
to the $\gamma\in\overline{V}_{y}$ given by $(-1,0,u,v)$ with respect
to the basis $y,z$ of $\shF$. In particular we can reduce again
the problem and prove that $\codim_{\overline{R}^{2}}\overline{W}'\geq s+1$.
We want to show that $\overline{W}'=W_{1}\cup W_{2}$, where
\[
W_{1}=\{e^{2}=4c^{3}\text{ and }c_{0},e_{0}\neq0\}\comma W_{2}=\{c_{0}=e_{0}=0\text{ and }c,e\text{ dependent in }m_{\overline{R}}/m_{\overline{R}^{2}}\}
\]
Remember that the parameter $m$ associated with $(-1,0,c,e)$ is
given by $(e^{2}-4c^{3})/4$ and, taking into account Theorem \ref{thm:regular Sthree covers first},
the only non trivial point to show in the equality above is the following:
if $c,e\in R$ are such that $e^{2}=4c^{3}$ then there exists $c',e'\in R$
such that $\overline{e}'=\overline{e}$, $\overline{c}'=\overline{c}$
and $e'^{2}\neq4c'^{3}$. Assume by contradiction that this is not
possible. Notice that $e^{2}=4c^{3}$ implies that $c,e$ are invertible
or that $e=c=0$. Indeed if $c,e\neq0$ and $c\in m_{R}^{t}-m_{R}^{t+1}$,
$e\in m_{R}^{l}-m_{R}^{l+1}$ then $2t=3l$, which implies $t=l=0$.
Consider $e'=e+w$, with $w\in m_{R}^{2}-m_{R}^{3}$, $c'=c$. If
$e=0$ we are fine. Otherwise, modulo $m_{R}^{3}$, we get the equality
$2ew=0$ and also in this case we get a contradiction.

If we write $c=c_{0}+c'$, $e=e_{0}+e'$ then $W_{1}$ is contained
in the irreducible component of the locus $\{e_{0}^{2}=4c_{0}^{3}\comma e_{0}e'=6c_{0}c'\}$
which is not $\{c_{0}=e_{0}=0\}$, that has codimension $s+1$. Finally
it is easy to check that also $W_{2}$ has codimension $\geq s+1$.\end{proof}
\begin{rem}
\label{rem:connectdness for regular Sthree covers}Let $Y$ be a proper,
smooth and geometrically connected scheme with $\dim Y\geq1$ over
a field $k$ with $\car k\neq2,3$ and $\chi=(\shL,\shF,m,\alpha,\beta,\la-,-\ra)\in\GCov(Y)$
such that $X_{\delta}$ is regular. Then $X_{\delta}$ is geometrically
connected if and only if $\shL\not\simeq\odi Y$ and $\Hl^{0}(Y,\shF)=0$.
Indeed $X_{\delta}$ is geometrically connected if and only if $\Hl^{0}(X_{\delta}\times\overline{k},\odi{X_{\delta}\times\overline{k}})=\overline{k}$,
which means $\Hl^{0}(Y\times\overline{k},\shL\otimes\overline{k})=\Hl^{0}(Y\times\overline{k},\shF\otimes\overline{k})=0$.
On the other hand, this is also equivalent to $\Hl^{0}(Y,\shL)=\Hl^{0}(Y,\shF)=0$
and, since $m$ induces an injective map $\shL\arr\shL^{-1}$, we
also have that $\Hl^{0}(Y,\shL)\neq0$ is equivalent to $\shL\simeq\odi Y$.\end{rem}
\begin{proof}
(of Proposition \ref{prop:when F admits regular Sthree covers}) We
can assume that $k$ is algebraically closed. Note that the property
of being strongly generated is stable for direct sums and quotients.
Moreover, if $i\colon Y\arr\PP_{k}^{m}$ is the closed immersion defined
by $\odi Y(1)$, then $\E$ is strongly generated if and only if $i_{*}\E$
is so. Since we have surjective maps $\odi Y(1)^{N}\arr\E$ by hypothesis
and $\odi{\PP_{k}^{m}}(1)\arr i_{*}\odi Y(1)$, we only need to prove
that $\odi{\PP_{k}^{m}}(1)$ is strongly generated. This last condition
means that, if $p$ is a maximal ideal of $R=k[x_{1},\dots,x_{m}]$,
then $R/p^{2}$ is generated as $k$-vector space by the elements
$1,x_{1},\dots,x_{m}$. This property clearly holds because $p=(x_{1}-a_{1},\dots,x_{m}-a_{m})$
for some $a_{i}\in k$.
\end{proof}

\begin{proof}
(of Theorem \ref{thm:construction of Sthree covers}) Denote by $V$
the vector bundle over $k$ associated with $\Hl^{0}(Y,\E)$ and by
$g\colon Y\times V\arr Y$ and $\pi\colon Y\times V\arr V$ the projections.
By definition of $V$, there exists $\chi=(g^{*}\shF,\mu)\in\shC_{3}(Y\times V)$
such that, for any $\Spec k\arrdi{\delta}V$, we have $(\id_{Y}\times\delta)^{*}(g^{*}\shF,\mu)=(\shF,\delta)$.
Consider the $G$-cover $f_{\chi}\colon X_{\chi}\arr Y\times V$ associated
with $\chi$ as in \ref{thm:The-locus when omega is invertible} and
let $U\subseteq V$ the smooth locus of the flat map $\pi\circ f_{\chi}\colon X_{\chi}\arr V$.
We claim that we have to prove that $U\neq\emptyset$, so that we
will assume $k$ algebraically closed. Indeed, since $k$ is infinite,
there will exists $\delta\in U(k)$. If $f_{\delta}\colon X_{\delta}\arr Y$
is the base change $f_{\chi,\delta}\colon X_{\chi,\delta}\arr Y\times\{\delta\}$
then, by construction, $X_{\delta}$ is smooth and, taking into account
\ref{thm:regular G covers and triple}, $\codim_{Y}Y_{\delta}=2$
if $Y_{\delta}\neq\emptyset$. Moreover $\det\shF\simeq\shL$ and
the last claim about connectedness follows by \ref{rem:connectdness for regular Sthree covers}. 

Given $\delta\in V$ we will denote by $f_{\delta}\colon X_{\delta}\arr Y$
the base change of $f_{\chi}$ over $Y\times\{\delta\}$. Since $\pi\circ f_{\chi}$
is flat, given $p\in X(k)$ and $\delta=\pi(f_{\chi}(p))$ we have
\[
X\text{ regular in }p\iff X_{\delta}\text{ regular in }p
\]
In particular, if $Z_{X}\subseteq X$ is the singular locus of $X$
and $Z=f_{\chi}(Z_{X})$, then $(q,\delta)\in Z(k)$ if and only if
$X_{\delta}$ has a singular point over $q\in Y$. Moreover $\pi(Z)$
is the complementary of $U$ in $V$. Therefore it is enough to prove
that $\dim Z\leq\dim V-1$,

If $q\in Y(k)$, then $g^{-1}(q)\cap Z\subseteq V$ is the locus of
$\delta\in V(k)$ such that $X_{\delta}$ is not regular over $q$.
In particular, if we denote by $\phi_{q}$ the map
\[
\phi_{q}\colon V\arr\E\otimes(\odi{Y,q}/m_{q}^{2})
\]
and by $\overline{W}_{q}$ the subspace of $\E\otimes(\odi{Y,q}/m_{q}^{2})$
of elements $\gamma$ such that there exists $\delta\in\phi_{q}^{-1}(\gamma)$
for which $X_{\delta}$ is not regular over $q$, then $\phi_{q}^{-1}(\overline{W}_{q})=g^{-1}(q)\cap Z$.
Indeed if $\delta,\delta'\in V$ are such that $\phi_{q}(\delta)=\phi_{q}(\delta')$,
then $X_{\delta}$ is regular over $q$ if and only if $X_{\delta'}$
is so. Moreover $\overline{W}_{q}$ is contained in the locus defined
in \ref{lem:key lemma for constructing Sthree covers}, where $R=\odi{Y,q}$,
and therefore $\codim_{\E\otimes(\odi{Y,q}/m_{q}^{2})}\overline{W}_{q}\geq\dim Y+1$.
Since $\phi_{q}$ is, by hypothesis, a surjective linear map of vector
spaces, we also have $\codim_{V}(g^{-1}(q)\cap Z)\geq\dim Y+1$. Therefore
$\dim(g^{-1}(q)\cap Z)\leq\dim V-\dim Y-1$ and
\[
\dim Z\leq\dim Y+\dim V-\dim Y-1=\dim V-1
\]

\end{proof}

\subsection{Invariants of regular $S_{3}$-covers of surfaces.}

The aim of this subsection is to compute the invariants of regular
$S_{3}$-covers of surfaces over an algebraically closed field. Here
and in the rest of the section by a surface over a field, we mean
a projective, smooth and integral scheme of dimension $2$. The result
is:
\begin{thm}
\label{thm:invariants of regular Sthree covers}Let $Y$ be a surface
over an algebraically closed field $k$ such that $\car k\neq2,3$
and $f\colon X\arr Y$ be a regular $S_{3}$-cover associated with
$(\shF,\delta)\in\shC_{3}$ as in \ref{thm:gamma for regular Sthree covers}.
The closed subscheme $Y_{\delta}$ of $Y$ defined by the map $\eta_{\delta}\colon\Sym^{2}\shF\arr\odi Y$
is the disjoint union of a divisor $D$ and a finite set $Y_{0}$
of rational points and $X$ is connected, that is a surface, if and
only if $\Hl^{0}(\shF)=0$ and $\odi Y(-D)\not\simeq\det\shF$. In
this case the invariants of $X$ are given by
\begin{alignat*}{1}
K_{X}^{2}= & \;6K_{Y}^{2}+6c_{1}(\shF)^{2}-12c_{1}(\shF)K_{Y}-\frac{10}{3}D^{2}-4DK_{Y}\\
p_{g}(X)= & \; p_{g}(Y)+2h^{2}(\shF)+h^{2}(\odi Y(D)\otimes\det\shF)\\
\chi(\odi X)= & \;6\chi(\odi Y)-2c_{2}(\shF)+\frac{1}{2}(3c_{1}(\shF)^{2}-3c_{1}(\shF)K_{Y}-DK_{Y}-D^{2})\\
|Y_{0}|= & \;3c_{2}(\shF)-\frac{2}{3}D^{2}
\end{alignat*}

\end{thm}
Before proving this Theorem we need several lemmas.
\begin{lem}
\label{lem:When omega is zero F split}Let $S$ be a finite disjoint
union of integral schemes, $(\shL,\shF,m,\alpha,\beta,\la-,-\ra)\in\stZ_{G}(S)$
and assume that $\beta$ is never zero, that $\la-,-\ra=0$ and that
$m$ is an isomorphism. Then there exist an isomorphism $\iota\colon\shL\arr\odi S$
whose square is $m$ and a decomposition $\shF=\shH_{1}\oplus\shH_{2}$
into invertible sheaves such that $\beta_{|\shH_{1}^{2}}$ is an isomorphism
$\shH_{1}^{2}\simeq\shH_{2}$ and $\beta_{|\shH_{2}^{2}},\beta_{|\shH_{1}\otimes\shH_{2}}=0$.
In particular
\[
\Coker\beta\simeq\shH_{1}\text{ and }\det\shF\simeq\shH_{1}^{3}
\]
\end{lem}
\begin{proof}
What we will really prove is that $\Imm\beta$ is an invertible sheaf,
that the map $\shL\arr\Endsh(\shF)$ induced by $\alpha$ yields an
isomorphism $\iota\colon\shL\arr\Endsh(\Imm\beta)\simeq\odi S$, that
$\iota^{2}=m$, that $\alpha'=\alpha\circ(\iota^{-1}\otimes\id_{\shF})\colon\shF\arr\shF$
is an isomorphism such that $\alpha'^{2}=\id$ and that the decomposition
into eigenspaces of $\alpha'$ is $\shF=\shH_{1}\oplus\shH_{2}$,
where $\shH_{i}$ are invertible sheaves and $\shH_{2}=\Imm\beta$.
Those conditions and the requests of the statement can be checked
locally.

So assume that $S$ is integral, that $\shL=\odi S$ and that we have
a basis $y,z=\beta(y^{2})$ of $\shF$. By \ref{lem:associated parameters for beta nowhere zero},
we have $B=\omega C^{2}=0$, $D=-A$ and therefore $A^{2}=m$. Set
$\iota=-A$, so that $\alpha'=\alpha/\iota$. Replacing $y$ by $y+(C/2A)z$,
we get new parameters, although $A$ remains the same, such that $B=C=0$,
i.e. such that $\alpha$ is diagonal. Therefore $\alpha'^{2}=\id$
and $\shH_{1}=\la y\ra,\shH_{2}=\la z\ra$ are the eigenspaces of
$\alpha'$ with respect to the eigenvalues $-1,1$ respectively. From
relations (\ref{eq:loc com and ass conditions}) we see that $a=c=be=0$.
Since $z\in\Imm\beta$, we get $e=0$, $b\in\odi S^{*}$. In particular
$\shH_{2}=\Imm\beta$, which is invertible, $\beta_{|\shH_{1}\otimes\shH_{2}\oplus\shH_{2}^{2}}=0$
and $\beta_{|\shH_{1}^{2}}$ is an isomorphism $\shH_{1}^{2}\simeq\Imm\beta=\shH_{2}$.
\end{proof}
We fix an integral, regular and noetherian scheme $Y$ with $\dim Y\geq1$
and an object $\chi\in\stZ_{\omega}(Y)$, $\chi=(\shM,\shF,\delta,\omega)$,
whose induced $G$-cover, denoted by $f_{\chi}\colon X_{\chi}\arr Y$,
is regular. We also denote by $\alA=\alA_{\chi}$ the algebra associated
with $\chi$, i.e. $\alA=f_{\chi*}\odi{X_{\chi}}$, and by 
\[
\shL,m,\zeta,\alpha,\beta,\la-,-\ra,(-,-)\qquad D_{m},D_{\omega},D_{\delta},Y_{\alpha},Y_{\delta}
\]
the objects associated with $\chi$ according to the inclusion $\stZ_{\omega}(Y)\arr\GCov(Y)$
(see \ref{thm:description of ZG when omega in Cartier}) and$ $ the
closed subschemes of $Y$ defined in \ref{def:closed subscheme associated to Gcovers first}
and \ref{def:closed subscheme associated to Gcovers second} respectively.
We will often make use of Theorems \ref{thm:regular Sthree covers first},
\ref{thm:gamma for regular Sthree covers} and \ref{thm:regular G covers and triple},
which yield several conditions on the closed subschemes introduced
above. In particular notice that $\beta$ is never zero. Therefore
we will often consider basis of $\shF$ of the form $y,\beta(y^{2})$,
thanks to \ref{lem:local basis for beta nowhere zero}, and the correspondent
parameters associated with $\chi$, given in \ref{lem:associated parameters for beta nowhere zero}.

We are going to describe two exact sequences over $Y$, as the first
step in the computation of the invariants of $X_{\chi}$.
\begin{rem}
If $i\colon Z\arr Y$ is a closed immersion of schemes defined by
the sheaf of ideals $\shI$ and $\shQ$ is a coherent sheaf on $Y$
such that $\shQ_{p}\simeq\odi{Y,p}/\shI_{p}$ for any $p\in Y$, then
$i^{*}\shQ$ is an invertible sheaf on $Z$ and $\shQ\simeq i_{*}i^{*}\shQ$.\end{rem}
\begin{lem}
We have an exact sequence
\begin{equation}
0\arr(\det\shF)^{2}\otimes\shM\arrdi{\zeta}\Sym^{2}\shF\arrdi{\beta}\shF\arr i_{*}\shH\oplus j_{*}\shQ\arr0\label{eq:first exact sequence for Sthree}
\end{equation}
where $i\colon D_{\omega}\arr Y$, $j\colon Y_{\alpha}\arr Y$ are
the immersions, $\shH,\shQ$ are invertible sheaves on $D_{\omega}$
and $Y_{\alpha}$ respectively and $\shH^{3}\simeq i^{*}\det\shF$.\end{lem}
\begin{proof}
We will prove that $\Coker\beta$ is schematically supported over
$Y_{\delta}=D_{\omega}\sqcup Y_{\alpha}$ and therefore of the form
$i_{*}\shH\oplus j_{*}\shQ$ where $i$ and $j$ are as in the statement
and $\shH,\shQ$ are coherent sheaves on $D_{\omega}$ and $Y_{\alpha}$
respectively. By \ref{lem:When omega is zero F split} we can then
conclude that $\shH^{3}\simeq i^{*}\det\shF$, since $D_{\omega}$
is regular. We can therefore work locally, i.e. assuming that $Y=\Spec R$,
where $R$ is a local regular ring, and that we have a basis $y,z=\beta(y^{2})$
of $\shF$. In particular
\[
\shF/\Imm\beta\simeq(Ry\oplus Rz)/(z,cy,ey-cz)\simeq R/(c,e)
\]
So $\shF/\Imm\beta$ is an invertible sheaf on $Y_{\delta}=D_{\omega}\sqcup Y_{\alpha}$
and therefore can be written as $\shF/\Imm\beta\simeq i_{*}\shH\oplus j_{*}\shQ$,
where $\shH,\shQ$ are invertible sheaves as in the statement.

By \ref{lem:for the description of Z Sthree}, we know that $\beta(\zeta)=0$.
By \ref{iso:primo alpha}, we can write $\zeta=By^{2}-2Ayz-Cz^{2}$.
If $\zeta=0$, then $m=A^{2}+\omega C^{3}=0$, which is not the case.
It remains to check that $\Ker\beta$ is generated by $\zeta$. Note
that $\beta$ is generically surjective, since $\chi$ is generically
a $G$-torsor and thanks to \ref{thm:description of Sthree torsors}.
Therefore we have that $\Ker\beta=(k(R)\zeta)\cap\Sym^{2}\shF\subseteq\Sym^{2}\shF\otimes k(R)$
and we need to prove that if $u\zeta\in R^{3}$, with $u\in k(R)$,
then $u\in R$. If $A$ or $C$ is invertible this is clear. If $A,C\in m_{R}$,
then they are independent in $m_{R}/m_{R}^{2}$ and therefore different
primes. Write $u=v/w$ with $v,w$ coprimes. By hypothesis, we have
relations $vA=wr,vC=wr'$, with $r,r'\in R$. If $w$ is not invertible,
any prime dividing $w$ will also divide $A$ and $C$, which is a
contradiction.
\end{proof}
In order to introduce the second exact sequence, we introduce the
following notation. Set $\alB=\alA^{\sigma}\simeq\odi Y\oplus\shF$
and $\pi'\colon X'=\Spec\alB\arr Y$. The map $X_{\chi}\arr X'$ is
a degree $2$ cover and we denote by $\Delta$ the invertible sheaf
over $\alB$ inducing it.
\begin{rem}
\label{rem:delta otimes delta}Notice that $\Delta\subseteq\alA$
is the eigenspace of $\sigma\in\Z/2\Z$ relative to $-1$. In particular
\[
\shL\oplus\shF\simeq\Delta=\{0\oplus s\oplus x\oplus(-x)\st s\in\shL\comma x\in\shF\}\subseteq\alA=\odi Y\oplus\shL\oplus\shF_{1}\oplus\shF_{2}
\]
Similarly to how we have identified $\alB=\alA^{\sigma}$ with $\odi Y\oplus\shF$
(see \ref{rem:triple covers and invariants by sigma}), we will always
identify $\shL\oplus\shF$ with $\Delta$ through the map $s\oplus x\longmapsto0\oplus s\oplus x\oplus(-x)$.
Note that $\alA_{\chi}=\alB\oplus\Delta$. In particular we get a
multiplication map $\Delta\otimes_{\alB}\Delta\arr\alB$ and it is
easy to check that
\[
m\oplus\alpha\oplus(\beta-\eta_{\delta})\colon\Sym_{\odi Y}^{2}\Delta=\shL^{2}\oplus\shL\otimes\shF\oplus\Sym^{2}\shF\arr\Delta\otimes_{\alB}\Delta\arr\alB
\]
Note that, in the above composition, the first map is surjective,
while the second is injective because $f_{\chi}\colon X_{\chi}\arr Y$
is generically a $G$-torsor. Therefore we will identify the sheaf
$\Delta\otimes_{\alB}\Delta$ with the image of $m\oplus\alpha\oplus(\beta-\eta_{\delta})$
in $\alB=\odi Y\oplus\shF$.\end{rem}
\begin{lem}
We have exact sequences
\begin{equation}
0\arr\shQ\arr\Delta\otimes_{\alB}\Delta\arr i_{*}\odi{D_{\omega}}\arr0\comma0\arr\Sym^{2}\shF\arr\shQ\arr i_{*}\shK\arr0\label{eq:second exact sequence for Sthree}
\end{equation}
where $i\colon D_{\omega}\arr Y$ is the inclusion and $\shK$ is
an invertible sheaf on $D_{\omega}$ such that $\shK^{3}\simeq i^{*}\det\shF$.\end{lem}
\begin{proof}
As explained in \ref{rem:delta otimes delta}, $\Delta\otimes_{\alB}\Delta$
can be identified to the image $\shG$ of $m\oplus\alpha\oplus(\beta-\eta_{\delta})$
in $\alB=\odi Y\oplus\shF$. Set 
\[
\shQ=\Imm((\beta-\eta_{\delta})\oplus\alpha\colon\Sym^{2}\shF\oplus\shL\otimes\shF\arr\alB)\subseteq\shG\comma\shH=\Coker((\beta-\eta_{\delta})\colon\Sym^{2}\shF\arr\shQ)
\]
and $\shT=\Coker(\shQ\arr\shG)$. We have exact sequences 
\[
0\arr\shQ\arr\Delta\otimes_{\alB}\Delta\arr\shT\arr0\comma\Sym^{2}\shF\arrdi{\beta-\eta_{\delta}}\shQ\arr\shH\arr0
\]
We first prove that $\shT\simeq i_{*}\odi{D_{\omega}}$. By definition
we have a surjective map $\shL^{2}\arrdi m\shG\arr\shT$ whose kernel
is $m^{-1}(\shQ)$. Since locally $\hat{\eta_{\delta}}=\omega\zeta$
by \ref{eq:second condition on M,F,delta,zeta,omega}, $\beta(\zeta)=0$
by \ref{lem:for the description of Z Sthree} and $\eta_{\delta}(\hat{\eta_{\delta}})=-4\omega^{2}m$
by \ref{rem: notation for Ydelta Ddelta} we have that $(\beta-\eta_{\delta})(\hat{\eta_{\delta}})=4\omega^{2}m\in\shQ$,
that $\omega^{2}\in m^{-1}(\shQ)$ and therefore that $|\Supp\shT|\subseteq|D_{\omega}|$.
So $\shT$ is supported in the locus where $m$ and therefore $\alpha$
are isomorphisms, in which we have $\shG=\alB$, $\shQ=\Imm\eta_{\delta}\oplus\shF$,
which yields the desired result.

We now consider $\shH$. Assume that we have already proved that $\shH\simeq i_{*}\shK$,
where $\shK$ is an invertible sheaf over $D_{\omega}$. We want to
prove that $\shK^{3}\simeq i^{*}\det\shF$. Let
\[
\overline{\shQ}=\Imm((\Sym^{2}\shF\oplus\shL\otimes\shF)\otimes\odi{D_{\omega}}\arrdi{(\beta-\eta_{\delta})\oplus\alpha}\alB\otimes\odi{D_{\omega}})
\]
Since, on $D_{\omega}$, $\alpha$ is an isomorphism and $\eta_{\delta}=0$,
we get $\overline{\shQ}=\shF\otimes\odi{D_{\omega}}$. Moreover we
have a surjective map $\shQ\otimes\odi{D_{\omega}}\arr\overline{\shQ}$
and a commutative diagram   \[   \begin{tikzpicture}[xscale=2.6,yscale=-1.2]     \node (A0_0) at (0, 0) {$\Sym^2\shF\otimes \odi{D_\omega}$};     \node (A0_1) at (1, 0) {$\shQ\otimes \odi{D_\omega}$};     \node (A0_2) at (2, 0) {$\shH\otimes \odi{D_\omega}\simeq \shK$};     
\node (A0_3) at (2.8, 0) {$0$};    
\node (A1_1) at (1, 1) {$\shF\otimes\odi{D_\omega}$};     
\node (A1_2) at (2, 1) {$\widehat\shK$};     
\path (A0_0) edge [->]node [auto] {$\scriptstyle{}$} (A0_1);     
\path (A0_1) edge [->]node [auto] {$\scriptstyle{}$} (A1_1);     
\path (A0_1) edge [->]node [auto] {$\scriptstyle{}$} (A0_2);     
\path (A0_2) edge [->>]node [auto] {$\scriptstyle{}$} (A1_2);     \path (A1_1) edge [->]node [auto] {$\scriptstyle{}$} (A1_2);     
\path (A0_0) edge [->]node [swap,auto] {$\scriptstyle{\beta\otimes\odi{D_\omega}}$} (A1_1);     \path (A0_2) edge [->]node [auto] {$\scriptstyle{}$} (A0_3);   \end{tikzpicture}   \]  where $\widehat{\shK}=\Coker(\beta\otimes\odi{D_{\omega}})$. Thanks
to \ref{lem:When omega is zero F split}, $\widehat{\shK}$ is an
invertible sheaf on $D_{\omega}$ such that $\widehat{\shK}^{3}\simeq i^{*}\det\shF$
and the surjective map $\shK\arr\widehat{\shK}$ is an isomorphism.
In order to prove that $\shH\simeq i_{*}\shK$, we can work on a regular
local ring $R$. Considering a basis $y,z=\beta(y^{2})$ of $\shF$
and basis $1,y,z$ of $\odi Y\oplus\shF$ we have
\[
\beta-\eta_{\delta}=\left(\begin{array}{ccc}
-2c & -e & -2c^{2}\\
0 & c & e\\
1 & 0 & -c
\end{array}\right)\comma e=2\omega A\comma c=-\omega C\comma B=\omega C^{2}\comma m=A^{2}+\omega C^{3}
\]
 In particular $\det(\beta-\eta_{\delta})=4c^{3}-e^{2}=-4\omega^{2}m$
and therefore $\beta-\eta_{\delta}$ is injective. In particular if
both $\omega$ and $m$ are invertible, then $\beta-\eta_{\delta}$
is an isomorphism and therefore $\shH=0$. In particular we can assume
that $\omega$ or $m$ is not invertible. By definition we have a
surjective map $\shL\otimes\shF\arr\shQ\arr\shH$ whose kernel is
$\alpha^{-1}(\Imm(\beta-\eta_{\delta}))$. Given $x=uy+vz\in\shL\otimes\shF$
we want to check when there exists $x'=wy^{2}+gyz+hz^{2}\in\Sym^{2}\shF$
such that $(\beta-\eta_{\delta})(x')=\alpha(x)$, i.e.
\[
u(Ay+Cz)+v(By-Az)=(-2cw-eg-2c^{2}h)+(gc+he)y+(w-hc)z
\]
which translates in the system of $3$ equations
\[
-2cw-eg-2c^{2}h=0\comma uA+vB=gc+he\comma uC-vA=w-hc
\]
We first get $w=uC+hc-vA$ and our equations become 
\[
-4c^{2}h-2ucC+2vcA-eg=0\comma uA+vB=gc+he
\]
First note that if $e=0\in m_{R}$, then $A=0$ and if $C\in m_{R}$
then $R/(A,C)$ cannot have codimension $2$. So $C$ is invertible
and, since $4\omega^{2}m=-4c^{3}$, $\omega$, $c$ and $m$ differs
by an invertible element. From $D_{\omega}\cap D_{m}=\emptyset$,
we can conclude that both $\omega$ and $m$ are invertible. Therefore
we have $e\neq0$. We can write
\[
h=(uA+vB-gc)/e
\]
 and substituting in the first equation we get
\[
g(4c^{3}-e^{2})=u(4c^{2}A+2ceC)+v(4c^{2}B-2ceA)
\]
Now note that $4c^{3}-e^{2}=-4\omega^{2}m$, $4c^{2}A+2ceC=0$ and
$4c^{2}B-2ceA=4\omega^{2}Cm$ and so the above equation become $4\omega^{2}m(g+Cv)$
whose unique solution is $g=-Cv$. In particular $vB-gc=0$ and our
last equation is $h=uA/e=u/2\omega$. So $\alpha(uy+vz)$ is in the
image of $\beta-\eta_{\delta}$ if and only if $\omega\mid u$, which
implies that $\shH\simeq R/(\omega)$.
\end{proof}
From now on we assume that $Y$ is a surface over an algebraically
closed field $k$. We write $\mu=c_{1}(\shM)=D_{\omega}$, $c_{1}=c_{1}(\shF)$,
$c_{2}=c_{2}(\shF)$ and $K_{Y}=K$, the canonical divisor of $Y$.
\begin{rem}
We have $\mu c_{1}=-\mu^{2}$. Indeed from $D_{m}\cap D_{\omega}=\emptyset$,
we get $\mu c_{1}(\shL)=0$. On the other hand we have $\shL\simeq\shM\otimes\det\shF$.\end{rem}
\begin{lem}
\label{lem:chern classes invertible sheaves Domega}Let $\shH$ be
an invertible sheaf on $D_{\omega}$ such that $\shH^{3}\simeq\det\shF^{l}\otimes\odi{D_{\omega}}$.
Then
\[
\chi(\shH)=-\frac{2l+3}{6}\mu{}^{2}-\frac{\mu K}{2}
\]
\end{lem}
\begin{proof}
Let $\mu=D_{\omega}=D_{1}+\cdots+D_{s}$ be the decomposition into
smooth integral components and set $\shH_{i}=\shH\otimes\odi{D_{i}}$
and $i\colon D_{\omega}\arr Y$ for the inclusion. Thanks to \ref{lem:When omega is zero F split},
$i^{*}\shL\simeq\odi{D_{\omega}}$ and therefore $i^{*}\det\shF\simeq i^{*}\shM^{-1}$.
In particular
\[
\shH_{i}^{3}\simeq(\det\shF)^{l}\otimes\odi{D_{i}}\simeq\shM^{-l}\otimes\odi{D_{i}}\simeq(\odi Y(-D_{i})\otimes\odi{D_{i}})^{l}\then\deg_{D_{i}}\shH_{i}=-lD_{i}^{2}/3
\]
By adjunction formula we also have $2(g(D_{i})-1)=D_{i}^{2}+D_{i}.K$
and moreover, by Riemann-Roch,
\[
\chi_{D_{i}}(\shH_{i})=-lD_{i}^{2}/3-D_{i}^{2}/2-D_{i}.K/2=-\frac{2l+3}{6}D_{i}^{2}-\frac{D_{i}.K}{2}
\]
Summing over all $i$ and taking into account the relation $\mu^{2}=D_{\omega}^{2}=\sum_{i}D_{i}^{2}$
we get the result.\end{proof}
\begin{rem}
Let $\E$ be a locally free sheaf of rank $r$ and $\shL$ be an invertible
sheaf over $Y$. We recall the following well known formulas for the
Chern classes. In particular the last one is Riemann-Roch for surfaces.
\[
c_{1}(\shE\otimes\shL)=rc_{1}(\shL)+c_{1}(\shE),c_{2}(\shE\otimes\shL)=\frac{r(r-1)}{2}c_{1}(\shL)^{2}+(r-1)c_{1}(\shL)c_{1}(\shE)+c_{2}(\shE)
\]
\[
c_{1}(\Sym^{2}\shE)=3c_{1}(\shE),c_{2}(\Sym^{2}\shE)=2c_{1}(\shE)^{2}+4c_{2}(\shE)\text{ if }r=2
\]
\[
\chi(\shE)=\frac{c_{1}(\shE)^{2}-2c_{2}(\shE)-c_{1}(\shE)K}{2}+r\chi(\odi Y)
\]
\end{rem}
\begin{lem}
\label{lem:cardinality of Yalpha}We have
\[
|Y_{\alpha}|=3c_{2}-\frac{2}{3}\mu^{2}
\]
\end{lem}
\begin{proof}
Consider the operator
\[
\Phi=2(\chi-\chi(\odi Y)\rk)
\]
on coherent sheaves, which is additive on exact sequences. We will
apply it on the exact sequence (\ref{eq:first exact sequence for Sthree}).
We have $c_{1}(\det\shF^{2}\otimes\shM)=2c_{1}+\mu$ and
\[
\Phi(\det\shF^{2}\otimes\shM)=(2c_{1}+\mu)^{2}-(2c_{1}+\mu)K=4c_{1}^{2}+\mu^{2}+4\mu c_{1}-2c_{1}K-\mu K=4c_{1}^{2}-2c_{1}K-3\mu^{2}-\mu K
\]
where we have used that $\mu c_{1}=-\mu^{2}$. By \ref{lem:chern classes invertible sheaves Domega},
we get $\Phi(i_{*}\shH)=2\chi(\shH)=-5\mu^{2}/3-\mu K$. Moreover
\[
\Phi(\Sym^{2}\shF)=9c_{1}^{2}-2(2c_{1}^{2}+4c_{2})-3c_{1}K=5c_{1}^{2}-8c_{2}-3c_{1}K
\]
Finally, since $Y_{\alpha}$ is regular of dimension $0$, we obtain
\[
2|Y_{\alpha}|=\Phi(j_{*}\shQ)=\Phi(\det\shF^{2}\otimes\shM)-\Phi(\Sym^{2}\shF)+\Phi(\shF)-\Phi(\shH)=6c_{2}-4\mu^{2}/3
\]
\end{proof}
\begin{lem}
\label{lem:computation characteristic of Delta square}We have
\[
2\chi(\Delta\otimes_{\alB}\Delta)=5c_{1}^{2}-8c_{2}-3c_{1}K-8\mu^{2}/3-2\mu K+6\chi(\odi Y)
\]
\end{lem}
\begin{proof}
Consider the exact sequence \ref{eq:second exact sequence for Sthree}.
Taking into account \ref{lem:chern classes invertible sheaves Domega},
the result follows from the relations $2\chi(\Sym^{2}\shF)=5c_{1}^{2}-8c_{2}-3c_{1}K+6\chi(\odi Y)$,
$2\chi(\odi{D_{\omega}})=-\mu^{2}-\mu k$, $2\chi(\shK)=-5\mu^{2}/3-\mu k$.\end{proof}
\begin{lem}
\label{lem:invariants of double covers}Let $f\colon X\arr X'$ be
a degree $2$ cover between surfaces and write $f_{*}\odi X=\odi{X'}\oplus\shW$,
where $\shW$ is the invertible sheaf inducing $f$. Then
\[
K_{X}^{2}=2K_{X'}^{2}+2c_{1}(\shW)^{2}-4c_{1}(\shW)K_{X'}\comma p_{g}(X)=p_{g}(X')+h^{2}(\shW)\comma\chi(\odi X)=\chi(\odi{X'})+\chi(\shW)
\]
\end{lem}
\begin{proof}
The last two formulas are clear. Therefore we focus on the first.
Set $\alB=f_{*}\odi X$. The map $\shW^{-1}\arr\duale{\alB}\simeq f_{*}\omega_{X/X'}$
induces a map
\[
f^{*}\shW^{-1}\arr f^{*}f_{*}\omega_{X/X'}\arr\omega_{X/X'}
\]
We want to prove that this map is surjective and therefore an isomorphism.
Locally, $\shW=\odi{X'}t$, $t^{2}=m\in\odi{X'}$, $\shW^{-1}=\odi{X'}t^{*}$.
$ $Since $t\cdot t^{*}=1^{*}$, where $t^{*}\in\duale{\alB}$, we
see that $t^{*}$ generates $\duale{\alB}$ as a $\alB$-module. Thus
$f^{*}\shW^{-1}\simeq\omega_{X/X'}$ and
\[
\omega_{X}\simeq f^{*}\omega_{X'}\otimes\omega_{X/X'}\simeq f^{*}(\omega_{X'}\otimes\shW^{-1})\then f_{*}\omega_{X}^{-1}\simeq\omega_{X'}^{-1}\otimes\shW\otimes\alB\simeq\omega_{X'}^{-1}\otimes\shW\oplus\omega_{X'}^{-1}\otimes\shW^{2}
\]
Now note that if $Z$ is a surface and $D$ is a divisor of it, by
Riemann-Roch formula we have
\[
2\chi(nD-K_{Z})=n^{2}D^{2}-3nDK_{Z}+2K_{Z}^{2}+2\chi(\odi Z)
\]
If we write $\shW=\odi Y(C)$, the result comes from the following
relations
\begin{alignat*}{1}
\chi(-K_{X}) & =\chi(f_{*}\omega_{X}^{-1})=\chi(C-K_{X'})+\chi(2C-K_{X'})=(5C^{2}-9CK_{X'})/2+2K_{X'}^{2}+2\chi(\odi{X'})\\
\chi(-K_{X}) & =K_{X}^{2}+\chi(\odi X)=K_{X}^{2}+\chi(\shW)+\chi(\odi{X'})=K_{X}^{2}+(C^{2}-CK_{X'})/2+2\chi(\odi{X'})
\end{alignat*}
\end{proof}
\begin{lem}
\label{lem:invariants of triple covers}Let $f\colon X'\arr Y$ be
a degree $3$ cover between surfaces induced by $(\shF,\delta)\in\shC_{3}(Y)$.
Then we have
\[
K_{X'}^{2}=3K_{Y}^{2}-4c_{1}(\shF)K_{Y}+2c_{1}(\shF)^{2}-3c_{2}(\shF)\comma\chi(\odi{X'})=\chi(\odi Y)+\chi(\shF)\comma p_{g}(X')=p_{g}(Y)+h^{2}(\shF)
\]
\end{lem}
\begin{proof}
The last two formulas are clear. The first one instead is proved in
\cite[Corollary 8.3]{Pardini1989} or \cite[Proposition 10.3]{Miranda1985}.
\end{proof}

\begin{proof}
(\emph{of Theorem }\ref{thm:invariants of regular Sthree covers})
The claim about connectedness follows from \ref{rem:connectdness for regular Sthree covers}.
The formula for $|Y_{0}|$ is given in \ref{lem:cardinality of Yalpha},
while the formula for $p_{g}$ is clear since $f_{*}\odi X\simeq\odi Y\oplus(\shM\otimes\det\shF)\oplus\shF\oplus\shF$.
This relation, Riemann-Roch formula and the computation 
\[
\chi(\shL)=\chi(\shM\otimes\det\shF)=((\mu+c_{1})^{2}-(\mu+c_{1})K)/2+\chi(\odi Y)=(-\mu^{2}-\mu k+c_{1}^{2}-c_{1}K)/2+\chi(\odi Y)
\]
yields the formula for $\chi(\odi X)$. Therefore we focus on the
formula for $K_{X}^{2}$. The map $f\colon X\arr Y$ factors as $X\arrdi{f_{2}}X'\arrdi{f_{3}}Y$
where $f_{2}$, $f_{3}$ are covers of degree $2,3$ respectively
and $X'$ is a surface, thanks to \ref{thm:regular G covers and triple}.
Moreover, by definition, $f_{2}$ is induced by the invertible sheaf
$\Delta$ on $X'$, while $f_{3}$ is induced by $(\shF,\delta)\in\shC_{3}$.
Notice that $X'$ is a complete smooth surface. Set $\Delta=\odi{X'}(-C)$,
where $C$ is a divisor over $X'$ and set $\Delta^{2}=\Delta\otimes_{\odi{X'}}\Delta.$
Recall that $\Delta\simeq\shL\oplus\shF$. By \ref{lem:invariants of triple covers}
and \ref{lem:invariants of double covers} we have
\begin{equation}
K_{X}^{2}=2K_{X'}^{2}+2C^{2}+4CK_{X'}=6K^{2}-8c_{1}K+4c_{1}^{2}-6c_{2}+2C^{2}+4CK_{X'}\label{eq:first equation for KX square for Sthree}
\end{equation}
By Riemann-Roch and the definition of the intersection product we
also get
\[
2(\chi(\Delta)-\chi(\odi{X'}))=C^{2}+CK_{X'}\comma C^{2}=\chi(\Delta^{2})-2\chi(\Delta)+\chi(\odi{X'})
\]
In particular $CK_{X'}=4\chi(\Delta)-\chi(\Delta^{2})-3\chi(\odi{X'})$.
Putting everything together and using \ref{lem:invariants of triple covers}
and \ref{lem:computation characteristic of Delta square} we obtain

\begin{alignat*}{1}
2C^{2}+4CK_{X'} & =-2\chi(\Delta^{2})+12\chi(\Delta)-10\chi(\odi{X'})=-2\chi(\Delta^{2})+12\chi(\shL)+2\chi(\shF)-10\chi(\odi Y)\\
 & =-5c_{1}^{2}+8c_{2}+3c_{1}K+8\mu^{2}/3+2\mu K-6\chi(\odi Y)-6\mu^{2}-6\mu K+6c_{1}^{2}-6c_{1}K+\\
 & 12\chi(\odi Y)+c_{1}^{2}-2c_{2}-c_{1}K+4\chi(\odi Y)-10\chi(\odi Y)\\
 & =2c_{1}^{2}+6c_{2}-4c_{1}K-10\mu^{2}/3-4\mu K
\end{alignat*}
Substituting in \ref{eq:first equation for KX square for Sthree}
we get the desired result.\end{proof}

\nocite{Stillman2002,Maclagan2002,Laumon1999,Ogus2006}

\newcommand{\etalchar}[1]{$^{#1}$}
\providecommand{\bysame}{\leavevmode\hbox to3em{\hrulefill}\thinspace}
\providecommand{\MR}{\relax\ifhmode\unskip\space\fi MR }
\providecommand{\MRhref}[2]{%
  \href{http://www.ams.org/mathscinet-getitem?mr=#1}{#2}
}
\providecommand{\href}[2]{#2}

\end{document}